\documentclass[11pt]{amsart}

\usepackage{amsmath,amsthm, amscd, amssymb, amsfonts}
\usepackage[all]{xy}
\usepackage{enumitem}

\usepackage{fontsmpl}
\usepackage{srcltx}

\usepackage[T2A,T1]{fontenc}
\newcommand\ztt{\mbox{\usefont{T2A}{\rmdefault}{m}{n}\cyrc}}

\newcommand\sh{\mbox{\usefont{T2A}{\rmdefault}{m}{n}\cyrsh}}
\newcommand\Sh{\mbox{\usefont{T2A}{\rmdefault}{m}{n}\CYRSH}}
\newcommand\ch{\mbox{\usefont{T2A}{\rmdefault}{m}{n}\cyrch}}
\newcommand\ya{\mbox{\usefont{T2A}{\rmdefault}{m}{n}\cyrya}}
\newcommand\ur{\mbox{\usefont{T2A}{\rmdefault}{m}{n}\cyryu}}
\newcommand\sch{\mbox{\usefont{T2A}{\rmdefault}{m}{n}\cyrshch}}
\newcommand\Sch{\mbox{\usefont{T2A}{\rmdefault}{m}{n}\CYRSHCH}}
\newcommand\zh{\mbox{\usefont{T2A}{\rmdefault}{m}{n}\cyrzh}}
\newcommand\Zh{\mbox{\usefont{T2A}{\rmdefault}{m}{n}\CYRZH}}

\newcommand\brus{\mbox{\usefont{T2A}{\rmdefault}{m}{n}\cyrb}}
\newcommand\drus{\mbox{\usefont{T2A}{\rmdefault}{m}{n}\cyrd}}

\renewcommand{\_}[1]{_{\left( #1 \right)}}

\allowdisplaybreaks

\newcommand{\supera}[2]{{\mathbf A}(#1|#2)}
\newcommand{\superb}[2]{{\mathbf B}(#1|#2)}

\newcommand{\superd}[2]{{\mathbf D}(#1|#2)}
\newcommand{\superda}[1]{{\mathbf D}(2,1;#1)}
\newcommand{\superf}{{\mathbf F}(4)}
\newcommand{\superg}{{\mathbf G}(3)}

\newcommand{\bm}{\mathbf{m} }
\newcommand{\bM}{\mathbf{M} }

\newcommand{\toba}{{\mathcal B}}
\def\wtoba{\widetilde{\toba}}
\newcommand{\lu}{\mathcal{L}}
\newcommand{\luq}{\lu_{\bq}}
\newcommand{\fd}{finite-dimensional \ }
\newcommand{\cR}{\mathcal R}
\newcommand{\Oc}{{\mathcal O}}
\newcommand{\cI}{{\mathcal I}}
\def\Uc{\mathcal{U}}

\def\cW{\mathcal{W}}
\def\cV{\mathcal{V}}
\def\cX{\mathcal{X}}

\def\wpartial{\widetilde{\partial}}
\newcommand{\VGamma}{\widehat{\Gamma}}

\newcommand{\vsp}{\vspace*{-0.7cm}}

\newcommand{\ztu}{\overline{\zeta}}

\newcommand{\Dchaintwo}[3]{\xymatrix@C-4pt{\overset{#1}{\underset{\ }{\circ}}\ar
@{-}[r]^{#2}
& \overset{#3}{\underset{\ }{\circ}}}}
\newcommand{\Dchainthree}[5]{\xymatrix@C-6pt{
\overset{#1}{\underset{\ }{\circ}}\ar  @{-}[r]^{#2}  & \overset{#3}{\underset{\
}{\circ}}\ar  @{-}[r]^{#4}
& \overset{#5}{\underset{\ }{\circ}} }}

\newcommand{\Dchainfour}[7]{\xymatrix@C-6pt{\overset{#1}{\underset{\ }{\circ}}\ar
@{-}[r]^{#2}
& \overset{#3}{\underset{\ }{\circ}}\ar  @{-}[r]^{#4}  & \overset{#5}{\underset{\
}{\circ}} \ar  @{-}[r]^{#6}
& \overset{#7}{\underset{\ }{\circ}}}}

\newcommand{\Dchainfive}[9]{\xymatrix@C-6pt{\overset{#1}{\underset{\ }{\circ}}\ar
@{-}[r]^{#2}  & \overset{#3}{\underset{\ }{\circ}}\ar  @{-}[r]^{#4}  &
\overset{#5}{\underset{\ }{\circ}}
\ar  @{-}[r]^{#6}  & \overset{#7}{\underset{\ }{\circ}}\ar  @{-}[r]^{#8}  &
\overset{#9}{\underset{\ }{\circ}}}}

\newcommand{\Dtriangle}[6]{
\xymatrix@R-12pt{  &    \overset{#2}{\circ} \ar  @{-}[dl]_{#4}\ar  @{-}[dr]^{#5} & \\
\overset{#1}{\circ} \ar  @{-}[rr]^{#6}  &  &\overset{#3}{\circ} }}

\newcommand{\Dthreefork}[8]{
\rule[-9\unitlength]{0pt}{12\unitlength}
\begin{picture}(28,12)(0,9)
\put(2,10){\ifthenelse{\equal{#1}{l}}{\circle*{2}}{\circle{2}}}
\put(3,10){\line(1,0){10}}
\put(14,10){\ifthenelse{\equal{#1}{m}}{\circle*{2}}{\circle{2}}}
\put(15,10){\line(1,1){7}}
\put(15,10){\line(1,-1){7}}
\put(22,18){\ifthenelse{\equal{#1}{t}}{\circle*{2}}{\circle{2}}}
\put(22,2){\ifthenelse{\equal{#1}{b}}{\circle*{2}}{\circle{2}}}
\put(2,12){\makebox[0pt]{\scriptsize #2}}
\put(8,11){\makebox[0pt]{\scriptsize #3}}
\put(14,12){\makebox[0pt]{\scriptsize #4}}
\put(19,16){\makebox[0pt][r]{\scriptsize #5}}
\put(19,4){\makebox[0pt][r]{\scriptsize #6}}
\put(24,17){\makebox[0pt][l]{\scriptsize #7}}
\put(24,2){\makebox[0pt][l]{\scriptsize #8}}
\end{picture}}

\newcommand{\Drightofway}[8]{\xymatrix@R-6pt{  &    \overset{#6}{\circ} \ar
@{-}[d]_{#4}\ar  @{-}[dr]^{#7} & \\
\overset{#1}{\circ} \ar  @{-}[r]^{#2}  &\overset{#3}{\circ} \ar  @{-}[r]^{#5}
&\overset{#8}{\circ} }}

\usepackage{color}

\numberwithin{equation}{section}\theoremstyle{plain}

\newtheorem{theorem}{Theorem}[section]

\newtheorem{cor}[theorem]{Corollary}

\newtheorem{conjecture}[theorem]{Conjecture}
\newtheorem{pro}[theorem]{Proposition}

\newtheorem{prop}{Proposition}[subsection]

\theoremstyle{definition}

\newtheorem{definition}[theorem]{Definition}
\newtheorem{exa}[theorem]{Example}
\newtheorem{example}[theorem]{Example}

\newtheorem{question}[equation]{Question}

\theoremstyle{remark}
\newtheorem{rem}[theorem]{Remark}
\newtheorem{remark}[theorem]{Remark}

\newcommand{\ydh}{{}^{H}_{H}\mathcal{YD}}

\newcommand{\ydg}{{}^{\ku\Gamma}_{\ku\Gamma}\mathcal{YD}}

\newcommand{\ydl}{{}_{L}^{L}{\mathcal{YD}}}

\def\ydgr{{}^{\ku G}_{\ku G}\mathcal{YD}}

\newcommand{\uno}{{\bf 1}}
\def\Ss{\mathcal{S}}

\newcommand{\SJ}{\Zh}

\def\zt{\Z^{\theta}}
\def\G{\mathbb{G}}

\def\Jb{{\tiny \mathbb J}}
\def\I{\mathbb{I}}
\newcommand\id{\operatorname{id}}
\newcommand\tp{\operatorname{top}}

\newcommand\ord{\operatorname{ord}}

\newcommand\Alg{\operatorname{Alg}}

\newcommand\img{\operatorname{Im}}

\newcommand\GK{\operatorname{GK-dim}}
\newcommand\sdim{\operatorname{sdim}}
\newcommand\End{\operatorname{End}}

\newcommand\rk{\operatorname{rank}}

\newcommand\car{\operatorname{char}}
\newcommand\mult{\operatorname{mult}}
\newcommand\gr{\operatorname{gr}}

\newcommand\ad{\operatorname{ad}}

\newcommand\Rep{\operatorname{Rep}}

\newcommand\re{\operatorname{re}}
\newcommand\im{\operatorname{im}}

\def\k{\Bbbk}
\def\ku{\Bbbk}
\def\kk{\mathbb F}

\def\qt{\mathtt{q}}

\def\bq{\mathfrak{q}}
\def\bp{\mathfrak{p}}
\newcommand{\pa}{\mathbf p}

\newcommand\nct{\mathbb{N}_0^\theta}

\newcommand\n{\mathfrak{n}}
\newcommand\ntp{\widetilde{\mathfrak{n}}_+}
\newcommand\ntm{\widetilde{\mathfrak{n}}_-}

\newcommand\np{\mathfrak{n}_+}
\newcommand\nm{\mathfrak{n}_-}

\newcommand\rg{\mathfrak{r}}
\newcommand\nt{\widetilde{\mathfrak{n}}}

\def\ot{\otimes}

\def\s{\mathbb{S}}

\def\Z{\mathbb{Z}}
\def\N{\mathbb{N}}

\def\B{\mathfrak{B}}

\def\eps{\varepsilon}

\def\cC{\mathcal{C}}

\def\D{\mathcal{D}}

\def\mP{\mathcal{P}}

\def\cQ{\mathcal{Q}}

\def\xx{\mathbb{X}}

\def\br{\mathfrak{br}}
\def\brj{\mathfrak{brj}}
\def\bgl{\mathfrak{wk}}
\def\el{\mathfrak{el}}
\def\g{\mathfrak{g}}
\def\h{\mathfrak{h}}
\def\gt{\widetilde{\mathfrak{g}}}
\def\sli{\mathfrak{sl}}
\def\ufo{\mathfrak{ufo}}

\newcommand{\Brown}{\mathtt{br}}
\newcommand{\Sbrown}{\texttt{brj}}
\newcommand{\El}{\mathtt{el}}
\newcommand{\Bgl}{\mathtt{wk}}

\newcommand{\Br}{\mathbf{Br}}
\newcommand{\SBrown}{\mathbf{Brj}}

\def\Ufo{\mathtt{ufo}}
\def\gtt{\mathtt{g}}
\def\Gtt{\mathtt{G}}
\def\Ftt{\mathtt{F}}

\def\Att{\mathtt{A}}
\def\att{\mathtt{a}}
\def\oatt{\overline{\mathtt{a}}}
\def\Btt{\mathtt{B}}

\def\Dtt{\mathtt{D}}

\newcommand{\dn}[1]{{\mathbf D}\left(#1\right)}
\newcommand{\dnp}[1]{{\mathbf D'}\left(#1\right)}

\newcommand{\vtxgpd}{\textcolor{blue}{\bullet}}

\newcommand{\J}{{\mathcal J}}
\newcommand{\Gc}{{\mathcal G}}
\newcommand{\Ec}{{\mathcal E}}

\def\pf{\begin{proof}}
\def\epf{\end{proof}}

\begin{document}


\title[Finite dimensional Nichols algebras of diagonal type]{On Finite dimensional
Nichols algebras of diagonal type}
\author[Andruskiewitsch; Angiono]
{Nicol\'as Andruskiewitsch, Iv\'an Angiono}

\address{FaMAF-CIEM (CONICET), Universidad Nacional de C\'ordoba,
Medina A\-llen\-de s/n, Ciudad Universitaria, 5000 C\' ordoba, Rep\'
ublica Argentina.} \email{(andrus|angiono)@famaf.unc.edu.ar}

\thanks{The work of N. A. and I. A. was partially supported by CONICET, Secyt (UNC), the
MathAmSud project GR2HOPF. The work of I. A. was partially supported by ANPCyT
(Foncyt). The work of N. A., respectively I. A., was partially done during a visit to the
University of Hamburg, respectively the MPI (Bonn), supported by the Alexander von
Humboldt Foundation}

\begin{abstract}
This is a survey on Nichols algebras of diagonal type with finite dimension, or more generally with arithmetic root system.
The knowledge of these algebras is the cornerstone of the classification
program of pointed Hopf algebras with finite dimension, or finite Gelfand-Kirillov dimension;
and their structure should be indispensable for the understanding of the representation theory, the computation of the various cohomologies, and many other aspects of finite dimensional pointed Hopf algebras.
These Nichols algebras were classified in \cite{H-classif RS} as a notable application of the notions of Weyl groupoid and generalized root system \cite{H-Weyl gpd,HY}.
In the first part of this monograph, we give an overview of the theory of 
Nichols algebras of diagonal type. This includes  a discussion of 
the notion of generalized root system and its  appearance  in the contexts of Nichols algebras of diagonal type and  (modular) Lie superalgebras.
In the second and third part, we describe for each Nichols algebra in the list of \cite{H-classif RS}
the following basic information:
the  generalized root system; its label in terms of Lie theory; 
the defining relations found in \cite{A-jems,A-presentation};
the PBW-basis;  the dimension or the Gelfand-Kirillov dimension;
the associated Lie algebra as in \cite{AAR2}.
Indeed the second part deals with Nichols algebras related to Lie algebras and superalgebras in arbitrary characteristic,
while the third contains the information on Nichols algebras related to Lie algebras and superalgebras only in small characteristic, and the few examples
yet unidentified in terms of Lie theory.
\end{abstract}

\maketitle
\begin{center}{\textit{Alles Gescheidte ist schon gedacht worden, man mu\ss\ nur versuchen \newline es noch einmal zu denken.}}\end{center}

 \begin{flushright}
{\sc   Goethe}
 \end{flushright}

\setcounter{tocdepth}{2}
\tableofcontents

\newpage
\section*{Introduction}

\subsection*{What is a Nichols algebra?}

\renewcommand\thesubsubsection{\textbf{\arabic{subsubsection}}}

\subsubsection{}
Let $\ku$ be a field, $V$  a vector space and $c  \in GL(V \ot V)$. The  braid equation on $c$  is
\begin{align}
\label{eq:braid}(c\ot \id)(\id\ot c)(c\ot \id) = (\id\ot c)(c\ot \id)(\id\ot c).
\end{align}
If $c$ satisfies  \eqref{eq:braid}, then the pair $(V, c)$ is a \emph{braided vector space}.
The braid equation, or the closely related quantum Yang-Baxter equation, is the key to many
developments in the last 50 years in several areas in mathematics and theoretical physics. 
Ultimately these applications come from the representations $\varrho_n$ of the braid groups 
$\mathbb B_n$ on $T^n(V)$ induced by \eqref{eq:braid}, for $n \ge 2$.
Indeed, let $\I_{n} := \{1, 2, \dots, n\}$, where $n$ is a natural number. 
Recall that $\mathbb B_n$ is presented by generators $(\sigma_j)_{j \in \I_{n-1}}$
with relations
\begin{align}
\label{eq:braid-rel}\sigma_j  \sigma_k &= \sigma_k\sigma_j, & \vert j-k \vert &\geq 2, &
\sigma_j  \sigma_k\sigma_j &= \sigma_k\sigma_j\sigma_k, & \vert j-k \vert &= 1.
\end{align}
Thus $\varrho_n$ applies $\sigma_j \mapsto \id_{V^{\ot (j-1)}} \ot c \ot \id_{V^{\ot (n - j-1)}}$.

\subsubsection{} Assume that $\car \ku \neq 2$.
Let $c$ be a \emph{symmetry}, i.e. a solution  of \eqref{eq:braid} such that $c^2 = \id$. 
Then $\varrho_n$ factorizes through the representation $\widetilde{\varrho}_n$ of the symmetric group $\s_n$ given by
$s_j := (j \, j+1) \mapsto \id_{V^{\ot (j-1)}} \ot c \ot \id_{V^{\ot (n - j-1)}}$.
The symmetric algebra of $(V,c)$ is the quadratic algebra 
\begin{align*}
S_c(V) = T(V)/ \langle \ker (c + \id) \rangle = \oplus_{n\in \N_0} S^n_c(V).
\end{align*}
For instance, if $c = \tau$ is the usual transposition, then $S_c(V) = S(V)$, the classical symmetric algebra;
while if $V = V_0 \oplus V_1$ is a super vector space and $c$ is the super transposition, 
then $S_c(V) \simeq S(V_0) \underline{\ot} \Lambda (V_1)$, 
the super symmetric algebra.

\smallbreak
The adequate setting for such symmetries is that of symmetric tensor categories,
advocated by Mac Lane in 1963. In this context, the symmetric algebra satisfies the  same universal property
as in the classical definition. 
In particular, symmetric algebras are Hopf algebras in symmetric tensor categories.
Assume that $\car \ku = 0$. Then, as vector spaces,
\begin{align}\label{eq:ScV}
S^n_c(V) \simeq T^n(V)^{\s_n} = \img \textstyle \int_n \simeq T^n(V) / \ker \int_n,
\end{align}
where $\int_n = \sum_{s \in \s_n}  \widetilde{\varrho}_n(s):T^n(V) \to T^n(V)$.

\subsubsection{}\label{subsec:intro-yd} The adequate setting for braided vector spaces is that of braided tensor categories \cite{JS}; 
there is a natural notion of Hopf algebra in such categories. 

Let $H$ be a Hopf algebra (with bijective antipode). Then $H$ gives rise to a braided tensor category $\ydh$ \cite{Dr},
and consequently is a source of examples of braided vector spaces. Namely, an object $M \in \ydh$, called a Yetter-Drinfeld
module over $H$, is simultaneously a left $H$-module and a left $H$-comodule satisfying the compatibility condition
\begin{align}\label{eq:yd-compatibility}
\delta(h \cdot v) &= h\_1v\_{-1} \Ss(h\_3) \ot h\_2 \cdot v\_0,& h\in H, \, v&\in V.
\end{align}
This is a braided tensor category with the usual tensor product of modules and comodules, 
and braiding
\begin{align}\label{eq:braid-yd-h} 
c_{M, N}(x\ot y) &= x\_{-1} \cdot y \ot x\_{0}, &M, N \in \ydh,\, x&\in M,\,  y\in N.
\end{align}
For $M \in \ydh$,  $c = c_{M, M} \in  GL(M \ot M)$ satisfies the braid equation \eqref{eq:braid}.

\smallbreak
If $M \in \ydh$,  then the tensor algebra $T(M)$ is a Hopf algebra in $\ydh$, whose coproduct is determined by 
$\Delta(x) = x \ot 1 + 1 \ot x$ for $x \in M$.
Also the tensor coalgebra  $T^c(M)$ is a Hopf algebra in $\ydh$, with braided shuffle product.
See \cite[Proposition 9]{Ro2}.

\subsubsection{}\label{subsec:intro-nichols-def}
Let $(V, c)$ be a braided vector space but $c$ not necessarily a symmetry. The Nichols algebra 
$\toba(V) = \oplus_{n\in \N_0} \toba^n (V)$ of $(V,c)$ is a graded connected algebra with a number of 
remarkable properties that has at least superficially  a resemblance with
a symmetric algebra\footnote{It is customary to omit $c$ in the notation of the Nichols algebra.}. 
For, let  $M_n: \s_n \to \mathbb B_n$ be the (set-theoretical)
Matsumoto section, that preserves the length and satisfies $M_n(s_j) = \sigma_j$. 
Let $\Omega_n = \sum_{\sigma \in \s_n} \varrho_n (M_n(\sigma))$ and $\J^n(V) = \ker \Omega_n$. Define
\begin{align}\label{eq:def-nichols-intro}
\J(V) &= \oplus_{n\ge 2}\J^n(V), &\toba(V) &=  T(V) / \J(V).
\end{align}

Despite the similarity of \eqref{eq:ScV} and \eqref{eq:def-nichols-intro}, Nichols algebras have profound divergences with symmetric algebras--and various analogies.
\begin{enumerate}[leftmargin=*, label=\rm{(\alph*)}]
\item  The subspace $\J(V)$ is actually a two-sided ideal of $T(V)$, so that $\toba(V)$ is a connected graded algebra generated in degree 1.
However $\J(V)$ is seldom quadratic, and  it might well be not finitely generated. 
The determination of $\J(V)$ is one of the central problems of the subject.

\item Although \eqref{eq:def-nichols-intro} is a compact definition, it hides the rich structure of Nichols algebras.
Indeed,  $\toba(V)$ is a Hopf algebra in $\ydh$ for suitable $H$.
Even more, it is a coradically graded coalgebra, a notion dual to generation in degree one. 

\item If $V$ is a \fd vector space, then  $S(V^*)$ is identified with the algebra of differential operators on $S(V)$ (with constant coefficients).
An analogous description is available for Nichols algebras, being useful to find relations of $\toba(V)$.

\item Nichols algebras appeared in various fronts. In \cite{Ni}, they were defined for the first time as a tool to construct new examples of Hopf algebras.
They are instrumental for the attempt in \cite{W} to define a non-commutative differential calculus on Hopf algebras. 
Also, the positive part $U^+_q(\g)$ of the quantized
enveloping algebra of a Kac-Moody algebra $\g$ at a generic parameter $q$ turns out to be a Nichols algebra \cite{Lu,Ro1,Sch}.
\end{enumerate}

By various reasons, we are also led to consider:

\begin{itemize}
\item \emph{Pre-Nichols algebras} of the braided vector space $(V,c)$ \cite{Masuoka,A-pre-Nichols}; these are graded connected Hopf algebras in $\ydh$,
say $\toba = \oplus_{n\in \N_0} \toba^n$, with $\toba^1 \simeq V$, that are generated in degree 1 (but not necessarily coradically graded). Thus we have epimorphisms of Hopf algebras in $\ydh$
\begin{align*}
\xymatrix{T(V) \ar@{->>}[rr]  && \toba \ar@{->>}[rr]  && \toba(V).}
\end{align*}

\item \emph{Post-Nichols algebras} of the braided vector space $(V,c)$ \cite{AAR}; these are graded connected Hopf algebras in $\ydh$,
say $\Ec = \oplus_{n\in \N_0} \Ec^n$, with $\Ec^1 \simeq V$, that are coradically graded (but not not necessarily generated in degree 1). Thus we have monomorphisms of Hopf algebras in $\ydh$
\begin{align*}
\xymatrix{\toba(V) \ar@{^{(}->}[rr]  && \Ec \ar@{^{(}->}[rr]  && T^c(V).}
\end{align*}
\end{itemize}
Thus, the only pre-Nichols algebra that is also post-Nichols is $\toba(V)$ itself.

\subsection*{Classes of Nichols algebras}

\subsubsection{} Nichols algebras are basic invariants of Hopf algebras that are
not generated by its coradical \cite{AS1,AC}; see the discussion in \S \ref{subsec:nichols-classification-Hopf}.
One is naturally led to the following questions:

\subsubsection*{Classify all $V \in \ydh$ such that the (Gelfand-Kirillov) dimension of $\toba(V)$ is finite.
For such $V$, determine the generators of the ideal $\J(V)$ and all post-Nichols algebras $\toba(V) \hookrightarrow \Ec$
with finite (Gelfand-Kirillov) dimension.}

\

\smallbreak
Now $\toba(V)$ is a Hopf algebra in $\ydh$ but the underlying algebra depends only on the braiding $c$, 
and reciprocally the same braided vector space can be realized in $\ydh$ in many ways and for many $H$'s. 
That is, we may deal with the above problems for
suitable classes of braided vector spaces. 

\smallbreak
Also, assume that $(V,c)$ satisfies, for some $\theta \in \N_{>1}$, 
\begin{align}\label{eq:braiding-generalform}
V &=  V_{1} \oplus \dots   \oplus V_\theta,& 
c(V_i \otimes V_j) &=  V_j \otimes V_i,& i,j&\in \I_{\theta}.
\end{align}
So, we may suppose that the $\toba(V_i)$'s are known 
and try to infer the shape of $\toba(V)$ from them and the cross-braidings $c_{\vert V_i \otimes V_j}$; 
this viewpoint leads to a rich combinatorial analysis \cite{H-Weyl gpd,HY,AHS,HS-london,HV,AAH}.

\subsubsection{} The simplest yet most fundamental examples are those $(V,c)$ satisfying 
\eqref{eq:braiding-generalform} with $\dim V_i = 1$, $i\in \I = \I_{\theta}$.
Pick $x_i \in V_i - 0$; then $(x_i)_{i\in \I}$ is a a basis of $V$, and 
$c(x_i\ot x_j)=q_{ij}\, x_j\ot x_i$,  $i,j\in \I$, where $q_{ij}  \in \ku^{\times}$.
We say that $(V, c)$ is a \emph{braided vector space of diagonal type} if
\begin{align*}
q_{ii} &\neq 1, & \text{for all } i &\in \I.
\end{align*}
Notice that this condition, assumed by technical reasons, is not always required in the literature.
See \cite[Lemma 2.8]{AAH}.

This class appears naturally when $H = \ku \Gamma$, where $\Gamma$ is an abelian group, but also 
lays behind any attempt to argue inductively. Other classes of braided vector spaces were considered in the literature:  

\begin{enumerate}[leftmargin=*, label=\rm{(\alph*)}]
\item Triangular type, see \cite{AAH,U}.

\item Rack type, arising from non-abelian groups, see references in \cite{A}.

\item  Semisimple (but not simple) Yetter-Drinfeld modules, \cite{AAH,HS,HS-london,HV} and references therein.

\item Yetter-Drinfeld modules over Hopf algebras that are not group algebras, see for example 
\cite{AGM,GGi,HX,AGi,AA-nichols-quantum}.

\end{enumerate}

\subsection*{Nichols algebras of diagonal type}

\subsubsection{} Assume now that $\ku$ is algebraically closed and of characteristic 0.
The classification of the braided vector spaces $(V,c)$ of diagonal type with finite-dimensional $\toba(V)$
was obtained in \cite{H-classif RS}. 
(When $\car\ku > 0$, the classification is known under the hypothesis $\dim V \leq 3$ \cite{HW,Wg}).
The core of the approach is the notion of generalized root system;
actually, the paper \cite{H-classif RS} contains the list of all $(V,c)$ of diagonal type with connected Dynkin diagram
and finite generalized root system (these are called \emph{arithmetic}). The list can be roughly split in several classes:

\begin{itemize} [leftmargin=*]\renewcommand{\labelitemi}{$\diamond$}
\medbreak\item  Standard type \cite{AA}, that includes Cartan type \cite{AS2}; related to the Lie algebras in the Killing-Cartan classification.

\medbreak\item Super type \cite{AAY}, related to the finite-dimensional contragredient Lie superalgebras in characteristic 0, classified in \cite{K-super}.

\medbreak\item Modular type \cite{AA-GRS-CLS-NA}, related to the finite-dimensional  contragredient Lie (super)algebras in positive characteristic, classified in \cite{KW-exponentials,BGL}.

\medbreak\item A short list of examples not (yet) related to Lie theory, baptised \textit{UFO}'s.
\end{itemize}

The goal of this work is to give exhaustive information on the structure of these Nichols algebras.

\subsubsection{} This monograph has three Parts. Part \ref{part:general} is an exposition of the basics of Nichols algebras of diagonal type.
Section \ref{sect:Preliminaries} is a potpourri of various topics needed for further discussions. The bulk of this Part is 
Section \ref{sec:nichols-diagonal} where the main notions that we display later are explained: PBW-basis, generalized root systems, and so on. 
In Parts \ref{part:cartan-super-standard} and III we give the list of all finite-dimensional Nichols algebra of diagonal type (with connected Dynkin diagram)
classified in \cite{H-classif RS} and for each of them, its fundamental information. For more details see Section \ref{sec:outline}, page \pageref{sec:outline}.

\renewcommand\thepart{\Roman{part}}
\renewcommand\thesubsubsection{\thesubsection.\arabic{subsubsection}}
\part{General facts}\label{part:general}

\section{Preliminaries}\label{sect:Preliminaries}

\subsection{Notation} In this paper, $\N = \{1, 2, 3, \dots\}$ and $\N_0 = \N \cup \{0\}$.
If $k < \theta \in \N_0$, then we denote $\I_{k, \theta} = \{n\in \N_0: k\le n \le \theta \}$.
Thus $\I_{\theta} = \I_{1, \theta}$.

The base field $\ku$ is algebraically closed of characteristic zero (unless explicitly stated); we set $\ku^{\times} = \ku - 0$. 
All algebras will
be considered over $\k$. If $R$ is an algebra and $J\subset R$,
we will denote by $\langle J\rangle$ the 2-sided ideal generated by $J$ and by $\k \langle
J\rangle$ the subalgebra generated by $J$, or $\ku[J]$ if $R$ is commutative. Also $\Alg(R,\ku )$ denotes the set of algebra maps from $R$ to $\ku$.

For each integer $N>1$, $\G_N$  denotes
the group of $N$-th roots of unity in $\ku$, and $\G'_N$ is the corresponding subset of primitive roots (of order $N$).
Also  $\G_{\infty} = \bigcup_{N \in \N} \G_N$, $\G'_{\infty} =\G_{\infty} - \{1\}$.
We will denote  by $\Gamma$ an abelian group and by $\VGamma$ the group of  characters of $\Gamma$.

We shall use the notation for $q$-factorial numbers: for  $q\in\ku^\times$,
$n\in\N$,
\begin{align*}
(0)_q! =1,&
(n)_q&=1+q+\ldots+q^{n-1}, & (n)_q!&=(1)_q(2)_q\cdots(n)_q.
\end{align*}

\subsection{Kac-Moody algebras}\label{subsec:kac-moody}
Recall from \cite{K-libro} that $A = (a_{ij})\in \Z^{\theta \times \theta}$ is a \emph{generalized Cartan matrix} (GCM) if
for all $i, j \in \I$
\begin{align}\label{eq:gcm}
a_{ii} &= 2, & a_{ij} &\le 0, \ i\neq j,& a_{ij} &= 0 \iff a_{ji} =0.
\end{align}
It is equivalent to give Cartan matrix or to  give a Dynkin diagram, cf. \cite{K-libro}. 
Also, $A$ is indecomposable if its Dynkin diagram is connected, see \cite{K-libro};
and is symmetrizable if there exists a diagonal matrix $D$ such that $DA$ is symmetric.
Indecomposable and symmetrizable GCM's fall into one of three classes:
\begin{enumerate}[leftmargin=*, label=\rm{(\alph*)}]
\item Finite; those whose corresponding Kac-Moody algebra has finite dimension, i.e. those in Killing-Cartan classification.

\item Affine; those whose corresponding Kac-Moody algebra has infinite dimension but is of polynomial growth.

\item Indefinite; the rest.

\end{enumerate}

Let $A$ be a GCM. We denote by $\g(A)$ the corresponding Kac-Moody algebra, see \cite{K-libro} and \S \ref{subsec:Weyl-gpd-super} below.
Also $\varDelta^A_+$ denotes the set of positive roots,
so that $\varDelta^A = \varDelta^A_+ \cup -\varDelta^A_+$ is the set of all roots. Further, $\varDelta^{A, \text{re}}_+$
, $\varDelta^{A, \text{im}}_+$, are the sets of positive \emph{real}, respectively \emph{imaginary},  roots.

\subsection{Hopf algebras}
We use standard notation for coalgebras and Hopf algebras: the coproduct is denoted by $\Delta$, the counit by $\varepsilon$ and the antipode by $\Ss$. 
For the first we use the Heyneman-Sweedler notation $\Delta(x)=\sum x_{(1)}\ot x_{(2)}$; the summation sign will be often omitted. 
All Hopf algebras are supposed to have bijective antipode; the composition inverse of $\Ss$ is denoted by $\overline{\Ss}$.
Let $H$ be a Hopf algebra. 
We denote by $G(H)$ the set of group-like elements of $H$.
The tensor category of finite-dimensional representations of $H$ is denoted $\Rep H$.
If the group of group-likes $G(H)=\Gamma$ is abelian, $g\in \Gamma$ and
$\chi\in\widehat\Gamma$, then
$\mP_{(1,g)}^\chi(H)$ denotes the isotypical component of type $\chi$ of the space of
$(1,g)$-primitive elements.

For more information on Hopf algebras see \cite{Mo,radford-book}.

\subsection{Yetter-Drinfeld modules}\label{subsec:yd}
The definition of these was given in \S \ref{subsec:intro-yd} of the Introduction.

\medbreak As  in every monoidal category, there are algebras and coalgebras in $\ydh$:
\begin{itemize}[leftmargin=*]
\item $(A, \mu)$\emph{ is an algebra in} $\ydh$ means that $A$ is an object in $\ydh$ that bears an associative unital multiplication $\mu$ such that $\mu: A \ot A \to A$ 
and the unit $u: \ku \to A$ are morphisms in $\ydh$.
\item $(C, \Delta)$\emph{ is a coalgebra in} $\ydh$ means that $C$ is an object in $\ydh$ that bears a  coassociative counital 
comultiplication $\Delta$ such that $\Delta: C \to C \ot C$ 
and the counit $\eps: C\to \ku$ are morphisms in $\ydh$.
\end{itemize}

The category of algebras in $\ydh$ is again monoidal; if $(A, \mu_A)$, $(B, \mu_B)$ are algebras in $\ydh$,
then  $A \underline{\ot} B := (A \ot B, \mu_{A \ot B})$ also is, where 
\begin{align}\label{eq:tensor-algebras}
\mu_{A \ot B} = (\mu_{A} \ot \mu_{B}) (\id_A \ot c_{B, A} \ot \id_B).
\end{align}
Analogously, if $(C, \Delta_C)$, $(D, \Delta_D)$ are coalgebras in $\ydh$,
then $C \underline{\ot} D := (C \ot D, \Delta_{C \ot D})$ also is, where 
\begin{align}\label{eq:tensor-coalgebras}
\Delta_{C \ot D} = (\id_C \ot c_{C, D} \ot \id_D) (\Delta_{C} \ot \Delta_{D}).
\end{align}

\medbreak
We are mainly interested in the case 
$H=\k \Gamma$, where $\Gamma$ is an abelian group;
a Yetter-Drinfeld module over $\ku \Gamma$ is a $\Gamma$-graded vector space $M =\bigoplus_{t\in \Gamma}M_t$ 
provided with a linear action of $\Gamma$ such that 
\begin{align}\label{eq:yd-compatibility-G}
t\cdot M_h &= M_{h},& t,h&\in \Gamma.
\end{align}
Here \eqref{eq:yd-compatibility-G} is just \eqref{eq:yd-compatibility} in this setting.
Morphisms in $\ydg$ are linear maps preserving the action and the grading.
Let  $M\in\ydg$. Then  we set
\begin{align*}
M^{\chi}_t &= \{v\in M_t: h\cdot v=\chi(h)v, \forall h\in \Gamma \},& t&\in \Gamma, \, \chi\in\VGamma.
\end{align*}
The braiding $c_{M, N}: M\ot N\to N\ot M$, cf. \eqref{eq:braid}, is given by
\begin{align}\label{eq:braid-yd-gps} 
c_{M, N}(x\ot y) &= t \cdot y \ot x, & x&\in M_t,\, t\in \Gamma,\, y\in N.
\end{align}

\subsection{Braided Hopf algebras}\label{subsec:braidedhopf}
Since it is a braided monoidal category, there are also Hopf algebras in $\ydh$. 
Let us describe them explicitly when $ H = \ku \Gamma$, $\Gamma$  an abelian group, for illustration.
A \emph{Hopf algebra} in  $\ydg$  is a collection $(R, \mu, \Delta)$, where
\begin{itemize}[leftmargin=*]
\item $R\in \ydg$; 
\item $(R, \mu)$ is an algebra  in $\ydg$ and 
$(R, \Delta)$ is a coalgebra in $\ydg$;
\item $\Delta: R \to R \underline{\ot} R$ and $\eps: R \to \ku$ are algebra maps;
\item $R$ has an antipode $\Ss_R$, i.~e. a convolution inverse of the identity of $R$.
\end{itemize}

Let $R$ be a Hopf algebra in $\ydg$. Then $R\#\ku \Gamma = R\ot \ku \Gamma$ with the smash product algebra and smash
coproduct coalgebra structures is a Hopf algebra, called the \emph{bosonization  of $R$ by $\ku \Gamma$}, see \cite[Theorem 6.2.2]{Ma}.

The adjoint representation  $\ad_c: R\to \End R$ is the linear map given by
\begin{align}\label{braided-adj}
\ad_c x(y) &= \mu(\mu\otimes\mathcal S)(\id\otimes c)(\Delta\otimes\id)(x\otimes y),&
x,y&\in R.
\end{align}
It can be shown that $\ad_c x(y) = \ad x(y)$, $x,y \in R$, where ad is the adjoint representation of $R\#\ku \Gamma$.

\subsection{Nichols algebras}\label{subsec:nichols}
We are now ready to discuss the central notion of this monograph. At the beginning we place ourselves in the context
of a general Hopf algebra $H$ with bijective antipode, although for the later discussion from Section 2 on, $ H = \ku \Gamma$, $\Gamma$  an abelian group, is general enough.

\smallbreak
Let $V \in \ydh$. Clearly the tensor algebra $T(V)$  and the tensor coalgebra $T^c(V)$ are objects in $\ydh$. 
Then:

\begin{itemize} [leftmargin=*]\renewcommand{\labelitemi}{$\circ$}
\medbreak\item  There is an algebra map $\Delta: T(V) \to T(V) \underline{\ot} T(V)$ determined by 
\begin{align*}
\Delta(v) &= v \otimes 1 + 1 \ot v,& v&\in V;
\end{align*}
with this, $T(V)$ is a graded Hopf algebra in $\ydh$.

\medbreak\item  There is a coalgebra map $\mu: T^c(V)  \underline{\ot} T^c(V) \to T^c(V)$ determined by 
\begin{align*}
\mu(v \otimes 1) &= v = \mu(1 \ot v),& v&\in V;
\end{align*}
with this, $T^c(V)$ is a graded Hopf algebra in $\ydh$  \cite[Proposition 9]{Ro2}.

\medbreak\item There is a  morphism $\Omega: T(V) \to T^c(V)$ of graded Hopf algebras in $\ydh$ such that $\Omega_{|V} = \id_V$. We denote $\Omega_n= \Omega_{| T^n(V)}$, so that $\Omega = \sum_n \Omega_n$.
\end{itemize}

\begin{definition}\label{def:nichols}
The Nichols algebra $\toba(V)$ is the quotient of the tensor algebra $T(V)$ by the ideal $\J(V) := \ker \Omega$,
which is (isomorphic to) the image of the map $\Omega$. 
Thus, $\J(V) = \oplus_{n \ge 2} \J^n(V)$, where $\J^n(V) = \ker \Omega_n$.
\end{definition}

Nichols algebras play a fundamental role in the classification of pointed Hopf algebras, see \cite{AS3}
and \S \ref{subsec:nichols-classification-Hopf} below. As algebras or coalgebras, their structure
depends only on the braided vector space $(V,c)$ and not on the realization in $\ydh$, cf. Proposition \ref{prop:nichols-alternative} \ref{item:nichols-alternative-a}. We now state
various alternative descriptions of $\toba(V)$, or more precisely of the ideal $\J(V) = \oplus_{n\ge 2} \J^n(V)$. To start with we introduce the left and right skew derivations. See e.g. \cite{AHS} for more details. 
Let $f \in V^*$. Let $\partial^L_{f} = \wpartial_{f} \in \End  T(V)$ be given by
\begin{align}\label{eq:partial-0}
\wpartial_f(1) &=0,& &\wpartial_f(v) = f(v), \forall v\in V,
\\\label{eq:partial-left}
\wpartial_f(xy) &= \wpartial_f(x)y + \sum_i x_i \wpartial_{f_i}(y), & &\text{where } c^{-1} (f \ot x) = \sum_i x_i \ot f_i.
\end{align}
Analogously, let  $\partial^R_{f} = \partial_{f}\in \End  T(V)$ be   given by \eqref{eq:partial-0} and
\begin{align}
\label{eq:partial-right}
\partial_f(xy) &= x\partial_f(y) + \sum_j  \partial_{f_j}(x) y_j, & \text{where } c^{-1} (y \ot f) = \sum_j f_j \ot y_j.
\end{align}

Let us fix a basis $(x_i)_{i\in \I}$ of $V$ and let $(f_i)_{i\in \I}$  be its dual basis; set $\partial_i = \partial_{f_i}$,
$\wpartial_i = \wpartial_{f_i}$, $i\in \I$. 
There is a particular instance where \eqref{eq:partial-right} has a simpler expression:
assume that there exists a family $(g_i)_{i\in \I}$ in $G(H)$
such that $\delta(x_i) = g_i \ot x_i$, for every $i \in \I$.  Then \eqref{eq:partial-right} for all $f$ is equivalent to
\begin{align}\label{eq:partial-right-diag}
\partial_i(xy) &= x\partial_i(y) + \partial_i(x) \,g_i\cdot y,&
x,y&\in T(V),& i&\in \I.
\end{align}

\medbreak
Recall the Matsumoto section $M_n: \s_n \to \mathbb B_n$, cf. \S \ref{subsec:intro-nichols-def} of the Introduction.
Here are the promised alternative descriptions.

\begin{pro}\label{prop:nichols-alternative}
\begin{enumerate}[leftmargin=*,label=\rm{(\alph*)}]
\item\label{item:nichols-alternative-a}  $\Omega_n = \sum_{\sigma \in \s_n} \varrho_n (M_n(\sigma))$.

\medbreak\item\label{item:nichols-alternative-b} $\J(V)$ is maximal  in the class of graded Hopf ideals $J = \oplus_{n \ge 2} J^n$ in $T(V)$ that are sub-objects in $\ydh$.

\medbreak\item\label{item:nichols-alternative-c} $\J(V)$ is maximal  in the class of graded Hopf ideals $J = \oplus_{n \ge 2} J^n$ in $T(V)$ that are categorical braided subspaces of $T(V)$ in the sense of \cite{T}.

\medbreak\item\label{item:nichols-alternative-d} $\J(V)$ is the radical of the natural Hopf pairing 
$T(V^*) \otimes T(V) \to \ku$ induced by the evaluation $V^*\times V\to \Bbbk$.

\medbreak\item\label{item:nichols-alternative-dbis} If $\toba =  \oplus_{n \ge 0} \toba^n$ is a graded Hopf algebra in $\ydh$ with 
$\toba^0 = \ku$, $\mathcal{P} (\toba) = \toba^1 \simeq V$, $\toba = \ku \langle \toba^1 \rangle$, then $\toba \simeq \toba(V)$.

\medbreak\item\label{item:nichols-alternative-e} Let $x \in T^n(V)$, $n \ge 2$.  
If $\partial_{f} (x) =0$ for all $f$ in a basis of $V^*$, then $x \in \J^n(V)$.

\medbreak\item\label{item:nichols-alternative-f} Let $x \in T^n(V)$, $n \ge 2$.  
If $\wpartial_{f} (x) =0$ for all $f$ in a basis of $V^*$, then $x \in \J^n(V)$. 
\end{enumerate}
\end{pro}

The proofs of various parts of this Proposition can be found e.g. in \cite{AG, AHS, AS3, Lu, Ro1, Ro2, Sch}.
The characterizations \ref{item:nichols-alternative-a}, \ref{item:nichols-alternative-b} and \ref{item:nichols-alternative-c} are useful theoretically;
in practice, \ref{item:nichols-alternative-a} is applicable only for small $n$, mostly $n=2$. 
In turn, \ref{item:nichols-alternative-d}, \ref{item:nichols-alternative-e} or \ref{item:nichols-alternative-f}
are suitable for explicit computations; notice that an iterated application of \ref{item:nichols-alternative-e} provides another description of $\J(V)$. In fact the skew derivations $\partial_{f}$ descend to $\toba(V)$ and 
\begin{align*}
\bigcap_{f\in V^*} \ker \partial_{f} &= \ku \text{ in } \toba(V).
\end{align*}

\subsection{Nichols algebras as invariants of Hopf algebras}\label{subsec:nichols-classification-Hopf}
The applications of Nichols algebras to classification problems of Hopf algebras go through the characterization \ref{item:nichols-alternative-dbis}.
As in every classification problem, one starts by considering various invariants, seeking eventually to list all objects in terms of them.
To explain this, let us consider a Hopf algebra $A$ (with bijective antipode). If $D, E$ are subspaces of $A$, then
\begin{align*}
D\wedge E &:= \{x\in A: \Delta(x) \in D \ot A + A \ot E \}.
\end{align*}

The first invariants of the Hopf algebra $A$ are:

\begin{itemize} [leftmargin=*]\renewcommand{\labelitemi}{$\circ$}

\item The coradical $A_0$, which is the sum of all simple subcoalgebras.

\item The coradical filtration $(A_n)_{n\in \N_0}$, where $A_{n+1} = A_n \wedge A_0$.

\item The subalgebra  generated by the coradical, denoted $A_{[0]}$ and called the Hopf coradical. 

\item The standard filtration $(A_{[n]})_{n\in \N_0}$, where $A_{[n+1]} = A_{[n]} \wedge A_{[0]}$.

\item The associated graded Hopf algebra $\gr A = \oplus_{n\ge 0} \gr^n A$, $\gr^0 A = A_{[0]}$, $\gr^{n+1} A = A_{[n+1]} / A_{[n]}$.
\end{itemize}
	
The first two are just  invariants of the underlying coalgebra, while the last three mix algebra and coalgebra information.

Clearly, $A_0$ is a subalgebra iff $A_0 = A_{[0]}$; 
in this case the method outlined below was introduced in \cite{AS1,AS3}, see also \cite{AS4},
the extension being proposed in \cite{AC}.
For simplicity we address the problem of classifying finite-dimensional Hopf algebras, but this could be
adjusted to finite Gelfand-Kirillov dimension. 
The method rests on the consideration of several questions. 
First, one needs to deal with the possibility $A = A_{[0]}$. 
Formally, this means:

\begin{question}\label{que:gen-by-coradical}
	Classify all finite-dimensional Hopf algebras  generated as algebras by their coradicals.
\end{question}

This seems to be out of reach presently; see the discussion in \cite{AC}.
Notice that there are plenty of finite-dimensional Hopf algebras  generated by the coradical;
pick one of them, say $L$.
Recall that $\ydl$ is the category of its Yetter-Drinfeld modules.
Let $A$ be a finite-dimensional Hopf algebra and suppose that $A_{[0]} \simeq L$. Then
$\gr A$ splits as the bosonization of $L$ by the subalgebra of coinvariants $\cR$,  i.e.
\begin{align*}
\gr A &\simeq \cR \# L, 
\end{align*} 
see e.g. \cite{AS3,Ma} for details. Actually this gives two more invariants of $A$: 

\begin{itemize} [leftmargin=*]\renewcommand{\labelitemi}{$\circ$}
	
	\item $\cR = \oplus_{n\ge 0} \cR^n$, a graded connected  Hopf algebra in $\ydl$.
It is called the diagram of $A$. 
	
	\item $\cV  = \cR^1$, an object of $\ydl$ called the infinitesimal braiding of $A$.
\end{itemize}

It is then natural to ask:

\begin{question}\label{que:conn-R}
Classify all graded connected  Hopf algebras $R$ in $\ydl$ such that $\dim R < \infty$.
\end{question}

The subalgebra $\ku \langle \cV \rangle$ of $\cR$ projects onto the Nichols algebra $\toba(\cV)$, 
i.e. it is a pre-Nichols algebra of $\cV$; this is how Nichols algebras enter into the picture. 
Thus the classification of all finite-dimensional Nichols algebras in $\ydl$ is not only part
of Question \ref{que:conn-R} but also a crucial ingredient of its solution.
Even more, in some cases all possible $\cR$'s are Nichols algebras.
We introduce a convenient terminology to describe them.

\begin{definition}
	An object $V\in \ydl$ is \emph{fundamentally finite} if
	\begin{enumerate}
		\item\label{item-ff1} $\dim \toba(V) < \infty$;
		
		\item\label{item-ff2} if $R$ is a  pre-Nichols algebra of $V$ and $\dim R < \infty$, then $R\simeq \toba(V)$;
		
		\item\label{item-ff3} if $R$ is a post-Nichols algebra of $V$ and $\dim R < \infty$, then $R\simeq \toba(V)$. 
	\end{enumerate}
\end{definition}

Notice that there is some redundancy in this Definition: for example, 
if $V$ is of diagonal type such that \eqref{item-ff1} holds, then \eqref{item-ff2}
and \eqref{item-ff3} are equivalent.

Thus, if the infinitesimal braiding $\cV$ of $A$ is fundamentally finite, then $\cR \simeq \toba(\cV)$.
In consequence, 
if every $V \in \ydl$ with $\dim \toba(V) < \infty$ is fundamentally finite, 
then any $R$ as in Question \ref{que:conn-R} is a Nichols algebra. 

Assume that $L$ is a cosemisimple Hopf algebra, i.e. the context where $A_0 = A_{[0]}$ \cite{AS3}.
Then the subalgebra $\ku \langle \cV \rangle$ of the diagram $\cR$ is isomorphic to the Nichols algebra 
$\toba(\cV)$. The question of whether every $V \in \ydl$ with $\dim \toba(V) < \infty$ is fundamentally finite,
when $L = \ku G$ is the group algebra of a finite group, is tantamount to

\begin{conjecture}\cite{AS2}
Every finite-dimensional pointed Hopf algebra is generated by group-like and skew-primitive elements.
\end{conjecture}

For instance, the Conjecture is true for abelian groups \cite{A-presentation}; this translates to the fact
all braided vector spaces $V$ of diagonal type and $\dim \toba(V) < \infty$, are  fundamentally finite.

\medbreak
Finally, here is the last Question to be addressed within the method \cite{AC}.

\begin{question}\label{que:deformations}
Given $L$ and $R$ as in Questions \ref{que:gen-by-coradical} and \ref{que:conn-R}, 
classify all their liftings, i.e. all Hopf algebras $H$
such that $\gr H \cong R\# L$. 
\end{question}

To solve this Question, we need to know not only the classification of 
all finite-dimensional Nichols algebras in $\ydl$, but also a minimal set of relations
of each of them.

\section{Nichols algebras of diagonal type}\label{sec:nichols-diagonal}

In this Section we present the main features of Nichols algebras of diagonal type.
The central examples are the (positive parts of the) quantized enveloping algebras that where intensively studied in the literature, 
see for instance \cite{Lu-jams1990,Lu-dedicata,Lu,Ro-pbw,Ro2,Y-pbw}. We motivate each of the notions by comparison with the quantum case.

\subsection{Braidings of diagonal type}\label{subsec:braid} 
In this Subsection and the next, we set up the notation to be used in the rest of the monograph.

\smallbreak
Let $\theta \in \N$ and $\I =\I_{\theta}$.
Let $\bq = (q_{ij})_{i,j\in \I} \in( \ku^{\times})^{\I\times \I}$ such that 
\begin{align}\label{eq:diag-dif-1}
q_{ii} &\neq 1,& \text{for all } i&\in \I.
\end{align}
Let $\widetilde{q_{ij}} := q_{ij}q_{ji}$. The \emph{generalized
Dynkin diagram} of the matrix $\bq$ is a
graph with $\theta$ vertices, the vertex $i$ labeled with $q_{ii}$, and
an arrow between the vertices $i$ and $j$ only if $\widetilde{q_{ij}}\neq 1$, 
labeled with this scalar $\widetilde{q_{ij}}$.
For instance, given $\zeta\in\G'_{12}$ and $\kappa$ a square root of $\zeta$,
the matrices
$ \begin{pmatrix} \zeta^4 & 1  \\ \zeta^{11} & -1   \end{pmatrix}$,
$\begin{pmatrix} \zeta^4 & \kappa^{11}  \\ \kappa^{11} & -1 \end{pmatrix}$
have the diagram:
\begin{align}\label{eq:diagram-used}
\xymatrix{\circ^{\zeta^4} \ar@{-}[r]^{\zeta^{11}} &
\circ^{-1} }.\end{align}

Let $V$ be a vector space with a basis $(x_i)_{i\in \I}$; define $c^{\bq} = c: V \ot V \to V \ot V$ by
$c(x_i\ot x_j)=q_{ij}\, x_j\ot x_i$,  $i,j\in \I$. Then $c$ is a solution of the braid equation \eqref{eq:braid}.
The pair $(V, c)$ is called a \emph{braided vector space of diagonal type}.

Two braided vector spaces of diagonal type with the same generalized Dynkin diagram are called \emph{twist equivalent};
then the corresponding Nichols algebras are isomorphic as graded vector spaces \cite[Proposition 3.9]{AS3}.
If they correspond to matrices $\bq = (q_{ij})_{i,j\in \I}$ and $\bp = (p_{ij})_{i,j\in \I}$, 
then twist equivalence means that
\begin{align*}
q_{ij}q_{ji} &= p_{ij}p_{ji}&  q_{ii} &= p_{ii}& &\text{for all } i \neq j \in \I.
\end{align*}
For example, $\bq$ and $\bq^t$ are twist equivalent.

Now let $\Gamma$ be an abelian group, $(g_i)_{i\in \I}$ a family in $\Gamma$ and 
$(\chi_i)_{i\in \I}$ a family in $\VGamma$. Let $V$ be a vector space with a basis $(x_i)_{i\in \I}$;
then $V \in \ydg$ by imposing $x_i \in V_{g_i}^{\chi_i}$, $i\in \I$. 
The corresponding braided vector space $(V,c)$ with $c$ given by 
\eqref{eq:braid-yd-gps} is of diagonal type; indeed 
\begin{align*}
c(x_i \ot x_j) &= \chi_j(g_i) x_j \ot x_i, & i,j &\in \I.
\end{align*}

\subsection{Braided commutators}\label{subsec:nichols-braided-commutators}
Let $(V,c)$ be a braided vector space of diagonal type attached to a matrix $\bq$ as in \S \ref{subsec:braid}. 
As in \cite{AAR}, we switch to the notation $\toba_{\bq} = \toba(V)$, $\J_{\bq} = \J(V)$
and so on.

\medbreak
Let $(\alpha_i)_{i\in \I}$ be the canonical basis of $\Z^{\I}$. It is clear that
$T(V)$ admits a unique $\Z^{\I}$-graduation such that $\deg x_i=\alpha_i$ (in what follows, $\deg$ is the $\Z^{\I}$-degree).
Since $c$ is of diagonal type,  $\J_{\bq}$ is a $\Z^{\I}$-homogeneous ideal 
and $\toba_{\bq}$ is $\Z^{\I}$-graded \cite[Proposition 2.10]{AS3}, \cite[Proposition 1.2.3]{Lu}.

\medbreak Next, the matrix $\bq$ defines a $\Z$-bilinear form $\bq:\Z^{\I}\times\Z^{\I}\to\ku^\times$ by $\bq(\alpha_j,\alpha_k)=q_{jk}$ for all $j,k\in\I$.
Set $\bq_{\alpha \beta}  = \bq(\alpha,\beta)$, $\alpha,\beta  \in \Z^{\I}$; also, $\bq_{i \beta}  = \bq_{\alpha_i \beta}$.

\medbreak
Let $R$ be an algebra in $\ydg$. The braided commutator is the linear map $[\, , \,]_c :R\ot R \to R$  given by
\begin{align}
\label{eq:q-corc} [x,y]_c &= \mu \circ \left( \id - c \right) \left( x \ot y \right),& x,y &\in R.
\end{align}
In the setting of braidings of diagonal type,  braided commutators are also called  $q$-commutators. Assume that $R = T(V)$ or any quotient thereof by a $\Z^{\I}$-homogeneous 
ideal in $\ydg$, so that $R = \oplus_{\alpha\in \Z^{\I}} R_{\alpha}$. Let $u,v \in R$ be $\Z^{\I}$-homogeneous with $\deg u = \alpha$, $\deg v = \beta$. Then 
\begin{equation}\label{eq:braiding tipo diagonal}
c(u \ot v)= \bq_{\alpha \beta}\, v \ot u.
\end{equation}
Thus, if $y\in R_\alpha$, then 
\begin{align}\label{eq:q-adj}
[x_i,y]_c &= x_iy - \bq_{i \alpha}\, y x_i,& &i\in \I.
\end{align} 
Notice that $\ad_c x (y) = [x,y]_c$, in case $x\in V$ and $y$ in $R$. 
The braided commutator is a braided derivation in each variable and satisfies a braided Jacobi  identity, i.e.
\begin{align}
\label{eq:derivacion}
\left[ u,v w \right]_c &= \left[ u,v \right]_c w + \bq_{\alpha \beta}\, v  \left[ u,w \right]_c,
\\ \label{eq:derivacion 2} \left[ u  v, w \right]_c &= \bq_{\beta \gamma}\, \left[ u,w \right]_c  v + u  \left[ v,w \right]_c,
\\
\label{eq:identidad jacobi}
\left[\left[ u, v \right]_c, w \right]_c &= \left[u, \left[ v, w
\right]_c \right]_c
- \bq_{\alpha \beta}\, v  \left[ u, w \right]_c + \bq_{\beta \gamma}\, \left[ u,
w \right]_c  v,
\end{align}
for $u,v,w$ homogeneous of degrees $\alpha, \beta, \gamma \in \N^{\theta}$, respectively. 

For brevity, we set 
\begin{align}\label{eq:xij}
x_{ij} &= \ad_c x_i (x_j),& i&\neq j \in \I;
\end{align}
more generally, the iterated braided commutators are
\begin{align}\label{eq:iterated}
x_{i_1i_2\cdots i_k}& :=(\ad_c x_{i_1})\cdots(\ad_c x_{i_{k-1}})\, (x_{i_k}),& &i_1, i_2, \cdots, i_k\in\I.
\end{align}
In particular, we will use repeatedly the following further abbreviation:
\begin{align}\label{eq:roots-Atheta}
x_{(k \, l)} &:= x_{k\,(k+1)\, (k+2) \dots l},& &k < l.
\end{align}
Beware of confusing \eqref{eq:roots-Atheta} with \eqref{eq:xij}.
Also, we define recursively $x_{(m+1)\alpha_i+m\alpha_j}$, $m\in \N$, by
\begin{align}\label{eq:not-reducida-raiz}
\begin{aligned}
x_{2\alpha_i+\alpha_j}  &= (\ad_c x_i)^2 x_j = x_{iij} , \\
x_{(m+2)\alpha_i+(m+1)\alpha_j} &= [ x_{(m+1)\alpha_i+m\alpha_j}, (\ad_c x_i) x_j ]_c.
\end{aligned}
\end{align}
These commutators are instrumental to reorder products. For example, we can prove recursively on $m$  that,  for all $m,n\in\N$,
\begin{align}\label{eq:adx1x2x2}
x_1^m (\ad_cx_1)^n x_2 &= \sum_{j=0}^{m} \binom{m}{j}_{q_{11}} q_{11}^{n(m-j)}q_{12}^{m-j} (\ad_cx_1)^{n+j} x_2x_1^{m-j}.
\end{align}

\subsection{ PBW-basis and Lyndon words}\label{subsec:pbw} 

An unavoidable first step in the study of quantum groups $U_q(\g)$ is the description of the PBW-basis (alluding to the Poincar\'e-Birkhoff-Witt theorem), obtained for type A in \cite{Burr-pbw,Ro-pbw,Y-pbw} and for general $\g$ in \cite{Lu-jams1990, Lu-dedicata}. 
Here $\g$ is simple finite-dimensional, for other Kac-Moody algebras see Remark \ref{rem:quantum-affine}.
Indeed, it is enough to define the PBW-basis for the positive part $U_q^+(\g)$, where it
is an ordered basis of monomials in some elements, called root vectors. 
These root vectors are  defined using the braid automorphisms defined by Lusztig; actually they can be expressed
 as iterated braided-commutators \eqref{eq:iterated}, and there as many as $\varDelta^A_+$, where $A$ is the Cartan matrix of $\g$. 
 However they are not uniquely defined (even not up to a scalar); 
 it is necessary to fix a reduced decomposition of the longest 
 element of the Weyl group to have the precise order in which the iterated commutators
produce the root vectors. For general Nichols algebras of diagonal type, 
there is a procedure that replaces the sketched method, which consists 
in the use of Lyndon words, as pioneered in \cite{Kh}, see also 
the detailed monograph \cite{Kh-libro} (this approach also appeared later in \cite{R2}). 
In this Subsection we give a quick overview of this procedure, fundamental for the description of the Nichols algebras of diagonal type.

\medbreak
We shall adopt the following definition of PBW-basis. Let $A$ be an algebra. We consider

\begin{itemize} [leftmargin=*]\renewcommand{\labelitemi}{$\circ$}

\item a subset $\emptyset \neq P \subset A$;

\item a  subset $\emptyset \neq S \subset A$ provided with a total order $<$ (the \emph{PBW-generators});

\item a function $h: S \mapsto \N \cup \{ \infty \}$ (the \emph{height}).

\end{itemize}

Let $B = B(P,S,<,h)$ be the set
\begin{align*}
B &= \big\{p\,s_1^{e_1}\dots s_t^{e_t}:& t &\in \N_0,\ s_i \in S, \ p \in P,& s_1&>\dots >s_t, 
&  0&<e_i<h(s_i) \big\}.
\end{align*}
If $B$ is a $\ku$-basis of $A$, then we say that it is a \emph{PBW-basis}.
Our goal is to describe PBW-bases of some graded Hopf algebras $R$ in $\ydg$ (with $P = \{1\}$), following \cite{Kh}; by bosonization, one gets
PBW-bases of the graded Hopf algebras $R\# \ku \Gamma$ (with $P=\Gamma$ this time).  

\begin{remark}\label{rem:pbw-def}
In the definition of PBW-basis, one would expect that 
\begin{align}\label{eq:pbw-right-defi}
h(s) &= \min\{t \in \N:  s^t =0\},& s&\in S;
\end{align}
however this requirement is not flexible enough. For instance, consider the algebra $\B = \ku \langle x_1, x_2| x_1^N - x_2^M, x_1x_2 - q x_2x_1 \rangle$, 
where $N, M \ge 2$ and $q \in \ku^{\times}$. Then $\B$ has a PBW-basis with $P = \{1\}$, $S = \{x_1, x_2\}$, $x_1 < x_2$, $h(x_1) = N$
and $h(x_2) = \infty$. Kharchenko's theory of hyperletters based on Lyndon words does not apply with the stronger \eqref{eq:pbw-right-defi}. Indeed
let $\bq =  \begin{pmatrix}
w_N & q \\ q^{-1} & w_M
\end{pmatrix}$ where $w_N \in \G'_N$, $w_M \in \G'_M$; then $\B$ is a quotient of $T(V)$ by a Hopf ideal, that is homogeneous if $M=N$, so that Theorem \ref{thm:PBW-basis  Kharchenko} provides 
the PBW-basis described, but without  \eqref{eq:pbw-right-defi}. 
\end{remark}

\medbreak \subsubsection{}
Let $X$ be a set with $\theta$ elements and  fix a numeration $x_1,\dots, x_{\theta}$ of $X$. Let $\xx$ be the  corresponding vocabulary, 
i.e. the set of words with  letters in $X$, endowed with the lexicographic order induced by the numeration. Let $\ell: \xx \to \N_0$ be the length.

\begin{definition}
An element $u \in \xx-\left\{ 1 \right\}$ is a \emph{Lyndon word} if $u$ is smaller than any of its proper ends; i. e.,
if $u=vw$, $v,w \in\xx - \left\{ 1 \right\}$, then  $u<w$. The set of all Lyndon words is denoted by $L$.
\end{definition}

Here are some  basic properties of the Lyndon words.

\begin{enumerate}[leftmargin=*,label=\rm{(\alph*)}]
\item Let $u \in \xx-X$. Then $u$ is Lyndon if and only if for each decomposition $u=u_1 u_2$, where $u_1,u_2 \in \xx - \left\{ 1 \right\}$,
one has $u_1u_2=u < u_2u_1$.

\medbreak\item Every Lyndon word starts by its lowest letter.

\medbreak\item (Lyndon). Every word $u \in \xx$ admits a unique
decomposition as a non-increasing product of Lyndon words (the \emph{Lyndon decomposition}):
\begin{equation}\label{eq:descly}
u=l_1l_2\dots  l_r, \qquad l_i \in L, l_r \leq \dots \leq l_1;
\end{equation}
the words $l_i \in L$  appearing in \eqref{eq:descly} are the \emph{Lyndon letters} of $u$.

\medbreak\item The lexicographic order of $\xx$ turns out to coincide with the lexicographic order in the  Lyndon letters--i.e. with respect to \eqref{eq:descly}.

\medbreak\item (Shirshov). Let $u \in\xx-X$. 
Then $u \in L$ if and only if there exist $u_1,u_2 \in L$ such that $u_1<u_2$ and $u=u_1u_2$.

\end{enumerate}

\begin{definition}
The \emph{Shirshov decomposition} of $u\in L-X$ is the decomposition $u=u_1u_2$, with $u_1,u_2 \in L$, such that $u_2$ is the lowest proper end among the ends of such decompositions of $u$.
\end{definition}

Let us identify $X$ with the basis of $V$ defining the braiding and consequently $\xx$ with a basis of $T(V)$. Using the braided commutator
\eqref{eq:q-corc} and the Shirshov decomposition, we define 
$\left[ \ \right]_c : L \to \End T(V)$, by
$$
\left[ u \right]_c := \begin{cases} u,& \text{if } u = 1
\text{ or }u \in X;\\
[\left[ v \right]_c, \left[ w \right]_c]_c,  & \text{if } u \in
L, \, \ell(u)>1 \text{ and }u=vw \text{ is the Shirshov }\\
&  \hspace{140pt} \text{  decomposition of } u.
\end{cases}
$$

The element $\left[u\right]_c$ is called the \emph{hyperletter} of $u \in L$.
This leads to define a \emph{hyperword} as a word in hyperletters.
We need a more precise notion.

\begin{definition} A 
\emph{ monotone hyperword} is a hyperword $\left[u_1\right]_c^{k_1}\dots
\left[u_m\right]_c^{k_m}$, where $u_1>\dots >u_m$ are Lyndon words.
\end{definition}

\begin{remark}\label{corchete} \cite[Lemma 2.3]{Kh-libro}
Let $u \in L$, $n = \ell(u)$. Then $\left[ u \right]_c$ is a linear combination 
\begin{align*}
\left[ u \right]_c =  u+ \sum_{u < z \in \xx: \, \deg z =\deg u} p_z z,
\end{align*}
where $p_z \in \mathbb{Z} \left[q_{ij}: i, j \in \I\right]$.
\end{remark}

\medbreak  The order of the Lyndon words induces 
an order on the hyperletters. Consequently we consider the lexicographic order in the
hyperwords. Given two monotone hyperwords   $W,V$, it can be shown that
\begin{align*}
W=\left[w_1\right]_c\dots \left[w_m\right]_c &> V=\left[v_1\right]_c\dots  \left[v_t\right]_c,
& w_1 &\geq \dots  \geq w_m,& 
v_1 &\geq \dots  \geq v_t,
\end{align*}
if and only if $w=w_1\dots w_{m} > v=v_1\dots v_t$.

\medbreak
\subsubsection{}
Let $I$ be a homogeneous proper 2-sided ideal of $T(V)$ such that $I \cap V =0$, $R=T(V)/I$, $\pi: T(V) \rightarrow R$ the  canonical projection. Set
$$G_I:= \Big\{ u \in \xx: u \notin \sum_{u < z \in \xx}\ku z + I  \Big\}.$$
Let $u, v, w \in \xx$ such that $u=vw$. If $u \in G_I$, then $v,w \in G_I$. Hence every $u \in G_I$ factorizes uniquely as a non-increasing  product of Lyndon words in $G_I$. Then the set $\pi(G_I)$ is a basis of $R$ \cite{Kh,R2}.

\medbreak Assume next that $I$ is a homogeneous Hopf ideal and set $P = \left\{1\right\}$, $S_I:= G_I \cap L$ and $h_I: S_I \rightarrow \N_{\ge 2}\cup \left\{ \infty \right\}$  given by
\begin{align*}
h_I(u):= \min \Big\{ t \in \N : u^t  \in \sum_{u < z \in \xx}\ku z + I  \Big\}.
\end{align*}
Let   $B_I:= B\left(P , \pi (\left[ S_I \right]_c), <, h_I \right)$, see the beginning of this Subsection.
The next fundamental result is due to Kharchenko.

\begin{theorem}\label{thm:PBW-basis  Kharchenko} \cite{Kh} If  $I$ is a homogeneous Hopf ideal, then $B_I$ is a PBW-basis  of $T(V)/I$. 
\end{theorem}

That is, we get a PBW-basis whose PBW-generators are the images of the hyperletters corresponding to Lyndon words that are in $G_I$.

\smallbreak
The finiteness of the height in the PBW-basis of Theorem \ref{thm:PBW-basis  Kharchenko} is controlled by the matrix $\bq$:

\begin{rem}\label{cor:segundo} \cite[Theorem 2.3]{Kh-libro}
Let $v \in S_I$ such that $h_I(v)< \infty$, $\deg v = \alpha$. Then $\bq_{\alpha \alpha} \in \G_{\infty}$ and 
$h_I(v)= \ord q_{\alpha \alpha}$.
\end{rem}

\medbreak
By Remark  \ref{cor:segundo}, it seems convenient to consider PBW-basis $B(P, S, <, h)$ of $\toba_\bq$ with the following constraints: 

\begin{definition} A PBW-basis is \emph{good} if 
$P = \{1\}$, the elements of $S$ are $\Z^{\I}$-homogeneous and $h(v)= \ord q_{\alpha \alpha}$  for $v \in S$, $\deg v = \alpha$. 
\end{definition}

\begin{exa} \label{exa:roots-fniteheight}
	If $\dim \toba_{\bq} < \infty$, then the PBW-basis $B_{\J_{\bq}}$ of $\toba_{\bq}$ 
	as in Theorem \ref{thm:PBW-basis  Kharchenko} is good. Indeed,  
	$h_{\J_{\bq}}(v)< \infty$ for all $v \in S_{\J_{\bq}}$.
	However, there are many examples of $v \in S_{\J_{\bq}}$, $\deg v = \alpha$ with $\bq_{\alpha \alpha} \in \G_{\infty}$, 
	with  $h_{\J_{\bq}}(v) = \infty$.

	For instance, take $\theta=2$. Set for simplicity $y_n=(\ad_cx_1)^n x_2$, $n\in\N$.
	Assume that  there exists $k \in \N$ such that $\alpha= k\alpha_1+\alpha_2$ is a root; particularly,
	$y_k \neq 0$. Notice that $\partial_2(y_k)=b_k \, x_1^k$, where  $b_k=\prod_{j=0}^{k-1} (1 -q_{11}^k \widetilde{q_{12}})$, 
	hence $(k)_{q_{11}}^! b_k \neq 0$. 
	We claim that
	\begin{align*}
	(y_k)^2 &=0 \text{ in } \toba_{\bq} & &\iff &y_{k + 1} &=0, & \bq_{\alpha \alpha}&=-1.
	\end{align*}
	Indeed, $(y_k)^2 =0$ iff $\partial_1(y_k^2) \overset{\star}{=} 0 = \partial_2(y_k^2) $; but $\star$ holds always. 
	Then
	\begin{align*}
	\partial_2(y_k^2) & = b_k (y_kx_1^k + q_{21}^k q_{22} \, x_1^k y_k) 
	\\
	& \overset{\eqref{eq:adx1x2x2}}{=}  b_k \left((1+\bq_{\alpha \alpha}) \ y_kx_1^k + \sum_{j=1}^{k} \binom{k}{j}_{q_{11}} q_{11}^{k(k-j)}q_{12}^{k-j} q_{21}^k q_{22} \ y_{k+j}x_1^{k-j}\right)
	\end{align*}
	and this is 0 iff $\bq_{\alpha \alpha} =-1$ and $y_{k + 1} = 0$ (the last implies $y_{k + j} = 0$ for all $j\in \N$).
	Thus, if $\bq_{\alpha \alpha}=-1$ but $y_{k + 1} \neq 0$, then $y_k$ has infinite height by Remark \ref{cor:segundo}.
	Concretely, the matrix $\bq$ with Dynkin diagram $\xymatrix{\circ^{q} \ar@{-}[r]^{q^{-2}} &
		\circ^{-q}}$, where $q\neq -1$, gives the desired example by $k=1$.
\end{exa}

Let us illustrate the strength of Theorem \ref{thm:PBW-basis  Kharchenko} in the following example.

\begin{exa}
Let $q \in \G'_N$, $N >1$, $q_{12} \in \ku^{\times}$ and $\bq = \begin{pmatrix} q & q_{12} \\ q^{-1}q_{12}^{-1} & q\end{pmatrix}$, so that $\theta = 2$.
Then a PBW-basis of $\toba_{\bq}$ is 
\begin{align*}
B = \Big\{x_2^{e_2} x_{12}^{e_{12}} x_1^{e_1}: 0 \le e_j < N, j = 2,12,1 \Big\}.
\end{align*}
Here we use the notation \eqref{eq:iterated}.
\end{exa}

Let us outline the proof of this statement.

\begin{itemize}[leftmargin=*]\renewcommand{\labelitemi}{$\circ$}
\item The quantum Serre relations $x_{112} =0$, $x_{221} =0$ hold in $\toba_{\bq}$. This can be checked using derivations, see Proposition \ref{prop:nichols-alternative} \ref{item:nichols-alternative-e}. Alternatively, apply \cite[Appendix]{AS2}.

\medbreak
\item The  power relations $x_1^N =0$, $x_2^N =0$, $x_{12}^N =0$ hold in $\toba_{\bq}$. The first two follow directly from the quantum bilinear formula, and the third from the quantum Serre relations using derivations. By Remark \ref{cor:segundo}, we conclude that $h_{\J_{\bq}}(x_2) = h_{\J_{\bq}}(x_{12}) =h_{\J_{\bq}}(x_2)=N$.

\medbreak
\item The set $B$ generates the algebra $\toba_{\bq}$, or more generally any $R$ where the quantum Serre and the power relations hold. This follows because the subspace spanned by $B$ is a left ideal.

\medbreak
\item The element $x_1x_2 \in S_{\J_{\bq}} =  G_{\J_{\bq}} \cap L$. Now $x_{12} = [x_1x_2]_c$; hence $B$, being a subset of $B_{\J_{\bq}}$, is linearly independent by Theorem \ref{thm:PBW-basis  Kharchenko}. Alternatively, one can check the linear independence of $B$ by successive applications of derivations. 
Together with the previous claim, this implies the statement.
\end{itemize}

\begin{rem}
The results of the theory sketched in this Subsection depend heavily on the numeration of the starting set $X$, that is on a fixed total order of $X$.
Changing the numeration gives rise to different Lyndon words and so on. The outputs are equivalent but not equal.
\end{rem}

\subsection{The roots of a Nichols algebra}\label{subsec:nichols-RS} 

Let $(V,c)$ be a braided vector space of diagonal type attached to a matrix $\bq$ as in \S \ref{subsec:braid}.
Now that we have the PBW-basis $B_{\bq}$ of $\toba_\bq$ given by Theorem \ref{thm:PBW-basis  Kharchenko}, whose PBW-generators  (i.e. the elements of $S_{\J_{\bq}}$) are $\Z^{\I}$-homogeneous,
we reverse the reasoning outlined at the beginning of the previous Subsection and define following  \cite{H-Weyl gpd}
the positive roots of $\toba_\bq$ as the degrees of the PBW-generators, and the roots as plus or minus the positive ones; i.e.
\begin{align}\label{eq:root-def}
\varDelta^{\bq}_+ &= (\deg u)_{u\in S_{\J_{\bq}}}, & \varDelta^{\bq} &= \varDelta^{\bq}_+ \cup -\varDelta^{\bq}_+.
\end{align}

In principle, there might be several $u\in S_{\J_{\bq}}$ with the same deg. It is natural to define
the multiplicity of $\beta \in \Z^{\I}$ as 
\begin{align*}
\mult \beta = \mult_{\bq} \beta = \vert \{u\in S_{\J_{\bq}}: \deg u = \beta\} \vert.
\end{align*}

\begin{definition}\label{def:arithmetic} \cite{H-classif RS}
The matrix $\bq$ (or the braided vector space $(V,c)$, or the Nichols algebra $\toba_{\bq}$) is  \emph{arithmetic}
if $|\varDelta^\bq_+|< \infty$.
\end{definition}

For instance, if $\dim \toba_{\bq} < \infty$, then $(V,c)$ is arithmetic.
If $\toba_{\bq}$ is  arithmetic, then all roots are real, i.e. conjugated to simple roots by the Weyl groupoid; see the discussion
in \ref{subsection:weylgroupoid-def} below. Using this, one can prove:

\begin{rem} Assume that $\toba_{\bq}$ is  arithmetic. 
\begin{itemize} [leftmargin=*]\renewcommand{\labelitemi}{$\circ$}
\item The heights of the generators in $S_{\J_{\bq}}$ satisfy \eqref{eq:pbw-right-defi}. 

\item $\varDelta^{\bq}_+$ does not depend on the PBW-basis: if $B = B(P, S, <, h)$ is any PBW-basis with $P= \{1\}$ satisfying \eqref{eq:pbw-right-defi},
then $\varDelta^{\bq}_+ = (\deg u)_{u\in S}$. See \cite{AA}.
\end{itemize}
\end{rem}

We next explain the recursive procedure to describe the hyperletters.

\begin{rem}\label{rem:lyndon-word}
Assume that $\bq$ is arithmetic. Then every root has multiplicity one \cite{CH-at most 3}, 
so we can label the Lyndon words with $\varDelta^\bq_+$; let $l_{\beta}$ be the Lyndon word of degree $\beta \in \varDelta^\bq_+$.
The Lyndon words $l_{\beta}$'s are computed recursively \cite[Corollary 3.17]{A-jems}:
$l_{\alpha_i}=x_i$, and for $\beta \neq \alpha_i$,
\begin{equation}\label{eq:caracterizacion-palabras-Lyndon}
l_{\beta}= \max \{ l_{\delta_1}l_{\delta_2}: \ \ \delta_1, \delta_2 \in \varDelta^\bq_+, \ \delta_1+\delta_2=\beta, \ l_{\delta_1}<l_{\delta_2} \}.
\end{equation}
Let $x_{\beta}$ be the hyperletter corresponding to  the Lyndon word $l_{\beta}$, $\beta \in \varDelta^\bq_+$.
 Then
\begin{align}\label{eq:x-raiz}
\begin{aligned}
x_{\alpha_i} &= x_i,& i&\in \I, \\
x_{\beta}    &= [x_{\delta_1}, x_{\delta_2}]_c, & &\text{if }
l_\beta = l_{\delta_1}l_{\delta_2} \text{ is the Shirshov decomposition}.
\end{aligned}
\end{align}
This gives explicit formulas for the PBW-generators of the PBW-basis $B_{\bq}$ given by Theorem \ref{thm:PBW-basis  Kharchenko}.
\end{rem}

\subsubsection*{Braidings of Cartan type}
To explain the importance of the roots as defined in \eqref{eq:root-def}, 
we discuss the class of braidings of Cartan type, closely related with quantum groups.

\begin{definition}\label{def:cartantype} \cite{AS2}
The matrix $\bq$ (or $V$, or   $\toba_{\bq}$) is of \emph{Cartan type}
if there exists a GCM $A = (a_{ij})$ such that
\begin{align}\label{eq:cartan-type}
q_{ij}q_{ji} &= q_{ii}^{a_{ij}}, & \forall i,j \in \I.
\end{align}
Assume that this is the case. We fix a choice of $A$ by
\begin{align}\label{eq:cartan-type-fix}
-N_i &< a_{ij} \le 0, & \forall j\neq i \in \I & \text{ when } N_i := \ord q_{ii} \in (1, \infty) .
\end{align}
\end{definition}

This was the first class of Nichols algebras to be studied in depth.

\begin{theorem}\label{th:cartan-finite} \cite{H-Weyl gpd}, see also \cite{AS2}.
Let $\bq$ be of Cartan type with GCM $A$ indecomposable and normalized by \eqref{eq:cartan-type-fix}. Then the following are equivalent:
\begin{enumerate}[leftmargin=*]
\item\label{it:cartan-arithmetic1} The Nichols algebra $\toba_{\bq}$ is arithmetic.

\item\label{it:cartan-arithmetic2} The GCM $A$ is of finite type.
\end{enumerate}

Consequently, the following are equivalent:
\begin{enumerate}[leftmargin=*]
\item\label{it:cartan-finite1} The Nichols algebra $\toba_{\bq}$ has finite dimension.

\item\label{it:cartan-finite2} \begin{enumerate}
\item\label{it:cartan-finite2a} The GCM $A$ is of finite type.
\item\label{it:cartan-finite2b} $N_i \in (1, \infty)$ for all $i\in \I$.
\end{enumerate}
\end{enumerate}
\end{theorem}

\begin{remark}
If $\bq$ is of Cartan type as in Theorem \ref{th:cartan-finite}, \eqref{it:cartan-finite2a} holds but $N_i = \infty$ (for one or equivalently for any $i$), then 
$\GK \toba_\bq = \vert \varDelta^A_+ \vert$. 
\end{remark}

\begin{remark}\label{rem:quantum-affine}
	Let $\bq$ be as in Theorem \ref{th:cartan-finite}. Then $\varDelta^{\bq}_+ \supseteq \varDelta^{A, \text{re}}_+$.
	\begin{enumerate}[leftmargin=*, label=\rm{(\alph*)}]
		\item\label{it:cartan-finite} \cite{H-Weyl gpd,AS2} If $A$ is of finite type, then  $\varDelta^{\bq}_+ = \varDelta^{A}_+$.
		
		\medbreak
		\item\label{it:cartan-affine} Assume that $A$ is of affine type. 
		If $N_i = \infty$ for some (or  all) $i$, then  $\varDelta^{\bq}_+ \overset{\star}{=} \varDelta^{A}_+$ \cite{beck,Da}. 
		When $N_i < \infty$ for all $i$,   the last equality does not hold. 
		
		\medbreak
		\item\label{it:cartan-affine-polynomial-growth}
		In any case, $\varDelta^{\bq}_+ \subseteq \varDelta^{A}_+$ 
		and thus the function 
		\begin{align*}
		n \longmapsto \wp_n &:= \vert\{\beta \in \varDelta^{\bq}_+: \vert\beta\vert = n \} \vert,& n &\in \N,
		\end{align*}
		has polynomial growth. Here $\vert\beta\vert = \sum_i n_i$, if $\beta = \sum_i n_i \alpha_{i}$.

		\medbreak
		\item\label{it:cartan-indefinite} If $A$ is of indefinite type, then we do not know if 
		$(\wp_n)_{n \in \N}$ has polynomial growth.
	\end{enumerate}
	
We discuss an example for the last claim in \ref{it:cartan-affine}.
Let $\zeta \in \G'_N$,  $N>2$, and $\bq$ with Dynkin diagram $\xymatrix{\circ^{\zeta} \ar@{-}[r]^{\zeta^{-2}} &
\circ^{\zeta}}$; this is of affine Cartan 	type $A^{(1)}_1$, see \eqref{eq:A1-(1)}.
The height of the (imaginary) root $\alpha = \alpha_1+\alpha_2$ is infinite since
$\bq_{\alpha \alpha}=1$. 
Now the only possible root vector of degree $2 \alpha$ is $y = [x_1, [x_{12}, x_2]_c]_c$. 
Since 
\begin{align*}
(\widetilde{q_{12}}-1)(q_{11}+1)(q_{22}+1)(q_{11}\widetilde{q_{12}}-1)(\widetilde{q_{12}}q_{22}-1) \neq 0,
\end{align*}
we conclude from \cite{AAH2} that $y$ is  a root vector if and only if
$q_{11}(\widetilde{q_{12}})^2q_{22} \neq -1$,
iff $N \neq 4$. So if $\zeta^2=-1$, then  $2 \alpha \notin \varDelta^{\bq}_+$.
\end{remark}

Before explaining how to describe the positive roots for an arbitrary $\bq$ we need the Drinfeld double of the bosonization of $\toba_{\bq}$.

\subsection{The double of a Nichols algebra}\label{subsec:double-Nichols}

The natural construction of the Drinfeld double of the bosonization of a Nichols algebra by an 
appropriate Hopf algebra was considered by many authors. 
For a smooth exposition, we start by a general construction as in \cite{ARS}.
Let $\bq = (q_{ij})_{i, j \in \I}$ be a matrix of elements in $\ku^{\times}$ such that $q_{ii} \neq 1$ for all $i\in \I$.
Let $G$ be an abelian group. 
A \emph{ reduced YD-datum}   is 
a collection $\D_{red} = (L_i, K_i, \vartheta_i)_{i \in \I}$
where $K_i$, $L_i \in G$, $\vartheta_i \in \widehat{G}$, $i \in \I$, such that
\begin{align}
q_{ij} = \vartheta_j(K_i)&=\vartheta_i(L_j) & &\text{ for all }  i,j  \in \I,\label{eq:reduced1}\\
K_iL_i &\neq 1 &  &\text{ for all } i \in \I.\label{eq:reduced2}
\end{align}

Let us fix a reduced YD-datum as above and define
\begin{align*}
V &= \oplus_{i \in \I}\ku x_i \in {\ydgr}, &  &\text{ with basis }x_i \in V_{K_i}^{\vartheta_i}, i \in \I,\\
W &= \oplus_{i \in \I} \ku y_i \in {\ydgr}, &  & \text{ with basis }y_i \in W_{L_i}^{\vartheta_i^{-1}}, i \in \I.
\end{align*}

Let $\Uc(\D_{red})$ be the quotient of  $T(V \oplus W) \# \ku G$ by the ideal
generated by the relations of the Nichols algebras $\J(V)$ and $\J(W)$, together with
\begin{align*}
x_iy_{j} - \vartheta_j^{-1}(K_i) y_{j}x_i &- \delta_{ij} (K_iL_i - 1),&  i,j  &\in \I.
\end{align*}
Thus $\Uc(\D_{red})$ is a Hopf algebra quotient of $T(V \oplus W) \# \ku G$, with coproduct determined by $\Delta (g)=g\ot g$, $g\in G$,
\begin{align*}
\Delta(x_i)&=x_i\ot 1+K_i\ot x_i,& \Delta(y_i)&=y_i\ot 1 +L_i\ot y_i,& i & \in\I.
\end{align*}

Our exclusive  interest is in the examples of the following shape.

\begin{example}\label{exa:reducedYD-datum} Let $\Gamma$ be an abelian group.
A \emph{realization} of $\bq$ over $\Gamma$ is a pair $(\mathbf{g}, \mathbf{\chi})$ of families
$\mathbf{g} = (g_i)_{i \in \I}$ in  $\Gamma$, $\mathbf{\chi} = (\chi_i)_{i \in \I}$ in $\widehat{\Gamma}$, 
such that $q_{ij} =\chi_i(g_j)$ for all $i,j  \in \I$. 
Any realization gives rise to a reduced datum $\D_{red}$ over $G = \Gamma \times \widetilde{\Gamma}$, 
where $\widetilde{\Gamma} = \langle \chi_i: i \in \I \rangle 
\hookrightarrow \widehat{\Gamma}$, by
\begin{align*}
K_i &=g_i ,& L_i &=\chi_i, & \vartheta_i &= (\chi_i, g_i), & i&\in \I.
\end{align*}

\smallbreak
To stress the analogy with quantum groups, we set  as in \cite{ARS,H-lusztig iso,HY}, 
\begin{align}
E_i&=x_i,& F_i&=y_i\chi_i^{-1} & \text{in } \Uc(\D_{red}) \text{ for } i=1,2.
\end{align}

Let $\Uc(\D_{red})^{\mp} $be the subalgebra of $\Uc(\D_{red})$ generated by the $F_i$'s,  respectively the  $E_i$'s.
\end{example}

\begin{example}\label{exa:reducedYD-datum-free} We shall consider the following particular instance of Example \ref{exa:reducedYD-datum},
where $(\alpha_i)_{i\in \I}$ is the canonical basis of $\Z^{\I}$:
\begin{align*}
\Gamma &= \Z^{\I},& g_i &= \alpha_i, & \chi_i: \Z^{\I} &\to \ku^{\times}, \ \chi_j(\alpha_i) = q_{ij}, & i&\in \I.
\end{align*}
In this context,   we set $\Uc_{\bq} := \Uc(\D_{red})$.
\end{example}

In the following statement, \emph{quasi-triangular} has to be understood in a formal sense, as the R-matrix would belong to an appropriate completion.

\begin{theorem} Let $\D_{red}$ be a reduced datum as in Example \ref{exa:reducedYD-datum}. Then $\Uc(\D_{red})$ is a quasi-triangular Hopf algebra.
\end{theorem}

\begin{proof} (Sketch). We argue as in \cite[Theorem 3.7]{ARS}. Let  $H=\toba(V)\# \ku\Gamma$ and $U=\toba(W)\# \ku\widetilde{\Gamma}$.
There is a non-degenerate skew-Hopf bilinear form  $(\, | \,):H\ot U\to \ku$ given by
\begin{align*}
(x_i|y_j)&=\delta_{ij},& (v_i|\vartheta)&=0,& (g|y_j)&=0,& (g|\vartheta)&=\vartheta(g),& g\in\Gamma, \vartheta\in\VGamma, i,j\in\I.
\end{align*}
Then $\Uc(\D_{red})\simeq (U\ot H)_{\sigma}$, where $\sigma:(U\ot H)\ot (U\ot H)\to \ku$ 
is the  2-cocycle given by $\sigma(f\ot h,f'\ot h')= \varepsilon(f)(h|f')\varepsilon(h')$ for $f,f'\in H^*$ and $h,h'\in H$. 
Therefore $\Uc(\D_{red})$ is the Drinfeld double of $H$.
\end{proof}

\medbreak
\subsection{The Weyl groupoid of a Nichols algebra}\label{subsec:Weyl-gpd} The proofs of the claims on braidings 
of Cartan type, see \S \ref{subsec:nichols-RS},
rely on the action of the braid group described by Lusztig \cite{Lu-jams1990,Lu-dedicata}, as generalized in 
\cite{H-Weyl gpd}. It turns out that this action has a subtle extension to the context of Nichols algebras of diagonal, but not necessarily Cartan, type. As we shall see, the adequate language to express this extension is that of groupoids acting on bundles of sets.

\medbreak
Recall that  $(V,c)$ and  $\bq$ are as in \S \ref{subsec:braid}.

\begin{definition}\label{def:admissible}
The matrix $\bq$ (or $V$, or   $\toba_{\bq}$) is \emph{admissible}
if for all $i  \neq j$ in $\I$, the set $\left\{ n \in \N_0: (n+1)_{q_{ii}}
(1-q_{ii}^n q_{ij}q_{ji} )=0 \right\}$ is non-empty. If this happens, then we 
consider the matrix $C^{\bq} =(c_{ij}^{\bq}) \in \Z^{\I \times \I}$ given by $c_{ii}^{\bq} = 2$ and 
\begin{align}\label{eq:defcij}
c_{ij}^{\bq}&:= -\min \left\{ n \in \N_0: (n+1)_{q_{ii}}
(1-q_{ii}^n q_{ij}q_{ji} )=0 \right\},& i&\neq j.
\end{align}
It is easy to see that $C^{\bq}$ is a GCM. If $\bq$ is of Cartan type with GCM $A$, then $C^{\bq} = A$. Thus the matrix $C^{\bq}$ suggests an approximation to Cartan type.
\end{definition}
The GCM $C^{\bq}$ induces reflections $s_i^{\bq}\in GL(\Z^\theta)$, namely 
\begin{align*}
s_i^{\bq}(\alpha_j)&=\alpha_j-c_{ij}^{\bq}\alpha_i,& i,j&\in \I.
\end{align*}
If $\bq$ is of Cartan type, then these reflections generate the Weyl group $W$ and  lift to an action of the braid group 
(corresponding to $W$) on $\Uc_{\bq}(\g)$  \cite{Lu-jams1990,Lu-dedicata}.
In general this is not quite true, but a weaker claim holds. Namely, for $i\in \I$, 
let $\rho_i(\bq)$ be defined by
\begin{align}\label{eq:rhoiq}
\rho_i(\bq)_{jk} &= \bq(s_i^{\bq}(\alpha_j),s_i^{\bq}(\alpha_k)),& j, k &\in \I.
\end{align}

The new braiding matrix $\rho_i(\bq)$ might be different from $\bq$, but nevertheless:

\begin{theorem}\label{th:heck-iso} \cite{H-Weyl gpd}
The reflection $s_i^{\bq}$ lifts to an isomorphism of algebras $T_i:\Uc_{\bq} \to \Uc_{\rho_i(\bq)}$.
\end{theorem}
Recall $\Uc_{\bq}$ from Example \ref{exa:reducedYD-datum-free}.
See \cite{HS} for a generalization and categorical explanation of this result.

Assume now that $\GK \toba_{\bq} < \infty$, for instance that $\dim \toba_{\bq} < \infty$.
Then $\bq$ is admissible by \cite{Ro2}; since $\GK \toba_{\rho_i(\bq)} < \infty$
by Theorem \ref{th:heck-iso}, $\rho_i(\bq)$ is also admissible.
It can be shown without complications that
\begin{align}\label{eq:condicion Cartan scheme-grs}
c^\bq_{ij}&=c^{\rho_i(\bq)}_{ij} &  \mbox{for all }&i,j \in \I.
\end{align}

The fact that $\rho_i(\bq)$ might be different from $\bq$ 
is dealt with by considering 
\begin{align}\label{eq:Xq}
\cX_{\bq} &= \{\rho_{i_k} \dots \rho_{i_2}\rho_{i_1}(\bq),& k&\in \N_0, &i_1, \dots, i_k &\in \I\}.
\end{align}
Thus all $\bp \in \cX_{\bq}$ are admissible.
The set $\cX_{\bq}$ comes equipped with maps $\rho_i: \cX_{\bq} \to \cX_{\bq}$ given by 
\eqref{eq:rhoiq} for each $\bp \in \cX_{\bq}$. It is easy to see that $\rho_i^2 =\id$.

Thus each $\bp \in \cX_{\bq}$ gives rise to a Nichols algebra with  finite GK-dimension (or finite dimension, according to the assumption).
The \emph{generalized root system} of the Nichols algebra $\toba_{\bq}$ is the collection of sets
\begin{align}\label{eq:Deltaq}
(\varDelta^{\bp})_{\bp \in \cX_{\bq}}.
\end{align}

We also have reflections $s_i^{\bp}$ for $i\in \I$ and $\bp \in \cX_{\bq}$; they satisfy
\begin{align}\label{eq:Deltaq-action}
s_i^{\bp}(\varDelta^{\bp}) &= \varDelta^{\rho_i(\bp)}.
\end{align}

Altogether, the reflections $s_i^{\bp}$, $i\in \I$ and $\bp \in \cX_{\bq}$,  generate the so called  Weyl groupoid
$\cW$, as a subgroupoid of $\cX_{\bq} \times GL(\zt) \times \cX_{\bq}$.
By \eqref{eq:condicion Cartan scheme-grs} and \eqref{eq:Deltaq-action}, the Weyl groupoid
$\cW$ acts on  $(C^{\bp})_{\bp \in \cX_{\bq}}$ and on the generalized root system $(\varDelta^{\bp})_{\bp \in \cX_{\bq}}$,
that we think of as bundles of matrices and sets over $\cX_{\bq}$, respectively.
These actions are crucial for the study of Nichols algebras of diagonal type.

\subsection{The axiomatics}\label{subsec:axiomatics}

The ideas outlined in \S \ref{subsec:Weyl-gpd} fit into an axiomatic framework designed in \cite{HY} 
for this purpose. 
We overview now this approach following the conventions in \cite{AA-GRS-CLS-NA}. 
We denote by $\s_{\cX}$ the group of symmetries of a set $\cX$.

\subsubsection{Basic data}\label{subsubsection:basic-data}
All the combinatorial structures in this context are sorts of bundles over a \emph{basis with prescribed changes}, that
we formally call a \emph{basic datum}. This is a pair $(\cX, \rho)$, where $\I \neq \emptyset$ is a finite set, $\cX \neq \emptyset$ is  a  set and $\rho: \I  \to \s_{\cX}$ satisfies $\rho_i^2 = \id$ for all $i\in \I$.
We assume safely that $\I = \I_{\theta}$, for some $\theta\in \N$.
We say that the datum has base $\cX$  and size $\I$ (or $\theta$).

\smallbreak
We associate to a basic datum $(\cX, \rho)$ the quiver
\begin{align*}
\cQ_{\rho} = \{\sigma_i^x := (x, i, \rho_i(x)): i\in \I, x\in \cX\}
\end{align*} 
over $\cX$ (i.e. with set of points $\cX$),  source $s(\sigma_i^x) = \rho_i(x)$ and  target
$t(\sigma_i^x) = x$, $x \in \cX$. The diagram of $(\cX, \rho)$ is
the graph with points $\cX$ and one arrow between $x$ and $y$ decorated with the labels $i$
for each pair $(x, i, \rho_i(x))$, $(\rho_i(x), i, x)$ such that $x \neq \rho_i(x) = y$.
Thus we omit the loops, that can be deduced from the diagram and $\theta$. 
Here is an example with $\theta = 4$:

\begin{align}\label{eq:basicdatum-diagram}
&  \xymatrix{  &     \underset{x}{\bullet} \ar@/^1,5pc/^{1}[dd] \ar@/^{-1,5pc}/_{2}[dd]  \ar@(lu,ld)[]_{3}  \ar@(rd,ru)[]_{4}     &
\\ \\\underset{y}{\bullet} \ar@(dr,dl)[]^{3} \ar@(ur,ul)[]_{1} \ar@(lu,ld)[]_{2} \ar@/^0,5pc/^{4}[r] 
& \underset{z}{\bullet}  \ar@/^0,5pc/^{4}[l] \ar@/^0,5pc/^{3}[r] \ar@/^1pc/_{2}[uu] \ar@/^{-1pc}/^{1}[uu] &
\ar@/^0,5pc/^{3}[l]
\underset{w}{\bullet}
\ar@(ur,ul)[]_{1} \ar@(dr,dl)[]^{4}   \ar@(rd,ru)[]_{2}} &
&\leftrightsquigarrow &&\xymatrix{ \underset{x}{\vtxgpd} \ar@{-} ^{1,2}[d]
\\ 
\underset{y}{\vtxgpd} \hspace{3pt} \text{\raisebox{3pt}{$\overset{4}{\rule{25pt}{0.5pt}}$}}
\hspace{3pt}\underset{z}{\vtxgpd} \hspace{3pt} \text{\raisebox{3pt}{$\overset{3}{\rule{25pt}{0.5pt}}$}}
\hspace{3pt}\underset{w}{\vtxgpd}  }
\end{align}

We say that $(\cX, \rho)$ is connected when $\cQ_{\rho}$ is connected.
Also, we write $x \sim y$ if $x, y$ are in the same connected component of $\cQ_{\rho}$.

\subsubsection{Coxeter groupoids}\label{subsubsection:coxeter-gpd}
we assume that the reader has some familiarity with the theory of groupoids. For example, if $\cX$ is a set and $G$ is a group,
then $\cX \times G  \times \cX$ is a groupoid with the multiplication defined by
\begin{align*}
(x, g, y)(y,h,z) &= (x, gh, z),& x,y,z &\in \cX, \, g,h \in G,
\end{align*}
while $(x, g, y)(y',h,z)$ is not defined if $y \neq y'$.
Recall that

\begin{itemize} [leftmargin=*]\renewcommand{\labelitemi}{$\diamondsuit$}
\medbreak\item  A Coxeter matrix of size $\I$ is a symmetric matrix $\bm = (m_{ij})_{i,j \in \I}$ with entries in
$\Z_{\ge 0} \cup \{+\infty\}$ such that $m_{ii} = 1$ and $m_{ij} \geq 2$, for all $i\neq j\in \I$.

\medbreak\item The  forgetful functor from the category of groupoids over $\cX$ to that of quivers over $\cX$  admits a left adjoint; i.e. every quiver $\cQ$ over $\cX$ determines 
a  free groupoid $F(\cQ)$ over $\cX$, whose construction  is pretty much the same as the construction of the free group.

\medbreak\item Consequently we may speak of the groupoid $\Gc$ presented by a quiver $\cQ$ with relations $N$ (which has to be a set of loops), 
namely $\Gc = F(\cQ)/ \mathcal N$, where $\mathcal N$ is the normal subgroup bundle of $F(\cQ)$ generated by $N$.
\end{itemize}

We fix a basic datum $(\cX, \rho)$ of size $\I$. We denote in the free groupoid $F(\cQ_{\rho})$, or any quotient thereof, 
\begin{align}\label{eq:conventio}
\sigma_{i_1}^x\sigma_{i_2}\cdots \sigma_{i_t} =
\sigma_{i_1}^x\sigma_{i_2}^{\rho_{i_1}(x)}\cdots \sigma_{i_t}^{\rho_{i_{t-1}} \cdots \rho_{i_1}(x)}
\end{align}
i.~e., the implicit superscripts are the only possible allowing compositions.

\begin{definition}
A \emph{Coxeter datum} for $(\cX, \rho)$  is a bundle of Coxeter matrices $\bM = (\bm^x)_{x\in \cX}$, $\bm^x = (m^x_{ij})_{i,j\in \I}$, such that
\begin{align}\label{eq:coxeter-datum}
s((\sigma^x_i\sigma_j)^{m^x_{ij}}) &= x,& i,j&\in \I,& x&\in \cX,
\\
\label{eq:coxeter-datum2}
m^x_{ij}&=m^{\rho_i(x)}_{ij} &  \mbox{for all }&x \in \cX, \, i,j \in \I.
\end{align}
Alternatively, we say that $(\cX, \rho, \bM)$ is a Coxeter datum.
The \emph{Coxeter groupoid} $\cW = \cW(\cX, \rho, \bM)$ 
is the groupoid  generated by $\cQ_{\rho}$ with relations
\begin{align}\label{eq:def-coxeter-gpd}
(\sigma_i^x\sigma_j)^{m^x_{ij}} &= \id_x, &i, j\in \I,&  &x\in \cX.
\end{align}
\end{definition}

The requirement \eqref{eq:coxeter-datum} just says that \eqref{eq:def-coxeter-gpd} makes sense.

\begin{example}\label{ex:coxetergroupoid-base1}
A Coxeter grupoid over a basic datum of size one is just a Coxeter group.
\end{example}

It is equivalent to give a groupoid over a basis $X$ or an equivalence relation on $X$
together with one group for each equivalence class. Coxeter groupoids are more intricate than
equivalence relations with  one Coxeter group for each class.
We now describe all Coxeter groupoids over a basis of 2 elements.

\begin{example}\label{ex:coxetergroupoid-base2} \cite{AA-GRS-CLS-NA}
Let $(\cX, \rho)$ be the basic datum of size $\theta$ with $\cX =\{x, y\}$ and $\ell$ loops at each point, that
we label as follows:
\begin{align*}
\underset{x}{\vtxgpd}\hspace{3pt}
\raisebox{3pt}{$\overset{\ell + 1,\, \dots, \,\theta}{\rule{50pt}{0.5pt}}$}
\hspace{3pt} \underset{y}{\vtxgpd}.
\end{align*}
Let  $\bm^x = (m_{ij})_{i,j \in \I}$ and $\bm^y = (n_{ij})_{i,j \in \I}$ be  Coxeter matrices
such that 
\begin{align*}
m_{ih},n_{ih}&\in 2\Z, & i&\in \I_{\ell}, \, h\in \I_{\ell + 1, \theta},& 
&\text{what is tantamount to \eqref{eq:coxeter-datum},} \\
m_{kj} &= n_{kj}, & k&\in \I_{\ell + 1, \theta}, \, j\in \I,& 
&\text{what is tantamount to \eqref{eq:coxeter-datum2}.}
\end{align*}
By symmetry, $m_{jk} = n_{jk}$, for $k\in \I_{\ell + 1, \theta}$, $j\in \I$. Set
\begin{align*}
c_{ih} &= \frac{m_{ih}}2 = \frac{n_{ih}}2, \,i\in \I_{\ell} , \, h\in \I_{\ell + 1, \theta}.
\end{align*}
The associated Coxeter groupoid is isomorphic to
$\cX\times H \times \cX$, where $H$ is the group presented by generators
\begin{align*}
&s_i,& &t_i,&  i&\in \I_{\ell} ,& &u_{h},& h&\in \I_{\ell + 1, \theta},\, h<\theta;
\end{align*}
with defining relations
\begin{align*}
(s_is_j)^{m_{ij}}  &= e =(t_it_j)^{n_{ij}}, & i, j&\in \I_{\ell} ; \\
(s_it_{ih})^{c_{ih}} &= e, & i&\in \I_{\ell}, \quad  h\in \I_{\ell + 1, \theta};\\
u_{hk}^{m_{hk}} &= e,& h,k&\in \I_{\ell + 1, \theta}, \quad h<k.
\end{align*}
Here we denote
\begin{align*}
u_{hk} &= u_{h}u_{h+1} \dots u_{k-1} = u_{kh}^{-1},& t_{ih} &= u_{h\theta}t_iu_{h\theta}^{-1},&
i&\in \I_{\ell},& h < k&\in \I_{\ell + 1, \theta}.
\end{align*}
In particular, we see
that the isotropy groups of a Coxeter groupoid are not necessarily Coxeter groups.
\end{example}

In a Coxeter groupoid, we may speak of the length and a reduced expression of any element, as for Coxeter groups.

\subsubsection{Generalized root systems}\label{subsection:weylgroupoid-def}
We are now ready for the main definition of this Subsection (that is not the 
same as the one considered in \cite{Ser-root systems}).
Let $\cR = (\cX, \rho)$ be a connected basic datum of size $\I = \I_{\theta}$.
Recall that $\{\alpha_i\}_{i\in\I}$ denotes the  canonical basis of $\Z^{\I}$.

\begin{definition} \cite{HY}
A \emph{generalized root system} for $(\cX, \rho)$ (abbreviated GRS)
is a pair $(\cC, \varDelta)$,
where 
\begin{itemize} [leftmargin=*]
\medbreak\item  $\cC = (C^x)_{x\in \cX}$
is a bundle of generalized Cartan matrices $C^x = (c^x_{ij})_{i,j \in \I}$, cf. \eqref{eq:gcm}, satisfying
\begin{align}\label{eq:condicion Cartan scheme}
c^x_{ij}&=c^{\rho_i(x)}_{ij} &  \mbox{for all }&x \in \cX, \, i,j \in \I.
\end{align}
As usual, these GCM give rise to reflections  $s_i^x\in GL (\Z^{\I})$  by
\begin{align}\label{eq:reflection-x}
s_i^x(\alpha_j)&=\alpha_j-c_{ij}^x\alpha_i, & j&\in \I,&  i &\in \I, x \in \cX.
\end{align}
By \eqref{eq:condicion Cartan scheme}, $s_i^x$ is the inverse of $s_i^{\rho_i(x)}$.

\medbreak\item $\varDelta = (\varDelta^x)_{x\in \cX}$ is a bundle of subsets $\varDelta^x \subset \Z^{\I}$ (we call this a \emph{bundle of root sets}) such that
\begin{align}
\label{eq:def root system 1}
\varDelta^x &= \varDelta^x_+ \cup \varDelta^x_-, & \varDelta^x_{\pm} &:= \pm(\varDelta^x \cap \N_0^{\I}) \subset \pm\N_0^{\I};
\\ \label{eq:def root system 2}
\varDelta^x \cap \Z \alpha_i &= \{\pm \alpha_i \};& &
\\ \label{eq:def root system 3}
s_i^x(\varDelta^x)&=\varDelta^{\rho_i(x)},& \text{cf. }&\eqref{eq:reflection-x};
\\ \label{eq:def root system 4}
(\rho_i\rho_j)^{m_{ij}^x}(x)&=(x), & m_{ij}^x &:=|\varDelta^x \cap (\N_0\alpha_i+\N_0 \alpha_j)|,
\end{align}
for all $x \in \cX$, $i \neq j \in \I$.
\end{itemize}
We call $\varDelta^x_+$, respectively $ \varDelta^x_-$, 
the set of \emph{positive}, respectively \emph{negative}, roots.
\end{definition}

\begin{definition}\label{def:weylgpd} Let $\cR =  (\cC, \varDelta)$ be a generalized root system.

\begin{itemize}[leftmargin=*]\renewcommand{\labelitemi}{$\circ$}
\item
The Weyl groupoid  $\cW$ is
the subgroupoid of $\cX \times GL(\Z^\theta) \times \cX$ generated by all
$\varsigma_i^x = (x, s_i^x,\rho_i(x))$, $i \in \I$, $x \in \cX$.

\medbreak \item  If  $x\in \cX$, then we set
$\bm^x = (m^x_{ij})_{i,j\in \I}$, where $m^x_{ij}$ is defined as in \eqref{eq:def root system 4}.
By the axioms above,
$\bM = (\bm^x)_{x\in \cX}$ is a Coxeter datum for $(\cX, \rho)$.

\medbreak \item
Let $x, y \in \cX$. If $w \in \cW(x, y)$, then $w(\varDelta^x)= \varDelta^y$, by \eqref{eq:def root system 3}.
Thus the sets of \emph{real} and \emph{imaginary} roots at $x$ are
\begin{align*}
(\varDelta^{\re})^x &= \bigcup_{y\in \cX}\{ w(\alpha_i): \ i \in \I, \ w \in \cW(y,x) \}, &
(\varDelta^{\im})^x &= \varDelta^{x} - (\varDelta^{\re})^x.
\end{align*}
\end{itemize}
\end{definition}

In analogy with  Cartan matrices of finite type, finite GRS are characterized by all roots being real. 
Let $\cR =  (\cC, \varDelta)$ be a generalized root system. We say that $\cR$ is \emph{finite} if $\vert \cW \vert < \infty$.

\begin{theorem}\label{th:finite-grs}
\begin{enumerate}[leftmargin=*, label=\rm{(\alph*)}]
\item\label{item:prop-grs2}  \cite[2.11]{CH-at most 3} $\cR$ is finite $\iff \vert\varDelta^x\vert < \infty, \forall x\in \cX\iff$
\begin{align*}
\exists x\in \cX: \vert\varDelta^x\vert < \infty   
\iff 
\vert(\varDelta^{\re})^x\vert < \infty, \forall x\in \cX.
\end{align*}

\medbreak \item\label{item:prop-grs2.5} \cite[Corollary 5]{HY}
Assume that $\cR$ is finite. Pick $x\in \cX$. Then there is a unique
$\omega_0^x \in \cW$ ending at $x$ of maximal length $\ell$; 
all reduced expressions of $\omega_0^x$ have length $\ell$. 
If $y \in \cX$, then $\omega_0^y$ has length $\ell$.

\medbreak \item\label{item:prop-grs3}   \cite[Prop. 2.12]{CH-at most 3}
Assume that $\cR$ is finite. Pick $x\in \cX$ and fix
a reduced expression $\omega_0^x =\sigma^x_{i_1} \cdots \sigma_{i_\ell}$.
Then $$\varDelta^x_+ = \{\beta_j: j \in \I_\ell \},$$ where 
$\beta_j := s_{i_1}^x\cdots s_{i_{j-1}}(\alpha_{i_j}) \in \varDelta^x$, 
$j \in \I_\ell$. Hence, all roots are real.

\medbreak \item\label{item:prop-grs1} \cite{HY}
There is an epimorphism of groupoids $\cW(\cX, \rho, \bM)\to \cW(\cX, \rho, \cC)$.
If $\cR$ is finite, then this is an isomorphism.

\medbreak \item\label{item:prop-grs4} \cite[Theorem 4.2]{HV}
If $\cR$ is finite, then there is $x \in \cX$ such that $C^x$ is of finite type.

\end{enumerate}
\end{theorem}

An outcome of the Theorem is that a finite GRS is determined by the bundle $\cC$ of generalized Cartan matrices. 
It would be coherent to call \emph{arithmetic} root system to a  finite GRS. 

As expected, Nichols algebras of diagonal type with finite dimension or GK-dimension give rise to 
generalized root systems.

\begin{exa}\label{prop:q-gkdim-root-system} Let $(V,c)$ and $\bq$ be as in \S \ref{subsec:braid}.
Assume that $\GK \toba_{\bq} < \infty$. 
Let $(\cX_{\bq}, \rho)$ be as in \eqref{eq:Xq}, let
$\cC = (C^{\bp})_{\bp \in \cX_{\bq}} $ be the bundle of generalized Cartan matrices defined by \eqref{eq:defcij},
and let $\varDelta = (\varDelta^{\bp})_{\bp \in \cX_{\bq}} $ be as in \eqref{eq:Deltaq}.
Then $\cR= (\cC, \varDelta)$ is a generalized root system for  $(\cX_{\bq}, \rho)$.
\end{exa}

We summarize the relation between generalized root systems and Nichols algebras:

\begin{rem}\begin{enumerate}[leftmargin=*, label=\rm{(\alph*)}]
\item\label{it:rem-grs-a} The classification of the arithmetic Nichols algebras  of diagonal type 
(characteristic 0) was achieved in \cite{H-classif RS}, as said.

\medbreak\item\label{it:rem-grs-b} Later, the classification of the finite generalized root systems 
was obtained in \cite{CH-classification}. There are finite GRS that 
do not arise from arithmetic Nichols algebras; at least one of them 
arises from a finite dimensional Nichols algebras  of diagonal type in positive characteristic.

\medbreak\item\label{it:rem-grs-c}
Let $\cR$ be a finite GRS arising from a Nichols algebra $\toba_{\bq}$ of diagonal type. 
We say that $\toba_{\bq}$ is an incarnation of $\cR$; incarnations are by no means unique, see Remark 
\ref{rem:incarnation-not-unique}.

\medbreak\item\label{it:rem-grs-d} More generally, let $\toba$ be the 
Nichols algebra of a semisimple Yetter-Drinfeld module. If $\dim \toba < \infty$, then it gives rise
to a finite GRS \cite{HS-london}. No explicit examples are known, except diagonal type and the following:
The classification of the finite dimensional Nichols algebras 
over finite groups, semisimple but neither simple nor of diagonal type
(arbitrary characteristic) was achieved in \cite{HV}. It turns out that all GRS appearing here arise also
in diagonal type; explicitly they are standard with $|\cX|=1$ of types $A_{\theta}$, $B_{\theta}$,
$\theta \geq 2$, $C_{\theta}$,  $\theta \geq 3$, $D_{\theta}$, $\theta \geq 4$, $E_{\theta}$, $\theta \in\I_{6,8}$, $F_4$, $G_2$; the root systems of types $\Br(2)$, $\Br(3)$, and $\SBrown(2,3)$.

\medbreak\item\label{it:rem-grs-e} When $\GK\toba_{\bq} < \infty$,  it is conjectured that the associated GRS
is finite \cite{AAH}. 
There is some evidence: the conjecture is true for $\dim V = 2$ or for affine Cartan type.
We observe that apparently the imaginary roots of a GRS are not determined by the real ones, contrarily
to what happens with generalized Cartan matrices. 
See \S \ref{sec:appendix-finiteGK} for the list of Nichols algebras with arithmetic root systems and positive $\GK$.
\end{enumerate}
\end{rem}

The following Proposition will be used when discussing incarnations.

\begin{prop}\label{prop:equality-root-systems}
Let $(\cX, \rho)$ be a basic datum and let  $\cR= (\cC, \varDelta)$ and $\cR' = \cR(\cC, \varDelta')$ be two generalized root systems for  $(\cX, \rho)$
with the same bundle $\cC$. Then the bundles $\varDelta^{\re} = ((\varDelta^{\re})^{x})_{x \in \cX} $ and ${\varDelta'}^{\re} = (({\varDelta'}^{\re})^{x})_{x \in \cX} $
are equal. In particular $\cR$ is finite if and only if $\cR'$ is finite; if this happens, then $\cR = \cR'$.
\end{prop}

\pf
Both $\cR$ and $\cR'$ have the same Weyl groupoid because this is defined by $\cC$, implying the first claim. Now the second claim follows from the first and Theorem \ref{th:finite-grs}.
\epf

\subsection{The Weyl groupoid of a (modular) Lie (super)algebra}\label{subsec:Weyl-gpd-super}

The generalized root systems appear in other settings. Important for this monograph is that of 
contragredient Lie superalgebras. All results in this Subsection are from  \cite{AA-GRS-CLS-NA}, unless explicitly quoted otherwise.

\medbreak
Let $\theta \in \N$, $\I = \I_\theta$. We fix 
\begin{itemize}[leftmargin=*]\renewcommand{\labelitemi}{$\circ$}
\item a  field $\kk$  of characteristic $\ell$, 

\item $A =(a_{ij})\in\kk^{\I\times\I}$ 

\item $\pa=(p_i)\in \G_2^\I$,  when $\ell\neq 2$, 

\item a vector space $\h$  of dimension $2\theta-\rk A$,  with
a basis $(h_i)_{ i\in \I_{2\theta-\rk A}}$;

\item a linearly independent family $(\xi_i)_{i\in \I}$ in $\h^*$ such that 
\begin{align*}
\xi_j(h_i) &= a_{ij},& i,j&\in\I.
\end{align*}
\end{itemize}

We call $\pa$ the \emph{parity vector}; the additive version is  $\vert i \vert = \frac{1- p_i}2 \in \Z/2$. 

\smallbreak
From these data, we define a Lie superalgebra (a Lie algebra whenever $\pa = \uno := (1, \dots, 1)$) in the usual way.
First we define the Lie superalgebra $\gt:=\gt(A,\pa)$ by generators $e_i$, $f_i$, $i\in \I$, and $\h$,
subject to the relations:
\begin{align}
\label{eq:relaciones gtilde}
[h,h']&=0, &[h,e_i] &= \xi_i(h)e_i, & [h,f_i] &= -\xi_i(h)f_i, & [e_i,f_j]&=\delta_{ij}h_i,
\end{align}
for all $i, j \in \I$, $h, h'\in \h$, with parity given by
\begin{align*}
|e_i|&=|f_i|= \vert i \vert,&  i&\in\I, & |h|&=0,&  h&\in\h.
\end{align*}
This Lie superalgebra has a triangular decomposition $\gt=\ntp\oplus\h\oplus\ntm$ that arises from the 
$\Z$-grading $\gt = \mathop{\oplus}\limits_{k\in \Z}\gt_k$ determined by $e_i\in \gt_1$, $f_i \in \gt_{-1}$,
$\h = \gt_0$.
The \emph{contragredient Lie superalgebra} associated to $A$, $\pa$ is  
\begin{align*}
\g(A,\pa):= \gt(A,\pa) / \rg
\end{align*}
where $\rg=\rg_+\oplus \rg_-$ is the maximal $\Z$-homogeneous ideal intersecting $\h$ trivially.
We set $\g:=\g(A,\pa)$,  and identify $e_i$, $f_i$, $h_i$, $\h$ with their images in $\g$.
Clearly $\g$ inherits the grading of $\gt$ and $\g=\np\oplus\h\oplus\nm$, where $\n_\pm =\nt_\pm / \rg_\pm$.
As in \cite{K-libro,BGL}, we assume from now on that $A$ satisfies
\begin{align}\label{eq:symmetrizable}
&a_{ij}=0\mbox{ if and only if }a_{ji}=0, & \mbox{for all }j&\neq i.
\end{align}
By \cite[Section 4.3]{BGL} or \cite[Remark 4.2]{CE}, $\g$ is  $\Z^{\I}$-graded by
\begin{align*}
\deg e_i&=-\deg f_i=\alpha_i, & \deg h&=0, & i\in \I,\ h\in \h.
\end{align*}
The roots, respectively the positive, or negative, roots, are the elements of
\begin{align*}
\nabla^{(A,\pa)} &=\{\alpha\in\Z^{\I} - 0:\g_\alpha\neq0\}, & \nabla_{\hspace{2pt} \pm} &=\nabla^{(A,\pa)}\cap(\pm\nct).
\end{align*}
For instance, the simple roots are $\alpha_i \in \nabla^{(A,\pa)}$, $i\in\I$. Then
\begin{align*}
\nabla^{(A,\pa)}&=\nabla_{\hspace{2pt} +}^{(A,\pa)}\cup\nabla_{\hspace{1pt} -}^{(A,\pa)}, & 
\nabla_{\hspace{2pt} -}^{(A,\pa)}&= -\nabla_{\hspace{2pt} +}^{(A,\pa)}.
\end{align*}
We say that $(A,\pa)$ is \emph{admissible} if 
\begin{align}\label{eq:f-nilpotent}
\ad f_i &\text{ is locally nilpotent in } \g = \g(A,\pa)
\end{align}
for all  $i\in\I$, cf. \cite{Ser-superKM}. For instance, $(A,\pa)$ is admissible when
$\g(A, \pa)$ is finite-dimensional, or   $\ell >0$ \cite{AA-GRS-CLS-NA}.

If $(A,\pa)$ is admissible, then we define  $C^{(A,\pa)}= \big(c_{ij}^{(A,\pa)}\big)_{i,j\in\I} \in \Z^{\I \times \I}$
by
\begin{align*}
c_{ii}^{(A,\pa)}&= 2, &
c_{ij}^{(A,\pa)}&:=-\min\{m\in\N_0:(\ad f_i)^{m+1} f_j= 0 \},& i&\neq j \in \I.
\end{align*}
Let $s_i^{(A,\pa)}\in GL(\Z^{\I})$ be the  involution given by
\begin{align}\label{eq:definicion si}
s_i^{(A,\pa)}(\alpha_j)&:= \alpha_j-c_{ij}^{(A,\pa)}\alpha_i, & j\in\I.
\end{align}

Let $i\in \I$; set $\rho_i\pa = (\overline{p}_j)_{j\in\I}$,
$\overline{p}_j = p_jp_i^{c_{ij}^{(A,\pa)}}$. 
In \cite{AA-GRS-CLS-NA}, we introduce a matrix $\rho_iA$ and the pair $\rho_i(A, \pa) := (\rho_iA, \rho_i\pa)$. 

\begin{theorem}\label{thm:isomorfismo Ti} 
Let $A\in \kk^{\I\times \I}$  satisfying \eqref{eq:symmetrizable} and
$\pa\in (\G_2)^{\I}$. Assume that $(A,\pa)$ is admissible.
Then there are Lie superalgebra isomorphisms 
\begin{align}\label{eq:Ti-Liesuperalg}
T_i^{(A,\pa)}:\g(\rho_i A,\rho_i\pa)&\to\g(A,\pa),& i&\in \I, \\ \intertext{such that}
\label{eq:Ti mueve grados via si}
T_i^{(A,\pa)}\left(\g(\rho_i A,\rho_i\pa)_\beta\right) &= \g(A,\pa)_{s_i^{(A,\pa)}(\beta)}, &  \beta &\in\pm\N_0^\theta.
\end{align}
\end{theorem}

We consider the equivalence relation $\sim$ in $\kk^{\I\times \I}\times \G_2^{\I}$ generated by 
\begin{itemize}
\item $(A,\pa)\equiv (B,\bq)$  iff the rows of $B$ are obtained from those of $A$ multiplying by non-zero scalars,

\item $(A,\pa)\approx (B,\bq)$  iff $A$ is satisfies \eqref{eq:symmetrizable}
and there exists $i\in \I$ fulfilling \eqref{eq:f-nilpotent} such that $\rho_i(A,\pa)\equiv (B,\bq)$.
\end{itemize}

We denote by 
$\cX^{(A, \pa)}$ the equivalence class of $(A, \pa)$ with respect to $\sim$.

\begin{definition} A  pair $(A,\pa)$ is \emph{regular} if and only if  every $(B,\bq) \in \cX^{(A, \pa)}$ is  admissible
and satisfies \eqref{eq:symmetrizable}. Evidently, all $(B,\bq) \in \cX^{(A, \pa)}$ are regular too.
Therefore there are reflections $T_i$ for all $(B,\bq) \in \cX^{(A, \pa)}$ and  $i\in \I$.
\end{definition}

If $\ell>0$, then   `$(A,\pa)$ regular' says that all $(B,\bq) \in \cX^{(A, \pa)}$ satisfy \eqref{eq:symmetrizable}.

\medbreak
Let $(A,\pa)$ be a regular pair. We set
\begin{align}\label{eq:root system (A,p)}
\varDelta_+^{(A,\pa)} &= \nabla_{\hspace{2pt} +}^{(A,\pa)} - \{k\, \alpha: \, \alpha\in\nabla_{\hspace{2pt} +}^{(A,\pa)}, k\in\N, k\geq 2\},
\\
\cC^{(A,\pa)} &=\left(\I,\cX^{(A, \pa)},(\rho_i)_{i\in\I},(C^{(B,\bq)})_{(B,\bq)\in\cX^{(A, \pa)}}\right)
\end{align}

\begin{theorem}\label{thm:root system}
$(\cC^{(A,\pa)}, (\varDelta^{(B,\bq)})_{(B,\bq)\in\cX^{(A, \pa)}})$ is a generalized root system. 
\end{theorem}

This is the point we wanted to reach:

\begin{prop}\label{prop:fd regular} Let $(A,\pa)$ as above, i.e. $A \in \kk^{\I \times \I}$ and $\pa \in \G_2^{\I}$. 
If $A$ satisfies \eqref{eq:symmetrizable} and $ \dim\g(A,\pa)$ is finite, then $(A,\pa)$ is regular, thus it has a generalized root system. 
\end{prop}

Now the classification of the finite-dimensional contragredient Lie superalgebras is known and consists of the following:

\begin{itemize} [leftmargin=*]\renewcommand{\labelitemi}{$\diamond$}
\medbreak\item  If $\ell =0$ and $\pa = \uno$, then this is the Killing-Cartan classification of simple Lie algebras of types $A, \dots, G$.

\medbreak\item If $\ell =0$ and $\pa \neq \uno$, then this belongs the classification of simple Lie superalgebras \cite{K-super}.

\medbreak\item If $\ell > 0$ and $\pa = \uno$, then the analogous of Lie algebras in characteristic 0, the Brown algebras $\br(2;a)$, $\br(2)$, $\br(3)$ \cite{BGL,Br,Skryabin} for $\ell=3$, and the Kac-Weisfeiler algebras $\bgl(3;a)$, $\bgl(4;a)$ \cite{BGL,KW-exponentials} for $\ell=2$.

\medbreak\item If $\ell > 0$ and $\pa \neq \uno$, then the analogous of Lie algebras in characteristic 0, the Brown superalgebra $\brj(2;3)$, the Elduque superalgebra $\el(5;3)$, the Lie superalgebras $\g(1,6)$, $\g(2,3)$, $\g(3,3)$, $\g(4,3)$, $\g(3,6)$, $\g(2,6)$, $\g(8,3)$, $\g(4,6)$, $\g(6,6)$, $\g(8,6)$ \cite{BGL,CE,E1,E2} for $\ell=3$, and the Brown superalgebra $\brj(2;5)$, the Elduque superalgebra $\el(5;5)$ \cite{BGL} for $\ell=5$.

\end{itemize}

\subsection{Classification}\label{subsec:classification}
Recall that  $(V,c)$ is a finite-dimensional braided vector space of diagonal 
type with braiding matrix  $\bq$ as in \S \ref{subsec:braid}.
In the celebrated article \cite{H-classif RS}, the classification of the Nichols algebras of diagonal type
with arithmetic root system was presented in the form of several tables. 

Roughly, the method of the proof consists in deciding when the Weyl groupoid is finite, iterating the construction \eqref{eq:rhoiq}.
The procedure goes recursively on $\theta$; it could be shortened using the reduction given 
by Theorem \ref{th:finite-grs} \ref{item:prop-grs4}, see also \cite{HW}.

\medbreak
We propose an alternative organization of  the classification. Assume that $\bq$ is arithmetic, e.g. that
$\dim \toba_{\bq} < \infty$,  and 	let $\cR$ be its GRS.

\begin{itemize} [leftmargin=*]\renewcommand{\labelitemi}{$\diamond$}
\medbreak\item  If the bundles of matrices and of root sets are constant, then we say that $(V,c)$ is  
of \emph{standard type}; braided vector spaces of Cartan type fit here.

\medbreak\item If $\cR$ is isomorphic to the GRS of a Lie superalgebra in Kac's list above, then we say
that $(V,c)$ is of \emph{super type} \cite{AAY}.

\medbreak\item If $\cR$ is isomorphic to the GRS of a contragredient Lie superalgebra in characteristic $\ell > 0$,
 as above, then we say that $(V,c)$ is of \emph{modular type}.

\end{itemize}

Most of the $\bq$ with the assumption above fall into one of these three classes,
showing the deep relation between Nichols algebras and Lie theory. 
From the list of \cite{H-classif RS}, there are still 12 examples whose GRS 
could not be identified in Lie theory; we call them \emph{UFO}'s.
Actually, they come from 11 different GRS, as one of them incarnates in two distinct Nichols algebras.

\subsection{The relations of a  Nichols algebra and convex orders}\label{subsec:rels-convex}

\subsubsection{Convex orders}\label{subsubsec:convex}
We start by the concept of convex order in an arithmetic root system.
Let $\varDelta$ be the root system of a finite-dimensional simple Lie algebra, $W$ its Weyl group and $\omega_0 \in W$
the longest element.
Then a total order $<$ on $\varDelta_{+}$ is \emph{convex} if 
\begin{align}\label{eq:convex}
\alpha,\beta &\in \varDelta_{+}, & \alpha <\beta \text{ and }\alpha+\beta &\in\varDelta_{+}& \implies & \alpha < \alpha+\beta < \beta.
\end{align}
A priori, it is not evident why convex orders do exist, but this result gives them all:

\begin{theorem}\label{th:papi}\cite{P}
There is a bijective correspondence between the set of convex orders in $\varDelta_{+}$ and the set of reduced decompositions of $\omega_0$.
\end{theorem}

Let now $\cR =  (\cC, \varDelta)$ be an arithmetic root system over a basic datum $(\cX, \rho)$.
Fix $x \in \cX$. Following \cite{A-jems}, we say that total order $<$ on $\varDelta_{+}^x$ is \emph{convex} if \eqref{eq:convex} holds for all 
$\alpha,\beta \in \varDelta^x_{+}$. Recall $\omega^x_0$, Theorem \ref{th:finite-grs} \ref{item:prop-grs2.5}.

\begin{theorem}\label{th:ivan-convex}\cite{A-jems}
There is a bijective correspondence between the set of convex orders in $\varDelta^x_{+}$ and the set of reduced decompositions of $\omega^x_0$.
\end{theorem}
The correspondence is easy to describe: given a reduced decomposition $\omega_0^x =\sigma^x_{i_1} \cdots \sigma_{i_\ell}$,
the convex total order in $\varDelta^x_{+}$ is induced from the numeration given in Theorem \ref{th:finite-grs} \ref{item:prop-grs3}.

\subsubsection{Defining relations}\label{subsubsec:rels}
We keep the notation $(V,c)$, $\bq$, etc. from \S \ref{subsec:braid} and we assume that $\bq$ is arithmetic. 
The total order in the set of simple roots given by the numeration by $\I$ induces a total order in $\varDelta_{+}^{\bq}$
(the restriction of the lexicographic order) as explained in \S \ref{subsec:pbw}.
This total order turns out to be convex (but there are more convex orders than these when $\theta >2$).
Let $(\beta_k)_{k \in \I_{\ell}}$ be the numeration of $\varDelta^{\bq}_+$ induced by this order.
For every $k \in \I_{\ell}$, let $x_{\beta_k}$ be the corresponding root vector as in Remark \ref{rem:lyndon-word}.

Let $i<j \in \I_{\ell}$, $n_{i+1}, \dots, n_{j-1} \in\N_0$. Because the total order is convex, we conclude that 
there exist $c_{n_{i+1}, \dots, n_{j-1}}^{(i,j)} \in \ku$ such that
\begin{equation}\label{quantumSerregeneralizadas}
\left[ x_{\beta_i}, x_{\beta_j} \right]_c= 
\sum_{n_{i+1}, \dots, n_{j-1} \in\N_0} c_{n_{i+1}, \dots, n_{j-1}}^{(i,j)} \ x_{\beta_{j-1}}^{n_{j-1}} \dots x_{\beta_{i+1}}^{n_{i+1}}.
\end{equation}
The scalars $c_{n_{i+1}, \dots, n_{j-1}}^{(i,j)}$ can be computed explicitly \cite[Lemma 4.5]{A-jems}. Notice that if $\sum n_k\beta_k \neq \beta_i+\beta_j$, 
then $c_{n_{i+1}, \dots, n_{j-1}}^{(i,j)}  =0$, since $\toba_{\bq}$ is $\N_0^{\I}$-graded.

Let $\beta \in \varDelta^{\bq}_+$; we set $N_{\beta} = \ord \bq_{\beta\beta}$. If $N_\beta$ is finite, then
\begin{equation}\label{powerrootvector}
x_{\beta}^{N_{\beta}}=0.
\end{equation}

\begin{theorem}\label{Thm:presentacion} \cite[4.9]{A-jems}
The relations \eqref{quantumSerregeneralizadas},
$i<j \in \I_{\ell}$,  and \eqref{powerrootvector}, $\beta \in \varDelta^{\bq}_+$ with $N_\beta$ finite, 
generate the ideal $\J_{\bq}$ defining the Nichols algebra $\toba_{\bq}$.
\end{theorem}

The proof of this Theorem does not appeal to the classification in \cite{H-classif RS}, but to the theory of finite GRS \cite{H-Weyl gpd,HY}
and the study of coideal subalgebras in \cite{HS-london}.
Starting from Theorem \ref{Thm:presentacion}, the defining relations of $\toba_{\bq}$ for the various $\bq$ in the list in \cite{H-classif RS}  
was given explicitly in \cite[Theorem 3.1]{A-presentation}. The approach in \emph{loc. cit.} does not follow the list 
but the possible local subdiagrams, i.e. of rank 2,3,4, up to insuring the existence of the Lusztig isomorphisms $T_i$, 
analogous to those in Theorem  \ref{th:heck-iso}. The final argument uses that the bundle of Cartan matrices determines the GRS.

\subsection{The Lie algebra of a finite-dimensional Nichols algebra}\label{subsec:Lie-algebra}

A finite-dimensional Nichols algebra $\toba_{\bq}$ of diagonal type gives rise to some remarkable objects: 
its distinguished pre-Nichols algebra $\wtoba_{\bq}$ \cite{A-presentation,A-pre-Nichols}, 
its Lusztig algebra $\luq$ \cite{AAR} and its associated Lie algebra $\mathfrak g$ \cite{AAR2}. 
If $\bq$ is of Cartan type (with entries of odd order), then $\wtoba_{\bq}$ is isomorphic to the
positive part of the quantum group defined by De Concini and Procesi \cite{DP},
while $\luq$ is the positive part of the algebra of divided powers 
introduced by Lusztig \cite{Lu-dedicata,Lu}.
We  expect that these algebras $\wtoba_{\bq}$ and $\luq$
 would give rise to interesting representation theories.
 We discuss succinctly these three notions.

\subsubsection{The distinguished pre-Nichols algebra}\label{subsubsec:distinguished}
We start with the concept of Cartan roots \cite{A-pre-Nichols} and then discuss the definition of  $\wtoba_{\bq}$. 

First, $i\in\I$ is a \emph{Cartan vertex}  of $\bq$  if
$q_{ij}q_{ji} = q_{ii}^{c_{ij}^{\bq}}$, for all $j \in \I$.
Then the set of \emph{Cartan roots} of $\bq$ is
\begin{align*}
\Oc^{\bq} &= \{s_{i_1}^{\bq} s_{i_2} \dots s_{i_k}(\alpha_i) \in \varDelta^{\bq}:
i\in \I  \text{ is a Cartan vertex of } \rho_{i_k} \dots \rho_{i_2}\rho_{i_1}(\bq) \}.
\end{align*}
Thus $\Oc^{\bq} = \Oc^{\bq}_+ \cup \Oc^{\bq}_-$, where $\Oc^{\bq}_{\pm} = \Oc^{\bq} \cap \varDelta_{\pm}^{\bq}$.
 
\medbreak
The distinguished pre-Nichols algebra is defined in terms of the presentation of $\J_{\bq}$ evoked above. 
To explain this, we need the notation:
\begin{align*}
\widetilde N_{\alpha} &= \begin{cases} N_{\alpha} =\ord \bq_{\alpha\alpha} &\mbox{ if }\alpha\notin\Oc_+^{\bq},
\\ \infty  &\mbox{ if }\alpha\in\Oc_+^{\bq}, \end{cases}& \alpha &\in \varDelta_+^{\bq}.
\end{align*}

Let $\cI_{\bq} \subset \J_{\bq}$ be the ideal of $T(V)$ generated by all the relations in \cite[Theorem 3.1]{A-presentation}, but
\begin{itemize}
\item excluding the power root vectors $x_\alpha^{N_\alpha}$, $\alpha\in\Oc_+^{\bq}$,
\item adding the quantum Serre relations $(\ad_c x_i)^{1-c_{ij}^{\bq}} x_j$ for those $i\neq j$ such that
$q_{ii}^{c_{ij}^{\bq}}=q_{ij}q_{ji}=q_{ii}$.
\end{itemize}

\begin{definition} \cite{A-presentation}
The distinguished pre-Nichols algebra $\wtoba_{\bq}$ is the quotient $T(V)/ \cI_{\bq}$. 
\end{definition}

This pre-Nichols algebra is useful for the computation of the liftings; 
it should also be present in the classification of pointed Hopf algebras with finite $\GK$.
See \cite{A-pre-Nichols} for the basic properties of $\wtoba_{\bq}$.

\subsubsection{The Lusztig algebra and the associated Lie algebra}\label{subsubsec:Lusztig-Lie}
For an easy exposition, we suppose that $\bq$ is symmetric, and by technical reasons, that
\begin{align}\label{condition cart roots}
&q_{\alpha\beta}^{N_\beta}=1, & \forall \alpha,\beta\in\Oc^{\bq}.
\end{align}

\begin{definition}
The \emph{Lusztig algebra} $\lu_{\bq} $ of $(V,c)$ is the graded dual 
of the distinguished pre-Nichols algebra $\wtoba_{\bq}$ of $(V^*,\bq)$;
thus, $\toba_{\bq} \subseteq \lu_{\bq}$.
\end{definition}

To describe the associated Lie algebra, we begin by
considering the subalgebra $Z_{\bq}$ of $\wtoba_{\bq}$ generated by $x_{\beta}^{N_{\beta}}$, $\beta \in \Oc_+^{\bq}$.
Then $Z_{\bq}$ is a commutative normal Hopf subalgebra of $\wtoba_{\bq}$ 
\cite{A-pre-Nichols}. In turn, the graded dual $\mathfrak Z_\bq$ of  $Z_\bq$; 
is a cocommutative Hopf algebra,
isomorphic to the enveloping algebra $\Uc(\mathfrak n_{\bq})$ of a Lie algebra 
$\mathfrak n_{\bq}$.
Then $\luq$ is an extension of braided Hopf algebras:
\begin{align*}
\toba_{\bq} \hookrightarrow \luq \twoheadrightarrow \mathfrak Z_{\bq}.
\end{align*}

\begin{theorem} \cite{AAR2,AAR3} The Lie algebra $\mathfrak n_{\bq}$ is either 0 or else
isomorphic to the positive part of a semisimple Lie algebra $\mathfrak g_{\bq}$.
\end{theorem}

The proof we dispose of this Theorem is by computation case-by-case. 
If $\bq$ is of Cartan type, then $\mathfrak g_{\bq}$ corresponds to the same Cartan matrix,
except in even order and type $B$ or $C$, in which case is of type $C$ or $B$.
If $\bq$ is of super type, with corresponding Lie superalgebra $\g$, then $\g_{\bq}$ is isomorphic to the even part
$\g_0$.
Assume that $\bq$ is of modular type. 
Then the Theorem shows, in particular, that the modular Lie algebras or superalgebras
listed above give rise to Lie algebras in characteristic 0. 
However, the latter do not control completely the behaviour of the former.

\begin{rem}
There is a bijection from the set $\Oc_+^{\bq}$ to the set of positive roots of $\g_{\bq}$.
Thus $\Oc^{\bq}$ bears a structure of root system, albeit we do not dispose of a direct proof of this fact.
\end{rem}

\subsection{The degree of the integral}\label{subsec:ia-invariant}
Again,  $(V,c)$ and  $\bq$ are as in \S \ref{subsec:braid}. We assume that $\dim \toba_{\bq} < \infty$.
We introduce an element of the lattice $\Z^{\I}$ which is a relative of the  semi-sum of the positive roots
in the theory of semisimple Lie algebras.
Let $N_{\beta}= \ord(q_{\beta})=\text{h}(x_{\beta})$  for $\beta\in \varDelta_+^{\bq}$. 
Let
\begin{align}\label{eq:ia}
\ya &:= \sum_{\beta\in\varDelta_+^{\bq}} (N_{\beta}-1)\beta \in \Z^{\I}.
\end{align}
Suppose that $\bq$ is of Cartan type. If $\ord q_{ii}$ is relatively prime to the entries of the Cartan matrix $A$,
then $N_{\beta} = N$ is constant and $\ya = (N-1) \sum_{\beta\in\varDelta_+^{A}} \beta$.

\smallbreak
Since $\toba_{\bq} = \oplus_{n\ge 0} \toba_{\bq}^n$ is finite-dimensional, there exists
$d \in \N$ such that $\toba_{\bq}^d \neq 0$ and $\toba_{\bq}^n =0$ for $n >d$. 
Then $\toba_{\bq}$ is unimodular with $\toba_{\bq}^d$  
equal to the space of left and right integrals,
 $\dim \toba_{\bq}^d = 1$ and $\dim \toba_{\bq}^j = \dim \toba_{\bq}^{d-j}$ if $j \in \I_{0,d}$.
See \cite{AG} for details. We set ${\tp} = d$. Then 
\begin{align}\label{eq:top}
\ya &= \deg \toba_{\bq}^{\tp}.
\end{align}

Let now $\Gamma$ be a finite abelian group and  $(\mathbf{g}, \mathbf{\chi})$
a realization of $\bq$, 
cf. \S \ref{subsec:double-Nichols}, 
with $\mathbf{g} = (g_i)_{i\in\I}$, $\mathbf{\chi}) = (\chi_i)_{i\in\I}$
Then we have morphisms of groups $\Z^{\I} \to \Gamma$ and
$\Z^{\I} \to \VGamma$ given by
\begin{align*}
\beta &\longmapsto g_{\beta},& g_{\alpha_i} &= g_i;&  \beta &\longmapsto \chi_{\beta},&
\chi_{\alpha_i} &= \chi_i,& i\in \I.
\end{align*}

We refer to \cite[\S 8.11]{egno}
for the basics of ribbon Hopf algebras. Here is an application of  $\ya$. 

\begin{pro}\label{prop:ribbon-criterio1}
	The distinguished group-likes of $H=\toba(V)\# \ku\Gamma$ are
\begin{align*}
\ch_H&=\chi_{\ya}^{-1}
& &\mbox{and} & 
\drus_H&= g_{\ya}.
\end{align*}
The Drinfeld double $(D(H),\mathcal{R})$ is ribbon if and only if there exist $\brus\in\Gamma$, 
$\zh\in\widehat{\Gamma}$ such that $\brus^2=\drus_H = g_{\ya}$, $\zh^2=\ch_H =\chi_{\ya}^{-1}$ and
\begin{align}\label{eq:ribbon-cond}
	q_{ii}^{-1} &= \chi_i(\brus)\zh^{-1}(g_i), & \mbox{for all }& i\in\I.
\end{align}
\end{pro}

\begin{proof} 
The first claim follows from \cite[4.8, 4.10]{Bu} and the discussion above.
The second claim is a consequence of \cite[Theorem 3]{RK}.
\end{proof}

We next discuss how the preceding construction behaves under the action of the Weyl groupoid.
Let us fix $i\in\I$ and set $\bq' = \rho_i \bq= (q_{jk}')_{j,k\in\I}$.
First we notice that the degree of the integral corresponding to $\bq'$ is 
\begin{align}\label{eq:iaprima}
\ya' &=  s_i^{\bq}(\ya) +2(N_i-1)\alpha_i.
\end{align}
Indeed,
\begin{align*}
\ya' &= \sum_{\beta\in\varDelta_+^{\bq'}} (N_{\beta}-1)\beta 
= \sum_{\gamma\in\varDelta_+^{\bq}-\{\alpha_i\}} (N_{s_i(\gamma)}-1)s_i^{\bq}(\gamma) +(N_i-1)\alpha_i \\
& = \sum_{\gamma\in\varDelta_+^{\bq}-\{\alpha_i\}} (N_{\gamma}-1)s_i^{\bq}(\gamma) +(N_i-1)\alpha_i = s_i^{\bq}(\ya) +2(N_i-1)\alpha_i.
\end{align*}

Now $\mathbf{g'} = (g_j')_{j\in \I}$, $\mathbf{\chi'} =(\chi_j')_{j\in\I}$
gives a a realization of $\bq'$ where
\begin{align}\label{eq:reduced-YD-datum-rhoi}
g_j'&=g_jg_i^{-c_{ij}^{\bq}}=g_{s_i^{\bq}(\alpha_j)}, &  \chi_j'&=\chi_j\chi_i^{-c_{ij}^{\bq}}=\chi_{s_i^{\bq}(\alpha_j)}, 
& j\in\I.
\end{align}

\begin{cor}\label{cor:dist-group-like-rhoi} The distinguished group-likes of
$H'=\toba_{\bq'}\#\ku\Gamma$ are
\begin{align*}
\ch_{H'}&= \ch_{H} \chi_i^{2(N_i-1)}
& &\mbox{and} & 
\drus_{H'}&= \drus_{H} g_i^{2(1-N_i)}.
\end{align*}
If $(D(H),\mathcal{R})$ is ribbon, then 
$(D(H'),\mathcal{R}')$ is also ribbon.
\end{cor}
\pf
The first claim follows by a direct computation from \eqref{eq:iaprima}.
Let
\begin{align*}
\brus' &=\brus g_i^{1-N_i}\in\Gamma,& \zh'&=\zh \chi_i^{N_i-1}\in\widehat{\Gamma}.
\end{align*}
Clearly, 
$(\brus')^2=\drus_H'$, $(\zh')^2=\ch_H'$.
We check that \eqref{eq:ribbon-cond} holds. Let $j\in\I$. Since  $\brus$, $\zh$ satisfy \eqref{eq:ribbon-cond} and $q_{ii}^{N_i}=1$, we have
\begin{align*}
\chi_j'(\brus') & (\zh')^{-1}(g_j') = \chi_j\chi_i^{-c_{ij}^{\bq}}(\brus g_i^{1-N_i}) \zh^{-1} \chi_i^{1-N_i}(g_j g_i^{-c_{ij}^{\bq}})
\\
&= \chi_j(\brus)(\zh)^{-1}(g_j) \left(\chi_i(\brus)\zh^{-1}(g_i) \right)^{-c_{ij}^{\bq}} q_{ii}^{2c_{ij}^{\bq}(N_i-1)} \widetilde{q_{ij}}^{1-N_i}
\\
& = q_{jj}^{-1} q_{ii}^{c_{ij}^{\bq}} q_{ii}^{-2c_{ij}^{\bq}} q_{ij}^{1-N_i} q_{ji}^{1-N_i}
= q_{jj}^{-1} q_{ii}^{-c_{ij}^{\bq}} q_{ij}^{1-N_i} q_{ji}^{1-N_i}.
\end{align*}
Now
\begin{itemize} [leftmargin=*]\renewcommand{\labelitemi}{$\diamond$}
\item  If $j=i$, then $q_{ii}'=q_{ii}$, and $q_{jj}^{-1} q_{ii}^{-c_{ij}^{\bq}} q_{ij}^{1-N_i} q_{ji}^{1-N_i}=q_{ii}^{-1} q_{ii}^{-2} q_{ii}^{2-2N_i}= q_{ii}^{-1}$.

\item If $j\neq i$ and $q_{ii}^{c_{ij}^{\bq}}=q_{ij}q_{ji}$, then $q_{jj}'=q_{jj}$, and
\begin{align*}
q_{jj}^{-1} q_{ii}^{-c_{ij}^{\bq}} q_{ij}^{1-N_i} q_{ji}^{1-N_i}&=
q_{jj}^{-1} q_{ii}^{-c_{ij}^{\bq}} q_{ii}^{c_{ij}^{\bq}(1-N_i)}= q_{jj}^{-1}.
\end{align*}

\item If $j\neq i$ and $q_{ii}^{c_{ij}^{\bq}}\neq q_{ij}q_{ji}$, then $c_{ij}^{\bq}=1-N_i$, so
\begin{align*}
q_{jj}^{-1} q_{ii}^{-c_{ij}^{\bq}} q_{ij}^{1-N_i} q_{ji}^{1-N_i}&=
q_{jj}^{-1} q_{ii}^{-(c_{ij}^{\bq})^2} q_{ij}^{c_{ij}^{\bq}} q_{ji}^{c_{ij}^{\bq}}
= \chi_j\chi_i^{-c_{ij}^{\bq}}(g_j g_i^{-c_{ij}^{\bq}})^{-1}
= (q_{jj}')^{-1}.
\end{align*}

\end{itemize}
Thus $\chi_j'(\brus')(\zh')^{-1}(g_j')= (q_{jj}')^{-1}$  for all $j\in\I$;
Proposition \ref{prop:ribbon-criterio1} applies.
\epf

\part{Arithmetic  root systems: Cartan, super, standard}\label{part:cartan-super-standard}

\section{Outline}\label{sec:outline}
\subsection{Notation}\label{subsec:outline-notation}
In what follows $q\in \k^{\times}- \{1\}$, $N := \ord q \in [2, \infty]$. Recall that $(\alpha_i) _{i\in \I}$ denotes
the canonical basis of $\Z^{\I}$.

The matrices $\bq$ considered here belong to the classification list in \cite{H-classif RS}; 
they may form part of an
infinite series--or not. In the second case, we often use $i$ to denote the root $\alpha_{i}$, and more generally
\begin{align}\label{eq:notation-root-exceptional}
\begin{aligned}
i_1i_2\dots i_k   \text{ denotes }  \alpha_{i_1} + \alpha_{i_2} + \dots \alpha_{i_k} \in \Z^{\I};
\\
\text{also  } i_1^{h_1}i_2^{h_2}\dots i_k^{h_k} \text{ denotes }  
h_1\alpha_{i_1} + h_2\alpha_{i_2} + \dots h_k\alpha_{i_k} \in \Z^{\I}
\end{aligned}
\end{align}

The implicit numeration of any generalized Dynkin diagram is from the left to the right and from bottom to top; otherwise,
the numeration appears below the vertices.

Basic data are described either explicitly or by the corresponding diagram as in page \pageref{eq:basicdatum-diagram}.

If a numbered display contains several equalities (or diagrams), they will be referred to
with roman letters from left to right and from top to bottom; e.g., (\ref{eq:rels-type-A-N>2} c) below means
$x_{(kl)}^N = 0$, $k \leq l$.

If $X_m$ is a generalized Dynkin diagram (or a subset of $\Z^m$ or any variation thereof) with $m$ vertices and $\sigma\in \s_m$, 
then $\sigma(X_m)$ is the generalized Dynkin diagram (or the object in question) with the numeration of
the vertices after applying $\sigma$ to the numeration of $X_m$.
In this respect, $s_{ij}$ denotes the transposition $(ij)$, what should not be confused with the reflection $s_i$.
For brevity, we abbreviate some permutations in $\s_4$ and $\s_5$ as follows:
\begin{align*}
\kappa_1 &=s_{1234}, & \kappa_2 &= s_{234}, 
& \kappa_3 &= s_{12}s_{34}, & \kappa_4 &= s_{13}s_{24},
\\  \kappa_5 &= s_{142}, & \kappa_6 &= s_{1324}, & \kappa_7 &= s_{134}, & \kappa_8 &= s_{324}.
\\
\varpi_1 &=s_{15}s_{234}, & \varpi_2 &= s_{354}, & \varpi_3 &= s_{15}s_{23}, & \varpi_4 &= s_{345}, & \varpi_5 &=s_{14}s_{23}.
\end{align*}

Along the way, we recall the Cartan matrices of types $A$, $B$, $C$, $D$, $E$, $F$ and $G$ with the numeration we use; see \eqref{eq:dynkin-system-A}, \eqref{eq:dynkin-system-B}, \eqref{eq:dynkin-system-C}, \eqref{eq:dynkin-system-D}, \eqref{eq:dynkin-system-E}, \eqref{eq:dynkin-system-F}, \eqref{eq:dynkin-system-G}.
We also need some other generalized Cartan matrices:
\begin{align} \label{eq:An-(1)}
& \xymatrix@C-5pt{A_{n}^{(1)}: & & & & \underset{n+1}{\circ} \ar  @{-}[rrrd] & & & \\
	& \underset{1}{\circ}\ar  @{-}[r] \ar @{-} @{-}[rrru]   & \underset{2}{\circ}\ar  @{-}[r] &\underset{3}{\circ}\ar@{.}[rr]& &  \underset{n-2}{\circ}\ar  @{-}[r] & \underset{n-1}{\circ} \ar  @{-}[r]  & \underset{n}{\circ}}
\\ \label{eq:A1-(1)}
& \xymatrix{A_1^{(1)}: && \underset{1}{\circ}  \ar  @{<=>}[r]  & \underset{2}{\circ}}
\\ \label{eq:A2-(2)}
& \xymatrix{A_2^{(2)}: && }\underset{1}{\circ}  \big\langle \begin{tabular}[ht]{c} \vspace{-10pt}\\ \hline\hline \vspace{-11pt}\\ \hline\hline\vspace{-10pt}
\end{tabular} \hspace{5pt} \underset{2}{\circ}
\\ \label{eq:Cn-(1)}
& \xymatrix@C-5pt{C_n^{(1)}: && \underset{1}{\circ}  \ar  @{=>}[r]  &  \underset{2}{\circ}\ar  @{-}[r]  &
	\underset{3}{\circ}\ar@{.}[r] &  \underset{n-1}{\circ}\ar  @{-}[r] & \underset{n}{\circ} \ar  @{<=}[r]  & \underset{n+1}{\circ}}
\\ \label{eq:A2-2n-1}
& \xymatrix@C-5pt{A_{2n-1}^{(2)}: & & & & & \circ \ar  @{-}[d] & \\	
	& \circ  \ar  @{=>}[r]  & \circ\ar  @{-}[r]  &  \circ\ar@{.}[r]&  \circ\ar  @{-}[r] & \circ \ar  @{-}[r]  & \circ}
\\ \label{eq:A2-2n}
&\xymatrix@C-5pt{A_{2n}^{(2)}: && \underset{1}{\circ}  \ar  @{<=}[r]  &  \underset{2}{\circ}\ar  @{-}[r]  & \underset{3}{\circ} \ar@{.}[r]&  \underset{n-1}{\circ}\ar  @{-}[r] & \underset{n}{\circ} \ar  @{<=}[r]  & \underset{n+1}{\circ}}
\\ \label{eq:D43}
&\xymatrix@C-5pt{D_{4}^{(3)}: && \underset{1}{\circ} &  \underset{2}{\circ}\ar  @{-}[l]  & \underset{3}{\circ} \ar  @3{->}[l]}
\\ \label{eq:E6(2)}
& \xymatrix@C-5pt{E_{6}^{(2)}: & & \underset{1}{\circ} \ar@{-}[r]  & \underset{2}{\circ} \ar@{-}[r]  & \underset{3}{\circ} \ar@{<=}[r] & \underset{4}{\circ} \ar@{-}[r] & \underset{5}{\circ} }
\\ \label{eq:F4(1)}
& \xymatrix@C-5pt{F_{4}^{(1)}: &&  \underset{1}{\circ} \ar  @{-}[r]  &
	\underset{2}{\circ} \ar  @{-}[r]  & \underset{3}{\circ}\ar  @{=>}[r] & \underset{4}{\circ} \ar  @{-}[r] & \underset{5}{\circ} }
\\ \label{eq:mTn}
&\xymatrix@C-5pt{{}_mT_n: &  && \overset{m+n+1}{\circ} \ar  @{-}[rd]  & \\
	\underset{1}{\circ}\ar  @{-}[r]  &  \underset{2}{\circ}\ar@{.}[r]&  
	\underset{m}{\circ} \ar @{-}[ru] \ar @{-}[rr] & &\underset{m +1}{\circ} \ar  @{-}[r]  
	&  \underset{m+2}{\circ}\ar@{.}[r]&  \underset{m + n}{\circ}}
\\ \label{eq:T2}
& \xymatrix@C-5pt{T^{(2)}: && & \underset{3}{\circ} \ar  @{-}[rd]   & \\ &&
	\underset{1}{\circ} \ar @{-}[ru] \ar @{<=}[rr] & &\underset{2}{\circ} }
\\ \label{eq:1T2}
& \xymatrix@C-5pt{{}_1T^{(2)}: && & & \underset{4}{\circ} \ar  @{-}[rd]   & \\
	&& \underset{1}{\circ} \ar  @{-}[r]  & \underset{2}{\circ} \ar @{-}[ru] \ar @{<=}[rr] & & \underset{3}{\circ} }
\\ \label{eq:1T2t}
& \xymatrix@C-5pt{{}_1\widetilde{T}^{(2)}: & && & \underset{4}{\circ} \ar  @{<=}[rd]   & \\
	&& \underset{1}{\circ} \ar  @{-}[r]  & \underset{2}{\circ} \ar @{-}[ru] \ar @{-}[rr] & & \underset{3}{\circ} }
\\ \label{eq:Cn1-indef}
& \xymatrix@C-5pt{ C_n^{(1)\ \wedge}: && \underset{1}{\circ} \ar  @{-}[r]  & \underset{2}{\circ}  \ar  @{=>}[r]  &  \underset{3}{\circ}\ar  @{-}[r]  &
	\underset{4}{\circ}\ar@{.}[r] &  \underset{n-1}{\circ}\ar  @{-}[r] & \underset{n}{\circ} \ar  @{<=}[r]  & \underset{n+1}{\circ}}
\\ \label{eq:CEn}
& \xymatrix@C-5pt{CE_n: && &  & & \underset{n}{\circ} \ar@{-}[d]  & \\ && \underset{1}{\circ}\ar  @{-}[r]  & \underset{2}{\circ}\ar  @{.}[r]  & \underset{n-3}{\circ} \ar  @{-}[r]  & \underset{n-2}{\circ} \ar@{<=}[r] & \underset{n-1}{\circ}}
\\ \label{eq:C2+++}
& \xymatrix@C-5pt{C_2^{+++}: &&  \underset{1}{\circ} \ar@{-}[r]  & \underset{2}{\circ} \ar@{-}[r]  &  \underset{3}{\circ} \ar@{=>}[r] & \underset{4}{\circ} \ar@{<=}[r] & \underset{5}{\circ} }
\\ \label{eq:F4(1)w}
& \xymatrix@C-5pt{F_{4}^{(1)\wedge}: &&  \underset{1}{\circ}\ar  @{-}[r]  & 	\underset{2}{\circ} \ar  @{-}[r]  & \underset{3}{\circ}\ar  @{-}[r] & 
	\underset{4}{\circ} \ar  @{=>}[r] & 
	\underset{5}{\circ} \ar  @{-}[r]  & \underset{6}{\circ}   & }
\\ \label{eq:E6(2)w}
& \xymatrix@C-5pt{E_{6}^{(2)\wedge}: &&  \underset{1}{\circ}\ar  @{-}[r]  & 	\underset{2}{\circ} \ar  @{-}[r]  & \underset{3}{\circ}\ar  @{-}[r] & 
	\underset{4}{\circ} \ar  @{<=}[r] & 
	\underset{5}{\circ} \ar  @{-}[r]  & \underset{6}{\circ}   & }
\\ \label{eq:D4(3)w}
& \xymatrix@C-5pt{D_{4}^{(3)\wedge}: && \underset{1}{\circ}  \ar@{-}[r]  &  \underset{2}{\circ} \ar  @{-}[r]  & \underset{3}{\circ} \ar@3{<-}[r] & \underset{4}{\circ} & & & }
\\ \label{eq:Hcd}
& \xymatrix@C-5pt{H_{c,d}: && \underset{1}{\circ} \ar  @{-}[r]^{c,d} & \underset{2}{\circ}, && cd\geq 4}
\end{align}
We also abbreviate ${}_mT = {}_mT_1$, $A_{2}^{(1)} = {}_1T_1$.

\subsection{Information}\label{subsec:outline-information}
In this Part, we give information on Nichols algebras $\toba_{\bq}$
for  matrices  $\bq$ satisfying \eqref{eq:diag-dif-1}
such that $\bq$ is arithmetic, see Definition \ref{def:arithmetic}, and has a connected Dynkin diagram.

We organize the information as follows.

\begin{itemize}[leftmargin=*]
\item  We first describe the (abstract) generalized root system $\cR$, including

\begin{itemize}
\medbreak\item The basic datum $(\cX, \rho)$.

\medbreak\item The bundles $(C^{x})_{x\in \cX}$ of Cartan matrices and $(\varDelta^{x})_{x\in \cX}$ of sets of roots.

\medbreak\item The Weyl groupoid, see Definition \ref{def:weylgpd}. Actually, since the basic datum is connected, the  groupoid  is determined by 
the isotropy group at any point; so we describe this last one--see \cite{AA-GRS-CLS-NA} for details of the calculations.

\medbreak\item The Lie algebra or superalgebra realizing the generalized root system as explained in \S \ref{subsec:Weyl-gpd-super}, when it exists.
\end{itemize}

\medbreak\item The possible families of  matrices $(\bq^x)_{x\in \cX}$ (actually the Dynkin diagrams) with the prescribed GRS. 
We call them the \emph{incarnations}. Concretely, we exhibit  families of  matrices  $(\bq^x)_{x\in \cX}$ such that

\medbreak\begin{enumerate}[leftmargin=*, label=\rm{(\alph*)}]
\medbreak\item\label{it:encarn-a} the Cartan matrix $C^{\bq^x}$ defined by \eqref{eq:defcij} equals $C^{x}$ for all $x\in \cX$.

\medbreak\item\label{it:encarn-b} The matrix $\rho_i(\bq^x)$ defined by \eqref{eq:rhoiq} equals $\bq^{\rho_i(x)}$ for all $i \in \I$, $x\in \cX$.

\end{enumerate}

By \ref{it:encarn-a} and \ref{it:encarn-b}, the Weyl groupoid of $\cR$ is isomorphic to the Weyl groupoid of $(\bq^x)_{x\in \cX}$.
It follows at once that $\cR$ and $(\bq^x)_{x\in \cX}$ have the same sets of real roots. But $\cR$ is finite, so all roots are real,
hence $(\bq^x)_{x\in \cX}$ has a finite set of real roots, and a fortiori it is finite.

\medbreak\item The PBW-basis, consequently the dimension or the GK-dimension. That is, we give the formulae
for  the root vectors as defined in \eqref{eq:x-raiz} in terms of braided commutators, see \S \ref{subsec:nichols-braided-commutators}.
Notice that the definition of the Lyndon words depends on the ordering of $\I$, which is in our context the order of the Dynkin diagram.
Furthermore this order happens to be convex.

\medbreak\item The defining relations.

\medbreak\item The set $\Oc^{\bq}$ of Cartan roots; notice that the concept of Cartan vertex depends on $\bq$, not just on the root system.

\medbreak\item The associated Lie algebra, see \S \ref{subsec:Lie-algebra} and the degree $\ya$ of $\toba_{\bq}^{\tp}$.

\end{itemize}

\begin{remark} \label{rem:incarnation-not-unique}
The same generalized root system could have different incarnations: of course, there is a dependence on the parameter $q$ but there could be more drastic differences.
For instance, the GRS $\superb{j}{\theta-j}$, $j\in\I_{\theta-1}$ has incarnations described in \S \ref{subsec:type-B-super} and another in \S \ref{subsec:type-B-standard}.
Notice that the Cartan roots and consequently the associated Lie algebra are different in these incarnations; thus some of the data above does not depend just on the generalized root system.
Besides this, it can be shown that all possible incarnations are as listed in the corresponding Subsection.
\end{remark}

\subsection{Organization}\label{subsec:outline-organization} This Part is organized as follows: 

\begin{itemize}[leftmargin=*]\renewcommand{\labelitemi}{$\circ$}
\item Section \ref{sec:by-diagram-cartan} contains the treatment of the matrices $\bq$ of Cartan type, 
Definition \ref{def:cartantype}, with Subsections devoted to each of the types A, B, \dots, G. 
Here the Weyl groupoid is just the Weyl group.

\item  In Section \ref{sec:by-diagram-super} we deal with the matrices $\bq$ of super type, meaning that the 
generalized root system coincides with that of a finite-dimensional contragredient Lie superalgebra in characteristic 0 (no \emph{a priori} characterization is available); Cartan type is excluded. Thus we have 
Subsections devoted to the types 
$\supera{m}{n}$, $\superb{m}{n}$, $\superd{m}{n}$, $\superda{\alpha}$, $\superf$, $\superg$.

\item   Section \ref{sec:by-diagram-standard} contains the treatment of the matrices of standard, but neither Cartan nor super, type. Recall that standard means that all Cartan matrices are equal. There are such diagrams only
in types B, Subsection \ref{subsec:type-B-standard}, and G, Subsection \ref{subsec:type-G-st}. 

\end{itemize}
In Part 3, we deal with Nichols algebras of:

\begin{itemize}
\item  matrices $\bq$ of modular type, meaning that the 
generalized root system coincides with that of a finite-dimensional contragredient Lie algebra in characteristic $>0$ (no \emph{a priori} characterization is available); Cartan type is excluded. Thus we have 
Subsections devoted to the types of the Lie algebras $\bgl(4,\alpha)$ (char 2), $\br(2, a)$ and $\br(3)$ (char 3); 

\item   matrices $\bq$ of super modular type in characteristic 3, not in the previous classes, meaning again 
coincidence with the generalized root system  of a finite-dimensional contragredient Lie superalgebra  (no \emph{a priori} characterization is available). There are 
Subsections devoted to the types of the Lie superalgebras 
$\brj(2;3)$, $\el(5;3)$, $\g(1,6)$, $\g(2,3)$, $\g(3,3)$, $\g(4,3)$, $\g(3,6)$, $\g(2,6)$, $\g(8,3)$, $\g(4,6)$, $\g(6,6)$, $\g(8,6)$;

\item   analogous to the preceding but in characteristic 5. 
There are Subsections on the types of the Lie superalgebras $\brj(2;5)$ and $\el(5;5)$;

\item  (yet) unidentified  generalized roots systems, i.e. that so far have not been recognized in other areas of Lie theory. These are called $\Ufo(1)$, \dots, $\Ufo(12)$, except that there is no $\Ufo(8)$. There is a 
Subsection for each of the corresponding Nichols algebras loosely called $\ufo(1)$, \dots, $\ufo(12)$--here $\ufo(8)$
has generalized root system $\Ufo (7)$. 
\end{itemize}

\subsection{Attribution} The presentations 
of the Nichols algebras that we describe here appeared already in the literature. 
A general approach to the relations was given in \cite{A-jems,A-presentation}.
Of course, those of Cartan type, giving the positive parts of 
the small quantum groups,
were discussed in many places, first of all in \cite{Lu-jams1990,Lu-dedicata,Lu}, for a parameter $q$ of odd order (and relatively prime to 3 if of type $G_2$). For Cartan type $A_{\theta}$, there are expositions from scratch in \cite{T-gln} for generic $q$, \cite{AD} for $q = -1$, \cite{AS3}  for $q \in \G'_N$, $N \ge 3$. Other Nichols algebras of rank 2 were presented in \cite{BDR,helbig,Hrk2}. Standard type appeared in \cite{A-standard}; this paper contains a self-contained proof of the defining relations of the Nichols algebras of Cartan type at a generic parameter, i.e. the sufficiency of the quantum Serre relations.
Nichols algebras associated to Lie superalgebras appeared first in the pioneering paper \cite{Y-super}; see also the exposition \cite{AAY}. The explicit relations of the remaining Nichols algebras were given in \cite{A-ufo}.

\subsection{Gelfand-Kirillov dimension}\label{sec:appendix-finiteGK}
As we said, the classification of the matrices $\bq$ such that $\GK \toba_{\bq} < \infty$ is not known.

\begin{conjecture}\label{conj:finiteGK} \cite{AAH}
If $\GK \toba_{\bq} < \infty$, then $\bq$ is arithmetic.
\end{conjecture}

The Conjecture is true when $\bq$ is of affine Cartan type or $\theta =2$.
For convenience, we collect the information on the arithmetic Nichols algebras 
with matrix $\bq$ such that $0 < \GK \toba_{\bq} < \infty$.

\subsubsection*{Cartan type} 
Let $\bq$ be of Cartan type with matrix $A$; we follow the conventions in \S \ref{sec:by-diagram-cartan}.
Then $0 < \GK \toba_{\bq}$ if and only if $q \notin \G_{\infty}$, in which case
$\GK \toba_{\bq} = \vert \varDelta_{+}^A \vert$.

\subsubsection*{Super type}
Let $\bq$ be of super type; we follow the conventions in \S \ref{sec:by-diagram-super}.
Then $0 < \GK \toba_{\bq}$ if and only if the following holds: 

\begin{itemize}[leftmargin=*]
\item Type $\supera{j-1}{\theta - j}$,  $j \in \I_{\lfloor\frac{\theta+1}{2} \rfloor}$, see \S \ref{subsubsec:type-A-super-PBW}:
$q \notin \G_{\infty}$, in which case 
\begin{align*}
\GK \toba_{\bq}= \binom{j}{2}+\binom{\theta-j}{2}.
\end{align*}

\item Type $\superb{j}{\theta-j}$, $j\in\I_{\theta-1}$, see \S \ref{subsubsec:type-B-super-PBW}: $q \notin \G_{\infty}$, in which case 
\begin{align*}
\GK \toba_{\bq}= \theta^2-2j(\theta-j-1).
\end{align*}

\item Type $\superd{j}{\theta-j}$, $j\in\I_{\theta - 1}$, see \S \ref{subsubsec:type-CD-super-PBW}: $q \notin \G_{\infty}$, in which case 
\begin{align*}
\GK \toba_{\bq}= (\theta-j)(\theta-j-1)+j^2.
\end{align*}

\item Type $\superda{\alpha}$, see \S \ref{subsec:type-D2-1-alpha}: Here $q,r,s\neq 1$, $qrs=1$; the condition is that either exactly 2 or all 3 of $q,r,s$ do not belong to $\G_{\infty}$, in which case $\GK \toba_{\bq}$ is either 2 or 3, accordingly.

\smallbreak
\item Type $\superf$, see \S \ref{subsubsec:type-F4-super-PBW}: $q \notin \G_{\infty}$, in which case $\GK \toba_{\bq} = 10$.

\smallbreak
\item Type $\superg$, see \S \ref{subsubsec:type-G3-super-PBW}: $q \notin \G_{\infty}$, in which case $\GK \toba_{\bq}=7$.
\end{itemize}

\subsubsection*{Modular type} Let $\bq$ be of modular type.
Then $0 < \GK \toba_{\bq}$ if and only if the following holds: 

\begin{itemize}[leftmargin=*]
	\item Type $\Bgl(4)$, see \S \ref{subsubsec:type-bgl4a-PBW}:
	$q \notin \G_{\infty}$, in which case  $\GK \toba_{\bq}= 6$.
	
	\item Type $\Brown(2)$, see \S \ref{subsubsec:type-br2a-PBW}: $q \notin \G_{\infty}$, in which case  $\GK \toba_{\bq}= 2$.

\end{itemize}

\section{Cartan type}\label{sec:by-diagram-cartan}

Here the basic datum has just one point, hence there is just one Cartan matrix and
the Weyl groupoid $\cW$ is  the corresponding Weyl group $W$.
All roots are of Cartan type.
Throughout, we shall use the notation
\begin{align}\label{eq:alfaij}
\alpha_{i j} &= \sum_{k \in \I_{i,j}} \alpha_k,&  i&\leq j \in \I.
\end{align}

\subsection{Type $A_{\theta}$, $\theta \ge 1$}\label{subsec:type-A}

\subsubsection{Root system}
The Cartan matrix is of type $A_{\theta}$, with the numbering determined by the Dynkin diagram, which is
\begin{align}\label{eq:dynkin-system-A}
\xymatrix@C-5pt{  \overset{\ }{\underset{1}{\circ}}\ar  @{-}[r]  &
\overset{\ }{\underset{2}{\circ}} \ar@{.}[r] & \overset{\ }{\underset{\theta - 1}{\circ}} \ar  @{-}[r]  & \overset{\ }{\underset{\theta}{\circ}}}.
\end{align}
The set of positive roots is
\begin{align}\label{eq:root-system-A}
\varDelta^+&=\{\alpha_{k\, j}\,|\, k, j \in \I,\, k\leq j\}.
\end{align}

\subsubsection{Weyl group}\label{subsubsec:type-A-Weyl} Let $s_i\in GL(\Z^\I)$, $s_i(\alpha_i) =  -\alpha_i$,
$s_i(\alpha_j) =  \alpha_j + \alpha_i$ if $\vert i-j\vert = 1$, $s_i(\alpha_j) =  \alpha_j$ 
if $\vert i-j\vert > 1$, $i,j \in \I$. Then $W = \langle s_i: i\in \I\rangle \simeq \s_{\theta + 1}$ \cite[Planche I]{Bourbaki}.

\subsubsection{Incarnation}
The generalized Dynkin
diagram is of the form
\begin{align}\label{eq:dynkin-type-A}
\xymatrix{ \overset{q}{\underset{\ }{\circ}}\ar  @{-}[r]^{q^{-1}}  &
\overset{q}{\underset{\ }{\circ}}\ar  @{-}[r]^{q^{-1}} &  \overset{q}{\underset{\
}{\circ}}\ar@{.}[r] & \overset{q}{\underset{\ }{\circ}} \ar  @{-}[r]^{q^{-1}}  &
\overset{q}{\underset{\ }{\circ}}}
\end{align}

\subsubsection{PBW-basis and (GK-)dimension}\label{subsubsec:type-A-PBW}
The root vectors are 
\begin{align*}
x_{\alpha_{ii}} &= x_{\alpha_{i}} = x_{i},& i \in \I, \\
x_{\alpha_{i j}} &= x_{(ij)} = [x_{i}, x_{\alpha_{i+1\, j}}]_c,& i <  j \in \I,
\end{align*}
cf. \eqref{eq:roots-Atheta}. Thus
\begin{align*}
\{ x_{ \theta}^{n_{\theta  \theta}} x_{(\theta-1  \theta)}^{n_{\theta-1  \theta}} x_{\theta-1}^{n_{\theta-1  \theta-1}} \dots x_{(1  \theta)}^{n_{1  \theta}} \dots x_{1}^{n_{11}} \, | \, 0\le n_{ij}<N \}
\end{align*}
is a PBW-basis of $\toba_{\bq}$. If $N<\infty$, then
\begin{align*}
\dim \toba_{\bq}= N^{\binom{\theta+1}{2}}.
\end{align*}
If $N=\infty$ (that is, if $q\notin \G_{\infty}$), then
\begin{align*}
\GK \toba_{\bq}= \binom{\theta+1}{2}.
\end{align*}

\subsubsection{Relations, $N > 2$}\label{subsubsec:type-A-N>2}
Recall the notations \eqref{eq:xij}, \eqref{eq:iterated},  \eqref{eq:roots-Atheta}. 
The Nichols algebra $\toba_{\bq}$ is generated by $(x_i)_{i\in \I}$ with defining
relations
\begin{align}\label{eq:rels-type-A-N>2}
x_{ij} &= 0, \quad i < j - 1; & x_{iij} &= 0, \quad \vert j - i\vert = 1;& x_{(kl)}^N &=0,
\quad k\leq l.
\end{align}

If $N = \infty$, i.e. $q\notin \G_{\infty}$, then we omit the last set of relations.

\subsubsection{Relations, $N = 2$}\label{subsubsec:type-A-N=2}
The Nichols algebra $\toba_{\bq}$ is generated by $(x_i)_{i\in \I}$ with defining
relations
\begin{align}\label{eq:rels-type-A-N=2}
x_{ij} &= 0, \quad i < j - 1; & [x_{(i\, i+2)}, x_{i+1}]_c &= 0; & x_{(kl)}^2 &=0, \quad
k\leq l.
\end{align}

\subsubsection{The associated Lie algebra and $\ya$}\label{subsubsec:type-A-Lie-alg} 
The first is of type $A_\theta$, while
\begin{align*}
\ya &=  (N-1) \sum_{i\in\I} i(\theta-i+1) \alpha_i.
\end{align*}

\subsection{Type $B_{\theta}$, $\theta \ge 2$}\label{subsec:type-B}
Here  $N > 2$.
\subsubsection{Root system}
The Cartan matrix is of type $B_{\theta}$, with the numbering determined by the Dynkin diagram, which is
\begin{align}\label{eq:dynkin-system-B}
\xymatrix@C-5pt{  \overset{\ }{\underset{1}{\circ}}\ar  @{-}[r]  &
\overset{\ }{\underset{2}{\circ}} \ar@{.}[r] & \overset{\ }{\underset{\theta - 1}{\circ}} \ar  @{=>}[r]  & \overset{\ }{\underset{\theta}{\circ}}}.
\end{align}

Recall the notation \eqref{eq:alfaij}.
The set of positive roots is
\begin{align}\label{eq:root-system-B}
\varDelta^+ &=\{\alpha_{ij}\,|\, i\leq j \in\I \} \cup 
\{\alpha_{i\theta} + \alpha_{j\theta}\,|\, i< j\in\I \}.
\end{align}

\subsubsection{Weyl group}\label{subsubsec:type-B-Weyl} 
Let $i \in \I$ and define $s_i\in GL(\Z^\I)$ by 
\begin{align*}
s_i(\alpha_j) &= \begin{cases} -\alpha_i, & i=j, \\
\alpha_j + \alpha_i, &\vert i-j\vert = 1, i <\theta, \\
\alpha_{\theta-1}+2\alpha_{\theta}, &j= \theta -1, i = \theta, \\
\alpha_j, & \vert i-j\vert > 1,
\end{cases}
\end{align*}
$j\in \I$. Then $W = \langle s_i: i\in \I\rangle \simeq (\Z/2)^{\theta}\rtimes \s_{\theta}$ \cite[Planche II]{Bourbaki}.

\subsubsection{Incarnation} The generalized Dynkin diagram is of the form
\begin{align}\label{eq:dynkin-type-B}
\xymatrix{ \overset{\,\,q^2}{\underset{\ }{\circ}}\ar  @{-}[r]^{q^{-2}}  &
\overset{\,\,q^2}{\underset{\ }{\circ}}\ar  @{-}[r]^{q^{-2}} &
\overset{\,\,q^2}{\underset{\ }{\circ}}\ar@{.}[r] & \overset{\,\,q^2}{\underset{\
}{\circ}} \ar  @{-}[r]^{q^{-2}}  & \overset{q}{\underset{\ }{\circ}}}
\end{align}

\subsubsection{PBW-basis and (GK-)dimension}\label{subsubsec:type-B-PBW}
The root vectors are 
\begin{align*}
x_{\alpha_{ii}} &= x_{\alpha_{i}} = x_{i},& i \in \I, \\
x_{\alpha_{ij}} &= x_{(ij)} = [x_{i}, x_{\alpha_{(i+1) j}}]_c,& i <  j \in \I, \\
x_{\alpha_{i\theta} + \alpha_{\theta}} &= [x_{\alpha_{i\theta}}, x_\theta]_c, & i  \in \I_{\theta - 1},
\\
x_{\alpha_{i\theta} + \alpha_{j\theta}} &= [x_{\alpha_{i\theta} + \alpha_{(j+1) \theta}}, x_j]_c, & i <  j \in \I_{\theta - 1},
\end{align*}
cf. \eqref{eq:roots-Atheta}. Let $M=\ord q^2$. Thus
\begin{multline*}
\{ x_{\theta  }^{n_{\theta  \theta}} 
x_{\alpha_{\theta-1\theta} + \alpha_{\theta\theta}}^{m_{\theta-1\theta}}
x_{\alpha_{\theta-1\theta}}^{n_{\theta-1  \theta}} 
x_{ \theta-1}^{n_{\theta-1  \theta-1}} 
\dots
x_{\alpha_{1\theta} + \alpha_{2\theta}}^{m_{12}}
\dots
x_{\alpha_{1\theta} + \alpha_{\theta\theta}}^{m_{1\theta}}
\dots
x_{\alpha_{1\theta}}^{n_{1  \theta}} 
\dots 
x_{1}^{n_{1 1}} \, \\ 
| \, 0\le n_{i\theta}<N; \, 0\le n_{ij}< M, \, j\neq \theta; 
\, 0\le m_{ij}<M\}
\end{multline*}
is a PBW-basis of $\toba_{\bq}$. If $N<\infty$, then
\begin{align*}
\dim \toba_{\bq}= M^{\theta(\theta-1)}N^{\theta}.
\end{align*}
If $N=\infty$ (that is, if $q$ is not a root of unity), then
\begin{align*}
\GK \toba_{\bq}= \theta^2.
\end{align*}

\subsubsection{Relations, $N>4$}\label{subsubsec:type-B-N>4}
The Nichols algebra $\toba_{\bq}$ is generated by $(x_i)_{i\in \I}$ with defining
relations
\begin{align}\label{eq:rels-type-B-N>4-qsr}
&x_{iii\pm1}=0, \quad i < \theta;  & x_{ij} &= 0, \quad i < j - 1; \quad 
x_{\theta\theta\theta\theta-1}=0;
\\ \label{eq:rels-type-B-N>4-even}
&\begin{aligned}
&x_{\alpha}^{N} =0,& \alpha&\in\{\alpha_{i\,\theta}\,|\,i\in\I\};
\\ & x_{\alpha}^{M} =0,& \alpha&\notin\{\alpha_{i\,\theta}\,|\,i\in\I\}, 
\end{aligned} & N &= 2M \text{ even.}
\\ \label{eq:rels-type-B-N>4-odd}
&x_{\alpha}^{N} =0, \quad\alpha \in\varDelta_{+},& N& \text{ odd.}
\end{align}

If $N = \infty$, i.e. $q\notin \G_{\infty}$, then we have only  the relations \eqref{eq:rels-type-B-N>4-qsr}.

\subsubsection{Relations, $N=4$}\label{subsubsec:type-B-N=4}
The Nichols algebra $\toba_{\bq}$ is generated by $(x_i)_{i\in \I}$ with defining
relations
\begin{align}\label{eq:rels-type-B-N=4}
\begin{aligned}
x_{ij} &= 0, \quad i < j - 1; \quad
x_{\theta\theta\theta\theta-1}=0; &
&[x_{(i\, i+2)}, x_{i+1}]_c=0, \quad i < \theta; 
\\
x_{\alpha}^2 &=0, \quad
\alpha\notin\{\alpha_{i\,\theta}\,|\,i\in\I\}; &
&x_{\alpha}^4 =0, \quad \alpha\in\{\alpha_{i\,\theta}\,|\,i\in\I\}. 
\end{aligned}
\end{align}

\subsubsection{Relations, $N=3$}\label{subsubsec:type-B-N=3}
The Nichols algebra $\toba_{\bq}$ is generated by $(x_i)_{i\in \I}$ with defining
relations
\begin{align}\label{eq:rels-type-B-N=3}
\begin{aligned}
&x_{ij} = 0, \quad i < j - 1; & x_{iii\pm1}&=0, \quad i < \theta;   \\
&[x_{\theta\theta\theta-1\theta-2},x_{\theta\theta-1}]_c=0; & x_{\alpha}^3 &=0, \quad \alpha\in\varDelta_{+}.
\end{aligned}\end{align}

\subsubsection{The associated Lie algebra and $\ya$}\label{subsubsec:type-B-Lie-alg} 
If $N$ is odd (respectively  even), the associated Lie algebra is of type $B_\theta$ 
(respectively  $C_\theta$), while
\begin{align*}
\ya &= \sum_{i\in\I} [(M-1)i(2\theta-i-1)+(N-1)(\theta-i)] \alpha_i.
\end{align*}

\subsection{Type $C_{\theta}$, $\theta \ge 3$}\label{subsec:type-C}
Here  $N > 2$.
\subsubsection{Root system}
The Cartan matrix is of type $C_{\theta}$, with the numbering determined by the Dynkin diagram, which is
\begin{align}\label{eq:dynkin-system-C}
\xymatrix@C-5pt{  \overset{\ }{\underset{1}{\circ}}\ar  @{-}[r]  &
\overset{\ }{\underset{2}{\circ}} \ar@{.}[r] & \overset{\ }{\underset{\theta - 1}{\circ}} \ar  @{<=}[r]  & \overset{\ }{\underset{\theta}{\circ}}}.
\end{align}

Recall the notation \eqref{eq:alfaij}.
The set of positive roots is
\begin{align}\label{eq:root-system-C}
\varDelta^+&=\{\alpha_{i\, j}\,|\, i\leq j\in\I \}\cup \{\alpha_{i\, \theta}+\alpha_{j\, \theta-1}\,|\, i\leq j\in\I_{\theta-1} \}.
\end{align}

\subsubsection{Weyl group}\label{subsubsec:type-C-Weyl} 
Let $i \in \I$ and define $s_i\in GL(\Z^\I)$ by 
\begin{align*}
s_i(\alpha_j) &= \begin{cases} -\alpha_i, & i=j, \\
\alpha_j + \alpha_i, &\vert i-j\vert = 1, j <\theta, \\
2\alpha_{\theta-1} + \alpha_{\theta}, &j= \theta, i = \theta -1, \\
\alpha_j, & \vert i-j\vert > 1,
\end{cases}
\end{align*}
$j\in \I$. Then $W = \langle s_i: i\in \I\rangle \simeq (\Z/2)^{\theta}\rtimes \s_{\theta}$ \cite[Planche III]{Bourbaki}.

\subsubsection{Incarnation} The generalized Dynkin diagram is of the form
\begin{align}\label{eq:dynkin-type-C}
\xymatrix{ \overset{q}{\underset{\ }{\circ}}\ar  @{-}[r]^{q^{-1}}  &
\overset{q}{\underset{\ }{\circ}}\ar  @{-}[r]^{q^{-1}} &  \overset{q}{\underset{\
}{\circ}}\ar@{.}[r] & \overset{q}{\underset{\ }{\circ}} \ar  @{-}[r]^{q^{-2}}  &
\overset{\,\,q^2}{\underset{\ }{\circ}}}
\end{align}

\subsubsection{PBW-basis and (GK-)dimension}\label{subsubsec:type-C-PBW}
The root vectors are 
\begin{align*}
x_{\alpha_{ii}} &= x_{\alpha_{i}} = x_{i},& i \in \I, \\
x_{\alpha_{ij}} &= x_{(ij)} = [x_{i}, x_{\alpha_{i+1\, j}}]_c,& i <  j \in \I, \\
x_{\alpha_{i\theta} + \alpha_{i\theta-1}} &= [x_{(i\theta)}, x_{(i\theta-1)}]_c,& i \in \I_{\theta-1}, \\
x_{\alpha_{i\theta} + \alpha_{\theta-1}} &= [x_{(i\theta)}, x_{\theta-1}]_c ,& i \in \I_{\theta-1}, \\
x_{\alpha_{i\theta} + \alpha_{j\theta-1}} &= [x_{\alpha_{i\theta} + \alpha_{j+1\theta-1}}, x_j]_c, & i <  j \in \I_{\theta-2},
\end{align*}
cf. \eqref{eq:roots-Atheta}. Let $M=\ord q^2$. Thus
\begin{multline*}
\{ x_{\theta}^{n_{\theta  \theta}} 
x_{\alpha_{\theta-1\theta} + \alpha_{\theta-1\theta-1}}^{m_{\theta-1\theta-1}}
x_{(\theta-1  \theta)}^{n_{\theta-1  \theta}} 
x_{\theta-1}^{n_{\theta-1  \theta-1}} 
\dots
x_{\alpha_{1\theta} + \alpha_{2\theta-1}}^{m_{12}}
\dots
x_{\alpha_{1\theta} + \alpha_{\theta-1\theta-1}}^{m_{1\theta-1}}
\dots
\\ 
x_{(1  \theta)}^{n_{1  \theta}} 
\dots 
x_{1}^{n_{1 1}}  |\,  0\le n_{i\theta}<M;  0\le n_{ij}<N,  j\neq \theta; 
\, 0\le m_{ij}<N \}
\end{multline*}
is a PBW-basis of $\toba_{\bq}$. If $N<\infty$, then
\begin{align*}
\dim \toba_{\bq}= M^{\theta}N^{\theta(\theta-1)}.
\end{align*}
If $N=\infty$ (that is, if $q$ is not a root of unity), then
\begin{align*}
\GK \toba_{\bq}= \theta^2.
\end{align*}

\subsubsection{Relations, $N>3$}\label{subsubsec:type-C-N>3}
The Nichols algebra $\toba_{\bq}$ is generated by $(x_i)_{i\in \I}$ with defining
relations
\begin{align}\label{eq:rels-type-C-N>3-qsr}
&x_{ij} = 0, \quad i < j - 1; &  x_{iij}&=0, \quad j=i\pm 1, (i,j)\neq(\theta-1,\theta);
\\ \label{eq:rels-type-C-N>3-qsr-bis}
&x_{iii\theta}=0, \ i = \theta - 1; &&
\\ \label{eq:rels-type-C-N>3-even}
&\begin{aligned}
&x_{\alpha}^{N}=0, &&
\alpha\in\varDelta_{+}\mbox{ short};  
\\& x_{\alpha}^M =0, && \alpha\in\varDelta_{+} \mbox{ long}.
\end{aligned} & N &= 2M \text{ even.}
\\ \label{eq:rels-type-C-N>3-odd}
&x_{\alpha}^{N} =0, \quad\alpha \in\varDelta_{+},& N& \text{ odd.}
\end{align}

If $N = \infty$, i.e. $q\notin \G_{\infty}$, then we have only  the relations \eqref{eq:rels-type-C-N>3-qsr}, \eqref{eq:rels-type-C-N>3-qsr-bis}.

\subsubsection{Relations, $N=3$}\label{subsubsec:type-C-N=3}

The Nichols algebra $\toba_{\bq}$ is generated by $(x_i)_{i\in \I}$ with defining
relations
\begin{align}\label{eq:rels-type-C-N=3}
\begin{aligned}
&x_{iij}=0, \quad j=i\pm 1, \, (i,j)\neq(\theta-1,\theta); & x_{ij} &= 0, \quad i < j - 1;
\\
&[[ x_{(\theta-2\theta)}, x_{\theta-1}]_c, x_{\theta-1}]_c=0; & x_{\alpha}^3&=0, \quad
\alpha\in\varDelta_{+}.
\end{aligned}\end{align}

\subsubsection{The associated Lie algebra and $\ya$}\label{subsubsec:type-C-Lie-alg} If $N$ is odd (respectively  even), 
the associated Lie algebra is of type $C_\theta$ (respectively  $B_\theta$), while
\begin{align*}
\ya &= \sum_{i\in\I} [(N-1)i(2\theta-i-1)+(M-1)(\theta-i)] \alpha_i.
\end{align*}

\subsection{Type $D_{\theta}$, $\theta \ge 4$}\label{subsec:type-D}
\subsubsection{Root system}
The Cartan matrix is of type $D_{\theta}$, with the numbering determined by the Dynkin diagram, which is
\begin{align}\label{eq:dynkin-system-D}
\xymatrix@C-5pt{  & & \overset{\ }{\underset{\theta}{\circ}} \ar  @{-}[d] & \\
\overset{\ }{\underset{1}{\circ}}\ar  @{-}[r]  &  \overset{\ }{\underset{2}{\circ}}\ar@{.}[r]&  \overset{\ }{\underset{\theta-2}{\circ}} \ar  @{-}[r]  & \overset{\ }{\underset{\theta-1}{\circ}}}.
\end{align}

Recall the notation \eqref{eq:alfaij}.
The set of positive roots is
\begin{align}\label{eq:root-system-D}
\begin{aligned}
\varDelta^+&=\{\alpha_{i\, j}\,|\, i\leq j\in\I, \, (i,j)\neq (\theta-1,\theta) \}
\\ 
& \qquad \cup \{\alpha_{i\, \theta-2}+\alpha_{\theta}\,|\, i\in\I_{\theta-2} \} \cup \{\alpha_{i\, \theta}+\alpha_{j\, \theta-2}\,|\, i<j\in\I_{\theta-2} \}.
\end{aligned}\end{align}

\subsubsection{Weyl group}\label{subsubsec:type-D-Weyl}
Let $i \in \I$ and define $s_i\in GL(\Z^\I)$ by 
\begin{align*}
s_i(\alpha_j) &= \begin{cases} -\alpha_i, & i=j, \\
\alpha_j + \alpha_i, &\vert i-j\vert = 1, i, j \in \I_{\theta - 1},\text{ or } \{i,j\}=\{\theta-2,\theta\},   \\
\alpha_j, & \text{otherwise,}
\end{cases}
\end{align*}
$j\in \I$. Then $W = \langle s_i: i\in \I\rangle \simeq (\Z/2)^{\theta-1}\rtimes \s_{\theta}$ \cite[Planche IV]{Bourbaki}.

\subsubsection{Incarnation}The generalized Dynkin diagram is of the form
\begin{align}\label{eq:dynkin-type-D}
\xymatrix{ & & & &  \overset{q}{\circ} &\\
\overset{q}{\underset{\ }{\circ}}\ar  @{-}[r]^{q^{-1}}  & \overset{q}{\underset{\
}{\circ}}\ar  @{-}[r]^{q^{-1}} &  \overset{q}{\underset{\ }{\circ}}\ar@{.}[r] &
\overset{q}{\underset{\ }{\circ}} \ar  @{-}[r]^{q^{-1}}  & \overset{q}{\underset{\
}{\circ}} \ar @<0.7ex> @{-}[u]_{q^{-1}}^{\qquad} \ar  @{-}[r]^{q^{-1}} &
\overset{q}{\underset{\ }{\circ}}}
\end{align}

\subsubsection{PBW-basis and (GK-)dimension}\label{subsubsec:type-D-PBW}
The root vectors are 
\begin{align*}
x_{\alpha_{ii}} &= x_{\alpha_{i}} = x_{i},& i \in \I, \\
x_{\alpha_{ij}} &= x_{(ij)} = [x_{i}, x_{\alpha_{i+1\, j}}]_c,& i <  j \in \I_{\theta-1}, \\
x_{\alpha_{i\theta-2}+\alpha_{\theta}} &= [x_{(i\theta-2)}, x_{\theta}]_c,& i \in \I_{\theta-2}, \\
x_{\alpha_{i\theta}} &= [x_{\alpha_{i\theta-2}+\alpha_{\theta}}, x_{\theta-1}]_c ,& i \in \I_{\theta-2}, \\
x_{\alpha_{i\theta} + \alpha_{j\theta-2}} &= [x_{\alpha_{i\theta} + \alpha_{j+1\theta-2}}, x_j]_c, & i <  j \in \I_{\theta-2},
\end{align*}
cf. \eqref{eq:roots-Atheta}. Thus
\begin{multline*}
\{ x_{\theta}^{n_{\theta  \theta}}
x_{\theta-1}^{n_{\theta-1  \theta-1}} 
x_{\alpha_{\theta-2}+\alpha_{\theta}}^{m_{\theta-2\theta}}
x_{(\theta-2\theta)}^{n_{\theta-2\theta}}
x_{(\theta-2  \theta-1)}^{n_{\theta-2  \theta-1}} 
x_{\theta-2}^{n_{\theta-2  \theta-2}} 
\dots
x_{\alpha_{1\theta} + \alpha_{2\theta-2}}^{m_{12}}
\dots
\\ 
x_{\alpha_{1\theta} + \alpha_{\theta-2\theta-2}}^{m_{1\theta-2}}
\dots
x_{(1  \theta)}^{n_{1  \theta}} 
\dots 
x_{1}^{n_{11}} \, | \, 0\le n_{ij}, \, m_{ij}<N \}
\end{multline*}
is a PBW-basis of $\toba_{\bq}$. If $N<\infty$, then
\begin{align*}
\dim \toba_{\bq}= N^{\theta(\theta-1)}.
\end{align*}
If $N=\infty$ (that is, if $q$ is not a root of unity), then
\begin{align*}
\GK \toba_{\bq}= \theta(\theta-1).
\end{align*}

\subsubsection{Relations, $N > 2$}\label{subsubsec:type-D-N>2}
The Nichols algebra $\toba_{\bq}$ is generated by $(x_i)_{i\in \I}$ with defining
relations
\begin{align}\label{eq:rels-type-D-N>2}
\begin{aligned}
x_{(\theta-1)\theta} &=0; & x_{ij} &= 0, \quad i < j - 1, (i,j)\neq (\theta-2,\theta); \\
x_{ii\theta}&=0, \ i= \theta-2; & x_{iij} &= 0, \quad \vert j - i\vert = 1, i,j\neq \theta;
\\
x_{ii(\theta-2)}&=0, \ i= \theta; & x_{\alpha}^{N} &=0,
\quad \alpha\in\varDelta_{+}.
\end{aligned}
\end{align}
If $N = \infty$, i.e. $q\notin \G_{\infty}$, then we omit the last set of relations.

\subsubsection{Relations, $N = 2$}\label{subsubsec:type-D-N=2}
 $\toba_{\bq}$ is presented by $(x_i)_{i\in \I}$ with defining
relations
\begin{align}\label{eq:rels-type-D-N=2}
\begin{aligned}
& x_{ij} = 0, \quad i < j - 1,\, (i,j)\neq (\theta-2,\theta);  &  &x_{(\theta-1)\theta} =0; \\
& [x_{(ii+2)},x_{i+1}]_c = 0, \quad i\leq \theta-3; &
&[x_{(\theta-3)(\theta-2)\theta},x_{\theta-2}]_c=0;
\\
& x_{\alpha}^{2}=0,
\quad \alpha\in\varDelta_{+}.
\end{aligned}
\end{align}

\subsubsection{The associated Lie algebra and $\ya$}\label{subsubsec:type-D-Lie-alg} 
The first is of type $D_\theta$, while
\begin{align*}
\ya &= (N-1) \big(\sum_{i\in \I_{\theta-2}} j(2\theta-j-1) \alpha_i + \frac{\theta(\theta-1)}{2} (\alpha_{\theta-1} + \alpha_{\theta}) \big).
\end{align*}

\subsection{Type $E_{\theta}$, $\theta \in \I_{6,8}$}\label{subsec:type-E}\subsubsection{Root system}
The Cartan matrix is of type $E_{\theta}$, with the numbering determined by the Dynkin diagram, which is

\begin{align}\label{eq:dynkin-system-E}
\xymatrix@C-5pt{  & & \overset{\ }{\underset{\theta}{\circ}} \ar  @{-}[d] & \\
\overset{\ }{\underset{1}{\circ}}\ar  @{-}[r]  &  \overset{\ }{\underset{2}{\circ}}\ar@{.}[r]&  \overset{\ }{\underset{\theta-3}{\circ}} \ar  @{-}[r]  & \overset{\ }{\underset{\theta-2}{\circ}} \ar  @{-}[r]  & \overset{\ }{\underset{\theta-1}{\circ}}}.
\end{align}

Recall the notation \eqref{eq:notation-root-exceptional}.
The  positive roots of $E_6$ are
\begin{align}\label{eq:root-system-E6} 
&{ \scriptsize\begin{aligned}
\Big\{ & 1, 12, 2, 123, 23, 3, 1234, 234, 34, 4, 12345, 2345, 345, 45, 5, 12^23^34^256, 12^23^24^256, \\
& 12^23^2456, 12^23^246, 123^24^256, 123^2456, 123^246, 23^24^256, 23^2456, 23^246, \\ 
& 12^23^34^256^2, 123456, 23456, 3456, 12346, 2346, 346, 1236, 236, 36, 6 \Big\}
\end{aligned}}
\\ \notag & =\{\beta_1, \dots, \beta_{36}\}.
\end{align}

The set of positive roots of $E_7$ is
\begin{align}\label{eq:root-system-E7}
&{ \scriptsize\begin{aligned}
\Big\{ &  1, 12, 2, 123, 23, 3, 1234, 234, 34, 4, 12345, 
2345, 345, 45, 5, 123456, 23456, 3456, \\
& 456, 56, 6, 12^23^34^35^267, 12^23^24^35^267, 12^23^24^25^267, 12^23^24^2567, 12^23^24^257, \\
& 123^24^35^267, 123^24^25^267, 123^24^2567, 123^24^257, 23^24^35^267, 23^24^25^267, 23^24^2567, \\ 
& 23^24^257, 12^23^34^45^36^27^2, 12^23^34^45^367^2, 1234^25^267, 234^25^267, 34^25^267, \\
& 12^23^34^45^267^2, 1234^2567, 234^2567, 34^2567, 12^23^34^35^267^2, 1234567, 234567, 34567, \\
& 12^23^24^35^267^2, 123^24^35^267^2, 1234^257, 123457, 12347, 23^24^35^267^2, 234^257, 23457, \\
& 2347, 4567, 34^257, 3457, 347, 457, 47, 7 \Big\}
\end{aligned}}
\\ \notag & =\{\beta_1, \dots, \beta_{63}\}.
\end{align}

The set of positive roots of $E_8$ is
\begin{align}
\label{eq:root-system-E8}
&{ \scriptsize\begin{aligned}
\Big\{ & 1, 12, 2, 123, 23, 3, 1234, 234, 34, 4, 12345, 2345, 345, 45, 5, 123456, 23456, 3456, \\
& 456, 56, 6, 1234567, 234567, 34567, 4567, 567, 67, 7, 12^23^34^35^36^278, 12^23^24^35^36^278, \\
& 12^23^24^25^36^278, 12^23^24^25^26^278, 12^23^24^25^2678, 12^23^24^25^268, 123^24^35^36^278, \\
& 123^24^25^36^278, 123^24^25^26^278, 123^24^25^2678, 123^24^25^268, 23^24^35^36^278, 23^24^25^36^278, \\
& 23^24^25^26^278, 23^24^25^2678, 23^24^25^268, 12^23^34^45^56^47^28^2, 12^23^34^45^56^37^28^2, \\
& 12^23^34^45^56^378^2, 1234^25^36^278, 234^25^36^278, 
34^25^36^278, 12^23^34^45^46^37^28^2, \\
& 12^23^34^45^46^378^2, 1234^25^26^278, 234^25^26^278, 
34^25^26^278, 12^23^34^35^46^37^28^2, \\
& 12^23^34^35^46^378^2, 12345^26^278, 2345^26^278, 345^26^278, 1^22^33^44^55^66^47^28^3, \\
& 12^33^44^55^66^47^28^3, 12^33^34^45^46^278^2, 12^23^34^35^46^278^2, 12^23^34^35^36^278^2, \\
& 12^23^24^35^46^37^28^2, 123^24^35^46^37^28^2, 1234^25^2678, 12345^2678, 12345678, \\
& 12^23^24^35^46^378^2, 123^24^35^46^378^2, 1234^25^268, 12345^268, 1234568, 12^23^34^55^66^47^28^3, \\
& 12^23^34^45^66^47^28^3, 12^23^34^45^56^47^28^3, 23^24^35^46^37^28^2, 23^24^35^46^378^2, 45^26^278, \\
& 12^23^24^35^46^278^2, 12^23^24^35^36^278^2, 234^25^2678, 234^25^268, 123^24^35^46^278^2, 
\\
& 123^24^35^36^278^2, 34^25^2678, 34^25^268, 12^23^34^45^56^37^28^3, 12^23^34^45^56^378^3, \\
& 12^23^24^25^36^278^2, 123^24^25^36^278^2, 1234^25^36^278^2, 123458, 23^24^35^46^278^2, 2345^2678, \\
& 345^2678, 45^2678, 23^24^35^36^278^2, 2345678, 345678, 45678, 23^24^25^36^278^2, \\
& 234^25^36^278^2, 2345^268, 234568, 23458, 34^25^36^278^2, 345^268, 34568, 3458, 5678, \\
& 45^268, 4568, 458, 568, 58, 8 \Big\} =\{\beta_1, \dots, \beta_{120}\}.
\end{aligned}}
\end{align}

For brevity, we introduce the notation
\begin{align*}
d_6 &= 36,& d_7 &= 63,& d_8 &= 120.
\end{align*}
Notice that the roots in \eqref{eq:root-system-E6}, respectively \eqref{eq:root-system-E7},
\eqref{eq:root-system-E8}, are ordered from left to right, justifying the notation $\beta_1, \dots, \beta_{d_\theta}$,
$\theta \in \I_{6,8}$.

\subsubsection{Weyl group}\label{subsubsec:type-E-Weyl} Let $i \in \I$ and define $s_i\in GL(\Z^\I)$ by 
\begin{align*}
s_i(\alpha_j) &= \begin{cases} -\alpha_i, & i=j, \\
\alpha_j + \alpha_i, &\vert i-j\vert = 1, i, j \in \I_{\theta - 1},\text{ or } \{i,j\}=\{\theta-3,\theta\},   \\
\alpha_j, & \text{otherwise,}
\end{cases}
\end{align*}
$j\in \I$. Then $W = \langle s_i: i\in \I\rangle$ \cite[Planches V-VII]{Bourbaki}.

\subsubsection{Incarnation}
The generalized Dynkin diagram is of the form
\begin{align}\label{eq:dynkin-type-E}
\xymatrix{ &  &   \overset{q}{\circ} &\\
\overset{q}{\underset{\ }{\circ}}\ar  @{-}[r]^{q^{-1}}  &  \overset{q}{\underset{\
}{\circ}}\ar@{.}[r]  & \overset{q}{\underset{\ }{\circ}} \ar @<0.7ex> @{-}[u]_{q^{-1}} \ar
@{-}[r]^{q^{-1}} &  \overset{q}{\underset{\ }{\circ}}  \ar  @{-}[r]^{q^{-1}} &
\overset{q}{\underset{\ }{\circ}}}
\end{align}

\subsubsection{PBW-basis and (GK-)dimension}\label{subsubsec:type-E-PBW}
The root vectors $x_{\beta_j}$ are explicitly described in \cite[pp. 63 ff]{A-standard}, see also \cite{LR}.
Thus a PBW-basis of $\toba_{\bq}$  is
\begin{align*}
\left\{ x_{\beta_{d_\theta}}^{n_{d_\theta}} x_{\beta_{d_\theta - 1}}^{n_{d_\theta - 1}} \dots x_{\beta_2}^{n_{2}}  x_{\beta_1}^{n_{1}} \, | \, 0\le n_{i}<N \right\}.
\end{align*}

If $N<\infty$, then $\dim \toba_{\bq}$ is $N^{{d_\theta}}$.
If $N=\infty$, then $\GK \toba_{\bq}$ is $d_\theta$.

\subsubsection{Relations, $N > 2$}\label{subsubsec:type-E-N>2}
 $\toba_{\bq}$ is generated by $(x_i)_{i\in \I}$ with relations
\begin{align}\label{eq:rels-type-E-N>2}
x_{ij} &= 0, \quad \widetilde q_{ij}=1; & x_{iij} &= 0, \quad \widetilde q_{ij}\neq 1;&
x_{\alpha}^N &=0, \quad \alpha\in\varDelta_{+}.
\end{align}

If $N = \infty$, i.e. $q\notin \G_{\infty}$, then we omit the last set of relations.

\subsubsection{Relations, $N = 2$}\label{subsubsec:type-E-N=2}
 $\toba_{\bq}$ is generated by $(x_i)_{i\in \I}$ with 
relations
\begin{align}\label{eq:rels-type-E-N=2}
x_{ij} &= 0, \quad \widetilde q_{ij}=1; & [x_{ijk},x_{j}]_c &= 0, \quad \widetilde
q_{ij}, \widetilde q_{jk}\neq 1;&
x_{\alpha}^2 &=0, \, \alpha\in\varDelta_{+}.
\end{align}

\subsubsection{The associated Lie algebras and $\ya$}\label{subsubsec:type-E-Lie-alg} 
Those are of type $E_\theta$,  while
\begin{align*}
&E_6:&
\ya &= (N-1)(16 \alpha_1 + 30\alpha_2 + 42\alpha_3 + 30 \alpha_4 + 16\alpha_5 + 22\alpha_6), 
\\
&E_7:&
\ya &= (N-1)(27\alpha_1 + 52\alpha_2 + 75 \alpha_3 + 96\alpha_4 + 66\alpha_5 + 34\alpha_6 + 49\alpha_7), 
\\
&E_8:&
\ya &= (N-1)(58\alpha_1 + 114\alpha_2 + 168\alpha_3 + 220\alpha_4 + 270\alpha_5 
\\
& & & \qquad \qquad + 182\alpha_6 + 92\alpha_7 + 136\alpha_8).
\end{align*}

\subsection{Type $F_{4}$}\label{subsec:type-F}
Here  $N > 2$.
\subsubsection{Root system}

The Cartan matrix is of type $F_{4}$, with the numbering determined by the Dynkin diagram, which is
\begin{align}\label{eq:dynkin-system-F}
\xymatrix@C-5pt{\overset{\ }{\underset{1}{\circ}}\ar  @{-}[r]  &  \overset{\ }{\underset{2}{\circ}}\ar@{<=}[r]& \overset{\ }{\underset{3}{\circ}} \ar  @{-}[r]  & \overset{\ }{\underset{4}{\circ}}}.
\end{align}

The set of positive roots is
\begin{align} \label{eq:root-system-F}
\begin{aligned}
\varDelta^+&=\{ 1, 12, 2, 1^22^23, 12^23, 123, 2^23, 23, 3, 1^22^43^34, 1^22^43^24,
\\ 
& \qquad  1^22^33^24, 1^22^23^24, 1^22^234, 12^33^24, 12^23^24, 1^22^43^34^2, 
\\
& \qquad 12^234, 1234, 2^23^24, 2^234, 234, 34, 4 \} =\{\beta_1, \dots, \beta_{24}\}.
\end{aligned}\end{align}

\subsubsection{Weyl group}\label{subsubsec:type-F-Weyl} 
Let $i \in \I$ and define $s_i\in GL(\Z^\I)$ by 
\begin{align*}
s_i(\alpha_j) &= \begin{cases} -\alpha_i, & i=j, \\
\alpha_j + \alpha_i, & \vert i-j\vert = 1, \, (i,j)\neq (2,3)\\ \alpha_3 + 2\alpha_2, & (i,j)=(2,3),   \\
\alpha_j, & \vert i-j\vert > 1,
\end{cases}
\end{align*}
$j\in \I$. Then $W = \langle s_i: i\in \I\rangle \simeq \big( (\Z/2)^{3}\rtimes \s_{4} \big) \rtimes \s_3$ \cite[Planche VIII]{Bourbaki}.

\subsubsection{Incarnation}
The generalized Dynkin diagram is of the form
\begin{align}\label{eq:dynkin-type-F}
\xymatrix{ \overset{\,\,q}{\underset{\ }{\circ}}\ar  @{-}[r]^{q^{-1}}  &
\overset{\,\,q}{\underset{\ }{\circ}}\ar  @{-}[r]^{q^{-2}} &   \overset{q^2}{\underset{\
}{\circ}} \ar  @{-}[r]^{q^{-2}}  &  \overset{q^2}{\underset{\ }{\circ}} }
\end{align}

\subsubsection{PBW-basis and (GK-)dimension}\label{subsubsec:type-F-PBW}
Let $M=\ord q^2$. The root vectors $x_{\beta_j}$ are explicitly described in \cite[pp. 65 ff]{A-standard}, see also \cite{LR}.
Thus a PBW-basis of $\toba_{\bq}$ for type $F_4$ is
\begin{align*}
\left\{ x_{\beta_{24}}^{n_{24}} x_{\beta_{23}}^{n_{23}} \dots  x_{\beta_1}^{n_{1}} \, | \, 0\le n_{j}<N 
\text{ if }\beta_j\mbox{ is short}; \, 0\le n_{j}<M 
\text{ if }\beta_j\mbox{ is long} \right\}.
\end{align*}

If $N<\infty$, then $\dim \toba_{\bq}=M^{12}N^{12}$.
If $N=\infty$, then $\GK \toba_{\bq}=24$.

\subsubsection{Relations, $N>4$}\label{subsubsec:type-F-N>4}
The Nichols algebra $\toba_{\bq}$ is generated by $(x_i)_{i\in \I_4}$ with defining
relations
\begin{align}\label{eq:rels-type-F-N>4-qsr}
&x_{ij} = 0, \, i < j - 1; \, x_{2223}=0; \ x_{iij}=0, \, |j-i|=1,&  &(i,j)\neq (2,3); 
\\ \label{eq:rels-type-F-N>4-even}
&\begin{aligned}
&x_{\alpha}^{N} =0, &&\alpha\in\varDelta_{+} \mbox{ short};\\
& x_{\alpha}^M =0, && \alpha\in\varDelta_{+} \mbox{ long}.
\end{aligned}  & &N = 2M \text{ even},
\\ \label{eq:rels-type-F-N>4-odd}
 &x_{\alpha}^{N} =0, \quad\alpha \in\varDelta_{+},& &N \text{ odd.}
\end{align}

If $N = \infty$, i.e. $q\notin \G_{\infty}$, then we have only  the relations \eqref{eq:rels-type-F-N>4-qsr}.

\subsubsection{Relations, $N=4$}\label{subsubsec:type-F-N=4}
The Nichols algebra $\toba_{\bq}$ is generated by $(x_i)_{i\in \I_4}$ with defining
relations
\begin{align}\label{eq:rels-type-F-N=4}
\begin{aligned}
&[x_{(24)}, x_{3}]_c=0; & x_{221}&=0; \quad x_{112}=0;  
\\
&x_{ij}= 0, \quad i < j - 1; & x_{2223}&=0;
\\
&x_{\alpha}^4 =0, \quad \alpha\in\varDelta_{+} \mbox{ short};
& x_{\alpha}^2 &=0, \quad \alpha\in\varDelta_{+} \mbox{ long}.
\end{aligned}
\end{align}

\subsubsection{Relations, $N=3$}\label{subsubsec:type-F-N=3}
The Nichols algebra $\toba_{\bq}$ is generated by $(x_i)_{i\in \I_4}$ with defining
relations
\begin{align}\label{eq:rels-type-F-N=3}
\begin{aligned}
x_{ij} &= 0, \quad i < j - 1; &  &[x_{2234},x_{23}]_c=0; \\
x_{iij}&=0, \quad j=i\pm 1, (i,j)\neq (2,3);
& & x_{\alpha}^3=0, \quad \alpha\in\varDelta_{+}.
\end{aligned}
\end{align}

\subsubsection{The associated Lie algebra and $\ya$}\label{subsubsec:type-F-Lie-alg} 
The first is of type $F_4$, while
\begin{multline*}
\ya = (12M+10N-22)\alpha_1 + (24M+18N-42) \alpha_2\\ + (18M+12N-30)\alpha_3 + (10M+6N-16) \alpha_4.
\end{multline*}

\subsection{Type  $G_{2}$}\label{subsec:type-G}
Here  $N > 3$.
\subsubsection{Root system}
The Cartan matrix is of type $G_{2}$, with the numbering determined by the Dynkin diagram, which is
\begin{align}\label{eq:dynkin-system-G}
\xymatrix@C-5pt{\overset{\ }{\underset{1}{\circ}}\ar  @3{<-}[r] &  \overset{\ }{\underset{2}{\circ}}}.
\end{align}

The set of positive roots is
\begin{align}\label{eq:root-system-G}
\varDelta_+&=\{\alpha_1, 3\alpha_1+\alpha_2, 2\alpha_1+\alpha_2, 3\alpha_1+2\alpha_2, \alpha_1+\alpha_2, \alpha_2 \}.
\end{align}

\subsubsection{Weyl group}\label{subsubsec:type-G-Weyl} Let $s_i\in GL(\Z^\I)$, $s_i(\alpha_i) =  -\alpha_i$,
$s_1(\alpha_2) =  \alpha_2 + 3\alpha_1$, $s_{2}(\alpha_1) =  \alpha_1 + \alpha_2$. Then $W = \langle s_i: i\in \I\rangle$ is the dihedral group of order 12 \cite[Planche IX]{Bourbaki}.

\subsubsection{Incarnation}The generalized Dynkin diagram is of the form
\begin{align}\label{eq:dynkin-type-G}
\xymatrix{  \overset{\,\,q}{\underset{\ }{\circ}} \ar  @{-}[r]^{q^{-3}} &
\overset{q^3}{\underset{\ }{\circ}}}
\end{align}

\subsubsection{PBW-basis and (GK-)dimension}\label{subsubsec:type-G-PBW}
The root vectors are 
\begin{align*}
x_{\alpha_{i}} &= x_{i}, &
x_{m\alpha_{1}+\alpha_2} &= \ad_c x_1^m (x_2) =x_{1 \dots 1 2}, & 
x_{3\alpha_{1}+2\alpha_2} = [x_{112}, x_{12}]_c,
\end{align*}
cf. \eqref{eq:roots-Atheta} and \eqref{eq:not-reducida-raiz}. 
Let $M=\ord q^3$. Thus
\begin{align*}
\{ x_{2}^{n_1} x_{12}^{n_2} x_{3\alpha_1+2\alpha_2}^{n_{3}}x_{112}^{n_4} x_{1112}^{n_5} x_1^{n_6} \, \,|\, \, 0\le n_{11}, n_{3}, n_{5}, <M, \, 0\le n_{2}, n_{4}, n_{6}<N\}.
\end{align*}
is a PBW-basis of $\toba_{\bq}$. If $N<\infty$, then
\begin{align*}
\dim \toba_{\bq}= M^{3}N^{3}.
\end{align*}
If $N=\infty$ (that is, if $q\notin \G_{\infty}$), then
$\GK \toba_{\bq}= 6$.

\subsubsection{Relations, $N>4$}\label{subsubsec:type-G-N>4}
The Nichols algebra $\toba_{\bq}$ is generated by $(x_i)_{i\in \I_2}$ with defining
relations
\begin{align}\label{eq:rels-type-G-N>4-qsr}
&x_{11112} = 0;& &x_{221}=0; 
\\ \label{eq:rels-type-G-N>4-div3}
&\begin{aligned}
& x_1^{N}=0; & & x_{112}^{N}=0; & & x_{12}^{N}=0; \\
& x_{1112}^{M}=0; & & x_{3\alpha_1+2\alpha_2}^{M}=0; & & x_2^{M}=0;
\end{aligned}  & &N = 3M \in 3\Z,
\\ \label{eq:rels-type-G-N>4-notdiv3}
&x_{\alpha}^{N} =0, \quad\alpha \in\varDelta_{+},& &N\notin 3\Z.
\end{align}

If $N = \infty$, i.e. $q\notin \G_{\infty}$, then we have only  the relations \eqref{eq:rels-type-G-N>4-qsr}.

\subsubsection{Relations, $N=4$}\label{subsubsec:type-G-N=4}
The Nichols algebra $\toba_{\bq}$ is generated by $(x_i)_{i\in \I_2}$ with defining
relations
\begin{align}\label{eq:rels-type-G}
[x_{3\alpha_1+2\alpha_2}, x_{12}]_c &=0; & x_{221} &= 0; & x_{\alpha}^4&=0, \quad
\alpha\in\varDelta_{+}.
\end{align}

\subsubsection{The associated Lie algebra and $\ya$}\label{subsubsec:type-G-Lie-alg} The first is of type $G_2$,
while
\begin{align*}
\ya &= (6M+4N-10)\alpha_1 + (4M+2N-6)\alpha_2.
\end{align*}

\section{Super type}\label{sec:by-diagram-super}

In this Section we consider the matrices $\bq$ of super type, i.e. with the same
generalized root system as that of a finite-dimensional contragredient Lie superalgebra (not a Lie algebra) in characteristic 0.

\medbreak
We start by a useful notation. As always $\theta \in \N$ and $\I = \I_{\theta}$.
Let $\qt\in\Bbbk^\times - \G_2$ and let $\Jb\subset\I$.
Let ${\bf A}_{\theta}(\qt;\Jb)$ be the generalized Dynkin diagram 
\begin{align*}
\xymatrix{ \overset{q_{11}}{\underset{\ }{\circ}}\ar  @{-}[rr]^{\widetilde{q}_{12}}  &&
\overset{q_{22}}{\underset{\ }{\circ}}\ar@{.}[r] &  \overset{\quad q_{\theta-1 \theta-1}}{\underset{\ }{\circ}} \ar  @{-}[rr]^{\widetilde{q}_{\theta-1 \theta}}  &&
\overset{\quad q_{\theta\theta}}{\underset{\ }{\circ}}}
\end{align*}
where the scalars satisfy the following requirements:
\begin{enumerate} [leftmargin=*]
\item\label{it:aqj-1} $\qt = q_{\theta\theta}^2\widetilde{q}_{\theta-1 \theta}$;

\smallbreak\item\label{it:aqj-2} if $i\in\Jb$, then $q_{ii}=-1$ and $\widetilde{q}_{i-1 i}=\widetilde{q}_{i i+1}^{-1}$;

\smallbreak\item\label{it:aqj-4} if $i\notin\Jb$, then $\widetilde{q}_{i-1 i}= q_{ii}^{-1} = \widetilde{q}_{i i+1}$ (only the second equality if $i=1$, only the first if $i= \theta$).
\end{enumerate}

This is a variation of an analogous notation in \cite{H-classif RS}.
We notice that the diagram ${\bf A}_{\theta}(\qt;\Jb)$ is determined by $\qt$ and $\Jb$, that $q_{ii} = \qt^{\pm 1}$ if $i\notin\Jb$,
and that $\widetilde{q}_{i i+1} = \qt^{\pm 1}$
for all $i < \theta$:

\begin{itemize}
\item If $\theta\in\Jb$, then $q_{\theta\theta} \overset{\eqref{it:aqj-2}}{=} -1$, hence $\widetilde{q}_{\theta-1 \theta} \overset{\eqref{it:aqj-1}}{=} \qt$. 

\item If $\theta\notin\Jb$, then $\widetilde{q}_{\theta-1 \theta} \overset{\eqref{it:aqj-4}}{=} q_{\theta\theta}^{-1}$;
hence $ q_{\theta\theta} \overset{\eqref{it:aqj-1}}{=} \qt$ and $\widetilde{q}_{\theta-1 \theta} = \qt^{-1}$.

\item Let $j \in \I$, $j < \theta$. Suppose we have determined $q_{ii}$ and $\widetilde{q}_{i-1 i}$ for all $i > j$.
Then \eqref{it:aqj-2}  and \eqref{it:aqj-4} determine $q_{jj}$ and $\widetilde{q}_{j-1 j}$.
\end{itemize}

\smallbreak
If $\Jb=\{j\}$, then ${\bf A}_{\theta}(\qt;\Jb)$ is
\begin{align*}
&\xymatrix{ \overset{\qt^{-1}}{\underset{\ }{\circ}}\ar  @{-}[r]^{\qt}  &
\overset{\qt^{-1}}{\underset{\ }{\circ}}\ar@{.}[r] &  \underset{j}{\overset{{-1}}{\circ}}
\ar@{.}[r] & \overset{\qt}{\underset{\ }{\circ}} \ar  @{-}[r]^{\qt^{-1}}  &
\overset{\qt}{\underset{\ }{\circ}}}
\end{align*}

Below we shall specialize $\qt$ to our parameter $q$ or variations thereof.
Also, the symbol $\xymatrix{ {\bf A}_{\theta}(\qt;\Jb) \ar  @{-}[r]^{\qquad p}  &
\overset{r}{\underset{\ }{\circ}}}$ means   a diagram with $\theta + 1$ points; the first $\theta$ of them span ${\bf A}_{\theta}(\qt;\Jb)$,
there is an edge labelled $p$ between the $\theta$  and the $\theta + 1$ points, the last labelled by $r$. 
Symbols like this appear here and there.

\smallbreak
Given $\Jb\subset\I$, $\Jb = \{i_1,\ldots,i_k\}$ with $i_1<\ldots<i_k$, we shall need
\begin{align*}
\SJ_{\Jb}:=\left|\sum_{l=1}^k (-1)^l i_l\right|.
\end{align*}

Notice that  $\Jb = \emptyset$, if and only if $\SJ_{\Jb} = 0$,  if and only if   ${\bf A}_{\theta}(\qt;\Jb)$ is of Cartan type $A_{\theta}$ (because $\qt\notin \G_2$). 

\subsection{Type $\supera{j-1}{\theta - j}$,  $j \in \I_{\lfloor\frac{\theta+1}{2} \rfloor}$}\label{subsec:type-A-super}
Here $N >2$.
We first define
\begin{align}\label{eq:basic-datum-A-super-att}
\att_{\theta, j} &= \left\{\Jb\subseteq \I : \SJ_{\Jb} = j  \right\}.
\end{align}
Observe that for $k \in \I$,  $\{k\} \in \att_{\theta, j}$ if and only if $k = j$.

\subsubsection{Basic datum, $1 \leq j < \frac{\theta+1}{2}$}

The basic datum is $(\Att_{\theta, j}, \rho)$, where 
\begin{align}\label{eq:basic-datum-A-super}
\Att_{\theta, j} &= \att_{\theta, j} \mathbin{\dot{\cup}} \att_{\theta, \theta +1 - j}
\end{align}
and
$\rho: \I \to \s_{\Att_{\theta, j}}$ is  as follows. If $i \in \I$, then $\rho_i:\Att_{\theta, j}\to\Att_{\theta, j}$ is given by
\begin{align}\label{eq:rho-A-super}
\rho_i(\Jb):=\left\{ \begin{array}{ll} \Jb, & i\notin\Jb, \\  \Jb\cup\{i-1,i+1\}, & i\in\Jb, i-1,i+1\notin\Jb, \\
(\Jb - \{i\mp1\})\cup\{i\pm1\}, & i,i\mp1\in\Jb, i\pm1\notin\Jb, \\  \Jb - \{i\pm1\}, & i,i-1,i+1\in\Jb. \end{array} \right.
\end{align} 
If $i = 1$, respectively $\theta$, then $i-1$, respectively $i+1$, is omitted in the definition above.
It is not difficult to see that 

\begin{itemize}
\item $\rho_i$ is well-defined and $\rho_i^2 = \id$. Hence $(\Att_{\theta, j},\rho)$ is a  basic datum. 

\item Let $\Jb \in \Att_{\theta, j}$. Then there exists $k \in \I$ such that $\Jb \sim \{k\}$. Hence $(\Att_{\theta, j},\rho)$ is connected. 
\end{itemize}

Indeed, we see that $\rho_i (\att_{\theta, j}) = \att_{\theta, j}$, $\rho_i (\att_{\theta, \theta +1 - j}) =  \att_{\theta, \theta +1 - j}$
if $i <\theta$, but 
\begin{align*}
\rho_\theta(\Jb) \in \att_{\theta, \theta +1 - j} \text{ if } \Jb \in \att_{\theta, j}, \theta \in \Jb. 
\end{align*}

\subsubsection{Basic datum, $\theta$ odd, $j= \frac{\theta+1}{2}$}

Here $j = \theta+1-j$ and  we need to consider two copies of $\att_{\theta, j}$, see \eqref{eq:basic-datum-A-super-att}.
Let $\oatt_{\theta, j} = \left\{\overline{\Jb}: \Jb\in \att_{\theta, j}  \right\}$ be a disjoint copy of $\att_{\theta, j}$.
Then the basic datum is $(\Att_{\theta, j}, \rho)$, where
\begin{align}\label{eq:basic-datum-A-super-bis}
\Att_{\theta, j} &= \att_{\theta, j} \mathbin{\dot{\cup}} \overline{\att}_{\theta, j}.
\end{align}
and
$\rho: \I \to \s_{\Att_{\theta, j}}$ defined as follows. If $i \in \I$, $i < \theta$, then $\rho_i:\att_{\theta, j}\to\att_{\theta, j}$
and $\rho_i:\oatt_{\theta, j}\to\oatt_{\theta, j}$ are  given by \eqref{eq:rho-A-super}. If $i = \theta$ and $\Jb \in \att_{\theta, j}$, then 
\begin{align}\label{eq:rho-A-super-bis}
\begin{tabular}{c | c c}
\noalign{\smallskip} $\rho_\theta$ &   $\Jb$  & $\overline{\Jb}$       \\ \hline
$\theta\notin\Jb$ & $\Jb$  & $\overline{\Jb}$  \\
$\theta\in\Jb$, $\theta-1\notin\Jb$ &  $\overline{\Jb\cup\{\theta-1\}}$ & $ \Jb\cup\{\theta-1\}$ \\
$\theta\in\Jb$, $\theta-1\in\Jb$ &  $\overline{\Jb-\{\theta-1\}}$ & $ \Jb - \{\theta-1\}$ 
\end{tabular}
\end{align}
It is not difficult to see that 

\begin{itemize}

\item $\rho_i$ is well-defined and $\rho_i^2 = \id$. Hence $(\Att_{\theta, j},\rho)$ is a  basic datum. 

\item $(\Att_{\theta, j},\rho)$ is connected. 
\end{itemize}

\subsubsection{Root system}
The bundle of Cartan matrices $(C^{\Jb})_{\Jb \in\Att_{\theta, j}}$  is constant: $C^{\Jb}$ is the Cartan matrix of type $A_\theta$ as in \eqref{eq:dynkin-system-A}
for any $\Jb \in\Att_{\theta, j}$.

The bundle of sets of roots $(\varDelta^{\Jb})_{\Jb \in\Att_{\theta, j}}$  is constant: 
\begin{align*}
\varDelta^{\Jb}:= \{\pm \alpha_{i,j}: i,j\in\I, i\leq j \}.
\end{align*} 

Hence the root system is standard, with  positive roots \eqref{eq:root-system-A}.
Notice that this is the generalized root system of type $\supera{j-1}{\theta - j}$, i.~e. of $\sli(j \vert\theta - j)$.

\subsubsection{Weyl groupoid}\label{subsubsec:type-superA-Weyl} 
The isotropy group  at $\{j\} \in \Att_{\theta, j}$ is 
\begin{align*}
\cW(\{j\})= \langle s_i: i \in \I, i\neq j\rangle \simeq \s_{j}\times \s_{\theta-j + 1} \leq  GL(\Z^\I).
\end{align*}

\subsubsection{Lie superalgebras realizing this generalized root system}\label{subsubsec:super-structure}

\

Let $\pa_{\Jb}:\Z^{\I}\to\G_2$ be the group homomorphism such that 
\begin{align}\label{eq:superstructure}
\pa_{\Jb}(\alpha_i)=-1  \iff i\in\Jb.
\end{align}
We say that $\alpha\in\Z^I$ is even, respectively odd, if $\pa_{\Jb}(\alpha)=1$, respectively $\pa_{\Jb}(\alpha)=-1$. Thus
$\Jb$ is just the set of simple odd roots. 

To describe the incarnation in the setting of Lie superalgebras, we need the matrices ${\bf A}_{\theta}(\Jb)= (a_{ij}^{\Jb})_{i,j\in\I} \in\ku^{\I\times\I}$, $\Jb \in\Att_{\theta, j}$, where for $i,j\in\I$,
\begin{align*}
a_{ii}^{\Jb}&=\left\{ \begin{array}{ll} 2, & i\notin\Jb, \\ 0, & i\in\Jb, \end{array} \right. &
a_{ij}^{\Jb}&=\left\{ \begin{array}{ll} -1, & i\notin\Jb, j=i\pm1 \\ \pm1, & i\in\Jb,j=i\pm1, \\ 0, & |i-j|\geq2.
\end{array} \right.
\end{align*}
The assignment
\begin{align}\label{eq:incarnation-A-super-Lie}
\Jb \longmapsto & \big({\bf A}_{\theta}(\Jb),\pa_{\Jb}\big)
\end{align}
provides an isomorphism of generalized root systems, cf. \S \ref{subsec:Weyl-gpd-super}.

\subsubsection{Incarnation, $1 \leq j < \frac{\theta+1}{2}$}
The assignment
\begin{align}\label{eq:dynkin-A-super}
\Jb \longmapsto & \left\{\begin{aligned}
&{\bf A}_{\theta}(q;\Jb),& &\text{ if } \SJ_{\Jb} =j; \\
&{\bf A}_{\theta}(q^{-1};\Jb),&  &\text{ if } \SJ_{\Jb} = \theta + 1 - j.
\end{aligned}\right.
\end{align}
gives an incarnation. Indeed, let $\bq$ be the matrix corresponding to ${\bf A}_{\theta}(q;\{j\})$.
\begin{itemize}[leftmargin=*]\renewcommand{\labelitemi}{$\circ$}

\item First, the map \eqref{eq:dynkin-A-super}  $\Att_{\theta, j} \to \cX_{\bq}$ is bijective. 
By the definition \eqref{eq:rhoiq},  $\rho_i({\bf A}_{\theta}(q;\Jb))$ equals ${\bf A}_{\theta}(q^{\pm 1};\rho_i(\Jb))$, depending on $\SJ_{\Jb}$, 
cf. \eqref{eq:rho-A-super},  for all $i\in\I$ and $\Jb\in\Att_{\theta, j}$. 
Thus the basic data $(\Att_{\theta, j}, \rho)$ and  $(\cX_{\bq}, \rho)$ are isomorphic, cf. Proposition \ref{prop:equality-root-systems}.

\item By the definition \eqref{eq:defcij}, $C^{{\bf A}_{\theta}(q;\Jb)}$ is the Cartan matrix of type $A_{\theta}$ for all $\Jb\in\Att_{\theta, j}$, thus coincides with $C^{\Jb}$.
By Proposition \ref{prop:equality-root-systems}, the generalized root systems are equal.
\end{itemize}

\subsubsection{Incarnation, $j = \frac{\theta+1}{2}$}
The assignment
\begin{align}\label{eq:dynkin-A-super-bis}
\Jb \longmapsto & {\bf A}_{\theta}(q;\Jb),&  \overline{\Jb} \longmapsto &{\bf A}_{\theta}(q^{-1};\Jb),& \Jb \in \att_{\theta, j},
\end{align}
gives an incarnation. Indeed, let $\bq$ be the matrix corresponding to ${\bf A}_{\theta}(q;\{j\})$.
\begin{itemize}[leftmargin=*]\renewcommand{\labelitemi}{$\circ$}

\item First, the map \eqref{eq:dynkin-A-super-bis}  $\Att_{\theta, j} \to \cX_{\bq}$ is bijective. 
By the definition \eqref{eq:rhoiq},  cf. \eqref{eq:rho-A-super}, we see that the basic data $(\Att_{\theta, j}, \rho)$ and  $(\cX_{\bq}, \rho)$ are isomorphic, cf. Proposition \ref{prop:equality-root-systems}.

\item By the definition \eqref{eq:defcij}, and Proposition \ref{prop:equality-root-systems}, the generalized root systems are equal.
\end{itemize}

\subsubsection{PBW-basis and (GK-)dimension}\label{subsubsec:type-A-super-PBW} Let $\Jb \in\Att_{\theta, j}$.
The root vectors are 
\begin{align*}
x_{\alpha_{ii}} &= x_{\alpha_{i}} = x_{i},& i \in \I, \\
x_{\alpha_{i j}} &= x_{(ij)} = [x_{i}, x_{\alpha_{i+1\, j}}]_c,& i <  j \in \I,
\end{align*}
with notation \eqref{eq:roots-Atheta}.  Recall the super structure defined in \S \ref{subsubsec:super-structure}. 
Then
\begin{align*}
\left\{ x_{\theta}^{n_{\theta \, \theta}} x_{(\theta-1 \, \theta)}^{n_{\theta-1 \, \theta}} x_{\theta-1
}^{n_{\theta-1 \, \theta-1}} \dots x_{(1 \, \theta)}^{n_{1 \, \theta}} \dots x_{1}^{n_{11}} \, \,\Big|\, \, \begin{aligned}&0\le n_{ij}<N \mbox{ if }\alpha_{ij}\mbox{ is even,}
\\
&0\le n_{ij}<2 \mbox{ if }\alpha_{ij}\mbox{ is odd}\end{aligned}\right\}.
\end{align*}
is a PBW-basis of $\toba_{\bq}$. 
Notice that it depends on $\Jb$.
If $N<\infty$, then
\begin{align*}
\dim \toba_{\bq}= 2^{j(\theta-j)} N^{\binom{j}{2}+\binom{\theta-j}{2}}.
\end{align*}
If $N=\infty$ (that is, if $q$ is not a root of unity), then
\begin{align*}
\GK \toba_{\bq}= \binom{j}{2}+\binom{\theta-j}{2}.
\end{align*}

\subsubsection{Presentation}\label{subsubsec:type-A-super}
First, the set of positive Cartan roots is 
\begin{align*}
\Oc_+^{\bq} = \{\alpha_{ij} \in \varDelta_{+}^{\bq}: \alpha_{ij} \text{ even}\}.
\end{align*}

The Nichols algebra $\toba_{\bq}$ is generated by $(x_i)_{i\in \I}$ with defining
relations
\begin{align}\label{eq:rels-type-A-super}
\begin{aligned}
& x_{ij} = 0, \quad i < j - 1; & x_{iii\pm1} &= 0, \quad
q_{ii}\neq-1;  \\
&[x_{(i-1i+1)},x_i]_c=0, \quad q_{ii}=-1; &
x_i^2&=0, \quad q_{ii}=-1; \\
&x_{\alpha_{ij}}^{N}=0, \quad \alpha_{ij}\in \Oc_+^{\bq}.
\end{aligned}
\end{align}

If $N = \infty$, i.e. $q\notin \G_{\infty}$, then we omit the last set of relations.

\subsubsection{The associated Lie algebra and $\ya$}\label{subsubsec:type-A-super-Lie-alg} As the roots of Cartan type are the  even  ones, 
the associated Lie algebra is of type $A_{j-1}\times A_{\theta-j}$.

\begin{align*}
\ya &= (N-1)\sum_{\alpha_{ij}\text{ even}} (\alpha_i + \alpha_{i+1} + \dots +\alpha_j) 
+ \sum_{\alpha_{ij}\text{ odd}} (\alpha_i + \alpha_{i+1} + \dots +\alpha_j).
\end{align*}

\subsubsection{Type $\supera{1}{1}$,  $N >2$}\label{subsubsec:type-A(1|1)}
We illustrate the material of this Subsection with the example $\theta = 2$.
Here  $\SJ(\{1\}) =  \SJ(\I_2) = 1$ and $\SJ(\{2\}) = 2 = 2 +1 -1$. Thus $\Att_{2,1} = \left\{\{1\},\I_2, \{2\}\right\}$.
The diagram of $(\Att_{2,1}, \rho)$ is
\begin{align*}
\xymatrix{\underset{\{1\}}{\vtxgpd} \hspace{3pt} \text{\raisebox{3pt}{$\overset{1}{\rule{25pt}{0.5pt}}$}}
\hspace{3pt}\underset{\I_{2}}{\vtxgpd} \hspace{3pt} \text{\raisebox{3pt}{$\overset{2}{\rule{25pt}{0.5pt}}$}}
\hspace{3pt}\underset{\{2\}}{\vtxgpd} . }
\end{align*}

In all cases, $x_{\alpha_{1}} = x_1$, $x_{\alpha_{2}} = x_2$, $x_{\alpha_{12}} = x_1x_2 - q_{12} x_2x_1 = x_{12}$.

\begin{description}
\bigbreak\item[$\Jb = \{1\}$] This is ${\bf A}_{2}(q;\{1\})$, i.e. $\xymatrix{ \overset{-1}{\underset{\ }{\circ}}\ar  @{-}[rr]^{q^{-1}}  &&
\overset{q}{\underset{\ }{\circ}}}$. 
Here $\alpha_{1}$ and $\alpha_{12}$ are odd;  the PBW-basis is
$\left\{ x_{2}^{n_{2}} x_{1 2}^{n_{12}}  x_{1}^{n_{1}} \,|\, 0\le n_{2}<N, 0\le n_{1}, n_{12}<2\right\}$.
Also $\toba_{\bq}$ is generated by $(x_i)_{i\in \I}$ with defining
relations
\begin{align}\label{eq:rels-type-A-super(1|1)-1}
x_{221} &= 0; & x_1^2&=0, &x_{2}^{N}=0.
\end{align}

\medbreak\item[$\Jb = \I_2$] This is ${\bf A}_{2}(q;\I_2)$, i.e.  $\xymatrix{ \overset{-1}{\underset{\ }{\circ}}\ar  @{-}[rr]^{q}  &&
\overset{-1}{\underset{\ }{\circ}}}$.
Here $\alpha_{1}$ and $\alpha_{2}$ are the odd roots;  the PBW-basis is
$\left\{ x_{2}^{n_{2}} x_{1 2}^{n_{12}}  x_{1}^{n_{1}} \,|\, 0\le n_{12}<N, 0\le n_{1}, n_{2}<2\right\}$.
Also $\toba_{\bq}$ is generated by $(x_i)_{i\in \I}$ with defining
relations
\begin{align}\label{eq:rels-type-A-super(1|1)-12}
x_1^2&=0,&
x_2^2&=0, & x_{12}^{N} &=0,
\end{align}

\medbreak\item[$\Jb = \{2\}$] This is ${\bf A}_{2}(q^{-1};\{2\})$, i.e. $\xymatrix{ \overset{q}{\underset{\ }{\circ}}\ar  @{-}[rr]^{q^{-1}}  &&
\overset{-1}{\underset{\ }{\circ}}}$.
Here $\alpha_{2}$ and $\alpha_{12}$ are odd;  the PBW-basis is
$\left\{ x_{2}^{n_{2}} x_{1 2}^{n_{12}}  x_{1}^{n_{1}} \,|\, 0\le n_{1}<N, 0\le n_{2}, n_{12}<2\right\}$.
Also $\toba_{\bq}$ is generated by $(x_i)_{i\in \I}$ with defining
relations
\begin{align}\label{eq:rels-type-A-super(1|1)-2}
x_{112} &= 0, &
x_2^2&=0, &x_{1}^{N}=0.
\end{align}
\end{description}
Clearly, there is a graded algebra isomorphism with the Nichols algebra corresponding to ${\bf A}_{2}(q;\{1\})$, interchanging $x_1$ with $x_2$.

\subsubsection{Example $\Att_{3,2}$} \label{subsubsec:type-A-super-exa32} 
Here $\att_{3, 2} = \{ \{2\}, \{1,3\}, \I_3\}$. We exemplify the incarnation when $j = \theta + 1 - j$ in the case $\theta = 3$, $j = 2$.
We give the diagram  of the basic datum and its incarnation:
\begin{align*}
\xymatrix{ &\vtxgpd  \ar@{-} ^{2}[d] &  &&&
\\ &\vtxgpd  \ar@{-}_{1}[dl] \ar@{-} ^{3}[dr] & &&& 
\\\vtxgpd \ar@{-}_{3}[dr] & &  \vtxgpd \ar@{-} ^{1}[dl]   & \ar@{~>}[r]&& 
\\ &\vtxgpd  \ar@{-} ^{2}[d] &  &&&
\\ &\vtxgpd   & &&& 
}
\xymatrix@R-10pt@C-20pt{ & \overset{\Att_3(\{2\}, q)}{\vtxgpd}  \ar@{-} ^{2}[d] &
\\  & \underset{\Att_3(\I_3, q)}{\vtxgpd}  \ar@{-}_{1}[dl] \ar@{-} ^{3}[dr] &
\\\overset{ \Att_3(\{1,3\}, q)}{\vtxgpd} \ar@{-}_{3}[dr] & &  \overset{ \Att_3(\{1,3\}, q^{-1})}{\vtxgpd} \ar@{-} ^{1}[dl] 
\\ &\overset{\Att_3(\I_3, q^{-1})}{\vtxgpd}  \ar@{-} ^{2}[d] & 
\\  & \underset{\Att_3(\{2\}, q^{-1})}{\vtxgpd}   &
}
\end{align*}

\subsection{Type $\superb{j}{\theta-j}$, $j\in\I_{\theta-1}$}\label{subsec:type-B-super} Here $q\notin\G_4$.

\subsubsection{Basic datum}
The basic datum is $(\mathtt B_{\theta, j}, \rho)$, where 
\begin{align*}
\mathtt B_{\theta, j} &= \left\{\Jb\subseteq \I : \SJ_{\Jb}=j   \right\}
\end{align*}
and
$\rho: \I \to \s_{\mathtt B_{\theta, j}}$ is defined by \eqref{eq:rho-A-super} if $i \in \I_{\theta-1}$, while  $\rho_{\theta}=\id$.
We see that 

\begin{itemize}
\item  If $k \in \I$, then $\{k\} \in \mathtt B_{\theta, j}$ if and only if $k=j$. 

\item $\rho_i$ is well-defined and $\rho_i^2 = \id$. Hence $(\mathtt B_{\theta, j},\rho)$ is a  basic datum. 

\item Let $\Jb \in \mathtt B_{\theta, j}$. Then there exists $k \in \I$ such that $\Jb \sim \{k\}$. Hence $(\mathtt B_{\theta, j},\rho)$ is connected. 
\end{itemize}

\subsubsection{Root system}\label{subsubsec:root-system-B theta j}
The bundle of Cartan matrices $(C^{\Jb})_{\Jb \in\mathtt B_{\theta, j}}$  is constant: $C^{\Jb}$ is the Cartan matrix of type $B_\theta$ as in \eqref{eq:dynkin-system-B}
for any $\Jb \in\mathtt B_{\theta, j}$.

The bundle of sets of roots $(\varDelta^{\Jb})_{\Jb \in\mathtt B_{\theta, j}}$  is constant; 
$\varDelta^{\Jb}$ is given by \eqref{eq:root-system-B}.

\begin{rem}\label{rem:type-superB-bijection}
There exists a bijection
\begin{align*}
\mathtt B_{\theta, j}& \to \mathtt B_{\theta, \theta-j}, &
\Jb & \mapsto \overline{\Jb} = 
\begin{cases}
\Jb\cup \{\theta\}, & \theta\notin \Jb; \\
\Jb - \{\theta\}, & \theta\in \Jb.
\end{cases}
\end{align*}
Notice that $\rho_i(\overline{\Jb})= \overline{\rho_i(\Jb)}$ for all $\Jb \in \mathtt B_{\theta, j}$ and all $i\in\I$, so the bijection above establishes an isomorphism of basic data between $(\mathtt B_{\theta, j}, \rho)$ and $(\mathtt B_{\theta, \theta-j}, \rho)$, which gives an isomorphism between the root systems. 
Otherwise, the GRS of type $\superb{j}{\theta-j}$ and $\superb{k}{\theta-k}$, $j, k\in\I_{\theta-1}$ with $k \neq j, \theta-j$ are not isomorphic, e.g. compute 
$\vert \cW(\{j\})\vert$.
\end{rem}

\subsubsection{Weyl groupoid}\label{subsubsec:type-superB-Weyl} 

The isotropy group  at $\{j\} \in \mathtt B_{\theta, j}$ is 
\begin{align*}
\cW(\{j\}) & = \left\langle \widetilde{\varsigma}^{\{j\}}_j, \varsigma_i^{\{j\}}: i \in \I, i\neq j \right\rangle \\ & \simeq \big((\Z/2)^{j}\rtimes \s_{j} \big)\times \big((\Z/2)^{\theta-j}\rtimes \s_{\theta-j} \big) \leq  GL(\Z^\I).
\end{align*}
where $\widetilde{\varsigma}^{\{j\}}_j= \varsigma^{\{j\}}_j \varsigma_{j+1} \dots \varsigma_{\theta-1}\varsigma_{\theta}\varsigma_{\theta-1} \dots \varsigma_{j}\in \cW(\{j\})$.
In other words, it is generated by $\theta - 1$ loops and one cycle (which is not a loop). 

\subsubsection{Lie superalgebras realizing this generalized root system}

\

Let $\pa_{\Jb}:\Z^{\I}\to\G_2$ be the group homomorphism as in
\eqref{eq:superstructure}.

\medbreak
To describe the incarnation in the setting of Lie superalgebras, we need the matrices ${\bf B}_{\theta}(\Jb)= (b_{ij}^{\Jb})_{i,j\in\I} \in\ku^{\I\times\I}$, $\Jb \in\Att_{\theta, j}$, where for $i,j \in\I$, $i\neq \theta$,

\begin{align*}
b_{ii}^{\Jb}&=\left\{ \begin{aligned} &2, & &i\notin\Jb, \\ &0, & &i\in\Jb; \end{aligned} \right. &
b_{ij}^{\Jb}&=\left\{ \begin{aligned} &-1, & &i\notin\Jb, j=i\pm1, \\ &\pm1, & &i\in\Jb,j=i\pm1, \\ &0, & &|i-j|\geq2; \end{aligned} \right. &
b_{\theta j}^{\Jb}&=\left\{ \begin{aligned} &2, & &j=\theta, \\ &-2, & &j=\theta-1, \\ &0, & &j\le \theta-2.
\end{aligned} \right.
\end{align*}
Then $\g\big({\bf B}_{\theta}(\Jb),\pa_{\Jb}\big) \simeq \mathfrak{osp}(2j+1,2(\theta-j))$.
The assignment 
\begin{align}\label{eq:incarnation-B-super-Lie}
\Jb \longmapsto & \big({\bf B}_{\theta}(\Jb),\pa_{\Jb}\big)
\end{align}
provides an isomorphism of generalized root systems
between the GRS of type $\superb{j}{\theta-j}$ 
and the root system of $\mathfrak{osp}(2j+1,2(\theta-j))$, cf. \S \ref{subsec:Weyl-gpd-super}.

\begin{rem}\label{rem:type-superB-bijection-lie}
The bijection in Remark \ref{rem:type-superB-bijection} gives also an isomorphism 
between the GRS of type $\superb{j}{\theta-j}$ 
and the root system of $\mathfrak{osp}(2(\theta-j)+1,2j) \not\simeq \mathfrak{osp}(2j+1,2(\theta-j))$.
\end{rem}

\subsubsection{Incarnation}
The assignment
\begin{align}\label{eq:dynkin-B-super}
\Jb \mapsto &
\begin{cases}
\xymatrix@R-6pt{ {\bf A}_{\theta-1}(q^2;\Jb) \ar @{-}[rr]^{\hspace*{1.2cm} q^{-2}} & & \overset{q}{\underset{\
}{\circ}}}, & \theta\notin\Jb; \\
\xymatrix@R-6pt{ {\bf A}_{\theta-1}(q^{-2};\Jb) \ar@{-}[rr]^{\hspace*{0.7cm} q^{2}} & & \overset{-q^{-1}}{\underset{\ }{\circ}}}, & \theta\in\Jb
\end{cases}
\end{align}
gives an incarnation. Here the diagram in the second row in \eqref{eq:dynkin-B-super} is obtained from the first one interchanging $q$ by $-q^{-1}$. 
Below we give information for the diagram in the first row; the information for the other follows as mentioned.

\begin{rem}\label{rem:type-superB-bijection-nichols}
The bijection in Remark \ref{rem:type-superB-bijection} provides  another incarnation of $\superb{j}{\theta-j}$,
which is different from the first one.
\end{rem}

\subsubsection{PBW-basis, dimension} \label{subsubsec:type-B-super-PBW}
With the notation \eqref{eq:roots-Atheta}, the root vectors are 
\begin{align*}
x_{\alpha_{ii}} &= x_{\alpha_{i}} = x_{i},& i \in \I, \\
x_{\alpha_{ij}} &= x_{(ij)} = [x_{i}, x_{\alpha_{(i+1) j}}]_c,& i <  j \in \I, \\
x_{\alpha_{i\theta} + \alpha_{\theta}} &= [x_{\alpha_{i\theta}}, x_\theta]_c, & i  \in \I_{\theta - 1},
\\
x_{\alpha_{i\theta} + \alpha_{j\theta}} &= [x_{\alpha_{i\theta} + \alpha_{(j+1) \theta}}, x_j]_c, & i <  j \in \I_{\theta - 1},
\end{align*}
Let $N_{\alpha} = \ord q_{\alpha\alpha}$; it takes different values when $N$ is even. Thus
\begin{multline*}
\big\{ x_{\theta  }^{n_{\theta  \theta}} 
x_{\alpha_{\theta-1\theta} + \alpha_{\theta\theta}}^{m_{\theta-1\theta}}
x_{\alpha_{\theta-1\theta}}^{n_{\theta-1  \theta}} 
x_{ \theta-1}^{n_{\theta-1  \theta-1}} 
\dots
x_{\alpha_{1\theta} + \alpha_{2\theta}}^{m_{12}}
\dots
x_{\alpha_{1\theta} + \alpha_{\theta\theta}}^{m_{1\theta}}
\dots
x_{\alpha_{1\theta}}^{n_{1  \theta}} 
\dots 
x_{1}^{n_{1 1}} \, \\ 
| \, 0\le n_{ij}<N_{\alpha_{ij}}; \, 0\le m_{ij}<N_{\alpha_{i\theta}+\alpha_{j\theta}}\big\}
\end{multline*}
is a PBW-basis of $\toba_{\bq}$. Let $M=\ord q^2$, $P=\ord (-q)$.  If $N<\infty$, then
\begin{align*}
\dim \toba_{\bq}= 2^{2j(\theta-j-1)} M^{\theta^2-\theta-2j(\theta-j-1)}N^{\theta-j}P^j.
\end{align*}
If $N=\infty$ (that is, if $q$ is not a root of unity), then
\begin{align*}
\GK \toba_{\bq}= \theta^2-2j(\theta-j-1).
\end{align*}

\subsubsection{Relations, $N>4$}\label{subsubsec:type-B-super-N>4}
The set of positive Cartan roots is 
\begin{align}\label{eq:Cartan-roots-B-super}
\begin{aligned}
\Oc_+^{\bq} = & \{\alpha_{ij}: i\le j\in\I_{\theta - 1}, \alpha_{ij} \text{ even}\} \cup \{\alpha_{i\theta}: i \in\I_{\theta - 1}\} \\ & \cup \{ \alpha_{i\theta}+ \alpha_{j\theta}: i< j\in\I_{\theta}, \alpha_{i(j-1)} \text{ even}\}.
\end{aligned}
\end{align}

The Nichols algebra $\toba_{\bq}$ is generated by $(x_i)_{i\in \I}$ with defining relations
\begin{align}\label{eq:rels-type-B-super-N>4}
\begin{aligned}
&x_{ij}=0, \quad i < j - 1; & &x_{ii(i\pm1)}=0, \quad i < \theta, q_{ii}\neq -1;  \\
&x_{\theta\theta\theta(\theta-1)}=0; & & [x_{(i-1\, i+1)},x_i]_c=0, \quad q_{ii}=-1;\\
&x_i^2=0, \quad q_{ii}=-1; & &x_{\alpha}^{N_\alpha}=0, \quad
\alpha\in\Oc_+^{\bq}.
\end{aligned}
\end{align}

If $N = \infty$, i.e. $q\notin \G_{\infty}$, then we omit the last set of relations.

\subsubsection{Relations, $N=3$}\label{subsubsec:type-B-super-N=3}
The set of positive Cartan roots is also
\eqref{eq:Cartan-roots-B-super}.
The Nichols algebra $\toba_{\bq}$ is generated by $(x_i)_{i\in \I}$ with defining relations
\begin{align}\label{eq:rels-type-B-super-N=3}
\begin{aligned}
&x_{ij}= 0, \quad i < j - 1; &  
&[x_{(i-1i+1)},x_i]_c=0, \quad q_{ii}=-1;  \\
&x_{ii(i\pm1)}= 0, \quad q_{ii}\neq -1; &
&[x_{\theta\theta(\theta-1)},x_{\theta(\theta-1)}]_c=0, \quad q_{\theta-1\theta-1}=-1;\\
& & 
&[x_{\theta\theta(\theta-1)(\theta-2)},x_{\theta(\theta-1)}]_c=0;  \\
&x_i^2=0, \quad q_{ii}=-1;&
&x_{\alpha}^{N_\alpha}=0,\quad\alpha\in\Oc_+^{\bq}.
\end{aligned}
\end{align}

\subsubsection{The associated Lie algebra and $\ya$}\label{subsubsec:type-B-super-Lie-alg} If $N$ is odd (respectively even), then the associated Lie algebra is of type $C_{j} \times B_{\theta-j}$ (respectively $C_{j}\times C_{\theta-j}$),
while
\begin{align*}
\ya &=\sum_{\substack{i\le j\in\I_{\theta-1},\\  \alpha_{ij} \text{ odd}}} \alpha_{ij} + 
\sum_{\substack{i\le j\in\I_{\theta-1},\\  \alpha_{ij} \text{ even}}} (M-1)\alpha_{ij} +
\sum_{\substack{i\in\I_{\theta},\\  \alpha_{i\theta} \text{ odd}}} (P-1)\alpha_{i\theta} \\ & + 
\sum_{\substack{i\in\I_{\theta},\\  \alpha_{i\theta} \text{ even}}} (N-1)\alpha_{i\theta} +
\sum_{\substack{i< j\in\I_{\theta},\\  \alpha_{i(j-1)} \text{ odd}}} (\alpha_{i\theta}+\alpha_{j\theta}) + 
\sum_{\substack{i< j\in\I_{\theta},\\  \alpha_{i(j-1)} \text{ even}}} (M-1)(\alpha_{i\theta}+\alpha_{j\theta}).
\end{align*}

\subsubsection{Types $\superb{2}{1}$ and $\superb{1}{2}$}\label{subsubsec:type-b(2|1)}
We exemplify the incarnation in the case $\theta = 3$, $j =1, 2$.
To describe the first example, we need the matrices:
\begin{align*}
\bp^{(1)}: & \xymatrix@C-5pt{ 
	\overset{-1}{\circ}\ar@{-}[r]^{q^{\text{-}12}}  &
	\overset{q^2}{\circ} \ar@{-}[r]^{q^{\text{-}2}}  & 
	\overset{q}{\circ}, }
& 
\bp^{(2)}: & \xymatrix@C-5pt{ 
	\overset{\text{-}1}{\circ}\ar@{-}[r]^{q^{2}}  &
	\overset{\text{-}1}{\circ} \ar@{-}[r]^{q^{\text{-}2}}  & 
	\overset{q}{\circ}, }
&
\bp^{(3)}: &\xymatrix@C-5pt{ 
	\overset{q^2}{\circ}\ar@{-}[r]^{q^{\text{-}2}}  &
	\overset{-1}{\circ} \ar@{-}[r]^{q^2}  & 
	\overset{\text{-}q^{\text{-}1}}{\circ}.}
\end{align*}
Here $\Btt_{3, 1} = \{ \{1\}, \{1,2\}, \{2,3\} \}$. 
The basic datum and the incarnation are:
\begin{align*}
&\xymatrix{ \overset{\{1\}}{\vtxgpd} \ar@{-}^{1}[r] & \overset{\{1,2\}}{\vtxgpd} \ar@{-}^{2}[r] & \overset{\{2,3\}}{\vtxgpd}, }
&
&\xymatrix{ 
	\overset{\bp^{(1)}}{\vtxgpd} \ar@{-}^{1}[r] & \overset{\bp^{(2)}}{\vtxgpd} \ar@{-}^{2}[r] & \overset{\bp^{(3)}}{\vtxgpd}.}
\end{align*}

For the second example, we need the matrices:
\begin{align*}
\bq^{(1)}: & \xymatrix@C-5pt{ 
	\overset{q^{\text{-}2}}{\circ}\ar@{-}[r]^{q^2}  &
	\overset{\text{-}1}{\circ} \ar@{-}[r]^{q^{\text{-}2}}  & 
	\overset{q}{\circ}, }
& 
\bq^{(2)}: & \xymatrix@C-5pt{ 
	\overset{\text{-}1}{\circ}\ar@{-}[r]^{q^{\text{-}2}}  &
	\overset{\text{-}1}{\circ} \ar@{-}[r]^{q^2}  & 
	\overset{\text{-}q^{\text{-}1}}{\circ}, }
&
\bq^{(3)}: &\xymatrix@C-5pt{ 
	\overset{\text{-}1}{\circ}\ar@{-}[r]^{q^2}  &
	\overset{q^{\text{-}2}}{\circ} \ar@{-}[r]^{q^2}  & 
	\overset{\text{-}q^{\text{-}1}}{\circ}.}
\end{align*}
Here $\Btt_{3, 2} = \{ \{2\}, \{1,3\}, \I_3\}$. 
The basic datum and the incarnation are:
\begin{align*}
&\xymatrix{ \overset{\{2\}}{\vtxgpd} \ar@{-}^{2}[r] & \overset{\I_3}{\vtxgpd} \ar@{-}^{1}[r] & \overset{\{1,3\}}{\vtxgpd}, }
&
&\xymatrix{ 
	\overset{\bq^{(1)}}{\vtxgpd} \ar@{-}^{2}[r] & \overset{\bq^{(2)}}{\vtxgpd} \ar@{-}^{1}[r] & \overset{\bq^{(3)}}{\vtxgpd}.}
\end{align*}

By Remark \ref{rem:type-superB-bijection}, 
$\superb{2}{1}$ and $\superb{1}{2}$ are isomorphic, but the incarnations are not.

\subsection{Type $\superd{j}{\theta-j}$, $j\in\I_{\theta - 1}$}\label{subsec:type-D-super}
Here $N > 2$.
We consider different settings, according to whether $j$ is $<$, $=$ or $> \frac{\theta}{2}$.
Below $c, \widetilde{c}, d$ are three different symbols alluding to Cartan types C and D. 

\subsubsection{Basic datum, $j < \frac{\theta}{2}$}

We first recall \eqref{eq:basic-datum-A-super}:
$\Att_{\theta - 1, j}  = \att_{\theta - 1, j} \mathbin{\dot{\cup}} \att_{\theta - 1,\theta-j}$, 
and define $\rho': \I_{\theta-1} \to \s_{\Att_{\theta - 1, j}}$  as in \eqref{eq:rho-A-super}. 
The basic datum is $(\Dtt_{\theta,j}, \rho)$, where 
\begin{align}\label{eq:basic-datum-D-super}
\Dtt_{\theta,j} &= (\att_{\theta-1, j}\times \{ c,\widetilde{c} \}) \mathbin{\dot{\cup}} (\att_{\theta-1,\theta-j} \times \{d\}).
\end{align}
and for  $i \in \I$,  $\rho_i:\Dtt_{\theta,j}\to\Dtt_{\theta,j}$ is given by
\begin{align}\label{eq:rho-D-super}
\begin{aligned}
\rho_i(\Jb,c) &:=\left\{ \begin{array}{ll} 
(\rho_i'(\Jb),c), & i\in\I_{\theta-2}, \\  
(\rho_{\theta-1}'(\Jb),d), & i=\theta-1\in\Jb, \\
(\rho_{\theta-1}'(\Jb),c), & i=\theta-1\notin\Jb, \\  
(\Jb, c), & i=\theta; \end{array} \right.
\\
\rho_i(\Jb,\widetilde{c})&:=\left\{ \begin{array}{ll} 
(\rho_i'(\Jb),\widetilde{c}), & i\in\I_{\theta-2}, \\  
(\rho_{\theta-1}'(\Jb),d), & i=\theta, \text{ when } \theta-1\in\Jb, \\
(\rho_{\theta-1}'(\Jb),\widetilde{c}), & i=\theta,  \text{ when } \theta-1\notin\Jb, \\  
(\Jb, \widetilde{c}), & i=\theta-1; \end{array} \right.
\\
\rho_i(\Jb, d) &:=\left\{ \begin{array}{ll} 
(\rho_i'(\Jb),d), & i\in\I_{\theta-2}, \\  
(\rho_{\theta-1}'(\Jb),c), & i=\theta-1\in\Jb, \\
(\rho_{\theta-1}'(\Jb),\widetilde{c}), & i=\theta, \text{ when } \theta-1\in\Jb, \\  
(\Jb, d), & i\in\{\theta-1,\theta\},  \text{ when } \theta-1\notin\Jb. \end{array} \right.
\end{aligned}
\end{align} 

\subsubsection{Basic datum, $j > \frac{\theta}{2}$}
Here we extend \eqref{eq:basic-datum-A-super} as follows:
\begin{align*}
\Att_{\theta - 1, j} :=\Att_{\theta - 1,\theta-j} &= \att_{\theta - 1, \theta-j} \mathbin{\dot{\cup}} \att_{\theta - 1,j}, 
& &\text{ for } \frac{\theta}{2}< j< \theta.
\end{align*}
Then we define $\rho': \I_{\theta-1} \to \s_{\Att_{\theta - 1, j}} $  as in \eqref{eq:rho-A-super}; with this, the basic datum is $(\Dtt_{\theta,j}, \rho)$, where $\Dtt_{\theta,j}$ and $\rho$ are defined exactly as in   \eqref{eq:basic-datum-D-super} and \eqref{eq:rho-D-super}.

\subsubsection{Basic datum, $j=\frac{\theta}{2}$}
We first recall \eqref{eq:basic-datum-A-super-bis}:
$\Att_{\theta - 1, j}  = \att_{\theta - 1, j} \mathbin{\dot{\cup}} \oatt_{\theta - 1, j}$, 
and define $\rho': \I_{\theta-1} \to \s_{\Att_{\theta - 1, j}}$  as in \eqref{eq:rho-A-super-bis}. 
The basic datum is $(\Dtt_{\theta,j}, \rho)$, where 
\begin{align*}
\Dtt_{\theta,j} &= (\att_{\theta-1, j}\times \{ c,\widetilde{c} \}) \mathbin{\dot{\cup}} (\oatt_{\theta-1, j}\times \{d\}),
\end{align*}
and $\rho_i:\Dtt_{\theta,j}\to\Dtt_{\theta,j}$, $i \in \I$, is defined as follows. If $\Jb \in \att_{\theta-1, j}$,  then  
\begin{align*}
\rho_i(\Jb,c) &:=\left\{ \begin{array}{ll} 
(\rho_i'(\Jb),c), & i\in\I_{\theta-2}, \\  
( \rho_{\theta-1}'(\Jb),d ), & i=\theta-1\in\Jb, \\
(\rho_{\theta-1}'(\Jb),c), & i=\theta-1\notin\Jb, \\  
(\Jb, c), & i=\theta; \end{array} \right.
\\
\rho_i(\Jb,\widetilde{c})&:=\left\{ \begin{array}{ll} 
(\rho_i'(\Jb),\widetilde{c}), & i\in\I_{\theta-2}, \\  
(\rho_{\theta-1}'(\Jb),d), & i=\theta, \text{ when } \theta-1\in\Jb, \\
(\rho_{\theta-1}'(\Jb),\widetilde{c}), & i=\theta,  \text{ when } \theta-1\notin\Jb, \\  
(\Jb, \widetilde{c}), & i=\theta-1; \end{array} \right.
\\
\rho_i \big( \overline{\Jb}, d \big) &:=\left\{ \begin{array}{ll} 
\big( \rho_{i}'(\overline{\Jb}),d \big), & i\in\I_{\theta-2}, \\  
(\rho_{\theta-1}'(\Jb),c), & i=\theta-1\in\Jb, \\
(\rho_{\theta-1}'(\Jb),\widetilde{c}), & i=\theta, \text{ when } \theta-1\in\Jb, \\  
\big( \overline{\Jb}, d \big), & i\in\{\theta-1,\theta\},  \text{ when } \theta-1\notin\Jb. \end{array} \right.
\end{align*} 
Here and in the previous cases, all $\rho_i$, $i\in\I_\theta$, are well-defined, and $(\Dtt_{\theta,j}, \rho)$ is a connected basic datum.

\subsubsection{Root system, $j\neq \frac{\theta}{2}$}\label{subsubsec:superD-jneq-theta/2-rootsystem}
The bundle of Cartan matrices $(C^{(\Jb,x)})_{(\Jb,x) \in\Dtt_{\theta,j}}$ is the following: 
\begin{itemize}
\item Let $\Jb \in\att_{\theta-1,j}$. If $\theta-1\notin \Jb$, then $C^{(\Jb,c)}$ is the Cartan matrix of type $C_\theta$ as in \eqref{eq:dynkin-system-C}. If $\theta-1\in \Jb$, then $C^{(\Jb,c)}$ is and of type $A_\theta$ as in \eqref{eq:dynkin-system-A}.
\item $C^{(\Jb,\widetilde{c})}$ 
has the same Dynkin diagram as $C^{(\Jb,c)}$, but  changing the numeration of the diagram by $\theta-1 \longleftrightarrow\theta$.
\item Let $\Jb \in\att_{\theta-1, \theta-j}$. If $\theta-1\notin \Jb$, then $C^{(\Jb,d)}$ 
is the Cartan matrix of type $D_\theta$ \eqref{eq:dynkin-system-D}.
If $\theta-1\in \Jb$, then the Dynkin diagram of $C^{(\Jb,d)}$ is of type ${}_{\theta-2}T$, see \eqref{eq:mTn}.
\end{itemize}

As in \eqref{eq:superstructure}, let $\pa_{\Jb}: \Z^{\theta-1} \to \G_2$ be the group homomorphism such that $\pa_{\Jb}(\alpha_k)= -1$ iff $k \in \Jb$. 
Using this parity vector, we define the bundle of root sets $(\varDelta^{(\Jb,x)})_{(\Jb,x) \in\Dtt_{\theta,j}}$  as follows:
\begin{align} \label{eq:root-D-super}
\begin{aligned}
\varDelta^{(\Jb,c)} &=  \{ \pm \alpha_{ij}: i \leq j \in\I \}
\cup \{ \pm (\alpha_{i,\theta} + \alpha_{j, \theta-1}): i < j \in\I_{\theta-1} \}  
\\ 
& \quad\cup \{ \pm(\alpha_{i,\theta-1} + \alpha_{i, \theta}): \ i \in\I_{\theta-1}, \ \pa_{\Jb}(\alpha_{i, \theta-1})= 1\},
\\
\varDelta^{(\Jb,\widetilde{c})} &= s_{\theta-1 \theta} \big(\varDelta^{(\Jb,c)} \big), 
\\
\varDelta^{(\Jb,d)}  & =  \{ \pm \alpha_{ij}: \ 1 \leq i \leq j \leq \theta, (i,j) \neq ( \theta-1, \theta) \}
\\ 
& \quad\cup \{ \pm (\alpha_{\theta-1}+ \alpha_\theta): \pa_{\Jb}(\alpha_{\theta-1})=-1 \} 
\\ & \quad\cup \{ \pm(\alpha_{i,\theta-2} + \alpha_{\theta}): i \in\I_{\theta-2} \} 
\\ & \quad\cup \{ \pm (\alpha_{i,\theta} + \alpha_{j, \theta-2}): i < j \in\I_{\theta-2} \}
\\ & \quad\cup \{ \pm (\alpha_{i,\theta} + \alpha_{i, \theta-2}): i \in\I_{\theta-2}, \ \pa_{\Jb}(\alpha_{i, \theta-1})=-1 \}.
\end{aligned}
\end{align}

\subsubsection{Root system, $j=\frac{\theta}{2}$} 
The bundle of Cartan matrices $(C^{(\Jb,x)})_{(\Jb,x) \in\Dtt_{\theta,j}}$ is the following, where  $\Jb \in\att_{\theta-1,j}$:
\begin{itemize}
\item  $C^{(\Jb,c)}$  and $C^{(\Jb,\widetilde{c})}$ are exactly as in \S \ref{subsubsec:superD-jneq-theta/2-rootsystem}.
\item If $\theta-1\notin \Jb$, then $C^{(\overline{\Jb},d)}$ 
is the Cartan matrix of type $D_\theta$ \eqref{eq:dynkin-system-D}.
If $\theta-1\in \Jb$, then the Dynkin diagram of $C^{(\overline{\Jb},d)}$ is of type ${}_{\theta-2}T$, see  \eqref{eq:mTn}.
\end{itemize}
The bundle of root sets $(\varDelta^{(\Jb,x)})_{(\Jb,x) \in\Dtt_{\theta,j}}$ is defined analogously:
$\varDelta^{(\Jb,c)}$ and $\varDelta^{(\Jb,\widetilde{c})}$, are as in \eqref{eq:root-D-super}, while $\varDelta^{(\overline{\Jb},d)}$ is as $\varDelta^{(\Jb,d)}$ in \eqref{eq:root-D-super}.

\subsubsection{Weyl groupoid}\label{subsubsec:type-superC-Weyl} 
The isotropy group  at $(\{j\},c) \in \Dtt_{\theta,j}$ is 
\begin{align*}
\cW(\{j\}) & = \langle \widetilde{\varsigma}^{\{j\}}_j, \varsigma_i^{\{j\}}: i \in \I, i\neq j\rangle \\ & \simeq \big((\Z/2\Z)^{j-1}\rtimes \s_{j} \big)\times \big((\Z/2\Z)^{\theta-j}\rtimes \s_{\theta-j} \big) \leq  GL(\Z^\I).
\end{align*}
where $\widetilde{\varsigma}^{\{j\}}_j= \varsigma^{\{j\}}_j \varsigma_{j+1} \dots \varsigma_{\theta-1}\varsigma_{\theta}\varsigma_{\theta-1} \dots \varsigma_{j}\in \cW(\{j\})$.

\subsubsection{Lie superalgebras realizing this generalized root system, $j\neq \frac{\theta}{2}$}

Let $(\Jb,x) \in\Dtt_{\theta,j}$.
Let $\pa_{(\Jb,x)}:\Z^{\I}\to\G_2$ be the group homomorphism such that 
\begin{itemize}[leftmargin=*]\renewcommand{\labelitemi}{$\circ$}
\item If $x=c$, then $\pa_{(\Jb,c)}(\alpha_i)=-1  \iff i\in\Jb$.
\item If $x=\widetilde{c}$, then $\pa_{(\Jb,\widetilde{c})}$ is $\pa_{(\Jb,c)}$ but changing the numeration by $\theta-1 \longleftrightarrow\theta$.
\item If $x=d$, then $\pa_{(\Jb,d)}(\alpha_i)=-1  \iff$ either $i\in\Jb$ or else $i=\theta$, $\theta-1\in\Jb$.
\end{itemize}


To describe the incarnation in this setting, we need matrices
\begin{itemize}[leftmargin=*]
\item ${\bf C}_{\theta}(\Jb)= (c_{ij}^{\Jb})_{i,j\in\I} \in \ku^{\I\times\I}$, $\Jb \in\att_{\theta, j}$, where for $i,j \in\I$,
\begin{align*}
c_{ii}^{\Jb}&=\left\{ \begin{array}{ll} 2, & i\notin\Jb, \\ 0, & i\in\Jb, \end{array} \right. &
c_{ij}^{\Jb}&=\left\{ \begin{array}{ll} 
-1, & i\notin\Jb, j=i\pm1, (i,j)\neq (\theta-1,\theta), \\ 
\mp1, & i\in\Jb,j=i\pm1, (i,j)\neq (\theta-1,\theta), \\ 
-2, & i=\theta-1,j=\theta, \\
0, & |i-j|\geq2.
\end{array} \right.
\end{align*}
\item ${\bf D}_{\theta}(\Jb)= (d_{ij}^{\Jb})_{i,j\in\I} \in\ku^{\I\times\I}$, $\Jb \in\att_{\theta,\theta-j}$, where for $i,j \in\I$,
\begin{align*}
d_{ii}^{\Jb}&=\left\{ \begin{array}{ll} 
2, & i\notin\Jb, \\ 
0, & i\in\Jb, \\
2, & i=\theta, \theta-1\notin\Jb, \\ 
0, & i=\theta, \theta-1\in\Jb,\end{array} \right. &
d_{ij}^{\Jb}&=\left\{ \begin{array}{ll} 
-1, & i,j \in\I_{\theta-1}, i\notin\Jb, j=i\pm1, \\ 
\mp1, & i,j \in\I_{\theta-1},i\in\Jb,j=i\pm1, \\ 
0, & i,j \in\I_{\theta-1},|i-j|\geq2.
\end{array} \right.
\\
d_{\theta j}^{\Jb}&=\left\{ \begin{array}{ll} 
-1, & j=\theta-2, \\ 
2, & j=\theta-1\in\Jb, \\
0, & j=\theta-1\notin\Jb, \\
0, & j\in\I_{\theta-3}.
\end{array} \right. &
d_{i\theta}^{\Jb}&=\left\{ \begin{array}{ll} 
-1, & i=\theta-2, \\ 
-1, & j=\theta-1\in\Jb, \\
0, & j=\theta-1\notin\Jb, \\
0, & j\in\I_{\theta-3}.
\end{array} \right.
\end{align*}
\end{itemize}

Then $\g\big({\bf C}_{\theta}(\Jb),\pa_{(\Jb,c)}\big) \simeq \mathfrak{osp}(2j,2(\theta-j))$.
The assignment 
\begin{align}\label{eq:incarnation-D-super-Lie}
\begin{aligned}
(\Jb,c) \longmapsto & \big({\bf C}_{\theta}(\Jb),\pa_{(\Jb,c)}\big),
\\
(\Jb,\widetilde{c}) \longmapsto & \, \big( s_{\theta - 1 \theta}({\bf C}_{\theta}(\Jb)), \pa_{(\Jb,\widetilde{c})}\big),
\\
(\Jb,d) \longmapsto & \big({\bf D}_{\theta}(\Jb), \pa_{(\Jb,d)}\big).
\end{aligned}
\end{align}
provides an isomorphism of generalized root systems
between $(\Dtt_{\theta, j}, \rho)$ 
and the root system of $\mathfrak{osp}(2j,2(\theta-j))$, cf. \S \ref{subsec:Weyl-gpd-super}.

\subsubsection{Lie superalgebras realizing this generalized root system, $j=\frac{\theta}{2}$}

\

There is an incarnation of  $\Dtt_{\theta,j}$ as follows: $(\Jb,c)$, respectively $(\Jb,\widetilde{c})$, $(\overline{\Jb},d)$, maps to the pairs in \eqref{eq:incarnation-D-super-Lie}, accordingly.

\subsubsection{Incarnation, $j\neq \frac{\theta}{2}$} 
Here is an incarnation of $\Dtt_{\theta,j}$:
\begin{align}\label{eq:dynkin-Dtheta-super-c}
(\Jb,c) \longmapsto &\xymatrix@R-6pt{
{\bf A}_{\theta-1}(q;\Jb) \ar  @{-}[r]^{\hspace*{1.1cm} q^{-2}}  & \overset{q^2}{\underset{\ }{\circ}}},
\\ \label{eq:dynkin-Dtheta-super-ct}
(\Jb,\widetilde{c}) \longmapsto & \, s_{\theta - 1 \theta}\Big(\xymatrix@R-6pt{
{\bf A}_{\theta-1}(q;\Jb) \ar  @{-}[r]^{\hspace*{1.1cm} q^{-2}}  & \overset{q^2}{\underset{\ }{\circ}}}\Big),
\\\label{eq:dynkin-Dtheta-super-d1}
(\Jb,d) \longmapsto &
\xymatrix@R-6pt{ & \overset{-1}{\circ} \ar  @{-}[d]_{q^{-1}}\ar  @{-}[dr]^{q^2} & \\
{\bf A}_{\theta-2}(q;\Jb\cap \I_{\theta - 2}) \hspace*{-1.2cm} & \ar  @{-}[r]^{ q^{-1}}  & \overset{-1}{\underset{\ }{\circ}},}
& \theta-1 \in \Jb, 
\\ \label{eq:dynkin-Dtheta-super-d2}
(\Jb,d) \longmapsto & \xymatrix@R-6pt{ & \overset{q^{-1}}{\circ} \ar @{-}[d]_{q} & \\
{\bf A}_{\theta-2}(q^{-1};\Jb\cap \I_{\theta - 2}) \hspace*{-1.1cm} & \ar  @{-}[r]^{q}  & \overset{q^{-1}}{\underset{\ }{\circ}},}
& \theta-1 \notin \Jb.
\end{align}

\subsubsection{Incarnation, $j=\frac{\theta}{2}$} 
There is an incarnation of  $\Dtt_{\theta,j}$ as follows: $(\Jb,c)$, respectively $(\Jb,\widetilde{c})$, $(\overline{\Jb},d)$, maps to the Dynkin diagram in \eqref{eq:dynkin-Dtheta-super-c}, respectively \eqref{eq:dynkin-Dtheta-super-ct}, \eqref{eq:dynkin-Dtheta-super-d1} or \eqref{eq:dynkin-Dtheta-super-d2}, accordingly.

\subsubsection{PBW-basis and (GK-)dimension} \label{subsubsec:type-CD-super-PBW}

The root vectors $x_{\beta_k}$ are described as in Cartan type  $C$, $D$, c.f. \S \ref{subsubsec:type-C-PBW} and \ref{subsubsec:type-D-PBW}, $k\in\I_{\theta^2-\theta+j}$.
Thus
\begin{align*}
\left\{ x_{\beta_{\theta^2-\theta+j}}^{n_{\theta^2-\theta+j}} x_{\beta_{\theta^2-\theta+j - 1}}^{n_{\theta^2-\theta+j- 1}} \dots x_{\beta_2}^{n_{2}}  x_{\beta_1}^{n_{1}} \, | \, 0\le n_{k}<N_{\beta_k} \right\}.
\end{align*}
is a PBW-basis of $\toba_{\bq}$. Let $M=\ord q^2$. If $N<\infty$, then
\begin{align*}
\dim \toba_{\bq}= 2^{2j(\theta-j)} M^{\theta-j} N^{(\theta-j-1)^2+j^2-1}.
\end{align*}
If $N=\infty$ (that is, if $q$ is not a root of unity), then
\begin{align*}
\GK \toba_{\bq}= (\theta-j)(\theta-j-1)+j^2.
\end{align*}

\begin{itemize}[leftmargin=*]\renewcommand{\labelitemi}{$\circ$}
\item The set of positive Cartan roots for \eqref{eq:dynkin-Dtheta-super-c} is 
\begin{align}\label{eq:Cartan-Dtheta-super-c}
\begin{aligned}
\Oc_+^{\bq} = & 
\{ \alpha_{ij}: i \leq j \in\I, \ \pa_{\Jb}(\alpha_{ij})= 1 \}
\\
& \quad \cup \{ \alpha_{i,\theta} + \alpha_{j, \theta-1}: i \leq j \in\I_{\theta-1}, \ \pa_{\Jb}(\alpha_{i,\theta} + \alpha_{j, \theta-1})= 1\}.
\end{aligned}
\end{align}

\item The set of positive Cartan roots for \eqref{eq:dynkin-Dtheta-super-ct} is $s_{\theta-1 \theta}$ of the set described in \eqref{eq:Cartan-Dtheta-super-c}.

\item The set of positive Cartan roots for \eqref{eq:dynkin-Dtheta-super-d1} or \eqref{eq:dynkin-Dtheta-super-d2} is 
\begin{align}\label{eq:Cartan-Dtheta-super-d1}
\begin{aligned}
\Oc_+^{\bq} = & 
\{\alpha_{ij}:i \leq j\in\I_{\theta}, \ \pa_{\Jb}(\alpha_{ij})= 1 \}
\\ & \quad\cup \{\alpha_{i,\theta-2} + \alpha_{\theta}: i \in\I_{\theta-2}, , \ \pa_{\Jb}(\alpha_{i(\theta-2)}+\alpha_{\theta} ) = 1 \} 
\\ & \quad\cup \{ \alpha_{i,\theta} + \alpha_{j, \theta-2}: i \leq j \in\I_{\theta-2} , \ \pa_{\Jb}(\alpha_{i,\theta} + \alpha_{j, \theta-2})= 1 \}.
\end{aligned}
\end{align}
\end{itemize}

We now provide the defining relations of the Nichols algebras according to the Dynkin diagram \eqref{eq:dynkin-Dtheta-super-c}, \eqref{eq:dynkin-Dtheta-super-d1} or \eqref{eq:dynkin-Dtheta-super-d2}. The relations for the Dynkin diagram \eqref{eq:dynkin-Dtheta-super-ct} follow from those in \eqref{eq:dynkin-Dtheta-super-c}
applying the transposition $s_{\theta - 1 \, \theta}$.

\smallbreak \subsubsection{The  Dynkin diagram \eqref{eq:dynkin-Dtheta-super-c}, $q_{\theta-1 \theta-1}\neq -1$, $N>4$}\label{subsubsec:type-D-super-a-N>4}

The Nichols algebra $\toba_{\bq}$ is generated by $(x_i)_{i\in \I}$ with defining relations
\begin{align}\label{eq:rels-type-D-super-a-N>4}
\begin{aligned}
&x_{ij}= 0, & &i < j - 1; & 
&x_{iii\pm1}=0,& &i \in\I_{\theta - 1}-\Jb ;  
\\
&x_{\theta-1\theta-1\theta-2}=0; & &&
&x_{\theta-1\theta-1\theta-1\theta}=0; &&
\\
&x_{\theta\theta\theta-1}=0; & &&
&[x_{(i-1i+1)},x_i]_c=0,& &i\in\I_{2,\theta - 2}\cap \Jb ;\\
&x_i^2=0, \quad i\in \Jb ; & && 
&x_{\alpha}^{N_\alpha}=0,&  &\alpha\in\Oc_+^{\bq}.
\end{aligned}
\end{align}

If $N = \infty$, i.e. $q\notin \G_{\infty}$, then we omit the last set of relations.

\smallbreak \subsubsection{The  Dynkin diagram \eqref{eq:dynkin-Dtheta-super-c},  $q_{\theta-1 \theta-1}\neq -1$, $N=4$}\label{subsubsec:type-D-super-a-N=4}

The Nichols algebra $\toba_{\bq}$ is generated by $(x_i)_{i\in \I}$ with defining
relations
\begin{align}\label{eq:rels-type-D-super-a-N=4}
\begin{aligned}
& [x_{(\theta-2\theta)},x_{\theta-1\theta}]_c=0; &
& x_{iii\pm1}=0,& &i\in\I_{\theta - 2}- \Jb ;  \\
& x_{\theta-1\theta-1\theta-2}=0; & 
& x_{ij} = 0,& &i < j - 1; \\
& x_{\theta-1\theta-1\theta-1\theta}=0; & 
& [x_{(i-1i+1)},x_i]_c=0,& &i\in\I_{2,\theta - 2}\cap \Jb ;\\
& x_i^2=0, \quad i\in \Jb ; &  &x_{\alpha}^{N_\alpha}=0,&
&\alpha\in\Oc_+^{\bq} .
\end{aligned}
\end{align}

\smallbreak \subsubsection{The  Dynkin diagram \eqref{eq:dynkin-Dtheta-super-c}, $q_{\theta-1 \theta-1}\neq -1$, $N=3$}\label{subsubsec:type-D-super-a-N=3}

The Nichols algebra $\toba_{\bq}$ is generated by $(x_i)_{i\in \I}$ with defining
relations
\begin{align}\label{eq:rels-type-D-super-a-N=3}
\begin{aligned}
&[[x_{(\theta-2\theta)},x_{\theta-1}]_c,x_{\theta-1}]_c=0;  & 
& x_{iii\pm1}=0, & &i\in\I_{\theta - 2}- \Jb;  \\
& x_{\theta-1\theta-1\theta-2}=0; &
& x_{ij} = 0, & &i < j - 1;\\
& x_{\theta\theta\theta-1}=0; & 
& [x_{(i-1i+1)},x_i]_c=0, && i\in\I_{2,\theta - 2} \cap \Jb ;\\
& x_i^2=0, \quad i\in \Jb ; & 
& x_{\alpha}^{N_\alpha} =0, & &\alpha\in\Oc_+^{\bq}.
\end{aligned}
\end{align}

\smallbreak \subsubsection{The  Dynkin diagram \eqref{eq:dynkin-Dtheta-super-c}, $q_{\theta-1 \theta-1}=-1$, $N \neq 4$}\label{subsubsec:type-D-super-b-Nneq4}

The Nichols algebra $\toba_{\bq}$ is generated by $(x_i)_{i\in \I}$ with defining
relations
\begin{align}\label{eq:rels-type-D-super-b-Nneq4}
\begin{aligned}
&[[x_{\theta-2\theta-1},x_{(\theta-2\theta)}]_c,x_{\theta-1}]_c=0, && \theta-2 \in \Jb ; \\
&[[[x_{(\theta-3\theta)},x_{\theta-1}]_c,x_{\theta-2}]_c,x_{\theta-1}]_c=0, && \theta-2 \notin \Jb ;\\
&[x_{(i-1i+1)},x_i]_c=0, \quad i\in\I_{2,\theta - 2}\cap \Jb; &&x_{ij} = 0, \quad i < j - 1; \\
&x_{iii\pm1}=0, \quad i\in\I_{\theta - 2}-\Jb ;  && x_{\theta\theta\theta-1}=0; \\
& x_i^2=0, \quad i\in \Jb ; &&
x_{\alpha}^{N_\alpha}=0, \quad \alpha\in\Oc_+^{\bq}.
\end{aligned}
\end{align}

If $N = \infty$, i.e. $q\notin \G_{\infty}$, then we omit the last set of relations.

\smallbreak \subsubsection{The  Dynkin diagram \eqref{eq:dynkin-Dtheta-super-c}, $q_{\theta-1 \theta-1}=-1$, $N=4$}\label{subsubsec:type-D-super-b-N=4}

The Nichols algebra $\toba_{\bq}$ is generated by $(x_i)_{i\in \I}$ with defining
relations
\begin{align}\label{eq:rels-type-D-super-b-N=4}
\begin{aligned}
& [[[x_{(\theta-3\theta)},x_{\theta-1}]_c, x_{\theta-2}]_c,x_{\theta-1}]_c=0, && \theta-2 \notin \Jb ;
\\
& [[x_{\theta-2\theta-1},x_{(\theta-2\theta)}]_c,x_{\theta-1}]_c=0, && \theta-2 \in \Jb ;  \\
& x_{iii\pm1}=0, \quad  i\in\I_{\theta - 2}- \Jb ;  && x_{ij} = 0, \ \ i < j - 1;  \\
&[x_{(i-1i+1)},x_i]_c=0, \quad i\in\I_{2,\theta - 2}\cap \Jb; && x_i^2=0, \ \ i\in \Jb; \\
& x_{\alpha}^{N_\alpha} =0, \quad \alpha\in\Oc_+^{\bq}; && x_{\theta-1\theta}^2=0.
\end{aligned}
\end{align}

\smallbreak \subsubsection{The  Dynkin diagram \eqref{eq:dynkin-Dtheta-super-d1},  $q_{\theta-2 \theta-2}\neq -1$,  $N\neq 4$}\label{subsubsec:type-D-super-c-Nneq4}

The Nichols algebra $\toba_{\bq}$ is generated by $(x_i)_{i\in \I}$ with defining
relations
\begin{align}\label{eq:rels-type-D-super-c-Nneq4}
\begin{aligned}
\begin{aligned}
& x_{\theta-2\, \theta-2 \, \theta}=0;&
& x_{iii\pm1}=0, \quad i\in\I_{\theta - 2}- \Jb ; \\
& x_{ij} = 0, \quad  i < j - 1,\theta-2;  &
& [x_{(i-1i+1)},x_i]_c=0, \quad i\in\I_{\theta - 3}\cap \Jb ;\\
& x_i^2=0, \quad  i\in \Jb ; & 
& x_{\alpha}^{N_\alpha} =0, \quad \alpha\in\Oc_+^{\bq};
\end{aligned}\\
\raggedright \begin{aligned}
x_{(\theta-2\theta)} = q_{\theta-2\theta-1}(1-q^2)x_{\theta-1}x_{\theta-2\theta}
-q_{\theta-1\theta}(1+q^{-1})[x_{\theta-2\theta},x_{\theta-1}]_c.
\end{aligned}
\end{aligned}
\end{align}

If $N = \infty$, i.e. $q\notin \G_{\infty}$, then we omit the relations $x_{\alpha}^{N_{\alpha}} =0$, $\alpha \in \Oc_+^{\bq}$.

\smallbreak \subsubsection{The  Dynkin diagram \eqref{eq:dynkin-Dtheta-super-d1}, $q_{\theta-2 \theta-2}\neq -1$, $N=4$}\label{subsubsec:type-D-super-c-N=4}

The Nichols algebra $\toba_{\bq}$ is generated by $(x_i)_{i\in \I}$ with defining
relations
\begin{align}\label{eq:rels-type-D-super-c-N=4}
\begin{aligned}
\begin{aligned}
& x_{iii\pm1}=0, \quad i\in\I_{\theta - 2}-\Jb; && [x_{(i-1i+1)},x_i]_c=0, \quad  i\in\I_{\theta - 3} \cap \Jb ; \\
& x_{\theta-2\, \theta-2 \, \theta}=0; &&x_{ij}= 0, \quad  i < j - 1,\theta-2;  \\
& x_i^2=0, \quad  i\in \Jb; && x_{\alpha}^{N_\alpha}=0, \quad \alpha\in\Oc_+^{\bq};
\end{aligned}
\\
\raggedright
\begin{aligned}
x_{(\theta-2\theta)}= q_{\theta-2\theta-1}(1-q^2)x_{\theta-1}x_{\theta-2\theta} -q_{\theta-1\theta}(1+q^{-1})[x_{\theta-2\theta},x_{\theta-1}]_c.
\end{aligned}
\end{aligned}
\end{align}

\smallbreak \subsubsection{The  Dynkin diagram \eqref{eq:dynkin-Dtheta-super-d1}, $q_{\theta-2 \theta-2}=-1$, $N\neq 4$}\label{subsubsec:type-D-super-d-Nneq4}

The Nichols algebra $\toba_{\bq}$ is generated by $(x_i)_{i\in \I}$ with defining
relations
\begin{align}\label{eq:rels-type-D-super-d-Nneq4}
\begin{aligned}
\begin{aligned}
& [x_{\theta-3\, \theta-2 \, \theta},x_{\theta-2}]_c=0; &
& x_{iii\pm1}=0, &&  i\in\I_{\theta - 2} \cap \Jb ; \\
& x_{ij} = 0, \quad i < j - 1,\theta-2;  &
& [x_{(i-1i+1)},x_i]_c=0, && i\in \I_{\theta - 3} \cap \Jb;\\
& x_i^2=0, \quad i\in \Jb; & 
& x_{\alpha}^{N_\alpha}=0, && \alpha\in\Oc_+^{\bq};
\end{aligned}
\\
\raggedright
\begin{aligned}
x_{(\theta-2\theta)} = q_{\theta-2\theta-1}(1-q^2)x_{\theta-1}x_{\theta-2\theta}
-q_{\theta-1\theta}(1+q^{-1})[x_{\theta-2\theta},x_{\theta-1}]_c.
\end{aligned}
\end{aligned}
\end{align}

If $N = \infty$, i.e. $q\notin \G_{\infty}$, then we omit the relations $x_{\alpha}^{N_{\alpha}} =0$, $\alpha \in \Oc_+^{\bq}$.

\smallbreak \subsubsection{The  Dynkin diagram \eqref{eq:dynkin-Dtheta-super-d1}, $q_{\theta-2 \theta-2}=-1$, $N=4$}\label{subsubsec:type-D-super-d-N=4}

The Nichols algebra $\toba_{\bq}$ is generated by $(x_i)_{i\in \I}$ with defining
relations
\begin{align}\label{eq:rels-type-D-super-d-N=4}
\begin{aligned}
\begin{aligned}
&x_{iii\pm1}=0, \quad i\in\I_{\theta - 2}-\Jb;  && x_{ij}= 0, \quad i < j - 1,\theta-2;  
\\
& [x_{\theta-3\, \theta-2 \, \theta},x_{\theta-2}]_c=0; &&
[x_{(i-1i+1)},x_i]_c=0, \quad i\in\I_{\theta - 3}\cap \Jb; 
\\
&x_i^2=0, \quad i\in\Jb;&& x_{\alpha}^{N_\alpha}=0, \quad \alpha\in\Oc_+^{\bq};
\end{aligned}
\\
\begin{aligned}
x_{(\theta-2\theta)}= 2 q_{\theta-2\theta-1} x_{\theta-1}x_{\theta-2\theta} -q_{\theta-1\theta}(1+q^{-1}) [x_{\theta-2\theta},x_{\theta-1}]_c.
\end{aligned}
\end{aligned}
\end{align}

\smallbreak \subsubsection{The  Dynkin diagram \eqref{eq:dynkin-Dtheta-super-d2}, $q_{\theta-2 \theta-2}\neq -1$, $N\neq 4$}\label{subsubsec:type-D-super-e-Nneq4}

The Nichols algebra $\toba_{\bq}$ is generated by $(x_i)_{i\in \I}$ with defining
relations
\begin{align}\label{eq:rels-type-D-super-e-Nneq4}
\begin{aligned}
& x_{iii\pm1}=0, \   i\in\I_{\theta - 2}-\Jb; && x_{\theta-2\, \theta-2 \, \theta}=0; & 
& x_{\theta-1\theta-1\theta-2}=0; \\
& x_{ij}=0, \  i < j - 1, \theta-2; && x_{\theta-1\theta}=0; & 
& x_{\theta\theta\theta-2}=0; \\
& [x_{(i-1i+1)},x_i]_c=0, \  i\in \I_{\theta - 3}\cap \Jb; &
& x_i^2=0, \  i\in \Jb ; & 
& x_{\alpha}^{N_\alpha}=0, \ \alpha\in\Oc_+^{\bq}.
\end{aligned}
\end{align}

If $N = \infty$, i.e. $q\notin \G_{\infty}$, then we omit the last set of relations.

\smallbreak \subsubsection{The  Dynkin diagram \eqref{eq:dynkin-Dtheta-super-d2}, $q_{\theta-2 \theta-2}\neq -1$, $N=4$}\label{subsubsec:type-D-super-e-N=4}

The Nichols algebra $\toba_{\bq}$ is generated by $(x_i)_{i\in \I}$ with defining
relations
\begin{align}\label{eq:rels-type-D-super-e-N=4}
\begin{aligned}
&x_{\theta\theta\theta-2}=0; &
& x_{iii\pm1}=0, &&   i\in\I_{\theta - 2}-\Jb ; \\
&x_{\theta-2\, \theta-2 \, \theta}=0; & 
& [x_{(i-1i+1)},x_i]_c=0, && i\in \I_{\theta - 3} \cap \Jb; \\
&x_{\theta-1\theta-1\theta-2}=0; & 
&x_{ij}= 0, &&  i < j - 1,\theta-2; \\
&x_{\theta-1\theta}=0; &
&x_i^2=0, \quad i\in \Jb ; & & x_{\alpha}^{N_\alpha}=0, \
\alpha\in\Oc_+^{\bq}.
\end{aligned}
\end{align}

\smallbreak \subsubsection{The  Dynkin diagram \eqref{eq:dynkin-Dtheta-super-d2}, $q_{\theta-2 \theta-2}=-1$, $N\neq 4$}\label{subsubsec:type-D-super-f-Nneq4}

The Nichols algebra $\toba_{\bq}$ is generated by $(x_i)_{i\in \I}$ with defining
relations
\begin{align}\label{eq:rels-type-D-super-f-Nneq4}
\begin{aligned}
&x_{iii\pm1}=0, \   i \in \I_{\theta - 2}-\Jb ; && x_{\theta\theta\theta-2} =0;  && x_{\theta-1\theta-1\theta-2}=0; 
\\
& [x_{(i-1i+1)},x_i]_c=0, \  i\in\I_{\theta - 2}\cap \Jb; & &x_{\theta-1\theta}=0; & 
& [x_{\theta-3\, \theta-2 \, \theta},x_{\theta-2}]_c=0;
\\
& x_{ij} = 0, \  i < j - 1, \theta-2;
&&
x_i^2=0, \   i\in \Jb; & 
& x_{\alpha}^{N_\alpha}=0, \  \alpha\in\Oc_+^{\bq}.
\end{aligned}
\end{align}

If $N = \infty$, i.e. $q\notin \G_{\infty}$, then we omit the last set of relations.

\smallbreak \subsubsection{The  Dynkin diagram \eqref{eq:dynkin-Dtheta-super-d2}, $q_{\theta-2 \theta-2}=-1$, $N=4$}\label{subsubsec:type-D-super-f-N=4}

The Nichols algebra $\toba_{\bq}$ is generated by $(x_i)_{i\in \I}$ with defining
relations
\begin{align}\label{eq:rels-type-D-super-f-N=4}
\begin{aligned}
&x_{\theta-1\theta}=0; & &x_{ij} = 0,&& i < j - 1, \theta-2; 
\\
& [x_{\theta-3\, \theta-2 \, \theta},x_{\theta-2}]_c=0; &
& [x_{(i-1i+1)},x_i]_c=0, && i\in \I_{\theta - 2} \cap \Jb;
\\
& x_{\theta\theta\theta-2}=0; &
& x_{iii\pm1}=0, &&   i \in\I_{\theta - 2}-\Jb ; 
\\
& x_{\theta-1\theta-1\theta-2}=0; & 
& x_i^2=0, \quad i\in \Jb ; & 
& x_{\alpha}^{N_\alpha} =0, \ \alpha\in\Oc_+^{\bq}.
\end{aligned}
\end{align}

\smallbreak \subsubsection{The associated Lie algebra and $\ya$}\label{subsubsec:type-CD-Lie-alg} 
If $N$ is odd (respectively even), then the corresponding Lie algebra is of type 
$D_{j} \times C_{\theta-j}$ (respectively $D_{j}\times B_{\theta-j}$).
We present $\ya$ for each generalized Dynkin diagram.

\begin{itemize}[leftmargin=*]\renewcommand{\labelitemi}{$\circ$}

\item For \eqref{eq:dynkin-Dtheta-super-c},
\begin{align*}
\ya &=
\sum_{\substack{i\le j\in\I_{\theta},\\  \alpha_{ij} \text{ odd}}} \alpha_{ij} + 
\sum_{\substack{i\le j\in\I_{\theta},\\  \alpha_{ij} \text{ even}}} (N-1)\alpha_{ij} +
\sum_{\substack{i\in\I_{\theta},\\  \alpha_{i\theta-1} \text{ even}}} (M-1)(\alpha_{i,\theta-1} + \alpha_{i, \theta}) \\ & + 
\sum_{\substack{i< j\in\I_{\theta-1},\\  \alpha_{ij-1} \text{ odd}}} (\alpha_{i\theta}+\alpha_{j\theta-1}) + 
\sum_{\substack{i< j\in\I_{\theta-1},\\  \alpha_{ij-1} \text{ even}}} (N-1)(\alpha_{i\theta}+\alpha_{j\theta-1}).
\end{align*}

\item The expression of $\ya$ for \eqref{eq:dynkin-Dtheta-super-ct} follows from the previous case by changing the numeration of the diagram by $\theta-1 \longleftrightarrow\theta$.

\item For \eqref{eq:dynkin-Dtheta-super-d1} or \eqref{eq:dynkin-Dtheta-super-d2},
\begin{align*}
\ya &=\sum_{\substack{i\le j\in\I_{\theta},\\  \alpha_{ij} \text{ odd}}} \alpha_{ij} + 
\sum_{\substack{i\le j\in\I_{\theta},\\  \alpha_{ij} \text{ even}}} (N-1)\alpha_{ij} +
\sum_{\substack{i\in\I_{\theta-2},\\  \alpha_{i\theta-1} \text{ odd}}} (N-1)(\alpha_{i\theta-1}+\alpha_{\theta}) \\ & + 
\sum_{\substack{i\in\I_{\theta-2},\\  \alpha_{i\theta-2} \text{ even}}} (\alpha_{i\theta-2}+\alpha_{\theta}) +
\sum_{\substack{i< j\in\I_{\theta-2},\\  \alpha_{ij-1} \text{ odd}}} (\alpha_{i\theta}+\alpha_{j\theta-2}) \\ & + 
\sum_{\substack{i< j\in\I_{\theta-2},\\  \alpha_{ij-1} \text{ even}}} (N-1)(\alpha_{i\theta}+\alpha_{j\theta-2}) + 
\sum_{\substack{i\in\I_{\theta-2},\\  \alpha_{i\theta-1} \text{ odd}}} (M-1)(\alpha_{i\theta}+\alpha_{i\theta-2}).
\end{align*}

\end{itemize}

\subsubsection{Example $\Dtt_{4,2}$} \label{subsubsec:type-D-super-exa42} 
Here $\att_{3, 2} = \{ \{2\}, \{1,3\}, \I_3\}$. We exemplify the incarnation when $j = \theta - j$ in the case $\theta = 4$, $j = 2$.
Here is the basic datum:
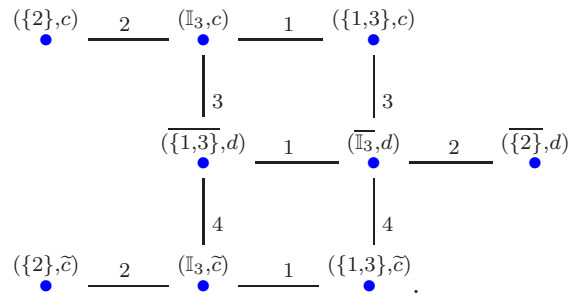
\begin{align*}
\xymatrix{ \overset{(\{2\},c)}{\vtxgpd} \ar@{-}^{2}[r] & \overset{(\I_3,c)}{\vtxgpd} \ar@{-}^{1}[r] \ar@{-}^{3}[d] & \overset{(\{1,3\},c)}{\vtxgpd} \ar@{-}^{3}[d] & 
\\
& \overset{(\overline{\{1,3\}},d)}{\vtxgpd} \ar@{-}^{1}[r] \ar@{-}^{4}[d] & \overset{(\overline{\I_3},d)}{\vtxgpd} \ar@{-}^{4}[d] \ar@{-}^{2}[r] & \overset{(\overline{\{2\}},d)}{\vtxgpd} 
\\
\overset{(\{2\},\widetilde{c})}{\vtxgpd} \ar@{-}^{2}[r] & \overset{(\I_3,\widetilde{c})}{\vtxgpd} \ar@{-}^{1}[r] & \overset{(\{1,3\},\widetilde{c})}{\vtxgpd}. & }
\end{align*}
To describe the incarnation, we need the matrices $(\bq^{(i)})_{i\in\I_6}$, from left to right and  from up to down:
\begin{align*}
&\xymatrix@C-9pt{ 
\overset{q^{\text{-}1}}{\circ}\ar@{-}[r]^{q}  &
\overset{\text{-}1}{\circ} \ar@{-}[r]^{q^{\text{-}1}}  & 
\overset{q}{\circ} \ar@{-}[r]^{q^{\text{-}2}}  & 
\overset{q^2}{\circ}, }
&&
\xymatrix@C-9pt{ 
\overset{\text{-}1}{\circ}\ar@{-}[r]^{q^{\text{-}1}}  &
\overset{\text{-}1}{\circ} \ar@{-}[r]^{q}  & 
\overset{\text{-}1}{\circ} \ar@{-}[r]^{q^{\text{-}2}}  & 
\overset{q^2}{\circ}, }
&
&\xymatrix@C-9pt{ 
\overset{\text{-}1}{\circ}\ar@{-}[r]^{q}  &
\overset{q^{\text{-}1}}{\circ} \ar@{-}[r]^{q}  & 
\overset{\text{-}1}{\circ} \ar@{-}[r]^{q^{\text{-}2}}  & 
\overset{q^2}{\circ}, }
\\&
\xymatrix{ 
& \overset{\text{-}1}{\circ}\ar@{-}[d]_{q^{\text{-}1}} \ar@{-}[rd]^{q^2} & \\
\overset{\text{-}1}{\circ} \ar@{-}[r]^{q^{\text{-}1}}  & 
\overset{q}{\circ} \ar@{-}[r]^{q^{\text{-}1}}  & 
\overset{\text{-}1}{\circ}, }
&&\xymatrix{ 
& \overset{\text{-}1}{\circ}\ar@{-}[d]_{q^{\text{-}1}} \ar@{-}[rd]^{q^2} & \\
\overset{\text{-}1}{\circ} \ar@{-}[r]^{q}  & 
\overset{\text{-}1}{\circ} \ar@{-}[r]^{q^{\text{-}1}}  & 
\overset{\text{-}1}{\circ}, }
&&
\xymatrix{ 
& \overset{q^{\text{-}1}}{\circ}\ar@{-}[d]_{q} & \\
\overset{q}{\circ} \ar@{-}[r]^{q^{\text{-}1}}  & 
\overset{\text{-}1}{\circ} \ar@{-}[r]^{q}  & 
\overset{q^{\text{-}1}}{\circ}. }
\end{align*}
 Now, this is the incarnation:
\begin{align*}
\xymatrix{ \overset{\bq^{(1)}}{\vtxgpd} \ar@{-}^{2}[r] & \overset{\bq^{(2)}}{\vtxgpd} \ar@{-}^{1}[r] \ar@{-}^{3}[d] & \overset{\bq^{(3)}}{\vtxgpd} \ar@{-}^{3}[d] & 
\\
& \overset{\bq^{(4)}}{\vtxgpd} \ar@{-}^{1}[r] \ar@{-}^{4}[d] & \overset{\bq^{(5)}}{\vtxgpd} \ar@{-}^{4}[d] \ar@{-}^{2}[r] & \overset{\bq^{(6)}}{\vtxgpd} 
\\
\overset{s_{34}(\bq^{(1)})}{\vtxgpd} \ar@{-}^{2}[r] & \overset{s_{34}(\bq^{(2)})}{\vtxgpd} \ar@{-}^{1}[r] & \overset{s_{34}(\bq^{(3)})}{\vtxgpd}. & }
\end{align*}

\subsection{Type $\superda{\alpha}$}\label{subsec:type-D2-1-alpha}
Here , $q,r,s\neq 1$, $qrs=1$.
$\superda{\alpha}$, $\alpha\neq 0,-1$, is a Lie superalgebra
of superdimension $9|8$ \cite[Proposition 2.5.6]{K-super}.
There exist 4 pairs $(A, \pa)$ of (families of) matrices and parity vectors as in \S \ref{subsec:Weyl-gpd-super}
such that the corresponding contragredient Lie superalgebra is isomorphic to $\superda{\alpha}$.

\subsubsection{Basic datum and root system}

The basic datum $(\cX, \rho)$, where $\cX = \{a_j: j \in \I_4\}$; 
and the bundle $(C^{a_j})_{j\in \I_4}$ of Cartan matrices are described by the following diagram:

\begin{align*}
\xymatrix@R-8pt{& \overset{A_3}{\underset{a_1}{\vtxgpd}} \ar@{-}^-{2 \qquad} [d]  & 
\\
\overset{s_{12}(A_3)}{\underset{a_2}{\vtxgpd}} \ar@{-}^{1} [r]
& \overset{A_2^{(1)}}{\underset{a_3}{\vtxgpd}} \ar@{-}^{3} [r]
& \overset{s_{23}(A_3)}{\underset{a_4}{\vtxgpd}}.
}
\end{align*}
Here the numeration of $A_3$ is as in \eqref{eq:dynkin-system-A} while $A_2^{(1)}$ is as in \eqref{eq:mTn} and below.
Using the notation \eqref{eq:notation-root-exceptional}, the bundle of root sets is the following:
\begin{align*}
\varDelta_{+}^{a_1} &= \{ 1,12,123,12^23,2,23,3 \}, & 
\varDelta_{+}^{a_2} &= s_{12}(\varDelta_{+}^{a_1}), \\
\varDelta_{+}^{a_3} &= \{1,12,13,123,2,23,3 \}, & 
\varDelta_{+}^{a_4} &= s_{23}(\varDelta_{+}^{a_1}).
\end{align*}
We denote this generalized root system by $\mathtt D(2,1)$.

\subsubsection{Weyl groupoid}
The isotropy group  at $a_1 \in \cX$ is 
\begin{align*}
\cW(a_1)= \langle \varsigma_1^{a_1}, \varsigma_2^{a_1}\varsigma_3 \varsigma_1\varsigma_3 \varsigma_2, \varsigma_3^{a_1} \rangle \simeq \Z/2 \times \Z/2 \times \Z/2 \leq  GL(\Z^\I).
\end{align*}

\subsubsection{Lie superalgebras realizing this generalized root system}

\

To describe the incarnation in the setting of Lie superalgebras, we need parity vectors $\pa=(1,-1,1)$, $\pa'=(-1,-1,-1) \in\G_2^3$, and matrices 
\begin{align*}
\dn{\alpha}&:= \begin{pmatrix} 2 & -1 & 0 \\ 1 & 0 & \alpha \\ 0 & -1 & 2 \end{pmatrix}   ,&
\dnp{\alpha} &:= \begin{pmatrix} 0 & 1 & -1-\alpha \\ 1 & 0 & \alpha \\ 1+\alpha & -\alpha & 0 \end{pmatrix}, & 
\alpha\neq\{0,-1\}.
\end{align*}

Let $\alpha\neq\{0,-1\}$. The assignment 
\begin{align}\label{eq:incarnation-D2-1-alpha-Lie}
\begin{aligned}
a_1 &\longmapsto \big(\dn{\alpha}, \pa\big),
&
a_2 &\longmapsto s_{12 }\big(\dn{-1-\alpha}, \pa\big),
\\
a_3 &\longmapsto \big(\dnp{\alpha}, \pa' \big),
& 
a_4 &\longmapsto s_{23}\left(\dn{-1-\alpha^{-1}}, \pa\right),
\end{aligned}
\end{align}
provides an isomorphism of generalized root systems, cf. \S \ref{subsec:Weyl-gpd-super}. Moreover,
the Lie superalgebras associated to $\big(\dn{\alpha}, \pa\big)$, $\big(\dn{\beta}, \pa\big)$ if and only if
$$ \beta\in\{ \alpha^{\pm 1}, -(1+\alpha)^{\pm 1}, -(1+\alpha^{-1})^{\pm 1}\}. $$

\subsubsection{Incarnation}

Here is an incarnation of $\mathtt D(2,1)$:
\begin{align}\label{eq:dynkin-D2-1-alpha}
\begin{aligned}
a_1 &\longmapsto \xymatrix{ \overset{q}{\underset{1 }{\circ}}\ar  @{-}[r]^{q^{-1}}  &
\overset{-1}{\underset{2 }{\circ}} \ar  @{-}[r]^{r^{-1}}  & \overset{r}{\underset{3}{\circ}}},&
a_2 &\longmapsto \xymatrix{ \overset{q}{\underset{2 }{\circ}}\ar  @{-}[r]^{q^{-1}}  &
\overset{-1}{\underset{1 }{\circ}} \ar  @{-}[r]^{s^{-1}}  & \overset{s}{\underset{3}{\circ}}},
\\
a_3 &\longmapsto \xymatrix@C-4pt@R-10pt{
&\overset{-1} {\underset{3}{\circ}} \ar  @{-}[dl]_{s} \ar  @{-}[dr]^{r} & \\
\overset{-1}{\underset{1}{\circ}} \ar  @{-}[rr]^{q}& & \overset{-1}{\underset{2}{\circ}}},
& a_4 &\longmapsto \xymatrix{ \overset{r}{\underset{2}{\circ}}\ar  @{-}[r]^{r^{-1}}  &
\overset{-1}{\underset{3}{\circ}} \ar  @{-}[r]^{s^{-1}}  & \overset{s}{\underset{1}{\circ}}}.
\end{aligned}
\end{align}

We set $N=\ord q$ (as always), $M=\ord r$, $L=\ord s$. Also, 
\begin{align*}
x_{123,2} := [x_{123},x_2]_c.
\end{align*}

\subsubsection{The generalized Dynkin diagram \emph{(\ref{eq:dynkin-D2-1-alpha} a)}, $q,r,s\neq -1$}\label{subsubsec:type-D2-1-alpha-a-neq-1}
The set
\begin{multline*}
\{ x_{3}^{n_1} x_{23}^{n_2} x_{2}^{n_{3}} x_{123,2}^{n_4} x_{123}^{n_5} x_{12}^{n_6} x_1^{n_7} \, | \, 0\le n_{1}<M, \, 0\le n_{4}<L, \, 0\le n_7<N, \\ 0\le n_2, n_{3}, n_{5}, n_6 <2\}.
\end{multline*}
is a PBW-basis of $\toba_{\bq}$.  If $N, M, L<\infty$, then
\begin{align*}
\dim \toba_{\bq}= 2^4LMN.
\end{align*}
If exactly two, respectively all, of $N, L, M$ are $\infty$, then
\begin{align*}
\GK \toba_{\bq}= 2, \text{ respectively } 3.
\end{align*}

\medskip
The set of positive Cartan roots is 
$\Oc_+^{\bq} = \{1,12^23,3\}$. The Nichols algebra $\toba_{\bq}$ is generated by $(x_i)_{i\in \I_3}$ with defining
relations
\begin{align}\label{eq:rels-type-D2-1-alpha-a-neq-1}
\begin{aligned}
x_1^N&=0; & x_2^2&=0; & x_3^M&=0; & x_{123,2}^L&=0;\\
x_{112}&=0; & x_{332}&=0; & x_{13}&=0.
\end{aligned}
\end{align}

If $N = \infty$, i.e. $q\notin \G_{\infty}$, respectively $L = \infty$, $M= \infty$, then we omit the relation where it appears as exponent.

\medskip
The degree of the integral is $\ya = (L+N)\alpha_1 + (2L+2)\alpha_2 + (L+M)\alpha_3$.

\subsubsection{The generalized Dynkin diagram \emph{(\ref{eq:dynkin-D2-1-alpha} a)}, $q=-1,r,s\neq -1$}\label{subsubsec:type-D2-1-alpha-a-q=-1}
The Nichols algebra $\toba_{\bq}$ is generated by $(x_i)_{i\in \I_3}$ with defining
relations
\begin{align}\label{eq:rels-type-D2-1-alpha-b-q=-1}
\begin{aligned}
x_1^2&=0; & x_2^2&=0; & x_3^M&=0; & x_{123,2}^L&=0;\\
x_{12}^2&=0; & x_{332}&=0; & x_{13}&=0.
\end{aligned}
\end{align}

If $L = \infty$, respectively  $M= \infty$, then we omit the relation where it appears as exponent.
The PBW basis, the dimension, the GK-dimension, the set of Cartan roots and $\ya$ are as in \S \ref{subsubsec:type-D2-1-alpha-a-neq-1}.

\subsubsection{The generalized Dynkin diagrams \emph{(\ref{eq:dynkin-D2-1-alpha}
b and d)}}

These diagrams are of the shape of (\ref{eq:dynkin-D2-1-alpha} a) but with $s$
interchanged with $q$, respectively with  $r$. Hence 
the corresponding Nichols algebras are as in \S 
\ref{subsubsec:type-D2-1-alpha-a-neq-1}  and \S \ref{subsubsec:type-D2-1-alpha-a-q=-1}.

\subsubsection{The generalized Dynkin diagram \emph{(\ref{eq:dynkin-D2-1-alpha}
c)}}\label{subsubsec:type-D2-1-alpha-c}

The set
\begin{multline*}
\{ x_{3}^{n_1} x_{23}^{n_2} x_{2}^{n_3} x_{123}^{n_4} x_{13}^{n_5} x_{12}^{n_6} x_1^{n_7} \, | \, 0\le n_{2}<M, \, 0\le n_{6}<L, \, 0\le n_5<N, \\ 0\le n_1, n_{3}, n_{4}, n_7 <2\}.
\end{multline*}
is a PBW-basis of $\toba_{\bq}$. The dimension and the GK-dimension are as in \S \ref{subsubsec:type-D2-1-alpha-a-neq-1}.

\medskip
The set of positive Cartan roots is 
$\Oc_+^{\bq} = \{12,13,23\}$. The Nichols algebra $\toba_{\bq}$ is generated by $(x_i)_{i\in \I_3}$ with defining
relations
\begin{align}\notag
x_1^2&=0; & x_2^2&=0; & x_3^2&=0; \\
\label{eq:rels-type-D2-1-alpha-c} x_{12}^N&=0; & x_{23}^M&=0; & x_{13}^L&=0;
\end{align}
\vspace*{-0.7cm}
\begin{align*}
x_{(13)}-\frac{1-s}{q_{23}(1-r)}[x_{13},x_2]_c-q_{12}(1-s)x_2x_{13}=0.
\end{align*}

If $N = \infty$, i.e. $q\notin \G_{\infty}$, respectively $L = \infty$, $M= \infty$, then we omit the relation where it appears as exponent. 
The degree of the integral is $$\ya = (L+N)\alpha_1 + (L+M)\alpha_2 + (M+N)\alpha_3.$$

\subsubsection{The associated Lie algebra} This is of type $A_1\times A_1\times A_1$.

\subsection{Type $\superf$}\label{subsec:type-F-super}
Here $N > 3$.
$\superf$ is a Lie superalgebra of superdimension $24|16$ \cite[Proposition 2.5.6]{K-super}.
There exist 6 pairs $(A, \pa)$ of matrices and parity vectors as in \S \ref{subsec:Weyl-gpd-super} such that the corresponding
contragredient Lie superalgebra is isomorphic to $\superf$.

\subsubsection{Basic datum and root system}
Below, $A_4$, $C_4$, $F_4$ and ${}_2T$ are numbered as in \eqref{eq:dynkin-system-A}, \eqref{eq:dynkin-system-C}, \eqref{eq:dynkin-system-F} and  
\eqref{eq:mTn}, respectively.
Also,  we denote 
\begin{align*}
\kappa_0 &= (1 4)(2 3),& \kappa_1 &=(1234), & \kappa_2 &= (234) \in \s_4. 
\end{align*}

The basic datum and the bundle of Cartan matrices are described the following diagram, that we call  $\Ftt(4)$:
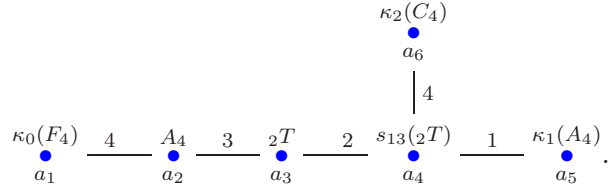
\begin{align*}
\xymatrix@R-8pt{& & & \overset{\kappa_2(C_4)}{\underset{a_6}{\vtxgpd}} \ar@{-}^-{4} [d]  & 
\\
\overset{\kappa_0(F_4)}{\underset{a_1}{\vtxgpd}} \ar@{-}^{4} [r]
&\overset{A_4}{\underset{a_2}{\vtxgpd}} \ar@{-}^{3} [r]
&\overset{{}_2T}{\underset{a_3}{\vtxgpd}} \ar@{-}^{2} [r]
& \overset{s_{13}({}_2T)}{\underset{a_4}{\vtxgpd}} \ar@{-}^{1} [r]
& \overset{\kappa_1(A_4)}{\underset{a_5}{\vtxgpd}}.  }
\end{align*}

Using the notation \eqref{eq:notation-root-exceptional}, the bundle of root sets is the following: {\scriptsize
\begin{align*}
\varDelta_{+}^{a_1}= & \{ 1, 12, 2, 123, 12^23^2, 123^2, 23,
23^2, 3, 12^23^34, 12^23^24, 123^24, 23^24, 12^23^34^2, 1234, 234, 34,
4 \}, \\
\varDelta_{+}^{a_2}= & \{ 1, 12, 2, 123, 23, 3, 12^23^34,
12^23^24, 123^24, 23^24, 12^23^34^2, 1234, 12^23^24^2, 
\\ & \qquad \qquad 123^24^2, 234, 23^24^2, 34,
4 \}, 
\\
\varDelta_{+}^{a_3}= & \{ 1, 12, 2, 123, 23, 3, 12^234,
12^24, 1234, 12^23^24^2, 234, 12^234^2, 34, 1234^2, 124, 234^2, 24,
4 \}, \\
\varDelta_{+}^{a_4}= & s_{13} (\{ 1, 12, 2, 12^23, 123, 12^23^2, 23,
3, 1^22^33^24, 12^33^24, 12^23^24, 12^234, 1234, 124, 234, 24, 34,
4 \}), \\
\varDelta_{+}^{a_5}= & \kappa_1(\{ 1, 12, 2, 123, 23, 3, 12^23^34,
12^23^24, 1^22^33^44^2, 123^24, 12^33^44^2, 12^23^44^2, 23^24, 12^23^34^2, 
\\ & \qquad \qquad 1234, 234, 34, 4 \}), 
\\
\varDelta_{+}^{a_6}= & \kappa_2 (\{ 1, 12, 12^2, 2, 12^23, 123, 23,
3, 1^22^33^24, 12^33^24, 12^23^24, 123^24, 12^234, 1234, 23^24, 234, 34,
4 \}).
\end{align*}
}

\subsubsection{Weyl groupoid}
\label{subsubsec:type-superF4-Weyl} 
The isotropy group  at $a_1 \in \cX$ is  
\begin{align*}
\cW(a_1)= \langle \varsigma_1^{a_1}, \varsigma_2^{a_1},  \varsigma_3^{a_1}, \varsigma_4^{a_1}\varsigma_3 \varsigma_2 \varsigma_1 \varsigma_4 \varsigma_1 
\varsigma_2 \varsigma_3 \varsigma_4 \rangle \simeq W(B_3) \times \Z/2 \leq  GL(\Z^\I).
\end{align*}

\subsubsection{Lie superalgebras realizing this generalized root system}

\

To describe the incarnation in the setting of Lie superalgebras, we need parity vectors $\pa_{\Jb}$ as in \eqref{eq:superstructure}, $\Jb\subset \I$, and matrices 
\begin{align*}
A_1&=\left( \begin{smallmatrix} 
2 & -1 & 0 & 0 \\ 
\text{--}1 & 2 & \text{--}1 & 0 \\ 
0 & \text{--}2 & 2 & \text{--}1 \\ 
0 & 0 & 1 & 0 \end{smallmatrix} \right), &
A_2&=\left( \begin{smallmatrix} 
2 & \text{--}1 & 0 & 0 \\ 
\text{--}1 & 2 & \text{--}1 & 0 \\ 
0 & \text{--}2 & 0 & 1 \\ 
0 & 0 & 1 & 0 \end{smallmatrix} \right), &
A_3&=\left( \begin{smallmatrix} 
2 & \text{--}1 & 0 & 0 \\ 
\text{--}1 & 0 & 2 & 1 \\ 
0 & \text{--}2 & 0 & 1 \\ 
0 & \text{--}1 & \text{--}1 & 2 \end{smallmatrix} \right), 
\\
A_4&=\left( \begin{smallmatrix} 
0 & 2 & 0 & \text{--}3 \\ 
2 & 0 & \text{--}2 & 1 \\ 
0 & \text{--}1 & 2 & 0 \\ 
3 & 1 & 0 & 0 \end{smallmatrix} \right), &
A_5&=\left( \begin{smallmatrix} 
2 & 0 & 0 & \text{--}1 \\ 
0 & 2 & \text{--}2 & \text{--}1 \\ 
0 & \text{--}1 & 2 & 0 \\ 
3 & \text{--}1 & 0 & 0 \end{smallmatrix} \right), &
A_6&=\left( \begin{smallmatrix} 
0 & 2 & 0 & 3 \\ 
\text{--}1 & 2 & \text{--}1 & 0 \\ 
0 & \text{--}1 & 2 & 0 \\ 
\text{--}1 & 0 & 0 & 2 \end{smallmatrix} \right).
\end{align*}
The assignment 
\begin{align}\label{eq:incarnation-F(4)-Lie}
\begin{aligned}
a_1 &\longmapsto \big(A_1, \pa_{\{4\}}\big),
&
a_2 &\longmapsto \big(A_2, \pa_{\{3,4\}}\big),
&
a_3 &\longmapsto \big(A_3, \pa_{\{2,3\}}\big),
\\ 
a_4 &\longmapsto \big(A_4, \pa_{\{1,2,4\}}\big),
&
a_5 &\longmapsto \big(A_5, \pa_{\{1\}}\big),
&
a_6 &\longmapsto \big(A_6, \pa_{\{4\}}\big),
\end{aligned}
\end{align}
provides an isomorphism of generalized root systems, cf. \S \ref{subsec:Weyl-gpd-super}.

\subsubsection{Incarnation}
Here it is:
\begin{align}\label{eq:dynkin-F4-super}
\begin{aligned}
a_1\longmapsto&\xymatrix@C-4pt{ \overset{\,\, q^2}{\underset{1}{\circ}}\ar  @{-}[r]^{q^{-2}}  & \overset{\,\,
q^2}{\underset{2}{\circ}} \ar  @{-}[r]^{q^{-2}}  & \overset{q}{\underset{3}{\circ}} &
\overset{-1}{\underset{4}{\circ}} \ar  @{-}[l]_{q^{-1}}} &
a_2\longmapsto&\xymatrix@C-4pt{ \overset{\,\, q^2}{\underset{1}{\circ}}\ar  @{-}[r]^{q^{-2}}  & \overset{\,\,
q^2}{\underset{2}{\circ}} \ar  @{-}[r]^{q^{-2}}  & \overset{-1}{\underset{3}{\circ}} &
\overset{-1}{\underset{4}{\circ}} \ar  @{-}[l]_{q}}
\\
a_3\longmapsto&\xymatrix@R-6pt@C-4pt{ &  \overset{q}{\underset{4}{\circ}}\ar  @{-}[d]_{q^{-1}} \ar
@{-}[rd]^{q^{-1}} &
\\ \overset{q^2}{\underset{1}{\circ}} \ar  @{-}[r]^{q^{-2}}  &  \overset{-1}{\underset{2}{\circ}} \ar  @{-}[r]^{q^{2}}  
& \overset{-1}{\underset{3}{\circ}}} &
a_4\longmapsto&\xymatrix@R-6pt@C-4pt{  & \overset{-1}{\underset{4}{\circ}} \ar  @{-}[d]^{q} \ar  @{-}[ld]_{q^{-3}} &
\\ 
\overset{-1}{\underset{1}{\circ}} \ar  @{-}[r]^{q^{2}}  & \overset{-1}{\underset{2}{\circ}} & \overset{q^2}{\underset{3}{\circ}} \ar  @{-}[l]_{q^{-2}}  }
\\
a_5\longmapsto&\xymatrix@C-4pt{ \overset{\,\, q^{-3}}{\underset{4}{\circ}}\ar  @{-}[r]^{q^{3}}  
& \overset{-1}{\underset{1}{\circ}} \ar  @{-}[r]^{q^{-2}}  
& \overset{q^2}{\underset{2}{\circ}}\ar @{-}[r]^{q^{-2}} & \overset{q^{2}}{\underset{3}{\circ}}}&
a_6\longmapsto& \xymatrix@C-4pt{ \overset{\,\, q^{-3}}{\underset{1}{\circ}}\ar  @{-}[r]^{q^{3}}  
& \overset{-1}{\underset{4}{\circ}} \ar  @{-}[r]^{q^{-1}}  
& \overset{q}{\underset{2}{\circ}}\ar @{-}[r]^{q^{-2}} & \overset{q^{2}}{\underset{3}{\circ}}}.
\end{aligned}
\end{align}

\subsubsection{PBW-basis and (GK-)dimension}\label{subsubsec:type-F4-super-PBW}
Notice that the roots in each $\varDelta_{+}^{a_i}$, $i\in\I_6$, are ordered from left to right, justifying the notation $\beta_1, \dots, \beta_{18}$.

The root vectors $x_{\beta_k}$ are described as in Remark \ref{rem:lyndon-word}.
Thus
\begin{align*}
\left\{ x_{\beta_{18}}^{n_{18}} x_{\beta_{17}}^{n_{17}} \dots x_{\beta_2}^{n_{2}}  x_{\beta_1}^{n_{1}} \, | \, 0\le n_{k}<N_{\beta_k} \right\}.
\end{align*}
is a PBW-basis of $\toba_{\bq}$. Let $L=\ord q^3$, $M=\ord q^2$. If $N<\infty$, then
\begin{align*}
\dim \toba_{\bq}= 2^8LM^3N^6.
\end{align*}
If $N=\infty$ (that is, if $q$ is not a root of unity), then
$\GK \toba_{\bq}= 10$.

\subsubsection{The generalized Dynkin diagram \emph{(\ref{eq:dynkin-F4-super}
a)}, $N> 4$}\label{subsubsec:type-F4-super-a-Nneq4}

The Nichols algebra $\toba_{\bq}$ is generated by $(x_i)_{i\in \I_4}$ with defining
relations
\begin{align}\label{eq:rels-type-F4-super-a-Nneq4}
\begin{aligned}
x_{13}&=0; & x_{14}&=0; & x_{24}&=0; \\
x_{112}&=0; & x_{221}&=0; & x_{223}&=0; \quad x_{334}=0;\\
x_{3332}&=0; & x_{4}^2&=0; & x_\alpha^{N_\alpha}&=0, \ \alpha\in\Oc_+^{\bq};
\end{aligned}
\end{align}
where $\Oc_+^{\bq}=\{1,12, 2, 123, 12^23^2, 123^2, 23, 23^2, 3, 12^23^34^2 \}$.
If $N = \infty$, i.e. $q\notin \G_{\infty}$, then we omit the last set of relations. 
Here
\begin{multline*}
\ya = (L+4M+N-2)\alpha_1 + (2L+6M+2N-2)\alpha_2 \\ + (3L+6M+3N)\alpha_3
+ (2L+6)\alpha_4
\end{multline*}

\subsubsection{The generalized Dynkin diagram \emph{(\ref{eq:dynkin-F4-super}
a)}, $N =4$}\label{subsubsec:type-F4-super-a-N=4}

The Nichols algebra $\toba_{\bq}$ is generated by $(x_i)_{i\in \I_4}$ with defining
relations
\begin{align}\label{eq:rels-type-F4-super-a-N=4}
\begin{aligned}
x_{13}&=0; & x_{14}&=0; & &[x_{(13)},x_2]_c=0; \\
x_{334}&=0; & x_{24}&=0; & &[x_{23},x_{(24)}]_c=0; \\
x_{3332}&=0; & x_{4}^2&=0; & &x_\alpha^{N_\alpha}=0, \ \alpha\in\Oc_+^{\bq};
\end{aligned}
\end{align}
here, $\Oc_+^{\bq}$, $\ya$  are as in \S \ref{subsubsec:type-F4-super-a-Nneq4}.

\subsubsection{The generalized Dynkin diagram \emph{(\ref{eq:dynkin-F4-super}
b)}, $N> 4$}\label{subsubsec:type-F4-super-b-Nneq4}

The Nichols algebra $\toba_{\bq}$ is generated by $(x_i)_{i\in \I_4}$ with defining
relations
\begin{align}\label{eq:rels-type-F4-super-b-Nneq4}
\begin{aligned}
x_{13}&=0; & x_{14}&=0; & &x_{24}=0; \quad x_{112}=0;\\
x_{221}&=0; & x_{223}&=0; & &[[x_{43},x_{432}]_c,x_3]_c=0;\\
x_{3}^2&=0; & x_{4}^2&=0; & &x_\alpha^{N_\alpha}=0, \ \alpha\in\Oc_+^{\bq};
\end{aligned}
\end{align}
where $\Oc_+^{\bq}=\{1, 12, 2, 12^23^34, 1234, 12^23^24^2, 123^24^2, 234, 23^24^2, 34 \}$.
If $N = \infty$, i.e. $q\notin \G_{\infty}$, then we omit the last set of relations. 
Here
\begin{multline*}
\ya= (L+4M+N-2)\alpha_1 + (2L+6M+2N-2)\alpha_2\\ + (3L+6M+3N)\alpha_3 + (L+6M+3N-4)\alpha_4.
\end{multline*}

\subsubsection{The generalized Dynkin diagram \emph{(\ref{eq:dynkin-F4-super}
b)}, $N =4$}\label{subsubsec:type-F4-super-b-N=4}

The Nichols algebra $\toba_{\bq}$ is generated by $(x_i)_{i\in \I_4}$ with defining
relations
\begin{align}\label{eq:rels-type-F4-super-b-N=4}
\begin{aligned}
x_{13}&=0; & x_{14}&=0; & &[x_{(13)},x_2]_c=0; \\
x_{24}&=0; & x_{23}^2&=0; & &[[x_{43},x_{432}]_c,x_3]_c=0; \\
x_{3}^2&=0; & x_{4}^2&=0; & &x_\alpha^{N_\alpha}=0, \ \alpha\in\Oc_+^{\bq};
\end{aligned}
\end{align}
here, $\Oc_+^{\bq}$, $\ya$  are as in \S \ref{subsubsec:type-F4-super-b-Nneq4}.

\subsubsection{The generalized Dynkin diagram \emph{(\ref{eq:dynkin-F4-super}
c)}, $N>4$}\label{subsubsec:type-F4-super-c-Nneq4}

The Nichols algebra $\toba_{\bq}$ is generated by $(x_i)_{i\in \I_4}$ with defining
relations
\begin{align}\label{eq:rels-type-F4-super-c-Nneq4}
\begin{aligned}
x_{13}&=0; & x_{14}&=0; & &x_{112}=0; \\
x_{442}&=0; & x_{443}&=0; & &[x_{(13)},x_2]_c=0; \\
x_{2}^2&=0; & x_{3}^2&=0; & &x_\alpha^{N_\alpha}=0, \ \alpha\in\Oc_+^{\bq};
\end{aligned}
\end{align}
\vspace*{-0.3cm}
\begin{align*}
x_{(24)}- q_{34}q[x_{24},x_3]_c-q_{23}(1-q^{-1})x_3x_{24}=0;
\end{align*}
where $\Oc_+^{\bq}=\{1, 123, 23, 12^24, 1234, 12^23^24^2, 234, 1234^2, 234^2, 4 \}$.
If $N = \infty$, i.e. $q\notin \G_{\infty}$, then we omit the relations $x_\alpha^{N_\alpha}=0,$ $\alpha\in\Oc_+^{\bq}$. 
Here
\begin{multline*}
\ya= (L+4M+N-2)\alpha_1 + (2L+6M+2N-2)\alpha_2\\ + (6M+2N-4)\alpha_3 + (L+6M+3N-4)\alpha_4.
\end{multline*}

\subsubsection{The generalized Dynkin diagram \emph{(\ref{eq:dynkin-F4-super}
c)}, $N =4$}\label{subsubsec:type-F4-super-c-N=4}

The Nichols algebra $\toba_{\bq}$ is generated by $(x_i)_{i\in \I_4}$ with defining
relations
\begin{align}\label{eq:rels-type-F4-super-c-N=4}
\begin{aligned}
x_{13}&=0; & x_{14}&=0; & &x_{12}^2=0; \\
x_{442}&=0; & x_{443}&=0; & &[x_{(13)},x_2]_c=0; \\
x_{2}^2&=0; & x_{3}^2&=0; & &x_\alpha^{N_\alpha}=0, \ \alpha\in\Oc_+^{\bq};
\end{aligned}
\end{align}
\vspace*{-0.3cm}
\begin{align*}
x_{(24)}- q_{34}q[x_{24},x_3]_c-q_{23}(1-q^{-1})x_3x_{24}=0;
\end{align*}
here, $\Oc_+^{\bq}$, $\ya$  are as in \S \ref{subsubsec:type-F4-super-c-Nneq4}.

\subsubsection{The generalized Dynkin diagram \emph{(\ref{eq:dynkin-F4-super}
d)}, $N> 4$}\label{subsubsec:type-F4-super-d-Nneq4}

The Nichols algebra $\toba_{\bq}$ is generated by $(x_i)_{i\in \I_4}$ with defining
relations
\begin{align}\label{eq:rels-type-F4-super-d-Nneq4}
\begin{aligned}
x_{13}&=0; & x_{14}&=0; & &[x_{124},x_2]_c=0; \\
x_{112}&=0; & x_{4}^2&=0; & &[[x_{32},x_{321}]_c,x_2]_c=0; \\
x_{2}^2&=0; & x_{3}^2&=0; & &x_\alpha^{N_\alpha}=0, \ \alpha\in\Oc_+^{\bq};
\end{aligned}
\end{align}
\vspace*{-0.3cm}
\begin{align*}
x_{(24)}+ q_{34}\frac{1-q^3}{1-q^2}[x_{24},x_3]_c-q_{23}(1-q^{-3})x_3x_{24}=0;
\end{align*}
where $\Oc_+^{\bq}=\{1,123, 12^23^2, 23, 1^22^33^24, 12^33^24, 12^234, 124, 24, 34\}$.
If $N = \infty$, i.e. $q\notin \G_{\infty}$, then we omit the relations $x_\alpha^{N_\alpha}=0,$ $\alpha\in\Oc_+^{\bq}$. 
Here
\begin{multline*}
\ya= (6M+2N-4)\alpha_1 + (10M+4N-6)\alpha_2\\ + (L+6M+3N-4)\alpha_3 + (L+4M+N-2)\alpha_4.
\end{multline*}

\subsubsection{The generalized Dynkin diagram \emph{(\ref{eq:dynkin-F4-super}
d)}, $N =4$}\label{subsubsec:type-F4-super-d-N=4}

The Nichols algebra $\toba_{\bq}$ is generated by $(x_i)_{i\in \I_4}$ with defining
relations
\begin{align}\label{eq:rels-type-F4-super-d-N=4}
\begin{aligned}
x_{13}&=0; & x_{14}&=0; & & [x_{124},x_2]_c=0; \\
x_{12}^2&=0; & x_{4}^2&=0; & & [[x_{32},x_{321}]_c,x_2]_c=0; \\
x_{2}^2&=0; & x_{3}^2&=0; & & x_\alpha^{N_\alpha}=0, \ \alpha\in\Oc_+^{\bq};
\end{aligned}
\end{align}
\vspace*{-0.3cm}
\begin{align*}
x_{(24)}+ q_{34}\frac{1-q^3}{1-q^2}[x_{24},x_3]_c-q_{23}(1-q^{-3})x_3x_{24}=0;
\end{align*}
here, $\Oc_+^{\bq}$, $\ya$  are as in \S \ref{subsubsec:type-F4-super-d-Nneq4}.

\subsubsection{The generalized Dynkin diagram \emph{(\ref{eq:dynkin-F4-super}
e)}, $N\neq 4,6$}\label{subsubsec:type-F4-super-e-Nneq46}

The Nichols algebra $\toba_{\bq}$ is generated by $(x_i)_{i\in \I_4}$ with defining
relations
\begin{align}\label{eq:rels-type-F4-super-e-Nneq46}
\begin{aligned}
x_{13}&=0; & x_{14}&=0; & &x_{24}=0; \quad x_{112}=0; \\
x_{221}&=0; & x_{223}&=0; &
& [[[x_{432},x_3]_c,[x_{4321},x_3]_c]_c,x_{32}]_c=0;\\
x_{443}&=0; & x_{3}^2&=0; & & x_\alpha^{N_\alpha}=0, \ \alpha\in\Oc_+^{\bq};
\end{aligned}
\end{align}
where $\Oc_+^{\bq}=\{1,12, 2, 12^23^24, 1^22^33^44^2, 123^24, 12^33^44^2, 12^23^44^2, 23^24, 4 \}$.
If $N = \infty$, i.e. $q\notin \G_{\infty}$, then we omit the last set of relations. 
Here
\begin{multline*}
\ya= (6M+2N-4)\alpha_1 + (10M+4N-6)\alpha_2\\ + (12M+6N-6)\alpha_3 + (L+6M+3N-4)\alpha_4.
\end{multline*}

\subsubsection{The generalized Dynkin diagram \emph{(\ref{eq:dynkin-F4-super}
e)}, $N =6$}\label{subsubsec:type-F4-super-e-N=6}

The Nichols algebra $\toba_{\bq}$ is generated by $(x_i)_{i\in \I_4}$ with defining
relations
\begin{align}\label{eq:rels-type-F4-super-e-N=6}
\begin{aligned}
x_{13}&=0; & x_{14}&=0; & & x_{24}=0; \quad x_{112}=0;  \\
x_{221}&=0; & x_{223}&=0; &
& [[[x_{432},x_3]_c,[x_{4321},x_3]_c]_c,x_{32}]_c=0;\\
x_{34}^2&=0; & x_{3}^2&=0; & & x_\alpha^{N_\alpha}=0, \ \alpha\in\Oc_+^{\bq};
\end{aligned}
\end{align}
here, $\Oc_+^{\bq}$, $\ya$  are as in \S \ref{subsubsec:type-F4-super-f-Nneq46}.

\subsubsection{The generalized Dynkin diagram \emph{(\ref{eq:dynkin-F4-super}
e)}, $N =4$}\label{subsubsec:type-F4-super-e-N=4}

The Nichols algebra $\toba_{\bq}$ is generated by $(x_i)_{i\in \I_4}$ with defining
relations
\begin{align}\label{eq:rels-type-F4-super-e-N=4}
\begin{aligned}
x_{13}&=0; & x_{14}&=0; & & [x_{(13)},x_2]_c=0; \\
x_{24}&=0; & x_{443}&=0; &
& [[[x_{432},x_3]_c,[x_{4321},x_3]_c]_c,x_{32}]_c=0;\\
x_{23}^2&=0; & x_{3}^2&=0; & & x_\alpha^{N_\alpha}=0, \ \alpha\in\Oc_+^{\bq};
\end{aligned}
\end{align}
here, $\Oc_+^{\bq}$, $\ya$  are as in \S \ref{subsubsec:type-F4-super-e-Nneq46}.

\subsubsection{The generalized Dynkin diagram \emph{(\ref{eq:dynkin-F4-super}
f)}, $N\neq 4,6$}\label{subsubsec:type-F4-super-f-Nneq46}

The Nichols algebra $\toba_{\bq}$ is generated by $(x_i)_{i\in \I_4}$ with defining
relations
\begin{align}\label{eq:rels-type-F4-super-f-Nneq46}
\begin{aligned}
x_{13}&=0; & x_{14}&=0; & x_{24}&=0; \\
x_{112}&=0; & x_{2221}&=0; & x_{223}&=0; \\
x_{443}&=0; & x_{3}^2&=0; & x_\alpha^{N_\alpha}&=0, \ \alpha\in\Oc_+^{\bq};
\end{aligned}
\end{align}
\vspace*{-0.3cm}
\begin{align*}
[[x_{(14)},x_2]_c,x_3]_c-q_{23}(q^2-q)[[x_{(14)},x_3]_c,x_2]_c=0;
\end{align*}
where $\Oc_+^{\bq}=\{ 1, 12, 12^2, 2, 1^22^33^24, 12^33^24, 12^23^24, 123^24, 23^24, 4 \}$.
If $N = \infty$, i.e. $q\notin \G_{\infty}$, then we omit the relations $x_\alpha^{N_\alpha}=0,$ $\alpha\in\Oc_+^{\bq}$. 
Here
\begin{multline*}
\ya= (6M+2N-4)\alpha_1 + (10M+4N-6)\alpha_2\\ + (8M+2N-2)\alpha_3 + (L+4M+N-2)\alpha_4.
\end{multline*}

\subsubsection{The generalized Dynkin diagram \emph{(\ref{eq:dynkin-F4-super}
f)}, $N =6$}\label{subsubsec:type-F4-super-f-N=6}

The Nichols algebra $\toba_{\bq}$ is generated by $(x_i)_{i\in \I_4}$ with defining
relations
\begin{align}\label{eq:rels-type-F4-super-f-N=6}
\begin{aligned}
x_{13}&=0; & x_{14}&=0; & x_{24}&=0; \\
x_{112}&=0; & x_{2221}&=0; & x_{223}&=0; \\
x_{34}^2&=0; & x_{3}^2&=0; & x_\alpha^{N_\alpha}&=0, \ \alpha\in\Oc_+^{\bq};
\end{aligned}
\end{align}
\vspace*{-0.3cm}
\begin{align*}
[[x_{(14)},x_2]_c,x_3]_c-q_{23}(q^2-q)[[x_{(14)},x_3]_c,x_2]_c=0;
\end{align*}
here, $\Oc_+^{\bq}$, $\ya$  are as in \S \ref{subsubsec:type-F4-super-f-Nneq46}.

\subsubsection{The generalized Dynkin diagram \emph{(\ref{eq:dynkin-F4-super}
f)}, $N =4$}\label{subsubsec:type-F4-super-f-N=4}

The Nichols algebra $\toba_{\bq}$ is generated by $(x_i)_{i\in \I_4}$ with defining
relations
\begin{align}\label{eq:rels-type-F4-super-f-N=4}
\begin{aligned}
x_{13}&=0; & x_{14}&=0; & x_{24}&=0; \\
x_{223}&=0; & x_{2221}&=0; & [x_{12}&,x_{(13)}]_c=0; \\
x_{443}&=0; & x_{3}^2&=0; & x_\alpha^{N_\alpha}&=0, \ \alpha\in\Oc_+^{\bq};
\end{aligned}
\end{align}
\vspace*{-0.3cm}
\begin{align*}
[[x_{(14)},x_2]_c,x_3]_c-q_{23}(q^2-q)[[x_{(14)},x_3]_c,x_2]_c=0;
\end{align*}
here, $\Oc_+^{\bq}$, $\ya$  are as in \S \ref{subsubsec:type-F4-super-f-Nneq46}.

\subsubsection{The associated Lie algebra} This is of type $A_1\times B_3$.

\subsection{Type  $\superg$}\label{subsec:type-G-super} Here $N > 3$.
$\superg$ is a Lie superalgebra (over a field of characteristic $\neq 2,3$)
of superdimension $17|14$ \cite[Proposition 2.5.6]{K-super}.
There exist 4 pairs $(A, \pa)$ of matrices and parity vectors as in \S \ref{subsec:Weyl-gpd-super}
such that the corresponding
contragredient Lie superalgebra is isomorphic to $\superg$.

\subsubsection{Basic datum and root system}
Below, $A_3$, $B_3$, $D_{4}^{(3)}$ and $T^{(2)}$ are numbered as in \eqref{eq:dynkin-system-A}, \eqref{eq:dynkin-system-B}, \eqref{eq:D43} and  
\eqref{eq:T2}, respectively.
Also,  we denote 
$\kappa = (123) \in \s_3$.
The basic datum and the bundle of Cartan matrices are described by the following diagram, that we call $\Gtt(3)$:
\begin{align*}
\xymatrix@R-8pt{ \overset{D_{4}^{(3)}}{\underset{a_1}{\vtxgpd}} \ar@{-}^{1} [r]
&\overset{A_3}{\underset{a_2}{\vtxgpd}} \ar@{-}^{2} [r]
&\overset{T^{(2)}}{\underset{a_3}{\vtxgpd}} \ar@{-}^{3} [r]
& \overset{\kappa(B_3)}{\underset{a_4}{\vtxgpd}}.  }
\end{align*}

Using the notation \eqref{eq:notation-root-exceptional}, the bundle of root sets is the following:
\begin{align*}
\varDelta_+^{a_1} &= \{ 1,12,123,12^23,12^33,12^33^2,12^43^2,2,23,2^23,3,2^33,2^33^2 \}, \\
\varDelta_+^{a_2} &= \{ 1,12,23,12^23,1^22^33,1^22^33^2,1^32^43^2,2,123,1^22^23,3,1^32^33,1^32^33^2 \}, \\
\varDelta_+^{a_3} &= \{ 12,1,3,13,1^23,1^223^2,1^323^2,2,123,1^223,23,1^323,1^32^23^2 \}, \\
\varDelta_+^{a_4} &= \{ 123^2,13,3,1,1^23,1^223,1^323^2,23,123,1^223^2,2,1^323^3,1^32^23^3 \}.
\end{align*}

\subsubsection{Weyl groupoid}
The isotropy group  at $a_1 \in \cX$ is  
\begin{align*}
\cW(a_1)= \langle \varsigma_1^{a_1}\varsigma_2 \varsigma_3\varsigma_1 \varsigma_3 \varsigma_2 \varsigma_1,
\varsigma_2^{a_1},  \varsigma_3^{a_1}   \rangle \simeq \Z/2  \times W(G_2).
\end{align*}

\subsubsection{Lie superalgebras realizing this generalized root system}
\

To describe the incarnation in the setting of Lie superalgebras, we need parity vectors $\pa_{\Jb}$ as in \eqref{eq:superstructure}, $\Jb\subset \I$, and matrices 
\begin{align*}
A_1&=\left( \begin{smallmatrix} 0 & 1 & 0 \\ -1 & 2 & -3 \\ 0 & -1 & 2 \end{smallmatrix} \right), &
A_2&=\left( \begin{smallmatrix} 0 & 1 & 0 \\ 1 & 0 & -3 \\ 0 & -1 & 2 \end{smallmatrix} \right), &
A_3&=\left( \begin{smallmatrix} 2 & -1 & -2 \\ 1 & 0 & -3 \\ 1 & 1 & 0 \end{smallmatrix} \right), &
A_4&=\left( \begin{smallmatrix} 2 & 0 & -2 \\ 0 & 2 & -1 \\ 1 & 3 & 0 \end{smallmatrix} \right).
\end{align*}
The assignment 
\begin{align}\label{eq:incarnation-G(3)-Lie}
\begin{aligned}
a_1 &\mapsto \big(A_1, \pa_{\{1\}}\big),
&
a_2 &\mapsto \big(A_2, \pa_{\{1,2\}}\big),
\\
a_3 &\mapsto \big(A_3, \pa_{\{2,3\}}\big),
&
a_4 &\mapsto \big(A_4, \pa_{\{1,3\}}\big),
\end{aligned}
\end{align}
provides an isomorphism of generalized root systems, cf. \S \ref{subsec:Weyl-gpd-super}.

\subsubsection{Incarnation}
Here it is:
\begin{align}\label{eq:dynkin-G-super}
\begin{aligned}
a_1 \longmapsto&\xymatrix{ \overset{-1}{\underset{1}{\circ}}\ar  @{-}[r]^{q^{-1}}  &
\overset{q}{\underset{2}{\circ}} \ar  @{-}[r]^{q^{-3}}  & \overset{\,\, q^3}{\underset{3}{\circ}},}&
a_2 \longmapsto&\xymatrix{ \overset{-1}{\underset{1}{\circ}} \ar  @{-}[r]^{q}  &
\overset{-1}{\underset{2}{\circ}} \ar  @{-}[r]^{q^{-3}}  & \overset{\,\, q^3}{\underset{3}{\circ}},}
\\
a_3 \longmapsto&\xymatrix@C-4pt{ & \overset{-1}{\underset{3}{\circ}} &
\\
\overset{q}{\underset{1}{\circ}} \ar  @{-}[ru]^{q^{-2}} \ar@{-}[rr]^{q^{-1}}
& &\overset{-1}{\underset{2}{\circ}} \ar  @{-}[ul]_{q^{3}},}
&
a_4 \longmapsto&\xymatrix{ \overset{-q^{-1}}{\underset{1}{\circ}} \ar  @{-}[r]^{q^2}  &
\overset{-1}{\underset{3}{\circ}} \ar  @{-}[r]^{q^{-3}}  & \overset{\,\, q^3}{\underset{2}{\circ}}.}
\end{aligned}
\end{align}

\subsubsection{PBW-basis and (GK-)dimension}\label{subsubsec:type-G3-super-PBW}
Notice that the roots in each $\varDelta_{+}^{a_i}$, $i\in\I_4$, are ordered from left to right, justifying the notation $\beta_1, \dots, \beta_{13}$.

The root vectors $x_{\beta_k}$ are described as in Remark \ref{rem:lyndon-word}.
Thus
\begin{align*}
\left\{ x_{\beta_{13}}^{n_{13}} x_{\beta_{12}}^{n_{12}} \dots x_{\beta_2}^{n_{2}}  x_{\beta_1}^{n_{1}} \, | \, 0\le n_{k}<N_{\beta_k} \right\}.
\end{align*}
is a PBW-basis of $\toba_{\bq}$. Let $L=\ord -q$, $M=\ord q^3$. If $N<\infty$, then
\begin{align*}
\dim \toba_{\bq}= 2^6LM^3N^3.
\end{align*}
If $N=\infty$ (that is, if $q$ is not a root of unity), then
$\GK \toba_{\bq}= 7$.

\subsubsection{The generalized Dynkin diagram \emph{(\ref{eq:dynkin-G-super}
a)}, $N\neq 4,6$}\label{subsubsec:type-G-super-a-Nneq46}

The Nichols algebra $\toba_{\bq}$ is generated by $(x_i)_{i\in \I_3}$ with defining
relations
\begin{align}\label{eq:rels-type-G-super-a-Nneq46}
\begin{aligned}
x_{13}&=0; & x_{221}&=0; &  x_{332}&=0; \\
x_{22223}&=0; & x_{1}^2&=0; &  x_{\alpha}^{N_\alpha}&=0, \ \alpha\in\Oc_+^{\bq};
\end{aligned}
\end{align}
where $\Oc_+^{\bq}=\{2,3,23,2^23,2^33,12^23,2^33^2\}$.
If $N = \infty$, i.e. $q\notin \G_{\infty}$, then we omit the last set of relations. 
Here
\begin{multline*}
\ya= (L+5)\alpha_1 + (2L+6M+4N)\alpha_2  + (L+4M+2N-1)\alpha_3.
\end{multline*}

\subsubsection{The generalized Dynkin diagram \emph{(\ref{eq:dynkin-G-super}
a)}, $N =6$}\label{subsubsec:type-G-super-a-N=6}

The Nichols algebra $\toba_{\bq}$ is generated by $(x_i)_{i\in \I_3}$ with defining
relations
\begin{align}\label{eq:rels-type-G-super-a-N=6}
\begin{aligned}
x_{13}&=0; & x_{221}&=0; & & [x_{12},x_{(13)}]_c=0; \\
x_{22223}&=0; & x_{1}^2&=0; & &x_{\alpha}^{N_\alpha}=0, \ \alpha\in\Oc_+^{\bq};
\end{aligned}
\end{align}
here, $\Oc_+^{\bq}$, $\ya$  are as in \S \ref{subsubsec:type-G-super-a-Nneq46}.

\subsubsection{The generalized Dynkin diagram \emph{(\ref{eq:dynkin-G-super}
a)}, $N =4$}\label{subsubsec:type-G-super-a-N=4}

The Nichols algebra $\toba_{\bq}$ is generated by $(x_i)_{i\in \I_3}$ with defining
relations
\begin{align}\label{eq:rels-type-G-super-a-N=4}
\begin{aligned}
x_{13}&=0; & x_{221}&=0; & [[[x_{(13)},&x_2]_c,x_2]_c,x_2]_c=0; \\
x_{332}&=0; & x_{1}^2&=0; &  x_{\alpha}^{N_\alpha}&=0, \ \alpha\in\Oc_+^{\bq};
\end{aligned}
\end{align}
here, $\Oc_+^{\bq}$, $\ya$  are as in \S \ref{subsubsec:type-G-super-a-Nneq46}.

\subsubsection{The generalized Dynkin diagram \emph{(\ref{eq:dynkin-G-super}
b)}, $N\neq 6$}\label{subsubsec:type-G-super-b-Nneq6}

The Nichols algebra $\toba_{\bq}$ is generated by $(x_i)_{i\in \I_3}$ with defining
relations
\begin{align}\label{eq:rels-type-G-super-b-Nneq6}
\begin{aligned}
x_{13}&=0; & x_{332}&=0; &  [[x_{12},&[x_{12},x_{(13)}]_c]_c,x_2]_c=0; \\
x_{1}^2&=0; & x_{2}^2&=0; &  x_{\alpha}^{N_\alpha}&=0, \ \alpha\in\Oc_+^{\bq};
\end{aligned}
\end{align}
where $\Oc_+^{\bq}=\{3,12,123,12^23,1^22^23,1^32^33,1^32^33^2\}$.
If $N = \infty$, i.e. $q\notin \G_{\infty}$, then we omit the last set of relations. 
Here
\begin{multline*}
\ya= (L+6M+4N-3)\alpha_1 + (2L+6M+4N)\alpha_2  + (L+4M+2N-1)\alpha_3.
\end{multline*}

\subsubsection{The generalized Dynkin diagram \emph{(\ref{eq:dynkin-G-super}
b)}, $N =6$}\label{subsubsec:type-G-super-b-N=6}

The Nichols algebra $\toba_{\bq}$ is generated by $(x_i)_{i\in \I_3}$ with defining
relations
\begin{align}\label{eq:rels-type-G-super-b-N=6}
\begin{aligned}
x_{13}&=0; & x_{23}^2&=0; &  [[x_{12},&[x_{12},x_{(13)}]_c]_c,x_2]_c=0; \\
x_{1}^2&=0; & x_{2}^2&=0; &  x_{\alpha}^{N_\alpha}&=0, \ \alpha\in\Oc_+^{\bq};
\end{aligned}
\end{align}
here, $\Oc_+^{\bq}$, $\ya$  are as in \S \ref{subsubsec:type-G-super-b-Nneq6}.

\subsubsection{The generalized Dynkin diagram \emph{(\ref{eq:dynkin-G-super}
c)}, $N\neq 6$}\label{subsubsec:type-G-super-c-Nneq6}

The Nichols algebra $\toba_{\bq}$ is generated by $(x_i)_{i\in \I_3}$ with defining
relations
\begin{align}\label{eq:rels-type-G-super-c-Nneq6}
\begin{aligned}
x_{13}&=0; \quad x_{332}=0; \quad x_{1112}=0; \quad x_{2}^2=0; \quad  x_{\alpha}^{N_\alpha}=0, \ \alpha\in\Oc_+^{\bq};
\\
[x_1, & [x_{123},x_2]_c]_c = \frac{q_{12}q_{32}}{1+q}[x_{12},x_{123}]_c-(q^{-1}-q^{-2})q_{12}q_{13} x_{123}x_{12};
\end{aligned}
\end{align}
where $\Oc_+^{\bq}=\{1,3,12,123,1^22^23,1^32^33,1^32^33^2\}$.
If $N = \infty$, i.e. $q\notin \G_{\infty}$, then we omit the relations $x_\alpha^{N_\alpha}=0,$ $\alpha\in\Oc_+^{\bq}$. 
Here
\begin{multline*}
\ya= (L+6M+4N-3)\alpha_1 + (L+4M+2N-1)\alpha_2  + (4M+2N-2)\alpha_3.
\end{multline*}

\subsubsection{The generalized Dynkin diagram \emph{(\ref{eq:dynkin-G-super}
c)}, $N =6$}\label{subsubsec:type-G-super-c-N=6}

The Nichols algebra $\toba_{\bq}$ is generated by $(x_i)_{i\in \I_3}$ with defining
relations
\begin{align}\label{eq:rels-type-G-super-c-N=6}
\begin{aligned}
x_{13}&=0; \ \ x_{23}^2=0; \ \ [x_{112},x_{12}]_c=0; \quad x_{2}^2=0; \quad x_{\alpha}^{N_\alpha}=0, \ \alpha\in\Oc_+^{\bq};
\\
[x_1, & [x_{123},x_2]_c]_c = \frac{q_{12}q_{32}}{1+q}[x_{12},x_{123}]_c-(q^{-1}-q^{-2})q_{12}q_{13} x_{123}x_{12};
\end{aligned}
\end{align}
here, $\Oc_+^{\bq}$, $\ya$  are as in \S \ref{subsubsec:type-G-super-c-Nneq6}.

\subsubsection{The generalized Dynkin diagram \emph{(\ref{eq:dynkin-G-super}
d)}}\label{subsubsec:type-G-super-d}

The Nichols algebra $\toba_{\bq}$ is generated by $(x_i)_{i\in \I_3}$ with defining
relations
\begin{align}\label{eq:rels-type-G-super-d}
\begin{aligned}
x_{1112}&=0; \quad x_{2}^2=0; \quad x_3^2=0; \quad x_{\alpha}^{N_\alpha}=0, \ \alpha\in\Oc_+^{\bq};
\\
x_{113}&=0; \ \   x_{(13)}+q^{-2}q_{23}\frac{1-q^3}{1-q}[x_{13},x_2]_c-q_{12}(1-q^3)x_2x_{13}=0;
\end{aligned}
\end{align}
where $\Oc_+^{\bq}=\{1,12,23,123,1^223,1^323,1^32^23^2\}$.
If $N = \infty$, i.e. $q\notin \G_{\infty}$, then we omit the relations $x_\alpha^{N_\alpha}=0,$ $\alpha\in\Oc_+^{\bq}$. 
Here
\begin{multline*}
\ya= (L+6M+4N-3)\alpha_1 + (6M+2N)\alpha_2  + (4M+2N-2)\alpha_3.
\end{multline*}

\subsubsection{The associated Lie algebra} This is of type $A_1\times G_2$.

\section{Standard type}\label{sec:by-diagram-standard}

\subsection{Standard type $B_{\theta, j}$, $\theta \ge 2$, $j\in \I_{\theta - 1}$}\label{subsec:type-B-standard}
Here $\zeta \in\G'_3$.

\subsubsection{Basic datum and root system}

The basic datum is $(\mathtt B_{\theta, j}, \rho)$ and the root system is  
$\superb{j}{\theta-j}$, $j\in\I_{\theta-1}$  as in \S 
\ref{subsubsec:root-system-B theta j}; hence the Weyl groupoid is as in \S 
\ref{subsubsec:type-superB-Weyl}.
But we have new incarnations.

\subsubsection{Incarnation}
The assignment
\begin{align}\label{eq:dynkin-B-stadard}
\Jb \mapsto &
\begin{cases}
\xymatrix{ {\bf A}_{\theta-1}(-\ztu;\Jb) \ar @{-}[r]^(.7){-\zeta }  & \overset{\zeta}{\underset{\ }{\circ}}}, & \theta\notin\Jb; \\
\xymatrix{ {\bf A}_{\theta-1}(-\zeta;\Jb) \ar @{-}[r]^(.7){-\ztu }  &
\overset{\ztu}{\underset{\ }{\circ}}}, & \theta\in\Jb.
\end{cases}
\end{align}
gives an incarnation. Notice that albeit $\theta$ is not a Cartan vertex, $\rho_{\theta} = \id$.

\subsubsection{PBW-basis and dimension}\label{subsubsec:type-B-standard-PBW}
The root vectors are 
\begin{align*}
x_{\alpha_{ii}} &= x_{\alpha_{i}} = x_{i},& i \in \I, \\
x_{\alpha_{ij}} &= x_{(ij)} = [x_{i}, x_{\alpha_{(i+1) j}}]_c,& i <  j \in \I, \\
x_{\alpha_{i\theta} + \alpha_{\theta}} &= [x_{\alpha_{i\theta}}, x_\theta]_c, & i  \in \I_{\theta - 1},
\\
x_{\alpha_{i\theta} + \alpha_{j\theta}} &= [x_{\alpha_{i\theta} + \alpha_{(j+1) \theta}}, x_j]_c, & i <  j \in \I_{\theta - 1},
\end{align*}
cf. \eqref{eq:roots-Atheta}. Thus
\begin{multline*}
\{ x_{\theta  }^{n_{\theta  \theta}} 
x_{\alpha_{\theta-1\theta} + \alpha_{\theta\theta}}^{m_{\theta-1\theta}}
x_{\alpha_{\theta-1\theta}}^{n_{\theta-1  \theta}} 
x_{ \theta-1}^{n_{\theta-1  \theta-1}} 
\dots
x_{\alpha_{1\theta} + \alpha_{2\theta}}^{m_{12}}
\dots
x_{\alpha_{1\theta} + \alpha_{\theta\theta}}^{m_{1\theta}}
\dots
x_{\alpha_{1\theta}}^{n_{1  \theta}} 
\dots 
x_{1}^{n_{1 1}} \, \\ 
| \, 0\le n_{ij}<N_{\alpha_{ij}}; \, 0\le m_{ij}<N_{\alpha_{i\theta}+\alpha_{j\theta}}\}
\end{multline*}
is a PBW-basis of $\toba_{\bq}$. Hence
\begin{align*}
\dim \toba_{\bq}= 2^{\theta(\theta-1)} 3^{j^2+(\theta-j)^2}.
\end{align*}

\subsubsection{Presentation}\label{subsubsec:type-B-standard}
The set of positive Cartan roots is
\begin{align}\label{eq:Cartan-roots-B-standard} 
\Oc_+^{\bq} = & \{\alpha_{ij}, \alpha_{i\theta}+ \alpha_{(j+1)\theta} \ : i\le j\in\I_{\theta - 1}, \alpha_{ij} \text{ even}\}.
\end{align}
Assume that $\theta = 2$. Then $\Oc_+^{\bq} =  \{\alpha_{1}, \alpha_{1}+ 2\alpha_{2}\}$  if $\alpha_{1}$ is even, i.e. $\Jb = \emptyset$, and $\Oc_+^{\bq} = \emptyset$
if $\alpha_{1}$ is odd, i.e. $\Jb = \{1\}$.

The Nichols algebra $\toba_{\bq}$ is generated by $(x_i)_{i\in \I}$ with defining relations
\begin{align}\label{eq:rels-type-B-standard}
\begin{aligned}
& x_{ij}= 0, && i < j - 1; &  
& [x_{(i-1i+1)},x_i]_c=0, && i\in\Jb;  
\\
& x_{ii(i\pm1)}= 0, && i\in\I_{\theta-1}-\Jb; & 
& [x_{\theta\theta(\theta-1)(\theta-2)}, x_{\theta(\theta-1)}]_c=0;  
\\
& x_i^2=0, && i\in\Jb; & & [x_{\theta\theta(\theta-1)}, x_{\theta(\theta-1)}]_c=0, && \theta-1\in\Jb; 
\\
&x_\theta^3=0; & &&
& x_{\alpha}^6 =0, && \alpha\in\Oc_+^{\bq}.
\end{aligned}
\end{align}

\subsubsection{The associated Lie algebra and $\ya$} This is of type $D_j\times D_{\theta-j}$.
In this case, the Weyl group of the associated  Lie algebra is isomorphic to a proper subgroup of the  isotropy group of the Weyl groupoid. Here
\begin{align*}
\ya &=
\sum_{\substack{i\le j\in\I_{\theta-1},\\  \alpha_{ij} \text{ odd}}} \alpha_{ij} + 
\sum_{\substack{i\le j\in\I_{\theta-1},\\  \alpha_{ij} \text{ even}}} 5 \alpha_{ij} +
\sum_{i\in\I_{\theta}} 2\alpha_{i\theta} \\ & \qquad +
\sum_{\substack{i< j\in\I_{\theta},\\  \alpha_{ij-1} \text{ odd}}} (\alpha_{i\theta}+\alpha_{j\theta}) + 
\sum_{\substack{i< j\in\I_{\theta},\\  \alpha_{ij-1} \text{ even}}} 5(\alpha_{i\theta}+\alpha_{j\theta}).
\end{align*}

\subsection{Standard type $G_{2}$}\label{subsec:type-G-st}
Here $\zeta \in \G'_8$.

\subsubsection{Basic datum}

This is described by the diagram
\begin{align*}
\xymatrix{ \underset{a_1}{\vtxgpd} \ar@{-}^{1} [r] & \underset{a_2}{\vtxgpd} \ar@{-}^{2} [r] & \underset{a_3}{\vtxgpd}.}
\end{align*}

\subsubsection{Root system}
The bundle of Cartan matrices $(C^{a_j})_{j\in\I_3}$  is constant: $C^{a_j}$ is the Cartan matrix of type $G_2$ as in \eqref{eq:dynkin-system-G}
for any $j \in\I_3$.

The bundle of root sets $(\varDelta^{a_j})_{j \in\I_3}$  is constant: 
\begin{align*}
\varDelta^{a_j}&=\{ \pm \alpha_1,\pm (3\alpha_1+\alpha_2), \pm(2\alpha_1+\alpha_2), \pm(3\alpha_1+2\alpha_2), \pm(\alpha_1+\alpha_2), \pm \alpha_2 \}.
\end{align*} 

\subsubsection{Weyl groupoid}\label{subsubsec:type-standard-G-Weyl} 
The isotropy group  at $a_1 \in \cX$ is 
\begin{align*}
\cW(a_1) & = \langle \varsigma^{a_1}_1 \varsigma_{2} \varsigma_{1} \varsigma_{2} \varsigma_{1}, \varsigma_2^{a_1} \rangle \simeq \Z/2 \times \Z/2 \leq  GL(\Z^\I).
\end{align*}

\subsubsection{Incarnation}
We assign the following Dynkin diagrams to $a_i$, $i\in\I_3$: 
\begin{align}\label{eq:dynkin-G2-st}
&\xymatrix{a_1 \ar  @{|->}[r]  & \overset{\,\, \zeta^2}{\underset{\ }{\circ}} \ar  @{-}[r]^{\zeta}  &
\overset{\ztu}{\underset{\ }{\circ}},} &
&\xymatrix{a_2 \ar  @{|->}[r]  & \overset{\,\, \zeta^2}{\underset{\ }{\circ}} \ar  @{-}[r]^{\zeta^3}  &
\overset{-1}{\underset{\ }{\circ}},}& 
&\xymatrix{a_3 \ar  @{|->}[r]  & \overset{\zeta}{\underset{\ }{\circ}} \ar
@{-}[r]^{\zeta^5}  & \overset{-1}{\underset{\ }{\circ}}.}
\end{align}

\subsubsection{The generalized Dynkin diagram \emph{(\ref{eq:dynkin-G2-st}a)}}\label{subsubsec:G2-st-a}

The set
\begin{multline*}
\{ x_{2}^{n_1} x_{12}^{n_2} x_{3\alpha_1+2\alpha_2}^{n_{3}}x_{112}^{n_4} x_{1112}^{n_5} x_1^{n_6} \, | \,  n_{1}, n_{4}\in\I_{0,7}, 
\  n_{2}, n_{6 }\in\I_{0,3}, \  n_{3}, n_{5}\in\I_{0,1}\}
\end{multline*}
is a PBW-basis of $\toba_{\bq}$. Hence $\dim \toba_{\bq}= 2^24^28^2=4096$.

\medskip
The Nichols algebra $\toba_{\bq}$ is generated by $(x_i)_{i\in \I_2}$ with defining relations
\begin{align}\label{eq:rels-G2-st-a}
\begin{aligned}
x_1^4&=0; & x_{221}&=0; & [x_{3\alpha_1+2\alpha_2}, x_{12}]_c &=0; \\
x_2^8&=0; &  x_{112}^8&=0.
\end{aligned}
\end{align}
In this case, $\Oc_+^{\bq}=\{\alpha_2, 2\alpha_1+\alpha_2\}$ and $\ya= 26\alpha_1 + 20 \alpha_2$.

\subsubsection{The generalized Dynkin diagram \emph{(\ref{eq:dynkin-G2-st} b)}}\label{subsubsec:G2-st-b}

The set
\begin{multline*}
\{ x_{2}^{n_1} x_{12}^{n_2} x_{3\alpha_1+2\alpha_2}^{n_{3}}x_{112}^{n_4} x_{1112}^{n_5} x_1^{n_6} \, | \,  n_{2}, n_{5}\in\I_{0,7}, \\  n_{4}, n_{6}\in\I_{0,3}, \,  n_{1}, n_{3}\in\I_{0,1}\}
\end{multline*}
is a PBW-basis of $\toba_{\bq}$. Hence $\dim \toba_{\bq}= 2^24^28^2=4096$.

\medskip
The Nichols algebra $\toba_{\bq}$ is generated by $(x_i)_{i\in \I_2}$ with defining relations
\begin{align}\label{eq:rels-G2-st-b}
\begin{aligned}
x_1^4&=0; & x_2^2&=0; &
[x_1,x_{3\alpha_1+2\alpha_2}]_c+\frac{q_{12}}{1-\zeta}x_{112}^2&=0;   \\
x_{12}^8&=0; & x_{1112}^8&=0.
\end{aligned}
\end{align}
In this case, $\Oc_+^{\bq}=\{\alpha_1+\alpha_2, 3\alpha_1+\alpha_2\}$.
and $\ya= 40\alpha_1 + 20 \alpha_2$.

\subsubsection{The generalized Dynkin diagram \emph{(\ref{eq:dynkin-G2-st} c)}}\label{subsubsec:G2-st-c}

The set
\begin{multline*}
\{ x_{2}^{n_1} x_{12}^{n_2} x_{3\alpha_1+2\alpha_2}^{n_{3}}x_{112}^{n_4} x_{1112}^{n_5} x_1^{n_6} \, | \,  n_{3}, n_{6}\in\I_{0,7}, 
\\  n_{2}, n_{4}\in\I_{0,3}, \,  n_{1}, n_{5}\in\I_{0,1}\}
\end{multline*}
is a PBW-basis of $\toba_{\bq}$. Hence $\dim \toba_{\bq}= 2^24^28^2=4096$.

\medskip

The Nichols algebra $\toba_{\bq}$ is generated by $(x_i)_{i\in \I_2}$ with defining relations
\begin{align}\label{eq:rels-G2-st-c}
\begin{aligned}
x_1^8&=0; & x_{11112}&=0;  & [x_{3\alpha_1+2\alpha_2}, x_{12}]_c &=0; \\
x_2^2&=0; & x_{3\alpha_1+2\alpha_2}^8&=0.
\end{aligned}
\end{align}
In this case, $\Oc_+^{\bq}=\{\alpha_1, 3\alpha_1+2\alpha_2\}$.
and $\ya= 40\alpha_1 + 22 \alpha_2$.

\subsubsection{The associated Lie algebra} This is of type $A_1\times A_1$.

\part{Arithmetic  root systems: modular, UFO}\label{part:modular-ufo}
\def\kk{\mathbb F}

\section{Modular type, characteristic 2 or 3}\label{sec:by-diagram-modular-char3}

\subsection{Type $\Bgl(4)$}\label{subsec:type-bgl(4,alpha)}
Here $\theta = 4$, $q\neq \pm 1$. Let $\kk$ be a field of characteristic 2, $\alpha \in \kk - \kk_2$ and
$$A = \begin{pmatrix}
0 & 1 & 0 & 0 \\ 
1 & 0 & 1 & 0 \\
0 & \alpha & 0 & 1 \\
0 & 0 & 1 & 0
\end{pmatrix} \in \kk^{4 \times 4}.$$ 
Let $\bgl(4,\alpha) = \g(A)$ be the corresponding  contragredient Lie algebra.
Then  $\dim \bgl(4,\alpha) = 34$ \cite{KW-exponentials}.
Notice that there are 4 other matrices $A'$ for which $\bgl(4,\alpha) \simeq \g(A')$. 
Here is the root system $\Bgl(4)$ of $\bgl(4,\alpha)$, see \cite{AA-GRS-CLS-NA} for details.

\subsubsection{Basic datum and root system}

Below, $A_4$ and ${}_1T$ are numbered as in \eqref{eq:dynkin-system-A} and  
\eqref{eq:mTn}, respectively.
The basic datum and the bundle of Cartan matrices are described 
by the following diagram: 
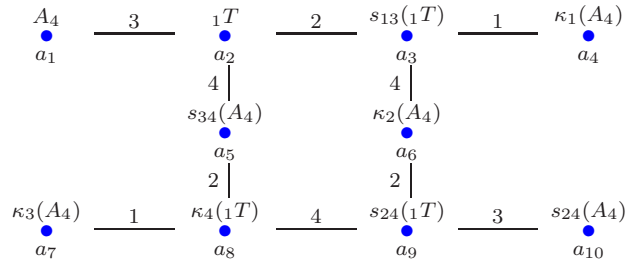
\begin{center}
	\begin{tabular}{c c c c c c c c c}
		$\overset{A_4}{\underset{a_{1}}{\vtxgpd}}$
		& \hspace{-5pt}\raisebox{3pt}{$\overset{3}{\rule{30pt}{0.5pt}}$}\hspace{-5pt}
		& $\overset{{}_1T} {\underset{a_{2}}{\vtxgpd}}$
		& \hspace{-5pt}\raisebox{3pt}{$\overset{2}{\rule{30pt}{0.5pt}}$}\hspace{-5pt}
		& $\overset{s_{13}({}_1T)} {\underset{a_{3}}{\vtxgpd}}$
		& \hspace{-5pt}\raisebox{3pt}{$\overset{1}{\rule{30pt}{0.5pt}}$}\hspace{-5pt}
		& $\overset{\kappa_1(A_4)} {\underset{a_{4}}{\vtxgpd}}$ & &
		\\
		& & {\scriptsize 4} \vline\hspace{5pt}  & & {\scriptsize 4} \vline\hspace{5pt}
		& & & &
		\\
		& &
		$\overset{s_{34}(A_4)} {\underset{a_{5}}{\vtxgpd}}$ & & $\overset{\kappa_2(A_4)} {\underset{a_{6}}{\vtxgpd}}$
		& & & &
		\\
		& & {\scriptsize 2} \vline\hspace{5pt} & & {\scriptsize 2} \vline\hspace{5pt}
		& & & &
		\\
		$\overset{\kappa_3(A_4)} {\underset{a_{7}}{\vtxgpd}}$
		& \hspace{-5pt}\raisebox{3pt}{$\overset{1}{\rule{30pt}{0.5pt}}$}\hspace{-5pt}
		& $\overset{\kappa_4({}_1T)} {\underset{a_{8}}{\vtxgpd}}$
		& \hspace{-5pt}\raisebox{3pt}{$\overset{4}{\rule{30pt}{0.5pt}}$}\hspace{-5pt}
		& $\overset{s_{24}({}_1T)} {\underset{a_{9}}{\vtxgpd}}$
		& \hspace{-5pt}\raisebox{3pt}{$\overset{3}{\rule{30pt}{0.5pt}}$}\hspace{-5pt}
		& $\overset{s_{24}(A_4)} {\underset{a_{10}}{\vtxgpd}}$  & &
	\end{tabular}
\end{center}
Using the notation \eqref{eq:notation-root-exceptional}, the bundle of root sets is the following: 
{ \scriptsize
	\begin{align*}
	\begin{aligned}
	\varDelta_{+}^{a_1}= & \{ 1, 12, 2, 123, 23, 3, 12^23^34, 12^23^24, 123^24, 23^24, 12^23^34^2, 1234, 234, 34, 4 \}, \\
	\varDelta_{+}^{a_2}= & \{ 1, 12, 2, 123, 23, 3, 12^234, 12^24, 1234, 234, 12^234^2, 34, 124, 24, 4 \}, \\
	\varDelta_{+}^{a_5}= & \{ 1, 12, 2, 12^23, 123, 23, 3, 
	12^23^24, 123^24, 23^24, 12^234, 1234, 234, 34, 4 \},
	\end{aligned}
	\\
	\begin{aligned}
	\varDelta_{+}^{a_3}= & s_{13}(\varDelta_{+}^{a_2}), &
	\varDelta_{+}^{a_4}= & \kappa_1(\varDelta_{+}^{a_1}), &
	\varDelta_{+}^{a_6}= & \kappa_2(\varDelta_{+}^{a_5}), &
	\varDelta_{+}^{a_7}= & \kappa_3(\varDelta_{+}^{a_1}),
	\\
	\varDelta_{+}^{a_8}= & \kappa_4(\varDelta_{+}^{a_2}), &
	\varDelta_{+}^{a_9}= & s_{24}(\varDelta_{+}^{a_2}), &
	\varDelta_{+}^{a_{10}}= & s_{24}(\varDelta_{+}^{a_1}).
	\end{aligned}
	\end{align*}
}

\subsubsection{Weyl groupoid}
\label{subsubsec:type-bgl4a-Weyl}
The isotropy group  at $a_1 \in \cX$ is
\begin{align*}
\cW(a_1)= \langle \varsigma_1^{a_1},\varsigma_2^{a_1}, \varsigma_3^{a_1}\varsigma_2 \varsigma_1 \varsigma_4 \varsigma_1\varsigma_2\varsigma_3, \varsigma_4^{a_1}  \rangle \simeq W(A_2) \times W(A_2).
\end{align*}

\subsubsection{Incarnation}
To describe it, we need the matrices $(\bq^{(i)})_{i\in\I_5}$ corresponding to the following Dynkin diagrams, from left to right and  from up to down
(also denoted below as a,\dots, e as customary).
\begin{align}\label{eq:dynkin-bgl(4,alpha)}
\begin{aligned}
&
\xymatrix@C-4pt{\overset{q}{\underset{\ }{\circ}}\ar  @{-}[r]^{q ^{-1}}  &
	\overset{q}{\underset{\ }{\circ}}
	\ar  @{-}[r]^{q^{-1}}  & \overset{-1}{\underset{\ }{\circ}}
	\ar  @{-}[r]^{-q}  & \overset{-q^{-1}}{\underset{\ }{\circ}}}
& &
\xymatrix@C-4pt{\overset{-q^{-1}}{\underset{\ }{\circ}}\ar  @{-}[r]^{-q}
	& \overset{-q^{-1}}{\underset{\ }{\circ}}
	\ar  @{-}[r]^{-q}  & \overset{-1}{\underset{\ }{\circ}}
	\ar  @{-}[r]^{q^{-1}}  & \overset{q}{\underset{\ }{\circ}}}
\\
&
\xymatrix@C-4pt{\overset{q}{\underset{\ }{\circ}}\ar  @{-}[r]^{q ^{-1}}  &
	\overset{-1}{\underset{\ }{\circ}}
	\ar  @{-}[r]^{-1}  & \overset{-1}{\underset{\ }{\circ}}
	\ar  @{-}[r]^{-q}  & \overset{-q^{-1}}{\underset{\ }{\circ}}}
& &
\\
&\xymatrix@R-6pt{ 
	& & \overset{-1}{\underset{\ }{\circ}} \ar@{-}[d]^{-q^{-1}} \ar@{-}[dl]_{-1} \\
	\overset{q}{\underset{\ }{\circ}} \ar@{-}[r]^{q^{-1}}  
	& \overset{-1}{\underset{\ }{\circ}} \ar@{-}[r]^{q}
	&\overset{-1}{\underset{\ }{\circ}} }
&& 
\xymatrix@R-6pt{ 
	& & \overset{-1}{\underset{\ }{\circ}} \ar@{-}[d]^{q} \ar@{-}[dl]_{-1} \\
	\overset{-q^{-1}}{\underset{\ }{\circ}} \ar@{-}[r]^{-q}  
	& \overset{-1}{\underset{\ }{\circ}} \ar@{-}[r]^{-q^{-1}}
	&\overset{-1}{\underset{\ }{\circ}} }
\end{aligned}
\end{align}

Now, this is the incarnation:
\begin{align*}
a_1 & \mapsto \bq^{(1)}, & 
a_2 & \mapsto \bq^{(4)}, & 
a_3 & \mapsto s_{13}(\bq^{(4)}), & 
a_4 & \mapsto \kappa_1(\bq^{(1)}), \\ 
& & a_5 & \mapsto s_{34}(\bq^{(3)}), &
a_6 & \mapsto \kappa_2(\bq^{(3)}), & & \\
a_7 & \mapsto \kappa_3(\bq^{(2)}), & 
a_8 & \mapsto \kappa_4(\bq^{(5)}), & 
a_9 & \mapsto s_{24}(\bq^{(5)}), & 
a_{10} & \mapsto s_{24}(\bq^{(2)}). & 
\end{align*}

We set $N=\ord q$, $M=\ord -q^{-1}$.

\subsubsection{PBW-basis and (GK-)dimension} \label{subsubsec:type-bgl4a-PBW}
Notice that the roots in each $\varDelta_{+}^{a_i}$, $i\in\I_{10}$, are ordered from left to right, justifying the notation $\beta_1, \dots, \beta_{15}$.

The root vectors $x_{\beta_k}$ are described as in Remark \ref{rem:lyndon-word}.
Thus
\begin{align*}
\left\{ x_{\beta_{15}}^{n_{15}} x_{\beta_{14}}^{n_{14}} \dots x_{\beta_2}^{n_{2}}  x_{\beta_1}^{n_{1}} \, | \, 0\le n_{k}<N_{\beta_k} \right\}.
\end{align*}
is a PBW-basis of $\toba_{\bq}$. If $N<\infty$, then
\begin{align*}
\dim \toba_{\bq}= 2^9M^3N^3.
\end{align*}
If $N=\infty$ (that is, if $q$ is not a root of unity), then
$\GK \toba_{\bq}= 6$.

\subsubsection{The Dynkin diagram \emph{(\ref{eq:dynkin-bgl(4,alpha)}
		a)}}\label{subsubsec:bgl(4,alpha)-a}

\

The Nichols algebra $\toba_{\bq}$ is generated by $(x_i)_{i\in \I_4}$ with defining relations
\begin{align}\label{eq:rels-bgl(4,alpha)-a}
\begin{aligned}
x_3^2&=0; & x_{13}&=0; & x_{14}&=0; & &x_{24}=0; \\
x_{112}&=0; & x_{221}&=0; & x_{\alpha}^N&=0, & &\alpha\in\{1,2,12\}; \\
x_{223}&=0; &  x_{443}&=0; & x_{\alpha}^M&=0, & &\alpha\in\{ 4,12^23^34,12^23^34^2\}.
\end{aligned}
\end{align}
If $N = \infty$, i.e. $q\notin \G_{\infty}$, then we omit the relations $x_\alpha^{N}=0$, $x_\alpha^{M}=0$. 
Here the degree of the integral is
\begin{equation*}
\ya= (2M+2N)\alpha_1 + (4M+2N+2)\alpha_2   + (6M+6)\alpha_3  + (4M+2)\alpha_4.
\end{equation*}

\subsubsection{The Dynkin diagram \emph{(\ref{eq:dynkin-bgl(4,alpha)}
		b)}}\label{subsubsec:bgl(4,alpha)-b}

\

This diagram is of the shape of (\ref{eq:dynkin-bgl(4,alpha)} a) but with $-q^{-1}$
instead of $q$. Thus the information on the corresponding Nichols algebra is analogous to
\S \ref{subsubsec:bgl(4,alpha)-a}.

\subsubsection{The Dynkin diagram \emph{(\ref{eq:dynkin-bgl(4,alpha)}
		c)}}\label{subsubsec:bgl(4,alpha)-c}

\

The Nichols algebra $\toba_{\bq}$ is generated by $(x_i)_{i\in \I_4}$ with defining relations
\begin{align}\label{eq:rels-bgl(4,alpha)-c}
\begin{aligned}
\begin{aligned}
x_2^2&=0; & x_{13}&=0; & x_{14}&=0; & &x_{24}=0; \\
x_{3}^2&=0; & x_{23}^2&=0; &    x_{\alpha}^N&=0, & & \alpha\in\{1,23^24,123^24\}; \\
x_{112}&=0; &  x_{443}&=0; & x_{\alpha}^M&=0, & & \alpha\in\{4,12^23,12^234\};
\end{aligned}
\\
[[x_{(14)},x_2 ]_c,x_3]_c-q_{23}\frac{1+q}{1-q}[[ x_{(14)},x_3]_c,x_2]_c=0.
\end{aligned}
\end{align}
If $N = \infty$, i.e. $q\notin \G_{\infty}$, then we omit the relations $x_\alpha^{N}=0$, $x_\alpha^{M}=0$. 
Here the degree of the integral is
\begin{equation*}
\ya= (2M+2N)\alpha_1 + (4M+2N+2)\alpha_2   + (2M+4N+2)\alpha_3  + (2M+2N)\alpha_4.
\end{equation*}

\subsubsection{The Dynkin diagram \emph{(\ref{eq:dynkin-bgl(4,alpha)}
		d)}}\label{subsubsec:bgl(4,alpha)-d}

\

The Nichols algebra $\toba_{\bq}$ is generated by $(x_i)_{i\in \I_4}$ with defining relations
\begin{align}\label{eq:rels-bgl(4,alpha)-d}
\begin{aligned}
\begin{aligned}
x_{112}&=0; & x_{13}&=0; & x_{14}&=0; & &[x_{(13)},x_2]_c=0; \\
x_{2}^2&=0; & x_{24}^2&=0; & x_{\alpha}^N&=0, & & \alpha\in\{1,23,123\}; \\
x_{3}^2&=0; &  x_{4}^2&=0; & x_{\alpha}^M&=0, & &\alpha\in\{34,12^24,12^234^2\};
\end{aligned}
\\
x_{(24)}+\frac{(1+q)q_{43}}{2}[x_{24},x_3]_c-q_{23}(1+q^{-1})x_3x_{24}=0.
\end{aligned}
\end{align}
If $N = \infty$, i.e. $q\notin \G_{\infty}$, then we omit the relations $x_\alpha^{N}=0$, $x_\alpha^{M}=0$. 
Here the degree of the integral is
\begin{equation*}
\ya= (2M+2N)\alpha_1 + (4M+2N+2)\alpha_2   + (2M+2N)\alpha_3  + (4M+2)\alpha_4.
\end{equation*}

\subsubsection{The Dynkin diagram \emph{(\ref{eq:dynkin-bgl(4,alpha)}
		e)}}\label{subsubsec:bgl(4,alpha)-e}

This diagram is of the shape of (\ref{eq:dynkin-bgl(4,alpha)} d) but with $-q^{-1}$
instead of $q$. Thus the information  is as in
\S \ref{subsubsec:bgl(4,alpha)-d}.

\subsubsection{The associated Lie algebra} This is of type $A_2\times A_2$.

\subsection{Type $\Brown(2)$}\label{subsec:type-br(2,a)}
Here $\theta = 2$, $\zeta \in \G_3$, $q\notin \G_3$. Let $\kk$ be a field of characteristic 3,  $a \in \kk - \kk_3$,
\begin{align*}
A &= \begin{pmatrix}
2 & -1 \\ a & 2
\end{pmatrix},& A' &= \begin{pmatrix}
2 & -1 \\ -1-a & 2
\end{pmatrix} \in \kk^{2\times 2}
\end{align*}
Let $\br(2,a) = \g(A) \simeq \g(A')$, the contragredient Lie algebras corresponding to $A, A'$. 
Then  $\dim \br(2,a) = 10$ \cite{BGL}.
We describe now the root system $\Brown(2)$ of $\br(2,a)$, see \cite{AA-GRS-CLS-NA} for details.

\subsubsection{Basic datum and root system}

Below, $B_2$ is numbered as in \eqref{eq:dynkin-system-B}.
The basic datum and the bundle of Cartan matrices are described by the diagram:
\begin{align*}
&\text{$\overset{B_2}{\vtxgpd}$   \hspace{-5pt}\raisebox{3pt}{$\overset{2}{\rule{40pt}{0.5pt}}$}\hspace{-5pt}  $\overset{B_2}{\vtxgpd}$}
&& \xymatrix{ \ar@{->}[rr]^{s_{12}} & & }
&\text{$\overset{C_2}{\vtxgpd}$   \hspace{-5pt}\raisebox{3pt}{$\overset{1}{\rule{40pt}{0.5pt}}$}\hspace{-5pt}  $\overset{C_2}{\vtxgpd}$.}
\end{align*}
This is a standard Dynkin diagram, that might be called of type $C_2$; indeed $\Br$ is not isomorphic to $\superb{1}{1}$. 
We include it here because of the relation 
with the modular Lie algebra $\br(2, a)$.
The bundle of root sets $(\varDelta^{a_j})_{j \in\I_2}$ is constant:
\begin{align*}
\varDelta^{a_j}&=\{ \pm \alpha_1, \pm(2\alpha_1+\alpha_2),  \pm(\alpha_1+\alpha_2), \pm \alpha_2 \}.
\end{align*}

\subsubsection{Weyl groupoid}
\label{subsubsec:type-br2a-Weyl}
The isotropy group  at $a_1 \in \cX$ is
\begin{align*}
\cW(a_1)= \langle \varsigma_1^{a_1}\varsigma_2\varsigma_1,
\varsigma_2^{a_1}, \rangle \simeq \Z/2  \times \Z/2.
\end{align*}

\subsubsection{Incarnation}
We assign the following Dynkin diagrams to $a_i$, $i\in\I_2$:
\begin{align}\label{eq:dynkin-br(2,a)}
&\xymatrix{a_1 \ar  @{|->}[r]  &  \overset{\zeta}{\underset{\ }{\circ}} \ar  @{-}[r]^{q^{-1}}  &
	\overset{q}{\underset{\ }{\circ}}},
& &\xymatrix{a_2 \ar  @{|->}[r]  &  \overset{\zeta}{\underset{\ }{\circ}} \ar  @{-}[r]^{\zeta^2q}  &
	\overset{\zeta q^{-1}}{\underset{\ }{\circ}}}.
\end{align}

The Dynkin diagram (\ref{eq:dynkin-br(2,a)} b)
has the same shape as (\ref{eq:dynkin-br(2,a)} a) but with $\zeta q^{-1}$
instead of $q$. Thus, we just discuss  the latter.

\subsubsection{PBW-basis and (GK-)dimension} \label{subsubsec:type-br2a-PBW}
We set $N=\ord q$, $M=\ord \zeta q^{-1}$. The root vectors $x_{\beta_k}$ are described as in Remark \ref{rem:lyndon-word}.
Thus
\begin{align*}
\{ x_{2}^{n_1} x_{12}^{n_2} x_{112}^{n_3} x_1^{n_4} \, | \, 0\le n_{3}<M, \, 0\le n_{4}<N, \, 0\le n_{1}, n_{2}<3\}.
\end{align*}
is a PBW-basis of $\toba_{\bq}$. If $N<\infty$, then
$\dim \toba_{\bq}= 3^2MN$. If $N=\infty$ (that is, if $q$ is not a root of unity), then $\GK \toba_{\bq}= 2$.

\subsubsection{Relations, $q=-1$}\label{subsubsec:br(2,a)-q=-1}
The Nichols algebra $\toba_{\bq}$ is generated by $(x_i)_{i\in \I_2}$ with defining relations
\begin{align}\label{eq:rels-br(2,a)-a-q=-1}
x_{1}^3 &= 0; & x_{12}^3 &= 0;& x_{112}^{6} &=0,& x_{2}^{2} &=0; & [x_{112}, x_{12}]_c
&=0.
\end{align}
Here the degree of the integral is
\begin{equation*}
\ya= 13\alpha_1 + 9\alpha_2.
\end{equation*}

\subsubsection{Relations, $q \neq -1$}\label{subsubsec:br(2,a)-qneq-1}
The Nichols algebra $\toba_{\bq}$ is generated by $(x_i)_{i\in \I_2}$ with defining relations
\begin{align}\label{eq:rels-br(2,a)-a-qneq-1}
x_{1}^3 &= 0; & x_{12}^3 &= 0;& x_{112}^{M} &=0,& x_{2}^{N} &=0; & x_{221}&=0.
\end{align}
If $N = \infty$, i.e. $q\notin \G_{\infty}$, then we omit the relations $x_2^{N}=0$, $x_{112}^{M}=0$. 
Here the degree of the integral is
\begin{equation*}
\ya= (2M+N-1)\alpha_1 + (M+3)\alpha_2.
\end{equation*}

\subsubsection{The associated Lie algebra} This is of type $A_1\times A_1$.

\subsection{Type $\Brown(3)$}\label{subsec:type-br(3)}
Here $\theta = 3$, $\zeta \in \G'_{9}$. Let $\kk$ be a field of characteristic 3 and
\begin{align*}
A &= \begin{pmatrix}
2 & -1 & 0 \\ -2 & 2 & -1 \\ 0 & 1 & 0
\end{pmatrix},&
A' &= \begin{pmatrix}
2 & -1 & 0 \\ -1 & 2 & -1 \\ 0 & 1 & 0
\end{pmatrix}
\end{align*}
Let $\br(3) = \g(A) \simeq \g(A')$, the contragredient Lie algebras corresponding to $A$, $A'$. 
Then  $\dim \br(3) = 29$ \cite{BGL}.
We describe now the root system $\Brown(3)$ of $\br(3)$, see \cite{AA-GRS-CLS-NA} for details.

\subsubsection{Basic datum and root system}

Below, $B_3$ and $A_4^{(2)}$ are numbered as in \eqref{eq:dynkin-system-B} and  
\eqref{eq:A2-2n}, respectively.
The basic datum and the bundle of Cartan matrices are described by the following diagram:
\begin{center}
	$\overset{B_3}{\underset{a_{1}}{\vtxgpd}}$   \hspace{-5pt}\raisebox{3pt}{$\overset{3}{\rule{40pt}{0.5pt}}$}\hspace{-5pt}  $\overset{\tau(A_4^{(2)})}{\underset{a_{2}}{\vtxgpd}}$.
\end{center}
Using the notation \eqref{eq:notation-root-exceptional}, the bundle of root sets is the following:
\begin{align*}
\varDelta_{+}^{a_1}= & \{ 1, 12, 123,  1^22^33^4, 12^23^2, 12^23^3, 12^23^4, 12^33^4, 123^2,  2, 23^2, 23, 3\}, \\
\varDelta_{+}^{a_2}= & \{ 1, 12^2, 12, 123^2, 12^33^2, 1^22^33^2, 12^23^2, 123, 12^23, 2, 23^2, 23, 3 \}.
\end{align*}

\subsubsection{Weyl groupoid}
\label{subsubsec:type-br3-Weyl}
The isotropy group at $a_1 \in \cX$ is
\begin{align*}
\cW(a_1)= \langle \varsigma_1^{a_1}, \varsigma_2^{a_1},  \varsigma_3^{a_1} \varsigma_2 \varsigma_3 \rangle \simeq W(B_3).
\end{align*}

\subsubsection{Incarnation}
We assign the following Dynkin diagrams to $a_i$, $i\in\I_2$: 
\begin{align}\label{eq:dynkin-br(3)}
&\xymatrix{a_1 \ar  @{|->}[r]  & \overset{\zeta}{\underset{\ }{\circ}}\ar  @{-}[r]^{\ztu}  &
	\overset{\zeta}{\underset{\ }{\circ}} \ar  @{-}[r]^{\ztu}  & \overset{\ztu^{\,
			3}}{\underset{\ }{\circ}}}
& &\xymatrix{a_2 \ar  @{|->}[r]  &  \overset{\zeta}{\underset{\ }{\circ}}\ar  @{-}[r]^{\ztu}  &
	\overset{\ztu^{\, 4}}{\underset{\ }{\circ}} \ar  @{-}[r]^{\zeta^{4}}  & \overset{\ztu^{\,
			3}}{\underset{\ }{\circ}}}
\end{align}

\subsubsection{PBW-basis and dimension} \label{subsubsec:type-br3-PBW}
Notice that the roots in each $\varDelta_{+}^{a_i}$, $i\in\I_{2}$, are ordered from left to right, justifying the notation $\beta_1, \dots, \beta_{13}$.

The root vectors $x_{\beta_k}$ are described as in Remark \ref{rem:lyndon-word}.
Thus
\begin{align*}
\left\{ x_{\beta_{13}}^{n_{13}} x_{\beta_{12}}^{n_{12}} \dots x_{\beta_2}^{n_{2}}  x_{\beta_1}^{n_{1}} \, | \, 0\le n_{k}<N_{\beta_k} \right\}.
\end{align*}
is a PBW-basis of $\toba_{\bq}$. Hence $\dim \toba_{\bq}= 3^49^9=3^{22}$.

\subsubsection{The Dynkin diagram \emph{(\ref{eq:dynkin-br(3)}
		a)}}\label{subsubsec:br(3)-a}

The Nichols algebra $\toba_{\bq}$ is generated by $(x_i)_{i\in \I_3}$ with defining
relations
\begin{align}\label{eq:rels-br(3)-a}
\begin{aligned}
x_{13}&=0; & x_{112}&=0; & x_3^3&=0; & & [[x_{332}, x_{3321}]_c,x_{32}]_c=0;\\
x_{221}&=0; & x_{223}&=0; & x_\alpha^{9}&=0, & & \alpha\in\Oc_+^{\bq};
\end{aligned}
\end{align}
where $\Oc_+^{\bq}=\{1,2,12,23^2,123^2,12^23^2,12^23^4,12^33^4,1^22^33^4\}$.
Here the degree of the integral is
\begin{equation*}
\ya= 68\alpha_1 + 120\alpha_2   +156\alpha_3 .
\end{equation*}

\subsubsection{The Dynkin diagram \emph{(\ref{eq:dynkin-br(3)}
		b)}}\label{subsubsec:br(3)-b}

The Nichols algebra $\toba_{\bq}$ is generated by $(x_i)_{i\in \I_3}$ with defining
relations
\begin{align}\label{eq:rels-br(3)-b}
\begin{aligned}
\begin{aligned}
x_{13}&=0; & x_{112}&=0; &
x_{2221}&=0; & x_{223}&=0; & x_\alpha^{9}&=0, \ \alpha\in\Oc_+^{\bq};
\end{aligned}
\\
\begin{aligned}
x_3^3&=0; & & (1+\zeta^4)[[x_{(13)}, x_{2}]_c,x_{3}]_c = q_{23}[[x_{(13)}, x_{3}]_c,x_{2}]_c;
\end{aligned}
\end{aligned}
\end{align}
where $\Oc_+^{\bq}=\{1,2,12,12^2,23^2,123^2,12^23^2,12^33^2,1^22^33^2\}$.
Here the degree of the integral is
\begin{equation*}
\ya= 68\alpha_1 + 120\alpha_2   + 88\alpha_3 .
\end{equation*}

\subsubsection{The associated Lie algebra} This is of type $B_3$.

\section{Super modular type, characteristic 3}\label{sec:by-diagram-super-modular-char3}
In this Section $\kk$ is a field of characteristic  $3$.

\subsection{\ Type $\Sbrown(2; 3)$}\label{subsec:type-brj(2,3)}
Here $\theta = 2$, $\zeta \in \G'_9$. Let 
\begin{align*}
A&=\begin{pmatrix} 0 & 1 \\ 1 & 0 \end{pmatrix}, &
A'&=\begin{pmatrix} 0 & 1 \\ -2 & 2 \end{pmatrix}, &
A''&=\begin{pmatrix} 2 & -1 \\ -1 & 0 \end{pmatrix}
 \in \kk^{2\times 2}; \\
\pa &= (-1,1),& \pa''&= (-1,-1) \in \G_2^2.
\end{align*}
Let $\brj(2; 3) = \g(A, \pa)\simeq \g(A',\pa)\simeq \g(A'',\pa'')$,
 the contragredient Lie superalgebras corresponding to $(A, \pa)$, $(A', \pa)$, $(A'', \pa'')$. 
We know \cite{BGL} that $$\sdim \brj(2; 3) = 10|8.$$ 
We describe now the root system $\Sbrown(2; 3)$ of $\brj(2; 3)$, see \cite{AA-GRS-CLS-NA} for details.

\subsubsection{Basic datum and root system}
Below, $A_1^{(1)}$, $C_2$ and $A_2^{(2)}$ are numbered as in \eqref{eq:An-(1)}, \eqref{eq:dynkin-system-C} and  	
\eqref{eq:A2-2n}, respectively.
The basic datum and the bundle of Cartan matrices are described by the following diagram:
\begin{center}
	$\overset{A_1^{(1)}}{\underset{a_1}{\vtxgpd}}$ \hspace{-5pt}\raisebox{3pt}{$\overset{1}{\rule{40pt}{0.5pt}}$}\hspace{-5pt}
	$\overset{C_2}{\underset{a_2}{\vtxgpd}}$  \hspace{-5pt}\raisebox{3pt}{$\overset{2}{\rule{40pt}{0.5pt}}$}\hspace{-5pt}
	$\overset{A_2^{(2)}}{\underset{a_3}{\vtxgpd}}$.
\end{center}
Using the notation \eqref{eq:notation-root-exceptional}, the bundle of root sets is the following:
\begin{align*}
\varDelta_{+}^{a_1}= & \{1,1^22,1^32^2,12,12^2,2 \}, &
\varDelta_{+}^{a_2}= & \{1,1^22,1^32^2,1^42^3,12,2 \}, \\
\varDelta_{+}^{a_3}= & \{1,1^42,1^32,1^22,12,2 \}.
\end{align*}

\subsubsection{Weyl groupoid}
\label{subsubsec:type-brj23-Weyl}
The isotropy group  at $a_2 \in \cX$ is
\begin{align*}
\cW(a_2)= \langle \varsigma_1^{a_2}\varsigma_2\varsigma_1, \varsigma_2^{a_2} \varsigma_1 \varsigma_2 \rangle \simeq \Z/2  \times \Z/2.
\end{align*}

\subsubsection{Incarnation}

We assign the following Dynkin diagrams to $a_i$, $i\in\I_3$:
\begin{align}\label{eq:dynkin-brj(2,3)}
& a_1\mapsto \xymatrix{ \overset{-\zeta}{\underset{\ }{\circ}} \ar  @{-}[r]^{\ztu^{\, 2}}  &
	\overset{\zeta^3}{\underset{\ }{\circ}}},
& &a_2\mapsto \xymatrix{ \overset{\zeta^3}{\underset{\ }{\circ}} \ar  @{-}[r]^{\ztu}  &
	\overset{-1}{\underset{\ }{\circ}}},&
&a_3\mapsto \xymatrix{ \overset{-\zeta^2}{\underset{\ }{\circ}} \ar  @{-}[r]^{\zeta}  &
	\overset{-1}{\underset{\ }{\circ}}}.
\end{align}

\subsubsection{PBW-basis and dimension} \label{subsubsec:type-br23j-PBW}
Notice that the roots in each $\varDelta_{+}^{a_i}$, $i\in\I_{3}$, are ordered from left to right, justifying the notation $\beta_1, \dots, \beta_{6}$.

The root vectors $x_{\beta_k}$ are described as in Remark \ref{rem:lyndon-word}.
Thus
\begin{align*}
\left\{ x_{\beta_{6}}^{n_{6}} x_{\beta_{5}}^{n_{5}} 
x_{\beta_4}^{n_{4}}  x_{\beta_3}^{n_{3}}
x_{\beta_2}^{n_{2}}  x_{\beta_1}^{n_{1}} \, | \, 0\le n_{k}<N_{\beta_k} \right\}.
\end{align*}
is a PBW-basis of $\toba_{\bq}$. Hence $\dim \toba_{\bq}= 2^23^218^2=11664$.

\subsubsection{The Dynkin diagram \emph{(\ref{eq:dynkin-brj(2,3)}
		a)}}\label{subsubsec:br(2,3)-a}

The Nichols algebra $\toba_{\bq}$ is generated by $(x_i)_{i\in \I_2}$ with defining relations
\begin{align}\label{eq:rels-brj(2,3)-a}
\begin{aligned}
x_1^{18}&=0; & x_2^3&=0; &  &[x_1,[x_{12},x_2]_c]_c=\frac{\zeta^7q_{12}}{1+\zeta}
x_{12}^2;\\
x_{1112}&=0; & x_{12}^{18}&=0.
\end{aligned}
\end{align}
Here $\Oc_+^{\bq}= \{1, 12 \}$ and the degree of the integral is
\begin{equation*}
\ya= 42\alpha_1 + 25\alpha_2.
\end{equation*}

\subsubsection{The Dynkin diagram \emph{(\ref{eq:dynkin-brj(2,3)}
		b)}}\label{subsubsec:br(2,3)-b}

The Nichols algebra $\toba_{\bq}$ is generated by $(x_i)_{i\in \I_2}$ with defining relations
\begin{align}\label{eq:rels-brj(2,3)-b}
x_1^{3}&=0; & x_2^2&=0; &  [x_{112},[x_{112},x_{12}]_c]_c&=0; &
x_{112}^{18}&=0; & x_{12}^{18}&=0.
\end{align}
Here $\Oc_+^{\bq}= \{1^22, 12 \}$ and the degree of the integral is
\begin{equation*}
\ya= 63\alpha_1 + 42\alpha_2.
\end{equation*}

\subsubsection{The Dynkin diagram \emph{(\ref{eq:dynkin-brj(2,3)}
		c)}}\label{subsubsec:br(2,3)-c}

The Nichols algebra $\toba_{\bq}$ is generated by $(x_i)_{i\in \I_2}$ with defining relations
\begin{align}\label{eq:rels-brj(2,3)-c}
x_1^{18}&=0; & x_2^2&=0; &  [x_{112},x_{12}]_c&=0; &
x_{111112}&=0; & x_{112}^{18}&=0.
\end{align}
Here $\Oc_+^{\bq}= \{1, 1^22 \}$ and the degree of the integral is
\begin{equation*}
\ya= 63\alpha_1 + 23\alpha_2.
\end{equation*}

\subsubsection{The associated Lie algebra} This is of type $A_1\times A_1$.

\subsection{\ Type $\gtt(1,6)$}\label{subsec:type-g(1,6)}
Here $\theta = 3$, $\zeta \in \G'_{3} \cup \G'_{6}$. Let 
\begin{align*}
A&=\begin{pmatrix} 2 & -1 & 0 \\ -2 & 2 & -2 \\ 0 & 1 & 0 \end{pmatrix}, &
A'&=\begin{pmatrix} 2 & -1 & 0 \\ -1 & 2 & -2 \\ 0 & 1 & 0 \end{pmatrix} \in \kk^{3\times 3},\\ 
\pa&= (1,-1,-1),& \pa'&= (1,1,-1) \in \G_2^3.
\end{align*}
Let $\g(1, 6) = \g(A, \pa) \simeq \g(A', \pa')$, the contragredient Lie
superalgebras corresponding to $(A, \pa)$, $(A', \pa')$. 
Then $\sdim \g(1, 6) = 21|14$ \cite{BGL}. We describe now its root system $\gtt(1,6)$,
see \cite{AA-GRS-CLS-NA}.

\subsubsection{Basic datum and root system}
Below, $C_3$ and $C_2^{(1)}$ are numbered as in \eqref{eq:dynkin-system-C} and \eqref{eq:Cn-(1)}, respectively.
The basic datum and the bundle of Cartan matrices are described by the following diagram:
\begin{center}
	$\overset{C_3}{\underset{a_1}{\vtxgpd}}$   \hspace{-5pt}\raisebox{3pt}{$\overset{3}{\rule{40pt}{0.5pt}}$}\hspace{-5pt}  $\overset{C_2^{(1)}}{\underset{a_2}{\vtxgpd}}$.
\end{center}

Using the notation \eqref{eq:notation-root-exceptional}, the bundle of root sets is the following:
\begin{align*}
\varDelta_{+}^{a_1}=&\{1,12,1^22^23,123,12^23,12^23^2,12^33^2,1^22^33^2,1^22^43^3,2,2^23,23,3 \}, \\
\varDelta_{+}^{a_2}=&\{1,12,12^2,1^22^23,1^22^33,1^22^43,12^33,123,12^2,2,2^23,23,3\}.
\end{align*}

\subsubsection{Weyl groupoid}
\label{subsubsec:type-g16-Weyl}
The isotropy group  at $a_1 \in \cX$ is
\begin{align*}
\cW(a_1)= \langle \varsigma_1^{a_1}, \varsigma_2^{a_1},  \varsigma_3^{a_1} \varsigma_2 \varsigma_3 \rangle \simeq W(C_3).
\end{align*}

\subsubsection{Incarnation}
We assign the following Dynkin diagrams to $a_i$, $i\in\I_2$:
\begin{align}\label{eq:dynkin-g(1,6)}
&a_1\mapsto \xymatrix{ \overset{\zeta}{\underset{\ }{\circ}}\ar  @{-}[r]^{\ztu}  &
	\overset{\zeta}{\underset{\ }{\circ}} \ar  @{-}[r]^{\ztu^{\, 2}}  &
	\overset{-1}{\underset{\ }{\circ}}}
& & a_2\mapsto \xymatrix{ \overset{\zeta}{\underset{\ }{\circ}}\ar  @{-}[r]^{\ztu}  &
	\overset{-\ztu}{\underset{\ }{\circ}} \ar  @{-}[r]^{\zeta^{2}}  & \overset{-1}{\underset{\
		}{\circ}}}
\end{align}

\subsubsection{PBW-basis and dimension} \label{subsubsec:type-g16-PBW}
Notice that the roots in each $\varDelta_{+}^{a_i}$, $i\in\I_{2}$, are ordered from left to right, justifying the notation $\beta_1, \dots, \beta_{13}$.

The root vectors $x_{\beta_k}$ are described as in Remark \ref{rem:lyndon-word}. Thus
\begin{align*}
\left\{ x_{\beta_{13}}^{n_{13}} x_{\beta_{12}}^{n_{12}} \dots
x_{\beta_2}^{n_{2}}  x_{\beta_1}^{n_{1}} \, | \, 0\le n_{k}<N_{\beta_k} \right\}
\end{align*}
is a PBW-basis of $\toba_{\bq}$. If $N=6$, then $\dim \toba_{\bq}= 2^43^36^6=2^{10}3^{9}$. If $N=3$, then $\dim \toba_{\bq}= 2^43^66^3=2^{7}3^{9}$.

\subsubsection{The Dynkin diagram \emph{(\ref{eq:dynkin-g(1,6)}
		a)}, $N=6$}\label{subsubsec:g(1,6)-a-N=6}

The Nichols algebra $\toba_{\bq}$ is generated by $(x_i)_{i\in \I_3}$ with defining relations
\begin{align}\label{eq:rels-g(1,6)-a-N=6}
\begin{aligned}
x_{221}&=0; &x_{112}&=0; & x_{\alpha}^6&=0, & &\alpha\in \{1,2,12,12^23^2,12^33^2,1^22^33^2\}; \\
x_{13}&=0; & x_{2223}&=0; & x_{3}^2&=0; &  & x_{\alpha}^3=0, \, \alpha\in \{23,123,12^23\}.
\end{aligned}
\end{align}
Here $\Oc_+^{\bq}= \{1, 2, 12, 12^23^2, 12^33^2, 1^22^33^2, 23, 123, 12^23 \}$ and the degree of the integral is
\begin{equation*}
\ya= 38\alpha_1 + 66\alpha_2 + 42\alpha_3.
\end{equation*}

\subsubsection{The Dynkin diagram \emph{(\ref{eq:dynkin-g(1,6)}
		a)}, $N=3$}\label{subsubsec:g(1,6)-a-N=3}

The Nichols algebra $\toba_{\bq}$ is generated by $(x_i)_{i\in \I_3}$ with defining relations
\begin{align}\label{eq:rels-g(1,6)-a-N=3}
\begin{aligned}
&x_{13}=0; & &[[x_{(13)},x_2]_c,x_2]_c=0; \ [x_{223},x_{23}]_c=0;
\\
&x_{112}=0; & &x_{3}^2=0; \ x_{\alpha}^6=0, \, \alpha\in \{23,123,12^23\}
\\
&x_{221}=0; & &x_{\alpha}^3=0, \, \alpha\in \{1,2,12,12^23^2,12^33^2,1^22^33^2\}.
\end{aligned}
\end{align}
Here $\Oc_+^{\bq}= \{1, 2, 12, 12^23^2, 12^33^2, 1^22^33^2, 23, 123, 12^23 \}$ and the degree of the integral is
\begin{equation*}
\ya= 26\alpha_1 + 48\alpha_2 + 33\alpha_3.
\end{equation*}

\subsubsection{The Dynkin diagram \emph{(\ref{eq:dynkin-g(1,6)}
		b)}, $N=6$}\label{subsubsec:g(1,6)-b-N=6}

The Nichols algebra $\toba_{\bq}$ is generated by $(x_i)_{i\in \I_3}$ with defining relations
\begin{align} \label{eq:rels-g(1,6)-b-N=6}
\begin{aligned}
\begin{aligned}
x_{13}&=0; & x_{112}&=0;  \quad x_{\alpha}^6=0, \, \alpha\in \{1,23,123,12^2,12^33,1^22^33\}; \\
x_{3}^2&=0; &  &[x_{223},x_{23}]_c=0; \quad x_{\alpha}^3=0, \, \alpha\in \{2,12,12^23\};
\end{aligned}
\\
[x_2, [x_{21},x_{23}]_c]_c+q_{13}q_{23}q_{21}[x_{223},x_{21}]_c+q_{21}x_{21}x_{223}=0.
\end{aligned}
\end{align}
Here $\Oc_+^{\bq}= \{1, 23, 123, 12^2, 12^33, 1^22^33, 2, 12, 12^23 \}$ and the degree of the integral is
\begin{equation*}
\ya= 38\alpha_1 + 66\alpha_2 + 26\alpha_3.
\end{equation*}

\subsubsection{The Dynkin diagram \emph{(\ref{eq:dynkin-g(1,6)}
		b)}, $N=3$}\label{subsubsec:g(1,6)-b-N=3}

The Nichols algebra $\toba_{\bq}$ is generated by $(x_i)_{i\in \I_3}$ with defining relations
\begin{align}\label{eq:rels-g(1,6)-b-N=3}
\begin{aligned}
\begin{aligned}
x_{2221}&=0; &x_{112}&=0; & x_{\alpha}^3&=0, \, \alpha\in \{1,23,123,12^2,12^33,1^22^33\}; \\
x_{2223}&=0; & x_{13}&=0; &x_{3}^2&=0; \quad x_{\alpha}^6=0, \, \alpha\in \{2,12,12^23\};
\end{aligned}
\\
[x_{1},x_{223}]_c+q_{23}[x_{(13)},x_{2}]_c+(\zeta^2-\zeta)q_{12}x_{2}x_{(13)}=0.
\end{aligned}
\end{align}
Here $\Oc_+^{\bq}= \{1, 23, 123, 12^2, 12^33, 1^22^33, 2, 12, 12^23 \}$ and the degree of the integral is
\begin{equation*}
\ya= 26\alpha_1 + 48\alpha_2 + 17\alpha_3.
\end{equation*}

\subsubsection{The associated Lie algebra} This is of type $C_3$.

\subsection{\ Type $\gtt(2,3)$}\label{subsec:type-g(2,3)}
Here $\theta = 3$, $\zeta \in \G'_{3}$. Let
\begin{align*}
A&=\begin{pmatrix} 0 & 1 & 0 \\ -1 & 2 & -2 \\ 0 & 1 & 0 \end{pmatrix}
\in \kk^{3\times 3},& 
\pa&= (-1,1,-1) \in \G_2^3.
\end{align*}
Let $\g(2, 3) = \g(A, \pa)$, the contragredient Lie
superalgebra corresponding to $(A, \pa)$. 
Then $\sdim \g(2, 3) = 12|14$ \cite{BGL}. 
There are 4 other pairs of matrices and parity vectors for which the associated contragredient Lie superalgebra is isomorphic to $\g(2,3)$. We describe now its root system $\gtt(2,3)$,
see \cite{AA-GRS-CLS-NA}.

\subsubsection{Basic datum and root system}
Below, $A_3$, $A_2^{(1)}$, $C_3$ and $C_2^{(1)}$ are numbered as in \eqref{eq:dynkin-system-A}, \eqref{eq:An-(1)}, \eqref{eq:dynkin-system-C} and  	
\eqref{eq:Cn-(1)}, respectively.
The basic datum and the bundle of Cartan matrices are described by the following diagram:
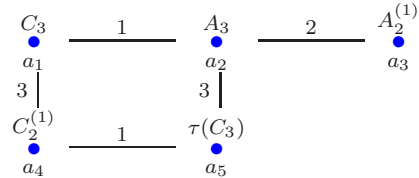
\begin{center}
	\begin{tabular}{c c c c c c c c c c c c}
		&&
		& $\overset{C_3}{\underset{a_1}{\vtxgpd}}$
		& \hspace{-5pt}\raisebox{3pt}{$\overset{1}{\rule{40pt}{0.5pt}}$}\hspace{-5pt}
		& $\overset{A_3}{\underset{a_2}{\vtxgpd}}$
		& \hspace{-5pt}\raisebox{3pt}{$\overset{2}{\rule{40pt}{0.5pt}}$}\hspace{-5pt}
		& $\overset{A_2^{(1)}}{\underset{a_3}{\vtxgpd}}$ & &
		\\
		&& & {\scriptsize 3} \vline\hspace{5pt} & & {\scriptsize 3} \vline\hspace{5pt} & & &&
		\\
		&& &  $\overset{C_2^{(1)}}{\underset{a_4}{\vtxgpd}}$
		&\hspace{-5pt}\raisebox{3pt}{$\overset{1}{\rule{40pt}{0.5pt}}$}\hspace{-5pt}
		& $\overset{\tau(C_3)}{\underset{a_5}{\vtxgpd}}$
		& & & &
	\end{tabular}
\end{center}
Using the notation \eqref{eq:notation-root-exceptional}, the bundle of root sets is the following:
\begin{align*}
\varDelta_{+}^{a_1} &= \{ 1, 12, 2, 12^23, 123, 12^33^2, 2^23, 12^23^2, 23, 3 \}=\tau (\varDelta_{+}^{a_5}),\\
\varDelta_{+}^{a_2} &= \{ 1, 12, 2, 1^22^23, 12^23, 1^22^33^2, 123, 12^23^2, 23, 3 \},\\
\varDelta_{+}^{a_3} &= \{ 1, 12, 2, 1^223, 123, 1^223^2, 13, 123^2, 23, 3 \} ,\\
\varDelta_{+}^{a_4} &= \{ 1, 12, 12^2, 2, 1^22^33, 1^22^23, 12^23, 123, 23, 3 \}.
\end{align*}

\subsubsection{Weyl groupoid}
\label{subsubsec:type-g23-Weyl}
The isotropy group  at $a_3 \in \cX$ is
\begin{align*}
\cW(a_3)= \langle \varsigma_2^{a_3}\varsigma_1 \varsigma_3 \varsigma_2 \varsigma_3 \varsigma_1 \varsigma_2, \varsigma_1^{a_3}, \varsigma_3^{a_3} \rangle \simeq \Z/2  \times W(A_2).
\end{align*}

\subsubsection{Incarnation}
To describe it, we need the matrices $(\bq^{(i)})_{i\in\I_4}$, from left to right and  from up to down:
\begin{align}\label{eq:dynkin-g(2,3)}
\begin{aligned}
&\xymatrix{ \overset{-1}{\underset{\ }{\circ}}\ar  @{-}[r]^{\ztu}  &
	\overset{\zeta}{\underset{\ }{\circ}} \ar  @{-}[r]^{\zeta}  & \overset{-1}{\underset{\
		}{\circ}}}
& &\xymatrix{ \overset{-1}{\underset{\ }{\circ}}\ar  @{-}[r]^{\zeta}  &
	\overset{-1}{\underset{\ }{\circ}} \ar  @{-}[r]^{\zeta}  & \overset{-1}{\underset{\
		}{\circ}}}
\\
&\xymatrix{ \overset{-1}{\underset{\ }{\circ}}\ar  @{-}[r]^{\ztu}  &
	\overset{-\ztu}{\underset{\ }{\circ}} \ar  @{-}[r]^{\ztu}  & \overset{-1}{\underset{\ }{\circ}}}
&
&\xymatrix@R8pt{ & \overset{\zeta}{\underset{\ }{\circ}} \ar@{-}[rd]^{\ztu} &
	\\
	\overset{\zeta}{\underset{\ }{\circ}} \ar@{-}[rr]^{\ztu} \ar@{-}[ru]^{\ztu} & & \overset{-1}{\underset{\ }{\circ}}}.
\end{aligned}
\end{align}
Now, this is the incarnation:
\begin{align*}
a_i & \mapsto \bq^{(i)}, \quad i\in\I_4; & 
a_5 & \mapsto \tau(\bq^{(1)}).
\end{align*}

\subsubsection{PBW-basis and dimension} \label{subsubsec:type-g23-PBW}
Notice that the roots in each $\varDelta_{+}^{a_i}$, $i\in\I_{5}$, are ordered from left to right, justifying the notation $\beta_1, \dots, \beta_{10}$.

The root vectors $x_{\beta_k}$ are described as in Remark \ref{rem:lyndon-word}.
Thus
\begin{align*}
\left\{ x_{\beta_{10}}^{n_{10}} \dots x_{\beta_2}^{n_{2}}  x_{\beta_1}^{n_{1}} \, | \, 0\le n_{k}<N_{\beta_k} \right\}.
\end{align*}
is a PBW-basis of $\toba_{\bq}$. Hence $\dim \toba_{\bq}= 2^63^36=2^73^4$.

\subsubsection{The Dynkin diagram \emph{(\ref{eq:dynkin-g(2,3)}
		a)}}\label{subsubsec:g(2,3)-a}

\

The Nichols algebra $\toba_{\bq}$ is generated by $(x_i)_{i\in \I_3}$ with defining relations
\begin{align}\label{eq:rels-g(2,3)-a}
\begin{aligned}
&[x_{223},x_{23}]_c=0; & x_{13}&=0; & x_{221}&=0; \\
&[[x_{(13)},x_2]_c,x_2]_c=0; & x_{2}^3&=0;  & x_1^2&=0; \quad x_3^2=0; \\
&[x_{(13)},x_2]_c^3=0; & x_{23}^6&=0; & x_{(13)}^3&=0.
\end{aligned}
\end{align}
Here $\Oc_+^{\bq}= \{123, 12^23, 2, 23 \}$ and the degree of the integral is
\begin{equation*}
\ya= 8\alpha_1 + 21\alpha_2 + 15\alpha_3.
\end{equation*}

\subsubsection{The Dynkin diagram \emph{(\ref{eq:dynkin-g(2,3)}
		b)}}\label{subsubsec:g(2,3)-b}

\

The Nichols algebra $\toba_{\bq}$ is generated by $(x_i)_{i\in \I_3}$ with defining relations
\begin{align}\label{eq:rels-g(2,3)-b}
\begin{aligned}
& [[x_{12},x_{(13)}]_c,x_2]_c=0; & x_{13}&=0; & x_{(13)}^6&=0; \\
& [[x_{32},x_{321}]_c,x_2]_c=0; & x_{1}^2&=0; & x_2^2&=0; \quad x_3^2=0; \\
& [x_{(13)},x_2]_c^3=0; & x_{12}^3&=0; & x_{23}^3&=0.
\end{aligned}
\end{align}
Here $\Oc_+^{\bq}= \{12, 12^23, 23, 123 \}$ and the degree of the integral is
\begin{equation*}
\ya= 15\alpha_1 + 21\alpha_2 + 15\alpha_3.
\end{equation*}

\subsubsection{The Dynkin diagram \emph{(\ref{eq:dynkin-g(2,3)}
		c)}}\label{subsubsec:g(2,3)-c}

\

The Nichols algebra $\toba_{\bq}$ is generated by $(x_i)_{i\in \I_3}$ with defining relations
\begin{align}
\begin{aligned}
\begin{aligned}
x_{13}&=0; & x_{2221}&=0; & x_{2223}&=0; &  & x_1^2=0; \quad x_3^2=0; \\ \label{eq:rels-g(2,3)-c}
x_{2}^6&=0; & x_{12}^3&=0; & x_{23}^3&=0; &  & [x_{(13)},x_2]_c^3=0;
\end{aligned}
\\
[x_{1},x_{223}]_c+q_{23}[x_{(13)},x_2]_c-(1-\zeta)q_{12}x_2x_{(13)}=0.
\end{aligned}
\end{align}
Here $\Oc_+^{\bq}= \{12, 12^23, 23, 2 \}$ and the degree of the integral is
\begin{equation*}
\ya= 8\alpha_1 + 21\alpha_2 + 8\alpha_3.
\end{equation*}

\subsubsection{The Dynkin diagram \emph{(\ref{eq:dynkin-g(2,3)}
		d)}}\label{subsubsec:g(2,3)-d}

\

The Nichols algebra $\toba_{\bq}$ is generated by $(x_i)_{i\in \I_3}$ with defining relations
\begin{align}
\begin{aligned}
&\begin{aligned}
x_{112}&=0; & x_{113}&=0; & x_{331}&=0; &  x_{332}&=0;\\ \label{eq:rels-g(2,3)-d}
x_{1}^3&=0; & x_{3}^3&=0; & x_{13}^3&=0; &  x_{(13)}^6&=0;
\end{aligned}
\\
& x_2^2=0; \qquad x_{(13)} -
q_{23}\zeta [x_{13},x_{2}]_c-q_{12}(1-\ztu)x_2x_{13}=0.
\end{aligned}
\end{align}
Here $\Oc_+^{\bq}= \{1, 13, 3, 123 \}$ and the degree of the integral is
\begin{equation*}
\ya= 15\alpha_1 + 11\alpha_2 + 15\alpha_3.
\end{equation*}

\subsubsection{The associated Lie algebra} This is of type $A_2\times A_1$.

\subsection{\ Type $\gtt(3, 3)$}\label{subsec:type-g(3,3)}
Here $\theta = 4$, $\zeta \in \G'_3$. Let 
\begin{align*}
A&=\begin{pmatrix} 2 & -1 & 0 & 0 \\ -1 & 2 & -1 & 0 \\ 0 & -2 & 2 & -1 \\ 0 & 0 & 1 & 0 \end{pmatrix}
\in \kk^{4\times 4},& 
\pa&= (1,1, -1, -1) \in \G_2^3.
\end{align*}
Let $\g(3, 3) = \g(A, \pa)$, the contragredient Lie
superalgebra corresponding to $(A, \pa)$. 
Then $\sdim \g(1, 6) = 23|16$ \cite{BGL}. 
There are 6 other pairs of matrices and parity vectors for which the associated contragredient Lie superalgebra is isomorphic to $\g(3,3)$. We describe now its root system $\gtt(3,3)$,
see \cite{AA-GRS-CLS-NA}.

\subsubsection{Basic datum and root system}
Below, $F_4$, $A_4$, $_{1}T$, $D_4$ and $A_5^{(2)}$ are numbered as in \eqref{eq:dynkin-system-F}, \eqref{eq:dynkin-system-A}, \eqref{eq:mTn}, \eqref{eq:dynkin-system-D} and  	
\eqref{eq:A2-2n-1}, respectively.
The basic datum and the bundle of Cartan matrices are described by the following diagram:
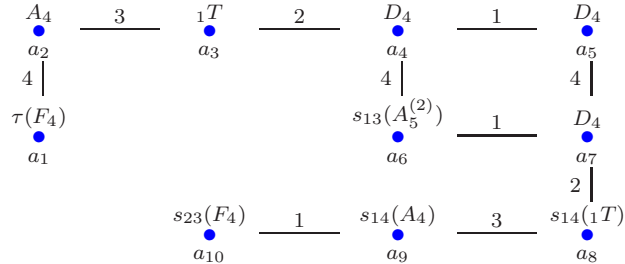
\begin{center}
	\begin{tabular}{c c c c c c c c c}
		$\overset{A_4}{\underset{a_2}{\vtxgpd}}$
		& \hspace{-5pt}\raisebox{3pt}{$\overset{3}{\rule{30pt}{0.5pt}}$}\hspace{-5pt}
		& $\overset{{}_1T}{\underset{a_3}{\vtxgpd}}$
		& \hspace{-5pt}\raisebox{3pt}{$\overset{2}{\rule{30pt}{0.5pt}}$}\hspace{-5pt}
		& $\overset{D_4}{\underset{a_4}{\vtxgpd}}$
		& \hspace{-5pt}\raisebox{3pt}{$\overset{1}{\rule{30pt}{0.5pt}}$}\hspace{-5pt}
		& $\overset{D_4}{\underset{a_5}{\vtxgpd}}$     & &
		\\
		{\scriptsize 4} \vline\hspace{5pt}
		& & & &
		{\scriptsize 4} \vline\hspace{5pt}
		& & {\scriptsize 4} \vline\hspace{5pt} & &
		\\
		$\overset{\tau(F_4)}{\underset{a_1}{\vtxgpd}}$
		& & & & $\overset{s_{13}(A_5^{(2)})}{\underset{a_6}{\vtxgpd}}$
		& \hspace{-5pt}\raisebox{3pt}{$\overset{1}{\rule{30pt}{0.5pt}}$}\hspace{-5pt}
		& $\overset{D_4}{\underset{a_7}{\vtxgpd}}$ & &
		\\
		& & & & & & {\scriptsize 2} \vline\hspace{5pt}   & &
		\\
		& & $\overset{s_{23}(F_4)}{\underset{a_{10}}{\vtxgpd}}$
		& \hspace{-5pt}\raisebox{3pt}{$\overset{1}{\rule{30pt}{0.5pt}}$}\hspace{-5pt}
		& $\overset{s_{14}(A_4)}{\underset{a_9}{\vtxgpd}}$
		& \hspace{-5pt}\raisebox{3pt}{$\overset{3}{\rule{30pt}{0.5pt}}$}\hspace{-5pt}
		& $\overset{s_{14}({}_1T)}{\underset{a_8}{\vtxgpd}}$ & &
	\end{tabular}
\end{center}
Using the notation \eqref{eq:notation-root-exceptional}, the bundle of root sets is the following:
{\scriptsize
	\begin{align*}
	\varDelta_{+}^{a_1}= & \{ 1, 12, 2, 123, 12^23^2, 123^2, 23, 23^2, 3, 12^23^34, 12^23^24, 123^24, 23^24, 1234,234, 34, 4 \}, \\
	\varDelta_{+}^{a_2}= & \{ 1, 12, 2, 123, 23, 3, 12^23^24, 123^24, 1234, 12^23^34^2, 23^24, 12^23^24^2, 234, 123^24^2, 23^24^2, 34, 4 \}, \\
	\varDelta_{+}^{a_3}= & \{ 1, 12, 2, 12^23, 123, 12^23^2, 23, 3, 1^22^33^24, 12^33^24, 12^23^24, 12^234, 1234,124, 234, 24, 4 \}, \\
	\varDelta_{+}^{a_4}= & \{ 1, 12, 2, 123, 23, 3, 12^23^24, 12^234, 1^22^33^24^2, 1234, 12^33^24^2, 12^23^24^2,234, 12^234^2, 124, 24, 4 \}, \\
	\varDelta_{+}^{a_5}= & \{ 1, 12, 12^2, 2, 12^23, 123, 23, 3, 1^22^334, 12^334, 12^234, 1234, 12^24, 124, 234, 24, 4 \}, \\
	\varDelta_{+}^{a_6}= & \{ 1, 12, 2, 123, 23, 3, 12^23^24, 12^234, 123^24, 23^24, 1234, 12^23^24^2, 124, 234, 34, 24, 4 \}, \\
	\varDelta_{+}^{a_7}= & s_{14}(\varDelta_{+}^{a_4}), \qquad \varDelta_{+}^{a_8}=s_{14}(\varDelta_{+}^{a_3}), \qquad \varDelta_{+}^{a_9}=s_{14}(\varDelta_{+}^{a_2}), \qquad \varDelta_{+}^{a_{10}}=s_{14}(\varDelta_{+}^{a_1}).
	\end{align*}
}%

\subsubsection{Weyl groupoid}
\label{subsubsec:type-g33-Weyl}
The isotropy group  at $a_5 \in \cX$ is
\begin{align*}
\cW(a_5)= \langle \varsigma_1^{a_5}, \varsigma_2^{a_5},  \varsigma_3^{a_5} \rangle \simeq W(B_3).
\end{align*}

\subsubsection{Incarnation}
To describe it, we need the matrices $(\bq^{(i)})_{i\in\I_6}$, from left to right and  from up to down:
\begin{align}\label{eq:dynkin-g(3,3)}
\begin{aligned}
&\xymatrix{ \overset{\ztu}{\underset{\ }{\circ}}\ar  @{-}[r]^{\zeta}  &
	\overset{\ztu}{\underset{\ }{\circ}} \ar  @{-}[r]^{\zeta}  & \overset{\zeta}{\underset{\
		}{\circ}}
	\ar  @{-}[r]^{\ztu}  & \overset{-1}{\underset{\ }{\circ}}}
& &\xymatrix{ \overset{\ztu}{\underset{\ }{\circ}}\ar  @{-}[r]^{\zeta}  &
	\overset{\ztu}{\underset{\ }{\circ}} \ar  @{-}[r]^{\zeta}  & \overset{-1}{\underset{\
		}{\circ}}
	\ar  @{-}[r]^{\zeta}  & \overset{-1}{\underset{\ }{\circ}}}
\\
& \xymatrix@R-6pt{  &    \overset{\zeta}{\circ} \ar  @{-}[d]_{\ztu}\ar  @{-}[dr]^{\ztu} & \\
	\overset{\ztu}{\underset{\ }{\circ}} \ar  @{-}[r]^{\zeta}  & \overset{-1}{\underset{\
		}{\circ}} \ar  @{-}[r]^{\ztu}  & \overset{-1}{\underset{\ }{\circ}}}
&& \xymatrix@R-6pt{  &    \overset{-1}{\circ} \ar  @{-}[d]^{\ztu} & \\ \overset{-1}{\underset{\ }{\circ}} \ar  @{-}[r]^{\zeta}  & \overset{-1}{\underset{\
		}{\circ}} \ar  @{-}[r]^{\zeta}  & \overset{\ztu}{\underset{\ }{\circ}}}
\\
&\xymatrix@R-6pt{  &    \overset{-1}{\circ} \ar  @{-}[d]^{\zeta} & \\ \overset{-1}{\underset{\ }{\circ}} \ar  @{-}[r]^{\zeta}  & \overset{\ztu}{\underset{\
		}{\circ}} \ar  @{-}[r]^{\zeta}  & \overset{\ztu}{\underset{\ }{\circ}}}
& &\xymatrix@R-6pt{  &    \overset{-1}{\circ} \ar  @{-}[d]^{\ztu} & \\ \overset{-1}{\underset{\ }{\circ}} \ar  @{-}[r]^{\zeta}  & \overset{\zeta}{\underset{\
		}{\circ}} \ar  @{-}[r]^{\ztu}  & \overset{\ztu}{\underset{\ }{\circ}}}
\end{aligned}
\end{align}
Now, this is the incarnation:
\begin{align*}
a_i & \mapsto \bq^{(i)}, \quad i\in\I_6; & 
a_i & \mapsto s_{14}(\bq^{(11-i)}), \quad i\in\I_{7,10}.
\end{align*}

\subsubsection{PBW-basis and dimension} \label{subsubsec:type-g33-PBW}
Notice that the roots in each $\varDelta_{+}^{a_i}$, $i\in\I_{10}$, are ordered from left to right, justifying the notation $\beta_1, \dots, \beta_{17}$.

The root vectors $x_{\beta_k}$ are described as in Remark \ref{rem:lyndon-word}.
Thus
\begin{align*}
\left\{ x_{\beta_{17}}^{n_{17}} \dots x_{\beta_2}^{n_{2}}  x_{\beta_1}^{n_{1}} \, | \, 0\le n_{k}<N_{\beta_k} \right\}.
\end{align*}
is a PBW-basis of $\toba_{\bq}$. Hence $\dim \toba_{\bq}= 2^83^9$.

\subsubsection{The Dynkin diagram \emph{(\ref{eq:dynkin-g(3,3)}
		a)}}\label{subsubsec:g(3,3)-a}

\

The Nichols algebra $\toba_{\bq}$ is generated by $(x_i)_{i\in \I_4}$ with defining relations
\begin{align}\label{eq:rels-g(3,3)-a}
\begin{aligned}
x_{13}&=0; &x_{14}&=0; & x_{24}&=0; & [x_{3321},&x_{32}]_c=0;\\
&& x_{112}&=0; & x_{221}&=0; & [[x_{(24)},&x_{3}]_c,x_3]_c=0;\\
x_{223}&=0; &x_{334}&=0; & x_{4}^2&=0; & x_{\alpha}^3&=0, \ \alpha\in\Oc_+^{\bq}.
\end{aligned}
\end{align}
Here $\Oc_+^{\bq}=\{ 1, 12, 2, 123, 12^23^2, 123^2, 23, 23^2, 3 \}$ 
and the degree of the integral is
\begin{equation*}
\ya= 14\alpha_1 + 24\alpha_2 + 30\alpha_3 + 8\alpha_4.
\end{equation*}

\subsubsection{The Dynkin diagram \emph{(\ref{eq:dynkin-g(3,3)}
		b)}}\label{subsubsec:g(3,3)-b}

\

The Nichols algebra $\toba_{\bq}$ is generated by $(x_i)_{i\in \I_4}$ with defining relations
\begin{align}\label{eq:rels-g(3,3)-b}
\begin{aligned}
x_{223}&=0; & x_{24}&=0; & [[x_{43},&x_{432}]_c,x_3]_c=0;\\
x_{112}&=0; & x_{221}&=0; & x_{13}&=0; \quad x_{14}=0; \\
x_{3}^2&=0; & x_{4}^2&=0; & x_{\alpha}^3&=0, \ \alpha\in\Oc_+^{\bq}.
\end{aligned}
\end{align}
Here $\Oc_+^{\bq}=\{ 1, 12, 2, 1234, 12^23^24^2, 234, 123^24^2, 23^24^2, 34  \}$
and the degree of the integral is
\begin{equation*}
\ya= 14\alpha_1 + 24\alpha_2 + 30\alpha_3 + 24\alpha_4.
\end{equation*}

\subsubsection{The Dynkin diagram \emph{(\ref{eq:dynkin-g(3,3)} c)}}\label{subsubsec:g(3,3)-c}

\

The Nichols algebra $\toba_{\bq}$ is generated by $(x_i)_{i\in \I_4}$ with defining relations
\begin{align}\label{eq:rels-g(3,3)-c}
\begin{aligned}
&\begin{aligned}
x_{14}&=0; & x_{112}&=0; & [x_{(13)},&x_{2}]_c=0;\\
x_{442}&=0; & x_{443}&=0; & [x_{124},&x_{2}]_c=0;\\
x_{2}^2&=0; & x_{3}^2&=0; & x_{\alpha}^3&=0, \ \alpha\in\Oc_+^{\bq};
\end{aligned}
\\
&x_{13}=0; \quad x_{(24)}=\zeta q_{34}[x_{24},x_3]_c+q_{23}(1-\ztu)x_3x_{24}.
\end{aligned}
\end{align}
Here  $\Oc_+^{\bq}=\{ 1, 4, 1234^2, 23^24, 1234, 12^23^24^2, 123, 234, 23 \}$
and the degree of the integral is
\begin{equation*}
\ya= 14\alpha_1 + 24\alpha_2 + 20\alpha_3 + 24\alpha_4.
\end{equation*}

\subsubsection{The Dynkin diagram \emph{(\ref{eq:dynkin-g(3,3)}
		d)}}\label{subsubsec:g(3,3)-d}

\

The Nichols algebra $\toba_{\bq}$ is generated by $(x_i)_{i\in \I_4}$ with defining relations
\begin{align}\label{eq:rels-g(3,3)-d}
\begin{aligned}
x_{13}&=0; &x_{14}&=0; & x_{34}&=0; & [x_{124},&x_{2}]_c=0;\\
&& x_{332}&=0; & x_{2}^2&=0; & [x_{324},&x_{2}]_c=0;\\
& & x_{1}^2&=0; & x_{4}^2&=0; & x_{\alpha}^3&=0, \ \alpha\in\Oc_+^{\bq}.
\end{aligned}
\end{align}
Here  $\Oc_+^{\bq}=\{ 3, 123, 1^22^23, 12, 1^22^33^24, 1^22^334, 12^234, 234, 24 \}$
and the degree of the integral is
\begin{equation*}
\ya= 24\alpha_1 + 36\alpha_2 + 20\alpha_3 + 14\alpha_4.
\end{equation*}

\subsubsection{The Dynkin diagram \emph{(\ref{eq:dynkin-g(3,3)}
		e)}}\label{subsubsec:g(3,3)-e}

\

The Nichols algebra $\toba_{\bq}$ is generated by $(x_i)_{i\in \I_4}$ with defining relations
\begin{align}\label{eq:rels-g(3,3)-e}
\begin{aligned}
x_{221}&=0; & x_{13}&=0; & & x_{332}=0;
\\
x_{223}&=0; & x_{14}&=0; & & x_{1}^2=0; \quad x_{4}^2=0;
\\
x_{224}&=0; & x_{34}&=0; & & x_{\alpha}^3=0, \ \alpha\in\Oc_+^{\bq}.
\end{aligned}
\end{align}
Here $\Oc_+^{\bq}=\{ 3, 23, 2, 12^234, 1^22^33^24^2, 1234, 1^22^334^2, 1^22^234^2, 124 \}$
and the degree of the integral is
\begin{equation*}
\ya= 24\alpha_1 + 36\alpha_2 + 20\alpha_3 + 24\alpha_4.
\end{equation*}

\subsubsection{The Dynkin diagram \emph{(\ref{eq:dynkin-g(3,3)}
		f)}}\label{subsubsec:g(3,3)-f}

\

The Nichols algebra $\toba_{\bq}$ is generated by $(x_i)_{i\in \I_4}$ with defining relations
\begin{align}\label{eq:rels-g(3,3)-f}
\begin{aligned}
x_{13}&=0; &x_{14}&=0; & x_{34}&=0; & [[x_{(13)},&x_2]_c,x_2]_c=0;\\
&& x_{223}&=0; & x_{224}&=0; & [[x_{324},&x_2]_c,x_2]_c=0;\\
x_{332}&=0 & x_{1}^2&=0; & x_{4}^2&=0; & x_{\alpha}^3&=0, \ \alpha\in\Oc_+^{\bq}.
\end{aligned}
\end{align}
Here $\Oc_+^{\bq}=\{ 1, 12, 12^2, 2, 1^22^334, 12^334, 12^234, 1234, 234 \}$
and the degree of the integral is
\begin{equation*}
\ya= 14\alpha_1 + 36\alpha_2 + 20\alpha_3 + 14\alpha_4.
\end{equation*}

\subsubsection{The associated Lie algebra} This is of type $B_3$.	

\subsection{\ Type $\gtt(4,3)$}\label{subsec:type-g(4,3)}

Here $\theta = 4$, $\zeta \in \G'_3$. Let 
\begin{align*}
A&=\begin{pmatrix}  2 & -1 & 0 & 0 \\ -1 & 0 & 1 & 0 \\ 0 & -2 & 2 & -2 \\ 0 & 0 & 1 & 0 \end{pmatrix}
\in \kk^{4\times 4},& 
\pa&= (1,-1,-1,-1) \in \G_2^4.
\end{align*}
Let $\g(4, 3) = \g(A, \pa)$, the contragredient Lie
superalgebra corresponding to $(A, \pa)$. 
Then $\sdim \g(2, 3) = 24|26$ \cite{BGL}. 
There are 9 other pairs of matrices and parity vectors for which the associated contragredient Lie superalgebra is isomorphic to $\g(4,3)$. We describe now its root system $\gtt(4,3)$,
see \cite{AA-GRS-CLS-NA}.

\subsubsection{Basic datum and root system}
Below, $C_n^{(1)\wedge}$, $C_4$, $F_4$, $A_4$, $_1 T$, $D_4$ and $A_5^{(2)}$ are numbered as in  \eqref{eq:Cn1-indef}, \eqref{eq:dynkin-system-C}, \eqref{eq:dynkin-system-F}, \eqref{eq:dynkin-system-A},  \eqref{eq:mTn}, \eqref{eq:dynkin-system-D} and \eqref{eq:A2-2n-1}, respectively.
The basic datum and the bundle of Cartan matrices are described by the following diagram, called $\gtt(4,3)$:
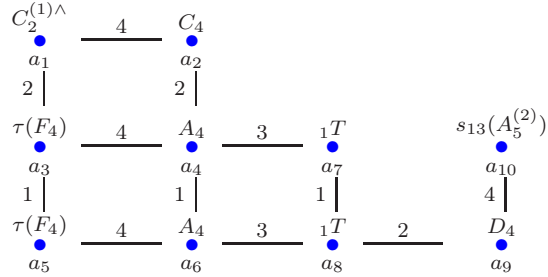
\begin{center}
	\begin{tabular}{c c c c c c c c c}
		$\overset{C_2^{(1)\wedge}}{\underset{a_1}{\vtxgpd}}$
		& \hspace{-5pt}\raisebox{3pt}{$\overset{4}{\rule{30pt}{0.5pt}}$}\hspace{-5pt}
		& $\overset{C_4}{\underset{a_2}{\vtxgpd}}$
		& & & & & &
		\\
		{\scriptsize 2} \vline\hspace{5pt}
		& & {\scriptsize 2} \vline\hspace{5pt}
		& &
		& & & &
		\\
		$\overset{\tau(F_4)}{\underset{a_3}{\vtxgpd}}$
		& \hspace{-5pt}\raisebox{3pt}{$\overset{4}{\rule{30pt}{0.5pt}}$}\hspace{-5pt}
		& $\overset{A_4}{\underset{a_4}{\vtxgpd}}$
		& \hspace{-5pt}\raisebox{3pt}{$\overset{3}{\rule{30pt}{0.5pt}}$}\hspace{-5pt}
		& $\overset{{}_1T}{\underset{a_7}{\vtxgpd}}$
		& & $\overset{s_{13}(A_5^{(2)})}{\underset{a_{10}}{\vtxgpd}}$ & &
		\\
		{\scriptsize 1} \vline\hspace{5pt}
		& &{\scriptsize 1} \vline\hspace{5pt}
		& & {\scriptsize 1} \vline\hspace{5pt}
		& & {\scriptsize 4} \vline\hspace{5pt} & &
		\\
		$\overset{\tau(F_4)}{\underset{a_5}{\vtxgpd}}$
		& \hspace{-5pt}\raisebox{3pt}{$\overset{4}{\rule{30pt}{0.5pt}}$}\hspace{-5pt}
		& $\overset{A_4}{\underset{a_6}{\vtxgpd}}$
		& \hspace{-5pt}\raisebox{3pt}{$\overset{3}{\rule{30pt}{0.5pt}}$}\hspace{-5pt}
		& $\overset{{}_1T}{\underset{a_8}{\vtxgpd}}$
		& \hspace{-5pt}\raisebox{3pt}{$\overset{2}{\rule{30pt}{0.5pt}}$}\hspace{-5pt}
		& $\overset{D_4}{\underset{a_9}{\vtxgpd}}$  & &
	\end{tabular}
\end{center}
Using the notation \eqref{eq:notation-root-exceptional}, the bundle of root sets is the following: {\scriptsize
	\begin{align*}
	\varDelta_{+}^{a_1}= & \{ 1, 12, 2, 123, 12^23^2, 123^2, 23,  23^2, 3, 12^23^44, 12^23^24, 123^34, 23^34, \\
	& 12^23^24, 123^24, 12^23^44^2, 1234, 23^24, 234, 3^24, 34, 4 \}, \\
	\varDelta_{+}^{a_2}= & \{ 1, 12, 2, 123, 23, 3, 12^23^24, 123^24, 1234, 12^23^44^2, 12^23^34^2, 23^24, \\
	& 12^23^24^2, 234, 12^23^44^3, 123^34^2, 123^24^2, 23^34^2, 23^24^2, 3^24, 34, 4 \}, \\
	\varDelta_{+}^{a_3}= & \{ 1, 12, 2, 123, 12^23^2, 123^2, 23, 23^2, 3, 1^22^33^44, 1^22^33^34, 1^22^23^34, 12^23^24, \\
	& 1^22^23^24, 12^23^24, 1^22^33^44^2, 123^24, 1234, 23^24, 234, 34, 4 \}, \\
	\varDelta_{+}^{a_4}= & \{ 1, 12, 2, 123, 23, 3, 1^22^23^24, 12^23^24, 1^22^33^44^2, 1^22^33^34^2, 123^24, \\
	& 1^22^23^34^2, 1234, 1^22^33^44^3, 12^23^34^2, 12^23^24^2, 123^24^2, 23^24, 234, 23^24^2, 34, 4 \}, \\
	\varDelta_{+}^{a_5}= & \{ 1, 12, 2, 123, 12^23^2, 123^2, 23, 23^2, 3, 12^33^44, 12^33^34, 12^23^24, 2^23^34, \\
	& 12^23^24, 2^23^24, 12^33^44^2, 123^24, 1234, 23^24, 234, 34, 4 \}, \\
	\varDelta_{+}^{a_6}= & \{ 1, 12, 2, 123, 23, 3, 12^23^24, 123^24, 1234, 12^33^44^2, 12^33^34^2, 2^23^24, \\
	& 12^23^34^2, 23^24, 12^33^44^3, 12^23^24^2, 123^24^2, 2^23^34^2, 234, 23^24^2, 34, 4 \}, \\
	\varDelta_{+}^{a_7}= & \{ 1, 12, 2, 123, 23, 3, 1^22^234, 12^234, 1^22^33^24^2, 1^22^334^2, 1234, 1^22^234^2, \\
	& 124, 1^22^33^24^3, 12^23^24^2, 12^234^2, 1234^2, 234, 34, 234^2, 24, 4 \}, \\
	\varDelta_{+}^{a_8}= & \{ 1, 12, 2, 123, 23, 3, 12^33^24, 12^23^24, 12^234, 2^23^24, 2^234, 12^33^34^2, 123^24, \\
	& 23^24, 12^33^24^2, 1234, 12^23^24^2, 124, 234, 24, 34, 4 \}, \\
	\varDelta_{+}^{a_9}= & \{ 1, 12, 2, 123, 23, 3, 1^22^33^24, 1^22^23^24, 1^22^234, 12^33^24, 12^23^24, 1^22^43^34^2, \\
	& 2^23^24, 12^234, 2^234, 1^22^33^24^2, 12^33^24^2, 1234, 124, 234, 24, 4 \}, \\
	\varDelta_{+}^{a_{10}}= & \{ 1, 12, 2, 123, 23, 3, 1^22^334, 1^22^234, 1^22^24, 12^334, 12^234, 1^22^434^2, 2^234, \\
	& 1234, 1^22^334^2, 234, 12^334^2, 12^24, 124, 2^24, 24, 4 \}.
	\end{align*}
}%

\subsubsection{Weyl groupoid}
\label{subsubsec:type-g43-Weyl}
The isotropy group  at $a_6 \in \cX$ is
\begin{align*}
\cW(a_6)= \langle \varsigma_1^{a_6}, \varsigma_2^{a_6}, \varsigma_3^{a_6}, \varsigma_4^{a_6}\varsigma_2 \varsigma_3\varsigma_4 \varsigma_1 \varsigma_2 \varsigma_3\varsigma_2 \varsigma_1 \varsigma_4 \varsigma_3 \varsigma_2 \varsigma_4 \rangle \simeq W(C_3) \times \Z/2.
\end{align*}

\subsubsection{Incarnation}
To describe it, we need the matrices $(\bq^{(i)})_{i\in\I_{10}}$, from left to right and  from up to down:

\begin{align}\label{eq:dynkin-g(4,3)}
\begin{aligned}
&\xymatrix{ \overset{\ztu}{\underset{\ }{\circ}}\ar  @{-}[r]^{\zeta}  &
	\overset{-1}{\underset{\ }{\circ}} \ar  @{-}[r]^{\ztu}  & \overset{-\ztu}{\underset{\
		}{\circ}}
	\ar  @{-}[r]^{\ztu}  & \overset{-1}{\underset{\ }{\circ}}}
& &\xymatrix{ \overset{\ztu}{\underset{\ }{\circ}}\ar  @{-}[r]^{\zeta}  &
	\overset{-1}{\underset{\ }{\circ}} \ar  @{-}[r]^{\ztu}  & \overset{\zeta}{\underset{\
		}{\circ}}
	\ar  @{-}[r]^{\zeta}  & \overset{-1}{\underset{\ }{\circ}}}
\\
& \xymatrix{ \overset{-1}{\underset{\ }{\circ}}\ar  @{-}[r]^{\ztu}  &
	\overset{-1}{\underset{\ }{\circ}} \ar  @{-}[r]^{\zeta}  & \overset{\zeta}{\underset{\
		}{\circ}}
	\ar  @{-}[r]^{\ztu}  & \overset{-1}{\underset{\ }{\circ}}}
& & \xymatrix{ \overset{-1}{\underset{\ }{\circ}}\ar  @{-}[r]^{\ztu}  &
	\overset{-1}{\underset{\ }{\circ}} \ar  @{-}[r]^{\zeta}  & \overset{-1}{\underset{\
		}{\circ}}
	\ar  @{-}[r]^{\zeta}  & \overset{-1}{\underset{\ }{\circ}}}
\\
& \xymatrix{ \overset{-1}{\underset{\ }{\circ}}\ar  @{-}[r]^{\zeta}  &
	\overset{\ztu}{\underset{\ }{\circ}} \ar  @{-}[r]^{\zeta}  & \overset{\zeta}{\underset{\
		}{\circ}}
	\ar  @{-}[r]^{\ztu}  & \overset{-1}{\underset{\ }{\circ}}}
&& \xymatrix{ \overset{-1}{\underset{\ }{\circ}}\ar  @{-}[r]^{\zeta}  &
	\overset{\ztu}{\underset{\ }{\circ}} \ar  @{-}[r]^{\zeta}  & \overset{-1}{\underset{\
		}{\circ}}
	\ar  @{-}[r]^{\zeta}  & \overset{-1}{\underset{\ }{\circ}}}
\\
& \xymatrix@R-6pt{  &    \overset{-1}{\circ} \ar  @{-}[d]^{\zeta} & \\
	\overset{\zeta}{\underset{\ }{\circ}} \ar  @{-}[r]^{\ztu}  & \overset{-1}{\underset{\
		}{\circ}} \ar  @{-}[r]^{\zeta}  & \overset{\ztu}{\underset{\ }{\circ}}}
&& \xymatrix@R-6pt{  &    \overset{\zeta}{\circ} \ar  @{-}[d]_{\ztu}\ar  @{-}[dr]^{\ztu} &
	\\
	\overset{-1}{\underset{\ }{\circ}} \ar  @{-}[r]^{\ztu}  & \overset{\zeta}{\underset{\
		}{\circ}} \ar  @{-}[r]^{\ztu}  & \overset{-1}{\underset{\ }{\circ}}}
\\
& \xymatrix@R-6pt{  &    \overset{-1}{\circ} \ar  @{-}[d]^{\ztu} & \\
	\overset{\zeta}{\underset{\ }{\circ}} \ar  @{-}[r]^{\ztu}  & \overset{\zeta}{\underset{\
		}{\circ}} \ar  @{-}[r]^{\zeta}  & \overset{\ztu}{\underset{\ }{\circ}}}
&&\xymatrix@R-6pt{  &    \overset{\zeta}{\circ} \ar  @{-}[d]_{\ztu}\ar  @{-}[dr]^{\ztu} & \\
	\overset{-1}{\underset{\ }{\circ}} \ar  @{-}[r]^{\zeta}  & \overset{-1}{\underset{\
		}{\circ}} \ar  @{-}[r]^{\ztu}  & \overset{-1}{\underset{\ }{\circ}}}
\end{aligned}
\end{align}
Now, this is the incarnation: $a_i\mapsto \bq^{(i)}$, $i\in\I_{10}$.

\subsubsection{PBW-basis and dimension} \label{subsubsec:type-g43-PBW}
Notice that the roots in each $\varDelta_{+}^{a_i}$, $i\in\I_{10}$, are ordered from left to right, justifying the notation $\beta_1, \dots, \beta_{22}$.

The root vectors $x_{\beta_k}$ are described as in Remark \ref{rem:lyndon-word}.
Thus
\begin{align*}
\left\{ x_{\beta_{22}}^{n_{22}} \dots x_{\beta_2}^{n_{2}}  x_{\beta_1}^{n_{1}} \, | \, 0\le n_{k}<N_{\beta_k} \right\}.
\end{align*}
is a PBW-basis of $\toba_{\bq}$. Hence $\dim \toba_{\bq}= 2^{12}3^96=2^{13}3^{10}$.

\subsubsection{The Dynkin diagram \emph{(\ref{eq:dynkin-g(4,3)}
		a)}}\label{subsubsec:g(4,3)-a}

\
The Nichols algebra $\toba_{\bq}$ is generated by $(x_i)_{i\in \I_4}$ with defining relations
\begin{align}\label{eq:rels-g(4,3)-a}
\begin{aligned}
\begin{aligned}
x_{14}&=0; & x_{24}&=0; & [x_{(13)},x_2]_c&=0; & &x_{13}=0;\\
x_{112}&=0; & x_{3332}&=0; & x_{3334}&=0;\\
x_{2}^2&=0; & x_{4}^2&=0; & x_{\alpha}^{N_\alpha}&=0, &  &\alpha\in\Oc_+^{\bq};
\end{aligned}
\\
[x_2,x_{334}]_c-q_{34}[x_{(24)},x_3]_c+(\zeta^2-\zeta)q_{23}x_3x_{(24)}=0.
\end{aligned}
\end{align}
Here  $\Oc_+^{\bq}=\{  1, 123, 12^23^2, 23, 3, 12^23^34, 123^24, 12^23^44^2, 23^24, 34  \}$ 
and the degree of the integral is
\begin{equation*}
\ya= 18\alpha_1 + 32\alpha_2 + 57\alpha_3 + 20\alpha_4.
\end{equation*}

\subsubsection{The Dynkin diagram \emph{(\ref{eq:dynkin-g(4,3)}
		b)}}\label{subsubsec:g(4,3)-b}

\

The Nichols algebra $\toba_{\bq}$ is generated by $(x_i)_{i\in \I_4}$ with defining relations
\begin{align}\label{eq:rels-g(4,3)-b}
\begin{aligned}
x_{13}&=0; &x_{14}&=0; & x_{24}&=0; & [[x_{(24)},&x_3]_c,x_3]_c=0;\\
&& x_{112}&=0; & x_{332}&=0; & [x_{334},&x_{34}]_c=0;\\
[x_{(13)},x_2]_c&=0; & x_{2}^2&=0; & x_{4}^2&=0; & x_{\alpha}^{N_\alpha}&=0, \ \alpha\in\Oc_+^{\bq}.
\end{aligned}
\end{align}
Here $\Oc_+^{\bq}=\{ 1, 3, 123^24, 1234, 12^23^44^2, 12^23^34^2, 23^24, 12^23^24^2, 234, 34 \}$
and the degree of the integral is
\begin{equation*}
\ya= 18\alpha_1 + 32\alpha_2 + 57\alpha_3 + 39\alpha_4.
\end{equation*}

\subsubsection{The Dynkin diagram \emph{(\ref{eq:dynkin-g(4,3)}
		c)}}\label{subsubsec:g(4,3)-c}

\

The Nichols algebra $\toba_{\bq}$ is generated by $(x_i)_{i\in \I_4}$ with defining relations
\begin{align}\label{eq:rels-g(4,3)-c}
\begin{aligned}
x_{13}&=0; &x_{14}&=0; & x_{24}&=0; & [[x_{(24)},&x_{3}]_c,x_3]_c=0;\\
[x_{332},x_{32}]_c&=0; & x_{334}&=0; & x_{1}^2&=0; & [x_{3321},&x_{32}]_c=0;\\
[x_{(13)},x_2]_c&=0; & x_{2}^2&=0; & x_{4}^2&=0; & x_{\alpha}^{N_\alpha}&=0, \ \alpha\in\Oc_+^{\bq}.
\end{aligned}
\end{align}
Here $\Oc_+^{\bq}=\{ 12, 123, 123^2, 23, 3, 12^23^34, 12^23^24, 12^33^44^2, 23^24, 234 \}$
and the degree of the integral is
\begin{equation*}
\ya= 18\alpha_1 + 45\alpha_2 + 57\alpha_3 + 20\alpha_4.
\end{equation*}

\subsubsection{The Dynkin diagram \emph{(\ref{eq:dynkin-g(4,3)}
		d)}}\label{subsubsec:g(4,3)-d}

\

The Nichols algebra $\toba_{\bq}$ is generated by $(x_i)_{i\in \I_4}$ with defining relations
\begin{align}\label{eq:rels-g(4,3)-d}
\begin{aligned}
x_{13}&=0; &x_{14}&=0; & x_{24}&=0; & [[x_{23},&x_{(24)}]_c,x_3]_c=0;\\
&& x_{1}^2&=0; & x_{2}^2&=0; & [[x_{43},&x_{432}]_c,x_3]_c=0;\\
[x_{(13)},x_2]_c&=0; & x_{3}^2&=0; & x_{4}^2&=0; & x_{\alpha}^{N_\alpha}&=0, \ \alpha\in\Oc_+^{\bq}.
\end{aligned}
\end{align}
Here  $\Oc_+^{\bq}=\{ 12, 23, 12^23^24, 1234, 12^33^44^2, 12^23^34^2, 23^24, 123^24^2, 234, 34 \}$
and the degree of the integral is
\begin{equation*}
\ya= 18\alpha_1 + 45\alpha_2 + 57\alpha_3 + 39\alpha_4.
\end{equation*}

\subsubsection{The Dynkin diagram \emph{(\ref{eq:dynkin-g(4,3)}
		e)}}\label{subsubsec:g(4,3)-e}

\

The Nichols algebra $\toba_{\bq}$ is generated by $(x_i)_{i\in \I_4}$ with defining relations
\begin{align}\label{eq:rels-g(4,3)-e}
\begin{aligned}
x_{13}&=0; &x_{14}&=0; & x_{24}&=0; & [[x_{(24)},&x_{3}]_c,x_3]_c=0;\\
& & x_{221}&=0; & x_{223}&=0; & [x_{3321},&x_{32}]_c=0;\\
x_{334}&=0; & x_{1}^2&=0; & x_{4}^2&=0; & x_{\alpha}^{N_\alpha}&=0, \ \alpha\in\Oc_+^{\bq}.
\end{aligned}
\end{align}
Here  $\Oc_+^{\bq}=\{ 2, 123, 23, 23^2, 3, 12^23^34, 12^23^24, 1^22^33^44^2, 123^24, 1234 \}$
and the degree of the integral is
\begin{equation*}
\ya= 29\alpha_1 + 45\alpha_2 + 57\alpha_3 + 20\alpha_4.
\end{equation*}

\subsubsection{The Dynkin diagram \emph{(\ref{eq:dynkin-g(4,3)}
		f)}}\label{subsubsec:g(4,3)-f}

\

The Nichols algebra $\toba_{\bq}$ is generated by $(x_i)_{i\in \I_4}$ with defining relations
\begin{align}\label{eq:rels-g(4,3)-f}
\begin{aligned}
x_{14}&=0; & x_{24}&=0; & & [[x_{43},x_{432}]_c,x_3]_c=0;\\
x_{221}&=0; & x_{223}&=0; & &x_1^2=0; \quad x_{13}=0;\\
x_{3}^2&=0; & x_{4}^2&=0; & &x_{\alpha}^{N_\alpha}=0, \ \alpha\in\Oc_+^{\bq}.
\end{aligned}
\end{align}
Here $\Oc_+^{\bq}=\{  2, 123, 12^23^24, 1^22^33^44^2, 123^24, 1234, 12^23^34^2, 234, 23^24^2, 34 \}$
and the degree of the integral is
\begin{equation*}
\ya= 29\alpha_1 + 45\alpha_2 + 57\alpha_3 + 39\alpha_4.
\end{equation*}

\subsubsection{The Dynkin diagram \emph{(\ref{eq:dynkin-g(4,3)}
		g)}}\label{subsubsec:g(4,3)-g}

\

The Nichols algebra $\toba_{\bq}$ is generated by $(x_i)_{i\in \I_4}$ with defining relations
\begin{align}\label{eq:rels-g(4,3)-g}
\begin{aligned}
&\begin{aligned}
x_{13}&=0; &x_{14}&=0; & x_{221}&=0; & x_{223}&=0; & x_{224}&=0;\\
x_{332}&=0; & x_{334}&=0;  & x_1^2&=0; & x_4^2&=0;  & x_{\alpha}^{N_\alpha}&=0, \ \alpha\in\Oc_+^{\bq};
\end{aligned}
\\
&x_{(24)}-\zeta q_{34}[x_{24},x_3]_c-(1-\ztu)q_{23}x_3x_{24}=0.
\end{aligned}
\end{align}
Here  $\Oc_+^{\bq}=\{ 2, 123, 12^234, 1234, 12^334^2, 12^234^2, 1234^2, 234, 24, 4 \}$
and the degree of the integral is
\begin{equation*}
\ya= 18\alpha_1 + 45\alpha_2 + 39\alpha_3 + 41\alpha_4.
\end{equation*}

\subsubsection{The Dynkin diagram \emph{(\ref{eq:dynkin-g(4,3)}
		h)}}\label{subsubsec:g(4,3)-h}

\

The Nichols algebra $\toba_{\bq}$ is generated by $(x_i)_{i\in \I_4}$ with defining relations
\begin{align}\label{eq:rels-g(4,3)-h}
\begin{aligned}
&\begin{aligned}
&[x_{(13)},x_2]_c=0; & x_{13}&=0; & x_{332}&=0; & & x_2^2=0; && x_4^2=0;
\\
&[x_{124},x_2]_c=0; & x_{14}&=0; & x_{334}&=0; & & x_{\alpha}^{N_\alpha}=0, && \alpha\in\Oc_+^{\bq};
\end{aligned}
\\
& x_{1}^2=0; \quad x_{(24)}-\zeta q_{34}[x_{24},x_3]_c-(1-\ztu)q_{23}x_3x_{24}=0.
\end{aligned}
\end{align}
Here  $\Oc_+^{\bq}=\{ 12, 123, 3, 1^22^33^24, 12^23^24, 12^234, 23^24, 1234, 234, 24 \}$
and the degree of the integral is
\begin{equation*}
\ya= 29\alpha_1 + 45\alpha_2 + 39\alpha_3 + 29\alpha_4.
\end{equation*}

\subsubsection{The Dynkin diagram \emph{(\ref{eq:dynkin-g(4,3)}
		i)}}\label{subsubsec:g(4,3)-i}

\

The Nichols algebra $\toba_{\bq}$ is generated by $(x_i)_{i\in \I_4}$ with defining relations
\begin{align}\label{eq:rels-g(4,3)-i}
\begin{aligned}
x_{13}&=0; &x_{14}&=0; & x_{34}&=0; & [x_{124},&x_2]_c=0;\\
&&  x_{112}&=0; & x_{442}&=0; & [[x_{32},&x_{324}]_c,x_2]_c=0;\\
[x_{(13)},x_2]_c&=0; & x_2^2&=0; & x_3^2&=0; & x_{\alpha}^{N_\alpha}&=0, \ \alpha\in\Oc_+^{\bq}.
\end{aligned}
\end{align}
Here  $\Oc_+^{\bq}=\{  1, 123, 23, 1^22^23^24, 12^23^24, 2^23^24, 12^234, 1234, 234, 4 \}$
and the degree of the integral is
\begin{equation*}
\ya= 29\alpha_1 + 54\alpha_2 + 39\alpha_3 + 29\alpha_4.
\end{equation*}

\subsubsection{The Dynkin diagram \emph{(\ref{eq:dynkin-g(4,3)}
		j)}}\label{subsubsec:g(4,3)-j}

\

The Nichols algebra $\toba_{\bq}$ is generated by $(x_i)_{i\in \I_4}$ with defining relations
\begin{align}\label{eq:rels-g(4,3)-j}
\begin{aligned}
x_{13}&=0; &x_{14}&=0; & x_{34}&=0; & [[x_{124},&x_2]_c,x_2]_c=0; \\
&& x_{112}&=0; & x_{221}&=0; & [[x_{324},&x_2]_c,x_2]_c=0;\\
x_4^2&=0; & x_{223}&=0; & x_{442}&=0; & x_{\alpha}^{N_\alpha}&=0, \ \alpha\in\Oc_+^{\bq}.
\end{aligned}
\end{align}
Here  $\Oc_+^{\bq}=\{ 1, 12, 2, 1^22^24, 12^234, 12^24, 124, 2^24, 24, 4 \}$and the degree of the integral is
\begin{equation*}
\ya= 29\alpha_1 + 54\alpha_2 + 17\alpha_3 + 29\alpha_4.
\end{equation*}

\subsubsection{The associated Lie algebra} This is of type $C_3\times A_1$.

\subsection{\ Type $\gtt(3, 6)$}\label{subsec:type-g(3,6)}
Here $\theta = 4$, $\zeta \in \G'_3$. Let 
\begin{align*}
A&=\begin{pmatrix}   0 & 1 & 0 & 0 \\ \text{--}1 & 2 & \text{--}1 & 0 \\ 0 & \text{--}2 & 2 & \text{--}2 \\ 0 & 0 & 1 & 0 \end{pmatrix}
\in \kk^{4\times 4},& 
\pa&= (-1,1,-1,-1) \in \G_2^4.
\end{align*}
Let $\g(3, 6) = \g(A, \pa)$, the contragredient Lie
superalgebra corresponding to $(A, \pa)$. 
Then $\sdim \g(3, 6) = 36|40$ \cite{BGL}. 
There are 6 other pairs of matrices and parity vectors for which the associated contragredient Lie superalgebra is isomorphic to $\g(3, 6)$. We describe now its root system $\gtt(3, 6)$,
see \cite{AA-GRS-CLS-NA}.

\subsubsection{Basic datum and root system}
Below, $C_n^{(1)\wedge}$, $C_4$, $F_4$, $A_4$ and $_1 T$ are numbered as in  \eqref{eq:Cn1-indef}, \eqref{eq:dynkin-system-C}, \eqref{eq:dynkin-system-F}, \eqref{eq:dynkin-system-A} and  \eqref{eq:mTn}, respectively.
The basic datum and the bundle of Cartan matrices are described by the following diagram:
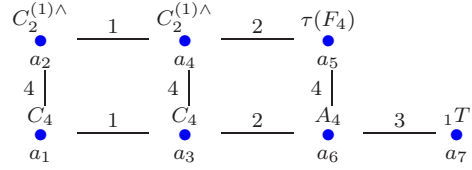
\begin{center}
	\begin{tabular}{c c c c c c c c c c c c}
		& $\overset{C_2^{(1)\wedge}}{\underset{a_2}{\vtxgpd}}$
		& \hspace{-7pt}\raisebox{3pt}{$\overset{1}{\rule{27pt}{0.5pt}}$}\hspace{-7pt}
		& $\overset{C_2^{(1)\wedge}}{\underset{a_4}{\vtxgpd}}$
		& \hspace{-7pt}\raisebox{3pt}{$\overset{2}{\rule{27pt}{0.5pt}}$}\hspace{-7pt}
		& $\overset{\tau(F_4)}{\underset{a_5}{\vtxgpd}}$
		& & & & & &
		\\
		& {\scriptsize 4} \vline\hspace{5pt}
		&
		& {\scriptsize 4} \vline\hspace{5pt}
		&
		& {\scriptsize 4} \vline\hspace{5pt}
		& & & & & &
		\\
		& $\overset{C_4}{\underset{a_1}{\vtxgpd}}$
		& \hspace{-7pt}\raisebox{3pt}{$\overset{1}{\rule{27pt}{0.5pt}}$}\hspace{-7pt}
		& $\overset{C_4}{\underset{a_3}{\vtxgpd}}$
		& \hspace{-7pt}\raisebox{3pt}{$\overset{2}{\rule{27pt}{0.5pt}}$}\hspace{-7pt}
		& $\overset{A_4}{\underset{a_6}{\vtxgpd}}$
		& \hspace{-7pt}\raisebox{3pt}{$\overset{3}{\rule{27pt}{0.5pt}}$}\hspace{-7pt}
		& $\overset{{}_1T}{\underset{a_7}{\vtxgpd}}$
		& & & &
	\end{tabular}
\end{center}
Using the notation \eqref{eq:notation-root-exceptional}, the bundle of root sets is the following: { \scriptsize
	\begin{align*}
	\varDelta_{+}^{a_1}= & \{ 1, 12, 2, 123, 23, 3, 1^22^23^24, 12^23^24, 1^22^33^44^2, 1^22^33^34^2, 123^24, 1^22^23^34^2, 1234, 1^32^43^64^4, \\
	& 1^22^33^54^3, 1^22^33^44^3, 1^22^23^44^3, 12^23^44^2, 12^23^34^2, 1^22^33^64^4, 23^24, 123^34^2, 3^24, 1^22^33^54^4, \\
	& 12^23^44^3, 12^23^24^2, 123^24^2, 23^34^2, 234, 23^24^2, 34, 4 \}, \\
	\varDelta_{+}^{a_2}= & \{ 1, 12, 2, 123, 12^23^2, 123^2, 23, 23^2, 3, 1^22^33^54, 1^22^33^44, 1^22^23^44, 12^23^44, 1^22^33^34, 1^22^23^34, \\
	& 1^32^43^64^2, 1^22^23^24, 12^23^24, 12^23^24, 1^22^33^64^2, 1^22^33^54^2, 123^34, 23^34, 1^22^33^44^2, 123^24, \\
	& 12^23^44^2, 1234, 23^24, 234, 3^24, 34, 4 \}, \\
	\varDelta_{+}^{a_3}= & \{ 1, 12, 2, 123, 23, 3, 12^23^24, 123^24, 1234, 12^33^44^2, 12^33^34^2, 2^23^24, 12^23^44^2, 12^23^34^2, \\
	& 12^33^54^3, 123^34^2, 23^24, 3^24, 12^43^64^4, 12^33^64^4, 12^33^44^3, 2^23^34^2, 12^23^44^3, 23^34^2, 12^33^54^4, \\
	& 12^23^24^2, 123^24^2, 2^23^44^3, 234, 23^24^2, 34, 4 \}, \\
	\varDelta_{+}^{a_4}= & \{ 1, 12, 2, 123, 12^23^2, 123^2, 23, 23^2, 3, 12^33^54, 12^33^44, 12^23^44, 2^23^44, 12^33^34, 12^23^24, \\
	& 12^43^64^2, 12^23^24, 2^23^34, 2^23^24, 12^33^64^2, 12^33^54^2, 123^34, 23^34, 12^33^44^2, 123^24, 12^23^44^2, \\
	& 1234, 23^24, 234, 3^24, 34, 4 \}, \\
	\varDelta_{+}^{a_5}= & \{ 1, 12, 2, 123, 12^23^2, 123^2, 23, 23^2, 3, 1^22^43^54, 1^22^43^44, 1^22^33^44, 12^33^44, 1^22^33^34, 12^33^34, \\
	& 1^32^53^64^2, 1^22^23^34, 1^22^23^24, 1^22^53^64^2, 12^23^24, 1^22^43^54^2, 2^23^34, 12^23^24, 2^23^24, 1^22^33^44^2, \\
	& 12^33^44^2, 123^24, 1234, 23^24, 234, 34, 4 \}, \\
	\varDelta_{+}^{a_6}= & \{ 1, 12, 2, 123, 23, 3, 1^22^23^24, 12^23^24, 1^22^33^44^2, 1^22^33^34^2, 123^24, 1^22^23^34^2, 1234, 1^32^53^64^4, \\
	& 1^22^43^54^3, 1^22^43^44^3, 1^22^33^44^3, 12^33^44^2, 12^33^34^2, 1^22^53^64^4, 2^23^24, 12^23^34^2, 23^24, \\
	& 1^22^43^54^4, 12^33^44^3, 12^23^24^2, 123^24^2, 2^23^34^2, 234, 23^24^2, 34, 4 \}, \\
	\varDelta_{+}^{a_7}= & \{ 1, 12, 2, 123, 23, 3, 1^22^33^24, 1^22^23^24, 1^22^234, 12^33^24, 12^23^24, 1^22^43^34^2, 2^23^24, 123^24, \\
	& 1^22^33^34^2, 23^24, 1^32^53^44^3, 12^33^34^2, 1^22^53^44^3, 1^22^43^44^3, 12^234, 1^22^33^24^2, 2^234, 12^33^24^2, \\
	& 1^22^43^34^3, 1234, 12^23^24^2, 124, 234, 24, 34, 4 \}.
	\end{align*}
	
}%

\subsubsection{Weyl groupoid}
\label{subsubsec:type-g36-Weyl}
The isotropy group  at $a_7 \in \cX$ is
\begin{align*}
\cW(a_7)= \langle \varsigma_1^{a_7}, \varsigma_2^{a_7}, \varsigma_4^{a_7}, \varsigma_3^{a_7} \varsigma_4\varsigma_2 \varsigma_3 \varsigma_2 \varsigma_4\varsigma_3 \rangle \simeq W(C_4).
\end{align*}

\subsubsection{Incarnation}
To describe it, we need the matrices $(\bq^{(i)})_{i\in\I_{7}}$, from left to right and  from up to down:
\begin{align}\label{eq:dynkin-g(3,6)}
\begin{aligned}
&\xymatrix{ \overset{-1}{\underset{\ }{\circ}}\ar  @{-}[r]^{\ztu}  &
	\overset{\zeta}{\underset{\ }{\circ}} \ar  @{-}[r]^{\ztu}  & \overset{\zeta}{\underset{\
		}{\circ}}\ar  @{-}[r]^{\zeta}  & \overset{-1}{\underset{\ }{\circ}}}
& &\xymatrix{ \overset{-1}{\underset{\ }{\circ}}\ar  @{-}[r]^{\ztu}  &
	\overset{\zeta}{\underset{\ }{\circ}} \ar  @{-}[r]^{\ztu}  & \overset{-\ztu}{\underset{\
		}{\circ}}
	\ar  @{-}[r]^{\ztu}  & \overset{-1}{\underset{\ }{\circ}}}
\\
& \xymatrix{ \overset{-1}{\underset{\ }{\circ}}\ar  @{-}[r]^{\zeta}  &
	\overset{-1}{\underset{\ }{\circ}} \ar  @{-}[r]^{\ztu}  & \overset{\zeta}{\underset{\
		}{\circ}}
	\ar  @{-}[r]^{\zeta}  & \overset{-1}{\underset{\ }{\circ}}}
&& \xymatrix{ \overset{-1}{\underset{\ }{\circ}}\ar  @{-}[r]^{\zeta}  &
	\overset{-1}{\underset{\ }{\circ}} \ar  @{-}[r]^{\ztu}  & \overset{-\ztu}{\underset{\
		}{\circ}}
	\ar  @{-}[r]^{\ztu}  & \overset{-1}{\underset{\ }{\circ}}}
\\ &\xymatrix{ \overset{\zeta}{\underset{\ }{\circ}}\ar  @{-}[r]^{\ztu}  &
	\overset{-1}{\underset{\ }{\circ}} \ar  @{-}[r]^{\zeta}  & \overset{\zeta}{\underset{\
		}{\circ}}
	\ar  @{-}[r]^{\ztu}  & \overset{-1}{\underset{\ }{\circ}}}
&&\xymatrix{ \overset{\zeta}{\underset{\ }{\circ}}\ar  @{-}[r]^{\ztu}  &
	\overset{-1}{\underset{\ }{\circ}} \ar  @{-}[r]^{\zeta}  & \overset{-1}{\underset{\
		}{\circ}}
	\ar  @{-}[r]^{\zeta}  & \overset{-1}{\underset{\ }{\circ}}}
\\ &\xymatrix@R-6pt{  &    \overset{-1}{\circ} \ar  @{-}[d]_{\ztu}\ar  @{-}[dr]^{\ztu} &
	\\
	\overset{\zeta}{\underset{\ }{\circ}} \ar  @{-}[r]^{\ztu}  & \overset{\zeta}{\underset{\
		}{\circ}} \ar  @{-}[r]^{\ztu}  & \overset{\zeta}{\underset{\ }{\circ}}}
&&
\end{aligned}
\end{align}
Now, this is the incarnation: $a_i\mapsto \bq^{(i)}$, $i\in\I_{10}$.

\subsubsection{PBW-basis and dimension} \label{subsubsec:type-g36-PBW}
Notice that the roots in each $\varDelta_{+}^{a_i}$, $i\in\I_{7}$, are ordered from left to right, justifying the notation $\beta_1, \dots, \beta_{32}$.

The root vectors $x_{\beta_k}$ are described as in Remark \ref{rem:lyndon-word}.
Thus
\begin{align*}
\left\{ x_{\beta_{32}}^{n_{32}} \dots x_{\beta_2}^{n_{2}}  x_{\beta_1}^{n_{1}} \, | \, 0\le n_{k}<N_{\beta_k} \right\}.
\end{align*}
is a PBW-basis of $\toba_{\bq}$. Hence $\dim \toba_{\bq}= 2^{16}3^{12}6^4=2^{20}3^{16}$.

\subsubsection{The Dynkin diagram \emph{(\ref{eq:dynkin-g(3,6)}
		a)}}\label{subsubsec:g(3,6)-a}

\

The Nichols algebra $\toba_{\bq}$ is generated by $(x_i)_{i\in \I_4}$ with defining relations
\begin{align}\label{eq:rels-g(3,6)-a}
\begin{aligned}
x_{13}&=0; &x_{14}&=0; & x_{24}&=0; & [[x_{(24)},&x_{3}]_c,x_{3}]_c=0;\\
&& x_{221}&=0; & x_{223}&=0; & [x_{334},&x_{34}]_c=0;\\
x_{332}&=0; &x_{1}^2&=0; & x_{4}^2&=0; & x_{\alpha}^{N_\alpha}&=0, \ \alpha\in\Oc_+^{\bq}.
\end{aligned}
\end{align}
Here
Here {\scriptsize$\Oc_+^{\bq}=\{ 2, 23, 3, 12^23^24, 123^24, 1234, 12^23^34^2,
	12^33^54^3, 23^24, 12^33^44^3, 2^23^34^2, \\ 12^23^44^3, 23^34^2, 234, 23^24^2, 34 \}$}
and the degree of the integral is
\begin{equation*}
\ya= 29\alpha_1 + 84\alpha_2 + 135\alpha_3 + 91\alpha_4.
\end{equation*}

\subsubsection{The Dynkin diagram \emph{(\ref{eq:dynkin-g(3,6)}
		b)}}\label{subsubsec:g(3,6)-b}

\

The Nichols algebra $\toba_{\bq}$ is generated by $(x_i)_{i\in \I_4}$ with defining relations
\begin{align}\label{eq:rels-g(3,6)-b}
\begin{aligned}
& \begin{aligned}
x_{14}&=0; & x_{221}&=0; & x_{3332}&=0; & &x_{1}^2=0; & & x_{4}^2=0;
\\
x_{24}&=0; & x_{223}&=0; & x_{3334}&=0; & & x_{\alpha}^{N_\alpha}=0, & & \alpha\in\Oc_+^{\bq};
\end{aligned}
\\
& x_{13}=0; \quad [x_2,x_{334}]_c +q_{34}[x_{(24)},x_3]_c +(\zeta^2-\zeta)q_{23} x_3x_{(24)}=0.
\end{aligned}
\end{align}
Here  {\scriptsize$\Oc_+^{\bq}=\{ 2, 123, 23, 23^2, 3, 12^33^44, 12^23^44,
	12^23^34, 12^23^24, 2^23^34, 12^33^54^2, 23^34, 123^24, 23^24,\linebreak 234, 34 \}$}
and the degree of the integral is
\begin{equation*}
\ya= 29\alpha_1 + 84\alpha_2 + 135\alpha_3 + 46\alpha_4.
\end{equation*}

\subsubsection{The Dynkin diagram \emph{(\ref{eq:dynkin-g(3,6)}
		c)}}\label{subsubsec:g(3,6)-c}

\

The Nichols algebra $\toba_{\bq}$ is generated by $(x_i)_{i\in \I_4}$ with defining relations
\begin{align}\label{eq:rels-g(3,6)-c}
\begin{aligned}
x_{13}&=0; &x_{14}&=0; & x_{24}&=0; & [[x_{(24)},&x_{3}]_c,x_{3}]_c=0;\\
&& x_{332}&=0; & [x_{(13)},x_{2}]_c&=0; & [x_{334},&x_{34}]_c=0;\\
x_1^2&=0; & x_{2}^2&=0; & x_{4}^2&=0; & x_{\alpha}^{N_\alpha}&=0, \ \alpha\in\Oc_+^{\bq}.
\end{aligned}
\end{align}
Here {\scriptsize$\Oc_+^{\bq}=\{ 12, 123, 3, 12^23^24, 123^24, 1^22^23^34^2, 1234,
	1^22^33^54^3, 1^22^33^44^3, 12^23^34^2, 23^24, \\ 123^34^2, 12^23^44^3, 123^24^2, 234, 34 \}$}
and the degree of the integral is
\begin{equation*}
\ya= 57\alpha_1 + 84\alpha_2 + 135\alpha_3 + 91\alpha_4.
\end{equation*}

\subsubsection{The Dynkin diagram \emph{(\ref{eq:dynkin-g(3,6)}
		d)}}\label{subsubsec:g(3,6)-d}

\

The Nichols algebra $\toba_{\bq}$ is generated by $(x_i)_{i\in \I_4}$ with defining relations
\begin{align}\label{eq:rels-g(3,6)-d}
\begin{aligned}
& \begin{aligned}
x_{13}&=0; & x_{3332}&=0; & x_{1}^2&=0; & & [x_{(13)},x_2]_c=0; \quad x_{4}^2=0;
\\
x_{14}&=0; & x_{3334}&=0; & x_{2}^2&=0; & & x_{\alpha}^{N_\alpha} =0, \ \alpha\in\Oc_+^{\bq};
\end{aligned}
\\
& x_{24}=0; \quad [x_2,x_{334}]_c +q_{34}[x_{(24)},x_3]_c +(\zeta^2-\zeta)q_{23}x_3x_{(24)}=0.
\end{aligned}
\end{align}
Here {\scriptsize$\Oc_+^{\bq}=\{ 12, 123, 123^2, 23, 3, 1^22^33^44, 12^23^44,
	1^22^23^34, 12^23^34, 12^23^24, 1^22^33^54^2, 123^34$, \\ $123^24, 1234, 23^24, 34 \}$}
and the degree of the integral is
\begin{equation*}
\ya= 57\alpha_1 + 84\alpha_2 + 135\alpha_3 + 46\alpha_4.
\end{equation*}

\subsubsection{The Dynkin diagram \emph{(\ref{eq:dynkin-g(3,6)}
		e)}}\label{subsubsec:g(3,6)-e}

\

The Nichols algebra $\toba_{\bq}$ is generated by $(x_i)_{i\in \I_4}$ with defining relations
\begin{align}\label{eq:rels-g(3,6)-e}
\begin{aligned}
x_{13}&=0; &x_{14}&=0; & x_{24}&=0; & [[x_{(24)},&x_3]_c,x_3]_c=0;\\
[x_{(13)},&x_2]_c=0; & x_{112}&=0; & x_{334}&=0; & [x_{3321},&x_{32}]_c=0;\\
[x_{332},&x_{32}]_c=0; & x_{2}^2&=0; & x_{4}^2&=0; & x_{\alpha}^{N_\alpha}&=0, \ \alpha\in\Oc_+^{\bq}.
\end{aligned}
\end{align}
Here
{\scriptsize $\Oc_+^{\bq}=\{ 1, 123, 12^23^2, 23, 3, 1^22^33^44, 12^33^44,
	1^22^33^34, 12^33^34, 12^23^34, 1^22^43^54^2, 12^23^24$, 
	\newline $123^24, 1234, 23^24, 234 \}$}
and the degree of the integral is
\begin{equation*}
\ya= 57\alpha_1 + 110\alpha_2 + 135\alpha_3 + 46\alpha_4.
\end{equation*}

\subsubsection{The Dynkin diagram \emph{(\ref{eq:dynkin-g(3,6)}
		f)}}\label{subsubsec:g(3,6)-f}

\

The Nichols algebra $\toba_{\bq}$ is generated by $(x_i)_{i\in \I_4}$ with defining relations
\begin{align}\label{eq:rels-g(3,6)-f}
\begin{aligned}
x_{13}&=0; & x_{14}&=0; & x_{24}&=0; & [[x_{23},&x_{(24)}]_c,x_{3}]_c=0;\\
& & x_{112}&=0; & x_{2}^2&=0; & [[x_{43},&x_{432}]_c,x_{3}]_c=0;\\
[x_{(13)},&x_{2}]_c=0; & x_{3}^2&=0; & x_{4}^2&=0; & x_{\alpha}^{N_\alpha}&=0, \ \alpha\in\Oc_+^{\bq}.
\end{aligned}
\end{align}
Here {\scriptsize$\Oc_+^{\bq}=\{  1, 123, 23, 12^23^24, 1^22^33^34^2, 123^24, 1234,
	1^22^43^54^3, 1^22^33^44^3, 12^33^34^2, \\ 12^23^34^2, 23^24, 12^33^44^3, 12^23^24^2, 234, 34 \}$}  
and the degree of the integral is
\begin{equation*}
\ya= 57\alpha_1 + 110\alpha_2 + 135\alpha_3 + 91\alpha_4.
\end{equation*}

\subsubsection{The Dynkin diagram \emph{(\ref{eq:dynkin-g(3,6)}
		g)}}\label{subsubsec:g(3,6)-g}

\

The Nichols algebra $\toba_{\bq}$ is generated by $(x_i)_{i\in \I_4}$ with defining relations
\begin{align}\label{eq:rels-g(3,6)-g}
\begin{aligned}
& \begin{aligned}
x_{221}&=0; & x_{224}&=0; & x_{332}&=0; & & x_{13}=0; & & x_{14}=0; 
\\
x_{223}&=0; & x_{112}&=0; & x_{334}&=0;  & &x_{\alpha}^{N_\alpha}=0, && \alpha\in\Oc_+^{\bq};
\end{aligned}
\\
& x_4^2=0; \quad x_{(24)}=\zeta q_{34}[x_{24},x_3]_c+(1-\ztu)q_{23}x_3x_{24}.
\end{aligned}
\end{align}
Here
{\scriptsize$\Oc_+^{\bq}=\{  1, 12, 2, 123, 23, 3, 12^23^24,
	1^22^43^34^2, 1^22^33^34^2, 12^33^34^2, 12^234, 1^22^33^24^2, 12^33^24^2,\newline 1234,  12^23^24^2, 234 \}$}
and the degree of the integral is
\begin{equation*}
\ya= 57\alpha_1 + 110\alpha_2 + 68\alpha_3 + 91\alpha_4.
\end{equation*}

\subsubsection{The associated Lie algebra} This is of type $C_4$.

\subsection{\ Type $\gtt(2, 6)$}\label{subsec:type-g(2,6)}
Here $\theta = 5$, $\zeta \in \G'_3$. Let 
\begin{align*}
A&=\begin{pmatrix} 2 & \text{--}1 & 0 & 0 & 0 \\ \text{--}1 & 2 & 0 & \text{--}1 & 0 \\ 0 & 0 & 2 & \text{--}1 & 0
\\ 0 & 1 & 1 & 0 & \text{--}1 \\ 0 & 0 & 0 & 1 & 0 \end{pmatrix}
\in \kk^{5\times 5}; & \pa &= (1,1, 1,-1,-1) \in \G_2^5.
\end{align*}
Let $\g(2, 6) = \g(A, \pa)$,
the contragredient Lie superalgebra corresponding to $(A, \pa)$. 
We know \cite{BGL} that $\sdim \g(2, 6) = 36|20$. 
There are 5 other pairs of matrices and parity vectors for which the associated contragredient Lie superalgebra is isomorphic to $\g(2, 6)$.
We describe now the root system $\gtt(2, 6)$ of $\g(2, 6)$, see \cite{AA-GRS-CLS-NA} for details.

\subsubsection{Basic datum and root system}
Below, $A_5$, $_1T_1$ and $D_5$ are numbered as in \eqref{eq:dynkin-system-A}, \eqref{eq:mTn} and  	
\eqref{eq:dynkin-system-D}, respectively.
The basic datum and the bundle of Cartan matrices are described by the following diagram:
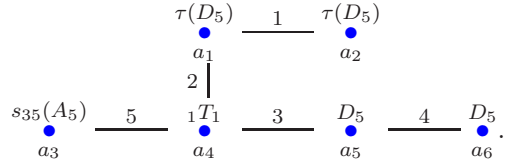
\begin{center}
	\begin{tabular}{c c c c c c c c c c c c}
		& & &  $\overset{\tau(D_5)}{\underset{a_1}{\vtxgpd}}$
		& \hspace{-7pt}\raisebox{3pt}{$\overset{1}{\rule{27pt}{0.5pt}}$}\hspace{-7pt}
		& $\overset{\tau(D_5)}{\underset{a_2}{\vtxgpd}}$
		& & & & & &
		\\
		& & & {\scriptsize 2} \vline\hspace{5pt} & & & & & & & &
		\\
		& $\overset{s_{35}(A_5)}{\underset{a_3}{\vtxgpd}}$
		& \hspace{-7pt}\raisebox{3pt}{$\overset{5}{\rule{27pt}{0.5pt}}$}\hspace{-7pt}
		& $\overset{{}_1T_1}{\underset{a_4}{\vtxgpd}}$
		& \hspace{-7pt}\raisebox{3pt}{$\overset{3}{\rule{27pt}{0.5pt}}$}\hspace{-7pt}
		& $\overset{D_5}{\underset{a_5}{\vtxgpd}}$
		& \hspace{-7pt}\raisebox{3pt}{$\overset{4}{\rule{27pt}{0.5pt}}$}\hspace{-7pt}
		& $\overset{D_5}{\underset{a_6}{\vtxgpd}}$.
		& & & &
	\end{tabular}
\end{center}
Using the notation \eqref{eq:notation-root-exceptional}, the bundle of root sets is the following:
{\scriptsize
	\begin{align*}
	\varDelta_{+}^{a_1}= & \tau(\varDelta_{+}^{a_6}), \qquad \qquad \qquad \varDelta_{+}^{a_2}=\tau(\varDelta_{+}^{a_5}), \\ 
	\varDelta_{+}^{a_3} =& \big\{ 1, 12, 2, 123, 23, 3, 12^23^24, 123^24, 1234, 23^24, 234, 34, 4, 12^23^34^25, \\
	&12^23^24^25, 123^24^25, 23^24^25, 12^23^245,123^245, 23^245, 12345, 2345, 345, 45, 5 \big\}, \\
	\varDelta_{+}^{a_4} =& \big\{ 1, 12, 2, 123, 23, 3, 12^234, 1234, 124, 234, 24, 34, 4, 12^23^24^25, 12^234^25, \\
	&1234^25, 234^25, 12^2345, 12345, 2345, 345, 1245, 245, 45, 5 \big\}, \\
	\varDelta_{+}^{a_5} =& \big\{ 1, 12, 2, 123, 23, 3, 12^23^24, 123^24, 1234, 23^24, 234, 34, 4, 12^23^34^25, \\
	&12^23^345, 12^23^245, 123^245, 23^245, 12345, 2345, 345, 1235, 235, 35, 5\big\}, \\
	\varDelta_{+}^{a_6} =& \big\{ 1, 12, 2, 123, 23, 3, 1234, 234, 34, 4, 12^23^245, 123^245, 12345, 1235, \\
	&12^23^34^25^2, 12^23^345^2, 23^245, 12^23^245^2, 2345, 235, 123^245^2, 23^245^2, 345, 35, 5\big\}.
	\end{align*}
}%

\subsubsection{Weyl groupoid}
\label{subsubsec:type-g26-Weyl}
The isotropy group  at $a_3 \in \cX$ is
\begin{align*}
\cW(a_3)= \langle \varsigma_1^{a_3}, \varsigma_2^{a_3},  \varsigma_5^{a_3} \varsigma_2 \varsigma_3 \varsigma_2 \varsigma_5, \varsigma_3^{a_3},  \varsigma_4^{a_3} \rangle \simeq W(A_5).
\end{align*}

\subsubsection{Incarnation}
To describe it, we need the matrices $(\bq^{(i)})_{i\in\I_{4}}$, from left to right and  from up to down:
\begin{align}\label{eq:dynkin-g(2,6)}
\begin{aligned}
&\xymatrix{\\ \overset{\zeta}{\underset{\ }{\circ}}\ar  @{-}[r]^{\ztu}  &
	\overset{\zeta}{\underset{\ }{\circ}} \ar  @{-}[r]^{\ztu}  & \overset{-1}{\underset{\
		}{\circ}}
	\ar  @{-}[r]^{\ztu}  & \overset{\zeta}{\underset{\ }{\circ}} \ar  @{-}[r]^{\ztu}  &
	\overset{\zeta}{\underset{\ }{\circ}}}
& & \xymatrix@R-8pt{  &    \overset{-1}{\circ} \ar  @{-}[d]_{\zeta}\ar  @{-}[dr]^{\zeta} &
	\\
	\overset{\zeta}{\underset{\ }{\circ}} \ar  @{-}[r]^{\ztu}  & \overset{-1}{\underset{\
		}{\circ}} \ar  @{-}[r]^{\zeta}  & \overset{-1}{\underset{\ }{\circ}}\ar  @{-}[r]^{\ztu}  &
	\overset{\zeta}{\underset{\ }{\circ}}}
\\
& \xymatrix@R-10pt{ & &    \overset{\zeta}{\circ} \ar  @{-}[d]^{\ztu} & \\
	\overset{\zeta}{\underset{\ }{\circ}} \ar  @{-}[r]^{\ztu} & \overset{\zeta}{\underset{\
		}{\circ}} \ar  @{-}[r]^{\ztu}  & \overset{-1}{\underset{\ }{\circ}} \ar  @{-}[r]^{\zeta}
	& \overset{-1}{\underset{\ }{\circ}}}
&& \xymatrix@R-10pt{  &  &  \overset{\zeta}{\circ} \ar  @{-}[d]^{\ztu} & \\
	\overset{\zeta}{\underset{\ }{\circ}} \ar  @{-}[r]^{\ztu} & \overset{\zeta}{\underset{\
		}{\circ}} \ar  @{-}[r]^{\ztu}  & \overset{\zeta}{\underset{\ }{\circ}} \ar  @{-}[r]^{\ztu}
	& \overset{-1}{\underset{\ }{\circ}}}
\end{aligned}
\end{align}
Now, this is the incarnation: 
\begin{align*}
& a_i\mapsto \tau(\bq^{(5-i)}), \ i\in\I_{2}; &
& a_3\mapsto s_{35}(\bq^{(1)}); &
& a_i\mapsto \bq^{(i-2)}, \ i\in\I_{4,6}.
\end{align*}

\subsubsection{PBW-basis and dimension} \label{subsubsec:type-g26-PBW}
Notice that the roots in each $\varDelta_{+}^{a_i}$, $i\in\I_{6}$, are ordered from left to right, justifying the notation $\beta_1, \dots, \beta_{25}$.

The root vectors $x_{\beta_k}$ are described as in Remark \ref{rem:lyndon-word}.
Thus
\begin{align*}
\left\{ x_{\beta_{25}}^{n_{25}} \dots x_{\beta_2}^{n_{2}}  x_{\beta_1}^{n_{1}} \, | \, 0\le n_{k}<N_{\beta_k} \right\}.
\end{align*}
is a PBW-basis of $\toba_{\bq}$. Hence $\dim \toba_{\bq}= 2^{10}3^{15}$.

\subsubsection{The Dynkin diagram \emph{(\ref{eq:dynkin-g(2,6)}
		a)}}\label{subsubsec:g(2,6)-a}

\

The Nichols algebra $\toba_{\bq}$ is generated by $(x_i)_{i\in \I_5}$ with defining relations
\begin{align}\label{eq:rels-g(2,6)-a}
\begin{aligned}
x_{13}&=0; & x_{14}&=0; & x_{15}&=0; & [[[x_{(14)},&x_3]_c , x_2]_c , x_3]_c =0;\\
x_{24}&=0; & x_{221}&=0; & x_{112}&=0; & [[[x_{5432},&x_3]_c, x_4]_c , x_3]_c=0;\\
x_{25}&=0; & x_{223}&=0; & x_{443}&=0; & x_{445}&=0; \\
x_{35}&=0; & x_{554}&=0; & x_{3}^2&=0; & x_{\alpha}^{3}&=0, \ \alpha\in\Oc_+^{\bq}.
\end{aligned}
\end{align}
Here {\scriptsize$\Oc_+^{\bq}=\{ 1, 12, 2, 12^23^24, 123^24, 23^24,
	4, 12^23^24^25, 123^24^25, 23^24^25, 12^23^245, 123^245, \newline 23^245, 45,
	5 \}$}
and the degree of the integral is
\begin{equation*}
\ya= 20\alpha_1 + 36\alpha_2 + 48\alpha_3 + 36\alpha_4 + 20\alpha_5.
\end{equation*}

\subsubsection{The Dynkin diagram \emph{(\ref{eq:dynkin-g(2,6)}
		b)}}\label{subsubsec:g(2,6)-b}

\

The Nichols algebra $\toba_{\bq}$ is generated by $(x_i)_{i\in \I_5}$ with defining relations
\begin{align}\label{eq:rels-g(2,6)-b}
\begin{aligned}
& \begin{aligned}
& [x_{125},x_2]_c=0; & x_{13}&=0; & x_{112}&=0; && x_{24}=0; && x_{45}=0;
\\
& [x_{(24)},x_3]_c=0; & x_{14}&=0; & x_{443}&=0; && x_{2}^2=0; && x_{3}^2=0;
\\
& [x_{(13)},x_2]_c=0; & x_{15}&=0; & [x_{435},& x_3]_c=0; & &x_{\alpha}^{3}=0, && \alpha\in\Oc_+^{\bq};
\end{aligned}
\\
& x_{5}^2=0; \quad x_{235}=q_{35}\ztu [x_{25},x_3]_c +q_{23}(1-\zeta)x_3x_{25}.
\end{aligned}
\end{align}
Here {\scriptsize$\Oc_+^{\bq}=\{ 1, 123, 23, 12^234, 124, 24,
	34, 12^23^24^25, 1234^25, 234^25, 12^2345, 345, 1245, 245,
	5 \}$}
and the degree of the integral is
\begin{equation*}
\ya= 20\alpha_1 + 36\alpha_2 + 36\alpha_3 + 20\alpha_4 + 26\alpha_5.
\end{equation*}

\subsubsection{The Dynkin diagram \emph{(\ref{eq:dynkin-g(2,6)}
		c)}}\label{subsubsec:g(2,6)-c}

\

The Nichols algebra $\toba_{\bq}$ is generated by $(x_i)_{i\in \I_5}$ with defining relations
\begin{align}\label{eq:rels-g(2,6)-c}
\begin{aligned}
x_{13}&=0; & x_{223}&=0; & x_{24}&=0; & x_{25}&=0; & & [x_{435},x_3]_c=0; \\
x_{14}&=0; & x_{112}&=0; & x_{221}&=0; & x_{45}&=0; & & [x_{(24)},x_3]_c=0; \\
x_{15}&=0; & x_{553}&=0; & x_{3}^2&=0; & x_{4}^2&=0; & & x_{\alpha}^{3}=0, \ \alpha\in\Oc_+^{\bq}.
\end{aligned}
\end{align}
Here {\scriptsize$\Oc_+^{\bq}=\{ 1, 12, 2, 12^23^24, 123^24, 23^24,
	4, 12^23^34^25, 12^23^345, 12345, 2345, 345, 1235, 235,
	35 \}$} \newline
and the degree of the integral is
\begin{equation*}
\ya= 20\alpha_1 + 36\alpha_2 + 48\alpha_3 + 20\alpha_4 + 26\alpha_5.
\end{equation*}

\subsubsection{The Dynkin diagram \emph{(\ref{eq:dynkin-g(2,6)}
		d)}}\label{subsubsec:g(2,6)-d}

\

The Nichols algebra $\toba_{\bq}$ is generated by $(x_i)_{i\in \I_5}$ with defining relations
\begin{align}\label{eq:rels-g(2,6)-d}
\begin{aligned}
x_{13}&=0; & x_{112}&=0; & x_{24}&=0; & x_{25}&=0; & x_{45}&=0;\\
x_{14}&=0; & x_{332}&=0; & x_{221}&=0; & x_{223}&=0; & x_{334}&=0; \\
x_{15}&=0; & x_{335}&=0; & x_{553}&=0; & x_{4}^2&=0; & x_{\alpha}^{3}&=0, \ \alpha\in\Oc_+^{\bq}.
\end{aligned}
\end{align}
Here {\scriptsize$\Oc_+^{\bq}=\{ 1, 12, 2, 123, 23, 3,
	1234, 234, 34, 4, 12^23^34^25^2, 12^23^345^2, 12^23^245^2, 123^245^2,
	23^245^2 \}$}

and the degree of the integral is
\begin{equation*}
\ya= 20\alpha_1 + 36\alpha_2 + 48\alpha_3 + 30\alpha_4 + 26\alpha_5.
\end{equation*}

\subsubsection{The associated Lie algebra} This is of type $A_5$.

\subsection{\ Type $\El(5;3)$}\label{subsec:type-el(5;3)}
Here $\theta = 5$, $\zeta \in \G'_3$. Let 
\begin{align*}
A&=\begin{pmatrix} 0 & 1 & 0 & 0 & 0 \\ -1 & 2 & -1 & 0 & 0 \\ 0 & -1 & 2 & -1 & -1
\\ 0 & 0 & -1 & 2 & 0 \\ 0 & 0 & 1 & 0 & 0 \end{pmatrix}
\in \kk^{5\times 5}; & \pa &= (-1,1, 1,1,-1) \in \G_2^5.
\end{align*}
Let $\el(5;3) = \g(A, \pa)$,
the contragredient Lie superalgebra corresponding to $(A, \pa)$. 
We know \cite{BGL} that $\sdim \el(5;3) = 39|32$. 
There are 14 other pairs of matrices and parity vectors for which the associated contragredient Lie superalgebra is isomorphic to $\el(5;3)$.
We describe now the root system $\El(5;3)$ of $\el(5;3)$, see \cite{AA-GRS-CLS-NA} for details.

\subsubsection{Basic datum and root system}
Below, $D_5$, $CE_5$, $A_5$, $F_4^{(1)}$, $E_6^{(2)}$, $_1T_1$ and $_2 T$ are numbered as in \eqref{eq:dynkin-system-D}, \eqref{eq:CEn}, \eqref{eq:dynkin-system-A}, \eqref{eq:F4(1)}, \eqref{eq:E6(2)} and \eqref{eq:mTn}, respectively.
The basic datum and the bundle of Cartan matrices are described by the following diagram:

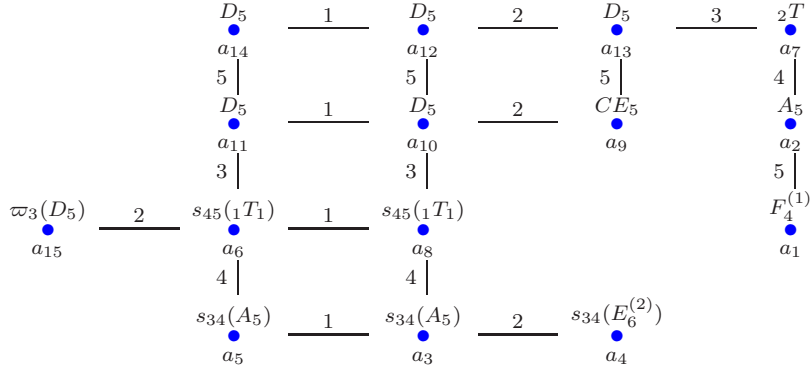
\begin{center}
	\begin{tabular}{c c c c c c c c c c}
		& &  &
		$\overset{D_5}{\underset{a_{14}}{\vtxgpd}}$ & \hspace{-5pt}\raisebox{3pt}{$\overset{1}{\rule{30pt}{0.5pt}}$}\hspace{-5pt}  &
		$\overset{D_5}{\underset{a_{12}}{\vtxgpd}}$ & \hspace{-5pt}\raisebox{3pt}{$\overset{2}{\rule{30pt}{0.5pt}}$}\hspace{-5pt}  &
		$\overset{D_5}{\underset{a_{13}}{\vtxgpd}}$ & \hspace{-5pt}\raisebox{3pt}{$\overset{3}{\rule{30pt}{0.5pt}}$}\hspace{-5pt}  &
		$\overset{{}_2T}{\underset{a_7}{\vtxgpd}}$
		\\
		& & & {\scriptsize 5} \vline\hspace{5pt} & & {\scriptsize 5} \vline\hspace{5pt} & & {\scriptsize 5} \vline\hspace{5pt} & &  {\scriptsize 4} \vline\hspace{5pt}
		\\
		& &
		& $\overset{D_5}{\underset{a_{11}}{\vtxgpd}}$
		& \hspace{-5pt}\raisebox{3pt}{$\overset{1}{\rule{30pt}{0.5pt}}$}\hspace{-5pt}
		& $\overset{D_5}{\underset{a_{10}}{\vtxgpd}}$
		& \hspace{-5pt}\raisebox{3pt}{$\overset{2}{\rule{30pt}{0.5pt}}$}\hspace{-5pt}
		& $\overset{CE_5}{\underset{a_9}{\vtxgpd}}$
		&
		& $\overset{A_5}{\underset{a_2}{\vtxgpd}}$
		\\
		& & & {\scriptsize 3} \vline\hspace{5pt} & & {\scriptsize 3} \vline\hspace{5pt} & & & &  {\scriptsize 5} \vline\hspace{5pt}
		\\
		& $\overset{\varpi_3(D_5)}{\underset{a_{15}}{\vtxgpd}}$
		& \hspace{-5pt}\raisebox{3pt}{$\overset{2}{\rule{30pt}{0.5pt}}$}\hspace{-5pt}
		& $\overset{s_{45}({}_1T_1)}{\underset{a_6}{\vtxgpd}}$
		& \hspace{-5pt}\raisebox{3pt}{$\overset{1}{\rule{30pt}{0.5pt}}$}\hspace{-5pt}
		& $\overset{s_{45}({}_1T_1)}{\underset{a_8}{\vtxgpd}}$
		&
		&
		&
		& $\overset{F_4^{(1)}}{\underset{a_1}{\vtxgpd}}$
		\\
		& & & {\scriptsize 4} \vline\hspace{5pt} & & {\scriptsize 4} \vline\hspace{5pt} & & & &
		\\
		& &
		& $\overset{s_{34}(A_5)}{\underset{a_5}{\vtxgpd}}$
		& \hspace{-5pt}\raisebox{3pt}{$\overset{1}{\rule{30pt}{0.5pt}}$}\hspace{-5pt}
		& $\overset{s_{34}(A_5)}{\underset{a_3}{\vtxgpd}}$
		& \hspace{-5pt}\raisebox{3pt}{$\overset{2}{\rule{30pt}{0.5pt}}$}\hspace{-5pt}
		& $\overset{s_{34}(E_6^{(2)})}{\underset{a_4}{\vtxgpd}}$
		& &
	\end{tabular}
\end{center}

Using the notation \eqref{eq:notation-root-exceptional}, the bundle of root sets is the following: { \scriptsize
	\begin{align*}
	\varDelta_{+}^{a_{1}}= & \{ 1, 12, 2, 123, 23, 3, 1234, 12^23^24^2, 123^24^2, 1234^2, 234, 23^24^2, 234^2, 34, 34^2, 4, 12^23^34^45, \\
	& 12^23^34^35, 12^23^24^35, 123^24^35, 23^24^35, 12^23^24^25, 123^24^25, 23^24^25, 12^23^34^45^2, 1234^25, \\
	& 12345, 234^25, 2345, 34^25, 345, 45, 5 \}, \\
	\varDelta_{+}^{a_{2}}= & \{ 1, 12, 2, 123, 23, 3, 1234, 234, 34, 4, 12^23^24^25, 123^24^25, 1234^25, 12345, 12^23^34^45^2, \\
	& 12^23^34^35^2, 23^24^25, 12^23^24^35^2, 234^25, 12^23^24^25^2, 2345, 12^23^34^45^3, 123^24^35^2, 123^24^25^2,\\
	& 1234^25^2, 23^24^35^2, 23^24^25^2, 234^25^2, 34^25, 345, 34^25^2, 45, 5 \}, \\
	\varDelta_{+}^{a_{3}}= & s_{34}(\{ 1, 12, 2, 123, 23, 3, 1^22^23^24, 12^23^24, 123^24, 1234, 23^24, 234, 34, 4, 1^22^33^44^35, 1^22^33^44^25, \\
	& 1^22^33^34^25, 1^22^23^34^25, 12^23^34^25, 1^22^23^24^25, 12^23^24^25, 123^24^25, 23^24^25, 1^22^33^44^35^2, \\
	& 1^22^23^245, 12^23^245, 123^245, 12345, 23^245, 2345, 345, 45, 5 \}), \\
	\varDelta_{+}^{a_{4}}= & s_{34}(\{ 1, 12, 2, 123, 23, 3, 12^23^24, 123^24, 1234, 23^24, 234, 3^24, 34, 4, 12^23^44^35, 12^23^44^25,  \\
	& 12^23^34^25, 123^34^25, 23^34^25, 12^23^24^25, 123^24^25, 23^24^25, 3^24^25, 12^23^44^35^2, 12^23^245, 123^245,   \\
	& 12345, 23^245, 2345, 3^245, 345, 45, 5 \}), \\
	\varDelta_{+}^{a_{5}}= & s_{34}(\{ 1, 12, 2, 123, 23, 3, 12^23^24, 123^24, 1234, 2^23^24, 23^24, 234, 34, 4, 12^33^44^35, 12^33^44^25, \\
	& 12^33^34^25, 12^23^34^25, 2^23^34^25, 12^23^24^25, 2^23^24^25, 123^24^25, 23^24^25, 12^33^44^35^2, 12^23^245, \\
	& 123^245, 12345, 2^23^245, 23^245, 2345, 345, 45, 5 \}), \\
	\varDelta_{+}^{a_{6}}= & s_{45}(\{ 1, 12, 2, 123, 23, 3, 12^234, 1234, 124, 2^234, 234, 24, 34, 4, 12^33^24^35, 12^33^24^25, 12^23^24^25, \\
	& 2^23^24^25, 12^334^25, 12^234^25, 2^234^25, 1234^25, 234^25, 12^33^24^35^2, 12^2345, 12345, 1245, 2^2345, \\
	& 2345, 245, 345, 45, 5 \}), \\
	\varDelta_{+}^{a_{7}}= & s_{45}(\{ 1, 12, 2, 123, 23, 3, 1234, 234, 34, 4, 12^23^245, 123^245, 12345, 1235, 12^23^34^25^2, 12^23^345^2, \\
	& 23^245, 12^23^24^25^2, 12^23^245^2, 2345, 235, 12^23^34^25^3, 123^24^25^2, 123^245^2, 12345^2, 23^24^25^2, \\
	& 23^245^2, 2345^2, 345, 45, 345^2, 35, 5 \}), \\
	\varDelta_{+}^{a_{8}}= & \{ 1, 12, 2, 123, 23, 3, 1^22^234, 12^234, 1234, 124, 234, 24, 34, 4, 1^22^33^24^35, 1^22^33^24^25, \\
	& 1^22^23^24^25, 12^23^24^25, 1^22^334^25, 1^22^234^25, 12^234^25, 1234^25, 234^25, 1^22^33^24^35^2, \\
	& 1^22^2345, 12^2345, 12345, 1245, 2345, 245, 345, 45, 5 \}, \\
	\varDelta_{+}^{a_{9}}= & \{ 1, 12, 2, 123, 23, 3, 12^23^24, 123^24, 1234, 23^24, 234, 3^24, 34, 4, 12^23^44^25, 12^23^34^25, \\
	& 123^34^25, 23^34^25, 12^23^345, 123^345, 23^345, 12^23^245, 123^245, 12^23^44^25^2, 12345, 1235, \\
	& 23^245, 2345, 235, 3^245, 345, 35, 5 \}, \\
	\varDelta_{+}^{a_{10}}= & \{ 1, 12, 2, 123, 23, 3, 1^22^23^24, 12^23^24, 123^24, 1234, 23^24, 234, 34, 4, 1^22^33^44^25, 1^22^33^34^25, \\
	& 1^22^23^34^25, 12^23^34^25, 1^22^33^345, 1^22^23^345, 12^23^345, 1^22^23^245, 12^23^245, 1^22^33^44^25^2, \\
	& 123^245, 12345, 1235, 23^245, 2345, 235, 345, 35, 5 \}, \\
	\varDelta_{+}^{a_{11}}= & \{ 1, 12, 2, 123, 23, 3, 12^23^24, 123^24, 1234, 2^23^24, 23^24, 234, 34, 4, 12^33^44^25, 12^33^34^25, \\
	& 12^23^34^25, 2^23^34^25, 12^33^345, 12^23^345, 2^23^345, 12^23^245, 2^23^245, 12^33^44^25^2, 123^245, \\
	& 12345, 1235, 23^245, 2345, 235, 345, 35, 5 \}, \\
	\varDelta_{+}^{a_{12}}= & \{ 1, 12, 2, 123, 23, 3, 1234, 234, 34, 4, 1^22^23^245, 12^23^245, 1^22^33^44^25^2, 1^22^33^34^25^2, \\
	& 1^22^33^345^2, 123^245, 1^22^23^34^25^2, 1^22^23^345^2, 12345, 1^22^23^245^2, 1235, 1^22^33^44^25^3, \\
	& 12^23^34^25^2, 12^23^345^2, 12^23^245^2, 123^245^2, 23^245, 2345, 345, 23^245^2, 235, 35, 5 \}, \\
	\varDelta_{+}^{a_{13}}= & \{ 1, 12, 2, 123, 23, 3, 1234, 234, 34, 4, 12^23^245, 123^245, 12345, 1235, 12^23^44^25^2, 12^23^34^25^2, \\
	& 12^23^345^2, 23^245, 12^23^245^2, 2345, 235, 12^23^44^25^3, 123^34^25^2, 123^345^2, 123^245^2, 23^34^25^2, \\
	& 23^345^2, 23^245^2, 3^245, 345, 3^245^2, 35, 5 \}, \\
	\varDelta_{+}^{a_{14}}= & \{ 1, 12, 2, 123, 23, 3, 1234, 234, 34, 4, 12^23^245, 123^245, 12345, 1235, 12^33^44^25^2, 12^33^34^25^2, \\
	& 12^33^345^2, 2^23^245, 12^23^34^25^2, 12^23^345^2, 23^245, 12^33^44^25^3, 12^23^245^2, 123^245^2, 2^23^34^25^2, \\
	& 2345, 345, 2^23^345^2, 2^23^245^2, 23^245^2, 235, 35, 5 \}, \\
	\varDelta_{+}^{a_{15}}= & \varpi_3(\{ 1, 12, 2, 123, 23, 3, 1^22^234, 12^234, 1234, 124, 2^234, 234, 24, 4, 1^22^43^24^35, 1^22^43^24^25, \\
	& 1^22^33^24^25, 12^33^24^25, 1^22^334^25, 12^334^25, 1^22^234^25, 12^234^25, 2^234^25, 1^22^43^24^35^2, \\
	& 1^22^2345, 12^2345, 12345, 1245, 2^2345, 2345, 245, 45, 5 \}).
	\end{align*}
	
}%

\subsubsection{Weyl groupoid}
\label{subsubsec:type-el53-Weyl}
The isotropy group  at $a_{12} \in \cX$ is
\begin{align*}
\cW(a_{12}) &= \langle \varsigma_1^{a_{12}}, \varsigma_2^{a_{12}}, \varsigma_3^{a_{12}}, \varsigma_4^{a_{12}}, \varsigma_5^{a_{12}} \varsigma_4 \varsigma_3\varsigma_2  \varsigma_5 \varsigma_3 \varsigma_4 \varsigma_2\varsigma_1 \varsigma_2\varsigma_4 \varsigma_3\varsigma_5 \varsigma_2\varsigma_3 \varsigma_4\varsigma_5 \rangle \\ 
& \simeq W(B_4)  \times \Z/2.
\end{align*}

\subsubsection{Incarnation}
We set the matrices $(\bq^{(i)})_{i\in\I_{15}}$, from left to right and  from up to down:
\begin{align}\label{eq:dynkin-el(5;3)}
&\xymatrix{\overset{\zeta}{\underset{\ }{\circ}}\ar  @{-}[r]^{\ztu}  &
	\overset{\zeta}{\underset{\ }{\circ}} \ar  @{-}[r]^{\ztu}  & \overset{\zeta}{\underset{\
		}{\circ}}
	\ar  @{-}[r]^{\ztu}  & \overset{\ztu}{\underset{\ }{\circ}} \ar  @{-}[r]^{\zeta}  &
	\overset{-1}{\underset{\ }{\circ}}}
& &
\xymatrix{\overset{\zeta}{\underset{\ }{\circ}}\ar  @{-}[r]^{\ztu}  &
	\overset{\zeta}{\underset{\ }{\circ}} \ar  @{-}[r]^{\ztu}  & \overset{\zeta}{\underset{\
		}{\circ}}
	\ar  @{-}[r]^{\ztu}  & \overset{-1}{\underset{\ }{\circ}} \ar  @{-}[r]^{\ztu}  &
	\overset{-1}{\underset{\ }{\circ}}}
\\ \notag
&\xymatrix{\overset{-1}{\underset{\ }{\circ}} \ar  @{-}[r]^{\zeta} & \overset{-1}{\underset{\ }{\circ}} \ar  @{-}[r]^{\ztu}  & \overset{-1}{\underset{\
		}{\circ}} \ar  @{-}[r]^{\ztu}  & \overset{\zeta}{\underset{\ }{\circ}} \ar  @{-}[r]^{\ztu}
	& \overset{\zeta}{\underset{\ }{\circ}}}
& & 
\xymatrix{\overset{\zeta}{\underset{\ }{\circ}} \ar  @{-}[r]^{\ztu} &
	\overset{-1}{\underset{\ }{\circ}} \ar  @{-}[r]^{\zeta}  & \overset{\ztu}{\underset{\
		}{\circ}} \ar  @{-}[r]^{\ztu}  & \overset{\zeta}{\underset{\ }{\circ}} \ar  @{-}[r]^{\ztu}
	& \overset{\zeta}{\underset{\ }{\circ}}}
\\ \notag
&\xymatrix@C-4pt{ &&&& \\ \overset{-1}{\underset{\ }{\circ}} \ar  @{-}[r]^{\ztu} &
	\overset{\zeta}{\underset{\ }{\circ}} \ar  @{-}[r]^{\ztu}  & \overset{-1}{\underset{\
		}{\circ}} \ar  @{-}[r]^{\ztu}  & \overset{\zeta}{\underset{\ }{\circ}} \ar  @{-}[r]^{\ztu}
	& \overset{\zeta}{\underset{\ }{\circ}}}
& &
\xymatrix@R-8pt{  &  &  \overset{-1}{\circ} \ar  @{-}[dl]_{\zeta}\ar  @{-}[d]^{\zeta} &
	\\ \overset{-1}{\underset{\ }{\circ}} \ar  @{-}[r]^{\ztu}  & \overset{-1}{\underset{\
		}{\circ}} \ar  @{-}[r]^{\zeta}  & \overset{-1}{\underset{\ }{\circ}} \ar  @{-}[r]^{\ztu}
	& \overset{\zeta}{\underset{\ }{\circ}}}
\\ \notag&
\xymatrix@R-8pt{  &  & & \overset{\ztu}{\circ} \ar  @{-}[d]^{\zeta}\ar  @{-}[dl]_{\zeta}
	\\
	\overset{\zeta}{\underset{\ }{\circ}} \ar  @{-}[r]^{\ztu}  & \overset{\zeta}{\underset{\
		}{\circ}} \ar  @{-}[r]^{\ztu}  & \overset{-1}{\underset{\ }{\circ}} \ar  @{-}[r]^{\zeta}
	& \overset{-1}{\underset{\ }{\circ}}}
& &
\xymatrix@R-8pt{  &  &  \overset{-1}{\circ} \ar  @{-}[dl]_{\zeta}\ar  @{-}[d]^{\zeta} &
	\\
	\overset{-1}{\underset{\ }{\circ}} \ar  @{-}[r]^{\zeta}  & \overset{\ztu}{\underset{\
		}{\circ}} \ar  @{-}[r]^{\zeta}  & \overset{-1}{\underset{\ }{\circ}} \ar  @{-}[r]^{\ztu}
	& \overset{\zeta}{\underset{\ }{\circ}}}
\\\notag
&\xymatrix@R-8pt{  &  &  \overset{-1}{\circ} \ar  @{-}[d]^{\zeta} & \\
	\overset{\zeta}{\underset{\ }{\circ}} \ar  @{-}[r]^{\ztu}  & \overset{-1}{\underset{\
		}{\circ}} \ar  @{-}[r]^{\zeta}  & \overset{\ztu}{\underset{\ }{\circ}} \ar
	@{-}[r]^{\ztu}  & \overset{\zeta}{\underset{\ }{\circ}}}
& &\xymatrix@R-8pt{  &  &  \overset{-1}{\circ} \ar  @{-}[d]^{\zeta} & \\
	\overset{-1}{\underset{\ }{\circ}} \ar  @{-}[r]^{\zeta}  & \overset{-1}{\underset{\
		}{\circ}} \ar  @{-}[r]^{\ztu}  & \overset{-1}{\underset{\ }{\circ}} \ar  @{-}[r]^{\ztu}
	& \overset{\zeta}{\underset{\ }{\circ}}}
\\\notag
&\xymatrix@R-8pt{  &  &  \overset{-1}{\circ} \ar  @{-}[d]^{\zeta} & \\
	\overset{-1}{\underset{\ }{\circ}} \ar  @{-}[r]^{\ztu}  & \overset{\zeta}{\underset{\
		}{\circ}} \ar  @{-}[r]^{\ztu}  & \overset{-1}{\underset{\ }{\circ}} \ar  @{-}[r]^{\ztu}
	& \overset{\zeta}{\underset{\ }{\circ}}}
& &\xymatrix@R-8pt{  &  &  \overset{-1}{\circ} \ar  @{-}[d]^{\ztu} & \\
	\overset{-1}{\underset{\ }{\circ}} \ar  @{-}[r]^{\zeta}  & \overset{-1}{\underset{\
		}{\circ}} \ar  @{-}[r]^{\ztu}  & \overset{\zeta}{\underset{\ }{\circ}} \ar  @{-}[r]^{\ztu}
	& \overset{\zeta}{\underset{\ }{\circ}}}
\\\notag
&\xymatrix@R-8pt{  &  &  \overset{-1}{\circ} \ar  @{-}[d]^{\ztu} & \\
	\overset{\zeta}{\underset{\ }{\circ}} \ar  @{-}[r]^{\ztu}  & \overset{-1}{\underset{\
		}{\circ}} \ar  @{-}[r]^{\zeta}  & \overset{-1}{\underset{\ }{\circ}} \ar  @{-}[r]^{\ztu}
	& \overset{\zeta}{\underset{\ }{\circ}}}
& &\xymatrix@R-8pt{  &  &  \overset{-1}{\circ} \ar  @{-}[d]^{\ztu} & \\
	\overset{-1}{\underset{\ }{\circ}} \ar  @{-}[r]^{\ztu}  & \overset{\zeta}{\underset{\
		}{\circ}} \ar  @{-}[r]^{\ztu}  & \overset{\zeta}{\underset{\ }{\circ}} \ar  @{-}[r]^{\ztu}
	& \overset{\zeta}{\underset{\ }{\circ}}}
\\ \notag
&\xymatrix@R-8pt{  &   \overset{\zeta}{\circ} \ar  @{-}[d]^{\ztu} & & \\
	\overset{\ztu}{\underset{\ }{\circ}} \ar  @{-}[r]^{\zeta}  & \overset{-1}{\underset{\
		}{\circ}} \ar  @{-}[r]^{\ztu}  & \overset{\zeta}{\underset{\ }{\circ}} \ar  @{-}[r]^{\ztu}
	& \overset{\zeta}{\underset{\ }{\circ}}}
& &
\end{align}
Now, this is the incarnation: $a_{15}\mapsto \varpi_3(\bq^{(15)})$,
\begin{align*}
& a_i\mapsto s_{34}(\bq^{(5-i)}), \ i\in\I_{3,5}; &
& a_i\mapsto s_{45}(\bq^{(i)}), \ i=6,8; &
& a_i\mapsto \bq^{(i)}, \text{ otherwise}.
\end{align*}

\subsubsection{PBW-basis and dimension} \label{subsubsec:type-el53-PBW}
Notice that the roots in each $\varDelta_{+}^{a_i}$, $i\in\I_{15}$, are ordered from left to right, justifying the notation $\beta_1, \dots, \beta_{33}$.

The root vectors $x_{\beta_k}$ are described as in Remark \ref{rem:lyndon-word}.
Thus
\begin{align*}
\left\{ x_{\beta_{33}}^{n_{33}} \dots x_{\beta_2}^{n_{2}}  x_{\beta_1}^{n_{1}} \, | \, 0\le n_{k}<N_{\beta_k} \right\}.
\end{align*}
is a PBW-basis of $\toba_{\bq}$. Hence $\dim \toba_{\bq}=2^{16}3^{17}$.

\subsubsection{The Dynkin diagram \emph{(\ref{eq:dynkin-el(5;3)}
		a)}}\label{subsubsec:el(5;3)-a}

\

The Nichols algebra $\toba_{\bq}$ is generated by $(x_i)_{i\in \I_5}$ with defining relations
\begin{align}\label{eq:rels-el(5;3)-a}
\begin{aligned}
x_{13}&=0; & x_{14}&=0; & x_{15}&=0; & [[x_{(35)},&x_4]_c,x_4]_c=0;\\
x_{24}&=0; & x_{25}&=0; & x_{35}&=0; & x_{112}&=0;\\
x_{221}&=0; & x_{223}&=0; & x_{332}&=0; & [x_{4432},&x_{43}]_c=0; \\
x_{334}&=0; & x_{445}&=0; & x_{5}^2&=0; & x_{\alpha}^{3}&=0, \ \alpha\in\Oc_+^{\bq}.
\end{aligned}
\end{align}
Here {\scriptsize$\Oc_+^{\bq}=\{ 1, 12, 2, 123, 23, 3,
	1234, 12^23^24^2, 123^24^2, 1234^2, 234, 23^24^2, 234^2, 34,
	34^2, 4,\newline 12^23^34^45^2  \}$}  
and the degree of the integral is
\begin{equation*}
\ya= 24\alpha_1 + 44\alpha_2 + 60\alpha_3 + 72\alpha_4 + 20\alpha_5.
\end{equation*}

\subsubsection{The Dynkin diagram \emph{(\ref{eq:dynkin-el(5;3)}
		b)}}\label{subsubsec:el(5;3)-b}

\

The Nichols algebra $\toba_{\bq}$ is generated by $(x_i)_{i\in \I_5}$ with defining relations
\begin{align}\label{eq:rels-el(5;3)-b}
\begin{aligned}
x_{24}&=0; & x_{13}&=0; & x_{14}&=0; & x_{15}&=0; & & [[x_{54},x_{543}]_c,x_4]_c=0;\\
x_{25}&=0; & x_{112}&=0; & x_{221}&=0; & x_{223}&=0; & &x_{332}=0; \\
x_{35}&=0; & x_{334}&=0; & x_{4}^2&=0; & x_{5}^2&=0; & & x_{\alpha}^{3}=0, \ \alpha\in\Oc_+^{\bq}.
\end{aligned}
\end{align}
Here {\scriptsize$\Oc_+^{\bq}=\{ 1, 12, 2, 123, 23, 3,
	12345, 12^23^34^45^2, 12^23^24^25^2, 2345, 123^24^25^2, 1234^25^2, 23^24^25^2$, \\ $234^25^2,
	345, 34^25^2, 45 \}$}  
and the degree of the integral is
\begin{equation*}
\ya= 24\alpha_1 + 44\alpha_2 + 60\alpha_3 + 72\alpha_4 + 58\alpha_5.
\end{equation*}

\subsubsection{The Dynkin diagram \emph{(\ref{eq:dynkin-el(5;3)}
		c)}}\label{subsubsec:el(5;3)-c}

\

The Nichols algebra $\toba_{\bq}$ is generated by $(x_i)_{i\in \I_5}$ with defining relations
\begin{align}\label{eq:rels-el(5;3)-c}
\begin{aligned}
x_{24}&=0; & x_{13}&=0; & x_{14}&=0; & x_{15}&=0; & & [[x_{23},x_{(24)}]_c,x_3]_c=0;\\
x_{25}&=0; & x_{443}&=0; & x_{445}&=0; & x_{554}&=0; & &[x_{(13)},x_2]_c=0; \\
x_{35}&=0; & x_{1}^2&=0; & x_{2}^2&=0; & x_{3}^2&=0; & & x_{\alpha}^{3}=0, \ \alpha\in\Oc_+^{\bq}.
\end{aligned}
\end{align}
Here {\scriptsize$\Oc_+^{\bq}=\{  12, 23, 123^24, 2^23^24, 234, 4,
	12^33^44^35, 12^33^44^25, 12^23^34^25, 2^23^24^25, 123^24^25,\newline 12^33^44^35^2,  123^245, 2^23^245,
	2345, 45, 5 \}$}
and the degree of the integral is
\begin{equation*}
\ya= 24\alpha_1 + 66\alpha_2 + 84\alpha_3 + 60\alpha_4 + 32\alpha_5.
\end{equation*}

\subsubsection{The Dynkin diagram \emph{(\ref{eq:dynkin-el(5;3)}
		d)}}\label{subsubsec:el(5;3)-d}

\

The Nichols algebra $\toba_{\bq}$ is generated by $(x_i)_{i\in \I_5}$ with defining relations
\begin{align}\label{eq:rels-el(5;3)-d}
\begin{aligned}
x_{13}&=0; & x_{14}&=0; & x_{15}&=0; & [[x_{(24)},&x_{3}]_c,x_3]_c=0;\\
x_{24}&=0; & x_{25}&=0; & x_{35}&=0; & [x_{3345},&x_{34}]_c=0;\\
x_{112}&=0; & x_{332}&=0; & x_{443}&=0; & [x_{(13)},&x_2]_c=0; \\
x_{445}&=0; & x_{554}&=0; & x_{2}^2&=0; & x_{\alpha}^{3}&=0, \ \alpha\in\Oc_+^{\bq}.
\end{aligned}
\end{align}
Here {\scriptsize$\Oc_+^{\bq}=\{  1, 3, 12^23^24, 3^24, 34, 4,
	12^23^44^35, 12^23^44^25, 12^23^34^25, 12^23^24^25, 3^24^25, 12^23^44^35^2, 
	\newline12^23^245, 3^245,
	345, 45, 5 \}$}
and the degree of the integral is
\begin{equation*}
\ya= 24\alpha_1 + 44\alpha_2 + 84\alpha_3 + 60\alpha_4 + 32\alpha_5.
\end{equation*}

\subsubsection{The Dynkin diagram \emph{(\ref{eq:dynkin-el(5;3)}
		e)}}\label{subsubsec:el(5;3)-e}

\

The Nichols algebra $\toba_{\bq}$ is generated by $(x_i)_{i\in \I_5}$ with defining relations
\begin{align}\label{eq:rels-el(5;3)-e}
\begin{aligned}
x_{13}&=0; & x_{14}&=0; & x_{15}&=0; & [[[x_{5432},&x_{3}]_c,x_4]_c,x_3]_c=0;\\
x_{24}&=0; & x_{25}&=0; & x_{35}&=0; & x_{221}&=0;\\
x_{223}&=0; & x_{443}&=0; & x_{445}&=0; & [[[x_{(14)},&x_{3}]_c,x_2]_c,x_3]_c=0; \\
x_{554}&=0; & x_{1}^2&=0; & x_{3}^2&=0; & x_{\alpha}^{3}&=0, \ \alpha\in\Oc_+^{\bq}.
\end{aligned}
\end{align}
Here {\scriptsize$\Oc_+^{\bq}=\{ 2, 123, 1^22^23^24, 1234, 23^24, 4,
	1^22^33^44^35, 1^22^33^44^25, 12^23^34^25, 1^22^23^24^25, 23^24^25$, \\ $1^22^33^44^35^2, 1^22^23^245, 12345,
	23^245, 45, 5 \}$}  
and the degree of the integral is
\begin{equation*}
\ya= 44\alpha_1 + 66\alpha_2 + 84\alpha_3 + 60\alpha_4 + 32\alpha_5.
\end{equation*}

\subsubsection{The Dynkin diagram \emph{(\ref{eq:dynkin-el(5;3)}
		f)}}\label{subsubsec:el(5;3)-f}

\

The Nichols algebra $\toba_{\bq}$ is generated by $(x_i)_{i\in \I_5}$ with defining relations
\begin{align}\label{eq:rels-el(5;3)-f}
\begin{aligned}
\begin{aligned}
x_{13}&=0; & x_{24}&=0; & [x_{125},&x_{2}]_c=0; & & [x_{435},x_{3}]_c=0;
\\
x_{14}&=0; & x_{45}&=0; & [x_{(24)},&x_{3}]_c=0; & & x_{443}=0; \quad x_{2}^2=0; 
\\
x_{15}&=0; & x_{1}^2&=0; & [x_{(13)},&x_{2}]_c=0; & &x_{\alpha}^{3}=0, \ \alpha\in\Oc_+^{\bq};
\end{aligned}
\\
\begin{aligned}
x_{3}^2&=0; & x_{5}^2&=0; & x_{235}=q_{35}\ztu[x_{25},x_3]_c+q_{23}(1-\zeta)x_3x_{25}=0.
\end{aligned}
\end{aligned}
\end{align}
Here {\scriptsize$\Oc_+^{\bq}=\{ 12, 23, 1^22^234, 1234, 24, 34,
	1^22^33^24^35, 1^22^23^24^25, 1^22^334^25, 12^234^25, 234^25, \newline 1^22^33^24^35^2, 1^22^2345, 12345,
	245, 345, 5 \}$}
and the degree of the integral is
\begin{equation*}
\ya= 44\alpha_1 + 66\alpha_2 + 60\alpha_3 + 32\alpha_4 + 44\alpha_5.
\end{equation*}

\subsubsection{The Dynkin diagram \emph{(\ref{eq:dynkin-el(5;3)}
		g)}}\label{subsubsec:el(5;3)-g}

\

The Nichols algebra $\toba_{\bq}$ is generated by $(x_i)_{i\in \I_5}$ with defining relations
\begin{align}\label{eq:rels-el(5;3)-g}
\begin{aligned}
& \begin{aligned}
x_{13}&=0; & x_{14}&=0; & x_{15}&=0; & [x_{(24)},&x_{3}]_c=0;& & x_{24}=0;\\
x_{221}&=0; & x_{223}&=0; & x_{443}&=0; & [x_{235},&x_{3}]_c=0;& &x_{25}=0; \\
x_{445}&=0; & x_{3}^2&=0; & x_{5}^2&=0; & x_{\alpha}^{3}&=0, && \alpha\in\Oc_+^{\bq};
\end{aligned}
\\
& x_{112}=0; \quad  x_{(35)}=q_{45}\ztu[x_{35},x_4]_c +q_{34}(1-\zeta)x_4x_{35}.
\end{aligned}
\end{align}
Here {\scriptsize$\Oc_+^{\bq}=\{  1, 12, 2, 1234, 234, 34,
	12345, 12^23^345^2, 12^23^24^25^2, 2345, 123^24^25^2, 12345^2, 23^24^25^2, \newline 2345^2, 345, 345^2, 5\}$}
and the degree of the integral is
\begin{equation*}
\ya= 24\alpha_1 + 44\alpha_2 + 60\alpha_3 + 58\alpha_4 + 44\alpha_5.
\end{equation*}

\subsubsection{The Dynkin diagram \emph{(\ref{eq:dynkin-el(5;3)}
		h)}}\label{subsubsec:el(5;3)-h}

The Nichols algebra $\toba_{\bq}$ is generated by $(x_i)_{i\in \I_5}$ with defining relations
\begin{align}\label{eq:rels-el(5;3)-h}
\begin{aligned}
& \begin{aligned}
x_{13}&=0; & x_{14}&=0; & x_{15}&=0; & [x_{(24)},&x_{3}]_c=0; & & x_{24}=0;\\
x_{223}&=0; & x_{225}&=0; & x_{443}&=0; & [x_{435},&x_{3}]_c=0; & & x_{45}=0;\\
x_{1}^2&=0; & x_{3}^2&=0; & x_{5}^2&=0; & x_{\alpha}^{3}&=0, & & \alpha\in\Oc_+^{\bq};
\end{aligned}
\\ 
& x_{221}=0; \quad x_{235}= q_{35}\ztu[x_{25},x_3]_c +q_{23}(1-\zeta)x_3x_{25}.
\end{aligned}
\end{align}
Here {\scriptsize$\Oc_+^{\bq}=\{ 2, 123, 124, 2^234, 234, 34,
	12^33^24^35, 2^23^24^25, 12^334^25, 12^234^25, 1234^25, 12^33^24^35^2, \\ 1245, 2^2345, 2345, 345, 5 \}$}
and the degree of the integral is
\begin{equation*}
\ya= 24\alpha_1 + 66\alpha_2 + 60\alpha_3 + 32\alpha_4 + 44\alpha_5.
\end{equation*}

\subsubsection{The Dynkin diagram \emph{(\ref{eq:dynkin-el(5;3)}
		i)}}\label{subsubsec:el(5;3)-i}

\

The Nichols algebra $\toba_{\bq}$ is generated by $(x_i)_{i\in \I_5}$ with defining relations
\begin{align}\label{eq:rels-el(5;3)-i}
\begin{aligned}
x_{13}&=0; & x_{14}&=0; & x_{15}&=0; & [[x_{235},&x_{3}]_c,x_3]_c=0;\\
x_{24}&=0; & x_{25}&=0; & x_{45}&=0; & [x_{(13)},&x_2]_c=0;\\
x_{112}&=0; & x_{332}&=0; & x_{334}&=0; & [[x_{435},&x_{3}]_c,x_3]_c=0; \\
x_{553}&=0; & x_{2}^2&=0; & x_{4}^2&=0; & x_{\alpha}^{3}&=0, \ \alpha\in\Oc_+^{\bq}.
\end{aligned}
\end{align}
Here {\scriptsize$\Oc_+^{\bq}=\{ 1, 3, 12^23^24, 3^24, 34, 4,
	123^34^25, 23^34^25, 123^345, 23^345, 123^245, 12^23^44^25^2, 12345, \\ 1235, 23^245, 2345, 235 \}$}
and the degree of the integral is
\begin{equation*}
\ya= 24\alpha_1 + 44\alpha_2 + 84\alpha_3 + 32\alpha_4 + 44\alpha_5.
\end{equation*}

\subsubsection{The Dynkin diagram \emph{(\ref{eq:dynkin-el(5;3)}
		j)}}\label{subsubsec:el(5;3)-j}

\

The Nichols algebra $\toba_{\bq}$ is generated by $(x_i)_{i\in \I_5}$ with defining relations
\begin{align}\label{eq:rels-el(5;3)-j}
\begin{aligned}
x_{13}&=0; & x_{14}&=0; & x_{15}&=0; & [[x_{23},&x_{235}]_c,x_3]_c=0;\\
x_{24}&=0; & x_{25}&=0; & x_{45}&=0; & [x_{(13)},&x_2]_c=0;\\
x_{553}&=0; & [x_{(24)},&x_3]_c=0; & x_{1}^2&=0; & [x_{435},&x_{3}]_c=0; \\
x_{2}^2&=0; & x_{3}^2&=0; & x_{4}^2&=0; & x_{\alpha}^{3}&=0, \ \alpha\in\Oc_+^{\bq}.
\end{aligned}
\end{align}
Here
{\scriptsize$\Oc_+^{\bq}=\{ 12, 23, 123^24, 2^23^24, 234, 4,
	12^33^34^25, 2^23^34^25, 12^33^345, 2^23^345, 12^23^245, 12^33^44^25^2$, \\ $12345, 1235,
	23^245, 345, 35 \}$}
and the degree of the integral is
\begin{equation*}
\ya= 24\alpha_1 + 66\alpha_2 + 84\alpha_3 + 32\alpha_4 + 44\alpha_5.
\end{equation*}

\subsubsection{The Dynkin diagram \emph{(\ref{eq:dynkin-el(5;3)}
		k)}}\label{subsubsec:el(5;3)-k}

\

The Nichols algebra $\toba_{\bq}$ is generated by $(x_i)_{i\in \I_5}$ with defining relations
\begin{align}\label{eq:rels-el(5;3)-k}
\begin{aligned}
x_{13}&=0; & x_{14}&=0; & x_{15}&=0; & [[[x_{1235},&x_{3}]_c,x_2]_c,x_3]_c=0;\\
x_{24}&=0; & x_{25}&=0; & x_{45}&=0; & [x_{(24)},&x_3]_c=0;\\
x_{221}&=0; & x_{223}&=0; & x_{553}&=0; & [x_{435},&x_{3}]_c=0; \\
x_{1}^2&=0; & x_{3}^2&=0; & x_{4}^2&=0; & x_{\alpha}^{3}&=0, \ \alpha\in\Oc_+^{\bq}.
\end{aligned}
\end{align}
Here
{\scriptsize$\Oc_+^{\bq}=\{ 2, 123, 1^22^23^24, 1234, 23^24, 4,
	1^22^33^34^25, 1^22^23^34^25, 1^22^33^345, 1^22^23^345, 12^23^245$, \\ $1^22^33^44^25^2, 123^245, 2345,
	235, 345, 35 \}$}
and the degree of the integral is
\begin{equation*}
\ya= 44\alpha_1 + 66\alpha_2 + 84\alpha_3 + 32\alpha_4 + 44\alpha_5.
\end{equation*}

\subsubsection{The Dynkin diagram \emph{(\ref{eq:dynkin-el(5;3)}
		l)}}\label{subsubsec:el(5;3)-l}

\

The Nichols algebra $\toba_{\bq}$ is generated by $(x_i)_{i\in \I_5}$ with defining relations
\begin{align}\label{eq:rels-el(5;3)-l}
\begin{aligned}
x_{24}&=0; & x_{13}&=0; & x_{14}&=0; & x_{15}&=0; & [x_{(13)},&x_2]_c=0;\\
x_{25}&=0; & x_{332}&=0; & x_{334}&=0; & x_{335}&=0; & x_{553}&=0; \\
x_{45}&=0; & x_{1}^2&=0; & x_{2}^2&=0; & x_{4}^2&=0; & x_{\alpha}^{3}&=0, \ \alpha\in\Oc_+^{\bq}.
\end{aligned}
\end{align}
Here
{\scriptsize$\Oc_+^{\bq}=\{ 12, 123, 3, 1234, 34, 4,
	12^23^245, 12^33^44^25^2, 12^33^34^25^2, 12^33^345^2, 23^245, 123^245^2, \\ 2^23^34^25^2, 2345,
	2^23^345^2, 2^23^245^2, 235 \}$}
and the degree of the integral is
\begin{equation*}
\ya= 24\alpha_1 + 66\alpha_2 + 84\alpha_3 + 54\alpha_4 + 44\alpha_5.
\end{equation*}

\subsubsection{The Dynkin diagram \emph{(\ref{eq:dynkin-el(5;3)}
		m)}}\label{subsubsec:el(5;3)-m}

\

The Nichols algebra $\toba_{\bq}$ is generated by $(x_i)_{i\in \I_5}$ with defining relations
\begin{align}\label{eq:rels-el(5;3)-m}
\begin{aligned}
x_{13}&=0; & x_{14}&=0; & x_{15}&=0; & [[x_{43},&x_{435}]_c,x_3]_c=0;\\
x_{24}&=0; & x_{25}&=0; & x_{45}&=0; &  [x_{(24)},&x_3]_c=0;\\
[x_{(13)},&x_2]_c=0; & x_{112}&=0; & x_{553}&=0; & [x_{235},&x_3]_c=0; \\
x_{2}^2&=0; & x_{3}^2&=0; & x_{4}^2&=0; & x_{\alpha}^{3}&=0, \ \alpha\in\Oc_+^{\bq}.
\end{aligned}
\end{align}
Here
{\scriptsize$\Oc_+^{\bq}=\{ 1, 123, 23, 1234, 234, 4,
	123^245, 12^23^44^25^2, 23^245, 12^23^245^2, 123^34^25^2, 123^345^2, \\ 23^34^25^2, 23^345^2,
	345, 3^245^2, 35 \}$}
and the degree of the integral is
\begin{equation*}
\ya= 24\alpha_1 + 44\alpha_2 + 84\alpha_3 + 54\alpha_4 + 44\alpha_5.
\end{equation*}

\subsubsection{The Dynkin diagram \emph{(\ref{eq:dynkin-el(5;3)}
		n)}}\label{subsubsec:el(5;3)-n}

\

The Nichols algebra $\toba_{\bq}$ is generated by $(x_i)_{i\in \I_5}$ with defining relations
\begin{align}\label{eq:rels-el(5;3)-n}
\begin{aligned}
x_{24}&=0; & x_{13}&=0; & x_{14}&=0; & x_{15}&=0; & x_{221}&=0;\\
x_{25}&=0; & x_{223}&=0; & x_{332}&=0; & x_{334}&=0; & x_{335}&=0; \\
x_{45}&=0; & x_{553}&=0; & x_{1}^2&=0; & x_{4}^2&=0; & x_{\alpha}^{3}&=0, \ \alpha\in\Oc_+^{\bq}.
\end{aligned}
\end{align}
Here
{\scriptsize$\Oc_+^{\bq}=\{ 2, 23, 3, 234, 34, 4,
	12^23^245, 1^22^33^44^25^2, 1^22^33^34^25^2, 1^22^33^345^2, 123^245, 1^22^23^34^25^2$, \\ $1^22^23^345^2, 12345,
	1^22^23^245^2, 1235, 23^245^2 \}$}
and the degree of the integral is
\begin{equation*}
\ya= 48\alpha_1 + 66\alpha_2 + 84\alpha_3 + 54\alpha_4 + 44\alpha_5.
\end{equation*}

\subsubsection{The Dynkin diagram \emph{(\ref{eq:dynkin-el(5;3)} \~{n})}}
\label{subsubsec:el(5;3)-enie}

\

The Nichols algebra $\toba_{\bq}$ is generated by $(x_i)_{i\in \I_5}$ with defining relations
\begin{align}\label{eq:rels-el(5;3)-enie}
\begin{aligned}
x_{13}&=0; & x_{14}&=0; & x_{15}&=0; & [[[x_{4325},&x_{2}]_c,x_3]_c,x_2]_c=0;\\
x_{24}&=0; & x_{35}&=0; & x_{45}&=0; &  [x_{(13)},&x_2]_c=0;\\
x_{112}&=0; & x_{332}&=0; & x_{334}&=0; & [x_{125},&x_2]_c=0; \\
x_{443}&=0; & x_{552}&=0; & x_{2}^2&=0; & x_{\alpha}^{3}&=0, \ \alpha\in\Oc_+^{\bq}.
\end{aligned}
\end{align}
Here
{\scriptsize$\Oc_+^{\bq}=\{ 1, 3, 1^22^234, 12^234, 4, 2^234,
	1^22^43^24^35, 1^22^43^24^25, 1^22^234^25, 12^234^25, 2^234^25,\\ 1^22^43^24^35^2, 1^22^2345, 12^2345,
	45, 2^2345, 5 \}$}
and the degree of the integral is
\begin{equation*}
\ya= 44\alpha_1 + 84\alpha_2 + 60\alpha_3 + 32\alpha_4 + 44\alpha_5.
\end{equation*}

\subsubsection{The associated Lie algebra} This is of type $B_4\times A_1$.

\subsection{\ Type $\gtt(8,3)$}\label{subsec:type-g(8,3)}
Here $\theta = 5$, $\zeta \in \G'_3$. Let 
\begin{align*}
A&=\begin{pmatrix} 0 & 1 & 0 & 0 & 0 \\ \text{--}1 & 2 & 0 & \text{--}1 & 0 \\ 0 & 0 & 2 & \text{--}1 & \text{--}1
\\ 0 & \text{--}1 & \text{--}2 & 2 & 0 \\ 0 & 0 & \text{--}1 & 0 & 2 \end{pmatrix}
\in \kk^{5\times 5}; & \pa &= (-1,1, 1,1, 1) \in \G_2^5.
\end{align*}
Let $\g(8,3) = \g(A, \pa)$,
the contragredient Lie superalgebra corresponding to $(A, \pa)$. 
We know \cite{BGL} that $\sdim \g(8,3) = 55|50$. 
There are 20 other pairs of matrices and parity vectors for which the associated contragredient Lie superalgebra is isomorphic to $\g(8,3)$.
We describe now the root system $\gtt(8,3)$ of $\g(8,3)$, see \cite{AA-GRS-CLS-NA} for details.

\subsubsection{Basic datum and root system}
Below, $D_5$, $CE_5$, $A_5$, $C_5$, $E_6^{(2)}$, $F_4^{(1)}$, $C_2^{+++}$, $_1T_1$ and $_2 T$ are numbered as in \eqref{eq:dynkin-system-D}, \eqref{eq:CEn}, \eqref{eq:dynkin-system-A}, \eqref{eq:dynkin-system-C}, 
\eqref{eq:E6(2)}, \eqref{eq:F4(1)}, \eqref{eq:C2+++} and \eqref{eq:mTn}, respectively.
The basic datum and the bundle of Cartan matrices are described by the following diagram:
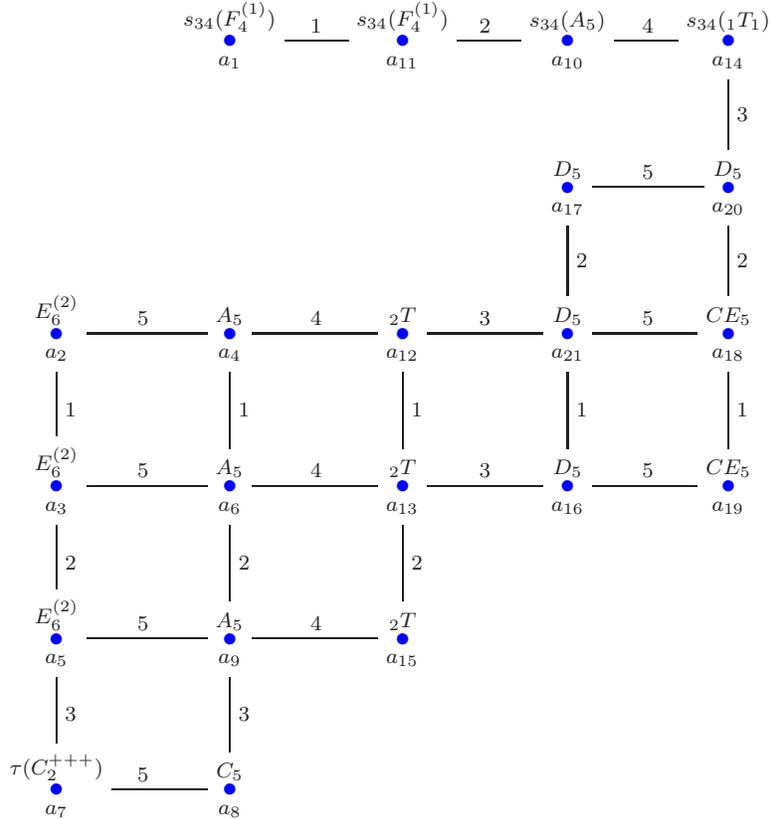
\begin{align*}
\xymatrix{ &
	\overset{s_{34}(F_4^{(1)})}{\underset{a_1}{\vtxgpd}} \ar@{-}[r]^1 &
	\overset{s_{34}(F_4^{(1)})}{\underset{a_{11}}{\vtxgpd}} \ar@{-}[r]^2 &
	\overset{s_{34}(A_5)}{\underset{a_{10}}{\vtxgpd}} \ar@{-}[r]^4 &
	\overset{s_{34}({}_1T_1)}{\underset{a_{14}}{\vtxgpd}} \ar@{-}[d]^3
	\\
	& & &
	\overset{D_5}{\underset{a_{17}}{\vtxgpd}} \ar@{-}[r]^5 \ar@{-}[d]^2 &
	\overset{D_5}{\underset{a_{20}}{\vtxgpd}} \ar@{-}[d]^2
	\\
	\overset{E_6^{(2)}}{\underset{a_2}{\vtxgpd}} \ar@{-}[r]^5 \ar@{-}[d]^1 &
	\overset{A_5}{\underset{a_4}{\vtxgpd}} \ar@{-}[r]^4 \ar@{-}[d]^1 &
	\overset{{}_2T}{\underset{a_{12}}{\vtxgpd}} \ar@{-}[r]^3 \ar@{-}[d]^1 &
	\overset{D_5}{\underset{a_{21}}{\vtxgpd}} \ar@{-}[r]^5 \ar@{-}[d]^1 &
	\overset{CE_5}{\underset{a_{18}}{\vtxgpd}} \ar@{-}[d]^1
	\\
	\overset{E_6^{(2)}}{\underset{a_3}{\vtxgpd}} \ar@{-}[r]^5 \ar@{-}[d]^2 &
	\overset{A_5}{\underset{a_6}{\vtxgpd}} \ar@{-}[r]^4 \ar@{-}[d]^2 &
	\overset{{}_2T}{\underset{a_{13}}{\vtxgpd}} \ar@{-}[r]^3 \ar@{-}[d]^2 &
	\overset{D_5}{\underset{a_{16}}{\vtxgpd}} \ar@{-}[r]^5 &
	\overset{CE_5}{\underset{a_{19}}{\vtxgpd}}
	\\
	\overset{E_6^{(2)}}{\underset{a_5}{\vtxgpd}} \ar@{-}[r]^5 \ar@{-}[d]^3 &
	\overset{A_5}{\underset{a_9}{\vtxgpd}} \ar@{-}[r]^4 \ar@{-}[d]^3 &
	\overset{{}_2T}{\underset{a_{15}}{\vtxgpd}} & &
	\\
	\overset{\tau(C_2^{+++})}{\underset{a_7}{\vtxgpd}} \ar@{-}[r]^5 &
	\overset{C_5}{\underset{a_8}{\vtxgpd}} & & & }
\end{align*}
Using the notation \eqref{eq:notation-root-exceptional}, the bundle of root sets is the following: { \tiny
	\begin{align*}
	\varDelta_{+}^{a_{1}}= & s_{34}(\{ 1, 12, 2, 123, 23, 3, 1^22^23^24, 12^23^24, 123^24, 1234, 23^24, 234, 3^24, 34, 4, 1^22^33^54^35, 1^22^33^44^35, \\
	& 1^22^23^44^35,  12^23^44^35, 1^22^33^44^25, 1^22^23^44^25, 12^23^44^25, 1^22^33^34^25, 1^22^23^34^25, 1^32^43^64^45^2, \\
	& 1^22^23^24^25, 1^22^23^245, 12^23^34^25, 12^23^24^25, 12^23^245, 1^22^33^64^45^2, 1^22^33^54^45^2, 1^22^33^54^35^2, \\
	& 123^34^25, 23^34^25, 1^22^33^44^35^2, 123^24^25, 23^24^25, 1^22^23^44^35^2, 12^23^44^35^2, 123^245, 12345, \\
	& 23^245, 2345, 3^24^25, 3^245, 345, 45, 5 \}), \\
	\varDelta_{+}^{a_{2}}= & \{ 1, 12, 2, 123, 23, 3, 1234, 12^23^24^2, 123^24^2, 1234^2, 234, 23^24^2, 234^2, 34, 34^2, 4, 12^33^44^55, \\
	& 12^33^44^45, 12^33^34^45, 12^23^34^45, 2^23^34^45, 12^33^34^35, 12^23^34^35, 2^23^34^35,  1^22^43^54^65^2, \\
	& 12^43^54^65^2, 12^23^24^35, 12^23^24^25, 2^23^24^35, 2^23^24^25, 12^33^54^65^2, 12^33^44^65^2, 12^33^44^55^2, \\
	& 123^24^35, 23^24^35, 12^33^44^45^2, 123^24^25, 23^24^25, 12^33^34^45^2, 12^23^34^45^2, 1234^25, 12345, \\
	& 2^23^34^45^2, 234^25, 2345, 34^25, 345, 45,  5 \}, \\
	\varDelta_{+}^{a_{3}}= & \{ 1, 12, 2, 123, 23, 3, 1234, 12^23^24^2, 123^24^2, 1234^2, 234, 23^24^2, 234^2, 34, 34^2, 4, 1^22^33^44^55, \\
	& 1^22^33^44^45, 1^22^33^34^45, 1^22^23^34^45, 12^23^34^45, 1^22^33^34^35, 1^22^23^34^35, 12^23^34^35, 1^32^43^54^65^2, \\
	& 1^22^23^24^35, 1^22^23^24^25, 1^22^43^54^65^2, 12^23^24^35, 12^23^24^25, 1^22^33^54^65^2, 1^22^33^44^65^2, \\
	& 1^22^33^44^55^2, 123^24^35, 23^24^35, 1^22^33^44^45^2, 123^24^25, 23^24^25, 1^22^33^34^45^2, 1^22^23^34^45^2, \\
	& 12^23^34^45^2, 1234^25, 12345, 234^25, 2345, 34^25, 345, 45, 5 \}, \\
	\varDelta_{+}^{a_{4}}= & \{ 1, 12, 2, 123, 23, 3, 1234, 234, 34, 4, 12^23^24^25, 123^24^25, 1234^25, 12345, 12^33^44^45^2, 12^33^34^45^2, \\
	& 12^33^34^35^2, 2^23^24^25, 12^23^34^45^2, 12^23^34^35^2, 23^24^25, 1^22^43^54^65^4, 12^33^44^55^3, 12^23^24^35^2, \\
	& 123^24^35^2, 2^23^34^45^2, 234^25, 34^25, 12^43^54^65^4, 12^33^54^65^4,  12^33^44^45^3, 2^23^34^35^2, 12^33^44^65^4, \\
	& 12^33^34^45^3, 2^23^24^35^2, 12^23^34^45^3, 23^24^35^2, 12^33^44^55^4, 12^23^24^25^2, 123^24^25^2, 1234^25^2, \\
	& 2^23^34^45^3, 2345, 23^24^25^2, 234^25^2, 345, 34^25^2, 45, 5 \}, \\
	\varDelta_{+}^{a_{5}}= & \{ 1, 12, 2, 123, 23, 3, 1234, 12^23^24^2, 123^24^2, 1234^2, 234, 23^24^2,  234^2, 34, 34^2, 4, 12^23^44^55, \\
	& 12^23^44^45, 12^23^34^45, 123^34^45, 23^34^45, 12^23^34^35, 123^34^35, 23^34^35, 1^22^33^54^65^2, 12^33^54^65^2, \\
	& 12^23^24^35, 12^23^24^25, 12^23^54^65^2, 123^24^35, 123^24^25, 12^23^44^65^2, 12^23^44^55^2, 23^24^35, 3^24^35, 12^23^44^45^2, \\
	& 23^24^25, 3^24^25, 12^23^34^45^2, 123^34^45^2, 1234^25, 12345, 23^34^45^2, 234^25, 2345, 34^25, 345, 45, 5 \}, \\
	\varDelta_{+}^{a_{6}}= & \{ 1, 12, 2, 123, 23, 3, 1234, 234, 34, 4, 1^22^23^24^25, 12^23^24^25, 1^22^33^44^45^2, 1^22^33^34^45^2, 1^22^33^34^35^2, \\
	& 123^24^25, 1^22^23^34^45^2, 1^22^23^34^35^2, 1234^25, 1^22^23^24^35^2, 12345, 1^32^43^54^65^4, 1^22^33^44^55^3, \\
	& 1^22^33^44^45^3, 1^22^33^34^45^3, 1^22^23^34^45^3, 12^23^34^45^2, 12^23^34^35^2, 1^22^43^54^65^4, 1^22^33^54^65^4, \\
	& 23^24^25, 12^23^24^35^2, 1^22^33^44^65^4, 234^25, 123^24^35^2, 34^25, 1^22^33^44^55^4, 12^23^34^45^3, 12^23^24^25^2, \\
	& 123^24^25^2, 1234^25^2, 23^24^35^2, 2345, 23^24^25^2, 234^25^2, 345, 34^25^2, 45, 5 \}, \\
	\varDelta_{+}^{a_{7}}= & \{ 1, 12, 2, 123, 23, 3, 1234, 12^23^24^2, 123^24^2, 1234^2, 234, 23^24^2, 234^2, 34, 34^2, 4, 12^23^34^55, 12^23^34^45, \\
	& 12^23^24^45, 123^24^45, 23^24^45, 12^23^34^35, 12^23^24^35, 1^22^33^44^65^2, 12^33^44^65^2, 12^23^24^25, 123^24^35, \\
	& 12^23^44^65^2, 123^24^25, 23^24^35, 23^24^25, 12^23^34^65^2, 12^23^34^55^2, 1234^35, 234^35, 34^35, \\
	& 12^23^34^45^2, 1234^25, 12^23^24^45^2, 123^24^45^2, 12345, 234^25, 23^24^45^2, 2345, 34^25, 345, 4^25, 45, 5 \}, \\
	\varDelta_{+}^{a_{8}}= & \{ 1, 12, 2, 123, 23, 3, 1234, 234, 34, 4, 12^23^24^25, 123^24^25, 1234^25, 12345, 12^23^34^45^2, 12^23^34^35^2, \\
	& 23^24^25, 12^23^24^45^2, 12^23^24^35^2, 234^25, 12^23^24^25^2, 2345, 1^22^33^44^65^4, 12^33^44^65^4, 12^23^34^55^3, \\
	& 12^23^34^45^3, 12^23^24^45^3, 123^24^45^2, 123^24^35^2, 1234^35^2, 123^24^25^2, 1234^25^2, 12^23^44^65^4,  \\
	& 12^23^34^65^4, 12^23^34^55^4, 123^24^45^3, 23^24^45^2, 23^24^35^2, 34^25, 23^24^25^2, 345, 23^24^45^3, 234^35^2, \\
	& 234^25^2, 34^35^2, 34^25^2, 4^25, 45, 5 \}, \\
	\varDelta_{+}^{a_{9}}= & \{ 1, 12, 2, 123, 23, 3, 1234, 234, 34, 4, 12^23^24^25, 123^24^25, 1234^25, 12345, 12^23^44^45^2, 12^23^34^45^2, \\
	& 12^23^34^35^2, 23^24^25, 12^23^24^35^2, 234^25, 12^23^24^25^2, 2345, 1^22^33^54^65^4, 12^33^54^65^4, 12^23^44^55^3, \\
	& 12^23^44^45^3, 12^23^34^45^3, 123^34^45^2, 123^34^35^2, 123^24^35^2, 123^24^25^2, 1234^25^2, 12^23^54^65^4, \\
	& 12^23^44^65^4, 12^23^44^55^4, 123^34^45^3, 23^34^45^2, 23^34^35^2, 3^24^25, 23^24^35^2, 34^25, 23^34^45^3, \\
	& 23^24^25^2,234^25^2, 3^24^35^2, 345, 34^25^2, 45, 5 \}, \\
	\varDelta_{+}^{a_{10}}= & s_{34}(\{ 1, 12, 2, 123, 23, 3, 1^22^23^24, 12^23^24, 123^24, 1234, 2^23^24, 23^24, 234, 34, 4, 1^22^43^54^35, 1^22^43^44^35, \\
	& 1^22^33^44^35, 12^33^44^35, 1^22^43^44^25, 1^22^33^44^25, 12^33^44^25, 1^22^33^34^25, 1^22^23^34^25, 1^32^53^64^45^2, \\
	& 1^22^23^24^25, 1^22^23^245, 12^33^34^25, 1^22^53^64^45^2, 12^23^34^25, 1^22^43^54^45^2, 1^22^43^54^35^2, 2^23^34^25, \\
	& 12^23^24^25, 1^22^43^44^35^2, 2^23^24^25, 12^23^245, 2^23^245, 1^22^33^44^35^2, 12^33^44^35^2, 123^24^25, 123^245, \\
	& 12345, 23^24^25, 23^245, 2345, 345, 45, 5 \}), \\
	\varDelta_{+}^{a_{11}}= & s_{34} (\{ 1, 12, 2, 123, 23, 3, 12^23^24, 123^24, 1234, 2^23^24, 23^24, 234, 3^24, 34, 4, 12^33^54^35, 12^33^44^35, 12^23^44^35, \\
	& 2^23^44^35, 12^33^44^25, 12^23^44^25, 2^23^44^25, 12^33^34^25, 12^23^34^25, 12^43^64^45^2, 12^23^24^25, 12^23^245, \\
	& 2^23^34^25, 2^23^24^25, 2^23^245, 12^33^64^45^2, 12^33^54^45^2, 12^33^54^35^2, 123^34^25, 23^34^25, 12^33^44^35^2, \\
	& 123^24^25, 23^24^25, 12^23^44^35^2, 123^245, 12345, 2^23^44^35^2, 23^245, 2345, 3^24^25, 3^245, 345, 45, 5 \}), \\
	\varDelta_{+}^{a_{12}}= & \{ 1, 12, 2, 123, 23, 3, 1234, 234, 34, 4, 12^23^245, 123^245, 12345, 1235, 12^33^44^25^2, 12^33^34^25^2, 12^33^345^2, \\
	& 2^23^245, 12^23^34^25^2, 12^23^345^2, 23^245, 1^22^43^54^35^4, 12^33^44^35^3, 12^23^24^25^2, 123^24^25^2, 12^33^44^25^3, \\
	& 12^23^245^2, 123^245^2, 12^43^54^35^4, 12^33^54^35^4, 2^23^34^25^2, 2^23^345^2, 12^33^34^25^3, 12^23^34^25^3, 12^33^44^35^4, \\
	& 12^33^44^25^4, 12345^2, 2345, 2^23^245^2, 235, 2^23^34^25^3, 23^24^25^2, 23^245^2, 2345^2, 345, 45, 345^2, 35, 5 \}, \\
	\varDelta_{+}^{a_{13}}= & \{ 1, 12, 2, 123, 23, 3, 1234, 234, 34, 4, 1^22^23^245, 12^23^245, 1^22^33^44^25^2, 1^22^33^34^25^2, 1^22^33^345^2, \\
	& 123^245, 1^22^23^34^25^2, 1^22^23^345^2, 12345, 1^22^23^245^2, 1235, 1^32^43^54^35^4, 1^22^33^44^35^3, 1^22^33^44^25^3,\\
	& 1^22^33^34^25^3, 1^22^23^34^25^3, 12^23^34^25^2, 12^23^24^25^2, 123^24^25^2, 1^22^43^54^35^4, 1^22^33^54^35^4, \\
	& 12^23^345^2, 23^245, 1^22^33^44^35^4, 12^23^245^2, 2345, 1^22^33^44^25^4, 12^23^34^25^3, 123^245^2, 12345^2, \\
	& 23^24^25^2, 345, 45, 23^245^2, 2345^2, 235, 345^2, 35, 5 \}, \\
	\varDelta_{+}^{a_{14}}= & s_{34} (\{ 1, 12, 2, 123, 23, 3, 1^22^234, 12^234, 1234, 124, 2^234, 234, 24, 34, 4, 1^22^43^34^35, 1^22^43^24^35, 1^22^33^24^35, \\
	& 12^33^24^35, 1^22^43^24^25, 1^22^33^24^25, 12^33^24^25, 1^22^23^24^25, 12^23^24^25, 2^23^24^25, 1^32^53^34^45^2, \\
	& 1^22^53^34^45^2, 1^22^43^34^45^2, 1^22^43^34^35^2, 1^22^334^25, 1^22^234^25, 1^22^2345, 12^334^25, 12^234^25, 1^22^43^24^35^2, \\
	& 2^234^25, 12^2345, 2^2345, 1^22^33^24^35^2, 12^33^24^35^2, 1234^25, 12345, 1245, 234^25, 2345, 245, 345, 45, 5 \}), \\
	\varDelta_{+}^{a_{15}}= & \{ 1, 12, 2, 123, 23, 3, 1234, 234, 34, 4, 12^23^245, 123^245, 12345, 1235, 12^23^44^25^2, 12^23^34^25^2, 12^23^345^2, \\
	& 23^245, 12^23^24^25^2, 12^23^245^2, 2345, 235, 1^22^33^54^35^4, 12^33^54^35^4, 12^23^44^35^3, 12^23^44^25^3, 12^23^34^25^3, \\
	& 123^34^25^2, 123^24^25^2, 123^345^2, 123^245^2, 12345^2, 12^23^54^35^4, 12^23^44^35^4, 12^23^44^25^4, 123^34^25^3, \\
	& 23^34^25^2, 23^345^2, 3^245, 23^24^25^2, 23^245^2, 23^34^25^3, 2345^2, 345, 45, 3^245^2, 345^2, 35, 5 \}, \\
	\varDelta_{+}^{a_{16}}= & \{ 1, 12, 2, 123, 23, 3, 1234, 234, 34, 4, 1^22^23^245, 12^23^245,  1^22^33^44^25^2, 1^22^33^34^25^2, 1^22^33^345^2, \\
	& 123^245, 1^22^23^34^25^2, 1^22^23^345^2, 12345, 1^22^23^245^2, 1235, 1^32^43^64^35^4, 1^22^33^54^35^3, 1^22^33^54^25^3, \\
	& 1^22^33^44^25^3, 1^22^23^44^25^3, 12^23^44^25^2, 12^23^34^25^2, 123^34^25^2, 1^22^43^64^35^4, 1^22^33^64^35^4, 12^23^345^2, \\
	& 23^245, 123^345^2, 3^245, 1^22^33^54^35^4, 1^22^33^54^25^4, 12^23^44^25^3, 12^23^245^2, 123^245^2, 23^34^25^2, 2345, \\
	& 345, 23^345^2, 23^245^2, 235, 3^245^2, 35, 5 \}, \\
	\varDelta_{+}^{a_{17}}= & \{ 1, 12, 2, 123, 23, 3, 1234, 234, 34, 4, 1^22^23^245, 12^23^245, 1^22^33^44^25^2, 1^22^33^34^25^2, 1^22^33^345^2, \\
	& 123^245, 1^22^23^34^25^2, 1^22^23^345^2, 12345, 1^22^23^245^2, 1235, 1^32^53^64^35^4, 1^22^43^54^35^3, 1^22^43^54^25^3, \\
	& 1^22^43^44^25^3, 1^22^33^44^25^3, 12^33^44^25^2, 12^33^34^25^2, 12^23^34^25^2, 1^22^53^64^35^4, 12^33^345^2, 2^23^245, \\
	& 1^22^43^64^35^4, 12^23^345^2, 23^245, 1^22^43^54^35^4, 1^22^43^54^25^4, 12^33^44^25^3, 12^23^245^2, 123^245^2, \\
	& 2^23^34^25^2, 2345, 345, 2^23^345^2, 2^23^245^2, 23^245^2, 235, 35, 5 \}, \\
	\varDelta_{+}^{a_{18}}= & \{ 1, 12, 2, 123, 23, 3, 12^23^24, 123^24, 1234, 2^23^24, 23^24, 234, 3^24, 34, 4, 12^33^54^35, 12^33^54^25, \\
	& 12^33^44^25, 12^23^44^25, 2^23^44^25, 12^33^34^25, 12^23^34^25, 2^23^34^25, 123^34^25, 23^34^25, 1^22^43^64^35^2, \\
	& 12^43^64^35^2, 12^33^64^35^2, 12^33^54^35^2, 12^33^345, 12^23^345, 12^23^245, 123^345, 123^245, 12345, 12^33^54^25^2, \\
	& 12^33^44^25^2, 12^23^44^25^2, 1235, 2^23^345, 23^345, 2^23^245, 23^245, 2345, 235, 3^245, 345, 35, 5 \}, \\
	\varDelta_{+}^{a_{19}}= & \{ 1, 12, 2, 123, 23, 3, 1^22^23^24, 12^23^24, 123^24, 1234, 23^24, 234, 3^24, 34, 4, 1^22^33^54^35, 1^22^33^54^25, \\
	& 1^22^33^44^25, 1^22^23^44^25, 12^23^44^25, 1^22^33^34^25, 1^22^23^34^25, 12^23^34^25, 123^34^25, 23^34^25, 1^32^43^64^35^2, \\
	& 1^22^43^64^35^2, 1^22^33^64^35^2, 1^22^33^54^35^2, 1^22^33^345, 1^22^23^345, 1^22^23^245, 12^23^345, 12^23^245, 1^22^33^54^25^2, \\
	& 123^345, 23^345, 1^22^33^44^25^2, 123^245, 12^23^44^25^2, 12345, 1235, 23^245, 2345, 235, 3^245, 345, 35, 5 \}, \\
	\varDelta_{+}^{a_{20}}= & \{ 1, 12, 2, 123, 23, 3, 1^22^23^24, 12^23^24, 123^24, 1234, 2^23^24, 23^24, 234, 34, 4, 1^22^43^54^35, 1^22^43^54^25, \\
	& 1^22^43^44^25,  1^22^33^44^25, 12^33^44^25, 1^22^33^34^25, 12^33^34^25, 1^22^23^34^25, 12^23^34^25, 2^23^34^25, 1^32^53^64^35^2, \\
	& 1^22^53^64^35^2, 1^22^43^64^35^2, 1^22^43^54^35^2, 1^22^33^345, 1^22^23^345, 1^22^23^245, 12^33^345, 12^23^345, 1^22^43^54^25^2, \\
	& 2^23^345, 12^23^245, 2^23^245, 1^22^33^44^25^2, 12^33^44^25^2, 123^245, 12345, 1235, 23^245, 2345, 235, 345, 35, 5 \}, \\
	\varDelta_{+}^{a_{21}}= & \{ 1, 12, 2, 123, 23, 3, 1234, 234, 34, 4, 12^23^245, 123^245, 12345, 1235, 12^33^44^25^2, 12^33^34^25^2, 12^33^345^2, \\
	& 2^23^245, 12^23^44^25^2, 12^23^34^25^2, 1^22^43^64^35^4, 12^33^54^35^3, 123^34^25^2, 12^23^345^2, 12^33^54^25^3, 123^345^2, \\
	& 23^245, 3^245, 12^43^64^35^4, 12^33^64^35^4, 12^33^44^25^3, 2^23^34^25^2, 2^23^345^2, 12^23^44^25^3, 23^34^25^2, 23^345^2, \\
	& 12^33^54^35^4, 12^33^54^25^4, 12^23^245^2, 123^245^2, 2^23^44^25^3, 2345, 345, 2^23^245^2, 23^245^2, 235, 3^245^2, 35, 5 \}.
	\end{align*}
}%

\subsubsection{Weyl groupoid}
\label{subsubsec:type-g83-Weyl}
The isotropy group  at $a_1 \in \cX$ is
\begin{align*}
\cW(a_1) &= \langle 
\varsigma_1^{a_1}\varsigma_2 \varsigma_3\varsigma_4 
\varsigma_2 \varsigma_5 \varsigma_4 \varsigma_3
\varsigma_1 \varsigma_2 \varsigma_4 \varsigma_5 \varsigma_3 \varsigma_5 \varsigma_4 \varsigma_2 \varsigma_1 \varsigma_3\varsigma_4 \varsigma_5 \varsigma_2 \varsigma_4\varsigma_3 \varsigma_2 \varsigma_1,
\varsigma_2^{a_1},  \varsigma_3^{a_1}, \\  & \qquad  \varsigma_4^{a_1},  \varsigma_5^{a_1} \rangle   \simeq \Z/2  \times W(F_4).
\end{align*}

\subsubsection{Incarnation}
We set the matrices $(\bq^{(i)})_{i\in\I_{21}}$, from left to right and  from up to down:
\begin{align}\label{eq:dynkin-g(8,3)}
& && \xymatrix@C-4pt{\overset{-1}{\underset{\ }{\circ}}\ar  @{-}[r]^{\zeta}  &
	\overset{\ztu}{\underset{\ }{\circ}} \ar  @{-}[r]^{\zeta}  & \overset{\ztu}{\underset{\
		}{\circ}}\ar  @{-}[r]^{\ztu}
	& \overset{\zeta}{\underset{\ }{\circ}} \ar  @{-}[r]^{\ztu}  & \overset{\zeta}{\underset{\
		}{\circ}}}
\\ \notag
&
\xymatrix@C-4pt{\overset{-1}{\underset{\ }{\circ}}\ar  @{-}[r]^{\ztu}  &
	\overset{\zeta}{\underset{\ }{\circ}} \ar  @{-}[r]^{\ztu}  & \overset{\zeta}{\underset{\
		}{\circ}}
	\ar  @{-}[r]^{\ztu}  & \overset{\ztu}{\underset{\ }{\circ}} \ar  @{-}[r]^{\zeta}  &
	\overset{-1}{\underset{\ }{\circ}}}
& & \xymatrix@C-4pt{\overset{-1}{\underset{\ }{\circ}}\ar  @{-}[r]^{\zeta}  &
	\overset{-1}{\underset{\ }{\circ}} \ar  @{-}[r]^{\ztu}  & \overset{\zeta}{\underset{\
		}{\circ}}
	\ar  @{-}[r]^{\ztu}  & \overset{\ztu}{\underset{\ }{\circ}} \ar  @{-}[r]^{\zeta}  &
	\overset{-1}{\underset{\ }{\circ}}}
\\ \notag
&\xymatrix@C-4pt{\overset{-1}{\underset{\ }{\circ}}\ar  @{-}[r]^{\ztu}  &
	\overset{\zeta}{\underset{\ }{\circ}} \ar  @{-}[r]^{\ztu}  & \overset{\zeta}{\underset{\
		}{\circ}}
	\ar  @{-}[r]^{\ztu}  & \overset{-1}{\underset{\ }{\circ}} \ar  @{-}[r]^{\ztu}  &
	\overset{-1}{\underset{\ }{\circ}}}
&&\xymatrix@C-4pt{\overset{\zeta}{\underset{\ }{\circ}}\ar  @{-}[r]^{\ztu}  &
	\overset{-1}{\underset{\ }{\circ}} \ar  @{-}[r]^{\zeta}  & \overset{-1}{\underset{\
		}{\circ}}
	\ar  @{-}[r]^{\ztu}  & \overset{\ztu}{\underset{\ }{\circ}} \ar  @{-}[r]^{\zeta}  &
	\overset{-1}{\underset{\ }{\circ}}}
\\ \notag 
& \xymatrix@C-4pt{\overset{-1}{\underset{\ }{\circ}}\ar  @{-}[r]^{\zeta}  &
	\overset{-1}{\underset{\ }{\circ}} \ar  @{-}[r]^{\ztu}  & \overset{\zeta}{\underset{\
		}{\circ}}
	\ar  @{-}[r]^{\ztu}  & \overset{-1}{\underset{\ }{\circ}} \ar  @{-}[r]^{\ztu}  &
	\overset{-1}{\underset{\ }{\circ}}}
&& \xymatrix@C-4pt{\overset{\zeta}{\underset{\ }{\circ}}\ar  @{-}[r]^{\ztu}  &
	\overset{\zeta}{\underset{\ }{\circ}} \ar  @{-}[r]^{\ztu}  & \overset{-1}{\underset{\
		}{\circ}}
	\ar  @{-}[r]^{\zeta}  & \overset{-\zeta}{\underset{\ }{\circ}} \ar  @{-}[r]^{\zeta}  &
	\overset{-1}{\underset{\ }{\circ}}}
\\ \notag
& \xymatrix@C-4pt{\overset{\zeta}{\underset{\ }{\circ}}\ar  @{-}[r]^{\ztu}  &
	\overset{\zeta}{\underset{\ }{\circ}} \ar  @{-}[r]^{\ztu}  & \overset{-1}{\underset{\
		}{\circ}}\ar  @{-}[r]^{\zeta}
	& \overset{\ztu}{\underset{\ }{\circ}} \ar  @{-}[r]^{\ztu}  & \overset{-1}{\underset{\
		}{\circ}}}
&&
\xymatrix@C-4pt{\overset{\zeta}{\underset{\ }{\circ}}\ar  @{-}[r]^{\ztu}  &
	\overset{-1}{\underset{\ }{\circ}} \ar  @{-}[r]^{\zeta}  & \overset{-1}{\underset{\
		}{\circ}}\ar  @{-}[r]^{\ztu}
	& \overset{-1}{\underset{\ }{\circ}} \ar  @{-}[r]^{\ztu}  & \overset{-1}{\underset{\
		}{\circ}}}
\\ \notag
& \xymatrix@C-4pt{\overset{\ztu}{\underset{\ }{\circ}}\ar  @{-}[r]^{\zeta}  &
	\overset{-1}{\underset{\ }{\circ}} \ar  @{-}[r]^{\ztu}  & \overset{-1}{\underset{\
		}{\circ}}\ar  @{-}[r]^{\ztu}
	& \overset{\zeta}{\underset{\ }{\circ}} \ar  @{-}[r]^{\ztu}  & \overset{\zeta}{\underset{\
		}{\circ}}}
&& \xymatrix@C-4pt{\overset{-1}{\underset{\ }{\circ}}\ar  @{-}[r]^{\ztu}  &
	\overset{-1}{\underset{\ }{\circ}} \ar  @{-}[r]^{\zeta}  & \overset{\ztu}{\underset{\
		}{\circ}}\ar  @{-}[r]^{\ztu}
	& \overset{\zeta}{\underset{\ }{\circ}} \ar  @{-}[r]^{\ztu}  & \overset{\zeta}{\underset{\
		}{\circ}}}
\end{align}
\begin{align*}
&\xymatrix@R-8pt{  &   & \overset{\ztu}{\circ} \ar  @{-}[d]^{\zeta}\ar@{-}[dr]^{\zeta}  & \\
	\overset{-1}{\underset{\ }{\circ}} \ar  @{-}[r]^{\ztu}  & \overset{\zeta}{\underset{\
		}{\circ}} \ar  @{-}[r]^{\ztu}  & \overset{-1}{\underset{\ }{\circ}} \ar  @{-}[r]^{\zeta}
	& \overset{-1}{\underset{\ }{\circ}}}
&& \xymatrix@R-8pt{  &   & \overset{\ztu}{\circ} \ar  @{-}[d]^{\zeta} \ar@{-}[dr]^{\zeta} &
	\\
	\overset{-1}{\underset{\ }{\circ}} \ar  @{-}[r]^{\zeta}  & \overset{-1}{\underset{\
		}{\circ}} \ar  @{-}[r]^{\ztu}  & \overset{-1}{\underset{\ }{\circ}} \ar  @{-}[r]^{\zeta}
	& \overset{-1}{\underset{\ }{\circ}}}
\\ & \xymatrix@R-8pt{  &    \overset{-1}{\circ} \ar  @{-}[d]^{\zeta} \ar  @{-}[dr]^{\zeta}
	&  & \\
	\overset{\ztu}{\underset{\ }{\circ}} \ar  @{-}[r]^{\zeta}  & \overset{\ztu}{\underset{\
		}{\circ}} \ar  @{-}[r]^{\zeta}  & \overset{-1}{\underset{\ }{\circ}} \ar  @{-}[r]^{\ztu}
	& \overset{\zeta}{\underset{\ }{\circ}}}
&& \xymatrix@R-8pt{  &   & \overset{\ztu}{\circ} \ar  @{-}[d]^{\zeta} \ar@{-}[dr]^{\zeta} &
	\\
	\overset{\zeta}{\underset{\ }{\circ}} \ar  @{-}[r]^{\ztu}  & \overset{-1}{\underset{\
		}{\circ}} \ar  @{-}[r]^{\zeta}  &
	\overset{\ztu}{\underset{\ }{\circ}} \ar  @{-}[r]^{\zeta}  & \overset{-1}{\underset{\
		}{\circ}}}
\end{align*}
\begin{align*}
& \xymatrix@R-8pt{  &   & \overset{-1}{\circ} \ar  @{-}[d]^{\ztu}  & \\
	\overset{-1}{\underset{\ }{\circ}} \ar  @{-}[r]^{\zeta}  & \overset{\ztu}{\underset{\
		}{\circ}} \ar  @{-}[r]^{\zeta}  & \overset{-1}{\underset{\ }{\circ}} \ar  @{-}[r]^{\ztu}
	& \overset{\zeta}{\underset{\ }{\circ}}}
&& \xymatrix@R-8pt{  &   & \overset{-1}{\circ} \ar  @{-}[d]^{\ztu}  & \\
	\overset{\ztu}{\underset{\ }{\circ}} \ar  @{-}[r]^{\zeta}  & \overset{-1}{\underset{\
		}{\circ}} \ar  @{-}[r]^{\ztu}  & \overset{\zeta}{\underset{\ }{\circ}} \ar  @{-}[r]^{\ztu}
	& \overset{\zeta}{\underset{\ }{\circ}}}
\\ & \xymatrix@R-8pt{  &   & \overset{-1}{\circ} \ar  @{-}[d]^{\ztu}  & \\
	\overset{-1}{\underset{\ }{\circ}} \ar  @{-}[r]^{\ztu}  & \overset{-1}{\underset{\
		}{\circ}} \ar  @{-}[r]^{\zeta}  & \overset{\zeta}{\underset{\ }{\circ}} \ar
	@{-}[r]^{\ztu}  & \overset{\zeta}{\underset{\ }{\circ}}}
&& \xymatrix@R-8pt{  &   & \overset{-1}{\circ} \ar  @{-}[d]^{\ztu}  & \\
	\overset{-1}{\underset{\ }{\circ}} \ar  @{-}[r]^{\zeta}  & \overset{\ztu}{\underset{\
		}{\circ}} \ar  @{-}[r]^{\zeta}  & \overset{\zeta}{\underset{\ }{\circ}} \ar
	@{-}[r]^{\ztu}  & \overset{\zeta}{\underset{\ }{\circ}}}
\\ & \xymatrix@R-8pt{  &   & \overset{-1}{\circ} \ar  @{-}[d]^{\zeta}  & \\
	\overset{\ztu}{\underset{\ }{\circ}} \ar  @{-}[r]^{\zeta}  & \overset{-1}{\underset{\
		}{\circ}} \ar  @{-}[r]^{\ztu}  & \overset{-1}{\underset{\ }{\circ}} \ar  @{-}[r]^{\ztu}
	& \overset{\zeta}{\underset{\ }{\circ}}}
&& \xymatrix@R-8pt{  &   & \overset{-1}{\circ} \ar  @{-}[d]^{\ztu}  & \\
	\overset{-1}{\underset{\ }{\circ}} \ar  @{-}[r]^{\ztu}  & \overset{-1}{\underset{\
		}{\circ}} \ar  @{-}[r]^{\zeta}  & \overset{-1}{\underset{\ }{\circ}} \ar  @{-}[r]^{\ztu}
	& \overset{\zeta}{\underset{\ }{\circ}}}
\end{align*}
Now, this is the incarnation:
\begin{align*}
& a_i\mapsto s_{34}(\bq^{i}), \ i\in \{1,10,11,14 \}; &
& a_i\mapsto \bq^{(i)}, \text{ otherwise}.
\end{align*}

\subsubsection{PBW-basis and dimension} \label{subsubsec:type-g83-PBW}
Notice that the roots in each $\varDelta_{+}^{a_i}$, $i\in\I_{21}$, are ordered from left to right, justifying the notation $\beta_1, \dots, \beta_{49}$.

The root vectors $x_{\beta_k}$ are described as in Remark \ref{rem:lyndon-word}.
Thus
\begin{align*}
\left\{ x_{\beta_{49}}^{n_{49}} \dots x_{\beta_2}^{n_{2}}  x_{\beta_1}^{n_{1}} \, | \, 0\le n_{k}<N_{\beta_k} \right\}.
\end{align*}
is a PBW-basis of $\toba_{\bq}$. Hence $\dim \toba_{\bq}=2^{24}3^{24}6=2^{25}3^{25}$.

\subsubsection{The Dynkin diagram \emph{(\ref{eq:dynkin-g(8,3)} a)}}
\label{subsubsec:g(8,3)-a}

\quad

The Nichols algebra $\toba_{\bq}$ is generated by $(x_i)_{i\in \I_5}$ with defining relations
\begin{align}\label{eq:rels-g(8,3)-a}
x_{221}&=0; & x_{223}&=0; & x_{332}&=0; & &x_{ij}=0, \ i<j, \, \widetilde{q}_{ij}=1;
\\ \notag
& & x_{443}&=0; & x_{445}&=0; & & [[x_{(24)},x_{3}]_c,x_3]_c=0;
\\ \notag
[x_{3345},&x_{34}]_c=0; & x_{554}&=0; & x_{1}^2&=0; & &x_{\alpha}^{N_\alpha}=0, \ \alpha\in\Oc_+^{\bq}.
\end{align}
Here  {\scriptsize$\Oc_+^{\bq}=\{ 2, 23, 3, 2^23^24, 23^24, 234,
	3^24, 34, 4, 2^23^44^35, 2^23^44^25, 12^23^34^25, 2^23^34^25, 2^23^24^25, \\
	2^23^245,  23^34^25, 23^24^25, 2^23^44^35^2, 23^245, 2345,
	3^24^25, 3^245, 345, 45, 5 \}$}
and the degree of the integral is
\begin{equation*}
\ya= 29\alpha_1 + 102\alpha_2 + 171\alpha_3 + 118\alpha_4 + 61\alpha_5.
\end{equation*}

\subsubsection{The Dynkin diagram \emph{(\ref{eq:dynkin-g(8,3)}
		b)}}\label{subsubsec:g(8,3)-b}

\

The Nichols algebra $\toba_{\bq}$ is generated by $(x_i)_{i\in \I_5}$ with defining relations
\begin{align}\label{eq:rels-g(8,3)-b}
\begin{aligned}
x_{221}&=0; & x_{223}&=0; & x_{332}&=0; & & x_{ij}=0, \ i<j, \, \widetilde{q}_{ij}=1;\\
& & x_{334}&=0; & x_{445}&=0; & & [[x_{(35)},x_{4}]_c,x_4]_c=0;\\
[x_{4432},&x_{43}]_c=0; & x_{1}^2&=0; & x_{5}^2&=0; & & x_{\alpha}^{N_\alpha}=0, \ \alpha\in\Oc_+^{\bq}.
\end{aligned}
\end{align}
Here {\scriptsize$\Oc_+^{\bq}=\{ 2, 23, 3, 1234, 234, 23^24^2,
	234^2, 34, 34^2, 4, 12^23^34^45, 12^23^34^35, 1^22^43^54^65^2, 12^23^24^35$, \\ $12^23^24^25, 1^22^33^54^65^2, 1^22^33^44^65^2, 1^22^33^44^55^2, 123^24^35, 1^22^33^44^45^2, 123^24^25, 1^22^33^34^45^2$, \\ $1^22^23^34^45^2, 1234^25, 12345  \}$}  
and the degree of the integral is
\begin{equation*}
\ya= 75\alpha_1 + 117\alpha_2 + 155\alpha_3 + 189\alpha_4 + 64\alpha_5.
\end{equation*}

\subsubsection{The Dynkin diagram \emph{(\ref{eq:dynkin-g(8,3)}
		c)}}\label{subsubsec:g(8,3)-c}

\

The Nichols algebra $\toba_{\bq}$ is generated by $(x_i)_{i\in \I_5}$ with defining relations
\begin{align}\label{eq:rels-g(8,3)-c}
\begin{aligned}
&[x_{(13)},x_2]_c=0; & x_{332}&=0; & x_{334}&=0; & & x_{ij}=0, \ i<j, \, \widetilde{q}_{ij}=1;\\
& & x_{445}&=0; & x_{1}^2&=0; & &[[x_{(35)},x_{4}]_c,x_4]_c=0;\\
&[x_{4432},x_{43}]_c=0; & x_{2}^2&=0; & x_{5}^2&=0; & & x_{\alpha}^{N_\alpha}=0, \ \alpha\in\Oc_+^{\bq}.
\end{aligned}
\end{align}
Here {\scriptsize$\Oc_+^{\bq}=\{ 12, 123, 3, 1234, 123^24^2, 1234^2,
	234, 34, 34^2, 4, 12^23^34^45, 12^23^34^35, 1^22^43^54^65^2, \\ 
	12^23^24^35, 12^23^24^25, 12^33^54^65^2, 12^33^44^65^2, 12^33^44^55^2, 23^24^35, 12^33^44^45^2,
	23^24^25, 12^33^34^45^2$, \\ $2^23^34^45^2, 234^25, 2345 \}$}  
and the degree of the integral is
\begin{equation*}
\ya= 44\alpha_1 + 117\alpha_2 + 155\alpha_3 + 189\alpha_4 + 64\alpha_5.
\end{equation*}

\subsubsection{The Dynkin diagram \emph{(\ref{eq:dynkin-g(8,3)}
		d)}}\label{subsubsec:g(8,3)-d}

\

The Nichols algebra $\toba_{\bq}$ is generated by $(x_i)_{i\in \I_5}$ with defining relations
\begin{align}\label{eq:rels-g(8,3)-d}
\begin{aligned}
x_{221}&=0; & x_{223}&=0; & x_{332}&=0; & x_{ij}&=0, \ i<j, \, \widetilde{q}_{ij}=1;\\
& & x_{334}&=0; & x_{1}^2&=0; & [[x_{54},&x_{543}]_c,x_4]_c=0;\\
& & x_{4}^2&=0; & x_{5}^2&=0; & x_{\alpha}^{N_\alpha}&=0, \ \alpha\in\Oc_+^{\bq}.
\end{aligned}
\end{align}
Here {\scriptsize$\Oc_+^{\bq}=\{  2, 23, 3, 1234, 12^23^24^25, 1^22^33^44^45^2,
	1^22^33^34^45^2, 123^24^25, 1^22^23^34^45^2, 1234^25, \\ 12345, 1^22^33^44^55^3, 12^23^34^35^2, 1^22^43^54^65^4,
	1^22^33^54^65^4, 12^23^24^35^2, 1^22^33^44^65^4, 123^24^35^2, \\ 12^23^34^45^3, 2345,
	23^24^25^2, 234^25^2, 345, 34^25^2, 45 \}$}  
and the degree of the integral is
\begin{equation*}
\ya= 75\alpha_1 + 117\alpha_2 + 155\alpha_3 + 189\alpha_4 + 127\alpha_5.
\end{equation*}

\subsubsection{The Dynkin diagram \emph{(\ref{eq:dynkin-g(8,3)}
		e)}}\label{subsubsec:g(8,3)-e}

\

The Nichols algebra $\toba_{\bq}$ is generated by $(x_i)_{i\in \I_5}$ with defining relations
\begin{align}\label{eq:rels-g(8,3)-e}
\begin{aligned}
x_{112}&=0; & [x_{(13)},&x_2]_c=0; & x_{2}^2&=0; & & x_{ij}=0, \ i<j, \, \widetilde{q}_{ij}=1;\\
x_{445}&=0; & [x_{(24)},&x_3]_c=0; & x_{3}^2&=0; & &[[x_{(35)},x_4]_c,x_4]_c=0;\\
[x_{443},&x_{43}]_c=0; & [x_{4432},&x_{43}]_c=0; & x_{5}^2&=0; & & x_{\alpha}^{N_\alpha}=0, \ \alpha\in\Oc_+^{\bq}.
\end{aligned}
\end{align}
Here {\scriptsize$\Oc_+^{\bq}=\{ 1, 123, 23, 1234, 12^23^24^2, 1234^2,
	234, 234^2, 34, 4, 12^23^34^45, 12^23^34^35, 1^22^33^54^65^2, \\ 12^33^54^65^2,  123^24^35, 123^24^25, 12^23^44^65^2, 12^23^44^55^2, 23^24^35, 12^23^44^45^2,
	23^24^25, 123^34^45^2, \\23^34^45^2, 34^25, 345  \}$}  
and the degree of the integral is
\begin{equation*}
\ya= 44\alpha_1 + 117\alpha_2 + 155\alpha_3 + 189\alpha_4 + 64\alpha_5.
\end{equation*}

\subsubsection{The Dynkin diagram \emph{(\ref{eq:dynkin-g(8,3)}
		f)}}\label{subsubsec:g(8,3)-f}

\

The Nichols algebra $\toba_{\bq}$ is generated by $(x_i)_{i\in \I_5}$ with defining relations
\begin{align}\label{eq:rels-g(8,3)-f}
\begin{aligned}
x_{332}&=0; & x_{334}&=0; & x_{1}^2&=0; & x_{ij}&=0, \ i<j, \, \widetilde{q}_{ij}=1;\\
& & x_{2}^2&=0; & x_{4}^2&=0; & [[x_{54},&x_{543}]_c,x_4]_c=0;\\
& & [x_{(13)},&x_2]_c=0; & x_{5}^2&=0; & x_{\alpha}^{N_\alpha}&=0, \ \alpha\in\Oc_+^{\bq}.
\end{aligned}
\end{align}
Here {\scriptsize$\Oc_+^{\bq}=\{ 12, 123, 3, 234, 12^23^24^25, 12345,
	12^33^44^45^2, 12^33^34^45^2, 12^23^34^35^2, 23^24^25, \\ 1^22^43^54^65^4, 12^33^44^55^3, 12^23^24^35^2, 2^23^34^45^2,
	234^25, 12^33^54^65^4, 12^33^44^65^4, 12^23^34^45^3, 23^24^35^2,\\ 123^24^25^2, 
	1234^25^2, 2345, 345, 34^25^2, 45 \}$}  
and the degree of the integral is
\begin{equation*}
\ya= 44\alpha_1 + 117\alpha_2 + 155\alpha_3 + 189\alpha_4 + 127\alpha_5.
\end{equation*}

\subsubsection{The Dynkin diagram \emph{(\ref{eq:dynkin-g(8,3)}
		g)}}\label{subsubsec:g(8,3)-g}

\

The Nichols algebra $\toba_{\bq}$ is generated by $(x_i)_{i\in \I_5}$ with defining relations
\begin{align}\label{eq:rels-g(8,3)-g}
\begin{aligned}
\begin{aligned}
x_{112}&=0; & x_{4443}&=0; & [x_{(24)},&x_{3}]_c=0; & x_{ij}&=0, \ i<j, \, \widetilde{q}_{ij}=1;
\\
x_{221}&=0; & x_{4445}&=0; & x_{223}&=0; & x_{\alpha}^{N_\alpha}&=0, \ \alpha\in\Oc_+^{\bq};
\end{aligned}
\\
\begin{aligned}
x_{3}^2&=0; & x_{5}^2&=0; & 
&[x_3,x_{445}]_c= (\ztu-\zeta)
q_{34}x_4x_{(35)} -q_{45}[x_{(35)},x_4]_c.
\end{aligned}
\end{aligned}
\end{align}
Here {\scriptsize$\Oc_+^{\bq}=\{  1, 12, 2, 1234, 12^23^24^2, 123^24^2,
	234, 23^24^2, 34, 4, 12^23^34^45, 12^23^24^35, 1^22^33^44^65^2, \\ 12^33^44^65^2, 
	123^24^35, 12^23^44^65^2, 23^24^35, 12^23^34^55^2, 1234^25, 12^23^24^45^2,
	123^24^45^2, 234^25, 23^24^45^2, \\ 34^25, 45  \}$}  
and the degree of the integral is
\begin{equation*}
\ya= 44\alpha_1 + 84\alpha_2 + 120\alpha_3 + 189\alpha_4 + 64\alpha_5.
\end{equation*}

\subsubsection{The Dynkin diagram \emph{(\ref{eq:dynkin-g(8,3)}
		h)}}\label{subsubsec:g(8,3)-h}

The Nichols algebra $\toba_{\bq}$ is generated by $(x_i)_{i\in \I_5}$ with defining relations
\begin{align}\label{eq:rels-g(8,3)-h}
\begin{aligned}
& [x_{(24)},x_3]_c=0; & x_{112}&=0; & x_{221}&=0; & &x_{ij}=0, \ i<j, \, \widetilde{q}_{ij}=1;\\
& & x_{223}&=0; & x_{445}&=0; & & [[x_{(35)},x_{4}]_c,x_4]_c=0; \\
& [x_{445},x_{45}]_c=0; & x_{3}^2&=0; & x_{5}^2&=0; & &x_{\alpha}^{N_\alpha}=0, \ \alpha\in\Oc_+^{\bq}.
\end{aligned}
\end{align}
Here {\scriptsize$\Oc_+^{\bq}=\{  1, 12, 2, 4, 1234^25, 12345,
	12^23^24^45^2, 12^23^24^35^2, 234^25, 12^23^24^25^2, 2345, 1^22^33^44^65^4$, 
	\\ $12^33^44^65^4, 12^23^34^55^3,
	12^23^34^45^3, 123^24^45^2, 123^24^35^2, 123^24^25^2, 12^23^44^65^4, 23^24^45^2$, 
	\\ $23^24^35^2, 34^25, 23^24^25^2, 345, 45 \}$} 
and the degree of the integral is
\begin{equation*}
\ya= 44\alpha_1 + 84\alpha_2 + 120\alpha_3 + 189\alpha_4 + 127\alpha_5.
\end{equation*}

\subsubsection{The Dynkin diagram \emph{(\ref{eq:dynkin-g(8,3)}
		i)}}\label{subsubsec:g(8,3)-i}

\

The Nichols algebra $\toba_{\bq}$ is generated by $(x_i)_{i\in \I_5}$ with defining relations
\begin{align}\label{eq:rels-g(8,3)-i}
\begin{aligned}
x_{112}&=0; & & [x_{(24)},x_3]_c=0; & x_{2}^2&=0; & & x_{ij}=0, \ i<j, \, \widetilde{q}_{ij}=1;
\\
& & & [x_{(13)},x_2]_c=0; & x_{3}^2&=0; & & [[x_{54},x_{543}]_c,x_4]_c=0;
\\
x_{5}^2&=0; & & [[x_{34},x_{(35)}]_c,x_4]_c=0; & x_{4}^2&=0; &  & x_{\alpha}^{N_\alpha}=0, \ \alpha\in\Oc_+^{\bq}.
\end{aligned}
\end{align}
Here {\scriptsize$\Oc_+^{\bq}=\{  1, 123, 23, 34, 123^24^25, 12345,
	12^23^44^45^2, 12^23^34^35^2, 23^24^25, 12^23^24^25^2, 2345, \\ 1^22^33^54^65^4, 
	12^33^54^65^4, 12^23^44^55^3,
	12^23^34^45^3, 123^34^45^2, 123^24^35^2, 1234^25^2, 12^23^44^65^4, 23^34^45^2,
	\\ 23^24^35^2, 34^25, 234^25^2, 345, 45  \}$}  
and the degree of the integral is
\begin{equation*}
\ya= 44\alpha_1 + 84\alpha_2 + 155\alpha_3 + 189\alpha_4 + 127\alpha_5.
\end{equation*}

\subsubsection{The Dynkin diagram \emph{(\ref{eq:dynkin-g(8,3)}
		j)}}\label{subsubsec:g(8,3)-j}

\

The Nichols algebra $\toba_{\bq}$ is generated by $(x_i)_{i\in \I_5}$ with defining relations
\begin{align}\label{eq:rels-g(8,3)-j}
\begin{aligned}
& [x_{(13)},x_2]_c=0; & x_{112}&=0; & x_{443}&=0; & x_{ij}&=0, \ i<j, \, \widetilde{q}_{ij}=1;\\
& & x_{445}&=0; & x_{554}&=0; & [[x_{23},&x_{(24)}]_c,x_3]_c=0;\\
& & x_{2}^2&=0; & x_{3}^2&=0; & x_{\alpha}^{N_\alpha}&=0, \ \alpha\in\Oc_+^{\bq}.
\end{aligned}
\end{align}
Here {\scriptsize$\Oc_+^{\bq}=\{  1, 123, 23, 1^22^23^24, 12^23^24, 1234,
	2^23^24, 234, 4, 1^22^43^44^35, 1^22^43^44^25, 1^22^33^34^25$, \\ 
	$1^22^23^24^25, 1^22^23^245, 12^33^34^25, 12^23^34^25, 12^23^24^25, 1^22^43^44^35^2, 2^23^24^25, 12^23^245$, \\ 
	$2^23^245, 12345, 2345, 45, 5 \}$}   
and the degree of the integral is
\begin{equation*}
\ya= 75\alpha_1 + 146\alpha_2 + 171\alpha_3 + 118\alpha_4 + 61\alpha_5.
\end{equation*}

\subsubsection{The Dynkin diagram \emph{(\ref{eq:dynkin-g(8,3)}
		k)}}\label{subsubsec:g(8,3)-k}

\

The Nichols algebra $\toba_{\bq}$ is generated by $(x_i)_{i\in \I_5}$ with defining relations
\begin{align}\label{eq:rels-g(8,3)-k}
\begin{aligned}
& [x_{(13)},x_2]_c=0; & x_{332}&=0; & x_{443}&=0; & & x_{ij}1=0, \ i<j, \, \widetilde{q}_{ij}=1;\\
& & x_{445}&=0; & x_{554}&=0; & & [[x_{(35)},x_{4}]_c,x_4]_c=0;\\
& [x_{3345},x_{34}]_c=0; & x_{1}^2&=0; & x_{2}^2&=0; & & x_{\alpha}^{N_\alpha}=0, \ \alpha\in\Oc_+^{\bq}.
\end{aligned}
\end{align}
Here {\scriptsize$\Oc_+^{\bq}=\{  12, 123, 3, 1^22^23^24, 123^24, 1234,
	3^24, 34, 4, 1^22^23^44^35, 1^22^23^44^25, 1^22^23^34^25, \\ 1^22^23^24^25, 
	1^22^23^245, 12^23^34^25, 123^34^25, 123^24^25, 1^22^23^44^35^2, 123^245, 12345,
	3^24^25, 3^245, \\ 345, 45, 5 \}$}  
and the degree of the integral is
\begin{equation*}
\ya= 75\alpha_1 + 102\alpha_2 + 171\alpha_3 + 118\alpha_4 + 61\alpha_5.
\end{equation*}

\subsubsection{The Dynkin diagram \emph{(\ref{eq:dynkin-g(8,3)}
		l)}}\label{subsubsec:g(8,3)-l}

\

The Nichols algebra $\toba_{\bq}$ is generated by $(x_i)_{i\in \I_5}$ with defining relations
\begin{align}\label{eq:rels-g(8,3)-l}
\begin{aligned}
& \begin{aligned}
x_{221}&=0; & x_{443}&=0; & [x_{(24)},&x_3]_c=0; & x_{ij}&=0, \ i<j, \, \widetilde{q}_{ij}=1;
\\
x_{223}&=0; & x_{445}&=0; & [x_{235},&x_3]_c=0; & x_{\alpha}^{N_\alpha}&=0, \ \alpha\in\Oc_+^{\bq};
\end{aligned}
\\
& x_1^2=0; \ x_3^2=0; \ x_5^2=0; \quad x_{(35)} =q_{45}\ztu[x_{35},x_4]_c +q_{34}(1-\zeta)x_4x_{35}.
\end{aligned}
\end{align}
Here {\scriptsize$\Oc_+^{\bq}=\{  2, 123, 234, 34, 12^23^245, 1^22^33^44^25^2,
	1^22^33^345^2, 123^245, 1^22^23^345^2, 12345, 1235$, \\ $1^22^33^44^25^3, 12^23^34^25^2, 1^22^43^54^35^4,
	1^22^33^54^35^4, 12^23^245^2, 2345, 1^22^33^44^25^4, 12^23^34^25^3$, \\ $123^245^2,
	23^24^25^2, 345, 2345^2, 345^2, 5 \}$}  
and the degree of the integral is
\begin{equation*}
\ya= 75\alpha_1 + 117\alpha_2 + 155\alpha_3 + 127\alpha_4 + 95\alpha_5.
\end{equation*}

\subsubsection{The Dynkin diagram \emph{(\ref{eq:dynkin-g(8,3)}
		m)}}\label{subsubsec:g(8,3)-m}

\

The Nichols algebra $\toba_{\bq}$ is generated by $(x_i)_{i\in \I_5}$ with defining relations
\begin{align}\label{eq:rels-g(8,3)-m}
\begin{aligned}
& \begin{aligned}
x_{443}&=0; & [x_{(13)},&x_2]_c=0; & x_1^2&=0; & x_{ij}&=0, & & i<j, \ \widetilde{q}_{ij}=1;
\\
x_{445}&=0; & [x_{(24)},&x_3]_c=0; & x_2^2&=0; & x_5^2&=0; & & x_{\alpha}^{N_\alpha}=0, \ \alpha\in\Oc_+^{\bq};
\end{aligned}
\\
& [x_{235},x_3]_c=0; \  x_3^2=0; \
x_{(35)}=q_{45}\ztu[x_{35},x_4]_c+q_{34}(1-\zeta)x_4x_{35}.
\end{aligned}
\end{align}
Here {\scriptsize$\Oc_+^{\bq}=\{  12, 23, 1234, 34, 12^23^245, 12345,
	12^33^44^25^2, 12^33^345^2, 12^23^34^25^2, 23^245, \\ 1^22^43^54^35^4, 
	123^24^25^2, 12^33^44^25^3, 12^23^245^2,
	12^33^54^35^4, 2^23^345^2, 12^23^34^25^3, 12^33^44^25^4, 12345^2, \\ 2345, 
	235, 23^245^2, 345, 345^2, 5 \}$}  
and the degree of the integral is
\begin{equation*}
\ya= 44\alpha_1 + 117\alpha_2 + 155\alpha_3 + 127\alpha_4 + 95\alpha_5.
\end{equation*}

\subsubsection{The Dynkin diagram \emph{(\ref{eq:dynkin-g(8,3)}
		n)}}\label{subsubsec:g(8,3)-n}

\

The Nichols algebra $\toba_{\bq}$ is generated by $(x_i)_{i\in \I_5}$ with defining relations
\begin{align}\label{eq:rels-g(8,3)-n}
\begin{aligned}
& \begin{aligned}
& [x_{(24)},x_3]_c=0; & x_{112}&=0; & x_{221}&=0; & x_{ij}&=0, & & i<j, \, \widetilde{q}_{ij}=1;
\\
& [x_{435},x_3]_c=0; & x_{223}&=0; & x_{225}&=0; & x_5^2&=0; & & x_{\alpha}^{N_\alpha}=0, \ \alpha\in\Oc_+^{\bq}; 
\end{aligned}
\\
& x_{443}=0; \quad x_3^2=0; \quad x_{235} =q_{35}\ztu[x_{25},x_3]_c +q_{23}(1-\zeta)x_3x_{25}.
\end{aligned}
\end{align}
Here {\scriptsize$\Oc_+^{\bq}=\{  1, 12, 2, 1^22^234, 12^234, 1234,
	2^234, 234, 34, 1^22^43^34^35, 1^22^43^24^25, 1^22^33^24^25, \\ 12^33^24^25, 1^22^23^24^25, 12^23^24^25, 2^23^24^25, 1^22^43^34^35^2, 1^22^2345, 12^234^25, 12^2345,
	2^2345, 12345, \\ 2345,  345, 5 \}$}  
and the degree of the integral is
\begin{equation*}
\ya= 75\alpha_1 + 146\alpha_2 + 118\alpha_3 + 61\alpha_4 + 95\alpha_5.
\end{equation*}

\subsubsection{The Dynkin diagram \emph{(\ref{eq:dynkin-g(8,3)} \~{n})}}
\label{subsubsec:g(8,3)-enie}

\

The Nichols algebra $\toba_{\bq}$ is generated by $(x_i)_{i\in \I_5}$ with defining relations
\begin{align}\label{eq:rels-g(8,3)-enie}
\begin{aligned}
& \begin{aligned}
x_{112}&=0; & x_{332}&=0; & x_{443}&=0; &x_2^2&=0; & x_{ij}&=0, \ i<j, \, \widetilde{q}_{ij}=1;
\\
x_{334}&=0; & x_{335}&=0; & x_{445}&=0; & x_5^2&=0; & x_{\alpha}^{N_\alpha}&=0, \ \alpha\in\Oc_+^{\bq};
\end{aligned}
\\
& [x_{(13)},x_2]_c=0; \quad  x_{(35)}=q_{45}\ztu[x_{35},x_4]_c+q_{34}(1-\zeta)x_4x_{35}.
\end{aligned}
\end{align}
Here {\scriptsize$\Oc_+^{\bq}=\{ 1, 3, 1234, 234, 123^245, 12345,
	12^23^44^25^2, 123^345^2, 12^23^34^25^2, 23^245, 1^22^33^54^35^4$, \\ $12^23^24^25^2, 12^23^44^25^3, 123^245^2,
	12^33^54^35^4, 23^345^2, 12^23^34^25^3, 12^23^44^25^4, 12345^2, 345$, \\ 
	$35, 23^245^2, 2345, 2345^2, 5 \}$}   
and the degree of the integral is
\begin{equation*}
\ya= 44\alpha_1 + 84\alpha_2 + 155\alpha_3 + 127\alpha_4 + 95\alpha_5.
\end{equation*}

\subsubsection{The Dynkin diagram \emph{(\ref{eq:dynkin-g(8,3)}
		o)}}\label{subsubsec:g(8,3)-o}

\

The Nichols algebra $\toba_{\bq}$ is generated by $(x_i)_{i\in \I_5}$ with defining relations
\begin{align}\label{eq:rels-g(8,3)-o}
\begin{aligned}
& [x_{(24)},x_3]_c=0; & x_{221}&=0; & x_{223}&=0; & x_{ij}&=0, \ i<j, \, \widetilde{q}_{ij}=1;\\
& & x_{553}&=0; & x_1^2&=0; & [[x_{43},&x_{435}]_c,x_3]_c=0;\\
& [x_{235},x_3]_c=0; & x_{3}^2&=0; & x_{4}^2&=0; & x_{\alpha}^{N_\alpha}&=0, \ \alpha\in\Oc_+^{\bq}.
\end{aligned}
\end{align}
Here {\scriptsize$\Oc_+^{\bq}=\{  2, 123, 1234, 4, 12^23^245, 123^245,
	12^33^34^25^2, 12^33^345^2, 12^23^34^25^2, 1^22^43^64^35^4, \\ 123^34^25^2, 
	12^23^345^2, 123^345^2, 23^245,
	12^33^44^25^3, 12^23^44^25^3, 12^33^54^35^4, 12^33^54^25^4, 2345, 345, \\
	2^23^245^2, 23^245^2, 235, 3^245^2, 35 \}$}  
and the degree of the integral is
\begin{equation*}
\ya= 44\alpha_1 + 117\alpha_2 + 186\alpha_3 + 127\alpha_4 + 95\alpha_5.
\end{equation*}

\subsubsection{The Dynkin diagram \emph{(\ref{eq:dynkin-g(8,3)}
		p)}}\label{subsubsec:g(8,3)-p}

\

The Nichols algebra $\toba_{\bq}$ is generated by $(x_i)_{i\in \I_5}$ with defining relations
\begin{align}\label{eq:rels-g(8,3)-p}
\begin{aligned}
x_{112}&=0; & x_{332}&=0; & x_{334}&=0; & x_{ij}&=0, \ i<j, \, \widetilde{q}_{ij}=1;\\
& & x_{335}&=0; & x_{553}&=0; & [x_{(13)},&x_2]_c=0;\\
& & x_{2}^2&=0; & x_{4}^2&=0; & x_{\alpha}^{N_\alpha}&=0, \ \alpha\in\Oc_+^{\bq}.
\end{aligned}
\end{align}
Here {\scriptsize$\Oc_+^{\bq}=\{ 1, 3, 34, 4, 12^23^245, 123^245,
	1^22^23^34^25^2, 1^22^23^345^2, 12345, 1^22^23^245^2, 1235, \\ 1^22^33^44^25^3, 
	12^23^34^25^2, 1^22^43^64^35^4,
	12^23^345^2, 23^245, 1^22^43^54^35^4, 1^22^43^54^25^4, 12^33^44^25^3, \\ 12^23^245^2,
	2^23^34^25^2, 2345, 2^23^345^2, 2^23^245^2, 235 \}$}  
and the degree of the integral is
\begin{equation*}
\ya= 75\alpha_1 + 146\alpha_2 + 186\alpha_3 + 127\alpha_4 + 95\alpha_5.
\end{equation*}

\subsubsection{The Dynkin diagram \emph{(\ref{eq:dynkin-g(8,3)}
		q)}}\label{subsubsec:g(8,3)-q}

\

The Nichols algebra $\toba_{\bq}$ is generated by $(x_i)_{i\in \I_5}$ with defining relations
\begin{align}\label{eq:rels-g(8,3)-q}
\begin{aligned}
& [x_{(13)},x_2]_c=0; & x_{332}&=0; & x_1^2&=0; & & x_{ij}=0, \ i<j, \, \widetilde{q}_{ij}=1;\\
& & x_{553}&=0; & x_2^2&=0; & & [[x_{235},x_3]_c,x_3]_c=0;\\
& [[x_{435},x_3]_c,x_3]_c=0; & x_{334}&=0; & x_{4}^2&=0; & & x_{\alpha}^{N_\alpha}=0, \ \alpha\in\Oc_+^{\bq}.
\end{aligned}
\end{align}
Here {\scriptsize$\Oc_+^{\bq}=\{ 12, 123, 3, 1^22^23^24, 123^24, 1234,
	3^24, 34, 4, 1^22^33^54^35, 1^22^33^54^25, 1^22^33^44^25, \\ 12^23^44^25, 
	1^22^33^34^25,
	12^23^34^25, 23^34^25, 1^22^43^64^35^2, 1^22^33^345, 12^23^345, 12^23^245,
	23^345, 123^245, \\ 23^245, 2345, 235 \}$}  
and the degree of the integral is 
\begin{equation*}
\ya= 75\alpha_1 + 117\alpha_2 + 186\alpha_3 + 61\alpha_4 + 95\alpha_5.
\end{equation*}

\subsubsection{The Dynkin diagram \emph{(\ref{eq:dynkin-g(8,3)}
		r)}}\label{subsubsec:g(8,3)-r}

\

The Nichols algebra $\toba_{\bq}$ is generated by $(x_i)_{i\in \I_5}$ with defining relations
\begin{align}\label{eq:rels-g(8,3)-r}
\begin{aligned}
x_{221}&=0; & x_{223}&=0; & x_{332}&=0; & & x_{ij}=0, \ i<j, \, \widetilde{q}_{ij}=1;\\
& & x_{334}&=0; & x_{553}&=0; & & [[x_{235},x_3]_c,x_3]_c=0;\\
[[x_{435},&x_3]_c,x_3]_c=0; & x_{1}^2&=0; & x_{4}^2&=0; & & x_{\alpha}^{N_\alpha}=0, \ \alpha\in\Oc_+^{\bq}.
\end{aligned}
\end{align}
Here {\scriptsize$\Oc_+^{\bq}=\{ 2, 23, 3, 2^23^24, 23^24, 234,
	3^24, 34, 4, 12^33^54^35, 12^33^54^25, 12^33^44^25, 12^23^44^25,
	\\ 12^33^34^25, 12^23^34^25, 123^34^25, 1^22^43^64^35^2, 12^33^345, 12^23^345, 12^23^245,
	123^345, 123^245, 12345, \\ 1235, 23^245 \}$}  
and the degree of the integral is
\begin{equation*}
\ya= 44\alpha_1 + 117\alpha_2 + 186\alpha_3 + 61\alpha_4 + 95\alpha_5.
\end{equation*}

\subsubsection{The Dynkin diagram \emph{(\ref{eq:dynkin-g(8,3)}
		s)}}\label{subsubsec:g(8,3)-s}

\

The Nichols algebra $\toba_{\bq}$ is generated by $(x_i)_{i\in \I_5}$ with defining relations
\begin{align}\label{eq:rels-g(8,3)-s}
\begin{aligned}
x_{112}&=0; & x_{2}^2&=0; & [x_{(13)},&x_2]_c=0; & x_{ij}&=0, \ i<j, \, \widetilde{q}_{ij}=1;
\\
& & x_{3}^2&=0; & [x_{(24)},&x_3]_c=0; & [[x_{23},&x_{235}]_c,x_3]_c=0;
\\
x_{553}&=0; & x_{4}^2&=0; & [x_{435},&x_3]_c=0; & x_{\alpha}^{N_\alpha}&=0, \ \alpha\in\Oc_+^{\bq}.
\end{aligned}
\end{align}
Here {\scriptsize$\Oc_+^{\bq}=\{ 1, 123, 23, 1^22^23^24, 12^23^24, 1234,
	2^23^24, 234, 4, 1^22^43^54^35, 1^22^43^54^25, 1^22^33^44^25$, \\ $12^33^44^25, 1^22^23^34^25,
	12^23^34^25, 2^23^34^25, 1^22^43^64^35^2, 1^22^23^345, 12^23^345, 2^23^345,
	12^23^245$, \\ $123^245, 23^245, 345, 35 \}$}  
and the degree of the integral is
\begin{equation*}
\ya= 75\alpha_1 + 146\alpha_2 + 186\alpha_3 + 61\alpha_4 + 95\alpha_5.
\end{equation*}

\subsubsection{The Dynkin diagram \emph{(\ref{eq:dynkin-g(8,3)} t)}}
\label{subsubsec:g(8,3)-t}

\

The Nichols algebra $\toba_{\bq}$ is generated by $(x_i)_{i\in \I_5}$ with defining relations
\begin{align}\label{eq:rels-g(8,3)-t}
\begin{aligned}
x_{443}&=0; & & [x_{(13)},x_2]_c=0; & x_{1}^2&=0; & & x_{ij}=0, \ i<j, \, \widetilde{q}_{ij}=1;
\\
& & & [x_{(24)},x_3]_c=0; & x_{2}^2&=0; & & [[x_{43},x_{435}]_c,x_3]_c=0;
\\
x_{3}^2&=0; & & [x_{235},x_{3}]_c=0; & x_{5}^2&=0; & 
& x_{\alpha}^{N_\alpha}=0, \ \alpha\in\Oc_+^{\bq}.
\end{aligned}
\end{align}
Here {\scriptsize$\Oc_+^{\bq}=\{ 12, 23, 234, 4, 12^23^245, 1^22^33^34^25^2,
	1^22^33^345^2, 123^245, 12345, 1^22^23^245^2, 1235, \\ 1^22^33^44^25^3, 12^23^34^25^2, 1^22^43^64^35^4,
	12^23^345^2, 23^245, 1^22^33^54^35^4, 1^22^33^54^25^4, 12^23^44^25^3, \\ 123^245^2, 
	23^34^25^2, 345, 23^345^2, 3^245^2, 35 \}$}  
and the degree of the integral is
\begin{equation*}
\ya= 75\alpha_1 + 117\alpha_2 + 186\alpha_3 + 127\alpha_4 + 95\alpha_5.
\end{equation*}

\subsubsection{The associated Lie algebra} This is of type $F_4\times A_1$.

\subsection{Type $\gtt(4,6)$}\label{subsec:type-g(4,6)}
Here $\theta = 6$, $\zeta \in \G'_3$. Let 
\begin{align*}
A&=\begin{pmatrix} 2 & -1 & 0 & 0 & 0 & 0 \\ 1 & 0 & -1 & 0 & 0 & 0 \\ 0 & -1 & 0 & 1 & 0 & 1
\\ 0 & 0 & -1 & 2 & -1 & 0 \\ 0 & 0 & 0 & -1 & 2 & 0 \\ 0 & 0 & -1 & 0 & 0 & 2 \end{pmatrix}
\in \kk^{6\times 6}; & \pa &= (-1,1, 1, 1,1, 1) \in \G_2^6.
\end{align*}
Let $\g(4,6) = \g(A, \pa)$,
the contragredient Lie superalgebra corresponding to $(A, \pa)$. 
We know \cite{BGL} that $\sdim \g(4,6) = 66|32$. 
There are 6 other pairs of matrices and parity vectors for which the associated contragredient Lie superalgebra is isomorphic to $\g(4,6)$.
We describe now the root system $\gtt(4,6)$ of $\g(4,6)$, see \cite{AA-GRS-CLS-NA} for details.

\subsubsection{Basic datum and root system}
Below, $A_6$, $D_6$, $E_6$ and $_{2}T_1$ are numbered as in \eqref{eq:dynkin-system-A}, \eqref{eq:dynkin-system-D}, \eqref{eq:dynkin-system-E} and  	
\eqref{eq:mTn}, respectively.
The basic datum and the bundle of Cartan matrices are described by the following diagram:
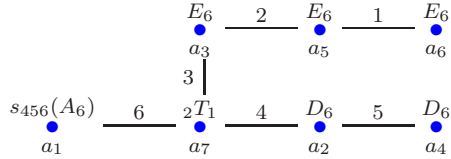
\begin{center}
	\begin{tabular}{c c c c c c c c c c c c}
		& &
		& $\overset{E_6}{\underset{a_3}{\vtxgpd}}$
		& \hspace{-7pt}\raisebox{3pt}{$\overset{2}{\rule{27pt}{0.5pt}}$}\hspace{-7pt}
		& $\overset{E_6}{\underset{a_5}{\vtxgpd}}$
		& \hspace{-7pt}\raisebox{3pt}{$\overset{1}{\rule{27pt}{0.5pt}}$}\hspace{-7pt}
		& $\overset{E_6}{\underset{a_6}{\vtxgpd}}$
		& & & &
		\\
		& & & {\scriptsize 3} \vline\hspace{5pt} & & & & & & & &
		\\
		& $\overset{s_{456}(A_6)}{\underset{a_1}{\vtxgpd}}$
		& \hspace{-7pt}\raisebox{3pt}{$\overset{6}{\rule{27pt}{0.5pt}}$}\hspace{-7pt}
		& $\overset{{}_2T_1}{\underset{a_7}{\vtxgpd}}$
		& \hspace{-7pt}\raisebox{3pt}{$\overset{4}{\rule{27pt}{0.5pt}}$}\hspace{-7pt}
		& $\overset{D_6}{\underset{a_2}{\vtxgpd}}$
		& \hspace{-7pt}\raisebox{3pt}{$\overset{5}{\rule{27pt}{0.5pt}}$}\hspace{-7pt}
		& $\overset{D_6}{\underset{a_4}{\vtxgpd}}$
		& & & &
	\end{tabular}
\end{center}
Using the notation \eqref{eq:notation-root-exceptional}, the bundle of root sets is the following: { \tiny
	\begin{align*}
	\varDelta_{+}^{a_1}= & s_{456}(\{ 1, 12, 2, 123, 23, 3, 1234, 234, 34, 4, 12^23^24^25, 123^24^25, 1234^25, 12345, 23^24^25, 234^25, 2345, \\
	& 34^25, 345, 45, 5, 12^23^34^45^36, 12^23^34^45^26, 12^23^34^35^26, 12^23^24^35^26, 123^24^35^26, 23^24^35^26, \\
	& 12^23^24^25^26, 123^24^25^26, 23^24^25^26, 1234^25^26, 234^25^26, 34^25^26, 12^23^34^45^36^2, 12^23^24^256,  \\
	& 123^24^256, 1234^256, 123456, 23^24^256, 234^256, 23456, 34^256, 3456, 456, 56, 6 \}), \\
	\varDelta_{+}^{a_2}= & \{ 1, 12, 2, 123, 23, 3, 1234, 234, 34, 4, 12345, 2345, 345, 45, 5, 12^23^34^356, 12^23^24^356, 12^23^24^256, \\
	& 12^23^24^26, 123^24^356, 123^24^256, 123^24^26, 23^24^356, 23^24^256, 23^24^26, 12^23^34^45^26^2, 12^23^34^456^2, \\
	& 1234^256, 234^256, 34^256, 12^23^34^356^2, 123456, 23456, 3456, 12^23^24^356^2, 123^24^356^2, 1234^26, \\
	& 12346, 23^24^356^2, 234^26, 2346, 456, 34^26, 346, 46, 6 \}, \\
	\varDelta_{+}^{a_3}= & \{ 1, 12, 2, 123, 23, 3, 1234, 234, 34, 4, 12345, 2345, 345, 45, 5, 12^23^34^256, 12^23^24^256, 12^23^2456, \\
	& 12^23^246, 123^34^256, 23^34^256, 123^24^256, 123^2456, 123^246, 12^23^44^35^26^2, 12^23^44^356^2, 23^24^256, \\
	& 3^24^256, 123456, 12346, 12^23^44^256^2, 12^23^34^256^2, 123^34^256^2, 1236, 23^2456, 3^2456, 23^246, 3^246, \\
	& 23^34^256^2, 23456, 2346, 236, 3456, 346, 36, 6 \}, \\
	\varDelta_{+}^{a_4}= & \{ 1, 12, 2, 123, 23, 3, 1234, 234, 34, 4, 12345, 2345, 345, 45, 5, 12^23^34^35^26, 12^23^24^35^26, 12^23^24^25^26, \\
	& 12^23^24^256, 123^24^35^26, 123^24^25^26, 123^24^256, 23^24^35^26, 23^24^25^26, 23^24^256, 12^23^34^45^36^2, \\
	& 1234^25^26, 234^25^26, 34^25^26, 12^23^34^45^26^2, 12^23^34^35^26^2, 1234^256, 234^256, 34^256, 12^23^24^35^26^2, \\
	& 123^24^35^26^2, 123456, 12346, 23^24^35^26^2, 23456, 2346, 3456, 346, 456, 46, 6 \}, \\
	\varDelta_{+}^{a_5}= & \{ 1, 12, 2, 123, 23, 3, 1234, 234, 34, 4, 12345, 2345, 345, 45, 5, 1^22^33^34^256, 1^22^23^34^256, 1^22^23^24^256, \\
	& 1^22^23^2456, 1^22^23^246, 12^23^34^256, 12^23^24^256, 12^23^2456, 12^23^246, 1^22^33^44^35^26^2, 1^22^33^44^356^2, \\
	& 123^24^256, 23^24^256, 1^22^33^44^256^2, 123^2456, 23^2456, 1^22^33^34^256^2, 123456, 23456, 1^22^23^34^256^2, \\
	& 12^23^34^256^2, 123^246, 12346, 1236, 3456, 23^246, 2346, 236, 346, 36, 6 \}, \\
	\varDelta_{+}^{a_6}= & \{ 1, 12, 2, 123, 23, 3, 1234, 234, 34, 4, 12345, 2345, 345, 45, 5, 12^33^34^256, 12^23^34^256, 12^23^24^256, \\
	& 12^23^2456, 12^23^246, 2^23^34^256, 2^23^24^256, 2^23^2456, 2^23^246, 12^33^44^35^26^2, 12^33^44^356^2, 123^24^256, \\
	& 23^24^256, 12^33^44^256^2, 123^2456, 23^2456, 12^33^34^256^2, 123456, 23456, 12^23^34^256^2, 123^246, 12346, \\
	& 1236, 2^23^34^256^2, 3456, 23^246, 2346, 236, 346, 36, 6 \}, \\
	\varDelta_{+}^{a_7}= & \{ 1, 12, 2, 123, 23, 3, 1234, 234, 34, 4, 12345, 2345, 345, 45, 5, 12^23^34^256, 12^23^24^256, 12^23^2456, \\
	& 12^23^246, 123^24^256, 123^2456, 123^246, 23^24^256, 23^2456, 23^246, 12^23^34^35^26^2, 12^23^34^356^2, 1234^256, \\
	& 234^256, 34^256, 12^23^34^256^2, 123456, 23456, 3456, 12^23^24^256^2, 123^24^256^2, 12346, 1236, \\
	& 23^24^256^2, 2346, 236, 456, 346, 36, 46, 6 \}.
	\end{align*}
}%

\subsubsection{Weyl groupoid}
\label{subsubsec:type-g46-Weyl}
The isotropy group  at $a_3 \in \cX$ is
\begin{align*}
\cW(a_3)= \langle \varsigma_1^{a_3}\varsigma_2 \varsigma_3 \varsigma_6 \varsigma_4 \varsigma_6 \varsigma_3 \varsigma_2 \varsigma_1,
\varsigma_2^{a_3},  \varsigma_3^{a_3}, \varsigma_4^{a_3},  \varsigma_5^{a_3},  \varsigma_6^{a_4} \rangle \simeq W(D_6).
\end{align*}

\subsubsection{Incarnation}
We set the matrices $(\bq^{(i)})_{i\in\I_{7}}$, from left to right and  from up to down:
\begin{align}\label{eq:dynkin-g(4,6)}
\begin{aligned}
& & &
\xymatrix@C-4pt{\overset{\zeta}{\underset{\ }{\circ}}\ar  @{-}[r]^{\ztu}  &
	\overset{\zeta}{\underset{\ }{\circ}} \ar  @{-}[r]^{\ztu}  & \overset{\zeta}{\underset{\
		}{\circ}}
	\ar  @{-}[r]^{\ztu}  & \overset{-1}{\underset{\ }{\circ}} \ar  @{-}[r]^{\ztu}  &
	\overset{\zeta}{\underset{\ }{\circ}}  \ar  @{-}[r]^{\ztu}  & \overset{\zeta}{\underset{\
		}{\circ}}}
\\
&\xymatrix@C-5pt@R-8pt{  &   & & \overset{\zeta}{\circ} \ar  @{-}[d]^{\ztu}  & \\
	\overset{\zeta}{\underset{\ }{\circ}} \ar  @{-}[r]^{\ztu}  & \overset{\zeta}{\underset{\
		}{\circ}} \ar  @{-}[r]^{\ztu}  & \overset{\zeta}{\underset{\ }{\circ}} \ar  @{-}[r]^{\ztu}
	& \overset{-1}{\underset{\ }{\circ}}
	\ar  @{-}[r]^{\zeta}  & \overset{-1}{\underset{\ }{\circ}}}
&& \xymatrix@C-5pt@R-8pt{  &   & \overset{\zeta}{\circ} \ar  @{-}[d]^{\ztu}  & &\\
	\overset{\zeta}{\underset{\ }{\circ}} \ar  @{-}[r]^{\ztu}  & \overset{-1}{\underset{\
		}{\circ}} \ar  @{-}[r]^{\zeta}  & \overset{-1}{\underset{\ }{\circ}} \ar  @{-}[r]^{\ztu}
	& \overset{\zeta}{\underset{\ }{\circ}}
	\ar  @{-}[r]^{\ztu}  & \overset{\zeta}{\underset{\ }{\circ}}}
\\
&\xymatrix@C-5pt@R-8pt{  &   & & \overset{\zeta}{\circ} \ar  @{-}[d]^{\ztu}  & \\
	\overset{\zeta}{\underset{\ }{\circ}} \ar  @{-}[r]^{\ztu}  & \overset{\zeta}{\underset{\
		}{\circ}} \ar  @{-}[r]^{\ztu}  & \overset{\zeta}{\underset{\ }{\circ}} \ar  @{-}[r]^{\ztu}
	& \overset{\zeta}{\underset{\ }{\circ}}
	\ar  @{-}[r]^{\ztu}  & \overset{-1}{\underset{\ }{\circ}}}
&& \xymatrix@C-5pt@R-8pt{  &   & \overset{\zeta}{\circ} \ar  @{-}[d]^{\ztu}  & &\\
	\overset{-1}{\underset{\ }{\circ}} \ar  @{-}[r]^{\zeta}  & \overset{-1}{\underset{\
		}{\circ}} \ar  @{-}[r]^{\ztu}  & \overset{\zeta}{\underset{\ }{\circ}} \ar  @{-}[r]^{\ztu}
	& \overset{\zeta}{\underset{\ }{\circ}}
	\ar  @{-}[r]^{\ztu}  & \overset{\zeta}{\underset{\ }{\circ}}}
\\
&\xymatrix@C-5pt@R-8pt{  &    & \overset{\zeta}{\circ} \ar  @{-}[d]^{\ztu} &  & \\
	\overset{-1}{\underset{\ }{\circ}} \ar  @{-}[r]^{\ztu}  & \overset{\zeta}{\underset{\
		}{\circ}} \ar  @{-}[r]^{\ztu}  & \overset{\zeta}{\underset{\ }{\circ}} \ar  @{-}[r]^{\ztu}
	& \overset{\zeta}{\underset{\ }{\circ}}
	\ar  @{-}[r]^{\ztu}  & \overset{\zeta}{\underset{\ }{\circ}}}
&& \xymatrix@C-5pt@R-8pt{  &   & \overset{-1}{\circ} \ar  @{-}[d]^{\zeta} \ar
	@{-}[dr]^{\zeta}  & &\\
	\overset{\zeta}{\underset{\ }{\circ}} \ar  @{-}[r]^{\ztu}  & \overset{\zeta}{\underset{\
		}{\circ}} \ar  @{-}[r]^{\ztu}  & \overset{-1}{\underset{\ }{\circ}} \ar  @{-}[r]^{\zeta}
	& \overset{-1}{\underset{\ }{\circ}}
	\ar  @{-}[r]^{\ztu}  & \overset{\zeta}{\underset{\ }{\circ}}}
\end{aligned}
\end{align}
Now, this is the incarnation:
\begin{align*}
& a_1\mapsto s_{456}(\bq^{1}); &
& a_i\mapsto \bq^{(i)}, \ i\in\I_{2,7}.
\end{align*}

\subsubsection{PBW-basis and dimension} \label{subsubsec:type-g46-PBW}
Notice that the roots in each $\varDelta_{+}^{a_i}$, $i\in\I_{7}$, are ordered from left to right, justifying the notation $\beta_1, \dots, \beta_{46}$.

The root vectors $x_{\beta_k}$ are described as in Remark \ref{rem:lyndon-word}.
Thus
\begin{align*}
\left\{ x_{\beta_{46}}^{n_{46}} \dots x_{\beta_2}^{n_{2}}  x_{\beta_1}^{n_{1}} \, | \, 0\le n_{k}<N_{\beta_k} \right\}.
\end{align*}
is a PBW-basis of $\toba_{\bq}$. Hence $\dim \toba_{\bq}=2^{16}3^{30}
$.

\subsubsection{The Dynkin diagram \emph{(\ref{eq:dynkin-g(4,6)}
		a)}}\label{subsubsec:g(4,6)-a}

\

The Nichols algebra $\toba_{\bq}$ is generated by $(x_i)_{i\in \I_6}$ with defining relations
\begin{align}\label{eq:rels-g(4,6)-a}
\begin{aligned}
x_{112}&=0; & x_{221}&=0; & & x_{ij}=0, \ i<j, \, \widetilde{q}_{ij}=1;
\\
x_{223}&=0; & x_{332}&=0; & & [[[x_{(25)},x_{4}]_c,x_3]_c,x_4]_c=0;
\\
x_{334}&=0; & x_{554}&=0; & & [[[x_{6543},x_{4}]_c,x_5]_c,x_4]_c=0;
\\
x_{556}&=0; & x_{665}&=0; & & x_{4}^2=0; \quad x_{\alpha}^{3}=0, \ \alpha\in\Oc_+^{\bq}.
\end{aligned}
\end{align}
Here {\scriptsize$\Oc_+^{\bq}=\{ 1, 12, 2, 123, 23, 3, 12^23^24^25, 123^24^25, 1234^25, 23^24^25, 234^25, 34^25, 5, 12^23^34^45^36$, \\ $12^23^34^45^26, 12^23^24^25^26, 123^24^25^26, 23^24^25^26, 1234^25^26, 234^25^26, 34^25^26, 12^23^34^45^36^2$, \\ $12^23^24^256, 123^24^256, 1234^256, 23^24^256, 234^256, 34^256, 56, 6 \}$}  
and the degree of the integral is
\begin{equation*}
\ya= 36\alpha_1 + 68\alpha_2 + 96\alpha_3 + 120\alpha_4 + 84\alpha_5 + 44\alpha_6.
\end{equation*}

\subsubsection{The Dynkin diagram \emph{(\ref{eq:dynkin-g(4,6)}
		b)}}\label{subsubsec:g(4,6)-b}

\

The Nichols algebra $\toba_{\bq}$ is generated by $(x_i)_{i\in \I_6}$ with defining relations
\begin{align}\label{eq:rels-g(4,6)-b}
\begin{aligned}
& [x_{546},x_{4}]_c=0; & x_{221}&=0; & x_{223}&=0; & & x_{ij}=0, \quad i<j, \widetilde{q}_{ij}=1;\\
& x_{332}=0; & x_{334}&=0; & x_{664}&=0; & & [[[x_{2346},x_{4}]_c,x_3]_c,x_4]_c=0;\\
& [x_{(35)},x_{4}]_c=0; & x_{112}&=0; & x_{5}^2&=0; & 
& x_{4}^2=0; \quad x_{\alpha}^{3}=0, \ \alpha\in\Oc_+^{\bq}.
\end{aligned}
\end{align}
Here {\scriptsize$\Oc_+^{\bq}=\{ 1, 12, 2, 123, 23, 3, 12345, 2345, 345, 45, 12^23^34^356, 12^23^24^356, 12^23^24^26, 123^24^356, \\ 123^24^26, 23^24^356, 23^24^26,
	12^23^44^45^26^2, 12^23^34^356^2, 123456, 23456, 3456, 12^23^24^356^2, 123^24^356^2,
	\\ 1234^26, 23^24^356^2, 234^26, 456, 34^26, 6 \}$}  
and the degree of the integral is
\begin{equation*}
\ya= 36\alpha_1 + 68\alpha_2 + 96\alpha_3 + 120\alpha_4 + 44\alpha_5 + 62\alpha_6.
\end{equation*}

\subsubsection{The Dynkin diagram \emph{(\ref{eq:dynkin-g(4,6)}
		c)}}\label{subsubsec:g(4,6)-c}

\

The Nichols algebra $\toba_{\bq}$ is generated by $(x_i)_{i\in \I_6}$ with defining relations
\begin{align}\label{eq:rels-g(4,6)-c}
\begin{aligned}
& [x_{(13)},x_{2}]_c=0; & x_{443}&=0; & x_{445}&=0; & & x_{ij}=0, \quad i<j, \widetilde{q}_{ij}=1;\\
& [x_{236},x_{3}]_c=0; & x_{554}&=0; & x_{663}&=0; & & [[[x_{5436},x_{3}]_c,x_4]_c,x_3]_c=0; \\
& [x_{(24)},x_{3}]_c=0; & x_{112}&=0; & x_{3}^2&=0; &   & x_{2}^2=0; \quad x_{\alpha}^{3}=0, \ \alpha\in\Oc_+^{\bq}.
\end{aligned}
\end{align}
Here {\scriptsize$\Oc_+^{\bq}=\{ 1, 123, 23, 1234, 234,
	4, 12345, 2345, 45, 5, 12^23^24^256, 12^23^2456, 12^23^246, 123^34^256$, \\ $23^34^256, 12^23^44^35^26^2, 12^23^44^356^2,
	3^24^256, 123456, 12346, 12^23^44^256^2, 123^34^256^2, 1236, 3^2456$, \\ $3^246, 23^34^256^2, 23456, 2346, 236, 6 \}$}  
and the degree of the integral is
\begin{equation*}
\ya= 36\alpha_1 + 68\alpha_2 + 120\alpha_3 + 84\alpha_4 + 44\alpha_5 + 62\alpha_6.
\end{equation*}

\subsubsection{The Dynkin diagram \emph{(\ref{eq:dynkin-g(4,6)}
		d)}}\label{subsubsec:g(4,6)-d}

\

The Nichols algebra $\toba_{\bq}$ is generated by $(x_i)_{i\in \I_6}$ with defining relations
\begin{align}\label{eq:rels-g(4,6)-d}
\begin{aligned}
x_{112}&=0; & x_{221}&=0; & x_{223}&=0; & &x_{ij}=0, \ i<j, \, \widetilde{q}_{ij}=1;\\
x_{332}&=0; & x_{334}&=0; & x_{443}&=0; & & x_{445}=0;\\
x_{446}&=0; & x_{664}&=0; & x_{5}^2&=0; & & x_{\alpha}^{3}=0, \ \alpha\in\Oc_+^{\bq}.
\end{aligned}
\end{align}
Here {\scriptsize$\Oc_+^{\bq}=\{ 1, 12, 2, 123, 23,
	3, 1234, 234, 34, 4, 12^23^34^35^26, 12^23^24^35^26, 12^23^24^25^26, 123^24^35^26$, \\ $123^24^25^26, 23^24^35^26, 23^24^25^26,
	1234^25^26, 234^25^26, 34^25^26, 12^23^44^45^26^2, 12^23^34^35^26^2$, \\ $12^23^24^35^26^2, 123^24^35^26^2,
	12346, 23^24^35^26^2, 2346, 346, 46, 6 \}$}  
and the degree of the integral is
\begin{equation*}
\ya= 36\alpha_1 + 68\alpha_2 + 96\alpha_3 + 120\alpha_4 + 78\alpha_5 + 62\alpha_6.
\end{equation*}

\subsubsection{The Dynkin diagram \emph{(\ref{eq:dynkin-g(4,6)}
		e)}}\label{subsubsec:g(4,6)-e}

\

The Nichols algebra $\toba_{\bq}$ is generated by $(x_i)_{i\in \I_6}$ with defining relations
\begin{align}\label{eq:rels-g(4,6)-e}
\begin{aligned}
x_{332}&=0; & x_{334}&=0; & x_{336}&=0; & & x_{ij}=0, \ i<j, \, \widetilde{q}_{ij}=1;\\
x_{663}&=0; & x_{443}&=0; & x_{445}&=0; & & [x_{(13)},x_2]_c=0;\\
x_{554}&=0; & x_{1}^2&=0; & x_{2}^2&=0; & & x_{\alpha}^{3}=0, \ \alpha\in\Oc_+^{\bq}.
\end{aligned}
\end{align}
Here {\scriptsize$\Oc_+^{\bq}=\{ 12, 123, 3, 1234, 34, 4, 12345, 345, 45, 5, 12^33^34^256, 2^23^34^256, 2^23^24^256, 2^23^2456, \\ 2^23^246, 12^33^44^35^26^2, 12^33^44^356^2,
	123^24^256, 12^33^44^256^2, 123^2456, 12^33^34^256^2, 123456, 123^246,\\
	12346, 1236, 2^23^34^256^2, 3456, 346, 36, 6 \}$}  
and the degree of the integral is
\begin{equation*}
\ya= 36\alpha_1 + 90\alpha_2 + 120\alpha_3 + 84\alpha_4 + 44\alpha_5 + 62\alpha_6.
\end{equation*}

\subsubsection{The Dynkin diagram \emph{(\ref{eq:dynkin-g(4,6)}
		f)}}\label{subsubsec:g(4,6)-f}

\

The Nichols algebra $\toba_{\bq}$ is generated by $(x_i)_{i\in \I_6}$ with defining relations
\begin{align}\label{eq:rels-g(4,6)-f}
\begin{aligned}
x_{221}&=0; & x_{223}&=0; & x_{332}&=0; & & x_{ij}=0, \ i<j, \, \widetilde{q}_{ij}=1;\\
x_{334}&=0; & x_{443}&=0; & x_{445}&=0; & & x_{336}=0;\\
x_{663}&=0; & x_{1}^2&=0; & x_{554}&=0; & & x_{\alpha}^{3}=0, \ \alpha\in\Oc_+^{\bq}.
\end{aligned}
\end{align}
Here {\scriptsize$\Oc_+^{\bq}=\{ 2, 23, 3, 234, 34,
	4, 2345, 345, 45, 5, 1^22^33^34^256, 1^22^23^34^256, 1^22^23^24^256, \\ 
	1^22^23^2456, 1^22^23^246, 1^22^33^44^35^26^2, 1^22^33^44^356^2,
	23^24^256, 1^22^33^44^256^2, 23^2456, 1^22^33^34^256^2, \\ 23456, 1^22^23^34^256^2, 3456,
	23^246, 2346, 236, 346, 36, 6 \}$}  
and the degree of the integral is
\begin{equation*}
\ya= 56\alpha_1 + 90\alpha_2 + 120\alpha_3 + 84\alpha_4 + 44\alpha_5 + 62\alpha_6.
\end{equation*}

\subsubsection{The Dynkin diagram \emph{(\ref{eq:dynkin-g(4,6)}
		g)}}\label{subsubsec:g(4,6)-g}

\

The Nichols algebra $\toba_{\bq}$ is generated by $(x_i)_{i\in \I_6}$ with defining relations
\begin{align}\label{eq:rels-g(4,6)-g}
\begin{aligned}
& \begin{aligned}
x_{112}&=0; & [x_{(24)},&x_3]_c=0; & x_{554}&=0; & & x_{ij}=0, \ i<j, \, \widetilde{q}_{ij}=1;
\\
x_{221}&=0; & [x_{236},&x_3]_c=0; & x_{3}^2&=0; & & [x_{546},x_4]_c=0;
\\
x_{223}&=0; & [x_{(35)},&x_4]_c=0; & x_{4}^2&=0; & & x_{\alpha}^{3}=0, \ \alpha\in\Oc_+^{\bq};
\end{aligned}
\\
& x_{6}^2=0; \quad x_{346}=q_{46}\ztu[x_{36},x_4]_c +q_{34}(1-\zeta)x_4x_{36}=0.
\end{aligned}
\end{align}
Here {\scriptsize$\Oc_+^{\bq}=\{ 1, 12, 2, 1234, 234,
	34, 12345, 2345, 345, 5, 12^23^34^256, 12^23^2456, 12^23^246, 123^2456, \\ 123^246, 23^2456, 23^246,
	12^23^34^35^26^2, 12^23^34^356^2, 1234^256, 234^256, 34^256, 12^23^24^256^2, 123^24^256^2,
	\\ 1236, 23^24^256^2, 236, 456, 36, 46 \}$}  
and the degree of the integral is
\begin{equation*}
\ya= 36\alpha_1 + 68\alpha_2 + 96\alpha_3 + 84\alpha_4 + 44\alpha_5 + 62\alpha_6.
\end{equation*}

\subsubsection{The associated Lie algebra} This is of type $D_6$.

\subsection{Type $\gtt(6,6)$}\label{subsec:type-g(6,6)}
Here $\theta = 6$, $\zeta \in \G'_3$. Let 
\begin{align*}
A&=\begin{pmatrix} 2 & -1 & 0 & 0 & 0 & 0 \\ -1 & 2 & -1 & 0 & 0 & 0 \\ 0 & -1 & 2 & -1 & 0 & 0
\\ 0 & 0 & -1 & 2 & 0 & -1 \\ 0 & 0 & 0 & 0 & 0 & 1 \\ 0 & 0 & 0 & -2 & -1 & 2 \end{pmatrix}
\in \kk^{6\times 6}; & \pa &= (1,1, 1,1, -1, 1) \in \G_2^6.
\end{align*}
Let $\g(6,6) = \g(A, \pa)$,
the contragredient Lie superalgebra corresponding to $(A, \pa)$. 
We know \cite{BGL} that $\sdim \g(6,6) = 78|64$. 
There are 20 other pairs of matrices and parity vectors for which the associated contragredient Lie superalgebra is isomorphic to $\g(6,6)$.
We describe now the root system $\gtt(6,6)$ of $\g(6,6)$, see \cite{AA-GRS-CLS-NA} for details.

\subsubsection{Basic datum and root system}
Below, $A_6$, $D_6$, $E_6$, $E_6^{(2)\wedge}$, $F_4^{(1)\wedge}$, $CE_6$ $_{3}T$ and $_{2}T_1$ are numbered as in \eqref{eq:dynkin-system-A}, \eqref{eq:dynkin-system-D}, \eqref{eq:dynkin-system-E}, \eqref{eq:E6(2)w}, \eqref{eq:F4(1)w}, \eqref{eq:CEn} and \eqref{eq:mTn}, respectively.
The basic datum and the bundle of Cartan matrices are described by the following diagram:

\begin{center}
	\begin{tabular}{c c}
		&
		$\xymatrix{&  &
			\overset{s_{56}(F_4^{(1)\wedge})}{\underset{a_1}{\vtxgpd}} \ar@{-}[r]^5 &
			\overset{s_{56}(A_6)}{\underset{a_2}{\vtxgpd}} \ar@{-}[r]^6 &
			\overset{{}_3T}{\underset{a_7}{\vtxgpd}} \ar@{-}[d]^4
			\\
			&
			\overset{D_6}{\underset{a_{16}}{\vtxgpd}} \ar@{-}[r]^1 \ar@{-}[d]^5 &
			\overset{D_6}{\underset{a_{14}}{\vtxgpd}} \ar@{-}[r]^2 \ar@{-}[d]^5 &
			\overset{D_6}{\underset{a_{12}}{\vtxgpd}} \ar@{-}[r]^3 \ar@{-}[d]^5 &
			\overset{D_6}{\underset{a_{11}}{\vtxgpd}} \ar@{-}[d]^5
			\\
			\overset{E_6}{\underset{a_{20}}{\vtxgpd}} \ar@{-}[rrd]^(.6){3}|!{[d];[rrd]}\hole \ar@{-}[d]^1 &
			\overset{D_6}{\underset{a_{19}}{\vtxgpd}} \ar@{-}[r]^1 \ar@{-}[d]^(.6){4} &
			\overset{D_6}{\underset{a_{17}}{\vtxgpd}} \ar@{-}[r]^2 \ar@{-}[d]^4 &
			\overset{D_6}{\underset{a_{15}}{\vtxgpd}} \ar@{-}[r]^3 \ar@{-}[d]^4 &
			\overset{CE_6}{\underset{a_{13}}{\vtxgpd}}
			\\
			\overset{E_6}{\underset{a_{21}}{\vtxgpd}} \ar@{-}[r]^3 \ar@{-}[d]^2 &
			\overset{{}_2T_1}{\underset{a_{8}}{\vtxgpd}} \ar@{-}[r]^1 \ar@{-}[d]^6 &
			\overset{{}_2T_1}{\underset{a_9}{\vtxgpd}} \ar@{-}[r]^2 \ar@{-}[d]^6 &
			\overset{{}_2T_1}{\underset{a_{10}}{\vtxgpd}} \ar@{-}[d]^6 &
			\\
			\overset{E_6}{\underset{a_{18}}{\vtxgpd}} &
			\overset{s_{456}(A_6)}{\underset{a_5}{\vtxgpd}} \ar@{-}[r]^1 &
			\overset{s_{456}(A_6)}{\underset{a_4}{\vtxgpd}} \ar@{-}[r]^2 &
			\overset{s_{456}(A_6)}{\underset{a_3}{\vtxgpd}} \ar@{-}[r]^3 &
			\overset{s_{456}(E_6^{(2)\wedge})}{\underset{a_6}{\vtxgpd}} }$
	\end{tabular}
\end{center}
Using the notation \eqref{eq:notation-root-exceptional}, the bundle of root sets is the following: { \scriptsize
	\begin{align*}
	\varDelta_{+}^{a_{1}}= & s_{56}(\{ 1, 12, 2, 123, 23, 3, 1234, 234, 34, 4, 12345, 12^23^24^25^2, 123^24^25^2, 1234^25^2, 12345^2, 2345, \\
	& 23^24^25^2, 234^25^2, 2345^2, 345, 34^25^2, 345^2, 45, 45^2, 5, 12^23^34^45^56, 12^23^34^45^46, 12^23^34^35^46, \\
	& 12^23^24^35^46, 123^24^35^46, 23^24^35^46, 12^23^34^35^36, 12^23^24^35^36, 123^24^35^36, 23^24^35^36, \\
	& 1^22^33^44^55^56^2, 12^33^44^55^56^2, 12^23^24^25^36, 12^23^24^25^26, 12^23^44^55^56^2, 123^24^25^36, \\
	& 123^24^25^26, 23^24^25^36, 23^24^25^26, 12^23^34^55^56^2, 12^23^34^45^66^2, 12^23^34^45^56^2, 1234^25^36, \\
	& 234^25^36, 34^25^36, 12^23^34^45^46^2, 1234^25^26, 234^25^26, 34^25^26, 12^23^34^35^46^2, 12^23^24^35^46^2, \\
	& 123^24^35^46^2, 12345^26, 123456, 23^24^35^46^2, 2345^26, 23456, 345^26, 3456, 45^26, 456, 56, 6 \}), \\
	\varDelta_{+}^{a_{2}}= & s_{56}(\{ 1, 12, 2, 123, 23, 3, 1234, 234, 34, 4, 12345, 2345, 345, 45, 5, 12^23^24^25^26, 123^24^25^26, \\
	& 1234^25^26, 12345^26, 123456, 12^23^34^45^46^2, 12^23^34^35^46^2, 12^23^34^35^36^2, 23^24^25^26, \\
	& 12^23^24^35^46^2, 12^23^24^35^36^2, 234^25^26, 12^23^24^25^36^2, 2345^26, 12^23^24^25^26^2, 23456, \\
	& 1^22^33^44^55^66^4, 12^33^44^55^66^4, 12^23^34^45^56^3, 12^23^34^45^46^3, 12^23^34^35^46^3, 12^23^24^35^46^3, \\
	& 123^24^35^46^2, 123^24^35^36^2, 123^24^25^36^2, 1234^25^36^2, 123^24^25^26^2, 1234^25^26^2, \\
	& 12345^26^2, 12^23^44^55^66^4, 12^23^34^55^66^4, 12^23^34^45^66^4, 12^23^34^45^56^4, 123^24^35^46^3, \\
	& 23^24^35^46^2, 23^24^35^36^2, 34^25^26, 23^24^25^36^2, 345^26, 23^24^25^26^2, 3456, 23^24^35^46^3, \\
	& 234^25^36^2, 234^25^26^2, 2345^26^2, 34^25^36^2, 34^25^26^2, 345^26^2, 45^26, 456, 45^26^2, 56, 6 \}), \\
	\varDelta_{+}^{a_{3}}= & s_{456}(\{ 1, 12, 2, 123, 23, 3, 1234, 234, 34, 4, 12^23^24^25, 123^24^25, 1234^25, 12345, 23^24^25, 234^25, \\
	& 2345, 3^24^25, 34^25, 345, 45, 5, 12^23^44^55^36, 12^23^44^45^36, 12^23^34^45^36, 123^34^45^36, 23^34^45^36, \\
	& 12^23^44^45^26, 12^23^34^45^26, 123^34^45^26, 23^34^45^26, 12^23^34^35^26, 12^23^24^35^26, 1^22^33^54^65^46^2, \\
	& 12^33^54^65^46^2, 12^23^24^25^26, 12^23^24^256, 123^34^35^26, 23^34^35^26, 12^23^54^65^46^2, 123^24^35^26, \\
	& 123^24^25^26, 123^24^256, 12^23^44^65^46^2, 12^23^44^55^46^2, 12^23^44^55^36^2, 23^24^35^26, 3^24^35^26, \\
	& 12^23^44^45^36^2, 23^24^25^26, 3^24^25^26, 23^24^256, 3^24^256, 12^23^34^45^36^2, 123^34^45^36^2, 1234^25^26, \\
	& 1234^256, 123456, 23^34^45^36^2, 234^25^26, 234^256, 23456, 34^25^26, 34^256, 3456, 456, 56, 6 \}), \\
	\varDelta_{+}^{a_{4}}= & s_{456}(\{ 1, 12, 2, 123, 23, 3, 1234, 234, 34, 4, 1^22^23^24^25, 12^23^24^25, 123^24^25, 1234^25, 12345, 23^24^25, \\
	& 234^25, 2345, 34^25, 345, 45, 5, 1^22^33^44^55^36, 1^22^33^44^45^36, 1^22^33^34^45^36, 1^22^23^34^45^36, \\
	& 12^23^34^45^36, 1^22^33^44^45^26, 1^22^33^34^45^26, 1^22^23^34^45^26, 12^23^34^45^26, 1^22^33^34^35^26, \\
	& 1^22^23^34^35^26, 12^23^34^35^26, 1^32^43^54^65^46^2, 1^22^23^24^35^26, 1^22^23^24^25^26, 1^22^23^24^256, \\
	& 1^22^43^54^65^46^2, 12^23^24^35^26, 12^23^24^25^26, 12^23^24^256, 1^22^33^54^65^46^2, 1^22^33^44^65^46^2, \\
	& 1^22^33^44^55^46^2, 1^22^33^44^55^36^2, 123^24^35^26, 23^24^35^26, 1^22^33^44^45^36^2, 123^24^25^26, 23^24^25^26, \\
	& 123^24^256, 23^24^256, 1^22^33^34^45^36^2, 1^22^23^34^45^36^2, 12^23^34^45^36^2, 1234^25^26, 1234^256, \\
	& 123456, 234^25^26, 234^256, 23456, 34^25^26, 34^256, 3456, 456, 56, 6 \}), \\
	\varDelta_{+}^{a_{5}}= & s_{456}(\{ 1, 12, 2, 123, 23, 3, 1234, 234, 34, 4, 12^23^24^25, 123^24^25, 1234^25, 12345, 2^23^24^25, 23^24^25, \\
	& 234^25, 2345, 34^25, 345, 45, 5, 12^33^44^55^36, 12^33^44^45^36, 12^33^34^45^36, 12^23^34^45^36, 2^23^34^45^36, \\
	& 12^33^44^45^26, 12^33^34^45^26, 12^23^34^45^26, 2^23^34^45^26, 12^33^34^35^26, 12^23^34^35^26, 2^23^34^35^26, \\
	& 1^22^43^54^65^46^2, 12^43^54^65^46^2, 12^23^24^35^26, 12^23^24^25^26, 12^23^24^256, 2^23^24^35^26, 2^23^24^25^26, \\
	& 2^23^24^256, 12^33^54^65^46^2, 12^33^44^65^46^2, 12^33^44^55^46^2, 12^33^44^55^36^2, 123^24^35^26, 23^24^35^26, \\
	& 12^33^44^45^36^2, 123^24^25^26, 23^24^25^26, 123^24^256, 23^24^256, 12^33^34^45^36^2, 12^23^34^45^36^2, 1234^25^26, \\
	& 1234^256, 123456, 2^23^34^45^36^2, 234^25^26, 234^256, 23456, 34^25^26, 34^256, 3456, 456, 56, 6 \}), \\
	\varDelta_{+}^{a_{6}}= & s_{456}(\{ 1, 12, 2, 123, 23, 3, 1234, 234, 34, 4, 12^23^24^25, 123^24^25, 1234^25, 12345, 23^24^25, 234^25, 2345, \\
	& 34^25, 345, 4^25, 45, 5, 12^23^34^55^36, 12^23^34^45^36, 12^23^24^45^36, 123^24^45^36, 23^24^45^36, 12^23^34^45^26, \\
	& 12^23^24^45^26, 123^24^45^26, 23^24^45^26, 12^23^34^35^26, 12^23^24^35^26, 1^22^33^44^65^46^2, 12^33^44^65^46^2, \\
	& 12^23^24^25^26, 12^23^24^256, 123^24^35^26, 12^23^44^65^46^2, 123^24^25^26, 123^24^256, 23^24^35^26, 23^24^25^26, \\
	& 23^24^256, 12^23^34^65^46^2, 12^23^34^55^46^2, 12^23^34^55^36^2, 1234^35^26, 234^35^26, 34^35^26, 12^23^34^45^36^2, \\
	& 1234^25^26, 234^25^26, 34^25^26, 12^23^24^45^36^2, 123^24^45^36^2, 1234^256, 123456, 23^24^45^36^2, 234^256, \\
	& 23456, 34^256, 3456, 4^25^26, 4^256, 456, 56, 6 \}), \\
	\varDelta_{+}^{a_{7}}= & \{ 1, 12, 2, 123, 23, 3, 1234, 234, 34, 4, 12345, 2345, 345, 45, 5, 12^23^34^35^26, 12^23^24^35^26, \\
	& 12^23^24^25^26, 12^23^24^256, 123^24^35^26, 123^24^25^26, 123^24^256, 23^24^35^26, 23^24^25^26, 23^24^256, \\
	& 12^23^34^45^46^2, 12^23^34^45^36^2, 1234^25^26, 234^25^26, 34^25^26, 12^23^34^35^36^2, 12345^26, 2345^26, 345^26, \\
	& 1^22^33^44^55^46^3, 12^33^44^55^46^3, 12^23^44^55^46^3, 12^23^34^45^26^2, 12^23^34^35^26^2, 12^23^24^35^36^2, \\
	& 123^24^35^36^2, 1234^256, 123456, 12^23^34^55^46^3, 12^23^34^45^46^3, 23^24^35^36^2, 45^26, 12^23^24^35^26^2, \\
	& 234^256, 123^24^35^26^2, 34^256, 12^23^34^45^36^3, 12^23^24^25^26^2, 123^24^25^26^2, 1234^25^26^2, 12346, \\
	& 23^24^35^26^2, 23456, 23^24^25^26^2, 234^25^26^2, 2346, 3456, 34^25^26^2, 346, 456, 46, 56, 6 \}, \\
	\varDelta_{+}^{a_{8}}= & \{ 1, 12, 2, 123, 23, 3, 1234, 234, 34, 4, 12345, 2345, 345, 45, 5, 12^33^34^256, 12^23^34^256, 12^23^24^256, \\
	& 12^23^2456, 12^23^246, 2^23^34^256, 2^23^24^256, 2^23^2456, 2^23^246, 12^33^44^45^26^2, 12^33^44^35^26^2, \\
	& 12^33^44^356^2, 123^24^256, 23^24^256, 12^33^34^35^26^2, 12^33^34^356^2, 1234^256, 234^256, 1^22^43^54^45^26^3, \\
	& 12^43^54^45^26^3, 12^33^44^256^2, 12^33^34^256^2, 12^23^34^35^26^2, 123^2456, 123456, 12^23^34^356^2, 123^246, \\
	& 12346, 12^33^54^45^26^3, 12^33^44^45^26^3, 2^23^34^35^26^2, 2^23^34^356^2, 34^256, 12^23^34^256^2, 23^2456, \\
	& 23^246, 12^33^44^35^26^3, 12^33^44^356^2, 12^23^24^256^2, 123^24^256^2, 1236, 2^23^34^256^2, 23456, 3456, \\
	& 2^23^24^256^2, 23^24^256^2, 2346, 236, 456, 346, 36, 46, 6 \}, \\
	\varDelta_{+}^{a_{9}}= & \{ 1, 12, 2, 123, 23, 3, 1234, 234, 34, 4, 12345, 2345, 345, 45, 5, 1^22^33^34^256, 1^22^23^34^256, \\
	& 1^22^23^24^256, 1^22^23^2456, 1^22^23^246, 12^23^34^256, 12^23^24^256, 12^23^2456, 12^23^246, 1^22^33^44^45^26^2, \\
	& 1^22^33^44^35^26^2, 1^22^33^44^356^2, 123^24^256, 23^24^256, 1^22^33^34^35^26^2, 1^22^33^34^356^2, 1234^256, \\
	& 234^256, 1^32^43^54^45^26^3, 1^22^43^54^45^26^3, 1^22^33^44^256^2, 1^22^33^34^256^2, 1^22^23^34^35^26^2, \\
	& 1^22^23^34^356^2, 12^23^34^35^26^2, 123^2456, 123456, 1^22^33^54^45^26^3, 1^22^33^44^45^26^3, 12^23^34^356^2, \\
	& 34^256, 23^2456, 23456, 1^22^33^44^35^26^3, 1^22^23^34^256^2, 12^23^34^256^2, 3456, 1^22^23^24^256^2, 12^23^24^256^2, \\
	& 456, 1^22^33^44^356^3, 123^24^256^2, 123^246, 12346, 1236, 23^24^256^2, 23^246, 2346, 236, 346, 36, 46, 6 \}, \\
	\varDelta_{+}^{a_{10}}= & \{ 1, 12, 2, 123, 23, 3, 1234, 234, 34, 4, 12345, 2345, 345, 45, 5, 12^23^34^256, 12^23^24^256, 12^23^2456, \\
	& 12^23^246, 123^34^256, 23^34^256, 123^24^256, 123^2456, 123^246, 12^23^44^45^26^2, 1234^256, 12^23^44^35^26^2, \\
	& 12^23^44^356^2, 23^24^256, 3^24^256, 1^22^33^54^45^26^3, 12^23^34^35^26^2, 123^34^35^26^2, 123456, 12^23^44^256^2, \\
	& 23^2456, 3^2456, 12^33^54^45^26^3, 12^23^54^45^26^3, 12^23^34^356^2, 12^23^34^256^2, 23^246, 23^34^35^26^2, 234^256, \\
	& 23456, 12^23^44^45^26^3, 123^34^356^2, 23^34^356^2, 34^256, 12^23^44^35^26^3, 12^23^24^256^2, 123^34^256^2, 23^34^256^2, \\
	& 3456, 12^23^44^356^3, 123^24^256^2, 12346, 1236, 3^246, 23^24^256^2, 2346, 236, 3^24^256^2, 456, 346, 36, 46, 6 \}, \\
	\varDelta_{+}^{a_{11}}= & \{ 1, 12, 2, 123, 23, 3, 1234, 234, 34, 4, 12345, 2345, 345, 45, 5, 12^23^34^35^26, 12^23^24^35^26, 12^23^24^25^26, \\
	& 12^23^24^256, 123^24^35^26, 123^24^25^26, 123^24^256, 23^24^35^26, 23^24^25^26, 23^24^256, 12^23^34^55^46^2, \\
	& 12^23^34^55^36^2, 1234^35^26, 234^35^26, 34^35^26, 12^23^34^45^36^2, 1234^25^26, 234^25^26, 34^25^26, \\
	& 1^22^33^44^65^46^3, 12^23^44^65^46^3, 12^23^44^65^46^3, 12^23^34^45^26^2, 12^23^34^35^26^2, 12^23^24^45^36^2, 123^24^45^36^2, \\
	& 1234^256, 123456, 12^23^34^65^46^3, 12^23^34^55^46^3, 23^24^45^36^2, 4^25^26, 12^23^24^45^26^2, 123^24^45^26^2, \\
	& 12^23^24^35^26^2, 234^256, 12^23^34^55^36^3, 123^24^35^26^2, 1234^35^26^2, 12346, 23456, 23^24^45^26^2, 23^24^35^26^2, \\
	& 234^35^26^2, 2346, 34^256, 4^256, 34^35^26^2, 3456, 346, 456, 46, 6 \}, \\
	\varDelta_{+}^{a_{12}}= & \{ 1, 12, 2, 123, 23, 3, 1234, 234, 34, 4, 12345, 2345, 345, 45, 5, 12^23^34^35^26, 12^23^24^35^26, 12^23^24^25^26, \\
	& 12^23^24^256, 123^34^35^26, 23^34^35^26, 123^24^35^26, 123^24^25^26, 123^24^256, 12^23^44^55^46^2, 12^23^44^55^36^2, \\
	& 23^24^35^26, 3^24^35^26, 1234^25^26, 1234^256, 1^22^33^54^65^46^3, 12^23^44^45^36^2, 12^23^34^45^36^2, 123^34^45^36^2, \\
	& 123456, 12^23^44^45^26^2, 12^23^34^45^26^2, 123^34^45^26^2, 12346, 12^33^54^65^46^3, 12^23^54^65^46^3, \\
	& 12^23^34^35^26^2, 23^24^25^26, 23^24^256, 12^23^44^65^46^3, 12^23^24^35^26^2, 234^25^26, 234^256, 12^23^44^55^46^3, \\
	& 12^23^44^55^36^3, 123^34^35^26^2, 123^24^35^26^2, 23^34^45^36^2, 3^24^25^26, 34^25^26, 23456, 23^34^45^26^2, \\
	& 23^34^35^26^2, 23^24^35^26^2, 2346, 3^24^256, 34^256, 3^24^35^26^2, 3456, 346, 456, 46, 6 \}, \\
	\varDelta_{+}^{a_{13}}= & \{ 1, 12, 2, 123, 23, 3, 1234, 234, 34, 4, 12345, 2345, 345, 45, 5, 12^23^34^356, 12^23^24^356, 12^23^24^256, \\
	& 12^23^24^26, 123^24^356, 123^24^256, 123^24^26, 23^24^356, 23^24^256, 23^24^26, 12^23^34^55^26^2, 12^23^34^556^2, \\
	& 1234^356, 234^356, 34^356, 12^23^34^45^26^2, 12^23^34^456^2, 1^22^33^44^65^26^3, 12^33^44^65^26^3, 12^23^44^65^26^3, \\
	& 12^23^34^356^2, 1234^256, 12^23^24^45^26^2, 123^24^45^26^2, 123456, 1234^26, 12^23^24^456^2, 123^24^456^2, 12346, \\
	& 12^23^34^65^26^3, 12^23^24^356^2, 234^256, 234^26, 12^23^34^55^26^3, 12^23^34^556^3, 123^24^356^2, 1234^356^2, 23^24^45^26^2, \\
	& 23456, 23^24^456^2, 34^256, 4^256, 23^24^356^2, 234^356^2, 2346, 34^26, 4^26, 34^356^2, 3456, 346, 456, 46, 6 \}, \\
	\varDelta_{+}^{a_{14}}= & \{ 1, 12, 2, 123, 23, 3, 1234, 234, 34, 4, 12345, 2345, 345, 45, 5, 1^22^33^34^35^26, 1^22^23^34^35^26, \\
	& 1^22^23^24^35^26, 1^22^23^24^25^26, 1^22^23^24^256, 12^23^34^35^26, 12^23^24^35^26, 12^23^24^25^26, 12^23^24^256, \\
	& 1^22^33^44^55^46^2, 1^22^33^44^55^36^2, 123^24^35^26, 23^24^35^26, 1^22^33^44^45^36^2, 123^24^25^26, 23^24^25^26, \\
	& 1^22^33^34^45^36^2, 1234^25^26, 234^25^26, 1^32^43^54^65^46^3, 1^22^43^54^65^46^3, 1^22^33^44^45^26^2, 1^22^33^34^45^26^2, \\
	& 1^22^33^34^35^26^2, 1^22^23^34^45^36^2, 12^23^34^45^36^2, 1^22^33^54^65^46^3, 1^22^33^44^65^46^3, 1^22^33^44^55^46^3, \\
	& 34^25^26, 123^24^256, 1^22^23^34^45^26^2, 1^22^23^34^35^26^2, 23^24^256, 12^23^34^45^26^2, 12^23^34^35^26^2, \\
	& 1^22^33^44^55^36^3, 1234^256, 234^256, 34^256, 1^22^23^24^35^26^2, 12^23^24^35^26^2, 123^24^35^26^2, 123456, \\
	& 12346, 23^24^35^26^2, 23456, 2346, 3456, 346, 456, 46, 6 \}, \\
	\varDelta_{+}^{a_{15}}= & \{ 1, 12, 2, 123, 23, 3, 1234, 234, 34, 4, 12345, 2345, 345, 45, 5, 12^23^34^356, 12^23^24^356, 12^23^24^256, \\
	& 12^23^24^26, 123^34^356, 23^34^356, 123^24^356, 123^24^256, 123^24^26, 12^23^44^55^26^2, 12^23^44^556^2, 23^24^356, \\
	& 3^24^356, 1234^256, 1234^26, 1^22^33^54^65^26^3, 12^23^44^45^26^2, 12^23^34^45^26^2, 123^34^45^26^2, 123456, \\
	& 12^23^44^456^2, 12^23^34^456^2, 123^34^456^2, 12346, 12^33^54^65^26^3, 12^23^54^65^26^3, 12^23^34^356^2, \\
	& 23^24^256, 23^24^26, 12^23^44^65^26^3, 12^23^24^356^2, 234^256, 234^26, 12^23^44^55^26^3, 12^23^44^556^3, \\
	& 123^34^356^2, 123^24^356^2, 23^34^45^26^2, 23456, 23^34^456^2, 3^24^256, 34^256, 23^34^356^2, 23^24^356^2, \\
	& 2346, 3^24^26, 34^26, 3^24^356^2, 3456, 346, 456, 46, 6 \}, \\
	\varDelta_{+}^{a_{16}}= & \{ 1, 12, 2, 123, 23, 3, 1234, 234, 34, 4, 12345, 2345, 345, 45, 5, 12^33^34^35^26, 12^23^34^35^26, 12^23^24^35^26, \\
	& 12^23^24^25^26, 12^23^24^256, 2^23^34^35^26, 2^23^24^35^26, 2^23^24^25^26, 2^23^24^256, 12^33^44^55^46^2, 12^33^44^55^36^2, \\
	& 123^24^35^26, 23^24^35^26, 12^33^44^45^36^2, 123^24^25^26, 23^24^25^26, 12^33^34^45^36^2, 1234^25^26, 234^25^26, \\
	& 1^22^43^54^65^46^3, 12^43^54^65^46^3, 12^33^44^45^26^2, 12^33^34^45^26^2, 12^33^34^35^26^2, 12^23^34^45^36^2, 123^24^256, \\
	& 1234^256, 123456, 12^33^54^65^46^3, 12^23^44^65^46^3, 12^33^44^55^46^3, 2^23^34^45^36^2, 34^25^26, 12^23^34^45^26^2, \\
	& 12^23^34^35^26^2, 23^24^256, 12^33^44^55^36^3, 12^23^24^35^26^2, 123^24^35^26^2, 12346, 2^23^34^45^26^2, 2^23^34^35^26^2, \\
	& 234^256, 34^256, 2^23^24^35^26^2, 23^24^35^26^2, 23456, 2346, 3456, 346, 456, 46, 6 \}, \\
	\varDelta_{+}^{a_{17}}= & \{ 1, 12, 2, 123, 23, 3, 1234, 234, 34, 4, 12345, 2345, 345, 45, 5, 1^22^33^34^356, 1^22^23^34^356, 1^22^23^24^356, \\
	& 1^22^23^24^256, 1^22^23^24^26, 12^23^34^356, 12^23^24^356, 12^23^24^256, 12^23^24^26, 1^22^33^44^55^26^2, 1^22^33^44^556^2, \\
	& 123^24^356, 23^24^356, 1^22^33^44^45^26^2, 1^22^33^34^45^26^2, 1^22^33^44^456^2, 123^24^256, 23^24^256, 1^32^43^54^65^26^3, \\
	& 1^22^43^54^65^26^3, 1^22^33^34^456^2, 1^22^33^34^356^2, 123^24^26, 23^24^26, 1^22^33^54^65^26^3, 1^22^23^34^45^26^2, \\
	& 1^22^23^34^456^2, 1^22^23^34^356^2, 12^23^34^45^26^2, 12^23^34^456^2, 12^23^34^356^2, 1^22^33^44^65^26^3, 1^22^33^44^55^26^3, \\
	& 1234^256, 234^256, 34^256, 1^22^33^44^556^3, 1234^26, 234^26, 34^26, 1^22^23^24^356^2, 12^23^24^356^2, 123^24^356^2, \\
	& 123456, 12346, 23^24^356^2, 23456, 2346, 3456, 346, 456, 46, 6 \}, \\
	\varDelta_{+}^{a_{18}}= & \{ 1, 12, 2, 123, 23, 3, 1234, 234, 34, 4, 12345, 2345, 345, 45, 5, 1^22^33^34^256, 1^22^23^34^256, 1^22^23^24^256, \\
	& 1^22^23^2456, 1^22^23^246, 12^33^34^256, 12^23^34^256, 1^22^43^54^45^26^2, 1^22^43^54^35^26^2, 1^22^43^54^356^2, \\
	& 12^23^24^256, 1^22^43^44^35^26^2, 1^22^43^44^356^2, 2^23^24^256, 123^24^256, 1^22^33^44^35^26^2, 2^23^34^256, \\
	& 1^22^33^44^356^2, 23^24^256, 1^32^53^64^45^26^3, 12^33^44^35^26^2, 12^23^2456, 123^2456, 123456, 1^22^53^64^45^26^3, \\
	& 1^22^43^64^44^26^3, 1^22^43^54^45^26^3, 12^33^44^356^2, 1^22^43^44^256^2, 2^23^2456, 1^22^33^44^256^2, 1^22^33^34^256^2, \\
	& 12^23^246, 1^22^43^54^34^26^3, 12^33^44^256^2, 23^2456, 1^22^23^34^256^2, 123^246, 1^22^43^54^356^3, 12^33^34^256^2, \\
	& 12^23^34^256^2, 12346, 1236, 2^23^246, 23^246, 2^23^34^256^2, 23456, 2346, 236, 3456, 346, 36, 6 \}, \\
	\varDelta_{+}^{a_{19}}= & \{ 1, 12, 2, 123, 23, 3, 1234, 234, 34, 4, 12345, 2345, 345, 45, 5, 12^33^34^356, 12^23^34^356, 12^23^24^356, \\
	& 12^23^24^256, 12^23^24^26, 2^23^34^356, 2^23^24^356, 2^23^24^256, 2^23^24^26, 12^33^44^55^26^2, 12^33^44^556^2, \\
	& 123^24^356, 23^24^356, 12^33^44^45^26^2, 12^33^34^45^26^2, 12^33^44^456^2, 123^24^256, 23^24^256, \\
	& 1^22^43^54^65^26^3, 12^43^54^65^26^3, 12^33^34^456^2, 12^33^34^356^2, 123^24^26, 23^24^26, 12^33^54^65^26^3, \\
	& 12^23^34^45^26^2, 12^23^34^456^2, 12^23^34^356^2, 2^23^34^45^26^2, 2^23^34^456^2, 2^23^34^356^2, 12^33^44^65^26^3, \\
	& 12^33^44^55^26^3, 1234^256, 234^256, 34^256, 12^33^44^556^3, 1234^26, 234^26, 34^26, 12^23^24^356^2, 123^24^356^2, \\
	& 123456, 12346, 2^23^34^356^2, 23^24^356^2, 23456, 2346, 3456, 346, 456, 46, 6 \}, \\
	\varDelta_{+}^{a_{20}}= & \{ 1, 12, 2, 123, 23,  3, 1234, 234, 34, 4, 12345, 2345, 345, 45, 5, 1^22^33^34^256, 1^22^23^34^256, 1^22^23^24^256, \\
	& 1^22^23^2456, 1^22^23^246, 12^23^34^256, 12^23^24^256, 12^23^2456, 12^23^246, 1^22^33^54^45^26^2, 1^22^33^54^35^26^2, \\
	& 1^22^33^54^356^2, 123^34^256, 23^34^256, 1^22^33^44^35^26^2, 1^22^33^44^356^2, 123^24^256, 23^24^256, 1^32^43^64^45^26^3, \\
	& 1^22^43^64^44^26^3, 1^22^33^44^256^2, 1^22^33^34^256^2, 1^22^23^44^35^26^2, 1^22^23^44^356^2, 12^23^44^35^26^2, 123^2456, \\
	& 123456, 1^22^33^64^45^26^3, 1^22^33^54^45^26^3, 12^23^44^356^2, 3^24^256, 23^2456, 23456, 1^22^33^54^35^26^3, \\
	& 1^22^23^44^256^2, 12^23^44^256^2, 3^2456, 1^22^23^34^256^2, 12^23^34^256^2, 3456, 1^22^33^54^356^3, 123^34^256^2, \\
	& 123^246, 12346, 1236, 23^34^256^2, 23^246, 2346, 236, 3^246, 346, 36, 6 \}, \\
	\varDelta_{+}^{a_{21}}= & \{ 1, 12, 2, 123, 23, 3, 1234, 234, 34, 4, 12345, 2345, 345, 45, 5, 12^33^34^256, 12^23^34^256, 12^23^24^256, \\
	& 12^23^2456, 12^23^246, 2^23^34^256, 2^23^24^256, 2^23^2456, 2^23^246, 12^33^54^45^26^2, 12^33^54^35^26^2, \\
	& 12^33^54^356^2, 123^34^256, 23^34^256, 12^33^44^35^26^2, 12^33^44^356^2, 123^24^256, 23^24^256, 1^22^43^64^44^26^3, \\
	& 12^43^64^45^26^3, 12^33^44^256^2, 12^33^34^256^2, 12^23^44^35^26^2, 123^2456, 123456, 12^23^44^356^2, 123^246, \\
	& 12346, 12^33^64^45^26^3, 12^33^54^45^26^3, 2^23^44^35^26^2, 2^23^44^356^2, 3^24^256, 12^23^44^256^2, \\
	& 12^23^34^256^2, 12^33^54^35^26^3, 12^33^54^356^3, 123^34^256^2, 1236, 23^2456, 2^23^44^256^2, 3^2456, 23456, \\
	& 2^23^34^256^2, 3456, 23^34^256^2, 23^246, 2346, 236, 3^246, 346, 36, 6 \}.
	\end{align*}
}%

\subsubsection{Weyl groupoid}
\label{subsubsec:type-g66-Weyl}
The isotropy group  at $a_1 \in \cX$ is
\begin{align*}
\cW(a_1)= \langle \varsigma_1^{a_1}, \varsigma_2^{a_1},  \varsigma_3^{a_1}, \varsigma_4^{a_1}, \varsigma_6^{a_1}, \varsigma_5^{a_1} \varsigma_6 \varsigma_4 \varsigma_6 \varsigma_3 \varsigma_2 \varsigma_4 \varsigma_3 \varsigma_5 \varsigma_3 \varsigma_4 \varsigma_2 \varsigma_3 \varsigma_6 \varsigma_4 \varsigma_6 \varsigma_5  \rangle \simeq W(B_6).
\end{align*}

\subsubsection{Incarnation}

We set the matrices $(\bq^{(i)})_{i\in\I_{21}}$, from left to right and  from up to down:
\begin{align}\label{eq:dynkin-g(6,6)}
\begin{aligned}
&
\xymatrix@C-6pt{\overset{\zeta}{\underset{\ }{\circ}}\ar  @{-}[r]^{\ztu}  &
	\overset{\zeta}{\underset{\ }{\circ}} \ar  @{-}[r]^{\ztu}  & \overset{\zeta}{\underset{\
		}{\circ}}
	\ar  @{-}[r]^{\ztu}  & \overset{\zeta}{\underset{\ }{\circ}} \ar  @{-}[r]^{\ztu}  &
	\overset{\ztu}{\underset{\ }{\circ}}  \ar  @{-}[r]^{\zeta}  & \overset{-1}{\underset{\
		}{\circ}}}
& &
\xymatrix@C-6pt{\overset{\zeta}{\underset{\ }{\circ}}\ar  @{-}[r]^{\ztu}  &
	\overset{\zeta}{\underset{\ }{\circ}} \ar  @{-}[r]^{\ztu}  & \overset{\zeta}{\underset{\
		}{\circ}}
	\ar  @{-}[r]^{\ztu}    & \overset{\zeta}{\underset{\ }{\circ}}  \ar  @{-}[r]^{\ztu} &
	\overset{-1}{\underset{\ }{\circ}} \ar  @{-}[r]^{\ztu}  & \overset{-1}{\underset{\
		}{\circ}}}
\\&
\xymatrix@C-7pt{\overset{\zeta}{\underset{\ }{\circ}} \ar  @{-}[r]^{\ztu}  &
	\overset{-1}{\underset{\ }{\circ}}\ar  @{-}[r]^{\zeta}  & \overset{-1}{\underset{\
		}{\circ}}\ar  @{-}[r]^{\ztu}  &
	\overset{-1}{\underset{\ }{\circ}}\ar  @{-}[r]^{\ztu}  & \overset{\zeta}{\underset{\
		}{\circ}}
	\ar  @{-}[r]^{\ztu}  & \overset{\zeta}{\underset{\ }{\circ}}}
&&\xymatrix@C-7pt{\overset{-1}{\underset{\ }{\circ}} \ar  @{-}[r]^{\zeta}  &
	\overset{-1}{\underset{\ }{\circ}}
	\ar  @{-}[r]^{\ztu}  & \overset{\zeta}{\underset{\ }{\circ}} \ar  @{-}[r]^{\ztu}  &
	\overset{-1}{\underset{\ }{\circ}} \ar  @{-}[r]^{\ztu}  & \overset{\zeta}{\underset{\
		}{\circ}}\ar  @{-}[r]^{\ztu}  & \overset{\zeta}{\underset{\ }{\circ}}}
\\ &
\xymatrix@C-6pt{\overset{-1}{\underset{\ }{\circ}}\ar  @{-}[r]^{\ztu}  &
	\overset{\zeta}{\underset{\ }{\circ}} \ar  @{-}[r]^{\ztu}  & \overset{\zeta}{\underset{\
		}{\circ}}
	\ar  @{-}[r]^{\ztu}  & \overset{-1}{\underset{\ }{\circ}} \ar  @{-}[r]^{\ztu}  &
	\overset{\zeta}{\underset{\ }{\circ}}\ar  @{-}[r]^{\ztu}  & \overset{\zeta}{\underset{\
		}{\circ}}}
&& \xymatrix@C-7pt{\overset{\zeta}{\underset{\ }{\circ}}\ar  @{-}[r]^{\ztu}  &
	\overset{\zeta}{\underset{\ }{\circ}} \ar  @{-}[r]^{\ztu}  & \overset{-1}{\underset{\
		}{\circ}}
	\ar  @{-}[r]^{\zeta}  & \overset{\ztu}{\underset{\ }{\circ}} \ar  @{-}[r]^{\ztu}  &
	\overset{\zeta}{\underset{\ }{\circ}}\ar  @{-}[r]^{\ztu}  & \overset{\zeta}{\underset{\
		}{\circ}}}
\end{aligned}
\end{align}

\begin{align*}
&
\xymatrix@R-8pt{  &  & & \overset{-1}{\circ} \ar  @{-}[d]^{\zeta}\ar@{-}[dr]^{\zeta}  & \\
	\overset{\zeta}{\underset{\ }{\circ}} \ar  @{-}[r]^{\ztu} & \overset{\zeta}{\underset{\
		}{\circ}} \ar  @{-}[r]^{\ztu} & \overset{\zeta}{\underset{\ }{\circ}} \ar  @{-}[r]^{\ztu}
	& \overset{-1}{\underset{\ }{\circ}} \ar  @{-}[r]^{\zeta}  & \overset{\ztu}{\underset{\
		}{\circ}}}
&&
\xymatrix@C-3pt@R-8pt{  &   & \overset{-1}{\circ} \ar  @{-}[d]^{\zeta} \ar
	@{-}[dr]^{\zeta}  & \\
	\overset{-1}{\underset{\ }{\circ}} \ar  @{-}[r]^{\ztu}  & \overset{\zeta}{\underset{\
		}{\circ}} \ar  @{-}[r]^{\ztu}
	& \overset{-1}{\underset{\ }{\circ}} \ar  @{-}[r]^{\zeta}
	& \overset{-1}{\underset{\ }{\circ}}\ar  @{-}[r]^{\ztu}  & \overset{\zeta}{\underset{\
		}{\circ}}}
\\ &
\xymatrix@C-3pt@R-8pt{  &   & \overset{-1}{\circ} \ar  @{-}[d]^{\zeta} \ar
	@{-}[dr]^{\zeta}  & \\
	\overset{-1}{\underset{\ }{\circ}} \ar  @{-}[r]^{\zeta}  & \overset{-1}{\underset{\
		}{\circ}} \ar  @{-}[r]^{\ztu}
	& \overset{-1}{\underset{\ }{\circ}} \ar  @{-}[r]^{\zeta}
	& \overset{-1}{\underset{\ }{\circ}}\ar  @{-}[r]^{\ztu}  & \overset{\zeta}{\underset{\
		}{\circ}}}
&&
\xymatrix@C-3pt@R-8pt{  &   & \overset{-1}{\circ} \ar  @{-}[d]^{\zeta} \ar
	@{-}[dr]^{\zeta}  & \\
	\overset{\zeta}{\underset{\ }{\circ}} \ar  @{-}[r]^{\ztu}  & \overset{-1}{\underset{\
		}{\circ}} \ar  @{-}[r]^{\zeta}
	& \overset{\ztu}{\underset{\ }{\circ}} \ar  @{-}[r]^{\zeta}
	& \overset{-1}{\underset{\ }{\circ}}\ar  @{-}[r]^{\ztu}  & \overset{\zeta}{\underset{\
		}{\circ}}}
\end{align*}

\begin{align*}
&
\xymatrix@C-3pt@R-8pt{  & &  & \overset{\zeta}{\circ} \ar  @{-}[d]^{\ztu}  & \\
	\overset{\zeta}{\underset{\ }{\circ}} \ar  @{-}[r]^{\ztu} & \overset{\zeta}{\underset{\
		}{\circ}} \ar  @{-}[r]^{\ztu}  & \overset{-1}{\underset{\ }{\circ}} \ar  @{-}[r]^{\zeta}
	& \overset{-1}{\underset{\ }{\circ}} \ar  @{-}[r]^{\ztu}  & \overset{-1}{\underset{\
		}{\circ}}}
&& \xymatrix@R-8pt{  & &  & \overset{\zeta}{\circ} \ar  @{-}[d]^{\ztu}  & \\
	\overset{\zeta}{\underset{\ }{\circ}} \ar  @{-}[r]^{\ztu} & \overset{-1}{\underset{\
		}{\circ}} \ar  @{-}[r]^{\zeta}  & \overset{-1}{\underset{\ }{\circ}} \ar  @{-}[r]^{\ztu}
	& \overset{\zeta}{\underset{\ }{\circ}} \ar  @{-}[r]^{\ztu}  & \overset{-1}{\underset{\
		}{\circ}}}
\\ &
\xymatrix@R-8pt{  & &  & \overset{\zeta}{\circ} \ar  @{-}[d]^{\ztu}  & \\
	\overset{\zeta}{\underset{\ }{\circ}} \ar  @{-}[r]^{\ztu} & \overset{\zeta}{\underset{\
		}{\circ}} \ar  @{-}[r]^{\ztu}  & \overset{-1}{\underset{\ }{\circ}} \ar  @{-}[r]^{\zeta}
	& \overset{\ztu}{\underset{\ }{\circ}} \ar  @{-}[r]^{\zeta}  & \overset{-1}{\underset{\
		}{\circ}}}
&& \xymatrix@R-8pt{  & &  & \overset{\zeta}{\circ} \ar  @{-}[d]^{\ztu}  & \\
	\overset{-1}{\underset{\ }{\circ}} \ar  @{-}[r]^{\zeta} & \overset{-1}{\underset{\
		}{\circ}} \ar  @{-}[r]^{\ztu}  & \overset{\zeta}{\underset{\ }{\circ}} \ar  @{-}[r]^{\ztu}
	& \overset{\zeta}{\underset{\ }{\circ}} \ar  @{-}[r]^{\ztu}  & \overset{-1}{\underset{\
		}{\circ}}}
\\ &
\xymatrix@C-4pt@R-8pt{  & &  & \overset{\zeta}{\circ} \ar  @{-}[d]^{\ztu}  & \\
	\overset{\zeta}{\underset{\ }{\circ}} \ar  @{-}[r]^{\ztu} & \overset{-1}{\underset{\
		}{\circ}} \ar  @{-}[r]^{\zeta}  & \overset{-1}{\underset{\ }{\circ}} \ar  @{-}[r]^{\ztu}
	& \overset{-1}{\underset{\ }{\circ}} \ar  @{-}[r]^{\zeta}  & \overset{-1}{\underset{\
		}{\circ}}}
&&
\xymatrix@R-8pt{  & &  & \overset{\zeta}{\circ} \ar  @{-}[d]^{\ztu}  & \\
	\overset{-1}{\underset{\ }{\circ}} \ar  @{-}[r]^{\ztu} & \overset{\zeta}{\underset{\
		}{\circ}} \ar  @{-}[r]^{\ztu}  & \overset{\zeta}{\underset{\ }{\circ}} \ar  @{-}[r]^{\ztu}
	& \overset{\zeta}{\underset{\ }{\circ}} \ar  @{-}[r]^{\ztu}  & \overset{-1}{\underset{\
		}{\circ}}}
\\ &
\xymatrix@C-4pt@R-8pt{  & &  & \overset{\zeta}{\circ} \ar  @{-}[d]^{\ztu}  & \\
	\overset{-1}{\underset{\ }{\circ}} \ar  @{-}[r]^{\zeta} & \overset{-1}{\underset{\
		}{\circ}} \ar  @{-}[r]^{\ztu}  & \overset{\zeta}{\underset{\ }{\circ}} \ar  @{-}[r]^{\ztu}
	& \overset{-1}{\underset{\ }{\circ}} \ar  @{-}[r]^{\zeta}  & \overset{-1}{\underset{\
		}{\circ}}}
&&
\xymatrix@C-4pt@R-8pt{  &   & \overset{\zeta}{\circ} \ar  @{-}[d]^{\ztu}  & \\
	\overset{\ztu}{\underset{\ }{\circ}} \ar  @{-}[r]^{\zeta}  & \overset{-1}{\underset{\
		}{\circ}} \ar  @{-}[r]^{\ztu}  & \overset{\zeta}{\underset{\ }{\circ}} \ar  @{-}[r]^{\ztu}
	& \overset{\zeta}{\underset{\ }{\circ}}\ar  @{-}[r]^{\ztu}  & \overset{\zeta}{\underset{\
		}{\circ}}}
\\ &
\xymatrix@C-3pt@R-8pt{  & &  & \overset{\zeta}{\circ} \ar  @{-}[d]^{\ztu}  & \\
	\overset{-1}{\underset{\ }{\circ}} \ar  @{-}[r]^{\ztu} & \overset{\zeta}{\underset{\
		}{\circ}} \ar  @{-}[r]^{\ztu}  & \overset{\zeta}{\underset{\ }{\circ}} \ar  @{-}[r]^{\ztu}
	& \overset{-1}{\underset{\ }{\circ}} \ar  @{-}[r]^{\zeta}  & \overset{-1}{\underset{\
		}{\circ}}}
&&
\xymatrix@C-3pt@R-8pt{  &   & \overset{\zeta}{\circ} \ar  @{-}[d]^{\ztu}  & \\
	\overset{-1}{\underset{\ }{\circ}} \ar  @{-}[r]^{\zeta}  & \overset{\ztu}{\underset{\
		}{\circ}} \ar  @{-}[r]^{\zeta}
	& \overset{-1}{\underset{\ }{\circ}} \ar  @{-}[r]^{\ztu}
	& \overset{\zeta}{\underset{\ }{\circ}}\ar  @{-}[r]^{\ztu}  & \overset{\zeta}{\underset{\
		}{\circ}}}
\\ &
\xymatrix@C-3pt@R-8pt{  &   & \overset{\zeta}{\circ} \ar  @{-}[d]^{\ztu}   & \\
	\overset{-1}{\underset{\ }{\circ}} \ar  @{-}[r]^{\ztu}  & \overset{-1}{\underset{\
		}{\circ}} \ar  @{-}[r]^{\zeta}
	& \overset{-1}{\underset{\ }{\circ}} \ar  @{-}[r]^{\ztu}
	& \overset{\zeta}{\underset{\ }{\circ}}\ar  @{-}[r]^{\ztu}  & \overset{\zeta}{\underset{\
		}{\circ}}}
&&
\end{align*}
Now, this is the incarnation:
\begin{align*}
& a_i\mapsto s_{56}(\bq^{i}), \ i\in\I_{2}; &
& a_i\mapsto s_{456}(\bq^{i}), \ i\in\I_{3,6}; &
& a_i\mapsto \bq^{(i)}, \ i\in\I_{7,21}.
\end{align*}

\subsubsection{PBW-basis and dimension} \label{subsubsec:type-g66-PBW}
Notice that the roots in each $\varDelta_{+}^{a_i}$, $i\in\I_{21}$, are ordered from left to right, justifying the notation $\beta_1, \dots, \beta_{68}$.

The root vectors $x_{\beta_k}$ are described as in Remark \ref{rem:lyndon-word}.
Thus
\begin{align*}
\left\{ x_{\beta_{68}}^{n_{68}} \dots x_{\beta_2}^{n_{2}}  x_{\beta_1}^{n_{1}} \, | \, 0\le n_{k}<N_{\beta_k} \right\}.
\end{align*}
is a PBW-basis of $\toba_{\bq}$. Hence $\dim \toba_{\bq}=2^{32}3^{36}
$.

\subsubsection{The Dynkin diagram \emph{(\ref{eq:dynkin-g(6,6)}
		a)}}\label{subsubsec:g(6,6)-a}

\

The Nichols algebra $\toba_{\bq}$ is generated by $(x_i)_{i\in \I_6}$ with defining relations
\begin{align}\label{eq:rels-g(6,6)-a}
\begin{aligned}
& x_{112}=0; & x_{221}&=0; & x_{223}&=0; & &x_{ij}=0, \ i<j, \, \widetilde{q}_{ij}=1;\\
&[x_{5543},x_{54}]_c=0; & x_{332}&=0; & x_{334}&=0; &  &[[x_{(46)},x_5]_c,x_5]_c=0;\\
&x_{443}=0; & x_{445}&=0; & x_{556}&=0; & &x_{6}^2=0; \quad x_{\alpha}^{3}=0, \ \alpha\in\Oc_+^{\bq}.
\end{aligned}
\end{align}
Here {\scriptsize$\Oc_+^{\bq}=\{ 1, 12, 2, 123, 23,
	3, 1234, 234, 34, 4, 12345, 12^23^24^25^2,
	123^24^25^2, 1234^25^2, 12345^2, \\ 2345, 23^24^25^2,
	234^25^2, 2345^2, 345, 34^25^2, 345^2, 45, 45^2,
	5, 1^22^33^44^55^66^2, 1^22^33^44^55^66^2, \\ 12^23^44^55^66^2, 12^23^34^55^66^2,
	12^23^34^45^66^2, 12^23^34^45^56^2, 12^23^34^45^46^2, 12^23^34^35^46^2, 12^23^24^35^46^2, \\
	123^24^35^46^2, 23^24^35^46^2 \}$}
and the degree of the integral is
\begin{equation*}
\ya= 56\alpha_1 + 108\alpha_2 + 156\alpha_3 + 200\alpha_4 + 240\alpha_5 + 76\alpha_6.
\end{equation*}

\subsubsection{The Dynkin diagram \emph{(\ref{eq:dynkin-g(6,6)}
		b)}}\label{subsubsec:g(6,6)-b}

\

The Nichols algebra $\toba_{\bq}$ is generated by $(x_i)_{i\in \I_6}$ with defining relations
\begin{align}\label{eq:rels-g(6,6)-b}
\begin{aligned}
x_{112}&=0; & x_{221}&=0; & x_{445}&=0; & & x_{ij}=0, \ i<j, \, \widetilde{q}_{ij}=1;\\
x_{223}&=0; & x_{332}&=0; & x_{5}^2&=0; &  & [[x_{65},x_{654}]_c,x_5]_c=0;\\
x_{334}&=0; & x_{443}&=0; & x_{6}^2&=0; & & x_{\alpha}^3=0, \ \alpha\in\Oc_+^{\bq}.
\end{aligned}
\end{align}
Here {\scriptsize$\Oc_+^{\bq}=\{ 1, 12, 2, 123, 23,
	3, 1234, 234, 34, 4, 123456, 12^23^34^45^46^2,
	12^23^34^35^46^2, 12^23^24^35^46^2$, \\ $12^23^24^25^26^2, 23456, 1^22^33^44^55^66^4,
	12^33^44^55^66^4, 12^23^34^45^56^3, 123^24^35^46^2, 123^24^25^26^2$, \\ $1234^25^26^2, 12345^26^2, 12^23^44^55^66^4,
	12^23^34^55^66^4, 12^23^34^45^66^4, 23^24^35^46^2, 23^24^25^26^2, 3456$, \\ $234^25^26^2, 2345^26^2, 34^25^26^2, 345^26^2, 456, 45^26^2, 56 \}$}  
and the degree of the integral is
\begin{equation*}
\ya= 56\alpha_1 + 108\alpha_2 + 156\alpha_3 + 200\alpha_4 + 240\alpha_5 + 166\alpha_6.
\end{equation*}

\subsubsection{The Dynkin diagram \emph{(\ref{eq:dynkin-g(6,6)}
		c)}}\label{subsubsec:g(6,6)-c}

\

The Nichols algebra $\toba_{\bq}$ is generated by $(x_i)_{i\in \I_6}$ with defining relations
\begin{align}\label{eq:rels-g(6,6)-c}
\begin{aligned}
x_{112}&=0; & x_{554}&=0; & x_{2}^2&=0; & & x_{ij}=0, \ i<j, \, \widetilde{q}_{ij}=1;
\\
[x_{(13)},&x_2]_c=0; & x_{556}&=0; & x_{3}^2&=0; & & [[x_{34},x_{(35)}]_c,x_4]_c=0;
\\
[x_{(24)},&x_3]_c=0; & x_{665}&=0; & x_{4}^2&=0; & 
& x_{\alpha}^{3}=0, \ \alpha\in\Oc_+^{\bq}.
\end{aligned}
\end{align}
Here {\scriptsize$\Oc_+^{\bq}=\{ 1, 123, 23, 34, 12^23^24^25,
	1234^25, 234^25, 3^24^25, 345, 5, 12^23^44^45^36, 123^34^45^36,
	23^34^45^36$, \\ $12^23^44^45^26, 123^34^45^26, 23^34^45^26, 12^23^34^35^26,
	1^22^33^54^65^46^2, 12^33^54^65^46^2, 12^23^24^25^26$, \\ $12^23^24^256, 123^24^35^26, 12^23^44^65^46^2, 23^24^35^26,
	12^23^44^45^36^2, 3^24^25^26, 3^24^256, 123^34^45^36^2$, \\ $1234^25^26,
	1234^256, 23^34^45^36^2, 234^25^26, 234^256, 3456, 56, 6 \}$}  
and the degree of the integral is
\begin{equation*}
\ya= 56\alpha_1 + 108\alpha_2 + 206\alpha_3 + 252\alpha_4 + 172\alpha_5 + 88\alpha_6.
\end{equation*}

\subsubsection{The Dynkin diagram \emph{(\ref{eq:dynkin-g(6,6)}
		d)}}\label{subsubsec:g(6,6)-d}

\

The Nichols algebra $\toba_{\bq}$ is generated by $(x_i)_{i\in \I_6}$ with defining relations
\begin{align}\label{eq:rels-g(6,6)-d}
\begin{aligned}
x_{332}&=0; & x_{334}&=0; & &x_{ij}=0, \ i<j, \, \widetilde{q}_{ij}=1;\\
x_{556}&=0; & x_{665}&=0; &  & [[[x_{(25)},x_4]_c,x_3]_c,x_4]_c=0;\\
x_1^2&=0; & [x_{(13)},&x_2]_c=0; & & [[[x_{6543},x_4]_c,x_5]_c,x_4]_c=0; \\
x_{554}&=0; & x_{2}^2&=0; & &x_{4}^2=0; \quad  x_{\alpha}^{3}=0, \ \alpha\in\Oc_+^{\bq}.
\end{aligned}
\end{align}
Here {\scriptsize$\Oc_+^{\bq}=\{ 12, 123, 3, 234, 123^24^25,
	1234^25, 2^23^24^25, 2345, 34^25, 5, 12^33^44^45^36, 12^33^34^45^36,
	2^23^34^45^36$, \\ $12^33^44^45^26, 12^33^34^45^26, 2^23^34^45^26, 12^23^34^35^26,
	1^22^43^54^65^46^2, 12^23^24^35^26, 2^23^24^25^26, 2^23^24^256$, \\ $12^33^54^65^46^2, 12^33^44^65^46^2, 23^24^35^26,
	12^33^44^45^36^2, 123^24^25^26, 123^24^256, 12^33^34^45^36^2$, \\ $1234^25^26,
	1234^256, 2^23^34^45^36^2, 23456, 34^25^26, 34^256, 56, 6 \}$}  
and the degree of the integral is
\begin{equation*}
\ya= 56\alpha_1 + 156\alpha_2 + 206\alpha_3 + 252\alpha_4 + 172\alpha_5 + 88\alpha_6.
\end{equation*}

\subsubsection{The Dynkin diagram \emph{(\ref{eq:dynkin-g(6,6)}
		e)}}\label{subsubsec:g(6,6)-e}

\

The Nichols algebra $\toba_{\bq}$ is generated by $(x_i)_{i\in \I_6}$ with defining relations
\begin{align}\label{eq:rels-g(6,6)-e}
\begin{aligned}
x_{223}&=0; & x_{332}&=0; & &x_{ij}=0, \ i<j, \, \widetilde{q}_{ij}=1;\\
x_{334}&=0; & x_{554}&=0; &  & [[[x_{(25)},x_4]_c,x_3]_c,x_4]_c=0;\\
x_{556}&=0; & x_{665}&=0; & & [[[x_{6543},x_4]_c,x_5]_c,x_4]_c=0; \\
x_{221}&=0; & x_{1}^2&=0; & & x_{4}^2=0; \quad x_{\alpha}^{3}=0, \ \alpha\in\Oc_+^{\bq}.
\end{aligned}
\end{align}
Here {\scriptsize$\Oc_+^{\bq}=\{ 2, 23, 3, 1234, 1^22^23^24^25,
	12345, 23^24^25, 234^25, 34^25, 5, 1^22^33^44^45^36, 1^22^33^34^45^36$, \\ $1^22^23^34^45^36, 1^22^33^44^45^26, 1^22^33^34^45^26, 1^22^23^34^45^26, 12^23^34^35^26,
	1^22^23^24^25^26, 1^22^23^24^256$, \\ $1^22^43^54^65^46^2, 12^23^24^35^26, 1^22^33^54^65^46^2, 1^22^33^44^65^46^2, 123^24^35^26,
	1^22^33^44^45^36^2, 23^24^25^26$, \\ $23^24^256, 1^22^33^34^45^36^2, 1^22^23^34^45^36^2,
	123456, 234^25^26, 234^256, 34^25^26, 34^256, 56, 6 \}$}  
and the degree of the integral is
\begin{equation*}
\ya= 102\alpha_1 + 156\alpha_2 + 206\alpha_3 + 252\alpha_4 + 172\alpha_5 + 88\alpha_6.
\end{equation*}

\subsubsection{The Dynkin diagram \emph{(\ref{eq:dynkin-g(6,6)}
		f)}}\label{subsubsec:g(6,6)-f}

\

The Nichols algebra $\toba_{\bq}$ is generated by $(x_i)_{i\in \I_6}$ with defining relations
\begin{align}\label{eq:rels-g(6,6)-f}
\begin{aligned}
x_{112}&=0; & x_{221}&=0; & & x_{ij}=0, \ i<j, \, \widetilde{q}_{ij}=1;\\
x_{223}&=0; & x_{443}&=0; &  & [[x_{(35)},x_4]_c,x_4]_c=0;\\
x_{554}&=0; & x_{556}&=0; & & [x_{4456},x_{45}]_c=0; \\
x_{665}&=0; & [x_{(24)},x_3]_c&=0; & & x_{3}^2=0; \quad x_{\alpha}^{3}=0, \ \alpha\in\Oc_+^{\bq}.
\end{aligned}
\end{align}
Here {\scriptsize$\Oc_+^{\bq}=\{ 1, 12, 2, 4, 12^23^24^25,
	123^24^25, 23^24^25, 4^25, 45, 5, 12^23^24^45^36, 123^24^45^36,
	23^24^45^36$, \\ $12^23^24^45^26, 123^24^45^26, 23^24^45^26, 12^23^24^35^26,
	1^22^33^44^65^46^2, 12^33^44^65^46^2, 12^23^24^25^26$, \\ $12^23^24^256, 123^24^35^26, 12^23^44^65^46^2, 123^24^25^26,
	123^24^256, 23^24^35^26, 23^24^25^26, 23^24^256$, \\ $12^23^24^45^36^2,
	123^24^45^36^2, 23^24^45^36^2, 4^25^26, 4^256, 456, 56, 6 \}$}  
and the degree of the integral is
\begin{equation*}
\ya= 56\alpha_1 + 108\alpha_2 + 156\alpha_3 + 252\alpha_4 + 172\alpha_5 + 88\alpha_6.
\end{equation*}

\subsubsection{The Dynkin diagram \emph{(\ref{eq:dynkin-g(6,6)}
		g)}}\label{subsubsec:g(6,6)-g}

\

The Nichols algebra $\toba_{\bq}$ is generated by $(x_i)_{i\in \I_6}$ with defining relations
\begin{align}\label{eq:rels-g(6,6)-g}
\begin{aligned}
& \begin{aligned}
x_{112}&=0; & x_{221}&=0; & x_{223}&=0; & & x_{ij}=0, \ i<j, \, \widetilde{q}_{ij}=1;\\
x_{554}&=0; & x_{332}&=0; & x_{334}&=0; &  & [x_{346},x_4]_c=0;\\
x_{556}&=0; & x_{4}^2&=0; & x_{5}^2&=0; & & x_{\alpha}^{3}=0, \ \alpha\in\Oc_+^{\bq};
\end{aligned}
\\
& [x_{(35)},x_4]_c=0; \quad  x_{(46)}= q_{56}\ztu [x_{46},x_5]_c +q_{45}(1-\zeta)x_5x_{46}.
\end{aligned}
\end{align}
Here {\scriptsize$\Oc_+^{\bq}=\{ 1, 12, 2, 123, 23,
	3, 5, 12^23^34^35^26, 12^23^24^35^26, 123^24^35^26, 23^24^35^26, 12^23^34^45^46^2$, \\ $12^23^34^45^36^2, 12345^26, 2345^26, 345^26, 1^22^33^44^55^46^3,
	12^33^44^55^46^3, 12^23^44^55^46^3, 12^23^44^45^26^2$, \\ $123456, 12^23^34^55^46^3, 45^26, 12^23^24^25^26^2,
	123^24^25^26^2, 1234^25^26^2, 12346, 23456, 23^24^25^26^2$, \\ $234^25^26^2, 2346, 3456, 34^25^26^2, 346, 456, 46 \}$}  
and the degree of the integral is
\begin{equation*}
\ya= 56\alpha_1 + 108\alpha_2 + 156\alpha_3 + 200\alpha_4 + 166\alpha_5 + 128\alpha_6.
\end{equation*}

\subsubsection{The Dynkin diagram \emph{(\ref{eq:dynkin-g(6,6)}
		h)}}\label{subsubsec:g(6,6)-h}

The Nichols algebra $\toba_{\bq}$ is generated by $(x_i)_{i\in \I_6}$ with defining relations
\begin{align}\label{eq:rels-g(6,6)-h}
\begin{aligned}
& \begin{aligned}
x_{221}&=0; & x_{223}&=0; & [x_{546},&x_4]_c=0; & & x_{ij}=0, \ i<j, \, \widetilde{q}_{ij}=1;\\
x_{554}&=0; & x_{1}^2&=0; & [x_{(24)},&x_3]_c=0; &  & [x_{236},x_3]_c=0;\\
x_{3}^2&=0; & x_{4}^2&=0; & [x_{(35)},&x_4]_c=0; & & x_{\alpha}^{3}=0, \ \alpha\in\Oc_+^{\bq};
\end{aligned}
\\
& x_{6}^2=0; \quad x_{346}= q_{46}\ztu [x_{36},x_4]_c + q_{34}(1-\zeta)x_4x_{36}.
\end{aligned}
\end{align}
Here {\scriptsize$\Oc_+^{\bq}=\{ 2, 123, 234, 34, 2345,
	345, 5, 1^22^33^34^256, 1^22^23^34^256, 1^22^23^2456, 1^22^23^246, 12^23^24^256$, \\ $1^22^33^44^45^26^2, 123^24^256, 1^22^33^34^35^26^2, 1^22^33^34^356^2, 234^256,
	1^22^43^54^45^26^3, 1^22^33^44^256^2$, \\ $1^22^23^34^35^26^2, 1^22^23^34^356^2, 123456, 1^22^33^54^45^26^3, 34^256,
	23^2456, 1^22^33^44^35^26^3, 12^23^34^256^2$, \\ $1^22^23^24^256^2, 456,
	1^22^33^44^356^3, 12346, 23^24^256^2, 23^246, 236, 36, 46 \}$}  
and the degree of the integral is
\begin{equation*}
\ya= 102\alpha_1 + 156\alpha_2 + 206\alpha_3 + 172\alpha_4 + 88\alpha_5 + 128\alpha_6.
\end{equation*}

\subsubsection{The Dynkin diagram \emph{(\ref{eq:dynkin-g(6,6)}
		i)}}\label{subsubsec:g(6,6)-i}

\

The Nichols algebra $\toba_{\bq}$ is generated by $(x_i)_{i\in \I_6}$ with defining relations
\begin{align}\label{eq:rels-g(6,6)-i}
\begin{aligned}
& \begin{aligned}
x_{554}&=0; & x_{1}^2&=0; & [x_{(13)},&x_2]_c=0; & & x_{ij}=0, \ i<j, \, \widetilde{q}_{ij}=1;\\
x_{2}^2&=0; & x_{3}^2&=0; & [x_{(24)},&x_3]_c=0; &  & [x_{236},x_3]_c=0;\\
x_{4}^2&=0; & x_{6}^2&=0; & [x_{(35)},&x_4]_c=0; & & x_{\alpha}^{3}=0, \ \alpha\in\Oc_+^{\bq};
\end{aligned}
\\
& [x_{546},x_4]_c=0; \quad 
x_{346}= q_{46}\ztu [x_{36},x_4]_c +q_{34}(1-\zeta)x_4x_{36}.
\end{aligned}
\end{align}
Here {\scriptsize$\Oc_+^{\bq}=\{ 12, 23, 1234, 34, 12345,
	345, 5, 12^33^34^256, 12^23^24^256, 2^23^34^256, 2^23^2456, 2^23^246$, \\ $12^33^44^45^26^2, 23^24^256, 12^33^34^35^26^2, 12^33^34^356^2, 1234^256,
	1^22^43^54^45^26^3, 12^33^44^256^2, 123^2456$, \\ $123^246, 12^33^54^45^26^3, 2^23^34^35^26^2, 2^23^34^356^2,
	34^256, 12^23^34^256^2, 12^33^44^35^26^3, 12^33^44^356^3$, \\ $123^24^256^2,
	1236, 23456, 2^23^24^256^2, 2346, 456, 36, 46 \}$}  
and the degree of the integral is
\begin{equation*}
\ya= 56\alpha_1 + 156\alpha_2 + 206\alpha_3 + 172\alpha_4 + 88\alpha_5 + 128\alpha_6.
\end{equation*}

\subsubsection{The Dynkin diagram \emph{(\ref{eq:dynkin-g(6,6)}
		j)}}\label{subsubsec:g(6,6)-j}

\

The Nichols algebra $\toba_{\bq}$ is generated by $(x_i)_{i\in \I_6}$ with defining relations
\begin{align}\label{eq:rels-g(6,6)-j}
\begin{aligned}
& \begin{aligned}
x_{112}&=0; & x_{332}&=0; & x_{334}&=0; & & x_{ij}=0, \ i<j, \, \widetilde{q}_{ij}=1;\\
x_{336}&=0; & x_2^2&=0; & [x_{(13)},&x_2]_c=0; &  & [x_{(35)},x_4]_c=0;\\
x_{554}&=0; & x_{4}^2&=0; & [x_{546},&x_4]_c=0; & & x_{\alpha}^{3}=0, \ \alpha\in\Oc_+^{\bq};
\end{aligned}
\\
& x_{6}^2=0; \quad x_{346}=q_{46}\ztu [x_{36},x_4]_c +q_{34}(1-\zeta)x_4x_{36}.
\end{aligned}
\end{align}
Here {\scriptsize$\Oc_+^{\bq}=\{ 1, 3, 1234, 234, 12345,
	2345, 5, 12^23^2456, 12^23^246, 123^34^256, 23^34^256, 123^24^256$, \\ $12^23^44^45^26^2, 1234^256, 23^24^256, 1^22^33^54^45^26^3, 123^34^35^26^2,
	12^23^44^256^2, 3^2456, 12^33^54^45^26^3$, \\ $12^23^34^256^2, 23^34^35^26^2, 234^256, 123^34^356^2,
	23^34^356^2, 12^23^44^35^26^3, 12^23^24^256^2, 3456$, \\ $12^23^44^356^3,
	1236, 3^246, 236, 3^24^256^2, 456, 346, 46 \}$} 
and the degree of the integral is 
\begin{equation*}
\ya= 56\alpha_1 + 108\alpha_2 + 206\alpha_3 + 172\alpha_4 + 88\alpha_5 + 128\alpha_6.
\end{equation*}

\subsubsection{The Dynkin diagram \emph{(\ref{eq:dynkin-g(6,6)}
		k)}}\label{subsubsec:g(6,6)-k}

\

The Nichols algebra $\toba_{\bq}$ is generated by $(x_i)_{i\in \I_6}$ with defining relations
\begin{align}\label{eq:rels-g(6,6)-k}
\begin{aligned}
x_{112}&=0; & x_{221}&=0; & [x_{346},&x_4]_c=0; & & x_{ij}=0, \ i<j, \, \widetilde{q}_{ij}=1;\\
x_{223}&=0; & x_{664}&=0; & [x_{(24)},&x_3]_c=0; &  & [[x_{54},x_{546}]_c,x_4]_c=0;\\
x_{4}^2&=0; & x_{3}^2&=0; & [x_{(35)},&x_4]_c=0; & & x_{5}^2=0; \quad x_{\alpha}^{3}=0, \ \alpha\in\Oc_+^{\bq}.
\end{aligned}
\end{align}
Here {\scriptsize$\Oc_+^{\bq}=\{ 1, 12, 2, 1234, 234,
	34, 45, 12^23^34^35^26, 12^23^24^25^26, 123^24^25^26, 23^24^25^26, 12^23^34^55^46^2$, \\ $1234^35^26, 234^35^26, 34^35^26, 12^23^34^45^36^2, 1^22^33^44^65^46^3,
	12^33^44^65^46^3, 12^23^44^65^46^3, 12^23^34^35^26^2$, \\ $1234^256, 12^23^34^55^46^3, 4^25^26, 12^23^24^45^26^2,
	123^24^45^26^2, 234^256, 1234^35^26^2, 12346, 23^24^45^26^2$, \\ $234^35^26^2, 2346, 34^256, 34^35^26^2, 346, 456, 6 \}$}  
and the degree of the integral is
\begin{equation*}
\ya= 56\alpha_1 + 108\alpha_2 + 156\alpha_3 + 252\alpha_4 + 166\alpha_5 + 128\alpha_6.
\end{equation*}

\subsubsection{The Dynkin diagram \emph{(\ref{eq:dynkin-g(6,6)}
		l)}}\label{subsubsec:g(6,6)-l}

\

The Nichols algebra $\toba_{\bq}$ is generated by $(x_i)_{i\in \I_6}$ with defining relations
\begin{align}\label{eq:rels-g(6,6)-l}
\begin{aligned}
x_{112}&=0; & x_{443}&=0; & [x_{(13)},&x_2]_c=0; & & x_{ij}=0, \ i<j, \, \widetilde{q}_{ij}=1;
\\
x_{445}&=0; & x_{446}&=0; & x_{664}&=0; & & [x_{(24)},x_3]_c=0;
\\
x_2^2&=0; & x_{3}^2&=0; & x_{5}^2&=0; & &x_{\alpha}^{3}=0, \ \alpha\in\Oc_+^{\bq}.
\end{aligned}
\end{align}
Here {\scriptsize$\Oc_+^{\bq}=\{ 1, 123, 23, 1234, 234,
	4, 345, 12^23^24^35^26, 12^23^24^25^26, 123^34^35^26, 23^34^35^26, 123^24^256$, \\ $12^23^44^55^46^2, 3^24^35^26, 1234^25^26, 1^22^33^54^65^46^3, 12^23^34^45^36^2,
	12^23^44^45^26^2, 123^34^45^26^2, 12346$, \\ $12^33^54^65^46^3, 23^24^256, 12^23^44^65^46^3, 12^23^24^35^26^2,
	234^25^26, 12^23^44^55^46^3, 123^34^35^26^2, 3^24^25^26$, \\ $23^34^45^26^2,
	23^34^35^26^2, 2346, 34^256, 3^24^35^26^2, 3456, 46, 6 \}$}  
and the degree of the integral is
\begin{equation*}
\ya= 56\alpha_1 + 108\alpha_2 + 206\alpha_3 + 252\alpha_4 + 166\alpha_5 + 128\alpha_6.
\end{equation*}

\subsubsection{The Dynkin diagram \emph{(\ref{eq:dynkin-g(6,6)}
		m)}}\label{subsubsec:g(6,6)-m}

\

The Nichols algebra $\toba_{\bq}$ is generated by $(x_i)_{i\in \I_6}$ with defining relations
\begin{align}\label{eq:rels-g(6,6)-m}
\begin{aligned}
&[x_{(24)},x_3]_c=0; & x_{221}&=0; & x_{223}&=0; & &x_{ij}=0, \ i<j, \, \widetilde{q}_{ij}=1;\\
& & x_{443}&=0; & x_{112}&=0; &  &[[x_{346},x_4]_c,x_4]_c=0;\\
& & x_{445}&=0; & x_{664}&=0; & &[[x_{546},x_4]_c,x_4]_c=0; \\
& & x_{3}^2&=0; & x_{5}^2&=0; & &x_{\alpha}^{3}=0, \ \alpha\in\Oc_+^{\bq}.
\end{aligned}
\end{align}
Here {\scriptsize$\Oc_+^{\bq}=\{ 1, 12, 2, 4, 12345,
	2345, 345, 12^23^34^356, 12^23^24^26, 123^24^26, 23^24^26, 12^23^34^556^2,
	1234^356$, \\ $234^356, 34^356, 12^23^34^456^2, 1^22^33^44^65^26^3, 12^33^44^65^26^3, 12^23^44^65^26^3, 12^23^34^356^2, 1234^256$, \\ $12^23^24^45^26^2, 123^24^45^26^2, 123456,
	234^256, 12^23^34^556^3, 1234^356^2, 23^24^45^26^2, 23456,
	34^256$, \\ $234^356^2, 4^26, 34^356^2, 3456, 46, 6 \}$}  
and the degree of the integral is
\begin{equation*}
\ya= 56\alpha_1 + 108\alpha_2 + 156\alpha_3 + 252\alpha_4 + 88\alpha_5 + 128\alpha_6.
\end{equation*}

\subsubsection{The Dynkin diagram \emph{(\ref{eq:dynkin-g(6,6)}
		n)}}\label{subsubsec:g(6,6)-n}

\

The Nichols algebra $\toba_{\bq}$ is generated by $(x_i)_{i\in \I_6}$ with defining relations
\begin{align}\label{eq:rels-g(6,6)-n}
\begin{aligned}
x_{332}&=0; & x_{334}&=0; & x_1^2&=0; & &x_{ij}=0, \ i<j, \, \widetilde{q}_{ij}=1;\\
x_{443}&=0; & x_{445}&=0; & x_{2}^2&=0; & & [x_{(13)},x_2]_c=0;\\
x_{446}&=0; & x_{664}&=0; & x_{5}^2&=0; & &x_{\alpha}^{3}=0, \ \alpha\in\Oc_+^{\bq}.
\end{aligned}
\end{align}
Here {\scriptsize$\Oc_+^{\bq}=\{ 12, 123, 3, 1234, 34,
	4, 2345, 12^33^34^35^26, 12^23^24^256, 2^23^34^35^26, 2^23^24^35^26, 2^23^24^25^26$, \\ $12^33^44^55^46^2, 123^24^35^26, 123^24^25^26, 1234^25^26, 1^22^43^54^65^46^3,
	12^33^44^45^26^2, 12^33^34^45^26^2$, \\ $12^33^34^35^26^2, 12^23^34^45^36^2, 12^33^54^65^46^3, 12^33^44^65^46^3, 12^33^44^55^46^3,
	34^25^26, 23^24^256$, \\ $123^24^35^26^2, 12346, 2^23^34^45^26^2,
	2^23^34^35^26^2, 234^256, 2^23^24^35^26^2, 23456, 346, 46, 6 \}$}  
and the degree of the integral is
\begin{equation*}
\ya= 56\alpha_1 + 156\alpha_2 + 206\alpha_3 + 252\alpha_4 + 166\alpha_5 + 128\alpha_6.
\end{equation*}

\subsubsection{The Dynkin diagram \emph{(\ref{eq:dynkin-g(6,6)} \~{n})}}
\label{subsubsec:g(6,6)-enie}

\

The Nichols algebra $\toba_{\bq}$ is generated by $(x_i)_{i\in \I_6}$ with defining relations
\begin{align}\label{eq:rels-g(6,6)-enie}
\begin{aligned}
x_{112}&=0; & [x_{(13)},&x_2]_c=0; & x_{2}^2&=0; & & x_{ij}=0, \ i<j, \, \widetilde{q}_{ij}=1;
\\
x_{664}&=0; & [x_{(24)},&x_3]_c=0; & x_{3}^2&=0; & &  [[x_{34},x_{346}]_c,x_4]_c=0;
\\
[x_{546},&x_4]_c=0; & [x_{(35)},&x_4]_c=0; & x_{4}^2&=0; & 
& x_{5}^2=0; \ x_{\alpha}^{3}=0, \alpha\in\Oc_+^{\bq}.
\end{aligned}
\end{align}
Here {\scriptsize$\Oc_+^{\bq}=\{ 1, 123, 23, 34, 12345,
	2345, 45, 123^34^356, 123^24^256, 23^34^356, 3^24^356, 3^24^26,
	12^23^44^556^2$, \\ $12^23^24^356, 12^23^44^45^26^2, 123^44^45^26^2, 23^24^256,
	1^22^33^54^65^26^3, 123^34^356^2, 12^23^24^26$, \\ $12^33^54^65^26^3, 12^23^34^456^2, 23^34^45^26^2, 23^34^356^2,
	12^23^44^65^26^3, 34^256, 12^23^44^556^3, 1234^26$, \\ $234^26,
	12^23^24^356^2, 123456, 3^24^356^2, 346, 23456, 456, 6 \}$}  
and the degree of the integral is
\begin{equation*}
\ya= 56\alpha_1 + 108\alpha_2 + 206\alpha_3 + 252\alpha_4 + 88\alpha_5 + 128\alpha_6.
\end{equation*}

\subsubsection{The Dynkin diagram \emph{(\ref{eq:dynkin-g(6,6)}
		o)}}\label{subsubsec:g(6,6)-o}

\

The Nichols algebra $\toba_{\bq}$ is generated by $(x_i)_{i\in \I_6}$ with defining relations
\begin{align}\label{eq:rels-g(6,6)-o}
\begin{aligned}
x_{221}&=0; & x_{223}&=0; & x_{332}&=0; & &x_{ij}=0, \ i<j, \, \widetilde{q}_{ij}=1;\\
x_{334}&=0; & x_{443}&=0; &  x_{445}&=0; & & x_{446}=0; 
\\
x_{664}&=0; & x_{1}^2&=0; & x_{5}^2&=0; & & x_{\alpha}^{3}=0, \ \alpha\in\Oc_+^{\bq}.
\end{aligned}
\end{align}
Here {\scriptsize$\Oc_+^{\bq}=\{ 2, 23, 3, 234, 34,
	4, 12345, 1^22^33^34^35^26, 1^22^23^34^35^26, 1^22^23^24^35^26, 1^22^23^24^25^26, 12^23^24^256$, \\ $1^22^33^44^55^46^2, 23^24^35^26, 23^24^25^26, 234^25^26, 1^22^43^54^65^46^3,
	1^22^33^44^45^26^2, 1^22^33^34^45^26^2$, \\ $1^22^33^34^35^26^2, 12^23^34^45^36^2, 1^22^33^54^65^46^3, 1^22^33^44^65^46^3, 1^22^33^44^55^46^3,
	34^25^26, 123^24^256$, \\ $1^22^23^34^45^26^2, 1^22^23^34^35^26^2, 1234^256,
	1^22^23^24^35^26^2, 123456, 23^24^35^26^2, 2346, 346, 46, 6 \}$}  
and the degree of the integral is
\begin{equation*}
\ya= 102\alpha_1 + 156\alpha_2 + 206\alpha_3 + 252\alpha_4 + 166\alpha_5 + 128\alpha_6.
\end{equation*}

\subsubsection{The Dynkin diagram \emph{(\ref{eq:dynkin-g(6,6)}
		p)}}\label{subsubsec:g(6,6)-p}

\

The Nichols algebra $\toba_{\bq}$ is generated by $(x_i)_{i\in \I_6}$ with defining relations
\begin{align}\label{eq:rels-g(6,6)-p}
\begin{aligned}
x_{332}&=0; & [x_{(13)},&x_2]_c=0; & x_{1}^2&=0; & & x_{ij}=0, \ i<j, \, \widetilde{q}_{ij}=1;
\\
x_{664}&=0; & [x_{(35)},&x_4]_c=0; & x_{2}^2&=0; & & [x_{564},x_4]_c=0;
\\
x_{334}&=0; & x_{4}^2&=0; & x_{5}^2&=0; & & x_{\alpha}^{3}=0, \ \alpha\in\Oc_+^{\bq}.
\end{aligned}
\end{align}
Here {\scriptsize$\Oc_+^{\bq}=\{ 12, 123, 3, 234, 12345,
	345, 45, 12^33^34^356, 12^23^24^256, 2^23^34^356, 2^23^24^356, 2^23^24^26$, \\ $12^33^44^556^2, 123^24^356, 12^33^44^45^26^2, 12^33^34^45^26^2, 23^24^256,
	1^22^43^54^65^26^3, 12^33^34^356^2, 123^24^26$, \\ $12^33^54^65^26^3, 12^23^34^456^2, 2^23^34^45^26^2, 2^23^34^356^2,
	12^33^44^65^26^3, 234^256, 12^33^44^556^3, 1234^26$, \\ $34^26,
	123^24^356^2, 123456, 2^23^24^356^2, 2346, 3456, 456, 6 \}$}  
and the degree of the integral is
\begin{equation*}
\ya= 56\alpha_1 + 156\alpha_2 + 206\alpha_3 + 252\alpha_4 + 88\alpha_5 + 128\alpha_6.
\end{equation*}

\subsubsection{The Dynkin diagram \emph{(\ref{eq:dynkin-g(6,6)}
		q)}}\label{subsubsec:g(6,6)-q}

\

The Nichols algebra $\toba_{\bq}$ is generated by $(x_i)_{i\in \I_6}$ with defining relations
\begin{align}\label{eq:rels-g(6,6)-q}
\begin{aligned}
x_{112}&=0; & x_{332}&=0; & x_{334}&=0; & & x_{ij}=0, \ i<j, \, \widetilde{q}_{ij}=1;\\
x_{336}&=0; & x_{443}&=0; & x_{445}&=0; & & [x_{(13)},x_2]_c=0;\\
x_{554}&=0; & x_{663}&=0; &  x_{2}^2&=0; & & x_{\alpha}^{3}=0, \ \alpha\in\Oc_+^{\bq}.
\end{aligned}
\end{align}
Here {\scriptsize$\Oc_+^{\bq}=\{ 1, 3, 34, 4, 345,
	45, 5, 1^22^23^34^256, 1^22^23^24^256, 1^22^23^2456, 1^22^23^246, 12^23^34^256$, \\ $1^22^43^54^45^26^2, 1^22^43^54^35^26^2, 1^22^43^54^356^2, 2^23^34^256, 12^23^24^256,
	1^22^43^44^35^26^2, 1^22^43^44^356^2$, \\ $2^23^24^256, 12^23^2456, 1^22^43^64^45^26^3, 1^22^43^54^45^26^3, 1^22^43^44^256^2,
	2^23^2456, 12^23^246$, \\ $1^22^43^54^35^26^3, 1^22^23^34^256^2, 1^22^43^54^356^3,
	12^23^34^256^2, 2^23^246, 2^23^34^256^2, 3456, 346, 36, 6 \}$}  
and the degree of the integral is
\begin{equation*}
\ya= 102\alpha_1 + 280\alpha_2 + 252\alpha_3 + 172\alpha_4 + 88\alpha_5 + 128\alpha_6.
\end{equation*}

\subsubsection{The Dynkin diagram \emph{(\ref{eq:dynkin-g(6,6)}
		r)}}\label{subsubsec:g(6,6)-r}

\

The Nichols algebra $\toba_{\bq}$ is generated by $(x_i)_{i\in \I_6}$ with defining relations
\begin{align}\label{eq:rels-g(6,6)-r}
\begin{aligned}
x_{221}&=0; & x_{334}&=0; & x_{332}&=0; & & x_{ij}=0, \ i<j, \, \widetilde{q}_{ij}=1;
\\
x_{223}&=0; & [x_{(35)},&x_4]_c=0; & x_{1}^2&=0; & &  [[[x_{2346},x_4]_c,x_3]_c,x_4]_c=0;
\\
x_{664}&=0; & [x_{546},&x_4]_c=0; & x_{4}^2&=0; & & x_{5}^2=0; \quad x_{\alpha}^{3}=0, \ \alpha\in\Oc_+^{\bq}.
\end{aligned}
\end{align}
Here {\scriptsize$\Oc_+^{\bq}=\{ 2, 23, 3, 1234, 2345,
	345, 45, 1^22^33^34^356, 1^22^23^34^356, 1^22^23^24^356, 1^22^23^24^26, 12^23^24^256$, \\ $
	1^22^33^44^556^2, 23^24^356, 1^22^33^44^45^26^2, 1^22^33^34^45^26^2, 123^24^256,
	1^22^43^54^65^26^3, 1^22^33^34^356^2$, \\ $23^24^26, 1^22^33^54^65^26^3, 1^22^23^34^45^26^2, 1^22^23^34^356^2, 12^23^34^456^2,
	1^22^33^44^65^26^3, 1234^256$, \\ $1^22^33^44^556^3, 234^26, 34^26,
	1^22^23^24^356^2, 12346, 23^24^356^2, 23456, 3456, 456, 6 \}$}  
and the degree of the integral is
\begin{equation*}
\ya= 102\alpha_1 + 156\alpha_2 + 206\alpha_3 + 252\alpha_4 + 88\alpha_5 + 128\alpha_6.
\end{equation*}

\subsubsection{The Dynkin diagram \emph{(\ref{eq:dynkin-g(6,6)}
		s)}}\label{subsubsec:g(6,6)-s}

\

The Nichols algebra $\toba_{\bq}$ is generated by $(x_i)_{i\in \I_6}$ with defining relations
\begin{align}\label{eq:rels-g(6,6)-s}
\begin{aligned}
x_{221}&=0; & x_{223}&=0; & [x_{(24)},&x_3]_c=0; & & x_{ij}=0, \ i<j, \, \widetilde{q}_{ij}=1;\\
x_{443}&=0; & x_{445}&=0; & [x_{236},&x_3]_c=0; & & [[[x_{5436},x_3]_c,x_4]_c,x_3]_c=0;\\
x_{554}&=0; & x_{663}&=0; & x_{1}^2&=0; &  & x_{3}^2=0; \quad x_{\alpha}^{3}=0, \ \alpha\in\Oc_+^{\bq}.
\end{aligned}
\end{align}
Here {\scriptsize$\Oc_+^{\bq}=\{ 2, 123, 1234, 4, 12345,
	45, 5, 12^33^34^256, 12^23^34^256, 2^23^24^256, 2^23^2456, 2^23^246,
	12^33^54^45^26^2$, \\ $12^33^54^35^26^2, 12^33^54^356^2, 123^34^256, 23^24^256,
	1^22^43^64^45^26^3, 12^33^34^256^2, 123456, 12346$, \\ $12^33^54^45^26^3, 2^23^44^35^26^2, 2^23^44^356^2,
	3^24^256, 12^23^34^256^2, 12^33^54^35^26^3, 12^33^54^356^3$, \\ $123^34^256^2,
	1236, 23^2456, 2^23^44^256^2, 3^2456, 23^246, 3^246, 6 \}$}  
and the degree of the integral is
\begin{equation*}
\ya= 56\alpha_1 + 156\alpha_2 + 252\alpha_3 + 172\alpha_4 + 88\alpha_5 + 128\alpha_6.
\end{equation*}

\subsubsection{The Dynkin diagram \emph{(\ref{eq:dynkin-g(6,6)} t)}}
\label{subsubsec:g(6,6)-t}

\

The Nichols algebra $\toba_{\bq}$ is generated by $(x_i)_{i\in \I_6}$ with defining relations
\begin{align}\label{eq:rels-g(6,6)-t}
\begin{aligned}
x_{443}&=0; & x_{445}&=0; & [x_{(13)},&x_2]_c=0; & & x_{ij}=0, \ i<j, \, \widetilde{q}_{ij}=1;\\
x_{554}&=0; &  x_{663}&=0; & [x_{(24)},&x_3]_c=0; & & [[[x_{5436},x_3]_c,x_4]_c,x_3]_c=0;\\
x_{1}^2&=0; & x_{2}^2&=0; & [x_{236},&x_3]_c=0; & & x_{3}^2=0; \quad x_{\alpha}^{3}=0, \ \alpha\in\Oc_+^{\bq}.
\end{aligned}
\end{align}
Here {\scriptsize$\Oc_+^{\bq}=\{ 12, 23, 234, 4, 2345,
	45, 5, 1^22^33^34^256, 1^22^23^24^256, 1^22^23^2456, 1^22^23^246, 12^23^34^256$, \\ $1^22^33^54^45^26^2, 1^22^33^54^35^26^2, 1^22^33^54^356^2, 23^34^256, 123^24^256,
	1^22^43^64^45^26^3, 1^22^33^34^256^2$, \\ $1^22^23^44^35^26^2, 1^22^23^44^356^2, 123^2456, 1^22^33^54^45^26^3, 3^24^256,
	23456, 1^22^33^54^35^26^3$, \\ $1^22^23^44^256^2, 3^2456, 12^23^34^256^2,
	1^22^33^54^356^3, 123^246, 23^34^256^2, 2346, 236, 3^246, 6 \}$}  
and the degree of the integral is
\begin{equation*}
\ya= 102\alpha_1 + 156\alpha_2 + 252\alpha_3 + 172\alpha_4 + 88\alpha_5 + 128\alpha_6.
\end{equation*}

\subsubsection{The associated Lie algebra} This is of type $B_6$.

\subsection{Type $\gtt(8,6)$}\label{subsec:type-g(8,6)}
Here $\theta = 7$, $\zeta \in \G'_3$. Let 
\begin{align*}
A&=\begin{pmatrix} 2 & \text{--}1 & 0 & 0 & 0 & 0 & 0 \\
\text{--}1 & 2 & \text{--}1 & 0 & 0 & 0 & 0 \\ 0 & 1 & 0 & \text{--}1 & 0 & 0 & 0 \\ 0 & 0 & \text{--}1 & 0 & 1 & 0 & 1
\\ 0 & 0 & 0 & \text{--}1 & 2 & \text{--}1 & 0 \\ 0 & 0 & 0 & 0 & \text{--}1 & 2 & 0 \\ 0 & 0 & 0 & \text{--}1 & 0 & 0 & 2 \end{pmatrix}
\in \kk^{7\times 7}; & \pa &= (1,1, -1,-1, 1, 1, 1) \in \G_2^7.
\end{align*}
Let $\g(8,6) = \g(A, \pa)$,
the contragredient Lie superalgebra corresponding to $(A, \pa)$. 
We know \cite{BGL} that $\sdim \g(8,6) = 133|56$. 
There are 7 other pairs of matrices and parity vectors for which the associated contragredient Lie superalgebra is isomorphic to $\g(8,6)$.
We describe now the root system $\gtt(8,6)$ of $\g(8,6)$, see \cite{AA-GRS-CLS-NA} for details.

\subsubsection{Basic datum and root system}
Below, $A_7$, $D_7$, $E_7$ and $_{3}T_1$ are numbered as in \eqref{eq:dynkin-system-A}, \eqref{eq:dynkin-system-D}, \eqref{eq:dynkin-system-E} and \eqref{eq:mTn}, respectively.
The basic datum and the bundle of Cartan matrices are described by the following diagram:

\begin{center}
	\begin{tabular}{c c c c c c c c c c}
		& &
		& $\overset{E_7}{\underset{a_1}{\vtxgpd}}$
		& \hspace{-7pt}\raisebox{3pt}{$\overset{3}{\rule{27pt}{0.5pt}}$}\hspace{-7pt}
		& $\overset{E_7}{\underset{a_2}{\vtxgpd}}$
		& \hspace{-7pt}\raisebox{3pt}{$\overset{2}{\rule{27pt}{0.5pt}}$}\hspace{-7pt}
		& $\overset{E_7}{\underset{a_3}{\vtxgpd}}$
		& \hspace{-7pt}\raisebox{3pt}{$\overset{1}{\rule{27pt}{0.5pt}}$}\hspace{-7pt}
		& $\overset{E_7}{\underset{a_4}{\vtxgpd}}$
		\\
		& & & {\scriptsize 4} \vline\hspace{5pt} & & & & & &
		\\
		& $\overset{s_{567}(A_7)}{\underset{a_5}{\vtxgpd}}$
		& \hspace{-7pt}\raisebox{3pt}{$\overset{7}{\rule{27pt}{0.5pt}}$}\hspace{-7pt}
		& $\overset{{}_3T_1}{\underset{a_6}{\vtxgpd}}$
		& \hspace{-7pt}\raisebox{3pt}{$\overset{5}{\rule{27pt}{0.5pt}}$}\hspace{-7pt}
		& $\overset{D_7}{\underset{a_7}{\vtxgpd}}$
		& \hspace{-7pt}\raisebox{3pt}{$\overset{6}{\rule{27pt}{0.5pt}}$}\hspace{-7pt}
		& $\overset{D_7}{\underset{a_8}{\vtxgpd}}$
		& &
	\end{tabular}
\end{center}
Using the notation \eqref{eq:notation-root-exceptional}, the bundle of root sets is the following: { \scriptsize
	\begin{align*}
	\varDelta_{+}^{a_1}& = \{ 1, 12, 2, 123, 23, 3, 1234, 234, 34, 4, 12345, 2345, 345, 45, 5, 12^23^24^256, 123^24^256, \\
	& 1234^256, 123456, 12346, 23^24^256, 234^256, 23456, 2346, 34^256, 3456, 346, 4^256, 456, 46, 6, \\
	& 12^23^34^55^36^37, 12^23^34^55^26^37, 12^23^34^45^26^37, 12^23^24^45^26^37, 123^24^45^26^37, 23^24^45^26^37, \\
	& 12^23^34^45^26^27, 12^23^24^45^26^27, 123^24^45^26^27, 23^24^45^26^27, 12^23^34^35^26^27, 12^23^24^35^26^27, \\
	& 123^24^35^26^27, 23^24^35^26^27, 1234^35^26^27, 234^35^26^27, 34^35^26^27, 1^22^33^44^65^36^47^2, 12^33^44^65^36^47^2, \\
	& 12^23^44^65^36^47^2, 12^23^34^65^36^47^2, 12^23^34^55^36^47^2, 12^23^34^55^36^37^2, 12^23^34^356^27, \\
	& 12^23^24^356^27, 12^23^24^256^27, 12^23^24^2567, 123^24^356^27, 123^24^256^27, 123^24^2567, 1234^356^27, \\
	& 1234^256^27, 1234^2567, 1234567, 12^23^34^55^26^47^2, 12^23^34^55^26^37^2, 12^23^34^45^26^37^2, 12^23^24^45^26^37^2, \\
	& 123^24^45^26^37^2, 123467, 23^24^356^27, 234^356^27, 34^356^27, 23^24^256^27, 234^256^27, 34^256^27, 4^256^27, \\
	& 23^24^45^26^37^2, 23^24^2567, 234^2567, 234567, 23467, 34^2567, 34567, 3467, 4^2567, 4567, 467, 67, 7 \}, \\
	\varDelta_{+}^{a_2}& = \{ 1, 12, 2, 123, 23, 3, 1234, 234, 34, 4, 12345, 2345, 345, 45, 5, 12^23^24^256, 123^24^256, \\
	& 1234^256, 123456, 12346, 23^24^256, 234^256, 23456, 2346, 3^24^256, 34^256, 3456, 346, 456, 46, 6, \\
	& 12^23^44^55^36^37, 12^23^44^55^26^37, 12^23^44^45^26^37, 12^23^34^45^26^37, 123^34^45^26^37, 23^34^45^26^37, \\
	& 12^23^44^45^26^27, 12^23^34^45^26^27, 123^34^45^26^27, 23^34^45^26^27, 12^23^34^35^26^27, 123^34^35^26^27, \\
	& 23^34^35^26^27, 12^23^24^35^26^27, 123^24^35^26^27, 23^24^35^26^27, 3^24^35^26^27, 1^22^33^54^65^36^47^2, \\
	& 12^33^54^65^36^47^2, 12^23^54^65^36^47^2, 12^23^44^65^36^47^2, 12^23^44^55^36^47^2, 12^23^44^55^36^37^2, \\
	& 12^23^34^356^27, 12^23^24^356^27, 12^23^24^256^27, 12^23^24^2567, 123^34^356^27, 123^24^356^27, 123^24^256^27, \\
	& 123^24^2567, 1234^256^27, 1234^2567, 1234567, 12^23^44^55^26^47^2, 12^23^44^55^26^37^2, 12^23^44^45^26^37^2, \\
	& 12^23^34^45^26^37^2, 123^34^45^26^37^2, 123467, 23^34^356^27, 23^24^356^27, 3^24^356^27, 23^24^256^27, \\
	& 3^24^256^27, 234^256^27, 34^256^27, 23^34^45^26^37^2, 23^24^2567, 234^2567, 234567, 23467, 3^24^2567, \\
	& 34^2567, 34567, 3467, 4567, 467, 67, 7 \}, \\
	\varDelta_{+}^{a_3}& = 
	\{ 1, 12, 2, 123, 23, 3, 1234, 234, 34, 4, 12345, 2345, 345, 45, 5, 12^23^24^256, 123^24^256, \\
	& 1234^256, 123456, 12346, 2^23^24^256, 23^24^256, 234^256, 23456, 2346, 34^256, 3456, 346, 456, \\
	& 46, 6, 12^33^44^55^36^37, 12^33^44^55^26^37, 12^33^44^45^26^37, 12^33^34^45^26^37, 12^23^34^45^26^37, \\
	& 2^23^34^45^26^37, 12^33^44^45^26^27, 12^33^34^45^26^27, 12^23^34^45^26^27, 2^23^34^45^26^27, 12^33^34^35^26^27, \\
	& 12^23^34^35^26^27, 2^23^34^35^26^27, 12^23^24^35^26^27, 2^23^24^35^26^27, 123^24^35^26^27, 23^24^35^26^27, \\
	& 1^22^43^54^65^36^47^2, 12^43^54^65^36^47^2, 12^33^54^65^36^47^2, 12^33^44^65^36^47^2, 12^33^44^55^36^47^2, \\
	& 12^33^44^55^36^37^2, 12^33^34^356^27, 12^23^34^356^27, 12^23^24^356^27, 12^23^24^256^27, 12^23^24^2567, \\
	& 123^24^356^27, 123^24^256^27, 123^24^2567, 1234^256^27, 1234^2567, 1234567, 12^33^44^55^26^47^2, \\
	& 12^33^44^55^26^37^2, 12^33^44^45^26^37^2, 12^33^34^45^26^37^2, 12^23^34^45^26^37^2, 123467, 2^23^34^356^27, \\
	& 2^23^24^356^27, 23^24^356^27, 2^23^24^256^27, 23^24^256^27, 234^256^27, 34^256^27, 2^23^34^45^26^37^2, \\
	& 2^23^24^2567, 23^24^2567, 234^2567, 234567, 23467, 34^2567, 34567, 3467, 4567, 467, 67, 7 \}, \\
	\varDelta_{+}^{a_4}& = \{ 1, 12, 2, 123, 23, 3, 1234, 234, 34, 4, 12345, 2345, 345, 45, 5, 1^22^23^24^256, 12^23^24^256, \\
	& 123^24^256, 1234^256, 123456, 12346, 23^24^256, 234^256, 23456, 2346, 34^256, 3456, 346, 456, \\
	& 46, 6, 1^22^33^44^55^36^37, 1^22^33^44^55^26^37, 1^22^33^44^45^26^37, 1^22^33^34^45^26^37, 1^22^23^34^45^26^37, \\
	& 12^23^34^45^26^37, 1^22^33^44^45^26^27, 1^22^33^34^45^26^27, 1^22^23^34^45^26^27, 12^23^34^45^26^27, \\
	& 1^22^33^34^35^26^27, 1^22^23^34^35^26^27, 12^23^34^35^26^27, 1^22^23^24^35^26^27, 12^23^24^35^26^27, \\
	& 123^24^35^26^27, 23^24^35^26^27, 1^32^43^54^65^36^47^2, 1^22^43^54^65^36^47^2, 1^22^33^54^65^36^47^2, \\
	& 1^22^33^44^65^36^47^2, 1^22^33^44^55^36^47^2, 1^22^33^44^55^36^37^2, 1^22^33^34^356^27, 1^22^23^34^356^27, \\
	& 1^22^23^24^356^27, 1^22^23^24^256^27, 1^22^23^24^2567, 12^23^34^356^27, 12^23^24^356^27, 12^23^24^256^27, \\
	& 12^23^24^2567, 1^22^33^44^55^26^47^2, 1^22^33^44^55^26^37^2, 123^24^356^27, 23^24^356^27, \\
	& 1^22^33^44^45^26^37^2, 123^24^256^27, 23^24^256^27, 123^24^2567, 23^24^2567, 1^22^33^34^45^26^37^2, \\
	& 1^22^23^34^45^26^37^2, 12^23^34^45^26^37^2, 1234^256^27, 1234^2567, 1234567, 123467, 234^256^27, \\
	& 234^2567, 234567, 23467, 34^256^27, 34^2567, 34567, 3467, 4567, 467, 67, 7 \}, \\
	\varDelta_{+}^{a_5}& = s_{567}( \{ 1, 12, 2, 123, 23, 3, 1234, 234, 34, 4, 12345, 2345, 345, 45, 5, 12^23^24^25^26, 123^24^25^26, \\
	& 1234^25^26, 12345^26, 123456, 23^24^25^26, 234^25^26, 2345^26, 23456, 34^25^26, 345^26, 3456, 45^26, \\
	& 456, 56, 6, 12^23^34^45^56^37, 12^23^34^45^46^37, 12^23^34^35^46^37, 12^23^24^35^46^37, 123^24^35^46^37, \\
	& 23^24^35^46^37, 12^23^34^45^46^27, 12^23^34^35^46^27, 12^23^24^35^46^27, 123^24^35^46^27, 23^24^35^46^27, \\
	& 12^23^34^35^36^27, 12^23^24^35^36^27, 123^24^35^36^27, 23^24^35^36^27, 1^22^33^44^55^66^47^2, 12^33^44^55^66^47^2, \\
	& 12^23^24^25^36^27, 12^23^24^25^26^27, 12^23^24^25^267, 12^23^44^55^66^47^2, 123^24^25^36^27, 123^24^25^26^27, \\
	& 123^24^25^267, 23^24^25^36^27, 23^24^25^26^27, 23^24^25^267, 12^23^34^55^66^47^2, 12^23^34^45^66^47^2, \\
	& 12^23^34^45^56^47^2, 12^23^34^45^56^37^2, 1234^25^36^27, 234^25^36^27, 34^25^36^27, 12^23^34^45^46^37^2, \\
	& 1234^25^26^27, 234^25^26^27, 34^25^26^27, 1234^25^267, 234^25^267, 34^25^267, 12^23^34^35^46^37^2, \\
	& 12^23^24^35^46^37^2, 123^24^35^46^37^2, 12345^26^27, 12345^267, 1234567, 23^24^35^46^37^2, \\
	& 2345^26^27, 2345^267, 234567, 345^26^27, 345^267, 34567, 45^26^27, 45^267, 4567, 567, 67, 7 \}), \\
	\varDelta_{+}^{a_6}& = \{ 1, 12, 2, 123, 23, 3, 1234, 234, 34, 4, 12345, 2345, 345, 45, 5, 12^23^24^256, 123^24^256, 1234^256, \\
	& 123456, 12346, 23^24^256, 234^256, 23456, 2346, 34^256, 3456, 346, 456, 46, 56, 6, 12^23^34^45^36^37, \\
	& 12^23^34^45^26^37, 12^23^34^35^26^37, 12^23^24^35^26^37, 123^24^35^26^37, 23^24^35^26^37, 12^23^34^45^26^27, \\
	& 12^23^34^35^26^27, 12^23^24^35^26^27, 123^24^35^26^27, 23^24^35^26^27, 12^23^24^25^26^27, 123^24^25^26^27, \\
	& 23^24^25^26^27, 1234^25^26^27, 234^25^26^27, 34^25^26^27, 1^22^33^44^55^36^47^2, 12^33^44^55^36^47^2, \\
	& 12^23^44^55^36^47^2, 12^23^34^55^36^47^2, 12^23^34^45^36^47^2, 12^23^34^45^36^37^2, 12^23^34^356^27, 12^23^24^356^27, \\
	& 12^23^24^256^27, 12^23^24^2567, 123^24^356^27, 123^24^256^27, 123^24^2567, 1234^256^27, 1234^2567, \\
	& 123456^27, 1234567, 12^23^34^45^26^47^2, 12^23^34^45^26^37^2, 12^23^34^35^26^37^2, 12^23^24^35^26^37^2, \\
	& 123^24^35^26^37^2, 123467, 23^24^356^27, 23^24^256^27, 234^256^27, 34^256^27, 23456^27, 3456^27, 456^27, \\
	& 23^24^35^26^37^2, 23^24^2567, 234^2567, 234567, 23467, 34^2567, 34567, 3467, 4567, 467, 567, 67, 7 \}, \\
	\varDelta_{+}^{a_7}& = \{ 1, 12, 2, 123, 23, 3, 1234, 234, 34, 4, 12345, 2345, 345, 45, 5, 12^23^24^25^26, 123^24^25^26, \\
	& 1234^25^26, 12345^26, 123456, 23^24^25^26, 234^25^26, 2345^26, 23456, 34^25^26, 345^26, 3456, 45^26, 456, \\
	& 56, 6, 12^23^34^45^56^37, 12^23^34^45^56^27, 12^23^34^45^46^27, 12^23^34^35^46^27, 12^23^24^35^46^27, 123^24^35^46^27, \\
	& 23^24^35^46^27, 12^23^34^35^36^27, 12^23^24^35^36^27, 123^24^35^36^27, 23^24^35^36^27, 12^23^24^25^36^27, \\
	& 123^24^25^36^27, 23^24^25^36^27, 1234^25^36^27, 234^25^36^27, 34^25^36^27, 1^22^33^44^55^66^37^2, 12^33^44^55^66^37^2, \\
	& 12^23^44^55^66^37^2, 12^23^34^55^66^37^2, 12^23^34^45^66^37^2, 12^23^34^45^56^37^2, 12^23^34^35^367, 12^23^24^35^367, \\
	& 12^23^24^25^367, 12^23^24^25^267, 123^24^35^367, 123^24^25^367, 123^24^25^267, 1234^25^367, 1234^25^267, \\
	& 12345^267, 1234567, 12^23^34^45^56^27^2, 12^23^34^45^46^27^2, 12^23^34^35^46^27^2, 12^23^24^35^46^27^2, \\
	& 123^24^35^46^27^2, 123457, 23^24^35^367, 23^24^25^367, 234^25^367, 34^25^367, 23^24^25^267, 234^25^267, \\
	& 34^25^267, 23^24^35^46^27^2, 2345^267, 234567, 23457, 345^267, 34567, 3457, 45^267, 4567, 457, 567, 57, 7 \}, \\
	\varDelta_{+}^{a_8}& = \{ 1, 12, 2, 123, 23, 3, 1234, 234, 34, 4, 12345, 2345, 345, 45, 5, 123456, 23456, 3456, 456, 56, 6, \\
	& 12^23^24^25^267, 123^24^25^267, 1234^25^267, 12345^267, 1234567, 123457, 12^23^34^45^46^27^2, \\
	& 12^23^34^35^46^27^2, 12^23^34^35^36^27^2, 12^23^34^35^367^2, 23^24^25^267, 12^23^24^35^46^27^2, 12^23^24^35^36^27^2, \\
	& 12^23^24^35^367^2, 234^25^267, 12^23^24^25^36^27^2, 12^23^24^25^367^2, 2345^267, 12^23^24^25^267^2, 234567, \\
	& 23457, 1^22^33^44^55^66^37^4, 12^33^44^55^66^37^4, 12^23^34^45^56^37^3, 12^23^34^45^56^27^3, \\
	& 12^23^34^45^46^27^3, 12^23^34^35^46^27^3, 12^23^24^35^46^27^3, 123^24^35^46^27^2, 123^24^35^36^27^2, 123^24^25^36^27^2, \\
	& 1234^25^36^27^2, 123^24^35^367^2, 123^24^25^367^2, 1234^25^367^2, 123^24^25^267^2, 1234^25^267^2, 12345^267^2, \\
	& 12^23^44^55^66^37^4, 12^23^34^55^66^37^4, 12^23^34^45^66^37^4, 12^23^34^45^56^37^4, 12^23^34^45^56^27^4, \\
	& 123^24^35^46^27^3, 23^24^35^46^27^2, 23^24^35^36^27^2, 23^24^35^367^2, 34^25^267, 23^24^25^36^27^2, 23^24^25^367^2, \\
	& 345^267, 23^24^25^267^2, 34567, 3457, 23^24^35^46^27^3, 234^25^36^27^2, 234^25^367^2, 234^25^267^2, 2345^267^2, \\
	& 34^25^36^27^2, 34^25^367^2, 34^25^267^2, 345^267^2, 45^267, 4567, 567, 45^267^2, 457, 57, 7 \}.
	\end{align*}
}%

\subsubsection{Weyl groupoid}
\label{subsubsec:type-g86-Weyl}
The isotropy group  at $a_4 \in \cX$ is
\begin{align*}
\cW(a_4)= \langle \varsigma_1^{a_4} \varsigma_2 \varsigma_3\varsigma_4 \varsigma_7 \varsigma_6 \varsigma_5 \varsigma_6 \varsigma_7\varsigma_4 \varsigma_3 \varsigma_2 \varsigma_1, \varsigma_2^{a_4},  \varsigma_3^{a_4}, 
\varsigma_4^{a_4},  \varsigma_5^{a_4},
\varsigma_6^{a_4},  \varsigma_7^{a_4} \rangle \simeq W(E_7).
\end{align*}

\subsubsection{Incarnation}
We set the matrices $(\bq^{(i)})_{i\in\I_{8}}$, from left to right and  from up to down:
\begin{align}\label{eq:dynkin-g(8,6)}
\begin{aligned}
& \xymatrix@R-8pt{  &   & & \overset{\zeta}{\circ} \ar  @{-}[d]^{\ztu}  & &\\
	\overset{\zeta}{\underset{\ }{\circ}} \ar  @{-}[r]^{\ztu}  & \overset{\zeta}{\underset{\
		}{\circ}} \ar  @{-}[r]^{\ztu}  & \overset{-1}{\underset{\ }{\circ}} \ar  @{-}[r]^{\zeta}
	& \overset{-1}{\underset{\ }{\circ}}
	\ar  @{-}[r]^{\ztu}  & \overset{\zeta}{\underset{\ }{\circ}} \ar  @{-}[r]^{\ztu}  &
	\overset{\zeta}{\underset{\ }{\circ}}}
\\
& \xymatrix@R-8pt{  &   & & \overset{\zeta}{\circ} \ar  @{-}[d]^{\ztu}  & & \\
	\overset{\zeta}{\underset{\ }{\circ}} \ar  @{-}[r]^{\ztu}  & \overset{-1}{\underset{\
		}{\circ}} \ar  @{-}[r]^{\zeta}  & \overset{-1}{\underset{\ }{\circ}} \ar  @{-}[r]^{\ztu}
	& \overset{\zeta}{\underset{\ }{\circ}}
	\ar  @{-}[r]^{\ztu}  & \overset{\zeta}{\underset{\ }{\circ}} \ar  @{-}[r]^{\ztu}  &
	\overset{\zeta}{\underset{\ }{\circ}}}
\\
& \xymatrix@R-8pt{  &   & & \overset{\zeta}{\circ} \ar  @{-}[d]^{\ztu}    & &\\
	\overset{-1}{\underset{\ }{\circ}} \ar  @{-}[r]^{\zeta}  & \overset{-1}{\underset{\
		}{\circ}} \ar  @{-}[r]^{\ztu}  & \overset{\zeta}{\underset{\ }{\circ}} \ar  @{-}[r]^{\ztu}
	& \overset{\zeta}{\underset{\ }{\circ}}
	\ar  @{-}[r]^{\ztu}  & \overset{\zeta}{\underset{\ }{\circ}} \ar  @{-}[r]^{\ztu}  &
	\overset{\zeta}{\underset{\ }{\circ}}}
\\
& \xymatrix@R-8pt{  &   & & \overset{\zeta}{\circ} \ar  @{-}[d]^{\ztu}    & &\\
	\overset{-1}{\underset{\ }{\circ}} \ar  @{-}[r]^{\ztu}  & \overset{\zeta}{\underset{\
		}{\circ}} \ar  @{-}[r]^{\ztu}  & \overset{\zeta}{\underset{\ }{\circ}} \ar  @{-}[r]^{\ztu}
	& \overset{\zeta}{\underset{\ }{\circ}}
	\ar  @{-}[r]^{\ztu}  & \overset{\zeta}{\underset{\ }{\circ}} \ar  @{-}[r]^{\ztu}  &
	\overset{\zeta}{\underset{\ }{\circ}}}
\end{aligned}
\end{align}

\begin{align*}
& \xymatrix{\overset{\zeta}{\underset{\ }{\circ}}\ar  @{-}[r]^{\ztu}  &
	\overset{\zeta}{\underset{\ }{\circ}} \ar  @{-}[r]^{\ztu}  & \overset{\zeta}{\underset{\
		}{\circ}}
	\ar  @{-}[r]^{\ztu}  & \overset{\zeta}{\underset{\ }{\circ}}
	\ar  @{-}[r]^{\ztu}  & \overset{-1}{\underset{\ }{\circ}}  \ar  @{-}[r]^{\ztu}  &
	\overset{\zeta}{\underset{\ }{\circ}}  \ar  @{-}[r]^{\ztu}  & \overset{\zeta}{\underset{\
		}{\circ}}
}
\\
& \xymatrix@R-8pt{  &   & & \overset{-1}{\circ} \ar  @{-}[d]^{\zeta}\ar  @{-}[dr]^{\zeta}
	& &\\
	\overset{\zeta}{\underset{\ }{\circ}} \ar  @{-}[r]^{\ztu}  & \overset{\zeta}{\underset{\
		}{\circ}} \ar  @{-}[r]^{\ztu}  & \overset{\zeta}{\underset{\ }{\circ}} \ar  @{-}[r]^{\ztu}
	& \overset{-1}{\underset{\ }{\circ}}
	\ar  @{-}[r]^{\zeta}  & \overset{-1}{\underset{\ }{\circ}} \ar  @{-}[r]^{\ztu}  &
	\overset{\zeta}{\underset{\ }{\circ}}}
\\
& \xymatrix@R-8pt{  &  & & & \overset{\zeta}{\circ} \ar  @{-}[d]^{\ztu}  & \\
	\overset{\zeta}{\underset{\ }{\circ}} \ar  @{-}[r]^{\ztu}  & \overset{\zeta}{\underset{\
		}{\circ}} \ar  @{-}[r]^{\ztu}  & \overset{\zeta}{\underset{\ }{\circ}} \ar  @{-}[r]^{\ztu}
	& \overset{\zeta}{\underset{\ }{\circ}}
	\ar  @{-}[r]^{\ztu}  & \overset{-1}{\underset{\ }{\circ}} \ar  @{-}[r]^{\zeta}  &
	\overset{-1}{\underset{\ }{\circ}}}
\\
& \xymatrix@R-8pt{  &  & & & \overset{\zeta}{\circ} \ar  @{-}[d]^{\ztu}  & \\
	\overset{\zeta}{\underset{\ }{\circ}} \ar  @{-}[r]^{\ztu}  & \overset{\zeta}{\underset{\
		}{\circ}} \ar  @{-}[r]^{\ztu}  & \overset{\zeta}{\underset{\ }{\circ}} \ar  @{-}[r]^{\ztu}
	& \overset{\zeta}{\underset{\ }{\circ}}
	\ar  @{-}[r]^{\ztu}  & \overset{\zeta}{\underset{\ }{\circ}} \ar  @{-}[r]^{\ztu}  &
	\overset{-1}{\underset{\ }{\circ}}}
\end{align*}
Now, this is the incarnation:
\begin{align*}
& a_5\mapsto s_{567}(\bq^{5}); &
& a_i\mapsto \bq^{(i)}, \ i\neq 5.
\end{align*}

\subsubsection{PBW-basis and dimension} \label{subsubsec:type-g86-PBW}
Notice that the roots in each $\varDelta_{+}^{a_i}$, $i\in\I_{8}$, are ordered from left to right, justifying the notation $\beta_1, \dots, \beta_{91}$.

The root vectors $x_{\beta_k}$ are described as in Remark \ref{rem:lyndon-word}.
Thus
\begin{align*}
\left\{ x_{\beta_{91}}^{n_{91}} \dots x_{\beta_2}^{n_{2}}  x_{\beta_1}^{n_{1}} \, | \, 0\le n_{k}<N_{\beta_k} \right\}.
\end{align*}
is a PBW-basis of $\toba_{\bq}$. Hence $\dim \toba_{\bq}=2^{28}3^{63}$.

\subsubsection{The Dynkin diagram \emph{(\ref{eq:dynkin-g(8,6)}
		a)}}\label{subsubsec:g(8,6)-a}

\

The Nichols algebra $\toba_{\bq}$ is generated by $(x_i)_{i\in \I_7}$ with defining relations
\begin{align}\label{eq:rels-g(8,6)-a}
\begin{aligned}
x_{112}&=0; & x_{221}&=0; & x_{223}&=0; & & x_{ij}=0, \ i<j, \, \widetilde{q}_{ij}=1;\\
& & x_{332}&=0; & x_{334}&=0; &  & [[[x_{(36)},x_5]_c,x_4]_c,x_5]_c=0;\\
x_{443}&=0; & x_{445}&=0; & x_{776}&=0; & & [[[x_{7654},x_5]_c,x_6]_c,x_5]_c=0; \\
x_{665}&=0; & x_{667}&=0; & x_{5}^2&=0; & &x_{\alpha}^{3}=0, \ \alpha\in\Oc_+^{\bq}.
\end{aligned}
\end{align}
Here {\scriptsize$\Oc_+^{\bq}=\{ 1, 12, 2, 1234, 234, 34, 12345, 2345, 345, 5, 12^23^24^256, 123^24^256, 123456, 12346, 23^24^256$, \\ $23456, 2346, 3456, 346, 4^256, 6, 12^23^34^55^36^37, 12^23^34^55^26^37, 12^23^24^45^26^37, 123^24^45^26^37$, \\ $23^24^45^26^37, 12^23^24^45^26^27, 123^24^45^26^27, 23^24^45^26^27, 12^23^34^35^26^27, 1234^35^26^27, 234^35^26^27$, \\ $34^35^26^27, 1^22^33^44^65^36^47^2, 12^33^44^65^36^47^2, 12^23^44^65^36^47^2, 12^23^34^55^36^47^2, 12^23^34^55^36^37^2$, \\ $12^23^34^356^27, 12^23^24^256^27, 12^23^24^2567, 123^24^256^27, 123^24^2567, 1234^356^27, 1234567$, \\ $12^23^34^55^26^47^2, 12^23^34^55^26^37^2, 12^23^24^45^26^37^2, 123^24^45^26^37^2, 123467, 234^356^27$, \\ $34^356^27, 23^24^256^27, 4^256^27, 23^24^45^26^37^2,
	23^24^2567, 234567, 23467, 34567, 3467, 4^2567, 67, 7 \}$}  
and the degree of the integral is
\begin{equation*}
\ya= 80\alpha_1 + 156\alpha_2 + 228\alpha_3 + 360\alpha_4 + 244\alpha_5 + 124\alpha_6 + 182\alpha_7.
\end{equation*}

\subsubsection{The Dynkin diagram \emph{(\ref{eq:dynkin-g(8,6)}
		b)}}\label{subsubsec:g(8,6)-b}

\

The Nichols algebra $\toba_{\bq}$ is generated by $(x_i)_{i\in \I_7}$ with defining relations
\begin{align}\label{eq:rels-g(8,6)-b}
\begin{aligned}
& \begin{aligned}
x_{112}&=0; & x_{221}&=0; & [x_{(35)},&x_4]_c=0; & & x_{ij}=0, \ i<j, \, \widetilde{q}_{ij}=1;\\
x_{223}&=0; & x_{332}&=0; & [x_{347},&x_4]_c=0; & & [x_{(46)},x_5]_c=0; \\ 
x_{334}&=0; & x_{665}&=0; & [x_{657},&x_5]_c=0; &
& x_{4}^2=0; \quad x_{\alpha}^{3}=0, \ \alpha\in\Oc_+^{\bq};
\end{aligned}
\\
& x_{5}^2=0; \quad x_{7}^2=0; \quad x_{457}= q_{57}\ztu[x_{47},x_5]_c +q_{45}(1-\zeta)x_5x_{47}.
\end{aligned}
\end{align}
Here {\scriptsize$\Oc_+^{\bq}=\{ 1, 123, 23, 1234, 234, 4, 12345, 2345, 45, 5, 12^23^24^256, 1234^256, 123456, 12346, 234^256, 23456$, \\ $2346, 3^24^256, 456, 46, 6, 12^23^44^55^36^37, 12^23^44^55^26^37, 12^23^44^45^26^37, 123^34^45^26^37, 23^34^45^26^37$, \\ $12^23^44^45^26^27, 123^34^45^26^27, 23^34^45^26^27, 123^34^35^26^27, 23^34^35^26^27, 12^23^24^35^26^27, 3^24^35^26^27$, \\ $1^22^33^54^65^36^47^2, 12^33^54^65^36^47^2, 12^23^44^65^36^47^2, 12^23^44^55^36^47^2, 12^23^44^55^36^37^2, 12^23^24^356^27$, \\ $12^23^24^256^27, 12^23^24^2567, 123^34^356^27, 1234^256^27, 1234^2567,
	1234567, 12^23^44^55^26^47^2$, \\ $12^23^44^55^26^37^2, 12^23^44^45^26^37^2, 123^34^45^26^37^2, 123467, 23^34^356^27, 3^24^356^27, 3^24^256^27$, \\ $234^256^27, 23^34^45^26^37^2, 234^2567, 234567, 23467, 3^24^2567, 4567, 467, 67, 7 \}$}  
and the degree of the integral is
\begin{equation*}
\ya= 80\alpha_1 + 156\alpha_2 + 290\alpha_3 + 360\alpha_4 + 244\alpha_5 + 124\alpha_6 + 182\alpha_7.
\end{equation*}

\subsubsection{The Dynkin diagram \emph{(\ref{eq:dynkin-g(8,6)}
		c)}}\label{subsubsec:g(8,6)-c}

\

The Nichols algebra $\toba_{\bq}$ is generated by $(x_i)_{i\in \I_7}$ with defining relations
\begin{align}\label{eq:rels-g(8,6)-c}
\begin{aligned}
x_{112}&=0; & x_{221}&=0; & x_{223}&=0; & & x_{ij}=0, \ i<j, \, \widetilde{q}_{ij}=1;\\
& & x_{332}&=0; & [x_{(35)},&x_4]_c=0; &  & [[x_{65},x_{657}]_c,x_5]_c=0;\\
x_{334}&=0; & x_{776}&=0; & [x_{(46)},&x_5]_c=0; & & [x_{457},x_5]_c=0; \\
x_{4}^2&=0; & x_{5}^2&=0; & x_{6}^2&=0; & & x_{\alpha}^{3}=0, \ \alpha\in\Oc_+^{\bq}.
\end{aligned}
\end{align}
Here {\scriptsize$\Oc_+^{\bq}=\{ 12, 123, 3, 1234, 34, 4, 12345, 345, 45, 5, 123^24^256,
	1234^256, 123456, 12346, 2^23^24^256$, \\ $34^256, 3456, 346, 456, 46, 6, 12^33^44^55^36^37, 12^33^44^55^26^37, 12^33^44^45^26^37, 12^33^34^45^26^37$, \\ $2^23^34^45^26^37, 12^33^44^45^26^27, 12^33^34^45^26^27, 2^23^34^45^26^27, 12^33^34^35^26^27, 2^23^34^35^26^27$, \\ $2^23^24^35^26^27, 123^24^35^26^27,
	1^22^43^54^65^36^47^2, 12^33^54^65^36^47^2, 12^33^44^65^36^47^2, 12^33^44^55^36^47^2$, \\ $12^33^44^55^36^37^2, 12^33^34^356^27, 123^24^356^27, 123^24^256^27, 123^24^2567, 1234^256^27, 1234^2567$, \\ $1234567, 12^33^44^55^26^47^2, 12^33^44^55^26^37^2, 12^33^44^45^26^37^2, 12^33^34^45^26^37^2, 123467, 2^23^34^356^27$, \\ $2^23^24^356^27, 2^23^24^256^27, 34^256^27, 2^23^34^45^26^37^2,
	2^23^24^2567, 34^2567, 34567, 3467, 4567,
	467, 67, 7 \}$}  
and the degree of the integral is
\begin{equation*}
\ya= 80\alpha_1 + 216\alpha_2 + 290\alpha_3 + 360\alpha_4 + 244\alpha_5 + 124\alpha_6 + 182\alpha_7.
\end{equation*}

\subsubsection{The Dynkin diagram \emph{(\ref{eq:dynkin-g(8,6)}
		d)}}\label{subsubsec:g(8,6)-d}

\

The Nichols algebra $\toba_{\bq}$ is generated by $(x_i)_{i\in \I_7}$ with defining relations
\begin{align}\label{eq:rels-g(8,6)-d}
\begin{aligned}
x_{112}&=0; & x_{221}&=0; & x_{223}&=0; & x_{ij}&=0, & &i<j, \, \widetilde{q}_{ij}=1;\\
x_{332}&=0; & x_{334}&=0; & x_{443}&=0; & x_{445}&=0; & &x_{554}=0; \\
x_{556}&=0; & x_{557}&=0; & x_{775}&=0; & 
x_{6}^2&=0; & & x_{\alpha}^{3}=0, \ \alpha\in\Oc_+^{\bq}.
\end{aligned}
\end{align}
Here {\scriptsize$\Oc_+^{\bq}=\{ 2, 23, 3, 234, 34, 4, 2345, 345, 45, 5, 1^22^23^24^256,
	23^24^256, 234^256, 23456, 2346, 34^256, 3456$, \\ $346, 456, 46, 6, 1^22^33^44^55^36^37, 1^22^33^44^55^26^37, 1^22^33^44^45^26^37, 1^22^33^34^45^26^37, 1^22^23^34^45^26^37$, \\ $1^22^33^44^45^26^27, 1^22^33^34^45^26^27, 1^22^23^34^45^26^27, 1^22^33^34^35^26^27, 1^22^23^34^35^26^27, 1^22^23^24^35^26^27$, \\ $23^24^35^26^27, 1^22^43^54^65^36^47^2, 1^22^33^54^65^36^47^2, 1^22^33^44^65^36^47^2, 1^22^33^44^55^36^47^2, 1^22^33^44^55^36^37^2$, \\ $1^22^33^34^356^27, 1^22^23^34^356^27, 1^22^23^24^356^27, 1^22^23^24^256^27, 1^22^23^24^2567, 1^22^33^44^55^26^47^2$, \\ $1^22^33^44^55^26^37^2, 23^24^356^27, 1^22^33^44^45^26^37^2, 23^24^256^27, 23^24^2567,
	1^22^33^34^45^26^37^2$, \\ $1^22^23^34^45^26^37^2, 234^256^27, 234^2567, 234567, 23467,
	34^256^27, 34^2567, 34567, 3467, 4567,
	467, 67, 7 \}$}  
and the degree of the integral is
\begin{equation*}
\ya= 138\alpha_1 + 216\alpha_2 + 290\alpha_3 + 360\alpha_4 + 244\alpha_5 + 124\alpha_6 + 182\alpha_7.
\end{equation*}

\subsubsection{The Dynkin diagram \emph{(\ref{eq:dynkin-g(8,6)}
		e)}}\label{subsubsec:g(8,6)-e}

\

The Nichols algebra $\toba_{\bq}$ is generated by $(x_i)_{i\in \I_7}$ with defining relations
\begin{align}\label{eq:rels-g(8,6)-e}
\begin{aligned}
x_{112}&=0; & x_{221}&=0; & x_{223}&=0; & & x_{ij}=0, \ i<j, \, \widetilde{q}_{ij}=1;\\
& & x_{554}&=0; & [x_{(24)},&x_3]_c=0; &  & [[[x_{6547},x_4]_c,x_5]_c,x_4]_c=0;\\
x_{556}&=0; & x_{665}&=0; & [x_{(35)},&x_4]_c=0; & & [x_{347},x_4]_c=0; \\
x_{774}&=0; & x_{3}^2&=0; & x_{4}^2&=0; & & x_{\alpha}^{3}=0, \ \alpha\in\Oc_+^{\bq}.
\end{aligned}
\end{align}
Here {\scriptsize$\Oc_+^{\bq}=\{ 1, 12, 2, 123, 23, 3, 1234, 234, 34, 4, 12^23^24^25^26,
	123^24^25^26, 1234^25^26, 12345^26, 23^24^25^26$, \\ $234^25^26, 2345^26, 34^25^26, 345^26, 45^26, 6, 12^23^34^45^46^37, 12^23^34^35^46^37, 12^23^24^35^46^37$, \\ $123^24^35^46^37, 23^24^35^46^37, 12^23^34^45^46^27, 12^23^34^35^46^27, 12^23^24^35^46^27, 123^24^35^46^27, 23^24^35^46^27$, \\ $1^22^33^44^55^66^47^2, 12^33^44^55^66^47^2, 12^23^24^25^26^27, 12^23^24^25^267, 12^23^44^55^66^47^2, 123^24^25^26^27$, \\ $123^24^25^267, 23^24^25^26^27, 23^24^25^267, 12^23^34^55^66^47^2, 12^23^34^45^66^47^2, 12^23^34^45^46^37^2, 1234^25^26^27$, \\ $234^25^26^27, 34^25^26^27, 1234^25^267, 234^25^267, 34^25^267, 12^23^34^35^46^37^2, 12^23^24^35^46^37^2, 123^24^35^46^37^2$, \\ $12345^26^27, 12345^267, 23^24^35^46^37^2, 2345^26^27, 2345^267, 345^26^27, 345^267, 45^26^27, 45^267, 67, 7 \}$}  
and the degree of the integral is
\begin{equation*}
\ya= 80\alpha_1 + 156\alpha_2 + 228\alpha_3 + 296\alpha_4 + 360\alpha_5 + 244\alpha_6 + 124\alpha_7.
\end{equation*}

\subsubsection{The Dynkin diagram \emph{(\ref{eq:dynkin-g(8,6)}
		f)}}\label{subsubsec:g(8,6)-f}

\

The Nichols algebra $\toba_{\bq}$ is generated by $(x_i)_{i\in \I_7}$ with defining relations
\begin{align}\label{eq:rels-g(8,6)-f}
\begin{aligned}
x_{112}&=0; & x_{443}&=0; & x_{445}&=0; & x_{ij}&=0, & &i<j, \, \widetilde{q}_{ij}=1;\\
x_{554}&=0; & x_{447}&=0; & x_{774}&=0; & x_{2}^2&=0; & &[x_{(13)},x_2]_c=0;\\
[x_{(24)},&x_3]_c=0; & x_{556}&=0; & x_{665}&=0; & x_{3}^2&=0; & & x_{\alpha}^{3}=0, \ \alpha\in\Oc_+^{\bq}.
\end{aligned}
\end{align}
Here {\scriptsize$\Oc_+^{\bq}=\{ 1, 12, 2, 123, 23, 3, 12345, 2345, 345, 45, 12^23^24^256,
	123^24^256, 1234^256, 12346, 23^24^256$, \\ $234^256, 2346, 34^256, 346, 46, 56, 12^23^34^45^36^37, 12^23^34^35^26^37, 12^23^24^35^26^37, 123^24^35^26^37$, \\ $23^24^35^26^37, 12^23^34^45^26^27, 12^23^24^25^26^27, 123^24^25^26^27, 23^24^25^26^27, 1234^25^26^27, 234^25^26^27$, \\ $34^25^26^27, 1^22^33^44^55^36^47^2, 12^33^44^55^36^47^2, 12^23^44^55^36^47^2, 12^23^34^55^36^47^2, 12^23^34^45^36^37^2$, \\ $12^23^34^356^27, 12^23^24^356^27, 12^23^24^2567, 123^24^356^27, 123^24^2567, 1234^2567, 123456^27$, \\ $12^23^34^45^26^47^2, 12^23^34^35^26^37^2, 12^23^24^35^26^37^2, 123^24^35^26^37^2, 123467, 23^24^356^27, 23456^27$, \\ $3456^27, 456^27, 23^24^35^26^37^2, 23^24^2567, 234^2567, 23467, 34^2567, 3467, 467, 567, 7 \}$}  
and the degree of the integral is
\begin{equation*}
\ya= 80\alpha_1 + 156\alpha_2 + 228\alpha_3 + 296\alpha_4 + 244\alpha_5 + 124\alpha_6 + 182\alpha_7.
\end{equation*}

\subsubsection{The Dynkin diagram \emph{(\ref{eq:dynkin-g(8,6)}
		g)}}\label{subsubsec:g(8,6)-g}

\

The Nichols algebra $\toba_{\bq}$ is generated by $(x_i)_{i\in \I_7}$ with defining relations
\begin{align}\label{eq:rels-g(8,6)-g}
\begin{aligned}
x_{332}&=0; & x_{334}&=0; & x_{443}&=0; & x_{ij}&=0, & & i<j, \, \widetilde{q}_{ij}=1;\\
x_{445}&=0; & x_{554}&=0; & x_{447}&=0; & x_{1}^2&=0; & 
& [x_{(13)},x_2]_c=0;\\
x_{774}&=0; & x_{556}&=0; & x_{665}&=0; & x_{2}^2&=0; & & x_{\alpha}^{3}=0, \ \alpha\in\Oc_+^{\bq}.
\end{aligned}
\end{align}
Here {\scriptsize$\Oc_+^{\bq}=\{ 1, 12, 2, 123, 23, 3, 1234, 234, 34, 4, 12^23^24^25^26,
	123^24^25^26, 1234^25^26, 12345^26, 23^24^25^26$, \\ $234^25^26,
	2345^26, 34^25^26, 345^26, 45^26, 6, 12^23^34^45^56^37, 12^23^34^45^56^27, 12^23^34^35^36^27$, \\ $12^23^24^35^36^27, 123^24^35^36^27, 23^24^35^36^27, 12^23^24^25^36^27, 123^24^25^36^27, 23^24^25^36^27, 1234^25^36^27$, \\ $234^25^36^27, 34^25^36^27, 1^22^33^44^55^66^37^2, 12^33^44^55^66^37^2, 12^23^44^55^66^37^2, 12^23^34^55^66^37^2$, \\ $12^23^34^45^66^37^2, 12^23^34^35^367, 12^23^24^35^367, 12^23^24^25^367, 123^24^35^367, 123^24^25^367, 1234^25^367$, \\ $1234567, 12^23^34^45^46^27^2, 12^23^34^35^46^27^2, 12^23^24^35^46^27^2, 123^24^35^46^27^2,
	123457, 23^24^35^367$, \\ $23^24^25^367, 234^25^367, 34^25^367, 23^24^35^46^27^2, 234567, 23457, 34567, 3457, 4567, 457, 567, 57 \}$}  
and the degree of the integral is
\begin{equation*}
\ya= 80\alpha_1 + 156\alpha_2 + 228\alpha_3 + 296\alpha_4 + 360\alpha_5 + 124\alpha_6 + 182\alpha_7.
\end{equation*}

\subsubsection{The Dynkin diagram \emph{(\ref{eq:dynkin-g(8,6)}
		h)}}\label{subsubsec:g(8,6)-h}

\

The Nichols algebra $\toba_{\bq}$ is generated by $(x_i)_{i\in \I_7}$ with defining relations
\begin{align}\label{eq:rels-g(8,6)-h}
\begin{aligned}
x_{221}&=0; & x_{223}&=0; & x_{332}&=0; & x_{ij}&=0, & & i<j, \, \widetilde{q}_{ij}=1;
\\
x_{334}&=0; & x_{443}&=0; & x_{445}&=0; & x_{447}&=0; & & x_{774}=0; 
\\ 
x_{554}&=0; & x_{556}&=0; & x_{665}&=0; & x_{1}^2&=0; & & x_{\alpha}^{3}=0, \ \alpha\in\Oc_+^{\bq}.
\end{aligned}
\end{align}
Here {\scriptsize$\Oc_+^{\bq}=\{ 1, 12, 2, 123, 23, 3, 1234, 234, 34, 4, 12345, 2345, 345, 45, 5, 123456, 23456, 3456, 456, 56, 6$, \\ $12^23^34^45^46^27^2, 12^23^34^35^46^27^2, 12^23^34^35^36^27^2, 12^23^34^35^367^2, 12^23^24^35^46^27^2, 12^23^24^35^36^27^2$, \\ $12^23^24^35^367^2, 12^23^24^25^36^27^2, 12^23^24^25^367^2, 12^23^24^25^267^2, 1^22^33^44^55^66^37^4, 12^33^44^55^66^37^4$, \\ $123^24^35^46^27^2, 123^24^35^36^27^2, 123^24^25^36^27^2, 1234^25^36^27^2, 123^24^35^367^2, 123^24^25^367^2, 1234^25^367^2$, \\ $123^24^25^267^2, 1234^25^267^2, 12345^267^2, 12^23^44^55^66^37^4, 12^23^34^55^66^37^4, 12^23^34^45^66^37^4$, \\ $12^23^34^45^56^37^4, 12^23^34^45^56^27^4, 23^24^35^46^27^2, 23^24^35^36^27^2, 23^24^35^367^2, 23^24^25^36^27^2$, \\ $23^24^25^367^2, 23^24^25^267^2, 234^25^36^27^2, 234^25^367^2, 234^25^267^2, 2345^267^2$, \\ $34^25^36^27^2, 34^25^367^2, 34^25^267^2, 345^267^2, 45^267^2 \}$}  
and the degree of the integral is
\begin{equation*}
\ya= 80\alpha_1 + 156\alpha_2 + 228\alpha_3 + 296\alpha_4 + 360\alpha_5 + 238\alpha_6 + 182\alpha_7.
\end{equation*}

\subsubsection{The associated Lie algebra} This is of type $E_7$.

\section{Super modular type, characteristic 5}\label{sec:by-diagram-super-modular-char5}

In this Section $\kk$ is a field of characteristic  $5$.

\subsection{Type $\Sbrown(2; 5)$}\label{subsec:type-brj(2;5)}
Here $\theta = 2$, $\zeta \in \G'_5$. Let 
\begin{align*}
A&=\begin{pmatrix} 2 & -3 \\ -1 & 0 \end{pmatrix}, \,
A'=\begin{pmatrix} 2 & -4 \\ -1 & 0 \end{pmatrix}\in \kk^{2\times 2}; & \pa &= (-1,1), \, \pa' = (-1,-1) \in \G_2^2.
\end{align*}
Let $\brj(2;5) = \g(A, \pa) \simeq \g(A', \pa')$,
the contragredient Lie superalgebras corresponding to $(A, \pa)$, $ \g(A', \pa')$. 
We know \cite{BGL} that $\sdim \brj(2;5) = 10|12$. 
We describe the root system $\Sbrown(2; 5)$ of $\brj(2;5)$, see \cite{AA-GRS-CLS-NA} for details.

\subsubsection{Basic datum and root system}
Below, $G_2$ and $A_2^{(2)}$ are numbered as in \eqref{eq:dynkin-system-G} and \eqref{eq:A2-2n}, respectively.
The basic datum and the bundle of Cartan matrices are described by the following diagram:
\begin{center}
	$\overset{G_2}{\underset{a_1}{\vtxgpd}}$   \hspace{-5pt}\raisebox{3pt}{$\overset{1}{\rule{40pt}{0.5pt}}$}\hspace{-5pt}  $\overset{A_2^{(2)}}{\underset{a_2}{\vtxgpd}}$.
\end{center}
Using the notation \eqref{eq:notation-root-exceptional}, the bundle of root sets is the following:
\begin{align*}
\varDelta_{+}^{a_1}= & \{ 1,1^32,1^22,1^52^3,1^32^2,1^42^3,12,2 \}, \\
\varDelta_{+}^{a_2}= & \{ 1,1^42,1^32,1^52^2,1^22,1^32^2,12,2 \}.
\end{align*}

\subsubsection{Weyl groupoid}
\label{subsubsec:type-brj25-Weyl}
The isotropy group  at $a_1 \in \cX$ is
\begin{align*}
\cW(a_1)= \langle \varsigma_1^{a_1}, \varsigma_2^{a_1} \varsigma_1 \varsigma_2 \rangle \simeq \mathbb{D}_4.
\end{align*}

\subsubsection{Incarnation}

We assign the following Dynkin diagrams to $a_i$, $i\in\I_2$:
\begin{align}\label{eq:dynkin-brj(2;5)}
a_1 \mapsto & \xymatrix{\overset{\zeta}{\underset{\ }{\circ}} \ar  @{-}[r]^{\zeta^2} &
	\overset{-1}{\underset{\ }{\circ}}},
&  a_2 \mapsto &\xymatrix{\overset{-\zeta^3}{\underset{\ }{\circ}} \ar  @{-}[r]^{\zeta^3} &
	\overset{-1}{\underset{\ }{\circ}}}.
\end{align}

\subsubsection{PBW-basis and dimension} \label{subsubsec:type-brj25-PBW}

Notice that the roots in each $\varDelta_{+}^{a_i}$, $i\in\I_{2}$, are ordered from left to right, justifying the notation $\beta_1, \dots, \beta_{8}$.

The root vectors $x_{\beta_k}$ are described as in Remark \ref{rem:lyndon-word}.
Thus
\begin{align*}
\left\{ x_{\beta_{8}}^{n_{8}} \dots x_{\beta_2}^{n_{2}}  x_{\beta_1}^{n_{1}} \, | \, 0\le n_{k}<N_{\beta_k} \right\}.
\end{align*}
is a PBW-basis of $\toba_{\bq}$. Hence $\dim \toba_{\bq}=2^45^210^2=40000$.

\subsubsection{The Dynkin diagram \emph{(\ref{eq:dynkin-brj(2;5)}
		a)}}\label{subsubsec:brj(2;5)-a}

\

The Nichols algebra $\toba_{\bq}$ is generated by $(x_i)_{i\in \I_2}$ with defining relations
\begin{align}\label{eq:rels-brj(2;5)-a}
\begin{aligned}
x_1^5&=0; & x_{112}^{10}&=0; & [x_{112},x_{12}]_c^5&=0; & [[[x_{112},x_{12}]_c,x_{12}]_c,x_{12}]_c&=0;  \\
x_2^2&=0; &x_{12}^{10}&=0; &x_{11112}&=0; & [x_{1112},x_{112}]_c&=0.
\end{aligned}
\end{align}
Here $\Oc_+^{\bq}= \{1, 1^22, 1^32^2, 12 \}$ and the degree of the integral is
\begin{equation*}
\ya= 55\alpha_1 + 34\alpha_2.
\end{equation*}

\subsubsection{The Dynkin diagram \emph{(\ref{eq:dynkin-brj(2;5)}
		b)}}\label{subsubsec:brj(2;5)-b}

\

The Nichols algebra $\toba_{\bq}$ is generated by $(x_i)_{i\in \I_2}$ with defining relations
\begin{align}\label{eq:rels-brj(2;5)-b}
\begin{aligned}
x_{1112}^5&=0; & x_{112}^{10}&=0; & x_1^{10}&=0; \qquad x_2^2=0;
\\ x_{12}^5&=0; & x_{111112}&=0; & [x_1,&[x_{112},x_{12}]_c]_c+q_{12}x_{112}^2=0.
\end{aligned}
\end{align}
Here $\Oc_+^{\bq}= \{1, 1^32, 1^22, 12 \}$ and the degree of the integral is
\begin{equation*}
\ya= 55\alpha_1 + 23\alpha_2.
\end{equation*}

\subsubsection{The associated Lie algebra} This is of type $B_2$.

\subsection{Type $\El(5;5)$}\label{subsec:type-el(5;5)}
Here $\theta = 5$, $\zeta \in \G'_5$. Let 
\begin{align*}
A&=\begin{pmatrix} 0 & 1 & 0 & 0 & 0 \\ \text{--}1 & 0 & 1 & 0 & 1 \\ 0 & \text{--}1 & 2 & \text{--}1 & 0
\\ 0 & 0 & \text{--}1 & 2 & 0 \\ 0 & \text{--}1 & 0 & 0 & 2 \end{pmatrix}
\in \kk^{5\times 5}; & \pa &= (-1,-1, 1,1, 1) \in \G_2^5.
\end{align*}
Let $\el(5;5) = \g(A, \pa)$,
the contragredient Lie superalgebra corresponding to $(A, \pa)$. 
We know \cite{BGL} that $\sdim \el(5;5) = 55|32$. 
There are 6 other pairs of matrices and parity vectors for which the associated contragredient Lie superalgebra is isomorphic to $\el(5;5)$.
We describe the root system $\El(5;5)$ of $\el(5;5)$, see \cite{AA-GRS-CLS-NA} for details.

\subsubsection{Basic datum and root system}
Below, $A_1^{(1)}$, $C_2$ and $A_2^{(2)}$ are numbered as in \eqref{eq:An-(1)}, \eqref{eq:dynkin-system-C} and  	
\eqref{eq:A2-2n}, respectively.
The basic datum and the bundle of Cartan matrices are described by the following diagram:

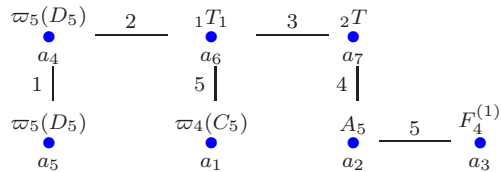
\begin{center}
	\begin{tabular}{c c c c c c c c c c c c}
		& $\overset{\varpi_5(D_5)}{\underset{a_4}{\vtxgpd}}$
		& \hspace{-7pt}\raisebox{3pt}{$\overset{2}{\rule{27pt}{0.5pt}}$}\hspace{-7pt}
		& $\overset{{}_1T_1}{\underset{a_6}{\vtxgpd}}$
		& \hspace{-7pt}\raisebox{3pt}{$\overset{3}{\rule{27pt}{0.5pt}}$}\hspace{-7pt}
		& $\overset{{}_2T}{\underset{a_7}{\vtxgpd}}$
		& & & & & &
		\\
		& {\scriptsize 1} \vline\hspace{5pt}
		&
		& {\scriptsize 5} \vline\hspace{5pt}
		&
		& {\scriptsize 4} \vline\hspace{5pt}
		& & & & & &
		\\
		& $\overset{\varpi_5(D_5)}{\underset{a_5}{\vtxgpd}}$
		&
		& $\overset{\varpi_4(C_5)}{\underset{a_1}{\vtxgpd}}$
		&
		& $\overset{A_5}{\underset{a_2}{\vtxgpd}}$
		& \hspace{-7pt}\raisebox{3pt}{$\overset{5}{\rule{27pt}{0.5pt}}$}\hspace{-7pt}
		& $\overset{F_4^{(1)}}{\underset{a_3}{\vtxgpd}}$
		& & & &
	\end{tabular}
\end{center}
Using the notation \eqref{eq:notation-root-exceptional}, the bundle of root sets is the following: { \scriptsize
	\begin{align*}
	\varDelta_{+}^{a_1}= & \varpi_4(\{ 1, 12, 2, 123, 23, 3, 12^23^24, 123^24, 1234, 23^24, 234, 34, 4, 12^23^34^35, 12^23^34^25, 12^23^24^35, \\
	& 123^24^35, 23^24^35, 12^23^24^25, 1^22^33^44^45^2, 12^33^44^45^2, 12^23^245, 123^24^25, 1234^25, 12^23^44^45^2, \\
	& 123^245, 12^23^34^45^2, 23^24^25, 12^23^34^35^2, 234^25, 34^25, 12^23^24^35^2, 123^24^35^2, 12345, 4^25, \\
	& 23^24^35^2, 23^245, 2345, 345, 45, 5 \}), \\
	\varDelta_{+}^{a_2}= & \{ 1, 12, 2, 123, 23, 3, 1234, 234, 34, 4, 12^23^34^35, 12^23^24^35, 12^23^24^25, 123^24^35, 123^24^25, \\
	& 23^24^35, 23^24^25, 12^23^34^45^2, 1234^25, 234^25, 34^25, 1^22^33^44^55^3, 12^33^44^55^3, 12^23^44^55^3, \\
	& 12^23^34^35^2, 12^23^34^55^3, 12^23^24^35^2, 123^24^35^2, 23^24^35^2, 12^23^34^45^3, 12345, 12^23^24^25^2, \\
	& 123^24^25^2, 1234^25^2, 2345, 23^24^25^2, 234^25^2, 345, 34^25^2, 45, 5 \}, \\
	\varDelta_{+}^{a_3}= & \{ 1, 12, 2, 123, 23, 3, 1234, 12^23^24^2, 123^24^2, 1234^2, 234, 23^24^2, 234^2, 34, 34^2, 4, 12^23^34^45, \\
	& 12^23^34^35, 12^23^24^35, 123^24^35, 23^24^35, 1^22^33^44^55^2, 12^33^44^55^2, 12^23^24^25, 12^23^44^55^2, \\
	& 123^24^25, 12^23^34^55^2, 1234^25, 12^23^34^45^2, 23^24^25, 234^25, 34^25, 12^23^34^35^2, 12^23^24^35^2, \\
	& 123^24^35^2, 12345, 23^24^35^2, 2345, 345, 45, 5 \}, \\
	\varDelta_{+}^{a_4}= & \varpi_5(\{ 1, 12, 2, 123, 23, 3, 1234, 234, 34, 4, 1^32^33^245, 1^22^33^245, 1^42^53^44^25^2, 1^42^53^34^25^2, \\
	& 1^42^53^345^2, 1^22^23^245, 1^42^43^34^25^2, 1^42^43^345^2, 1^22^2345, 1^42^43^245^2, 1^22^235, 1^52^63^44^25^3, \\
	& 1^32^43^34^25^2, 1^32^43^345^2, 1^32^43^245^2, 1^32^33^245^2, 1^42^63^44^25^3, 1^42^53^44^25^3, 12^23^245, \\
	& 1^42^53^34^25^3, 12^2345, 1^42^53^345^3, 12^235, 1^22^33^245^2, 12345, 1235, 125, 2345, 235, 25, 5 \}), \\
	\varDelta_{+}^{a_5}= & \varpi_5(\{ 1, 12, 2, 123, 23, 3, 1234, 234, 34, 4, 12^33^245, 12^23^245, 12^2345, 12^235, 12^53^44^25^2, \\
	& 12^53^34^25^2, 12^53^345^2, 2^33^245, 1^22^63^44^25^3, 12^43^34^25^2, 12345, 12^43^345^2, 1235, 12^43^245^2, \\
	& 125, 12^63^44^25^3, 12^33^245^2, 12^53^44^25^3, 12^53^34^25^3, 12^53^345^3, 2^23^245, 2^43^34^25^2, 2^43^345^2, \\
	& 2^2345, 2^43^245^2, 2^235, 2^33^245^2, 2345, 235, 25, 5 \}), \\
	\varDelta_{+}^{a_6}= & \{ 1, 12, 2, 123, 23, 3, 1234, 234, 34, 4, 12^23^245, 12^2345, 12^235, 123^245, 12^23^44^25^2, \\
	& 23^245, 12^23^34^25^2, 12^23^345^2, 3^245, 1^22^33^44^25^3, 123^34^25^2, 12345, 12^33^44^25^3, 23^34^25^2, \\
	& 2345, 12^23^44^25^3, 123^345^2, 23^345^2, 12^23^245^2, 12^23^34^25^3, 345, 123^245^2, 23^245^2, \\
	& 12^23^345^3, 1235, 125, 3^245^2, 235, 25, 35, 5 \}, \\
	\varDelta_{+}^{a_7}= & \{ 1, 12, 2, 123, 23, 3, 1234, 234, 34, 4, 12^23^345, 12^23^245, 12^23^25, 123^245, 123^25, 23^245, \\
	& 23^25, 12^23^34^25^2, 12^23^345^2, 1^22^33^44^25^3, 12^33^44^25^3, 12^23^44^25^3, 12345, 12^23^24^25^2, \\
	& 123^24^25^2, 2345, 23^24^25^2, 345, 12^23^34^25^3, 12^23^245^2, 123^245^2, 23^245^2, 45, 12^23^345^3, \\
	& 12345^2, 1235, 2345^2, 235, 345^2, 35, 5 \}.
	\end{align*}
}%

\subsubsection{Weyl groupoid}
\label{subsubsec:type-el55-Weyl} 
The isotropy group  at $a_3 \in \cX$ is
\begin{align*}
\cW(a_3)= \langle \varsigma_1^{a_3}, \varsigma_2^{a_3}, \varsigma_3^{a_3}, \varsigma_4^{a_3}, \varsigma_5^{a_3}\varsigma_4 \varsigma_3\varsigma_2 \varsigma_5 \varsigma_2 \varsigma_3 \varsigma_4 \varsigma_5  \rangle \simeq W(C_5).
\end{align*}

\subsubsection{Incarnation}
We set the matrices $(\bq^{(i)})_{i\in\I_{7}}$, from left to right and  from up to down:
\begin{align}\label{eq:dynkin-el(5;5)}
\begin{aligned}
&\xymatrix@C-4pt{\overset{\zeta^2}{\underset{\ }{\circ}} \ar  @{-}[r]^{\ztu^{\, 2}} &
	\overset{\zeta^{2}}{\underset{\ }{\circ}} \ar  @{-}[r]^{\ztu^{\, 2}}  &
	\overset{-1}{\underset{\ }{\circ}} \ar  @{-}[r]^{\ztu}  & \overset{\zeta}{\underset{\
		}{\circ}} \ar  @{-}[r]^{\ztu^{\, 2}}  & \overset{\zeta^2}{\underset{\ }{\circ}}}
& &\xymatrix@C-4pt{\overset{\zeta^2}{\underset{\ }{\circ}} \ar  @{-}[r]^{\ztu^{\, 2}} &
	\overset{\zeta^{2}}{\underset{\ }{\circ}} \ar  @{-}[r]^{\ztu^{\, 2}}  &
	\overset{\zeta^{2}}{\underset{\ }{\circ}} \ar  @{-}[r]^{\ztu^{\, 2}}  &
	\overset{-1}{\underset{\ }{\circ}} \ar  @{-}[r]^{\zeta}  & \overset{-1}{\underset{\
		}{\circ}}}
\\
&\xymatrix@C-4pt{\overset{\zeta^2}{\underset{\ }{\circ}} \ar  @{-}[r]^{\ztu^{\, 2}} &
	\overset{\zeta^{2}}{\underset{\ }{\circ}} \ar  @{-}[r]^{\ztu^{\, 2}}  &
	\overset{\zeta^{2}}{\underset{\ }{\circ}} \ar  @{-}[r]^{\ztu^{\, 2}}  &
	\overset{\zeta}{\underset{\ }{\circ}} \ar  @{-}[r]^{\ztu}  & \overset{-1}{\underset{\
		}{\circ}}}&&
\end{aligned}
\end{align}

\begin{align*}
&\xymatrix@R-8pt{  &   \overset{\zeta^2}{\circ} \ar  @{-}[d]^{\ztu^{\, 2}} & & \\
	\overset{-1}{\underset{\ }{\circ}} \ar  @{-}[r]^{\zeta^2}  & \overset{-1}{\underset{\
		}{\circ}} \ar  @{-}[r]^{\ztu^{\, 2}}  & \overset{\zeta^2}{\underset{\ }{\circ}} \ar
	@{-}[r]^{\ztu^{\, 2}}  & \overset{\zeta^2}{\underset{\ }{\circ}}}
& &\xymatrix@R-8pt{  &   \overset{\zeta^2}{\circ} \ar  @{-}[d]^{\ztu^{\, 2}} & & \\
	\overset{-1}{\underset{\ }{\circ}} \ar  @{-}[r]^{\ztu^{\, 2}}  &
	\overset{\zeta^2}{\underset{\ }{\circ}} \ar  @{-}[r]^{\ztu^{\, 2}}  &
	\overset{\zeta^2}{\underset{\ }{\circ}} \ar  @{-}[r]^{\ztu^{\, 2}}  &
	\overset{\zeta^2}{\underset{\ }{\circ}}}
\\
& \xymatrix@R-8pt{  &   \overset{-1}{\circ} \ar  @{-}[d]_{\zeta^{2}} \ar  @{-}[dr]^{\zeta}
	& & \\
	\overset{\zeta^2}{\underset{\ }{\circ}} \ar  @{-}[r]^{\ztu^{\, 2}}  &
	\overset{-1}{\underset{\ }{\circ}} \ar  @{-}[r]^{\zeta^{2}}  & \overset{-1}{\underset{\
		}{\circ}} \ar  @{-}[r]^{\ztu^{\, 2}}  & \overset{\zeta^2}{\underset{\ }{\circ}}}&
&\xymatrix@R-8pt{& &  \overset{\zeta}{\circ} \ar  @{-}[d]_{\ztu} \ar  @{-}[dr]^{\ztu} &
	\\
	\overset{\zeta^{2}}{\underset{\ }{\circ}}  \ar  @{-}[r]^{\ztu^{\, 2}} &
	\overset{\zeta^{2}}{\underset{\ }{\circ}} \ar  @{-}[r]^{\ztu^{\, 2}}  &
	\overset{-1}{\underset{\ }{\circ}} \ar  @{-}[r]^{\zeta^{2}}  & \overset{-1}{\underset{\
		}{\circ}} }
\end{align*}
Now, this is the incarnation:
\begin{align*}
& a_1\mapsto \varpi_4(\bq^{1}); &
& a_i\mapsto \bq^{(i)}, \ i\in\I_{2,7}.
\end{align*}

\subsubsection{PBW-basis and dimension} \label{subsubsec:type-el55-PBW}
Notice that the roots in each $\varDelta_{+}^{a_i}$, $i\in\I_{7}$, are ordered from left to right, justifying the notation $\beta_1, \dots, \beta_{41}$.

The root vectors $x_{\beta_k}$ are described as in Remark \ref{rem:lyndon-word}.
Thus
\begin{align*}
\left\{ x_{\beta_{41}}^{n_{41}} \dots x_{\beta_2}^{n_{2}}  x_{\beta_1}^{n_{1}} \, | \, 0\le n_{k}<N_{\beta_k} \right\}.
\end{align*}
is a PBW-basis of $\toba_{\bq}$. Hence $\dim \toba_{\bq}=2^{16}5^{25}$.

\subsubsection{The Dynkin diagram \emph{(\ref{eq:dynkin-el(5;5)}
		a)}}\label{subsubsec:el(5;5)-a}

\

The Nichols algebra $\toba_{\bq}$ is generated by $(x_i)_{i\in \I_5}$ with defining relations
\begin{align}\label{eq:rels-el(5;5)-a}
\begin{aligned}
\begin{aligned}
x_{112}&=0; & x_{221}&=0; & x_{223}&=0; & & [[[x_{(14)},x_3]_c,x_2]_c,x_3]_c=0;\\
x_{554}&=0; & x_{443}&=0; & x_{4445}&=0; & & x_{3}^2=0; \ x_{ij}=0, \ i<j, \ \widetilde{q}_{ij}=1; 
\end{aligned}
\\
\begin{aligned}
x_{\alpha}^{5}&=0, \ \alpha\in\Oc_+^{\bq};
& [[x_{5432},x_4]_c,x_3]_c &= q_{43}(\zeta^2-\zeta) [[x_{5432},x_3]_c,x_4]_c.
\end{aligned}
\end{aligned}
\end{align}
Here {\scriptsize$\Oc_+^{\bq}=\{ 1, 12, 2, 12^23^24, 123^24, 23^24,
	4, 12^23^24^35, 123^24^35, 23^24^35, 12^23^24^25, 1^22^33^44^45^2$, \\ $12^33^44^45^2, 12^23^245,
	123^24^25, 12^23^44^45^2, 123^245, 23^24^25, 12^23^24^35^2, 123^24^35^2$, \\ $4^25, 23^24^35^2, 23^245, 45, 5 \}$}  
and the degree of the integral is
\begin{equation*}
\ya= 72\alpha_1 + 136\alpha_2 + 192\alpha_3 + 208\alpha_4 + 108\alpha_5.
\end{equation*}

\subsubsection{The Dynkin diagram \emph{(\ref{eq:dynkin-el(5;5)}
		b)}}\label{subsubsec:el(5;5)-b}

\

The Nichols algebra $\toba_{\bq}$ is generated by $(x_i)_{i\in \I_5}$ with defining relations
\begin{align}\label{eq:rels-el(5;5)-b}
\begin{aligned}
x_{112}&=0; & x_{221}&=0; & x_{223}&=0; & & x_{ij}=0, \ i<j, \, \widetilde{q}_{ij}=1;\\
& & x_{332}&=0; & x_{334}&=0; & &[[x_{54},x_{543}]_c,x_4]_c=0;\\
& & x_{4}^2&=0; & x_{5}^2&=0; & &x_{\alpha}^{5}=0, \ \alpha\in\Oc_+^{\bq}.
\end{aligned}
\end{align}
Here {\scriptsize$\Oc_+^{\bq}=\{ 1, 12, 2, 123, 23, 3,
	12^23^34^35, 12^23^24^35, 123^24^35, 23^24^35, 12^23^34^45^2, 1^22^33^44^55^3, 12^33^44^55^3$, \\ $12^23^44^55^3,
	12^23^34^55^3, 12345, 12^23^24^25^2, 123^24^25^2, 1234^25^2, 2345,
	23^24^25^2, 234^25^2, 345, 34^25^2, 45 \}$}  
and the degree of the integral is
\begin{equation*}
\ya= 72\alpha_1 + 136\alpha_2 + 192\alpha_3 + 240\alpha_4 + 154\alpha_5.
\end{equation*}

\subsubsection{The Dynkin diagram \emph{(\ref{eq:dynkin-el(5;5)}
		c)}}\label{subsubsec:el(5;5)-c}

\

The Nichols algebra $\toba_{\bq}$ is generated by $(x_i)_{i\in \I_5}$ with defining relations
\begin{align}\label{eq:rels-el(5;5)-c}
\begin{aligned}
x_{112}&=0; & x_{221}&=0; & x_{223}&=0; & x_{332}&=0; & x_{ij}&=0, & i<j,& \, \widetilde{q}_{ij}=1;\\
x_{334}&=0; &  x_{4443}&=0; & x_{445}&=0; & x_{5}^2&=0; & x_{\alpha}^{5}&=0, & \alpha&\in\Oc_+^{\bq}.
\end{aligned}
\end{align}
Here {\scriptsize$\Oc_+^{\bq}=\{ 1, 12, 2, 123, 23, 3,
	1234, 12^23^24^2, 123^24^2, 1234^2, 234, 23^24^2, 234^2, 34,
	34^2, 4, 1^22^33^44^55^2$, \\ $12^33^44^55^2, 12^23^44^55^2, 12^23^34^55^2,
	12^23^34^45^2, 12^23^34^35^2, 12^23^24^35^2, 123^24^35^2, 23^24^35^2 \}$}  
and the degree of the integral is
\begin{equation*}
\ya= 72\alpha_1 + 136\alpha_2 + 192\alpha_3 + 240\alpha_4 + 88\alpha_5.
\end{equation*}

\subsubsection{The Dynkin diagram \emph{(\ref{eq:dynkin-el(5;5)}
		d)}}\label{subsubsec:el(5;5)-d}

\

The Nichols algebra $\toba_{\bq}$ is generated by $(x_i)_{i\in \I_5}$ with defining relations
\begin{align}\label{eq:rels-el(5;5)-d}
\begin{aligned}
x_{332}&=0; & x_{334}&=0; & [x_{(13)},&x_2]_c=0; & x_{1}^2&=0; & x_{ij}&=0, \ i<j, \, \widetilde{q}_{ij}=1;\\
x_{443}&=0; &  x_{552}&=0; & [x_{125},&x_2]_c=0; & x_{2}^2&=0; & x_{\alpha}^{5}&=0, \ \alpha\in\Oc_+^{\bq}.
\end{aligned}
\end{align}
Here {\scriptsize$\Oc_+^{\bq}=\{ 12, 123, 3, 1234, 34, 4, 12^33^245, 12^53^44^25^2, 12^53^34^25^2, 12^53^345^2, 1^22^63^44^25^3, 12345, 1235, 125$, \\ $12^33^245^2, 12^53^44^25^3, 12^53^34^25^3, 12^53^345^3, 2^23^245, 2^43^34^25^2,
	2^43^345^2, 2^2345, 2^43^245^2, 2^235, 5 \}$}  
and the degree of the integral is
\begin{equation*}
\ya= 72\alpha_1 + 300\alpha_2 + 208\alpha_3 + 108\alpha_4 + 154\alpha_5.
\end{equation*}

\subsubsection{The Dynkin diagram \emph{(\ref{eq:dynkin-el(5;5)}
		e)}}\label{subsubsec:el(5;5)-e}

\

The Nichols algebra $\toba_{\bq}$ is generated by $(x_i)_{i\in \I_5}$ with defining relations
\begin{align}\label{eq:rels-el(5;5)-e}
\begin{aligned}
x_{221}&=0; & x_{223}&=0; & x_{332}&=0; & x_{225}&=0; & x_{ij}&=0, \ i<j, \, \widetilde{q}_{ij}=1;\\
x_{552}&=0; &  x_{334}&=0; & x_{443}&=0; & x_{1}^2&=0; & x_{\alpha}^{5}&=0, \ \alpha\in\Oc_+^{\bq}.
\end{aligned}
\end{align}
Here {\scriptsize$\Oc_+^{\bq}=\{ 2, 23, 3, 234, 34, 4, 1^22^33^245, 1^42^53^44^25^2, 1^42^53^34^25^2, 1^42^53^345^2, 1^22^23^245, 1^42^43^34^25^2$, \\ $1^42^43^345^2, 1^22^2345, 1^42^43^245^2, 1^22^235, 1^42^63^44^25^3, 1^42^53^44^25^3, 1^42^53^34^25^3, 1^42^53^345^3$, \\ $1^22^33^245^2, 2345, 235, 25, 5 \}$}  
and the degree of the integral is
\begin{equation*}
\ya= 230\alpha_1 + 300\alpha_2 + 208\alpha_3 + 108\alpha_4 + 154\alpha_5.
\end{equation*}

\subsubsection{The Dynkin diagram \emph{(\ref{eq:dynkin-el(5;5)}
		f)}}\label{subsubsec:el(5;5)-f}

\

The Nichols algebra $\toba_{\bq}$ is generated by $(x_i)_{i\in \I_5}$ with defining relations
\begin{align}\label{eq:rels-el(5;5)-f}
\begin{aligned}
&\begin{aligned}
x_{112}&=0; & [x_{(13)},&x_2]_c=0; & x_{3}^2&=0; & & x_{ij}=0, \ i<j, \, \widetilde{q}_{ij}=1;
\\
x_{443}&=0; & [x_{(24)},&x_3]_c=0; & x_{5}^2&=0; \ & & [[x_{53},x_{534}]_c,x_3]_c=0;
\\
&& [x_{125},& x_2]_c=0; & x_{2}^2&=0; & & x_{\alpha}^{5}=0, \ \alpha\in\Oc_+^{\bq};
\end{aligned}
\\
& x_{235} =\frac{q_{35}}{\zeta^2+\zeta}[x_{25},x_3]_c +q_{23}(1-\zeta)x_3x_{25}.
\end{aligned}
\end{align}
Here {\scriptsize$\Oc_+^{\bq}=\{ 1, 123, 23, 1234, 234, 4, 12^2345, 12^235, 123^245, 12^23^44^25^2, 23^245, 1^22^33^44^25^3, 123^34^25^2$, \\ $12^33^44^25^3, 23^34^25^2, 123^345^2, 23^345^2, 12^23^245^2, 12^23^34^25^3, 345, 12^23^345^3, 125, 3^245^2, 25, 35 \}$}  
and the degree of the integral is
\begin{equation*}
\ya= 72\alpha_1 + 136\alpha_2 + 208\alpha_3 + 108\alpha_4 + 154\alpha_5.
\end{equation*}

\subsubsection{The Dynkin diagram \emph{(\ref{eq:dynkin-el(5;5)}
		g)}}\label{subsubsec:el(5;5)-g}

\

The Nichols algebra $\toba_{\bq}$ is generated by $(x_i)_{i\in \I_5}$ with defining relations
\begin{align}\label{eq:rels-el(5;5)-g}
\begin{aligned}
&\begin{aligned}
&&x_{112}&=0; & x_{221}&=0; & & x_{ij}=0, \ i<j, \, \widetilde{q}_{ij}=1;
\\
x_{3}^2&=0; &  x_{223}&=0; & x_{553}&=0; & &[[[x_{1235},x_3]_c,x_2]_c,x_3]_c=0;
\\
x_4^2&=0; & x_{554}&=0; & [x_{(24)},&x_3]_c=0; & & x_{\alpha}^{5}=0, \ \alpha\in\Oc_+^{\bq};
\end{aligned}
\\
&x_{(35)}=q_{45}\zeta[x_{35},x_4]_c+q_{34}(1-\ztu)x_4x_{35}.
\end{aligned}
\end{align}
Here {\scriptsize$\Oc_+^{\bq}=\{ 1, 12, 2, 1234, 234, 34, 12^23^345, 12^23^25, 123^25, 23^25, 12^23^345^2, 1^22^33^44^25^3, 12^33^44^25^3$, \\ $12^23^44^25^3, 12345, 12^23^24^25^2, 123^24^25^2, 2345, 23^24^25^2, 345, 12^23^345^3, 12345^2, 2345^2, 345^2, 5 \}$}  
and the degree of the integral is
\begin{equation*}
\ya= 72\alpha_1 + 136\alpha_2 + 192\alpha_3 + 108\alpha_4 + 154\alpha_5.
\end{equation*}

\subsubsection{The associated Lie algebra} This is of type $C_5$.

\section{Unidentified}\label{sec:by-diagram-Unidentified}

The root systems in this Section are denoted by $\Ufo(h)$, $8 \neq h \in \I_{12}$; the corresponding Nichols algebras 
are called collectively $\ufo(h)$. However $\Ufo(7)$ has two different incarnations, that are called
$\ufo(7)$ and $\ufo(8)$ respectively.

\subsection{Type $\Ufo(1)$}\label{subsec:type-ufo(1)}
Here $\zeta \in \G'_4$.
We describe first the root system $\Ufo(1)$.

\subsubsection{Basic datum and root system}
Below, $A_5$, $D_5$, $_{2}T$ and $_{1}T_1$ are numbered as in \eqref{eq:dynkin-system-A}, \eqref{eq:dynkin-system-D} and  	
\eqref{eq:mTn}, respectively.
The basic datum and the bundle of Cartan matrices are described by the following diagram:
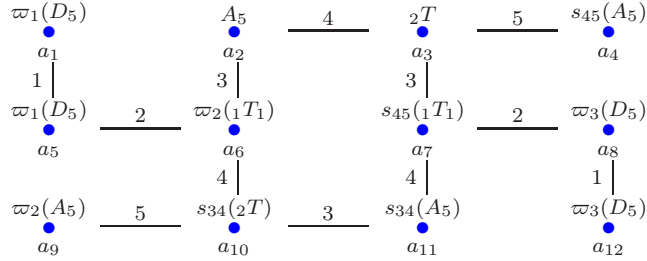
\begin{center}
	\begin{tabular}{c c c c c c c}
		$\overset{\varpi_1(D_5)}{\underset{a_1}{\vtxgpd}}$
		&
		& $\overset{A_5}{\underset{a_2}{\vtxgpd}}$
		& \hspace{-5pt}\raisebox{3pt}{$\overset{4}{\rule{30pt}{0.5pt}}$}\hspace{-5pt}
		& $\overset{{}_2T}{\underset{a_3}{\vtxgpd}}$
		& \hspace{-5pt}\raisebox{3pt}{$\overset{5}{\rule{30pt}{0.5pt}}$}\hspace{-5pt}
		& $\overset{s_{45}(A_5)}{\underset{a_4}{\vtxgpd}}$
		\\
		{\scriptsize 1} \vline\hspace{5pt}
		& & {\scriptsize 3} \vline\hspace{5pt}
		& & {\scriptsize 3} \vline\hspace{5pt}
		& &
		\\
		$\overset{\varpi_1(D_5)}{\underset{a_5}{\vtxgpd}}$
		& \hspace{-5pt}\raisebox{3pt}{$\overset{2}{\rule{30pt}{0.5pt}}$}\hspace{-5pt}
		& $\overset{\varpi_2({}_1T_1)}{\underset{a_6}{\vtxgpd}}$
		&
		& $\overset{s_{45}({}_1T_1)}{\underset{a_7}{\vtxgpd}}$
		& \hspace{-5pt}\raisebox{3pt}{$\overset{2}{\rule{30pt}{0.5pt}}$}\hspace{-5pt}
		& $\overset{\varpi_3(D_5)}{\underset{a_8}{\vtxgpd}}$
		\\
		& &{\scriptsize 4} \vline\hspace{5pt}
		& & {\scriptsize 4} \vline\hspace{5pt}
		& & {\scriptsize 1} \vline\hspace{5pt}
		\\
		$\overset{\varpi_2(A_5)}{\underset{a_9}{\vtxgpd}}$
		& \hspace{-5pt}\raisebox{3pt}{$\overset{5}{\rule{30pt}{0.5pt}}$}\hspace{-5pt}
		& $\overset{s_{34}({}_2T)}{\underset{a_{10}}{\vtxgpd}}$
		& \hspace{-5pt}\raisebox{3pt}{$\overset{3}{\rule{30pt}{0.5pt}}$}\hspace{-5pt}
		& $\overset{s_{34}(A_5)}{\underset{a_{11}}{\vtxgpd}}$
		&
		& $\overset{\varpi_3(D_5)}{\underset{a_{12}}{\vtxgpd}}$
	\end{tabular}
\end{center}
Using the notation \eqref{eq:notation-root-exceptional}, we set: { \scriptsize
	\begin{align*}
	\Delta_{+}^{(1)}= & \{ 1, 12, 2, 123, 23, 3, 12^23^24, 123^24, 1234, 23^24, 234, 34, 4, 12^23^34^35, 12^23^34^25, \\
	& 12^23^24^25, 123^24^25, 23^24^25, 1234^25, 234^25, 34^25, 12^23^34^35^2, 12^23^245, 123^245,  \\
	& 12345, 23^245, 2345, 345, 45, 5 \}, \\
	\Delta_{+}^{(2)}= & \{ 1, 12, 2, 123, 23, 3, 1234, 234, 34, 4, 12^23^34^35, 12^23^24^35, 12^23^24^25, 123^24^35, \\
	& 123^24^25, 23^24^35, 23^24^25, 12^23^34^45^2, 1234^25, 234^25, 34^25, 12^23^34^35^2, 12^23^24^35^2, \\
	& 123^24^35^2, 12345, 23^24^35^2, 2345, 345, 45, 5 \}, \\
	\Delta_{+}^{(3)}= & \{ 1, 12, 2, 123, 23, 3, 1234, 234, 34, 4, 12^23^245, 12^2345, 12^235, 123^245, 23^245, \\
	& 12^23^34^25^2, 12^23^345^2, 3^245, 12345, 2345, 12^23^245^2, 345, 123^245^2, 1235, 125, \\
	& 23^245^2, 235, 25, 35, 5 \}, \\
	\Delta_{+}^{(4)}= & \{ 1, 12, 2, 123, 23, 3, 1234, 234, 34, 4, 12^23^345, 12^23^245, 12^23^25, 123^245, 123^25, \\
	& 23^245, 23^25, 12^23^34^25^2, 12^23^345^2, 12345, 2345, 345, 12^23^245^2, 123^245^2, 1235, \\
	& 23^245^2, 235, 45, 35, 5 \}, \\
	\Delta_{+}^{(5)}= & \{ 1, 12, 2, 123, 23, 3, 1234, 234, 34, 4, 1^32^33^245, 1^22^33^245, 1^22^23^245, 1^22^2345, \\
	& 1^22^235, 1^32^43^34^25^2, 1^32^43^345^2, 12^23^245, 1^32^43^245^2, 12^2345, 1^32^33^245^2, 12345, \\
	& 1^22^33^245^2, 12^235, 1235, 125, 2345, 235, 25, 5 \}, \\
	\Delta_{+}^{(6)}= & \{ 1, 12, 2, 123, 23, 3, 1234, 234, 34, 4, 12^33^245, 12^23^245, 12^2345, 12^235, 2^33^245, \\
	& 12^43^34^25^2,12^43^345^2, 2^23^245, 12^43^245^2, 2^2345, 12345, 12^33^245^2, 1235, 125, 2345, \\
	& 2^33^245^2, 2^235, 235, 25, 5 \}.
	\end{align*}
}
Now the bundle of sets of (positive) roots is described as follows:
\begin{align*}
a_1 & \mapsto \varpi_1(\Delta_+^{(6)}), &
a_2 & \mapsto \Delta_+^{(1)}, &
a_3 & \mapsto \Delta_+^{(4)}, &
a_4 & \mapsto s_{45}(\Delta_+^{(2)}),
\\
a_5 & \mapsto \varpi_1(\Delta_+^{(5)}), &
a_6 & \mapsto \varpi_2(\Delta_+^{(3)}), &
a_7 & \mapsto s_{45}(\Delta_+^{(3)}), &
a_8 & \mapsto \varpi_3(\Delta_+^{(5)}),
\\
a_9 & \mapsto \varpi_2(\Delta_+^{(2)}), &
a_{10} & \mapsto s_{34}(\Delta_+^{(4)}), &
a_{11} & \mapsto s_{34}(\Delta_+^{(1)}), &
a_{12} & \mapsto \varpi_3(\Delta_+^{(6)}).
\end{align*}

\subsubsection{Weyl groupoid}
\label{subsubsec:type-ufo1-Weyl}
The isotropy group  at $a_1 \in \cX$ is
\begin{align*}
\cW(a_1)= \langle \varsigma_1^{a_1}\varsigma_2 \varsigma_3\varsigma_4 \varsigma_5 \varsigma_4 \varsigma_3 \varsigma_2 \varsigma_1, \varsigma_2^{a_1},  \varsigma_3^{a_1}, \varsigma_4^{a_1},  \varsigma_5^{a_1} \rangle \simeq W(A_5).
\end{align*}

\subsubsection{Incarnation}
We set the matrices $(\bq^{(i)})_{i\in\I_{6}}$, from left to right and  from up to down:
\begin{align}\label{eq:dynkin-ufo(1)}
\begin{aligned}
&
\Dchainfive{\zeta }{\ztu }{\zeta }{\ztu }{-1}{-1}{-1}{\ztu }{\zeta }
& &
\Dchainfive{\zeta }{\ztu }{\zeta }{\ztu }{\zeta }{\ztu }{-1}{\ztu }{\zeta }
\\
& \xymatrix@R-8pt{  &   \overset{-1}{\circ} \ar  @{-}[d]_{\zeta} \ar  @{-}[dr]^{-1} & & \\
	\overset{\zeta}{\underset{\ }{\circ}} \ar  @{-}[r]^{\ztu}  & \overset{-1}{\underset{\
		}{\circ}} \ar  @{-}[r]^{\zeta}  & \overset{-1}{\underset{\ }{\circ}} \ar  @{-}[r]^{\ztu}
	& \overset{\zeta}{\underset{\ }{\circ}}}
& &
\xymatrix@R-8pt{& &  \overset{-1}{\circ} \ar  @{-}[d]_{-1} \ar  @{-}[dr]^{\zeta} &  \\
	\overset{\zeta}{\underset{\ }{\circ}}  \ar  @{-}[r]^{\ztu} &
	\overset{\zeta}{\underset{\ }{\circ}} \ar  @{-}[r]^{\ztu}  & \overset{-1}{\underset{\
		}{\circ}} \ar  @{-}[r]^{\zeta}  & \overset{-1}{\underset{\ }{\circ}} }
\\
& \xymatrix@R-8pt{  &   \overset{\zeta}{\circ} \ar  @{-}[d]^{\ztu} & & \\
	\overset{-1}{\underset{\ }{\circ}} \ar  @{-}[r]^{\zeta}  & \overset{-1}{\underset{\
		}{\circ}} \ar  @{-}[r]^{\ztu}  & \overset{\zeta}{\underset{\ }{\circ}} \ar
	@{-}[r]^{\ztu}  & \overset{\zeta}{\underset{\ }{\circ}}}
& &
\xymatrix@R-8pt{  &   \overset{\zeta}{\circ} \ar  @{-}[d]^{\ztu} & & \\
	\overset{-1}{\underset{\ }{\circ}} \ar  @{-}[r]^{\ztu}  & \overset{\zeta}{\underset{\
		}{\circ}} \ar  @{-}[r]^{\ztu}  & \overset{\zeta}{\underset{\ }{\circ}} \ar
	@{-}[r]^{\ztu}  & \overset{\zeta}{\underset{\ }{\circ}}}
\end{aligned}
\end{align}
Now this is the incarnation:
\begin{align*}
a_1 & \mapsto \varpi_1(\bq^{(6)}), &
a_2 & \mapsto \bq^{(1)}, &
a_3 & \mapsto \bq^{(4)}, &
a_4 & \mapsto s_{45}(\bq^{(2)}),
\\
a_5 & \mapsto \varpi_1(\bq^{(5)}), &
a_6 & \mapsto \varpi_2(\bq^{(3)}), &
a_7 & \mapsto s_{45}(\bq^{(3)}), &
a_8 & \mapsto \varpi_3(\bq^{(5)}),
\\
a_9 & \mapsto \varpi_2(\bq^{(2)}), &
a_{10} & \mapsto s_{34}(\bq^{(4)}), &
a_{11} & \mapsto s_{34}(\bq^{(1)}), &
a_{12} & \mapsto \varpi_3(\bq^{(6)}).
\end{align*}

\subsubsection{PBW-basis and dimension} \label{subsubsec:type-ufo1-PBW}
Notice that the roots in each $\Delta_{+}^{a_i}$, $i\in\I_{12}$, are ordered from left to right, justifying the notation $\beta_1, \dots, \beta_{30}$.

The root vectors $x_{\beta_k}$ are described as in Remark \ref{rem:lyndon-word}.
Thus
\begin{align*}
\left\{ x_{\beta_{30}}^{n_{30}} \dots x_{\beta_2}^{n_{2}}  x_{\beta_1}^{n_{1}} \, | \, 0\le n_{k}<N_{\beta_k} \right\}.
\end{align*}
is a PBW-basis of $\toba_{\bq}$. Hence $\dim \toba_{\bq}=2^{15}4^{15}=2^{45}$.

\subsubsection{The Dynkin diagram \emph{(\ref{eq:dynkin-ufo(1)}
		a)}}\label{subsubsec:ufo(1)-a}

\

The Nichols algebra $\toba_{\bq}$ is generated by $(x_i)_{i\in \I_5}$ with defining relations
\begin{align}\label{eq:rels-ufo(1)-a}
\begin{aligned}
\begin{aligned}
x_{112}&=0; & x_{221}&=0; & x_{223}&=0; & & [[[x_{(14)},x_3]_c,x_2]_c,x_3]_c=0; \\
x_{554}&=0; & x_{34}^2&=0; & x_3^2&=0; & & x_4^2=0; \ x_{ij}=0, i<j,  \, \widetilde{q}_{ij}=1;
\end{aligned}
\\
\begin{aligned}
& x_{\alpha}^{4}=0, \ \alpha\in\Oc^{\bq}_+; & & [[x_{(25)},x_3]_c,x_4]_c= q_{34}\zeta[[x_{(25)},x_4]_c,x_3]_c.
\end{aligned}
\end{aligned}
\end{align}
Here {\scriptsize$\Oc^{\bq}_+=\{ 1, 12, 2, 12^23^24, 123^24, 23^24,
	12^23^34^35, 1234^25, 234^25, 34^25, 12^23^34^35^2, 12^23^245$, \\ $123^245, 23^245,
	5 \}$}, and the degree of the integral is
\begin{align*}
\ya &= 33 \alpha_1+ 60\alpha_2 +81\alpha_3 +70\alpha_4 +38\alpha_5.
\end{align*}

\subsubsection{The Dynkin diagram \emph{(\ref{eq:dynkin-ufo(1)}
		b)}}\label{subsubsec:ufo(1)-b}

\

The Nichols algebra $\toba_{\bq}$ is generated by $(x_i)_{i\in \I_5}$ with defining relations
\begin{align}\label{eq:rels-ufo(1)-b}
\begin{aligned}
x_{112}&=0; & x_{221}&=0; & x_{223}&=0; & x_{ij}&=0, & & i<j, \, \widetilde{q}_{ij}=1;\\
x_{332}&=0; &  x_{334}&=0; & x_{554}&=0; & x_{4}^2&=0; & & x_{\alpha}^{4}=0, \ \alpha\in\Oc^{\bq}_+.
\end{aligned}
\end{align}
Here {\scriptsize$\Oc^{\bq}_+=\{ 1, 12, 2, 123, 23, 3,
	12^23^34^35, 12^23^24^35, 123^24^35, 23^24^35, 12^23^34^35^2, 12^23^24^35^2$, \\ $123^24^35^2, 23^24^35^2, 5 \}$}, and the degree of the integral is
\begin{align*}
\ya &= 33 \alpha_1+ 60\alpha_2 +81\alpha_3 +96\alpha_4 +51\alpha_5.
\end{align*}

\subsubsection{The Dynkin diagram \emph{(\ref{eq:dynkin-ufo(1)}
		c)}}\label{subsubsec:ufo(1)-c}

\

The Nichols algebra $\toba_{\bq}$ is generated by $(x_i)_{i\in \I_5}$ with defining relations
\begin{align}\label{eq:rels-ufo(1)-c}
\begin{aligned}
&\begin{aligned}
&[x_{(13)},x_2]_c=0; & x_{112}&=0; & x_{443}&=0; & x_{ij}&=0, & & i<j, \, \widetilde{q}_{ij}=1; \\
&[x_{(24)},x_3]_c=0; & x_2^2&=0; & x_{35}^2&=0; & x_3^2&=0; & & x_{\alpha}^{4}=0, \ \alpha\in\Oc^{\bq}_+;
\end{aligned}
\\
& [x_{125},x_2]_c=0; \quad x_{5}^2=0; \quad x_{235}=2q_{23}x_3x_{25}-q_{35}(1+\zeta)[x_{25},x_3]_c.
\end{aligned}
\end{align}

Here {\scriptsize$\Oc^{\bq}_+=\{ 1, 123, 23, 1234, 234, 4,
	12^2345, 12^235, 12^23^34^25^2, 12^23^345^2, 3^245, 123^245^2, 125, 23^245^2,
	25 \}$}, and the degree of the integral is
\begin{align*}
\ya &= 33 \alpha_1+ 60\alpha_2 +81\alpha_3 +38\alpha_4 +51\alpha_5.
\end{align*}

\subsubsection{The Dynkin diagram \emph{(\ref{eq:dynkin-ufo(1)}
		d)}}\label{subsubsec:ufo(1)-d}

\

The Nichols algebra $\toba_{\bq}$ is generated by $(x_i)_{i\in \I_5}$ with defining relations
\begin{align}\label{eq:rels-ufo(1)-d}
\begin{aligned}
& \begin{aligned}
x_{112}&=0; & x_{221}&=0; & x_{223}&=0; & x_{ij}&=0, \ i<j, \, \widetilde{q}_{ij}=1;\\
& & x_{3}^2&=0; & x_4^2&=0; &  [[[x_{1235},&x_3]_c,x_2]_c,x_3]_c=0;\\
[x_{(24)},&x_3]_c=0; & x_5^2&=0; & x_{35}^2&=0; & x_{\alpha}^{4}&=0, \ \alpha\in\Oc^{\bq}_+;
\end{aligned}
\\
& x_{(35)} +\frac{q_{45}(1+\zeta)}{2}[x_{35},x_4]_c -q_{34}(1-\zeta)x_4x_{35}=0.
\end{aligned}
\end{align}
Here {\scriptsize$\Oc^{\bq}_+=\{ 1, 12, 2, 1234, 234, 34,
	12^23^345, 12^23^25, 123^25, 23^25, 12^23^34^25^2, 12^23^245^2$, \\ $123^245^2, 23^245^2,
	45 \}$}, and the degree of the integral is
\begin{align*}
\ya &= 33 \alpha_1+ 60\alpha_2 +81\alpha_3 +38\alpha_4 +51\alpha_5.
\end{align*}

\subsubsection{The Dynkin diagram \emph{(\ref{eq:dynkin-ufo(1)}
		e)}}\label{subsubsec:ufo(1)-e}

\

The Nichols algebra $\toba_{\bq}$ is generated by $(x_i)_{i\in \I_5}$ with defining relations
\begin{align}\label{eq:rels-ufo(1)-e}
\begin{aligned}
& [x_{(13)},x_2]_c=0; & x_{332}&=0; & x_{334}&=0; & x_1^2&=0; & x_{ij}&=0, \ \widetilde{q}_{ij}=1;\\
& [x_{125},x_2]_c=0; & x_{443}&=0; & x_{552}&=0; & x_2^2&=0; & x_{\alpha}^{4}&=0, \ \alpha\in\Oc^{\bq}_+.
\end{aligned}
\end{align}
Here {\scriptsize$\Oc^{\bq}_+=\{ 12, 123, 3, 1234, 34, 4, 2^33^245, 12^43^34^25^2, 12^43^345^2, 12^43^245^2, 12345, 1235, 125, 2^33^245^2, 5 \}$}, and the degree of the integral is
\begin{align*}
\ya &= 33 \alpha_1+ 96\alpha_2 +70\alpha_3 +38\alpha_4 +51\alpha_5.
\end{align*}

\subsubsection{The Dynkin diagram \emph{(\ref{eq:dynkin-ufo(1)}
		f)}}\label{subsubsec:ufo(1)-f}

\

The Nichols algebra $\toba_{\bq}$ is generated by $(x_i)_{i\in \I_5}$ with defining relations
\begin{align}\label{eq:rels-ufo(1)-f}
\begin{aligned}
x_{221}&=0; & x_{223}&=0; & x_{225}&=0; & x_{552}&=0; & x_{ij}&=0, \ i<j \ \widetilde{q}_{ij}=1;
\\
x_{332}&=0; & x_{334}&=0; & x_{443}&=0; & x_1^2&=0; & x_{\alpha}^{4}&=0, \ \alpha\in\Oc^{\bq}_+.
\end{aligned}
\end{align}
Here {\scriptsize$\Oc^{\bq}_+=\{ 2, 23, 3, 234, 34, 4, 1^32^33^245, 1^32^43^34^25^2, 1^32^43^345^2, 1^32^43^245^2, 1^32^33^245^2, 2345, 235, 25, 5 \}$}, and the degree of the integral is
\begin{align*}
\ya &= 65 \alpha_1+ 96\alpha_2 +70\alpha_3 +38\alpha_4 +51\alpha_5.
\end{align*}

\subsubsection{The associated Lie algebra} This is of type $A_5$.

\subsection{Type $\Ufo(2)$}\label{subsec:type-ufo(2)}
Here $\zeta \in \G'_4$.
We describe first  the root system $\Ufo(2)$.

\subsubsection{Basic datum and root system}
Below, $A_6$, $E_6$, $_{3}T$ and $_{2}T_1$ are numbered as in \eqref{eq:dynkin-system-A}, \eqref{eq:dynkin-system-E} and  	
\eqref{eq:mTn}, respectively.
The basic datum and the bundle of Cartan matrices are described by the following diagram:
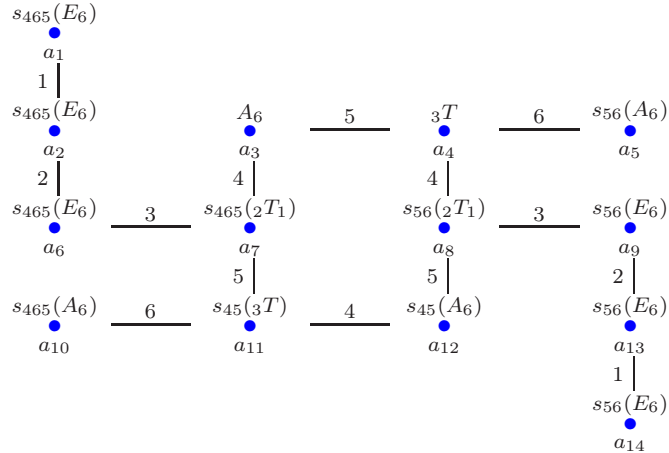
\begin{center}
	\begin{tabular}{c c c c c c c }
		$\overset{s_{465}(E_6)}{\underset{a_1}{\vtxgpd}}$
		&
		&
		&
		&
		&
		&
		\\
		{\scriptsize 1} \vline\hspace{5pt}
		& &
		& &
		& &
		\\
		$\overset{s_{465}(E_6)}{\underset{a_2}{\vtxgpd}}$
		&
		& $\overset{A_6}{\underset{a_3}{\vtxgpd}}$
		& \hspace{-5pt}\raisebox{3pt}{$\overset{5}{\rule{30pt}{0.5pt}}$}\hspace{-5pt}
		& $\overset{{}_3T}{\underset{a_4}{\vtxgpd}}$
		& \hspace{-5pt}\raisebox{3pt}{$\overset{6}{\rule{30pt}{0.5pt}}$}\hspace{-5pt}
		& $\overset{s_{56}(A_6)}{\underset{a_5}{\vtxgpd}}$
		\\
		{\scriptsize 2} \vline\hspace{5pt}
		& & {\scriptsize 4} \vline\hspace{5pt}
		& & {\scriptsize 4} \vline\hspace{5pt}
		& &
		\\
		$\overset{s_{465}(E_6)}{\underset{a_6}{\vtxgpd}}$
		& \hspace{-5pt}\raisebox{3pt}{$\overset{3}{\rule{30pt}{0.5pt}}$}\hspace{-5pt}
		& $\overset{s_{465}({}_2T_1)}{\underset{a_7}{\vtxgpd}}$
		&
		& $\overset{s_{56}({}_2T_1)}{\underset{a_8}{\vtxgpd}}$
		& \hspace{-5pt}\raisebox{3pt}{$\overset{3}{\rule{30pt}{0.5pt}}$}\hspace{-5pt}
		& $\overset{s_{56}(E_6)}{\underset{a_9}{\vtxgpd}}$
		\\
		& &{\scriptsize 5} \vline\hspace{5pt}
		& & {\scriptsize 5} \vline\hspace{5pt}
		& & {\scriptsize 2} \vline\hspace{5pt}
		\\
		$\overset{s_{465}(A_6)}{\underset{a_{10}}{\vtxgpd}}$
		& \hspace{-5pt}\raisebox{3pt}{$\overset{6}{\rule{30pt}{0.5pt}}$}\hspace{-5pt}
		& $\overset{s_{45}({}_3T)}{\underset{a_{11}}{\vtxgpd}}$
		& \hspace{-5pt}\raisebox{3pt}{$\overset{4}{\rule{30pt}{0.5pt}}$}\hspace{-5pt}
		& $\overset{s_{45}(A_6)}{\underset{a_{12}}{\vtxgpd}}$
		&
		& $\overset{s_{56}(E_6)}{\underset{a_{13}}{\vtxgpd}}$
		\\
		& &
		& &
		& & {\scriptsize 1} \vline\hspace{5pt}
		\\
		&
		&
		&
		&
		&
		& $\overset{s_{56}(E_6)}{\underset{a_{14}}{\vtxgpd}}$
	\end{tabular}
\end{center}
Using the notation \eqref{eq:notation-root-exceptional}, we set: { \scriptsize
	\begin{align*}
	\Delta_{+}^{(1)}= & \{ 1, 12, 2, 123, 23, 3, 1234, 234, 34, 4, 12^23^24^25, 123^24^25, 1234^25, 12345, 23^24^25, 234^25, \\
	& 2345, 34^25, 345, 45, 5, 12^23^34^45^36, 12^23^34^35^36, 12^23^24^35^36, 123^24^35^36, 23^24^35^36, \\
	& 12^23^34^45^26, 12^23^34^35^26, 12^23^24^35^26, 123^24^35^26, 23^24^35^26, 1^22^33^44^55^46^2, 12^33^44^55^46^2, \\
	& 12^23^24^25^26, 12^23^24^256, 12^23^44^55^46^2, 123^24^25^26, 123^24^256, 23^24^25^26, 23^24^256, \\
	& 12^23^34^55^46^2, 12^23^34^45^46^2, 12^23^34^45^36^2, 1234^25^26, 234^25^26, 34^25^26, 1234^256, \\
	& 234^256, 34^256, 12^23^34^35^36^2, 12^23^24^35^36^2, 123^24^35^36^2, 12345^26, 123456, 23^24^35^36^2, \\
	& 2345^26, 23456, 345^26, 3456, 45^26, 456, 56, 6 \}, \\
	\Delta_{+}^{(2)}= & \{ 1, 12, 2, 123, 23, 3, 1234, 234, 34, 4, 12345, 2345, 345, 45, 5, 12^23^34^35^36, 12^23^24^35^36, \\
	& 12^23^24^25^36, 12^23^24^25^26, 123^24^35^36, 123^24^25^36, 123^24^25^26, 23^24^35^36, 23^24^25^36, \\
	& 23^24^25^26, 12^23^34^45^56^2, 1234^25^36, 234^25^36, 34^25^36, 12^23^34^45^46^2, 12^23^34^35^46^2, \\
	& 1^22^33^44^55^66^3, 12^33^44^55^66^3, 12^23^44^55^66^3, 12^23^34^35^36^2, 1234^25^26, 12345^26, 12^23^24^35^46^2, \\
	& 123^24^35^46^2, 123456, 12^23^34^55^66^3, 12^23^24^35^36^2, 234^25^26, 12^23^34^45^66^3, 12^23^24^25^36^2, \\
	& 2345^26, 12^23^34^45^56^3, 123^24^35^36^2, 123^24^25^36^2, 1234^25^36^2, 23^24^35^46^2, 23456, 23^24^35^36^2, \\
	& 23^24^25^36^2, 234^25^36^2, 34^25^26, 345^26, 45^26, 34^25^36^2, 3456, 456, 56, 6 \}, \\
	\Delta_{+}^{(3)}= & \{ 1, 12, 2, 123, 23, 3, 1234, 234, 34, 4, 12^23^245, 123^245, 12345, 1235, 23^245, 2345, 235, 3^245, \\
	& 345, 35, 5, 12^23^54^35^36, 12^23^54^25^36, 12^23^44^25^36, 123^44^25^36, 23^44^25^36, 12^23^44^25^26, \\
	& 123^44^25^26, 23^44^25^26, 12^23^34^25^26, 123^34^25^26, 23^34^25^26, 3^34^25^26, 1^22^33^64^35^46^2, \\
	& 12^33^64^35^46^2, 12^23^64^35^46^2, 12^23^54^35^46^2, 12^23^54^35^36^2, 12^23^345^26, 12^23^245^26, \\
	& 12^23^2456, 123^345^26, 123^245^26, 123^2456, 123456, 12^23^54^25^46^2, 12^23^54^25^36^2, \\
	& 12^23^44^25^36^2, 123^44^25^36^2, 12356, 23^345^26, 3^345^26, 23^245^26, 3^245^26, 23^44^25^36^2, 23^2456, \\
	& 23456, 2356, 3^2456, 3456, 356, 56, 6 \}, \\
	\Delta_{+}^{(4)}= & \{ 1, 12, 2, 123, 23, 3, 1234, 234, 34, 4, 1^22^23^245, 12^23^245, 123^245, 12345, 1235, 23^245, 2345, \\
	& 235, 345, 35, 5, 1^32^43^54^35^36, 1^32^43^54^25^36, 1^32^43^44^25^36, 1^32^33^44^25^36, 1^22^33^44^25^36, \\
	& 1^32^43^44^25^26, 1^32^33^44^25^26, 1^22^33^44^25^26, 1^32^33^34^25^26, 1^22^33^34^25^26, 1^22^23^34^25^26, \\
	& 12^23^34^25^26, 1^42^53^64^35^46^2, 1^32^33^345^26, 1^32^53^64^35^46^2, 1^32^43^64^35^46^2, 1^32^43^54^35^46^2, \\
	& 1^32^43^54^35^36^2, 1^22^33^345^26, 1^22^23^345^26, 1^22^23^245^26, 1^22^23^2456, 1^32^43^54^25^46^2, \\
	& 1^32^43^54^25^36^2, 12^23^345^26, 1^32^43^44^25^36^2, 12^23^245^26, 12^23^2456, 1^32^33^44^25^36^2, 1^22^33^44^25^36^2, \\
	& 123^245^26, 123^2456, 123456, 12356, 23^245^26, 23^2456, 23456, 2356, 3456, 356, 56, 6 \}, \\
	\Delta_{+}^{(5)}= & \{ 1, 12, 2, 123, 23, 3, 1234, 234, 34, 4, 12^23^24^25, 123^24^25, 1234^25, 12345, 23^24^25, 234^25, 2345, \\
	& 34^25, 345, 45, 5, 12^23^34^45^36, 12^23^34^45^26, 12^23^34^35^26, 12^23^24^35^26, 123^24^35^26, 23^24^35^26, \\
	& 12^23^24^25^26, 123^24^25^26, 23^24^25^26, 1234^25^26, 234^25^26, 34^25^26, 1^22^33^44^55^36^2, 12^33^44^55^36^2, \\
	& 12^23^44^55^36^2, 12^23^34^55^36^2, 12^23^34^45^36^2, 12^23^34^356, 12^23^24^356, 12^23^24^256, 123^24^356, \\
	& 123^24^256, 1234^256, 123456, 12^23^34^45^26^2, 23^24^356, 12^23^34^35^26^2, 12^23^24^35^26^2, 123^24^35^26^2, \\
	& 12346, 23^24^256, 234^256, 34^256, 23^24^35^26^2, 23456, 2346, 3456, 346, 456, 56, 46, 6 \}, \\
	\Delta_{+}^{(6)}= & \{ 1, 12, 2, 123, 23, 3, 1234, 234, 34, 4, 12^23^245, 123^245, 12345, 1235, 2^23^245, 23^245, 2345, 235, \\
	& 345, 35, 5, 12^43^54^35^36, 12^43^54^25^36, 12^43^44^25^36, 12^33^44^25^36, 2^33^44^25^36, 12^43^44^25^26, \\
	& 12^33^44^25^26, 2^33^44^25^26, 12^33^34^25^26, 2^33^34^25^26, 12^23^34^25^26, 2^23^34^25^26, 1^22^53^64^35^46^2, \\
	& 12^53^64^35^46^2, 12^43^64^35^46^2, 12^43^54^35^46^2, 12^43^54^35^36^2, 12^33^345^26, 12^23^345^26, 12^23^245^26, \\
	& 12^23^2456, 123^245^26, 123^2456, 123456, 12^43^54^25^46^2, 12^43^54^25^36^2, 12^43^44^25^36^2, 12^33^44^25^36^2, \\
	& 12356, 2^33^345^26, 2^23^345^26, 2^23^245^26, 23^245^26, 2^33^44^25^36^2, 2^23^2456, 23^2456, 23456, 2356, \\
	& 3456, 356, 56, 6 \}, \\
	\Delta_{+}^{(7)}= & \{ 1, 12, 2, 123, 23, 3, 1234, 234, 34, 4, 12^23^245, 123^245, 12345, 1235, 23^245, 2345, 235, 345, 35, \\
	& 45, 5, 12^23^34^35^36, 12^23^34^25^36, 12^23^24^25^36, 123^24^25^36, 23^24^25^36, 12^23^34^25^26, 12^23^24^25^26, \\
	& 123^24^25^26, 23^24^25^26, 1234^25^26, 234^25^26, 34^25^26, 1^22^33^44^35^46^2, 12^33^44^35^46^2, \\
	& 12^23^44^35^46^2, 12^23^34^35^46^2, 12^23^34^35^36^2, 12^23^345^26, 12^23^245^26, 12^23^2456, 123^245^26, \\
	& 123^2456, 12345^26, 123456, 12^23^34^25^46^2, 12^23^34^25^36^2, 12^23^24^25^36^2, 123^24^25^36^2, 12356, \\
	& 23^245^26, 2345^26, 345^26, 45^26, 23^24^25^36^2, 23^2456, 23456, 2356, 3456, 356, 456, 56, 6 \}.
	\end{align*}
}
Now the bundle of sets of (positive) roots is described as follows:
\begin{align*}
a_1 & \mapsto s_{465}(\Delta_+^{(6)}), &
a_2 & \mapsto s_{465}(\Delta_+^{(4)}), &
a_3 & \mapsto \Delta_+^{(1)}, &
a_4 & \mapsto \Delta_+^{(5)},
\\
a_5 & \mapsto s_{56}(\Delta_+^{(2)}), &
a_6 & \mapsto s_{465}(\Delta_+^{(3)}), &
a_7 & \mapsto s_{465}(\Delta_+^{(7)}), &
a_8 & \mapsto s_{56}(\Delta_+^{(7)}),
\\
a_9 & \mapsto s_{56}(\Delta_+^{(3)}), &
a_{10} & \mapsto s_{465}(\Delta_+^{(2)}), &
a_{11} & \mapsto s_{45}(\Delta_+^{(5)}), &
a_{12} & \mapsto s_{45}(\Delta_+^{(1)}),
\\
a_{13} & \mapsto s_{56}(\Delta_+^{(4)}), &
a_{14} & \mapsto s_{56}(\Delta_+^{(6)}).
\end{align*}

\subsubsection{Weyl groupoid}
\label{subsubsec:type-ufo2-Weyl}
The isotropy group  at $a_1 \in \cX$ is
\begin{align*}
\cW(a_1)= \langle \varsigma_1^{a_1}\varsigma_2 \varsigma_3\varsigma_4 \varsigma_5 \varsigma_6 \varsigma_5 \varsigma_4 \varsigma_3 \varsigma_2 \varsigma_1, \varsigma_2^{a_1}, \varsigma_3^{a_1}, \varsigma_4^{a_1},  \varsigma_5^{a_1}, \varsigma_6^{a_1} \rangle \simeq W(E_6).
\end{align*}

\subsubsection{Incarnation}
We set the matrices $(\bq^{(i)})_{i\in\I_{7}}$, from left to right and  from up to down:
\begin{align}\label{eq:dynkin-ufo(2)}
\begin{aligned}
&
\xymatrix@C-4pt{\overset{\zeta}{\underset{\ }{\circ}}\ar  @{-}[r]^{\ztu}  &
	\overset{\zeta}{\underset{\ }{\circ}} \ar  @{-}[r]^{\ztu}  & \overset{\zeta}{\underset{\
		}{\circ}}
	\ar  @{-}[r]^{\ztu}  & \overset{-1}{\underset{\ }{\circ}} \ar  @{-}[r]^{-1}  &
	\overset{-1}{\underset{\ }{\circ}}  \ar  @{-}[r]^{\ztu}  & \overset{\zeta}{\underset{\
		}{\circ}}}
& &
\\
&
\xymatrix@C-4pt{\overset{\zeta}{\underset{\ }{\circ}}\ar  @{-}[r]^{\ztu}  &
	\overset{\zeta}{\underset{\ }{\circ}} \ar  @{-}[r]^{\ztu}  & \overset{\zeta}{\underset{\
		}{\circ}}
	\ar  @{-}[r]^{\ztu}  & \overset{\zeta}{\underset{\ }{\circ}} \ar  @{-}[r]^{\ztu}  &
	\overset{-1}{\underset{\ }{\circ}}  \ar  @{-}[r]^{\ztu}  & \overset{\zeta}{\underset{\
		}{\circ}}}
\end{aligned}
\end{align}

\begin{align*}
&\xymatrix@C-5pt@R-8pt{  &   &  \overset{\zeta}{\circ} \ar  @{-}[d]^{\ztu} & & \\
	\overset{\zeta}{\underset{\ }{\circ}} \ar  @{-}[r]^{\ztu}  & \overset{-1}{\underset{\
		}{\circ}} \ar  @{-}[r]^{\zeta}  & \overset{-1}{\underset{\ }{\circ}} \ar  @{-}[r]^{\ztu}
	& \overset{\zeta}{\underset{\ }{\circ}} \ar  @{-}[r]^{\ztu}  & \overset{\zeta}{\underset{\
		}{\circ}}}
&& \xymatrix@C-5pt@R-8pt{  &   & \overset{\zeta}{\circ} \ar  @{-}[d]^{\ztu}  & &\\
	\overset{-1}{\underset{\ }{\circ}} \ar  @{-}[r]^{\zeta}  & \overset{-1}{\underset{\
		}{\circ}} \ar  @{-}[r]^{\ztu}  & \overset{\zeta}{\underset{\ }{\circ}} \ar  @{-}[r]^{\ztu}
	& \overset{\zeta}{\underset{\ }{\circ}} \ar  @{-}[r]^{\ztu}  &
	\overset{\zeta}{\underset{\ }{\circ}}}
\\
&\xymatrix@C-5pt@R-8pt{  &   & & \overset{-1}{\circ} \ar  @{-}[d]^{-1} \ar
	@{-}[dr]^{\zeta}  & \\
	\overset{\zeta}{\underset{\ }{\circ}} \ar  @{-}[r]^{\ztu}  & \overset{\zeta}{\underset{\
		}{\circ}} \ar  @{-}[r]^{\ztu}  & \overset{\zeta}{\underset{\ }{\circ}} \ar  @{-}[r]^{\ztu}
	& \overset{-1}{\underset{\ }{\circ}} \ar  @{-}[r]^{\zeta}  & \overset{-1}{\underset{\
		}{\circ}}}
&& \xymatrix@C-5pt@R-8pt{  &   & \overset{\zeta}{\circ} \ar  @{-}[d]^{\ztu}  & &\\
	\overset{-1}{\underset{\ }{\circ}} \ar  @{-}[r]^{\ztu}  & \overset{\zeta}{\underset{\
		}{\circ}} \ar  @{-}[r]^{\ztu}  & \overset{\zeta}{\underset{\ }{\circ}} \ar  @{-}[r]^{\ztu}
	& \overset{\zeta}{\underset{\ }{\circ}} \ar  @{-}[r]^{\ztu}  &
	\overset{\zeta}{\underset{\ }{\circ}}}
\\
& \xymatrix@C-5pt@R-8pt{  &   & \overset{-1}{\circ} \ar  @{-}[d]^{\zeta} \ar
	@{-}[dr]^{-1}  & &\\
	\overset{\zeta}{\underset{\ }{\circ}} \ar  @{-}[r]^{\ztu}  & \overset{\zeta}{\underset{\
		}{\circ}} \ar  @{-}[r]^{\ztu}  & \overset{-1}{\underset{\ }{\circ}} \ar  @{-}[r]^{\zeta}
	& \overset{-1}{\underset{\ }{\circ}}
	\ar  @{-}[r]^{\ztu}  & \overset{\zeta}{\underset{\ }{\circ}}} &&
\end{align*}
Now this is the incarnation:
\begin{align*}
a_1 & \mapsto s_{465}(\bq^{(6)}), &
a_2 & \mapsto s_{465}(\bq^{(4)}), &
a_3 & \mapsto \bq^{(1)}, &
a_4 & \mapsto \bq^{(5)},
\\
a_5 & \mapsto s_{56}(\bq^{(2)}), &
a_6 & \mapsto s_{465}(\bq^{(3)}), &
a_7 & \mapsto s_{465}(\bq^{(7)}), &
a_8 & \mapsto s_{56}(\bq^{(7)}),
\\
a_9 & \mapsto s_{56}(\bq^{(3)}), &
a_{10} & \mapsto s_{465}(\bq^{(2)}), &
a_{11} & \mapsto s_{45}(\bq^{(5)}), &
a_{12} & \mapsto s_{45}(\bq^{(1)}),
\\
a_{13} & \mapsto s_{56}(\bq^{(4)}), &
a_{14} & \mapsto s_{56}(\bq^{(6)}).
\end{align*}

\subsubsection{PBW-basis and dimension} \label{subsubsec:type-ufo2-PBW}
Notice that the roots in each $\Delta_{+}^{a_i}$, $i\in\I_{14}$, are ordered from left to right, justifying the notation $\beta_1, \dots, \beta_{63}$.

The root vectors $x_{\beta_k}$ are described as in Remark \ref{rem:lyndon-word}.
Thus
\begin{align*}
\left\{ x_{\beta_{63}}^{n_{63}} \dots x_{\beta_2}^{n_{2}}  x_{\beta_1}^{n_{1}} \, | \, 0\le n_{k}<N_{\beta_k} \right\}.
\end{align*}
is a PBW-basis of $\toba_{\bq}$. Hence $\dim \toba_{\bq}=2^{27}4^{36}=2^{99}$.

\subsubsection{The Dynkin diagram \emph{(\ref{eq:dynkin-ufo(2)}
		a)}}\label{subsubsec:ufo(2)-a}

\

The Nichols algebra $\toba_{\bq}$ is generated by $(x_i)_{i\in \I_6}$ with defining relations
\begin{align}\label{eq:rels-ufo(2)-a}
\begin{aligned}
& \begin{aligned}
x_{112}&=0; & x_{221}&=0; & x_{223}&=0; & x_{ij}&=0, \ i<j, \, \widetilde{q}_{ij}=1;\\
x_{332}&=0; & x_{334}&=0; & x_{665}&=0; &  [[[x_{(25)},&x_4]_c,x_3]_c,x_4]_c=0;\\
x_4^2&=0; & x_5^2&=0; & x_{45}^2&=0; & x_{\alpha}^{4}&=0, \ \alpha\in\Oc^{\bq}_+;
\end{aligned}
\\
&[[x_{(36)},x_4]_c,x_5]_c=q_{45}\zeta[[x_{(36)},x_5]_c,x_4]_c.
\end{aligned}
\end{align}
Here {\scriptsize$\Oc^{\bq}_+=\{ 1, 12, 2, 123, 23, 3, 12^23^24^25, 123^24^25, 1234^25, 23^24^25, 234^25, 34^25, 12^23^34^35^36$, \\ $12^23^24^35^36,123^24^35^36, 23^24^35^36, 12^23^34^45^26, 1^22^33^44^55^46^2, 12^33^44^55^46^2, 12^23^24^256$, \\ $12^23^44^55^46^2, 123^24^256, 23^24^256, 12^23^34^55^46^2, 1234^256, 234^256, 34^256, 12^23^34^35^36^2$, \\ $12^23^24^35^36^2, 123^24^35^36^2, 12345^26, 23^24^35^36^2, 2345^26, 345^26, 45^26, 6 \}$}, and the degree of the integral is
\begin{align*}
\ya &= 78 \alpha_1+ 150\alpha_2 +216\alpha_3 +276\alpha_4 +226\alpha_5 +116\alpha_6.
\end{align*}

\subsubsection{The Dynkin diagram \emph{(\ref{eq:dynkin-ufo(2)}
		b)}}\label{subsubsec:ufo(2)-b}

\

The Nichols algebra $\toba_{\bq}$ is generated by $(x_i)_{i\in \I_6}$ with defining relations
\begin{align}\label{eq:rels-ufo(2)-b}
\begin{aligned}
x_{112}&=0; & x_{221}&=0; & x_{223}&=0; & x_{ij}&=0, \ i<j, \, \widetilde{q}_{ij}=1;\\
& & x_{332}&=0; &  x_{334}&=0; & x_{443}&=0; \\
x_{445}&=0; & x_{665}&=0; & x_{5}^2&=0; & x_{\alpha}^{4}&=0, \ \alpha\in\Oc^{\bq}_+.
\end{aligned}
\end{align}
Here {\scriptsize$\Oc^{\bq}_+=\{ 1, 12, 2, 123, 23, 3, 1234, 234, 34, 4, 12^23^34^35^36, 12^23^24^35^36, 12^23^24^25^36, 123^24^35^36$, \\ $123^24^25^36, 23^24^35^36, 23^24^25^36, 1234^25^36, 234^25^36, 34^25^36, 1^22^33^44^55^66^3, 12^33^44^55^66^3$, \\ $12^23^44^55^66^3, 12^23^34^35^36^2, 12^23^34^55^66^3, 12^23^24^35^36^2, 12^23^34^45^66^3, 12^23^24^25^36^2, 123^24^35^36^2$, \\ $123^24^25^36^2, 1234^25^36^2, 23^24^35^36^2, 23^24^25^36^2, 234^25^36^2, 34^25^36^2, 6 \}$}, and the degree of the integral is
\begin{align*}
\ya &= 78 \alpha_1+ 150\alpha_2 +216\alpha_3 +276\alpha_4 +330\alpha_5 +168\alpha_6.
\end{align*}

\subsubsection{The Dynkin diagram \emph{(\ref{eq:dynkin-ufo(2)}
		c)}}\label{subsubsec:ufo(2)-c}

\

The Nichols algebra $\toba_{\bq}$ is generated by $(x_i)_{i\in \I_6}$ with defining relations
\begin{align}\label{eq:rels-ufo(2)-c}
\begin{aligned}
& [x_{(13)},x_2]_c=0; & x_{112}&=0; & x_{443}&=0; & x_{ij}&=0, \ i<j, \, \widetilde{q}_{ij}=1;\\
& [x_{236},x_3]_c=0; & x_{445}&=0; & x_{554}&=0; & x_{663}&=0; \\
& [x_{(24)},x_3]_c=0; & x_2^2&=0; & x_3^2&=0; & x_{\alpha}^{4}&=0, \ \alpha\in\Oc^{\bq}_+.
\end{aligned}
\end{align}
Here {\scriptsize$\Oc^{\bq}_+=\{ 1, 123, 23, 1234, 234, 4, 12^23^245, 12345, 1235, 2345, 235, 5, 12^23^54^35^36, 12^23^54^25^36$, \\ $123^44^25^36, 23^44^25^36, 123^44^25^26, 23^44^25^26, 3^34^25^26, 1^22^33^64^35^46^2, 12^33^64^35^46^2, 12^23^54^35^46^2$, \\ $12^23^54^35^36^2, 12^23^245^26, 12^23^2456, 123456, 12^23^54^25^46^2, 12^23^54^25^36^2, 123^44^25^36^2, 12356$, \\ $3^345^26, 13^44^25^36^2, 23456, 2356, 56, 6 \}$}, and the degree of the integral is
\begin{align*}
\ya &= 78 \alpha_1+ 150\alpha_2 +330\alpha_3 +226\alpha_4 +116\alpha_5 +168\alpha_6.
\end{align*}

\subsubsection{The Dynkin diagram \emph{(\ref{eq:dynkin-ufo(2)}
		d)}}\label{subsubsec:ufo(2)-d}

\

The Nichols algebra $\toba_{\bq}$ is generated by $(x_i)_{i\in \I_6}$ with defining relations
\begin{align}\label{eq:rels-ufo(2)-d}
\begin{aligned}
x_{332}&=0; & x_{334}&=0; & x_{336}&=0; & & x_{ij}=0, \ i<j, \, \widetilde{q}_{ij}=1;\\
x_{663}&=0; & x_{443}&=0; & x_{445}&=0; &  & [x_{(13)},x_2]_c=0;\\
x_{554}&=0; & x_1^2&=0; & x_2^2&=0; & & x_{\alpha}^{4}=0, \ \alpha\in\Oc^{\bq}_+.
\end{aligned}
\end{align}
Here {\scriptsize$\Oc^{\bq}_+=\{ 12, 123, 3, 1234, 34, 4, 123^245, 12345, 1235, 345, 24, 5, 12^43^54^35^36, 12^43^54^25^36, 12^43^44^25^36$, \\ $2^33^44^25^36, 12^43^44^25^26, 2^33^44^25^26, 2^33^34^25^26, 1^22^53^64^35^46^2, 12^43^64^35^46^2, 12^43^54^35^46^2$, \\ $12^43^54^35^36^2, 123^245^26, 123^2456, 123456, 12^43^54^25^46^2, 12^43^54^25^36^2, 12^43^44^25^36^2, 12356$, \\ $2^33^345^26, 2^33^44^25^36^2, 3456, 245, 56, 6 \}$}, and the degree of the integral is
\begin{align*}
\ya &= 78 \alpha_1+ 260\alpha_2 +330\alpha_3 +226\alpha_4 +116\alpha_5 +168\alpha_6.
\end{align*}

\subsubsection{The Dynkin diagram \emph{(\ref{eq:dynkin-ufo(2)}
		e)}}\label{subsubsec:ufo(2)-e}

\

The Nichols algebra $\toba_{\bq}$ is generated by $(x_i)_{i\in \I_6}$ with defining relations
\begin{align}\label{eq:rels-ufo(2)-e}
\begin{aligned}
\begin{aligned}
x_{112}&=0; & x_{221}&=0; & x_{223}&=0; & &x_{ij}=0, \ i<j, \, \widetilde{q}_{ij}=1;\\
x_{46}^2&=0; & x_{332}&=0; & x_{334}&=0; &  & [[[x_{2346},x_4]_c,x_3]_c,x_4]_c=0;\\
x_{4}^2&=0; & x_5^2&=0; & x_6^2&=0; & & x_{\alpha}^{4}=0, \ \alpha\in\Oc^{\bq}_+;
\end{aligned}
\\
\begin{aligned}
& [x_{(35)},x_4]_c=0; & x_{(46)}& = q_{45}(1-\zeta)x_5x_{46}- \frac{q_{56}(1+\zeta)}{2}[x_{46},x_5]_c.
\end{aligned}
\end{aligned}
\end{align}
Here {\scriptsize$\Oc^{\bq}_+=\{ 1, 12, 2, 123, 23, 3, 12^23^24^25, 123^24^25, 1234^25, 23^24^25, 234^25, 34^25, 12^23^34^45^36$, \\ $12^23^24^25^26, 123^24^25^26, 23^24^25^26, 1234^25^26, 234^25^26, 34^25^26, 1^22^33^44^55^36^2, 12^33^44^55^36^2$, \\ $12^23^44^55^36^2, 12^23^34^55^36^2, 12^23^34^356, 12^23^24^356, 123^24^356, 23^24^356, 12^23^34^35^26^2$, \\ $12^23^24^35^26^2, 123^24^35^26^2, 12346, 23^24^35^26^2, 2346, 346, 56, 46 \}$}, and the degree of the integral is
\begin{align*}
\ya &= 78 \alpha_1+ 150\alpha_2 +216\alpha_3 +276\alpha_4 +116\alpha_5 +168\alpha_6.
\end{align*}

\subsubsection{The Dynkin diagram \emph{(\ref{eq:dynkin-ufo(2)}
		f)}}\label{subsubsec:ufo(2)-f}

\

The Nichols algebra $\toba_{\bq}$ is generated by $(x_i)_{i\in \I_6}$ with defining relations
\begin{align}\label{eq:rels-ufo(2)-f}
\begin{aligned}
x_{221}&=0; & x_{223}&=0; & x_{332}&=0; & x_{ij}&=0, \ i<j, \, \widetilde{q}_{ij}=1;\\
x_{334}&=0; & x_{336}&=0; & x_{443}&=0; &  x_{445}&=0;\\
x_{554}&=0; & x_{663}&=0; & x_1^2&=0; & x_{\alpha}^{4}&=0, \ \alpha\in\Oc^{\bq}_+.
\end{aligned}
\end{align}
Here {\scriptsize$\Oc^{\bq}_+=\{ 2, 23, 3, 234, 34, 4, 23^345, 2345, 235, 345, 24, 5,
	1^32^43^54^35^36, 1^32^43^54^25^36, 1^32^43^44^25^36$, \\ $1^32^33^44^25^36, 1^32^43^44^25^26, 1^32^33^44^25^26, 1^32^33^34^25^26, 1^32^33^345^26, 1^32^53^64^35^46^2$, \\ $1^32^43^64^35^46^2, 1^32^43^54^35^46^2, 1^32^43^54^35^36^2, 1^32^43^54^25^46^2, 1^32^43^54^25^36^2, 1^32^43^44^25^36^2$, \\ $1^32^33^44^25^36^2, 23^245^26, 23^2456, 23456, 2356, 3456, 245, 56, 6 \}$}, and the degree of the integral is
\begin{align*}
\ya &= 184 \alpha_1+ 260\alpha_2 +330\alpha_3 +226\alpha_4 +116\alpha_5 +168\alpha_6.
\end{align*}

\subsubsection{The Dynkin diagram \emph{(\ref{eq:dynkin-ufo(2)}
		g)}}\label{subsubsec:ufo(2)-g}

\

The Nichols algebra $\toba_{\bq}$ is generated by $(x_i)_{i\in \I_6}$ with defining relations
\begin{align}\label{eq:rels-ufo(2)-g}
\begin{aligned}
& \begin{aligned}
& [x_{(24)},x_3]_c=0; & x_{112}&=0; & x_{221}&=0; & x_{ij}&=0, & & i<j, \, \widetilde{q}_{ij}=1;\\
& [x_{236},x_3]_c=0; & x_{223}&=0; & x_{554}&=0; & x_6^2&=0; & & x_{46}^2=0; \\
& [x_{(35)},x_4]_c=0; & x_{3}^2&=0; & x_4^2&=0; & x_{\alpha}^{4}&=0, & & \alpha\in\Oc^{\bq}_+;
\end{aligned}
\\
& x_{346}+q_{46}(1+\zeta)[x_{36},x_4]_c-2q_{34}x_4x_{36}=0.
\end{aligned}
\end{align}
Here {\scriptsize$\Oc^{\bq}_+=\{ 1, 12, 2, 1234, 234, 34, 12^23^245, 123^245, 1235, 23^345, 235, 24, 12^23^34^35^36, 12^23^24^25^36$, \\ $123^24^25^36, 23^24^25^36, 1234^25^26,
	234^25^26, 34^25^26, 1^22^33^44^35^46^2, 12^33^44^35^46^2, 12^23^44^35^46^2$, \\ $12^23^34^35^36^2, 12^23^345^26, 12^23^2456, 123^2456, 12^23^34^25^46^2, 12^23^24^25^36^2, 123^24^25^36^2$, \\ $12356, 45^26, 23^24^25^36^2, 23^2456, 2356, 245, 6 \}$}, and the degree of the integral is
\begin{align*}
\ya &= 78 \alpha_1+ 150\alpha_2 +216\alpha_3 +226\alpha_4 +116\alpha_5 +168\alpha_6.
\end{align*}

\subsubsection{The associated Lie algebra} This is of type $E_6$.

\subsection{Type $\Ufo(3)$}\label{subsec:type-ufo(3)}
Here $\zeta \in \G'_3$.
We describe first  the root system $\Ufo(3)$.

\subsubsection{Basic datum and root system}
Below, $A_3$, $B_3$, $C_3$ and $T^{(2)}$ are numbered as in \eqref{eq:dynkin-system-A}, \eqref{eq:dynkin-system-B}, \eqref{eq:dynkin-system-C} and  	
\eqref{eq:T2}, respectively.
The basic datum and the bundle of Cartan matrices are described by the following diagram:

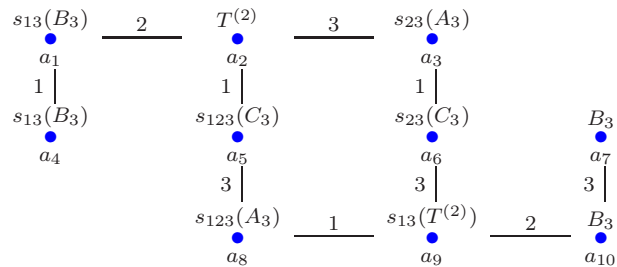
\begin{center}
	\begin{tabular}{c c c c c c c c c }
		$\overset{s_{13}(B_3)}{\underset{a_1}{\vtxgpd}}$
		& \hspace{-5pt}\raisebox{3pt}{$\overset{2}{\rule{30pt}{0.5pt}}$}\hspace{-5pt}
		& $\overset{T^{(2)}}{\underset{a_2}{\vtxgpd}}$
		&\hspace{-5pt}\raisebox{3pt}{$\overset{3}{\rule{30pt}{0.5pt}}$}\hspace{-5pt}
		& $\overset{s_{23}(A_3)}{\underset{a_3}{\vtxgpd}}$
		& & & &
		\\
		{\scriptsize 1} \vline\hspace{5pt}
		& & {\scriptsize 1} \vline\hspace{5pt}
		& & {\scriptsize 1} \vline\hspace{5pt}
		& & & &
		\\
		$\overset{s_{13}(B_3)}{\underset{a_4}{\vtxgpd}}$
		& & $\overset{s_{123}(C_3)}{\underset{a_5}{\vtxgpd}}$
		& & $\overset{s_{23}(C_3)}{\underset{a_6}{\vtxgpd}}$
		& & $\overset{B_3}{\underset{a_7}{\vtxgpd}}$ & &
		\\
		& & {\scriptsize 3} \vline\hspace{5pt}
		& & {\scriptsize 3} \vline\hspace{5pt}
		& & {\scriptsize 3} \vline\hspace{5pt} & &
		\\
		& &
		$\overset{s_{123}(A_3)}{\underset{a_8}{\vtxgpd}}$
		& \hspace{-5pt}\raisebox{3pt}{$\overset{1}{\rule{30pt}{0.5pt}}$}\hspace{-5pt}
		& $\overset{s_{13}(T^{(2)})}{\underset{a_9}{\vtxgpd}}$
		& \hspace{-5pt}\raisebox{3pt}{$\overset{2}{\rule{30pt}{0.5pt}}$}\hspace{-5pt}
		& $\overset{B_3}{\underset{a_{10}}{\vtxgpd}}$  & &
	\end{tabular}
\end{center}
Using the notation \eqref{eq:notation-root-exceptional}, we set: { \scriptsize
	\begin{align*}
	\Delta_{+}^{(1)}= & \{ 1, 12, 2, 1^22^33, 1^22^23, 12^23, 1^22^33^2, 123, 23, 3 \}, \\
	\Delta_{+}^{(2)}= & \{ 1, 12, 2, 12^33, 12^23, 2^23, 12^33^2, 123, 23, 3 \}, \\
	\Delta_{+}^{(3)}= & \{ 1, 1^22, 12, 2, 1^32^23, 1^22^23, 1^223, 123, 23, 3 \}, \\
	\Delta_{+}^{(4)}= & \{ 1, 1^22, 12, 2, 1^22^23, 12^23, 1^223, 123, 23, 3 \}, \\
	\Delta_{+}^{(5)}= & \{ 1, 1^22, 12, 2, 1^32^23, 1^323, 1^223, 123, 23, 3 \}.
	\end{align*}
}%
Now the bundle of sets of (positive) roots is described as follows:
\begin{align*}
a_1 & \mapsto s_{13}(\Delta_+^{(3)}), &
a_2 & \mapsto \Delta_+^{(5)}, &
a_3 & \mapsto s_{23}(\Delta_+^{(2)}), &
a_4 & \mapsto s_{13}(\Delta_+^{(4)}),
\\
a_5 & \mapsto s_{123}(\Delta_+^{(1)}), &
a_6 & \mapsto s_{23}(\Delta_+^{(1)}), &
a_7 & \mapsto \Delta_+^{(4)}, &
a_8 & \mapsto s_{123}(\Delta_+^{(2)}),
\\
a_9 & \mapsto s_{13}(\Delta_+^{(5)}), &
a_{10} & \mapsto \Delta_+^{(3)}.
\end{align*}

\subsubsection{Weyl groupoid}
\label{subsubsec:type-ufo3-Weyl}
The isotropy group  at $a_1 \in \cX$ is
\begin{align*}
\cW(a_4)= \langle \varsigma_2^{a_4}, \varsigma_3^{a_4} \rangle \simeq W(A_2).
\end{align*}

\subsubsection{Incarnation}
We set the matrices $(\bq^{(i)})_{i\in\I_{5}}$, from left to right and  from up to down:
\begin{align}\label{eq:dynkin-ufo(3)}
\begin{aligned}
&\Dchainthree{-1}{\ztu}{\zeta }{-\ztu}{-\zeta }
&& \Dchainthree{-1}{\zeta }{-1}{-\ztu}{-\zeta }
&& \Dchainthree{\zeta }{-1}{-1}{-\ztu}{-\zeta }
\\&\Dchainthree{\zeta }{-\ztu}{-\zeta }{-\ztu}{-\zeta }
&&\Dtriangle{\zeta }{-1}{-1}{\ztu}{-\zeta }{-1}
&&
\end{aligned}
\end{align}
Now this is the incarnation:
\begin{align*}
a_1 & \mapsto \Delta_+^{(3)}, &
a_2 & \mapsto \Delta_+^{(5)}, &
a_3 & \mapsto s_{23}(\Delta_+^{(2)}), &
a_4 & \mapsto \Delta_+^{(4)},
\\
a_5 & \mapsto s_{123}(\Delta_+^{(1)}), &
a_6 & \mapsto s_{23}(\Delta_+^{(1)}), &
a_7 & \mapsto s_{13}(\Delta_+^{(4)}), &
a_8 & \mapsto s_{123}(\Delta_+^{(2)}),
\\
a_9 & \mapsto s_{13}(\Delta_+^{(5)}), &
a_{10} & \mapsto s_{13}(\Delta_+^{(3)}).
\end{align*}

\subsubsection{PBW-basis and dimension} \label{subsubsec:type-ufo3-PBW}
Notice that the roots in each $\Delta_{+}^{a_i}$, $i\in\I_{10}$, are ordered from left to right, justifying the notation $\beta_1, \dots, \beta_{10}$.

The root vectors $x_{\beta_k}$ are described as in Remark \ref{rem:lyndon-word}.
Thus
\begin{align*}
\left\{ x_{\beta_{10}}^{n_{10}} \dots x_{\beta_2}^{n_{2}}  x_{\beta_1}^{n_{1}} \, | \, 0\le n_{k}<N_{\beta_k} \right\}.
\end{align*}
is a PBW-basis of $\toba_{\bq}$. Hence $\dim \toba_{\bq}=2^43^36^3=2^73^6$.

\subsubsection{The Dynkin diagram \emph{(\ref{eq:dynkin-ufo(3)}
		a)}}\label{subsubsec:ufo(3)-a}

\

The Nichols algebra $\toba_{\bq}$ is generated by $(x_i)_{i\in \I_3}$ with defining relations
\begin{align}\label{eq:rels-ufo(3)-a}
\begin{aligned}
x_{221}&=0; & x_{13}&=0; & x_1^2&=0; & [x_{(13)},&x_{223}]_c^6=0; \\
x_{332}&=0; & x_2^3&=0; &  x_3^6&=0; &  [[x_{(13)},&x_2]_c,x_2]_c^6=0.
\end{aligned}
\end{align}
Here, $\Oc^{\bq}_+=\{\alpha_3, \alpha_1+3\alpha_2+\alpha_3 \}$ and the degree of the integral is 
\begin{align*}
\ya&= 15\alpha_1 +42\alpha_2 +26\alpha_3.
\end{align*}

\subsubsection{The Dynkin diagram \emph{(\ref{eq:dynkin-ufo(3)}
		b)}}\label{subsubsec:ufo(3)-b}

\

The Nichols algebra $\toba_{\bq}$ is generated by $(x_i)_{i\in \I_3}$ with defining relations
\begin{align}\label{eq:rels-ufo(3)-b}
\begin{aligned}
x_{332}&=0; & x_{13}&=0; &  x_2^2&=0; & [x_{12},&[x_{(13)},x_2]_c]_c^6=0; \\
& & x_1^2&=0; & x_3^6&=0; &  [x_{(13)},&[x_{(13)},x_2]_c]_c^6=0.
\end{aligned}
\end{align}
Here, $\Oc^{\bq}_+=\{2\alpha_1 + 3\alpha_2 +\alpha_3, 2\alpha_1+3\alpha_2+2\alpha_3 \}$ and the degree of the integral is 
\begin{align*}
\ya&= 29\alpha_1 +42\alpha_2 +26\alpha_3.
\end{align*}

\subsubsection{The Dynkin diagram \emph{(\ref{eq:dynkin-ufo(3)}
		c)}}\label{subsubsec:ufo(3)-c}

\

The Nichols algebra $\toba_{\bq}$ is generated by $(x_i)_{i\in \I_3}$ with defining relations
\begin{align}\label{eq:rels-ufo(3)-c}
\begin{aligned}
x_{332}&=0; & & [[x_{12},x_{(13)}]_c,x_2]_c=0; & x_{112}^6&=0; & [x_{112},&x_{12}]_c=0; \\
x_{13}&=0; & & x_1^3=0; \quad x_2^2=0; & x_3^6&=0; &  [x_1,&x_{(13)}]_c^6=0.
\end{aligned}
\end{align}
Here, $\Oc^{\bq}_+=\{2\alpha_1 + \alpha_2, 2\alpha_1+\alpha_2+\alpha_3 \}$ and the degree of the integral is 
\begin{align*}
\ya&= 29\alpha_1 +20\alpha_2 +15\alpha_3.
\end{align*}

\subsubsection{The Dynkin diagram \emph{(\ref{eq:dynkin-ufo(3)}
		d)}}\label{subsubsec:ufo(3)-d}

\
The Nichols algebra $\toba_{\bq}$ is generated by $(x_i)_{i\in \I_3}$ with defining relations
\begin{align}\label{eq:rels-ufo(3)-d}
\begin{aligned}
x_{221}&=0; & x_{13}&=0; & x_1^3&=0; & x_{332}&=0; \\
x_{223}&=0; & x_2^6&=0; &  x_3^6&=0; &  x_{23}^6&=0.
\end{aligned}
\end{align}
Here, $\Oc^{\bq}_+=\{\alpha_2 +\alpha_3, \alpha_3 \}$ and the degree of the integral is 
\begin{align*}
\ya&= 15\alpha_1 +20\alpha_2 +15\alpha_3.
\end{align*}

\subsubsection{The Dynkin diagram \emph{(\ref{eq:dynkin-ufo(3)}
		e)}}\label{subsubsec:ufo(3)-e}

\

The Nichols algebra $\toba_{\bq}$ is generated by $(x_i)_{i\in \I_3}$ with defining relations
\begin{align}\label{eq:rels-ufo(3)-e}
\begin{aligned}
& \begin{aligned}
x_{113}&=0; & x_1^3&=0; & x_{23}^6&=0; &  [x_{112},x_{12}]_c&=0; \\
x_2^2&=0; &  x_3^2&=0; & x_{112}^6&=0; &  [x_{12},x_{(13)}]_c^6&=0;
\end{aligned}
\\
& x_{(13)}=\frac{q_{23}\zeta}{1-\ztu}[x_{13},x_2]_c-q_{12}\ztu x_2x_{13}.
\end{aligned}
\end{align}
Here, $\Oc^{\bq}_+=\{\alpha_2 +\alpha_3, 2\alpha_1+2\alpha_2+\alpha_3 \}$ and the degree of the integral is 
\begin{align*}
\ya&= 29\alpha_1 +26\alpha_2 +15\alpha_3.
\end{align*}

\subsubsection{The associated Lie algebra} This is of type $A_2$.

\subsection{Type $\Ufo(4)$}\label{subsec:type-ufo(4)}
Here $\zeta \in \G'_3$.
We describe first  the root system $\Ufo(4)$.

\subsubsection{Basic datum and root system}
Below, $A_3$, $C_3$, $C_2^{(1)}$ and $T^{(2)}$ are numbered as in \eqref{eq:dynkin-system-A}, \eqref{eq:dynkin-system-C}, \eqref{eq:Cn-(1)} and \eqref{eq:T2}, respectively.
The basic datum and the bundle of Cartan matrices are described by the following diagram:

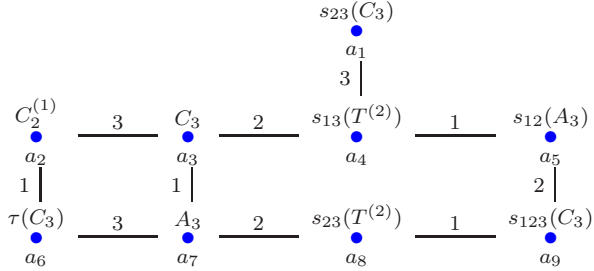
\begin{center}
	\begin{tabular}{c c c c c c c c c c}
		& & & & & $\overset{s_{23}(C_3)}{\underset{a_1}{\vtxgpd}}$ & &  & &
		\\
		& & & & & {\scriptsize 3} \vline\hspace{5pt} & &  & &
		\\
		& $\overset{C_2^{(1)}}{\underset{a_2}{\vtxgpd}}$
		& \hspace{-5pt}\raisebox{3pt}{$\overset{3}{\rule{30pt}{0.5pt}}$}\hspace{-5pt}
		& $\overset{C_3}{\underset{a_3}{\vtxgpd}}$
		& \hspace{-5pt}\raisebox{3pt}{$\overset{2}{\rule{30pt}{0.5pt}}$}\hspace{-5pt}
		& $\overset{s_{13}(T^{(2)})}{\underset{a_4}{\vtxgpd}}$
		& \hspace{-5pt}\raisebox{3pt}{$\overset{1}{\rule{30pt}{0.5pt}}$}\hspace{-5pt}
		& $\overset{s_{12}(A_3)}{\underset{a_5}{\vtxgpd}}$  & &
		\\
		& {\scriptsize 1} \vline\hspace{5pt} & & {\scriptsize 1} \vline\hspace{5pt} & & &
		& {\scriptsize 2} \vline\hspace{5pt}  & &
		\\
		& $\overset{\tau(C_3)}{\underset{a_6}{\vtxgpd}}$
		& \hspace{-5pt}\raisebox{3pt}{$\overset{3}{\rule{30pt}{0.5pt}}$}\hspace{-5pt}
		& $\overset{A_3}{\underset{a_7}{\vtxgpd}}$
		& \hspace{-5pt}\raisebox{3pt}{$\overset{2}{\rule{30pt}{0.5pt}}$}\hspace{-5pt}
		& $\overset{s_{23}(T^{(2)})}{\underset{a_8}{\vtxgpd}}$
		& \hspace{-5pt}\raisebox{3pt}{$\overset{1}{\rule{30pt}{0.5pt}}$}\hspace{-5pt}
		& $\overset{s_{123}(C_3)}{\underset{a_9}{\vtxgpd}}$  & &
	\end{tabular}
\end{center}
Using the notation \eqref{eq:notation-root-exceptional}, the bundle of root sets is the following: { \scriptsize
	\begin{align*}
	\Delta_{+}^{a_1}= & s_{23}(\{ 1,12,123,12^23,123^2,12^23^2,12^23^3,2,23,23^2,3 \}), \\
	\Delta_{+}^{a_2}= & \{ 1,12,123,12^2,12^23,12^33,12^33^2,2,2^23,23,3 \}, \\
	\Delta_{+}^{a_3}= & \{ 1,12,123,12^23,12^33,12^23^2,12^33^2,2,2^23,23,3 \}, \\
	\Delta_{+}^{a_4}= & \{ 1,12,13,123,12^23,12^23^2,12^33^2,2,2^23,23,3 \}, \\
	\Delta_{+}^{a_5}= & s_{12}(\{ 1,12,1^22^23,123,12^23,1^22^33,1^22^33^2,1^32^43^2,2,23,3 \}), \\
	\Delta_{+}^{a_6}= & \{ 1,12,12^2,1^22^23,123,12^23,1^22^33^2,1^22^33,2,23,3 \}, \\
	\Delta_{+}^{a_7}= & \{ 1,12,1^22^23,123,1^22^33,12^23,12^23^2,1^22^33^2,2,23,3 \}, \\
	\Delta_{+}^{a_8}= & \{ 1,12,12^2,1^22^23,13,1^223,123,12^23,2,23,3 \}, \\
	\Delta_{+}^{a_9}= & s_{123}(\{ 1,12,123,12^23,12^33,12^33^2,12^43^2,2,2^23,23,3 \}).
	\end{align*}
}%

\subsubsection{Weyl groupoid}
\label{subsubsec:type-ufo4-Weyl}
The isotropy group  at $a_2 \in \cX$ is
\begin{align*}
\cW(a_2)= \langle \varsigma_2^{a_2}, \varsigma_1^{a_2} \varsigma_2\varsigma_1 \rangle \simeq \Z/2  \times \Z/2.
\end{align*}

\subsubsection{Incarnation}
We set the matrices $(\bq^{(i)})_{i\in\I_{9}}$, from left to right and  from up to down:
\begin{align}\label{eq:dynkin-ufo(4)}
\begin{aligned}
&\Dchainthree{-1}{-1}{-1}{-\ztu}{\ztu}
&&\Dchainthree{-1}{\ztu}{\ztu}{-\ztu}{-1}
&&\Dchainthree{-1}{\ztu}{\zeta }{-\zeta }{-1}
\\
&\Dtriangle{-1}{-1}{\zeta }{-1}{\ztu}{-\zeta }
&&\Dchainthree{-1}{\zeta}{-1}{-1}{-1}
&&\Dchainthree{-1}{\zeta }{-\zeta }{-\ztu}{-1}
\\
&\Dchainthree{-1}{\zeta }{-1}{-\zeta }{-1}
&&\Dtriangle{-\zeta }{\zeta }{-1}{-\ztu}{\ztu}{-\ztu}
&& \Dchainthree{-1}{\ztu}{\zeta }{-1}{-1}
\end{aligned}
\end{align}
Now, this is the incarnation:
\begin{align*}
& a_1\mapsto s_{23}(\bq^{1}); &
& a_i\mapsto s_{13}(\bq^{i}), \ i=4,8; \\
& a_i\mapsto s_{12}(\bq^{i}), \ i=5,9; &
& a_i\mapsto \bq^{(i)}, \ i=2,3,6,7.
\end{align*}

\subsubsection{PBW-basis and dimension} \label{subsubsec:type-ufo4-PBW}
Notice that the roots in each $\Delta_{+}^{a_i}$, $i\in\I_{9}$, are ordered from left to right, justifying the notation $\beta_1, \dots, \beta_{11}$.

The root vectors $x_{\beta_k}$ are described as in Remark \ref{rem:lyndon-word}.
Thus
\begin{align*}
\left\{ x_{\beta_{11}}^{n_{11}} \dots x_{\beta_2}^{n_{2}}  x_{\beta_1}^{n_{1}} \, | \, 0\le n_{k}<N_{\beta_k} \right\}.
\end{align*}
is a PBW-basis of $\toba_{\bq}$. Hence $\dim \toba_{\bq}=2^{7}3^{3}6=2^83^4$.

\subsubsection{The Dynkin diagram \emph{(\ref{eq:dynkin-ufo(4)}
		a)}}\label{subsubsec:ufo(4)-a}

\
The Nichols algebra $\toba_{\bq}$ is generated by $(x_i)_{i\in \I_3}$ with defining relations
\begin{align}\label{eq:rels-ufo(4)-a}
\begin{aligned}
x_1^2&=0; & x_2^2&=0; & [x_{332},x_{32}]_c&=0;  \\
x_3^3&=0; &  [[x_{(13)},x_3]_c,x_2]_c^6&=0; & [x_{3321},x_{321}]_c&=0; \\
x_{13}&=0; &  x_{12}^2&=0; & [[x_{32},x_{321}]_c,x_2]_c&=0.
\end{aligned}
\end{align}
Here, $\Oc^{\bq}_+=\{\alpha_1 + 2\alpha_2 +2\alpha_3 \}$ and the degree of the integral is 
\begin{align*}
\ya&= 12\alpha_1 +22\alpha_2 +24\alpha_3.
\end{align*}

\subsubsection{The Dynkin diagram \emph{(\ref{eq:dynkin-ufo(4)}
		b)}}\label{subsubsec:ufo(4)-b}

\
The Nichols algebra $\toba_{\bq}$ is generated by $(x_i)_{i\in \I_3}$ with defining relations
\begin{align}\label{eq:rels-ufo(4)-b}
\begin{aligned}
\begin{aligned}
x_1^2&=0; & x_2^3&=0; &  x_{12}^6&=0; &
x_3^2&=0; &  x_{13}&=0;
\end{aligned}
\\
\begin{aligned}
& [x_{223}, x_{23}]_c=0; & [x_1,x_{223}]_c &= q_{12} x_2x_{123} -q_{23}[x_{123},x_2]_c.
\end{aligned}
\end{aligned}
\end{align}
Here, $\Oc^{\bq}_+=\{\alpha_1 + \alpha_2\}$ and the degree of the integral is 
\begin{align*}
\ya&= 12\alpha_1 +24\alpha_2 +10\alpha_3.
\end{align*}

\subsubsection{The Dynkin diagram \emph{(\ref{eq:dynkin-ufo(4)}
		c)}}\label{subsubsec:ufo(4)-c}

\
The Nichols algebra $\toba_{\bq}$ is generated by $(x_i)_{i\in \I_3}$ with defining relations
\begin{align}\label{eq:rels-ufo(4)-c}
\begin{aligned}
x_1^2&=0; & x_2^3&=0; & x_3^2&=0; &  x_{(13)}^6&=0; \\
&& x_{221}&=0; & x_{13}&=0; & [x_{223},x_{23}]_c&=0.
\end{aligned}
\end{align}
Here, $\Oc^{\bq}_+=\{\alpha_1 + \alpha_2 +\alpha_3 \}$ and the degree of the integral is 
\begin{align*}
\ya&= 12\alpha_1 +24\alpha_2 +16\alpha_3.
\end{align*}

\subsubsection{The Dynkin diagram \emph{(\ref{eq:dynkin-ufo(4)}
		d)}}\label{subsubsec:ufo(4)-d}

\
The Nichols algebra $\toba_{\bq}$ is generated by $(x_i)_{i\in \I_3}$ with defining relations
\begin{align}\label{eq:rels-ufo(4)-d}
\begin{aligned}
& \begin{aligned}
x_1^2&=0; & x_2^3&=0; & [x_{221},x_{21}]_c&=0;
\\
x_3^2&=0; & x_{13}^2&=0; & [x_{(13)},x_2]_c^6&=0; 
\end{aligned}
\\
& x_{2223}=0; \quad x_{(13)}= \frac{q_{23}(\zeta-1)}{2} [x_{13},x_2]_c +q_{12}(1-\ztu)x_2x_{13}.
\end{aligned}
\end{align}
Here, $\Oc^{\bq}_+=\{\alpha_1 + 2\alpha_2 +\alpha_3 \}$ and the degree of the integral is 
\begin{align*}
\ya&= 16\alpha_1 +24\alpha_2 +12\alpha_3.
\end{align*}

\subsubsection{The Dynkin diagram \emph{(\ref{eq:dynkin-ufo(4)}
		e)}}\label{subsubsec:ufo(4)-e}

\
The Nichols algebra $\toba_{\bq}$ is generated by $(x_i)_{i\in \I_3}$ with defining relations
\begin{align}\label{eq:rels-ufo(4)-e}
\begin{aligned}
x_1^2&=0; & x_2^2&=0; & x_3^2&=0; &  [x_{12},[x_{123},x_{123}]_c]_c&=0;\\
x_{12}^3&=0; & x_{23}^2&=0; & x_{13}&=0; &  [x_{12},x_{(13)}]_c^6&=0.
\end{aligned}
\end{align}
Here, $\Oc^{\bq}_+=\{2\alpha_1 + 2\alpha_2 +\alpha_3 \}$ and the degree of the integral is 
\begin{align*}
\ya&= 24\alpha_1 +30\alpha_2 +16\alpha_3.
\end{align*}

\subsubsection{The Dynkin diagram \emph{(\ref{eq:dynkin-ufo(4)}
		f)}}\label{subsubsec:ufo(4)-f}

\
The Nichols algebra $\toba_{\bq}$ is generated by $(x_i)_{i\in \I_3}$ with defining relations
\begin{align}\label{eq:rels-ufo(4)-f}
\begin{aligned}
x_1^2&=0; & x_2^6&=0; &  x_{2221}&=0; &
x_3^2&=0; &  x_{13}&=0; &  x_{223}&=0.
\end{aligned}
\end{align}
Here, $\Oc^{\bq}_+=\{\alpha_2 \}$ and the degree of the integral is 
\begin{align*}
\ya&= 14\alpha_1 +24\alpha_2 +10\alpha_3.
\end{align*}

\subsubsection{The Dynkin diagram \emph{(\ref{eq:dynkin-ufo(4)}
		g)}}\label{subsubsec:ufo(4)-g}

\
The Nichols algebra $\toba_{\bq}$ is generated by $(x_i)_{i\in \I_3}$ with defining relations
\begin{align}\label{eq:rels-ufo(4)-g}
\begin{aligned}
x_1^2&=0; & x_2^2&=0; & [[x_{32},x_{321}]_c,x_2]_c&=0; \\
x_3^2&=0; &  x_{13}&=0; &  x_{23}^6&=0.
\end{aligned}
\end{align}
Here, $\Oc^{\bq}_+=\{\alpha_2 +\alpha_3 \}$ and the degree of the integral is 
\begin{align*}
\ya&= 14\alpha_1 +24\alpha_2 +16\alpha_3.
\end{align*}

\subsubsection{The Dynkin diagram \emph{(\ref{eq:dynkin-ufo(4)}
		h)}}\label{subsubsec:ufo(4)-h}

\
The Nichols algebra $\toba_{\bq}$ is generated by $(x_i)_{i\in \I_3}$ with defining relations
\begin{align}\label{eq:rels-ufo(4)-h}
\begin{aligned}
& \begin{aligned}
x_1^6&=0; & x_2^2&=0; & x_3^3&=0; & x_{112}&=0; & x_{113}&=0; & x_{332}&=0;
\end{aligned}
\\
& x_{(13)} =q_{23}(\ztu-1)[x_{13},x_2]_c +q_{12}(1-\ztu)x_2x_{13}.
\end{aligned}
\end{align}
Here, $\Oc^{\bq}_+=\{\alpha_1 \}$ and the degree of the integral is 
\begin{align*}
\ya&= 16\alpha_1 +8\alpha_2 +14\alpha_3.
\end{align*}

\subsubsection{The Dynkin diagram \emph{(\ref{eq:dynkin-ufo(4)}
		i)}}\label{subsubsec:ufo(4)-i}

\
The Nichols algebra $\toba_{\bq}$ is generated by $(x_i)_{i\in \I_3}$ with defining relations
\begin{align}\label{eq:rels-ufo(4)-i}
\begin{aligned}
x_1^2&=0; & x_2^3&=0; & x_3^2&=0; & [x_{223}, x_{23}]_c&=0; \\
x_{221}&=0; & x_{13}&=0; & x_{223}^6&=0; & [x_{(13)},x_{23}]_c&=0.
\end{aligned}
\end{align}
Here, $\Oc^{\bq}_+=\{2\alpha_2 +\alpha_3 \}$ and the degree of the integral is 
\begin{align*}
\ya&= 8\alpha_1 +30\alpha_2 +16\alpha_3.
\end{align*}

\subsubsection{The associated Lie algebra} This is of type $A_1$.

\subsection{Type $\Ufo(5)$}\label{subsec:type-ufo(5)}
Here $\zeta \in \G'_3$.
We describe first  the root system $\Ufo(5)$.

\subsubsection{Basic datum and root system}
Below, $A_4$, $B_4$, $C_4$, $T_1$ and $_{1}\widetilde{T}^{(2)}$ are numbered as in \eqref{eq:dynkin-system-A}, \eqref{eq:dynkin-system-B}, \eqref{eq:dynkin-system-C}, \eqref{eq:mTn} and \eqref{eq:1T2t}, respectively.
The basic datum and the bundle of Cartan matrices are described by the following diagram:

\begin{center}
	\begin{tabular}{c c c c c c c }
		$\overset{B_4}{\underset{a_1}{\vtxgpd}}$
		&
		& $\overset{s_{34}(C_4)}{\underset{a_2}{\vtxgpd}}$
		& \hspace{-5pt}\raisebox{3pt}{$\overset{2}{\rule{30pt}{0.5pt}}$}\hspace{-5pt}
		& $\overset{s_{24}(T_1)}{\underset{a_3}{\vtxgpd}}$
		& \hspace{-5pt}\raisebox{3pt}{$\overset{1}{\rule{30pt}{0.5pt}}$}\hspace{-5pt}
		& $\overset{\kappa_3(A_4)}{\underset{a_4}{\vtxgpd}}$
		\\
		{\scriptsize 4} \vline\hspace{5pt}
		& & {\scriptsize 4} \vline\hspace{5pt}
		& & {\scriptsize 4} \vline\hspace{5pt}
		& &
		\\
		$\overset{B_4}{\underset{a_5}{\vtxgpd}}$
		& \hspace{-5pt}\raisebox{3pt}{$\overset{3}{\rule{30pt}{0.5pt}}$}\hspace{-5pt}
		& $\overset{{}_1\widetilde{T}^{(2)}}{\underset{a_6}{\vtxgpd}}$
		&
		& $\overset{s_{24}({}_1\widetilde{T}^{(2)})}{\underset{a_7}{\vtxgpd}}$
		& \hspace{-5pt}\raisebox{3pt}{$\overset{3}{\rule{30pt}{0.5pt}}$}\hspace{-5pt}
		& $\overset{s_{24}(B_4)}{\underset{a_8}{\vtxgpd}}$
		\\
		& &{\scriptsize 2} \vline\hspace{5pt}
		& & {\scriptsize 2} \vline\hspace{5pt}
		& & {\scriptsize 2} \vline\hspace{5pt}
		\\
		$\overset{\kappa_1(A_4)}{\underset{a_9}{\vtxgpd}}$
		& \hspace{-5pt}\raisebox{3pt}{$\overset{1}{\rule{30pt}{0.5pt}}$}\hspace{-5pt}
		& $\overset{T_1}{\underset{a_{10}}{\vtxgpd}}$
		& \hspace{-5pt}\raisebox{3pt}{$\overset{4}{\rule{30pt}{0.5pt}}$}\hspace{-5pt}
		& $\overset{\kappa_2(C_4)}{\underset{a_{11}}{\vtxgpd}}$
		&
		& $\overset{s_{24}(B_4)}{\underset{a_{12}}{\vtxgpd}}$
	\end{tabular}
\end{center}
Using the notation \eqref{eq:notation-root-exceptional}, we set: { \scriptsize
	\begin{align*}
	\Delta_{+}^{(1)}= & \{ 1, 12, 2, 123, 23, 3, 1234, 12^23^24^2, 123^24^2, 1234^2, 1^22^33^44^5, 12^23^34^3, 12^23^24^3, 123^24^3, \\
	& 12^23^34^4, 12^33^44^5, 12^23^44^5, 12^23^34^5, 234, 23^24^2, 234^2, 23^24^3, 34, 34^2, 4 \}, \\
	\Delta_{+}^{(2)}= & \{ 1, 12, 2, 123, 23, 3, 12^23^24, 12^23^24, 123^24, 23^24, 12^23^34^2, 1^22^33^44^3, 12^33^44^3, 12^23^44^3, \\
	& 1234, 12^23^24^2, 123^24^2, 12^23^34^3, 1234^2, 234, 23^24^2, 34, 234^2, 34^2, 4 \}, \\
	\Delta_{+}^{(3)}= & \{ 1, 12, 12^2, 2, 1^22^33, 12^33, 12^23, 123, 23, 3, 1^22^43^34, 1^22^43^24, 1^22^33^24, 12^33^24, 12^23^24, \\
	& 123^24, 23^24, 1^22^43^34^2, 1^22^334, 12^334, 12^234, 1234, 234, 34, 4 \}, \\
	\Delta_{+}^{(4)}= & \{ 1, 12, 2, 123, 23, 3, 12^23^24, 1^22^43^54^2, 1^22^33^54^2, 12^33^54^2, 12^23^24, 1^22^33^44^2, 12^33^44^2, \\
	& 1^22^43^64^3, 123^24, 12^23^44^2, 23^24, 1^22^43^54^3, 1^22^33^54^3, 12^33^54^3, 12^23^34^2, 1234, 234, 34, 4 \}, \\
	\Delta_{+}^{(5)}= & \{ 1, 12, 2, 12^23, 123, 23, 3, 12^23^24, 12^23^24, 123^24, 23^24, 3^24, 1^22^33^34^2, 12^33^34^2, 12^23^34^2, \\
	& 12^234, 1234, 12^23^24^2, 234, 123^24^2, 124, 23^24^2, 34, 24, 4 \}, \\
	\Delta_{+}^{(6)}= & \{ 1, 12, 2, 12^23, 1^22^33^2, 12^33^2, 123, 12^23^2, 23, 3, 1^22^43^34, 1^22^43^24, 1^22^33^34, 12^33^34, \\
	& 1^22^33^24, 12^33^24, 12^23^24, 1^22^43^34^2, 12^234, 1234, 124, 234, 34, 24, 4 \}.
	\end{align*}
}
Now the bundle of sets of (positive) roots is described as follows:
\begin{align*}
a_1 & \mapsto \Delta_+^{(1)}, &
a_2 & \mapsto \tau(\Delta_+^{(3)}), &
a_3 & \mapsto \kappa_6(\Delta_+^{(6)}), &
a_4 & \mapsto \kappa_4(\Delta_+^{(4)}),
\\
a_5 & \mapsto \Delta_+^{(2)}, &
a_6 & \mapsto s_{34}(\Delta_+^{(5)}), &
a_7 & \mapsto \kappa_2(\Delta_+^{(5)}), &
a_8 & \mapsto s_{24}(\Delta_+^{(2)}),
\\
a_9 & \mapsto s_{13}(\Delta_+^{(4)}), &
a_{10} & \mapsto \kappa_7(\Delta_+^{(6)}), &
a_{11} & \mapsto \kappa_7(\Delta_+^{(3)}), &
a_{12} & \mapsto s_{24}(\Delta_+^{(1)}).
\end{align*}

\subsubsection{Weyl groupoid}
\label{subsubsec:type-ufo5-Weyl}
The isotropy group at $a_1 \in \cX$ is
\begin{align*}
\cW(a_1)= \langle \varsigma_1^{a_1}, \varsigma_2^{a_1},  \varsigma_3^{a_1}, \varsigma_4^{a_1} \varsigma_3\varsigma_2 \varsigma_4 \varsigma_1 \varsigma_4\varsigma_2 \varsigma_3 \varsigma_4 \rangle \simeq W(A_4).
\end{align*}

\subsubsection{Incarnation}
We set the matrices $(\bq^{(i)})_{i\in\I_{6}}$, from left to right and  from up to down:
\begin{align}\label{eq:dynkin-ufo(5)}
\begin{aligned}
&\Dchainfour{-\zeta }{-\ztu}{-\zeta }{-\ztu}{-\zeta }{-\ztu}{\zeta } &&
\Dchainfour{-\zeta }{-\ztu}{-\zeta }{-\ztu}{-1}{-1}{\zeta }
\\
&\Dchainfour{-\zeta }{-\ztu}{\zeta }{\ztu}{-1}{-\ztu}{-\zeta }
&&\Dchainfour{-\zeta }{-\ztu}{-\zeta }{-\ztu}{-1}{-\ztu}{-\zeta }.
\\
& \Drightofway{-\zeta }{-\ztu}{-1}{-\zeta }{\ztu}{-1}{-1}{\zeta }
&&\Drightofway{-\zeta }{-\ztu}{-1}{-\zeta }{\zeta }{-1}{-\zeta }{-1}
\end{aligned}
\end{align}
Now this is the incarnation:
\begin{align*}
a_1 & \mapsto \bq^{(1)}, &
a_2 & \mapsto \tau(\bq^{(3)}), &
a_3 & \mapsto \kappa_6(\bq^{(6)}), &
a_4 & \mapsto \kappa_4(\bq^{(4)}),
\\
a_5 & \mapsto \bq^{(2)}, &
a_6 & \mapsto s_{34}(\bq^{(5)}), &
a_7 & \mapsto \kappa_2(\bq^{(5)}), &
a_8 & \mapsto s_{24}(\bq^{(2)}),
\\
a_9 & \mapsto s_{13}(\bq^{(4)}), &
a_{10} & \mapsto \kappa_7(\bq^{(6)}), &
a_{11} & \mapsto \kappa_7(\bq^{(3)}), &
a_{12} & \mapsto s_{24}(\bq^{(1)}).
\end{align*}

\subsubsection{PBW-basis and dimension} \label{subsubsec:type-ufo5-PBW}
Notice that the roots in each $\Delta_{+}^{a_i}$, $i\in\I_{12}$, are ordered from left to right, justifying the notation $\beta_1, \dots, \beta_{25}$.

The root vectors $x_{\beta_k}$ are described as in Remark \ref{rem:lyndon-word}.
Thus
\begin{align*}
\left\{ x_{\beta_{25}}^{n_{25}} \dots x_{\beta_2}^{n_{2}}  x_{\beta_1}^{n_{1}} \, | \, 0\le n_{k}<N_{\beta_k} \right\}.
\end{align*}
is a PBW-basis of $\toba_{\bq}$. Hence $\dim \toba_{\bq}=2^{10}3^56^{10}=2^{20}3^{15}$.

\subsubsection{The Dynkin diagram \emph{(\ref{eq:dynkin-ufo(5)}
		a)}}\label{subsubsec:ufo(5)-a}

\

The Nichols algebra $\toba_{\bq}$ is generated by $(x_i)_{i\in \I_4}$ with defining relations
\begin{align}\label{eq:rels-ufo(5)-a}
\begin{aligned}
x_{112}&=0; & x_{221}&=0; & x_{223}&=0; & x_{ij}&=0, \ i<j, \, \widetilde{q}_{ij}=1;\\
x_{332}&=0; &  x_{334}&=0; & x_{4}^3&=0; & x_{\alpha}^{6}&=0, \ \alpha\in\Oc^{\bq}_+.
\end{aligned}
\end{align}
Here {\scriptsize$\Oc^{\bq}_+=\{ 1, 12, 2, 123, 23, 3, 1^22^33^44^5, 12^33^44^5, 12^23^44^5, 12^23^34^5 \}$}
and the degree of the integral is
\begin{align*}
\ya &= 50 \alpha_1+ 90\alpha_2 +120\alpha_3 +140\alpha_4.
\end{align*}


\subsubsection{The Dynkin diagram \emph{(\ref{eq:dynkin-ufo(5)}
		b)}}\label{subsubsec:ufo(5)-b}

\

The Nichols algebra $\toba_{\bq}$ is generated by $(x_i)_{i\in \I_4}$ with defining relations
\begin{align}\label{eq:rels-ufo(5)-b}
\begin{aligned}
x_{112}&=0; & x_{221}&=0; & x_{ij}&=0, \ i<j, \, \widetilde{q}_{ij}=1;\\
x_{223}&=0; & [x_{443},&x_{43}]_c=0;  & [[x_{43},&x_{432}]_c,x_3]_c=0;\\
x_3^2&=0; & x_{4}^3&=0; & x_{\alpha}^{6}&=0, \ \alpha\in\Oc^{\bq}_+.
\end{aligned}
\end{align}
Here {\scriptsize$\Oc^{\bq}_+=\{ 1, 12, 2, 12^23^34, 1^22^33^44^3, 12^33^44^3, 12^23^44^3, 1234^2, 234^2, 34^2 \}$}
and the degree of the integral is
\begin{align*}
\ya &= 50 \alpha_1+ 90\alpha_2 +120\alpha_3 +104\alpha_4.
\end{align*}


\subsubsection{The Dynkin diagram \emph{(\ref{eq:dynkin-ufo(5)}
		c)}}\label{subsubsec:ufo(5)-c}

\

The Nichols algebra $\toba_{\bq}$ is generated by $(x_i)_{i\in \I_4}$ with defining relations
\begin{align}\label{eq:rels-ufo(5)-c}
\begin{aligned}
x_{112}&=0; & x_{443}&=0; & x_{ij}&=0, \ i<j, \, \widetilde{q}_{ij}=1;\\
x_2^3&=0; & x_{3}^2&=0; & x_{\alpha}^{6}&=0, \ \alpha\in\Oc^{\bq}_+.
\end{aligned}
\end{align}
Here {\scriptsize$\Oc^{\bq}_+=\{ 1, 1^22^33, 12^33, 1^22^43^34, 123^24, 23^24, 1^22^43^34^2, 1^22^334, 12^334, 4 \}$}
and the degree of the integral is
\begin{align*}
\ya &= 76\alpha_1+ 142\alpha_2 +90\alpha_3 +50\alpha_4.
\end{align*}


\subsubsection{The Dynkin diagram \emph{(\ref{eq:dynkin-ufo(5)}
		d)}}\label{subsubsec:ufo(5)-d}

\

The Nichols algebra $\toba_{\bq}$ is generated by $(x_i)_{i\in \I_4}$ with defining relations
\begin{align}\label{eq:rels-ufo(5)-d}
\begin{aligned}
x_{112}&=0; & x_{221}&=0; & x_{ij}&=0,  & & i<j, \, \widetilde{q}_{ij}=1;\\
x_{223}&=0;  & x_{443}&=0; & x_{3}^2&=0; & & x_{\alpha}^{6}=0, \ \alpha\in\Oc^{\bq}_+.
\end{aligned}
\end{align}
Here {\scriptsize$\Oc^{\bq}_+=\{ 1, 12, 2, 1^22^43^54^2, 1^22^33^54^2, 12^33^54^2, 1^22^43^54^3, 1^22^33^54^3, 12^33^54^3, 4 \}$}
and the degree of the integral is
\begin{align*}
\ya &= 76 \alpha_1+ 142\alpha_2 +198\alpha_3 +104\alpha_4.
\end{align*}


\subsubsection{The Dynkin diagram \emph{(\ref{eq:dynkin-ufo(5)}
		e)}}\label{subsubsec:ufo(5)-e}

\

The Nichols algebra $\toba_{\bq}$ is generated by $(x_i)_{i\in \I_4}$ with defining relations
\begin{align}\label{eq:rels-ufo(5)-e}
\begin{aligned}
\begin{aligned}
x_{112}&=0; & [x_{124},&x_2]_c=0; & x_{2}^2&=0; & x_{ij}&=0, \ i<j, \, \widetilde{q}_{ij}=1;
\\
x_{332}&=0; & [x_{334},&x_{34}]_c=0; & x_{4}^2&=0; & x_{\alpha}^{6}&=0, \ \alpha\in\Oc^{\bq}_+;
\end{aligned}
\\
\begin{aligned}
x_3^3&=0; & x_{(24)} &=2q_{34}\zeta[x_{24},x_3]_c +2q_{23}x_3x_{24}.
\end{aligned}
\end{aligned}
\end{align}
Here {\scriptsize$\Oc^{\bq}_+=\{ 1, 12^23, 12^23^34, 3^24, 1^22^33^34^2, 12^33^34^2, 123^24^2, 124, 23^24^2, 24 \}$}
and the degree of the integral is
\begin{align*}
\ya &= 50 \alpha_1+ 90\alpha_2 +104\alpha_3 +76\alpha_4.
\end{align*}


\subsubsection{The Dynkin diagram \emph{(\ref{eq:dynkin-ufo(5)}
		f)}}\label{subsubsec:ufo(5)-f}

\

The Nichols algebra $\toba_{\bq}$ is generated by $(x_i)_{i\in \I_4}$ with defining relations
\begin{align}\label{eq:rels-ufo(5)-f}
\begin{aligned}
& \begin{aligned}
& x_{112}=0; & x_{2}^2&=0; & x_{ij}&=0, && i<j, \, \widetilde{q}_{ij}=1;\\
& [x_{124},x_2]_c=0; & x_3^2&=0; & x_{4}^2&=0; & & x_{\alpha}^{6}=0, \ \alpha\in\Oc^{\bq}_+;
\end{aligned}
\\
& x_{(24)}=-q_{34}\ztu [x_{24},x_3]_c-q_{23}\ztu x_3x_{24}.
\end{aligned}
\end{align}
Here {\scriptsize$\Oc^{\bq}_+=\{ 1, 1^22^33^2, 12^33^2, 1^22^43^24, 1^22^33^34, 12^33^34, 1^22^43^34^2, 124, 34, 24 \}$}
and the degree of the integral is
\begin{align*}
\ya &= 76 \alpha_1+ 142\alpha_2 +104\alpha_3 +50\alpha_4.
\end{align*}


\subsubsection{The associated Lie algebra} This is of type $A_4$.

\subsection{Type $\Ufo(6)$}\label{subsec:type-ufo(6)}
Here $\zeta \in \G'_4$.
We describe first  the root system $\Ufo(6)$.

\subsubsection{Basic datum and root system}
Below, $A_4$, $F_4$, $D_{4}^{(3)\wedge}$, $T_1$ and $_{1}T^{(2)}$ are numbered as in \eqref{eq:dynkin-system-A}, \eqref{eq:dynkin-system-F}, \eqref{eq:D4(3)w}, \eqref{eq:mTn} and \eqref{eq:1T2}, respectively.
The basic datum and the bundle of Cartan matrices are described by the following diagram:

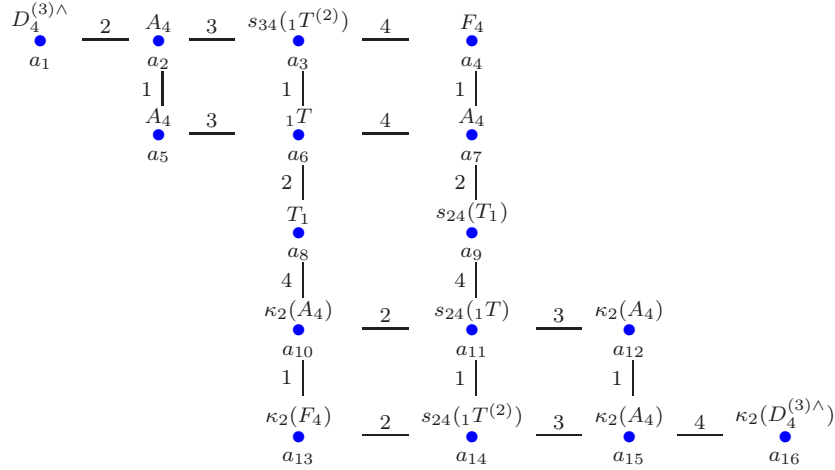
\begin{center}
	\begin{tabular}{c c c c c c c c c c c c}
		& $\overset{D_{4}^{(3)\wedge}}{\underset{a_1}{\vtxgpd}}$
		& \hspace{-5pt}\raisebox{3pt}{$\overset{2}{\rule{17pt}{0.5pt}}$}\hspace{-5pt}
		& $\overset{A_4}{\underset{a_2}{\vtxgpd}}$
		& \hspace{-5pt}\raisebox{3pt}{$\overset{3}{\rule{17pt}{0.5pt}}$}\hspace{-5pt}
		& $\overset{s_{34}({}_1T^{(2)})}{\underset{a_3}{\vtxgpd}}$
		& \hspace{-5pt}\raisebox{3pt}{$\overset{4}{\rule{17pt}{0.5pt}}$}\hspace{-5pt}
		& $\overset{F_4}{\underset{a_4}{\vtxgpd}}$ & & & &
		\\
		& & & {\scriptsize 1} \vline\hspace{5pt} & & {\scriptsize 1} \vline\hspace{5pt} & & {\scriptsize 1} \vline\hspace{5pt} & &  & &
		\\
		& &
		& $\overset{A_4}{\underset{a_5}{\vtxgpd}}$
		& \hspace{-5pt}\raisebox{3pt}{$\overset{3}{\rule{17pt}{0.5pt}}$}\hspace{-5pt}
		& $\overset{{}_1T}{\underset{a_6}{\vtxgpd}}$
		& \hspace{-5pt}\raisebox{3pt}{$\overset{4}{\rule{17pt}{0.5pt}}$}\hspace{-5pt}
		& $\overset{A_4}{\underset{a_7}{\vtxgpd}}$
		& & & &
		\\
		& & & & & {\scriptsize 2} \vline\hspace{5pt} & & {\scriptsize 2} \vline\hspace{5pt} & & & &
		\\
		&
		&
		&
		&
		& $\overset{T_1}{\underset{a_8}{\vtxgpd}}$
		&
		& $\overset{s_{24}(T_1)}{\underset{a_9}{\vtxgpd}}$
		&
		& & &
		\\
		& & & & & {\scriptsize 4} \vline\hspace{5pt} & & {\scriptsize 4} \vline\hspace{5pt} & &  & &
		\\
		&
		&
		&
		&
		& $\overset{\kappa_2(A_4)}{\underset{a_{10}}{\vtxgpd}}$
		& \hspace{-5pt}\raisebox{3pt}{$\overset{2}{\rule{17pt}{0.5pt}}$}\hspace{-5pt}
		& $\overset{s_{24}({}_1T)}{\underset{a_{11}}{\vtxgpd}}$
		& \hspace{-5pt}\raisebox{3pt}{$\overset{3}{\rule{17pt}{0.5pt}}$}\hspace{-5pt}
		& $\overset{\kappa_2(A_4)}{\underset{a_{12}}{\vtxgpd}}$ & &
		\\
		& & & & & {\scriptsize 1} \vline\hspace{5pt} & & {\scriptsize 1} \vline\hspace{5pt} & & {\scriptsize 1} \vline\hspace{5pt} & &
		\\
		&
		&
		&
		&
		& $\overset{\kappa_2(F_4)}{\underset{a_{13}}{\vtxgpd}}$
		& \hspace{-5pt}\raisebox{3pt}{$\overset{2}{\rule{17pt}{0.5pt}}$}\hspace{-5pt}
		& $\overset{s_{24}({}_1T^{(2)})}{\underset{a_{14}}{\vtxgpd}}$
		& \hspace{-5pt}\raisebox{3pt}{$\overset{3}{\rule{17pt}{0.5pt}}$}\hspace{-5pt}
		& $\overset{\kappa_2(A_4)}{\underset{a_{15}}{\vtxgpd}}$
		& \hspace{-5pt}\raisebox{3pt}{$\overset{4}{\rule{17pt}{0.5pt}}$}\hspace{-5pt}
		& $\overset{\kappa_2(D_{4}^{(3)\wedge})}{\underset{a_{16}}{\vtxgpd}}$
	\end{tabular}
\end{center}
Using the notation \eqref{eq:notation-root-exceptional}, we set: { \scriptsize
	\begin{align*}
	\Delta_{+}^{(1)}= & \big\{ 1, 12, 2, 123, 23, 3, 12^23^34, 12^23^24, 123^34, 23^34, 12^23^54^2, 3^34, 123^24, \\
	& 1^22^33^64^3, 12^23^44^2, 12^33^64^3, 23^24, 12^23^64^3, 12^23^34^2, 123^44^2, 23^44^2, \\
	& 12^23^54^3, 3^24, 123^34^2, 1234, 23^34^2, 234, 3^34^2, 34, 4 \big\}, \\
	\Delta_{+}^{(2)}= & \big\{ 1, 12, 2, 123, 23, 3, 1^32^33^34, 1^22^33^34, 1^22^23^34, 1^22^23^24, 1^32^43^54^2, \\
	& 12^23^34, 1^32^43^44^2, 1^32^33^34^2, 1^42^53^64^3, 1^32^33^34^2, 12^23^24, 1^32^53^64^3, 1^22^33^44^2, \\
	& 1^32^43^64^3, 123^24, 1^32^43^54^3, 23^24, 1^22^33^34^2, 1^22^23^34^2, 12^23^34^2, 1234, 234, 34, 4 \big\}, \\
	\Delta_{+}^{(3)}= & \big\{ 1, 12, 2, 123, 23, 3, 1^32^334, 1^22^334, 1^22^234, 1^22^24, 1^32^43^24^2, 1^32^434^2, \\
	& 12^234, 1^32^33^24^2, 1^42^53^24^3, 1^32^334^2, 1^22^33^24^2, 1^32^53^24^3, 12^24, 1^22^334^2,  \\
	& 1^32^43^24^3, 1^32^434^3, 1234, 234, 1^22^234^2, 12^234^2, 124, 24, 34, 4 \big\}, \\
	\Delta_{+}^{(4)}= & \big\{ 1, 12, 2, 1^22^23, 12^23, 123, 23, 3, 1^32^43^34, 1^32^43^24, 1^32^33^34, 1^22^33^34, \\
	& 1^32^33^24, 1^22^33^24, 1^42^53^44^2, 1^22^23^24, 1^22^234, 1^32^53^44^2, 1^32^43^44^2, 1^32^43^34^2, \\
	& 12^23^24, 12^234, 1^32^33^34^2, 1^22^33^34^2, 123^24, 1234, 23^24, 234, 34, 4 \big\}, \\
	\Delta_{+}^{(5)}= & \big\{ 1, 12, 2, 123, 23, 3, 12^33^34, 12^23^34, 12^23^24, 2^33^34, 12^43^54^2, 2^23^34, 12^43^44^2, \\
	& 1^22^53^64^3, 123^24, 12^33^44^2, 1234, 12^53^64^3, 12^43^64^3, 12^33^34^2, 2^23^24, 12^43^54^, \\
	& 12^23^34^2, 2^33^44^2, 23^24, 2^33^34^2, 2^23^34^2, 234, 34, 4 \big\}, \\
	\Delta_{+}^{(6)}= & \big\{ 1, 12, 2, 123, 23, 3, 12^334, 12^234, 12^24, 2^334, 12^43^24^2, 12^434^2, 2^234,  \\
	& 1^22^53^24^3, 12^33^24^2, 1234, 12^53^24^3, 12^334^2, 2^24, 2^33^24^2, 12^43^24^3, 2^334^2, \\
	& 234, 12^434^3, 12^234^2, 124, 2^234^2, 24, 34, 4 \big\}, \\
	\Delta_{+}^{(7)}= & \big\{ 1, 12, 2, 12^23, 123, 2^23, 23, 3, 12^43^34, 12^43^24, 12^33^34, 2^33^34, 12^33^24, \\
	& 1^22^53^44^2, 2^33^24, 12^53^44^2, 12^23^24, 12^234, 12^43^44^2, 12^43^34^2, 2^23^24, 2^234, \\
	& 12^33^34^2, 123^24, 1234, 2^33^34^2, 23^24, 234, 34, 4 \big\}, \\
	\Delta_{+}^{(8)}= & \big\{ 1, 12, 2, 123, 23, 3, 12^23^34, 12^23^24, 123^24, 23^24, 1^22^43^34^2, 1^22^33^34^2,  \\
	& 12^33^34^2, 12^234, 1^22^43^44^3, 1^22^33^24^2, 12^33^24^2, 12^23^34^2, 1^22^43^34^3, 1234,  \\
	& 1^22^33^34^3, 12^23^24^2, 12^33^34^3, 234, 12^24, 12^234^2, 124, 24, 34, 4 \big\}.
	\end{align*}
}
Now the bundle of sets of (positive) roots is described as follows:
\begin{align*}
a_8 & \mapsto \kappa_7(\Delta_+^{(8)}); & a_i & \mapsto \Delta_+^{(i)}, \ i=1,2,5; &
a_i & \mapsto s_{34}(\Delta_+^{(i)}), \ i=3,4,6,7;
\\
a_{9} & \mapsto \kappa_6(\Delta_+^{(8)}); & 
a_{i} & \mapsto \kappa_8(\Delta_+^{(17-i)}), \ i=11,14; &
a_{i} & \mapsto \kappa_2(\Delta_+^{(17-i)}), \ \text{otherwise}.
\end{align*}

\subsubsection{Weyl groupoid}
\label{subsubsec:type-ufo6-Weyl}
The isotropy group  at $a_1 \in \cX$ is
\begin{align*}
\cW(a_1)= \langle \varsigma_1^{a_1}, \varsigma_2^{a_1} \varsigma_3 \varsigma_4\varsigma_2 \varsigma_4 \varsigma_3 \varsigma_2, \varsigma_3^{a_1}, \varsigma_4^{a_1} \rangle \simeq W(G_2) \times W(G_2).
\end{align*}

\subsubsection{Incarnation}
We set the matrices $(\bq^{(i)})_{i\in\I_{8}}$, from left to right and  from up to down:
\begin{align}\label{eq:dynkin-ufo(6)}
\begin{aligned}
&\Dchainfour{\ztu }{\zeta }{-1}{\ztu }{\zeta }{\zeta }{\ztu }&&
\Dchainfour{-1}{\ztu }{-1}{\zeta }{-1}{\zeta }{\ztu }\\
&\Drightofway{-1}{\ztu }{\zeta }{\ztu }{-1}{-1}{\ztu }{-1} &&
\Dchainfour{-1}{\ztu }{\zeta }{-1}{-1}{\zeta }{\ztu } \\
&\Dchainfour{-1}{\zeta }{\ztu }{\zeta }{-1}{\zeta }{\ztu }&&
\Drightofway{-1}{\zeta }{-1}{\ztu }{-1}{-1}{\ztu }{-1}\\
&\Dchainfour{-1}{\zeta }{-1}{-1}{-1}{\zeta }{\ztu }&&
\Drightofway{\ztu }{\zeta }{-1}{\ztu }{-1}{\zeta }{\ztu }{-1}
\end{aligned}
\end{align}
Now this is the incarnation:
\begin{align*}
a_8 & \mapsto \kappa_7(\bq^{(8)}); & a_i & \mapsto \bq^{(i)}, \ i=1,2,5; &
a_i & \mapsto s_{34}(\bq^{(i)}), \ i=3,4,6,7;
\\
a_{9} & \mapsto \kappa_6(\bq^{(8)}); & 
a_{i} & \mapsto \kappa_8(\bq^{(17-i)}), \ i=11,14; &
a_{i} & \mapsto \kappa_2(\bq^{(17-i)}), \ \text{otherwise}.
\end{align*}

\subsubsection{PBW-basis and dimension} \label{subsubsec:type-ufo6-PBW}
Notice that the roots in each $\Delta_{+}^{a_i}$, $i\in\I_{16}$, are ordered from left to right, justifying the notation $\beta_1, \dots, \beta_{30}$.

The root vectors $x_{\beta_k}$ are described as in Remark \ref{rem:lyndon-word}.
Thus
\begin{align*}
\left\{ x_{\beta_{30}}^{n_{30}} \dots x_{\beta_2}^{n_{2}}  x_{\beta_1}^{n_{1}} \, | \, 0\le n_{k}<N_{\beta_k} \right\}.
\end{align*}
is a PBW-basis of $\toba_{\bq}$. Hence $\dim \toba_{\bq}=2^{18}4^{12}=2^{42}$.

\subsubsection{The Dynkin diagram \emph{(\ref{eq:dynkin-ufo(6)}
		a)}}\label{subsubsec:ufo(6)-a}

\

The Nichols algebra $\toba_{\bq}$ is generated by $(x_i)_{i\in \I_4}$ with defining relations
\begin{align}\label{eq:rels-ufo(6)-a}
\begin{aligned}
x_{112}&=0; & x_{332}&=0; & x_{ij}&=0, \ i<j, \, \widetilde{q}_{ij}=1;\\
& & [x_{(13)},&x_2]_c=0;  & [[[x_{(24)},&x_3]_c,x_3]_c,x_3]_c=0;\\
x_{443}&=0; & x_{2}^2&=0; & x_{\alpha}^{4}&=0, \ \alpha\in\Oc^{\bq}_+.
\end{aligned}
\end{align}
Here {\scriptsize$\Oc^{\bq}_+=\{ 1, 3, 3^34, 123^24, 1^22^33^64^3, 12^23^44^2, 12^33^64^3, 23^24, 3^24, 3^34^2, 34, 4 \}$}
and the degree of the integral is
\begin{align*}
\ya &= 30 \alpha_1+ 54\alpha_2 +138\alpha_3 +72\alpha_4.
\end{align*}


\subsubsection{The Dynkin diagram \emph{(\ref{eq:dynkin-ufo(6)} b)}} \label{subsubsec:ufo(6)-b}
\

The Nichols algebra $\toba_{\bq}$ is generated by $(x_i)_{i\in \I_4}$ with defining relations
\begin{align}\label{eq:rels-ufo(6)-b}
\begin{aligned}
x_{443}&=0; & x_1^2&=0; & x_{ij}&=0, \ i<j, \, \widetilde{q}_{ij}=1;\\
& & [x_{(13)},&x_2]_c=0;  & [[x_{23},&[x_{23},x_{(24)}]_c]_c,x_3]_c=0;\\
x_2^2&=0; & x_3^2&=0; & x_{\alpha}^{4}&=0, \ \alpha\in\Oc^{\bq}_+.
\end{aligned}
\end{align}
Here {\scriptsize$\Oc^{\bq}_+=\{ 12, 23, 12^23^24, 2^33^34, 1^22^53^64^3, 12^33^44^2, 12^43^64^3, 2^23^24, 23^24, 2^33^34^2, 234, 4 \}$}
and the degree of the integral is
\begin{align*}
\ya &= 30 \alpha_1+ 116\alpha_2 +138\alpha_3 +72\alpha_4.
\end{align*}


\subsubsection{The Dynkin diagram \emph{(\ref{eq:dynkin-ufo(6)}
		c)}}\label{subsubsec:ufo(6)-c}

\

The Nichols algebra $\toba_{\bq}$ is generated by $(x_i)_{i\in \I_4}$ with defining relations
\begin{align}\label{eq:rels-ufo(6)-c}
\begin{aligned}
& \begin{aligned}
x_{221}&=0; & x_{2223}&=0; & x_1^2&=0; & x_{ij}&=0, \ i<j, \, \widetilde{q}_{ij}=1;
\\
x_{224}&=0; & x_3^2&=0; & x_4^2&=0; & x_{\alpha}^{4}&=0, \ \alpha\in\Oc^{\bq}_+;
\end{aligned}
\\
& x_{(24)}=q_{34}\zeta [x_{24},x_3]_c +q_{23}(1+\zeta)x_3x_{24}.
\end{aligned}
\end{align}
Here {\scriptsize$\Oc^{\bq}_+=\{ 2, 123, 12^234, 2^334, 2^234, 1^22^53^24^3, 12^334^2, 2^33^24^2, 234, 12^434^3, 24, 34 \}$}
and the degree of the integral is
\begin{align*}
\ya &= 30 \alpha_1+ 116\alpha_2 +72\alpha_3 +52\alpha_4.
\end{align*}


\subsubsection{The Dynkin diagram \emph{(\ref{eq:dynkin-ufo(6)}
		d)}}\label{subsubsec:ufo(6)-d}

\

The Nichols algebra $\toba_{\bq}$ is generated by $(x_i)_{i\in \I_4}$ with defining relations
\begin{align}\label{eq:rels-ufo(6)-d}
\begin{aligned}
& \begin{aligned}
x_{221}&=0; & x_{443}&=0;  & x_{2223}&=0; & x_{ij}&=0, \ i<j, \, \widetilde{q}_{ij}=1;\\
& &  x_1^2&=0; & x_3^2&=0; & x_{\alpha}^{4}&=0, \ \alpha\in\Oc^{\bq}_+;
\end{aligned}
\\
& [x_2, [x_{(24)},x_3]_c]_c = \frac{q_{23}q_{43}}{1+\zeta}[x_{23},x_{(24)}]_c+(\zeta-1)q_{23}q_{24} x_{(24)}x_{23}.
\end{aligned}
\end{align}
Here {\scriptsize$\Oc^{\bq}_+=\{ 2, 23, 12^43^24, 2^33^34, 12^33^24, 1^22^53^44^2, 12^23^24, 2^23^24, 123^24, 2^33^34^2, 234, 4 \}$}
and the degree of the integral is
\begin{align*}
\ya &= 30 \alpha_1+ 116\alpha_2 +98\alpha_3 +52\alpha_4.
\end{align*}


\subsubsection{The Dynkin diagram \emph{(\ref{eq:dynkin-ufo(6)}
		e)}}\label{subsubsec:ufo(6)-e}
\

The Nichols algebra $\toba_{\bq}$ is generated by $(x_i)_{i\in \I_4}$ with defining relations
\begin{align}\label{eq:rels-ufo(6)-e}
\begin{aligned}
x_{221}&=0; & x_{223}&=0; & x_{ij}&=0, && i<j, \, \widetilde{q}_{ij}=1;\\
x_{443}&=0;  & x_1^2&=0; & x_3^2&=0; & & x_{\alpha}^{4}=0, \ \alpha\in\Oc^{\bq}_+.
\end{aligned}
\end{align}
Here {\scriptsize$\Oc^{\bq}_+=\{ 2, 123, 1^32^33^34, 1^22^23^24, 1^32^33^34^2, 12^23^24, 1^32^53^64^3, 1^22^33^44^2, 1^32^43^64^3, 123^24, 1234, 4 \}$}
and the degree of the integral is
\begin{align*}
\ya &= 88 \alpha_1+ 116\alpha_2 +138\alpha_3 +72\alpha_4.
\end{align*}


\subsubsection{The Dynkin diagram \emph{(\ref{eq:dynkin-ufo(6)}
		f)}}\label{subsubsec:ufo(6)-f}

\

The Nichols algebra $\toba_{\bq}$ is generated by $(x_i)_{i\in \I_4}$ with defining relations
\begin{align}\label{eq:rels-ufo(6)-f}
\begin{aligned}
& \begin{aligned}
x_1^2&=0; & [x_{124},&x_2]_c=0; & x_2^2&=0; & x_{ij}&=0, \ i<j, \, \widetilde{q}_{ij}=1;\\
x_3^2&=0; & [[x_{12},&x_{(13)}]_c,x_2]_c=0; &
x_4^2&=0; & x_{\alpha}^{4}&=0, \ \alpha\in\Oc^{\bq}_+;
\end{aligned}
\\
& x_{23}^2=0; \quad x_{(24)}=q_{34}\zeta [x_{24},x_3]_c +q_{23}(1+\zeta)x_3x_{24}.
\end{aligned}
\end{align}
Here {\scriptsize$\Oc^{\bq}_+=\{ 12, 23, 1^32^334, 1^22^234, 12^234, 1^32^33^24^2, 1^32^53^24^3, 1^22^334^2, 1^32^434^3, 1234, 124, 34 \}$}
and the degree of the integral is
\begin{align*}
\ya &= 88 \alpha_1+ 116\alpha_2 +72\alpha_3 +52\alpha_4.
\end{align*}


\subsubsection{The Dynkin diagram \emph{(\ref{eq:dynkin-ufo(6)}
		g)}}\label{subsubsec:ufo(6)-g}

\

The Nichols algebra $\toba_{\bq}$ is generated by $(x_i)_{i\in \I_4}$ with defining relations
\begin{align}\label{eq:rels-ufo(6)-g}
\begin{aligned}
x_{443}&=0; & x_1^2&=0; & x_2^2&=0;  & x_{ij}&=0, \ i<j, \, \widetilde{q}_{ij}=1;\\
[[x_{12},x_{(13)}]_c, x_2]_c &=0; & x_3^2&=0; & x_{23}^2&=0; & x_{\alpha}^{4}&=0, \ \alpha\in\Oc^{\bq}_+.
\end{aligned}
\end{align}
Here {\scriptsize$\Oc^{\bq}_+=\{ 12, 123, 1^32^43^24, 1^32^33^34, 1^22^33^24, 1^22^23^24, 1^32^53^44^2, 12^23^24, 1^32^33^34^2, 1234, 23^24, 4 \}$}
and the degree of the integral is
\begin{align*}
\ya &= 88 \alpha_1+ 116\alpha_2 +98\alpha_3 +52\alpha_4.
\end{align*}


\subsubsection{The Dynkin diagram \emph{(\ref{eq:dynkin-ufo(6)}
		h)}}\label{subsubsec:ufo(6)-h}

\

The Nichols algebra $\toba_{\bq}$ is generated by $(x_i)_{i\in \I_4}$ with defining relations
\begin{align}\label{eq:rels-ufo(6)-h}
\begin{aligned}
& \begin{aligned}
x_{112}&=0; & x_{443}&=0; & x_2^2&=0; & x_{ij}&=0, \ i<j, \, \widetilde{q}_{ij}=1;\\
x_{442}&=0; & x_{23}^2&=0;  & x_3^2&=0; & x_{\alpha}^{4}&=0, \ \alpha\in\Oc^{\bq}_+;
\end{aligned}
\\
& [x_{124},x_2]_c=0; \quad x_{(24)}=q_{34}\zeta [x_{24},x_3]_c+q_{23}(1+\zeta)x_3x_{24}.
\end{aligned}
\end{align}
Here {\scriptsize$\Oc^{\bq}_+=\{ 1, 3, 12^23^34, 12^23^24, 1^22^43^34^2, 12^234, 1234,
	1^22^33^34^3, 12^23^24^2, 12^33^34^3, 234, 12^24 \}$}
and the degree of the integral is
\begin{align*}
\ya &= 52\alpha_1+ 98\alpha_2 +72\alpha_3 +88\alpha_4.
\end{align*}


\subsubsection{The associated Lie algebra} This is of type $G_2\times G_2$.

\subsection{Type $\Ufo(7)$}\label{subsec:type-ufo(7)}
Here $\zeta \in \G'_{12}$.
We describe first  the root system $\Ufo(7)$.

\subsubsection{Basic datum and root system}\label{subsubsec:type-ufo7-roots}
Below, $A_1^{(1)}$, $B_2$, $C_2$ and $G_2$ are numbered as in \eqref{eq:An-(1)}, \eqref{eq:dynkin-system-B}, \eqref{eq:dynkin-system-C} and \eqref{eq:dynkin-system-G}, respectively.
The basic datum and the bundle of Cartan matrices are described by the following diagram:
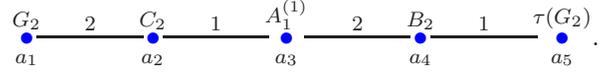
\begin{center}
	$\overset{G_2}{\underset{a_1}{\vtxgpd}}$
	\hspace{-5pt}\raisebox{3pt}{$\overset{2}{\rule{40pt}{0.5pt}}$}\hspace{-5pt}
	$\overset{C_2}{\underset{a_2}{\vtxgpd}}$
	\hspace{-5pt}\raisebox{3pt}{$\overset{1}{\rule{40pt}{0.5pt}}$}\hspace{-5pt}
	$\overset{A_{1}^{(1)}}{\underset{a_3}{\vtxgpd}}$
	\hspace{-5pt}\raisebox{3pt}{$\overset{2}{\rule{40pt}{0.5pt}}$}\hspace{-5pt}
	$\overset{B_2}{\underset{a_4}{\vtxgpd}}$
	\hspace{-5pt}\raisebox{3pt}{$\overset{1}{\rule{40pt}{0.5pt}}$}\hspace{-5pt}
	$\overset{\tau(G_2)}{\underset{a_5}{\vtxgpd}}$.
\end{center}
Using the notation \eqref{eq:notation-root-exceptional}, the bundle of (positive) root sets is the following: 
\begin{align*}
\Delta_{+}^{a_1}= & \{ 1, 1^32, 1^22, 12, 2 \}= \tau(\Delta_{+}^{a_5}), &
\Delta_{+}^{a_3}= & \{ 1, 1^22, 12, 12^2, 2 \}, \\
\Delta_{+}^{a_2}= & \{ 1, 1^22, 1^32^2, 12, 2 \}= \tau(\Delta_{+}^{a_4}).
\end{align*}

\subsubsection{Weyl groupoid}
\label{subsubsec:type-ufo7-Weyl}
The isotropy group  at $a_3 \in \cX$ is
\begin{align*}
\cW(a_3)= \langle \varsigma_1^{a_3}\varsigma_2 \varsigma_1\varsigma_2 \varsigma_1 = \varsigma_2^{a_3} \varsigma_1 \varsigma_2 \varsigma_1 \varsigma_2 \rangle \simeq \Z/2.
\end{align*}

\subsubsection{Incarnation} This is called $\ufo(7)$.
We assign the following Dynkin diagrams to $a_i$, $i\in\I_5$:
\begin{align}\label{eq:dynkin-ufo(7)}
\begin{aligned}
a_1\mapsto &\Dchaintwo{-\zeta ^3}{\zeta}{-1}, &
a_2\mapsto &\Dchaintwo{-\overline{\zeta}^{\,2}}{\overline{\zeta}}{-1}, &
a_3\mapsto &\Dchaintwo{-\overline{\zeta}^{\,2}}{-\zeta ^3}{-\zeta ^2}, \\
a_4 \mapsto & \Dchaintwo{-1}{-\zeta }{-\zeta ^2}, &
a_5 \mapsto & \Dchaintwo{-1}{-\overline{\zeta}}{-\zeta ^3}.
\end{aligned}
\end{align}

\subsubsection{PBW-basis and dimension} \label{subsubsec:type-ufo7-PBW}

Notice that the roots in each $\Delta_{+}^{a_i}$, $i\in\I_{5}$, are ordered from left to right, justifying the notation $\beta_1, \beta_2, \beta_3, \beta_4, \beta_{5}$.

The root vectors $x_{\beta_k}$ are described as in Remark \ref{rem:lyndon-word}.
Thus
\begin{align*}
\left\{ x_{\beta_{5}}^{n_{5}} x_{\beta_{4}}^{n_{4}} x_{\beta_{3}}^{n_{3}} x_{\beta_2}^{n_{2}}  x_{\beta_1}^{n_{1}} \, | \, 0\le n_{k}<N_{\beta_k} \right\}.
\end{align*}
is a PBW-basis of $\toba_{\bq}$. Hence $\dim \toba_{\bq}=2^23^24=144$.

\subsubsection{The Dynkin diagram \emph{(\ref{eq:dynkin-ufo(7)}
		a)}}\label{subsubsec:ufo(7)-a}

\

The Nichols algebra $\toba_{\bq}$ is generated by $(x_i)_{i\in \I_2}$ with defining relations
\begin{align}\label{eq:rels-ufo(7)-a}
x_1^4&=0; & x_2^2&=0; &  [ x_{112},x_{12}]_c&=0.
\end{align}
Here, $\Oc^{\bq}_+$ is empty and the degree of the integral is $\ya= 12\alpha_1 + 6\alpha_2$.

\subsubsection{The Dynkin diagram \emph{(\ref{eq:dynkin-ufo(7)}
		b)}}\label{subsubsec:ufo(7)-b}

\

The Nichols algebra $\toba_{\bq}$ is generated by $(x_i)_{i\in \I_2}$ with defining relations
\begin{align}\label{eq:rels-ufo(7)-b}
x_1^3&=0; & x_2^2&=0; & [[x_{112},x_{12}]_c,x_{12}]_c&=0.
\end{align}
Here, $\Oc^{\bq}_+$ is empty and the degree of the integral is $\ya= 12\alpha_1 + 8\alpha_2$.

\subsubsection{The Dynkin diagram \emph{(\ref{eq:dynkin-ufo(7)}
		c)}}\label{subsubsec:ufo(7)-c}

\

The Nichols algebra $\toba_{\bq}$ is generated by $(x_i)_{i\in \I_2}$ with defining relations
\begin{align}\label{eq:rels-ufo(7)-c}
x_1^3&=0; & x_2^3&=0; &  [x_1,x_{122}]_c +\frac{\zeta^4 q_{12}}{(3)_{\zeta}}x_{12}^2&=0.
\end{align}
Here, $\Oc^{\bq}_+$ is empty and the degree of the integral is $\ya= 8\alpha_1 + 8\alpha_2$.

\subsubsection{The Dynkin diagram \emph{(\ref{eq:dynkin-ufo(7)}
		d)}}\label{subsubsec:ufo(7)-d}

\

This diagram is of the shape of (\ref{eq:dynkin-ufo(7)} b) but with $\zeta^5$
instead of $\zeta$. Therefore the corresponding Nichols algebra has relations analogous to
\eqref{eq:rels-ufo(7)-b}.

\subsubsection{The Dynkin diagram \emph{(\ref{eq:dynkin-ufo(7)}
		e)}}\label{subsubsec:ufo(7)-e}

\

This diagram is of the shape of (\ref{eq:dynkin-ufo(7)} a) but with $\zeta^5$
instead of $\zeta$. Therefore the corresponding Nichols algebra has relations analogous to
\eqref{eq:rels-ufo(7)-a}.

\subsubsection{The associated Lie algebra} This is trivial.

\subsection{The Nichols algebras $\ufo(8)$}\label{subsec:type-ufo(8)}
Here $\zeta \in \G'_{12}$.

\subsubsection{Basic datum, root system and Weyl groupoid}
The  root system is of type $\Ufo(7)$ as in \S \ref{subsubsec:type-ufo7-roots}; hence the Weyl groupoid is as in \S \ref{subsubsec:type-ufo7-Weyl}. 

\subsubsection{Incarnation}
This is a new incarnation, denoted $\ufo(8)$.
We set the matrices $(\bq^{(i)})_{i\in\I_{3}}$, from left to right:
\begin{align}\label{eq:dynkin-ufo(8)}
\begin{aligned}
& \Dchaintwo{-\ztu}{-\zeta ^3}{-1}
& &\Dchaintwo{-\zeta ^2}{\zeta ^3}{-1}
& &\Dchaintwo{-\zeta ^2}{\zeta }{-\zeta ^2}
\end{aligned}
\end{align}
Now this is the incarnation:
\begin{align*}
a_i & \mapsto \bq^{(i)}, &
i & \in\I_3; &
a_i & \mapsto \tau(\bq^{(6-i)}), &
i & \in\I_{4,5}.
\end{align*}

\subsubsection{PBW-basis and dimension} \label{subsubsec:type-ufo8-PBW}
Notice that the roots in each $\Delta_{+}^{a_i}$, $i\in\I_{5}$, are ordered from left to right, justifying the notation $\beta_1, \beta_2, \beta_3, \beta_4, \beta_{5}$.

The root vectors $x_{\beta_k}$ are described as in Remark \ref{rem:lyndon-word}.
Thus
\begin{align*}
\left\{ x_{\beta_{5}}^{n_{5}} x_{\beta_{4}}^{n_{4}} x_{\beta_{3}}^{n_{3}} x_{\beta_2}^{n_{2}}  x_{\beta_1}^{n_{1}} \, | \, 0\le n_{k}<N_{\beta_k} \right\}.
\end{align*}
is a PBW-basis of $\toba_{\bq}$. Hence $\dim \toba_{\bq}=2^23^212=432$.

\subsubsection{The Dynkin diagram \emph{(\ref{eq:dynkin-ufo(8)} a)}} \label{subsubsec:ufo(8)-a}

\

The Nichols algebra $\toba_{\bq}$ is generated by $(x_i)_{i\in \I_2}$ with defining relations
\begin{align}\label{eq:rels-ufo(8)-a}
x_1^{12}&=0; & x_2^2&=0; & x_{11112}&=0; &  [x_{112},x_{12}]_c&=0.
\end{align}
Here, $\Oc^{\bq}_+=\{\alpha_1\}$ and the degree of the integral is $\ya= 20\alpha_1 + 16\alpha_2$.


\subsubsection{The Dynkin diagram \emph{(\ref{eq:dynkin-ufo(8)}
		b)}}\label{subsubsec:ufo(8)-b}

\

The Nichols algebra $\toba_{\bq}$ is generated by $(x_i)_{i\in \I_2}$ with defining relations
\begin{align}\label{eq:rels-ufo(8)-b}
x_1^3&=0; & x_2^2&=0; & x_{12}^{12}&=0; & [[x_{112},x_{12}]_c,x_{12}]_c&=0.
\end{align}
Here, $\Oc^{\bq}_+=\{\alpha_1+\alpha_2\}$ and the degree of the integral is $\ya= 20\alpha_1 + 16\alpha_2$.


\subsubsection{The Dynkin diagram \emph{(\ref{eq:dynkin-ufo(8)} c)}} \label{subsubsec:ufo(8)-c}

\

The Nichols algebra $\toba_{\bq}$ is generated by $(x_i)_{i\in \I_2}$ with defining relations
\begin{align}\label{eq:rels-ufo(8)-c}
\begin{aligned}
& x_1^3=0; && x_2^3=0; && x_{12}^{12}=0; & [x_1,x_{122}]_c&=(1+\zeta+\zeta^2)\zeta^4q_{12}x_{12}^2.
\end{aligned}
\end{align}
Here, $\Oc^{\bq}_+=\{\alpha_1+\alpha_2\}$ and the degree of the integral is $\ya= 16\alpha_1 + 16\alpha_2$.


\subsubsection{The associated Lie algebra} This is of type $A_1$.

\subsection{Type $\Ufo(9)$}\label{subsec:type-ufo(9)}
Here $\zeta \in \G'_{24}$.
We describe first  the root system $\Ufo(9)$.

\subsubsection{Basic datum and root system}
Below,$C_2$, $G_2$, $H_{5,1}$ and $H_{2,3}$ are numbered as in \eqref{eq:dynkin-system-G}, \eqref{eq:dynkin-system-C} and \eqref{eq:Hcd}, respectively.
The basic datum and the bundle of Cartan matrices are described by the following diagram:
\begin{center}
	$\overset{H_{5,1}}{{\underset{a_1}\vtxgpd}}$
	\hspace{-5pt}\raisebox{3pt}{$\overset{2}{\rule{40pt}{0.5pt}}$}\hspace{-5pt}
	$\overset{C_2}{{\underset{a_2}\vtxgpd}}$
	\hspace{-5pt}\raisebox{3pt}{$\overset{1}{\rule{40pt}{0.5pt}}$}\hspace{-5pt}
	$\overset{H_{2,3}}{{\underset{a_3}\vtxgpd}}$
	\hspace{-5pt}\raisebox{3pt}{$\overset{2}{\rule{40pt}{0.5pt}}$}\hspace{-5pt}
	$\overset{\tau(G_2)}{{\underset{a_4}\vtxgpd}}$.
\end{center}
Using the notation \eqref{eq:notation-root-exceptional}, the bundle of root sets is the following:
\begin{align*}
\Delta_{+}^{a_1}= & \{ 1,1^52,1^42,1^32,1^52^2,1^22,12,2 \}, \\
\Delta_{+}^{a_2}= & \{ 1,1^22,1^52^3,1^32^2,1^42^3,1^52^4,12,2   \}, \\
\Delta_{+}^{a_3}= & \{ 1,1^22,12,1^22^3,1^32^4,12^2,12^3,2 \}, \\
\Delta_{+}^{a_4}= & \{ 1,12,1^22^3,1^32^5,1^22^5,12^2,12^3,2 \}.
\end{align*}

\subsubsection{Weyl groupoid}
\label{subsubsec:type-ufo9-Weyl}
The isotropy group  at $a_2 \in \cX$ is
\begin{align*}
\cW(a_2)= \langle \varsigma_1^{a_2}\varsigma_2 \varsigma_1\varsigma_2 \varsigma_1, \varsigma_2^{a_2} \varsigma_1\varsigma_2 \rangle \simeq \Z/2  \times \Z/2.
\end{align*}

\subsubsection{Incarnation}
We set the matrices $(\bq^{(i)})_{i\in\I_{4}}$, from left to right:
\begin{align}\label{eq:dynkin-ufo(9)}
\begin{aligned}
&\Dchaintwo{\zeta ^6}{-\ztu}{\ \ -\ztu^{\, 4}}& &\Dchaintwo{\zeta ^6}{\zeta }{\ztu}&
&\Dchaintwo{-\ztu^{\, 4}}{\zeta ^5}{-1}& &\Dchaintwo{\zeta }{\ztu^{\, 5}}{-1}.
\end{aligned}
\end{align}
Now this is the incarnation:
\begin{align*}
a_1 & \mapsto \bq^{(4)}, &
a_2 & \mapsto \bq^{(3)}, &
a_3 & \mapsto \tau(\bq^{(1)}), &
a_4 & \mapsto \tau(\bq^{(2)}), &.
\end{align*}

\subsubsection{PBW-basis and dimension} \label{subsubsec:type-ufo9-PBW}
Notice that the roots in each $\Delta_{+}^{a_i}$, $i\in\I_{4}$, are ordered from left to right, justifying the notation $\beta_1, \dots, \beta_{8}$.

The root vectors $x_{\beta_k}$ are described as in Remark \ref{rem:lyndon-word}.
Thus
\begin{align*}
\left\{ x_{\beta_{8}}^{n_{8}} \cdots x_{\beta_2}^{n_{2}}  x_{\beta_1}^{n_{1}} \, | \, 0\le n_{k}<N_{\beta_k} \right\}.
\end{align*}
is a PBW-basis of $\toba_{\bq}$. Hence $\dim \toba_{\bq}=2^23^24^224^2=2^{12}3^4$.

\subsubsection{The Dynkin diagram \emph{(\ref{eq:dynkin-ufo(9)}
		a)}}\label{subsubsec:ufo(9)-a}

\

The Nichols algebra $\toba_{\bq}$ is generated by $(x_i)_{i\in \I_2}$ with defining relations
\begin{align}\label{eq:rels-ufo(9)-a}
\begin{aligned}
x_1^4&=0; & &x_{1112}^{24}=0; \quad x_{12}^{24}=0; \\
x_2^3&=0; & &[x_1,x_{122}]_c =\frac{1+\zeta^7}{1+\zeta} \zeta^{10}q_{12}x_{12}^2.
\end{aligned}
\end{align}
Here, $\Oc^{\bq}_+=\{3\alpha_1+\alpha_2, \alpha_1+\alpha_2\}$ and the degree of the integral is
\begin{align*}
\ya= 108\alpha_1 + 58\alpha_2.
\end{align*}

\subsubsection{The Dynkin diagram \emph{(\ref{eq:dynkin-ufo(9)}
		b)}}\label{subsubsec:ufo(9)-b}

\

The Nichols algebra $\toba_{\bq}$ is generated by $(x_i)_{i\in \I_2}$ with defining relations
\begin{align}\label{eq:rels-ufo(9)-b}
\begin{aligned}
x_1^4&=0; & x_2^{24}&=0; &  x_{112}^{24}&=0; & x_{221}&=0; & [[x_{112},x_{12}]_c,x_{12}]_c&=0.
\end{aligned}
\end{align}
Here, $\Oc^{\bq}_+=\{2\alpha_1+\alpha_2, \alpha_2\}$ and the degree of the integral is
\begin{align*}
\ya= 74\alpha_1 + 60\alpha_2.
\end{align*}

\subsubsection{The Dynkin diagram \emph{(\ref{eq:dynkin-ufo(9)}
		c)}}\label{subsubsec:ufo(9)-c}

\

The Nichols algebra $\toba_{\bq}$ is generated by $(x_i)_{i\in \I_2}$ with defining relations
\begin{align}\label{eq:rels-ufo(9)-c}
x_1^3&=0; & x_2^2&=0; & [x_{112},[x_{112},x_{12}]_c]_c^{24}&=0; & x_{12}^{24}&=0;
\end{align}
\vsp
\begin{align*}
[x_{112},[[x_{112},x_{12}]_c,x_{12}]_c]_c -\frac{1+\zeta+\zeta^6+2\zeta^7+\zeta^{17}}{1+\zeta^4+\zeta^6+\zeta^{11}} \zeta^{9}q_{12}
[x_{112},x_{12}]_c^2 =0.
\end{align*}
Here, $\Oc^{\bq}_+=\{5\alpha_1+3\alpha_2, \alpha_1+\alpha_2\}$ and the degree of the integral is
\begin{align*}
\ya= 168\alpha_1 + 112\alpha_2.
\end{align*}

\subsubsection{The Dynkin diagram \emph{(\ref{eq:dynkin-ufo(9)}
		d)}}\label{subsubsec:ufo(9)-d}

\

The Nichols algebra $\toba_{\bq}$ is generated by $(x_i)_{i\in \I_2}$ with defining relations
\begin{align}\label{eq:rels-ufo(9)-d}
\begin{aligned}
x_1^{24}&=0; & x_2^2&=0; & [x_{1112},x_{112}]_c^{24}&=0; \\
& & x_{1111112}&=0; & [x_{112},x_{12}]_c&=0.
\end{aligned}
\end{align}
Here, $\Oc^{\bq}_+=\{5\alpha_1+2\alpha_2, \alpha_1\}$ and the degree of the integral is
\begin{align*}
\ya= 168\alpha_1 + 58\alpha_2.
\end{align*}

\subsubsection{The associated Lie algebra} This is of type $A_1\times A_1$.

\subsection{Type $\Ufo(10)$}\label{subsec:type-ufo(10)}
Here $\zeta\in \G'_{20}$.
We describe first  the root system.

\subsubsection{Basic datum and root system}
Below, $G_2$ and $A_2^{(2)}$ are numbered as in \eqref{eq:dynkin-system-G} and  	
\eqref{eq:A2-(2)}, respectively.
The basic datum and the bundle of Cartan matrices are described by the following diagram:
\begin{center}
	$\overset{G_2}{\underset{a_1}{\vtxgpd}}$
	\hspace{-5pt}\raisebox{3pt}{$\overset{2}{\rule{40pt}{0.5pt}}$}\hspace{-5pt}
	$\overset{A_2^{(2)}}{\underset{a_1}{\vtxgpd}}$
	\hspace{-5pt}\raisebox{3pt}{$\overset{1}{\rule{40pt}{0.5pt}}$}\hspace{-5pt}
	$\overset{A_2^{(2)}}{\underset{a_1}{\vtxgpd}}$
	\hspace{-5pt}\raisebox{3pt}{$\overset{2}{\rule{40pt}{0.5pt}}$}\hspace{-5pt}
	$\overset{G_2}{\underset{a_1}{\vtxgpd}}$.
\end{center}
Using the notation \eqref{eq:notation-root-exceptional}, the bundle of root sets is the following:
\begin{align*}
\Delta_{+}^{a_1} = & \{ 1,1^32,1^22,1^52^3,1^32^2,1^42^3,12,2 \} = \Delta_{+}^{a_4}, \\
\Delta_{+}^{a_2} = & \{ 1,1^42,1^32,1^52^2,1^22,1^32^2,12,2 \}=\Delta_{+}^{a_3}.
\end{align*}

\subsubsection{Weyl groupoid}
\label{subsubsec:type-ufo10-Weyl}
The isotropy group  at $a_1 \in \cX$ is
\begin{align*}
\cW(a_1)= \langle \varsigma_1^{a_1}, \varsigma_2^{a_1} \varsigma_1\varsigma_2 \varsigma_1 \varsigma_2 \varsigma_1\varsigma_2 \rangle \simeq \Z/2  \times \Z/2.
\end{align*}

\subsubsection{Incarnation}

We assign the following Dynkin diagrams to $a_i$, $i\in\I_4$:
\begin{align}\label{eq:dynkin-ufo(10)}
\begin{aligned}
a_1\mapsto&\Dchaintwo{\zeta }{\ztu^{\, 3}}{-1},& a_2\mapsto&\Dchaintwo{-\ztu^{\, 2}}{\zeta ^3}{-1},
\\
a_3\mapsto&\Dchaintwo{-\ztu^{\, 2}}{-\zeta ^3}{-1},& a_4\mapsto&\Dchaintwo{-\zeta }{-\ztu^{\, 3}}{-1}.
\end{aligned}
\end{align}

\subsubsection{PBW-basis and dimension} \label{subsubsec:type-ufo10-PBW}
Notice that the roots in each $\Delta_{+}^{a_i}$, $i\in\I_{4}$, are ordered from left to right, justifying the notation $\beta_1, \dots, \beta_{8}$.

The root vectors $x_{\beta_k}$ are described as in Remark \ref{rem:lyndon-word}.
Thus
\begin{align*}
\left\{ x_{\beta_{8}}^{n_{8}} \cdots x_{\beta_2}^{n_{2}}  x_{\beta_1}^{n_{1}} \, | \, 0\le n_{k}<N_{\beta_k} \right\}.
\end{align*}
is a PBW-basis of $\toba_{\bq}$. Hence $\dim \toba_{\bq}=2^45^220^2=160000$.

\subsubsection{The Dynkin diagram \emph{(\ref{eq:dynkin-ufo(10)}
		a)}}\label{subsubsec:ufo(10)-a}

\

The Nichols algebra $\toba_{\bq}$ is generated by $(x_i)_{i\in \I_2}$ with defining relations
\begin{align}\label{eq:rels-ufo(10)-a}
\begin{aligned}
x_1^{20}&=0;  & x_2^2&=0;  &  [x_{112},x_{12}]_c^{20}&=0;  \\
&& x_{11112}&=0;  & [[[x_{112},x_{12}]_c,x_{12}]_c,x_{12}]_c&=0.
\end{aligned}
\end{align}
Here, $\Oc^{\bq}_+=\{\alpha_1, 3\alpha_1+2\alpha_2\}$ and the degree of the integral is
\begin{align*}
\ya= 100\alpha_1 + 54\alpha_2.
\end{align*}


\subsubsection{The Dynkin diagram \emph{(\ref{eq:dynkin-ufo(10)}
		b)}}\label{subsubsec:ufo(10)-b}

\

The Nichols algebra $\toba_{\bq}$ is generated by $(x_i)_{i\in \I_2}$ with defining relations
\begin{align}\label{eq:rels-ufo(10)-b}
\begin{aligned}
x_1^5&=0; & & x_{1112}^{20}=0; \quad x_{12}^{20}=0; \\
x_2^2&=0; & & [x_1,[x_{112},x_{12}]_c]_c +\frac{1-\zeta^{17}}{1-\zeta^2} q_{12}x_{112}^2=0.
\end{aligned}
\end{align}
Here, $\Oc^{\bq}_+=\{3\alpha_1+\alpha_2, \alpha_1+\alpha_2\}$ and the degree of the integral is
\begin{align*}
\ya= 100\alpha_1 + 48\alpha_2.
\end{align*}


\subsubsection{The Dynkin diagram \emph{(\ref{eq:dynkin-ufo(10)}
		c)}}\label{subsubsec:ufo(10)-c}

\

This diagram is of the shape of (\ref{eq:dynkin-ufo(10)} b) but with $-\zeta$
instead of $\zeta$. Thus the information on the  corresponding Nichols algebra is analogous to
\eqref{eq:rels-ufo(10)-b}.

\subsubsection{The Dynkin diagram \emph{(\ref{eq:dynkin-ufo(10)}
		d)}}\label{subsubsec:ufo(10)-d}

\

This diagram is of the shape of (\ref{eq:dynkin-ufo(10)} a) but with $-\zeta$ instead of $\zeta$. Thus
the information on the corresponding Nichols algebra is analogous to \eqref{eq:rels-ufo(10)-a}.

\subsubsection{The associated Lie algebra} This is of type $A_1\times A_1$.

\subsection{Type $\Ufo(11)$}\label{subsec:type-ufo(11)}
Here $\zeta\in \G'_{15}$.
We start by  the root system $\Ufo(11)$.

\subsubsection{Basic datum and root system}
Below, $A_2^{(2)}$, $C_2$ and $H_{2,3}$ are numbered as in \eqref{eq:A2-(2)}, \eqref{eq:dynkin-system-C} and  	
\eqref{eq:Hcd}, respectively.
The basic datum and the bundle of Cartan matrices are described by the following diagram:
\begin{center}
	$\overset{A_2^{(2)}}{\underset{a_1}{\vtxgpd}}$
	\hspace{-5pt}\raisebox{3pt}{$\overset{1}{\rule{40pt}{0.5pt}}$}\hspace{-5pt}
	$\overset{A_2^{(2)}}{\underset{a_2}{\vtxgpd}}$
	\hspace{-5pt}\raisebox{3pt}{$\overset{2}{\rule{40pt}{0.5pt}}$}\hspace{-5pt}
	$\overset{C_2}{\underset{a_3}{\vtxgpd}}$
	\hspace{-5pt}\raisebox{3pt}{$\overset{1}{\rule{40pt}{0.5pt}}$}\hspace{-5pt}
	$\overset{H_{2,3}}{\underset{a_4}{\vtxgpd}}$.
\end{center}
Using the notation \eqref{eq:notation-root-exceptional}, the bundle of root sets is the following:
\begin{align*}
\Delta_{+}^{a_1}= & \{ 1, 1^42, 1^32, 1^22, 1^32^2, 1^42^3, 12, 2 \}, \\
\Delta_{+}^{a_2}= & \{ 1, 1^42, 1^32, 1^82^3, 1^52^2, 1^22, 12, 2 \}, \\
\Delta_{+}^{a_3}= & \{ 1, 1^22, 1^52^3, 1^82^5, 1^32^2, 1^42^3, 12, 2 \}, \\
\Delta_{+}^{a_4}= & \{ 1, 1^22, 12, 1^22^3, 12^2, 1^22^5, 12^3, 2 \}.
\end{align*}

\subsubsection{Weyl groupoid}
\label{subsubsec:type-ufo11-Weyl}
The isotropy group  at $a_1 \in \cX$ is
\begin{align*}
\cW(a_1)= \langle \varsigma_1^{a_1}\varsigma_2 \varsigma_1 \varsigma_2 \varsigma_1 \varsigma_2 \varsigma_1,
\varsigma_2^{a_1}  \rangle \simeq \Z/2  \times \Z/2.
\end{align*}

\subsubsection{Incarnation}
We assign the following Dynkin diagrams to $a_i$, $i\in\I_4$:
\begin{align}\label{eq:dynkin-ufo(11)}
\begin{aligned}
a_1\mapsto &\Dchaintwo{\zeta ^3}{-\zeta ^4}{-\ztu^{\,
		4}},
& 
a_2\mapsto &\Dchaintwo{\zeta ^3}{-\zeta ^2}{-1},
\\
a_3\mapsto &\Dchaintwo{\zeta ^5}{-\ztu^{\, 2}}{-1},
& 
a_4\mapsto &\Dchaintwo{-\zeta }{-\ztu^{\, 3}}{\zeta ^5}.
\end{aligned}
\end{align}

\subsubsection{PBW-basis and dimension} \label{subsubsec:type-ufo11-PBW}
Notice that the roots in each $\Delta_{+}^{a_i}$, $i\in\I_{4}$, are ordered from left to right, justifying the notation $\beta_1, \dots, \beta_{8}$.

The root vectors $x_{\beta_k}$ are described as in Remark \ref{rem:lyndon-word}.
Thus
\begin{align*}
\left\{ x_{\beta_{8}}^{n_{8}} \cdots x_{\beta_2}^{n_{2}}  x_{\beta_1}^{n_{1}} \, | \, 0\le n_{k}<N_{\beta_k} \right\}.
\end{align*}
is a PBW-basis of $\toba_{\bq}$. Hence $\dim \toba_{\bq}=2^23^25^230^2=2^43^45^4$.

\subsubsection{The Dynkin diagram \emph{(\ref{eq:dynkin-ufo(11)}
		a)}}\label{subsubsec:ufo(11)-a}

\
The Nichols algebra $\toba_{\bq}$ is generated by $(x_i)_{i\in \I_2}$ with defining relations
\begin{align}\label{eq:rels-ufo(11)-a}
\begin{aligned}
x_1^5&=0; & x_2^{30}&=0; & & [x_1,[x_{112},x_{12}]_c]_c =\frac{1-\zeta^2}{1+\zeta^7}\zeta^9q_{12} x_{112}^2; \\
x_{112}^{30}&=0; & x_{221}&=0; & & [[[x_{112},x_{12}]_c,x_{12}]_c,x_{12}]_c=0.
\end{aligned}
\end{align}
Here, $\Oc^{\bq}_+=\{\alpha_2, 2\alpha_1+\alpha_2\}$ and the degree of the integral is
\begin{align*}
\ya= 86\alpha_1 + 72\alpha_2.
\end{align*}


\subsubsection{The Dynkin diagram \emph{(\ref{eq:dynkin-ufo(11)}
		b)}}\label{subsubsec:ufo(11)-b}

\
The Nichols algebra $\toba_{\bq}$ is generated by $(x_i)_{i\in \I_2}$ with defining relations
\begin{align}\label{eq:rels-ufo(11)-b}
\begin{aligned}
x_1^5&=0; & x_{112}^{30}&=0; & [x_{112},x_{12}]_c&=0; \\
x_2^2&=0; & x_{11112}^{30}&=0; & [[x_{1112},x_{112}]_c,x_{112}]_c&=0.
\end{aligned}
\end{align}
Here, $\Oc^{\bq}_+=\{4\alpha_1+\alpha_2, 2\alpha_1+\alpha_2\}$ and the degree of the integral is
\begin{align*}
\ya= 210\alpha_1 + 72\alpha_2.
\end{align*}


\subsubsection{The Dynkin diagram \emph{(\ref{eq:dynkin-ufo(11)}
		c)}}\label{subsubsec:ufo(11)-c}

\
The Nichols algebra $\toba_{\bq}$ is generated by $(x_i)_{i\in \I_2}$ with defining relations
\begin{align}\label{eq:rels-ufo(11)-c}
\begin{aligned}
x_1^3&=0; & x_2^2&=0; & [[x_{112},x_{12}]_c,x_{12}]_c^{30}&=0;  \\
& & x_{112}^{30}&=0; & [[[x_{112},x_{12}]_c,x_{12}]_c,x_{12}]_c& =0.
\end{aligned}
\end{align}
Here, $\Oc^{\bq}_+=\{4\alpha_1+3\alpha_2, 2\alpha_1+\alpha_2\}$ and the degree of the integral is
\begin{align*}
\ya= 210\alpha_1 + 140\alpha_2.
\end{align*}


\subsubsection{The Dynkin diagram \emph{(\ref{eq:dynkin-ufo(11)}
		d)}}\label{subsubsec:ufo(11)-d}

\

The Nichols algebra $\toba_{\bq}$ is generated by $(x_i)_{i\in \I_2}$ with defining relations
\begin{align}\label{eq:rels-ufo(11)-d}
\begin{aligned}
x_1^{30}&=0; & x_2^3&=0; &  & [x_1,x_{122}]_c +\frac{1+\zeta^{13}}{1+\zeta^{12}} \zeta^{10}q_{12} x_{12}^2=0;  \\
x_{1112}^{30}&=0; & x_{11112}&=0; & & [[x_{112},x_{12}]_c,x_{12}]_c=0.
\end{aligned}
\end{align}
Here, $\Oc^{\bq}_+=\{\alpha_1, 3\alpha_1+\alpha_2\}$ and the degree of the integral is
\begin{align*}
\ya= 140\alpha_1 + 74\alpha_2.
\end{align*}


\subsubsection{The associated Lie algebra} This is of type $A_1\times A_1$.

\subsection{Type $\Ufo(12)$}\label{subsec:type-ufo(12)}
Here $\zeta\in \G'_{7}$.
We start by the root system $\Ufo(12)$.

\subsubsection{Basic datum and root system}
Below, $G_2$ and $H_{5,1}$ are numbered as in \eqref{eq:dynkin-system-G} and \eqref{eq:Hcd}, respectively.
The basic datum and the bundle of Cartan matrices are described by the following diagram:
\begin{center}
	$\overset{G_2}{\underset{a_1}{\vtxgpd}}$   \hspace{-5pt}\raisebox{3pt}{$\overset{3}{\rule{40pt}{0.5pt}}$}\hspace{-5pt}  $\overset{H_{5,1}}{\underset{a_2}{\vtxgpd}}$.
\end{center}
Using the notation \eqref{eq:notation-root-exceptional}, the bundle of root sets is the following:
\begin{align*}
\Delta_{+}^{a_1}= & \{ 1,1^32,1^22,1^72^4,1^52^3,1^82^5,1^32^2,1^72^5,1^42^3,1^52^4,12,2 \}, \\
\Delta_{+}^{a_2}= & \{ 1,1^52,1^42,1^72^2,1^32,1^82^3,1^52^2,1^72^3,1^22,1^32^2,12,2 \}.
\end{align*}

\subsubsection{Weyl groupoid}
\label{subsubsec:type-ufo12-Weyl}
The isotropy group  at $a_1 \in \cX$ is
\begin{align*}
\cW(a_1)= \langle \varsigma_1^{a_1}, \varsigma_2^{a_1} \varsigma_1 \varsigma_2 \rangle \simeq \mathbb{D}_6.
\end{align*}

\subsubsection{Incarnation}
We assign the following Dynkin diagrams to $a_i$, $i\in\I_2$:
\begin{align}\label{eq:dynkin-ufo(12)}
\begin{aligned}
a_1\mapsto&\Dchaintwo{-\zeta }{-\ztu^{\, 3}}{-1}& a_2\mapsto&\Dchaintwo{-\ztu^{\, 2}}{-\zeta ^3}{-1}.
\end{aligned}
\end{align}

\subsubsection{PBW-basis and dimension} \label{subsubsec:type-ufo12-PBW}
Notice that the roots in each $\Delta_{+}^{a_i}$, $i\in\I_{2}$, are ordered from left to right, justifying the notation $\beta_1, \dots, \beta_{12}$.

The root vectors $x_{\beta_k}$ are described as in Remark \ref{rem:lyndon-word}.
Thus
\begin{align*}
\left\{ x_{\beta_{12}}^{n_{12}} \cdots x_{\beta_2}^{n_{2}}  x_{\beta_1}^{n_{1}} \, | \, 0\le n_{k}<N_{\beta_k} \right\}.
\end{align*}
is a PBW-basis of $\toba_{\bq}$. Hence $\dim \toba_{\bq}=2^614^6=2^{12}7^6$.

\subsubsection{The Dynkin diagram \emph{(\ref{eq:dynkin-ufo(12)}
		a)}}\label{subsubsec:ufo(12)-a}

\

The Nichols algebra $\toba_{\bq}$ is generated by $(x_i)_{i\in \I_2}$ with defining relations
\begin{align}\label{eq:rels-ufo(12)-a}
\begin{aligned}
& \begin{aligned}
x_{11112}&=0; & x_2^2&=0; & [[x_{112},x_{12}]_c,x_{12}]_c^{14}&=0;  \\
&    & x_{112}^{14}&=0; &  [x_{112},x_{12}]_c^{14}&=0; \\
x_1^{14}&=0; & x_{12}^{14}&=0;   & [x_{112},[x_{112},x_{12}]_c]_c^{14}&=0;
\end{aligned}
\\
& [x_{112},[[x_{112},x_{12}]_c,x_{12}]_c]_c=q_{12} \frac{\zeta-3\zeta^2-3\zeta^3+\zeta^4-3\zeta^6}{-2\zeta+2\zeta^3-\zeta^5+\zeta^6}[x_{112},x_{12}]_c^2.
\end{aligned}
\end{align}
Here, $\Oc^{\bq}_+=\{\alpha_1, 2\alpha_1+\alpha_2,  5\alpha_1+3\alpha_2, 4\alpha_1+3\alpha_2, 3\alpha_1+2\alpha_2,  \alpha_1+\alpha_2\}$ and the degree of the integral is
\begin{align*}
\ya= 238\alpha_1 + 150\alpha_2.
\end{align*}


\subsubsection{The Dynkin diagram \emph{(\ref{eq:dynkin-ufo(12)}
		b)}}\label{subsubsec:ufo(12)-b}

\

The Nichols algebra $\toba_{\bq}$ is generated by $(x_i)_{i\in \I_2}$ with defining relations
\begin{align}\label{eq:rels-ufo(12)-b}
\begin{aligned}
& \begin{aligned}
x_2^2&=0; & x_{112}^{14}&=0; & x_{11112}^{14}&=0; &  [x_{1112},x_{112}]_c^{14}&=0; \\
x_1^{14}&=0; & x_{12}^{14}&=0; & x_{1112}^{14}&=0;
\end{aligned}
\\
& x_{1111112}=0; \quad [x_1,[x_{112},x_{12}]_c]_c= q_{12}\frac{1+\zeta^4}{\zeta^2-1}x_{112}^2.
\end{aligned}
\end{align}
Here, $\Oc^{\bq}_+=\{\alpha_1, 4\alpha_1+\alpha_2,  3\alpha_1+\alpha_2, 5\alpha_1+2\alpha_2, 2\alpha_1+\alpha_2, \alpha_1+\alpha_2\}$ and the degree of the integral is
\begin{align*}
\ya= 238\alpha_1 + 90\alpha_2.
\end{align*}


\subsubsection{The associated Lie algebra} This is of type $G_2$.


\begin{thebibliography}{XXXX}


\bibitem[A]{A}  N. Andruskiewitsch, \emph{On finite-dimensional Hopf algebras}. Proc. of the ICM, Seoul 2014. S.-Y. Jang et al. (Editors). Vol II (2014), 117--142.

\bibitem[AA1]{AA} N. Andruskiewitsch, I. Angiono, \emph{On Nichols algebras with
generic braiding} Modules and Comodules. Trends in Mathematics. Brzezinski, T.; Gomez
Pardo, J.L.; Shestakov, I.; Smith, P.F. (Eds.), pp. 47-64 (2008). ISBN: 978-3-7643-8741-9.

\bibitem[AA2]{AA-GRS-CLS-NA} \bysame, \emph{Generalized root systems,
contragredient Lie superalgebras and Nichols algebras}, in preparation.

\bibitem[AA3]{AA-nichols-quantum} \bysame, 
\emph{On Nichols algebras over quantized enveloping algebras}, in preparation.


\bibitem[AAH1]{AAH} N. Andruskiewitsch, I. Angiono, I. Heckenberger. \emph{On finite GK-dimensional Nichols algebras over abelian groups}, \texttt{arXiv:1606.02521}.

\bibitem[AAH2]{AAH2} \bysame. \emph{On finite GK-dimensional Nichols algebras of diagonal type}, in preparation.


\bibitem[AAR1]{AAR} N. Andruskiewitsch, I. Angiono, F. Rossi Bertone. \emph{The quantum divided power algebra of a finite-dimensional Nichols algebra of diagonal type}, Math. Res. Lett., in press. \texttt{arXiv:1501.04518}.

\bibitem[AAR2]{AAR2} \bysame \emph{A finite-dimensional Lie algebra arising from a Nichols algebra of   diagonal type (rank 2)}, Bull. Belg. Math. Soc. Simon Stevin 24 (1) (2017), 15--34.

\bibitem[AAR3]{AAR3} \bysame \emph{Lie algebras arising from Nichols algebras of diagonal type}, in preparation.


\bibitem[AAY]{AAY} N. Andruskiewitsch, I. Angiono, H. Yamane. \emph{On pointed Hopf superalgebras}, Contemp. Math. \textbf{544} (2011), 123--140.

\bibitem[AC]{AC} N. Andruskiewitsch, J. Cuadra, \emph{On the structure of (co-Frobenius) Hopf algebras}. J.
Noncommut. Geom. 7 (2013), 83-104.

\bibitem[AD]{AD}  N. Andruskiewitsch,  S. Dascalescu. 
\emph{On finite quantum groups at -1}. Algebr. Represent. Theory \textbf{8} (2005), 11--34.

\bibitem[AGM]{AGM} N. Andruskiewitsch, C. Galindo, M.   M\"uller. \emph{On finite GK-dimensional Nichols algebras over abelian groups}, 	Publ. Mat., Barc., to appear.

\bibitem[AGi]{AGi} N. Andruskiewitsch, J. M. J. Giraldi, 
\emph{Nichols algebras that are quantum planes}, Linear Multilinear Algebra,
to appear.

\bibitem[AGr]{AG} N. Andruskiewitsch, M. Gra\~na,
\emph{Braided {H}opf algebras over non-abelian groups\/}, Bol.
Acad. Ciencias (Cordoba) \textbf{63} (1999), 45--78. Also in
\texttt{math.QA/9802074}.

\bibitem[AHS]{AHS} N. Andruskiewitsch, I. Heckenberger, H.-J. Schneider, 
\emph{The Nichols algebra of a semisimple Yetter-Drinfeld module}, Amer. J. Math. \textbf{132} (2010),  1493--1547.




\bibitem[ARS]{ARS} N. Andruskiewitsch, D. Radford, H.-J. Schneider,
\emph{Complete reducibility theorems for modules over pointed Hopf algebras}, J. Algebra \textbf{324}, 2932--2970 (2010).

\bibitem[AS1]{AS1} N. Andruskiewitsch, H.-J. Schneider, \emph{Lifting of
quantum linear spaces and pointed Hopf algebras of order $p^3$}, J. Algebra \textbf{209},
658-691 (1998).

\bibitem[AS2]{AS2} \bysame
\emph{Finite quantum groups and Cartan matrices}, Adv. Math.
\textbf{154}, 1--45 (2000).

\bibitem[AS3]{AS3} \bysame
\emph{Pointed Hopf algebras}, ``New directions in Hopf algebras'',
MSRI series Cambridge Univ. Press; 1--68 (2002).

\bibitem[AS4]{AS4} \bysame
\emph{On the classification of finite-dimensional pointed Hopf
algebras}, Annals of Mathematics Vol. \textbf{171} (2010), No. 1, 375--417.

\bibitem[Ang1]{A-standard} I. Angiono, \emph{On Nichols algebras with standard braiding},
Algebra and Number Theory Vol. 3, No. \textbf{1}, 35-106,  (2009).

\bibitem[Ang2]{A-jems} \bysame \emph{A presentation by generators and relations of Nichols
algebras of diagonal type and convex orders on root systems}. J. Europ. Math. Soc.  \textbf{17} (2015), 2643--2671.


\bibitem[Ang3]{A-presentation} \bysame \emph{On Nichols algebras of diagonal type}. J. Reine Angew. Math.  683 (2013), 189--251.

\bibitem[Ang4]{A-ufo}  \bysame \emph{Nichols algebras of unidentified diagonal type}. Comm. Alg \textbf{41} (2013), 4667--4693.

\bibitem[Ang5]{A-pre-Nichols}  \bysame \emph{Distinguished pre-Nichols algebras}. Transf. Groups \textbf{21} (2016), 1--33.


\bibitem[BDR]{BDR} M. Beattie, S. D\u{a}sc\u{a}lescu, \c{S}. Raianu. \emph{Lifting of
Nichols algebras of type $B_2$}. Israel J. Math. {\bf 132} (2002),  1--28.

\bibitem[Be]{beck} J. Beck, \emph{Braid group action and quantum affine algebras}, Comm. Math. Phys. \textbf{165} (1994),
555--568.


\bibitem[BGL]{BGL} S. Bouarroudj, P. Grozman, D. Leites, \emph{Classification of finite dimensional modular
Lie superalgebras with indecomposable Cartan matrix}. SIGMA Symmetry Integrability Geom. Methods Appl. \textbf{5} (2009),
Paper 060, 63 pp.

\bibitem[Bo]{Bourbaki} N. Bourbaki, \emph{Groupes et alg\`ebres de Lie, Ch. 4, 5 et 6}, Hermann, Paris, 1968.

\bibitem[Br]{Br} G. Brown, \emph{Properties of a 29-dimensional simple Lie algebra of characteristic three}, Math. Ann. \textbf{261} (1982), 487--492.

\bibitem[Bu]{Bu} S. Burciu, \emph{A class of Drinfeld doubles that are ribbon algebras}, J. Algebra \textbf{320} (2008), 2053--2078.

\bibitem[Bur]{Burr-pbw} N. Burroughs,  \emph{The universal R-matrix for $U_qsl(3)$ and beyond! }
Comm. Math. Phys. \textbf{127} (1990), 109--128.

\bibitem[CE]{CE} I. Cunha, A. Elduque, \emph{An extended Freudenthal magic square in characteristic 3}. J. Algebra \textbf{317} (2007), 471--509.

\bibitem[CH1]{CH-at most 3} M. Cuntz, I. Heckenberger. \emph{Weyl groupoids with at most three objects}. J. Pure Appl. Algebra \textbf{213} (2009), 1112--1128.

\bibitem[CH2]{CH-classification} \bysame, \emph{Finite Weyl groupoids}. J. Reine Angew. Math. \textbf{702} (2015), 77--108.

\bibitem[Da]{Da} I. Damiani, \emph{Drinfeld Realization of Affine Quantum Algebras: the Relations}, Publ. RIMS Kyoto Univ. \textbf{48} (2012), 661--733.

\bibitem[DP]{DP} C. De Concini, C. Procesi. \emph{Quantum groups}. D-modules, representation theory, and quantum groups, 31--
140, Lecture Notes in Math. \textbf{1565}, Springer, 1993.



\bibitem[Dr]{Dr} V.G. Drinfeld, \emph{Quantum groups}, Proc. Int. Cong. Math, Berkeley \textbf{1}, 789--820 (1986).



\bibitem[E1]{E1} A. Elduque, \emph{New simple Lie superalgebras in characteristic 3}. J. Algebra \textbf{296} (2006), 196--233.

\bibitem[E2]{E2} \bysame \emph{Some new simple modular Lie superalgebras}. Pacific J. Math. \textbf{231} (2007), 337--359.


\bibitem[EGNO]{egno} P. Etingof, S. Gelaki, D. Nikshych,V. Ostrik. \emph{Tensor categories}, Mathematical Surveys and Monographs 205. Providence, RI: American Mathematical Society (AMS). xvi, 343 p. (2015).

\bibitem[GGi]{GGi} G. A. Garc\'\i a, J. M. J. Giraldi. \emph{On Hopf Algebras over quantum subgroups}, \texttt{arXiv:1605.03995}.

\bibitem[H1]{H-Weyl gpd} I. Heckenberger, \emph{The Weyl groupoid of a Nichols algebra of diagonal type}, Invent. Math. \textbf{164}, 175--188 (2006).


\bibitem[H2]{Hrk2}  \bysame \emph{Examples of finite dimensional rank 2 Nichols algebras of diagonal type}, Compositio Math. \textbf{143} (2007), 165--190.

\bibitem[H3]{H-classif RS} \bysame \emph{Classification of arithmetic root
systems}, Adv. Math. \textbf{220} (2009) 59--124.

\bibitem[H4]{H-lusztig iso} \bysame \emph{Lusztig isomorphism for Drinfel'd doubles of bosonizations of Nichols algebras of diagonal type}, J. Algebra \textbf{323} (2010), 2130--2182.

\bibitem[HS1]{HS-london} I. Heckenberger and H.-J   Schneider, 
\emph{Root systems and Weyl groupoids for Nichols algebras}, Proc. Lond. Math. Soc. \textbf{101} (2010),  623--654.

\bibitem[HS2]{HS} \bysame
\emph{Yetter-Drinfeld modules over bosonizations of dually paired Hopf algebras},Adv. Math. \textbf{244} (2013) 354--394.





\bibitem[HeV]{HV} I. Heckenberger, L. Vendramin, \emph{A classification of Nichols algebras of semi-simple Yetter-Drinfeld modules over non-abelian groups}, J. Europ. Math. Soc., to appear.

\bibitem[HeW]{HW} I. Heckenberger, J. Wang. \emph{Rank 2 Nichols algebras of diagonal type over fields of positive characteristic}, SIGMA \textbf{11} (2015), 011.

\bibitem[HeY]{HY} I. Heckenberger, H. Yamane, \emph{A generalization of Coxeter groups, root systems, and Matsumoto's theorem}, Math. Z.
\textbf{259} (2008), 255--276.

\bibitem[Hg]{helbig}  M. Helbig.  {\em On the Lifting of Nichols Algebras}, 
Comm. Alg. \textbf{40} (2012), 3317--3351.

\bibitem[HX]{HX} Naihong Hu, Rongchuan Xiong. \emph{Some Hopf algebras of dimension 72 without the Chevalley property}, \texttt{arXiv:1612.04987}; \emph{Eight classes of new Hopf algebras of dimension 128 without the Chevalley property}, \texttt{arXiv:1701.01991}.


\bibitem[JS]{JS} Joyal, A., Street, R., Braided tensor categories, \emph{Adv. Math.} \textbf{102} (1993), 20--78.

\bibitem[K1]{K-super} V. Kac, \emph{Lie superalgebras}. Adv. Math. \textbf{26} (1977), 8--96.

\bibitem[K2]{K-libro} \bysame \emph{Infinite-dimensional Lie algebras}. Third edition. Cambridge University Press, Cambridge, 1990. xxii+400 pp.

\bibitem[KaW]{KW-exponentials} V. Kac, B. Weisfeiler, \emph{Exponentials in Lie algebras of characteristic p},   Math. USSR Izv. 5 (1971), 777--803.

\bibitem[Kh1]{Kh} V. Kharchenko, \emph{A quantum analogue of the Poincar\'e-Birkhoff-Witt theorem}.
Algebra and Logic \textbf{38} (1999), 259--276.

\bibitem[Kh2]{Kh-libro} \bysame \emph{Quantum Lie theory}.
Lect. Notes Math. \textbf{2150} (2015), Springer-Verlag.

\bibitem[KR]{RK} L. Kauffman, D. Radford, \emph{A necesary and sufficient condition for a finite dimensional Drinfel'd Double to be a ribbon Hopf algebra}, J. Algebra \textbf{159}, 98--114 (1993).

\bibitem[LaR]{LR} P. Lalonde, A. Ram, \emph{Standard Lyndon bases of Lie algebras and enveloping algebras}, Trans. Am. Math. Soc. \textbf{137} (1995), 1821--1830.

\bibitem[Lo]{Lo} M. Lothaire, \emph{Combinatorics on Words}. Encyclopedia Math. Appl. \textbf{17} , Addison-Wesley, 1983.

\bibitem[L1]{Lu-jams1990} G. Lusztig, \emph{Finite-dimensional Hopf algebras arising from quantized universal enveloping algebra}, 
J. Amer. Math. Soc. \textbf{3} (1990),  257--296.

\bibitem[L2]{Lu-dedicata} \bysame \emph{Quantum groups at roots of 1}, Geom. Dedicata \textbf{35} (1990), 89--113. 

\bibitem[L3]{Lu} \bysame \emph{Introduction to quantum groups}, Birkh\"auser (1993).

\bibitem[M]{M} S. MacLane, \emph{Categories for the working mathematician}, Springer (1971).

\bibitem[Ma]{Ma} S. Majid, \emph{Foundations of Quantum Group Theory}, Cambridge University Press (1995).


\bibitem[Mk]{Masuoka} A. Masuoka. \emph{Abelian and non-abelian second cohomologies of quantized
enveloping algebras}, J.  Algebra \textbf{320} (2008), 1--47.


\bibitem[Mo]{Mo} S. Montgomery, \emph{Hopf algebras and their action on rings},
CBMS Regional Conference Series \textbf{82} (1993).


\bibitem[N]{Ni}  Nichols, W. D., Bialgebras of type one, \emph{Comm. Algebra} \textbf{6} (1978),  1521--1552.

\bibitem[P]{P} P. Papi, \emph{A characterization of a special ordering in a root system}, Proc. Am. Math. Soc. \textbf{120}  (1994),  661--665.

\bibitem[R]{radford-book} D. E. Radford,   {\it Hopf algebras}. 
Series on Knots and Everything 49. Hackensack, NJ: World Scientific.  xxii, 559 p. (2012)







\bibitem[Ro1]{Ro-pbw} M. Rosso,  \emph{An analogue of P.B.W. theorem and the universal R-matrix for $U_hsl(N+1)$. }
Comm. Math. Phys. \textbf{124} (1989), 307--318.




\bibitem[Ro2]{Ro1} \bysame
\emph{Groupes quantiques et algebres of battage quantiques}, C. R.
A. S. (Paris) \textbf{320} (1995), 145--148.

\bibitem[Ro3]{Ro2} \bysame \emph{Quantum groups and quantum shuffles},
Inventiones Math. \textbf{133}   (1998),  399--416.

\bibitem[Ro4]{R2}  \bysame \emph{Lyndon words and Universal R-matrices}, talk at MSRI, October 26, 1999, available at {\tt
http://www.msri.org};  \emph{Lyndon basis and the multiplicative formula for R-matrices}, preprint (2003).



\bibitem[S1]{Sch} P.~Schauenburg. {\it A characterization of
the Borel-like subalgebras of quantum enveloping algebras},
Comm. Algebra {\bf24} (1996), 2811--2823.


\bibitem[Se1]{Ser-root systems} V. Serganova, \emph{On generalizations of root systems}. Comm. Algebra \textbf{24} (1996), 4281--4299.


\bibitem[Se2]{Ser-superKM} \bysame \emph{Kac-Moody superalgebras and integrability}.
Neeb, Karl-Hermann (ed.) et al., Developments and trends in infinite-dimensional Lie theory. Basel: Birkh\"auser.
Progress in Mathematics 288, 169-218 (2011).



\bibitem[Sh]{Sh} A.I. Shirshov. \emph{On bases for free Lie algebra}, Alg. Log. \textbf{1}, no. 1, 14--19 (1962).

\bibitem[Sk]{Skryabin} S. Skryabin, \emph{A contragredient 29-dimensional Lie algebra of characteristic 3}.  Siberian Math. J. \textbf{34} (1993), 548--554.

\bibitem[T1]{T-gln}  M. Takeuchi. \emph{Some topics on $\mathrm{GL}_q(n)$}. 
J. Algebra \textbf{147}  (1992), 379-410.


\bibitem[T2]{T} \bysame \emph{Finite Hopf algebras in braided tensor categories}. J. Pure Appl. Algebra \textbf{138} (1999), 59--82.

\bibitem[U]{U} S. Ufer. \emph{PBW bases for a class of braided Hopf algebras}, J. Alg. \textbf{280} (2004) 84-119.


\bibitem[W]{Wg} J. Wang.  \emph{Rank three Nichols algebras of diagonal type over fields of positive characteristic},   
Israel J. Math. 218 (2017) 1--26.

\bibitem[Wo]{W} S. L. Woronowicz. Differential calculus on compact matrix pseudogroups (quantum
groups), \emph{Comm. Math. Phys.} \textbf{122} (1989),  125--170.


\bibitem[Y]{Y-pbw} H. Yamane. \emph{A Poincar\'e-Birkhoff-Witt theorem for quantized universal enveloping algebras of type $A_N$}. 
Publ. Res. Inst. Math. Sci. \textbf{25} (1989) 503--520.

\bibitem[Y]{Y-super}  \bysame \emph{Quantized enveloping algebras associated to simple Lie superalgebras and their
universal R-matrices}, Publ. Res. Inst. Math. Sci. \textbf{30} (1994), 15--87.

\end{thebibliography}
\end{document}